%% file: main.tex
\documentclass[12pt,final,CPage]{ufthesis}
\include{packages}

\SetFullName{Wanyu Bian}%
\SetThesisType{Dissertation}%{Dissertation} %{Thesis}
\SetDegreeType{Doctor of Philosophy}% {Doctor of Philosophy} {Master of Science}
\SetGradMonth{August}%
\SetGradYear{2022}%
\SetDepartment{Mathematics}%
\SetChair{Yunmei Chen}%
\SetTitle{Optimization-Based Deep learning methods for Magnetic Resonance Imaging Reconstruction and Synthesis}

\include{usersetcommands}

\begin{document}

 % % % % % % % % % % % % % % % % % % % % % % % % % % % % % % % % % % % % % %
 % Remember - You MUST get a .bst file that matches the Journal in your
 % field that you choose as your Reference example
 % NONE of these examples will satisfy the Graduate Editorial Office
 % if they don't match your Journal example!!!!
 % NOTE: If you use a numbered reference system and your references
 % are set in parentheses rather than brackets you need to select the
 % Natbib option "numbers sort and compress" in the packages.tex file
 % % % % % % % % % % % % % % % % % % % % % % % % % % % % % % % % % % % % % %

 %Note that the path separator is a forward slash NOT a back slash
 %Place YOUR .bst file in the bst folder and use that filename (without the .bst extension)
 % as your Bibliography Style file

%\bibliographystyle{bst/abbrv}
%\bibliographystyle{bst/abbrvnat}
%\bibliographystyle{bst/abbrvurl_uf}
%\bibliographystyle{bst/alphaurl_uf}
%\bibliographystyle{bst/apa-good}
%\bibliographystyle{bst/Chicago_Web}
%\bibliographystyle{bst/ecology_web}
%\bibliographystyle{bst/IEEEtran}
%\bibliographystyle{bst/mla_web}
%\bibliographystyle{bst/mla-good}
%\bibliographystyle{bst/plainnat}
%\bibliographystyle{bst/plainurl_uf}
%\bibliographystyle{bst/Science_Web}
%\bibliographystyle{bst/uf_econ}
%\bibliographystyle{bst/uffull}
%\bibliographystyle{bst/ufinit}
%\bibliographystyle{bst/unsrtnat}
%\bibliographystyle{bst/unsrturl_uf}
%\bibliographystyle{bst/plain}
%\bibliographystyle{bst/ufinit}
%\bibliographystyle{bst/plainurl_uf}
\bibliographystyle{bst/splncs04}

%-----------------------------------------------------------------------%

\maketitle % % % % Creates the Title page from the information entered in userinfo.tex
\makecopyright

%------------------------------------------%

\dedication{\input{tex/dedication}} % %Creates the dedication - if your dedication is more than a single line
% % % % % % % % % % % % % % % % % %you will need to reduce the vspace amount to keep the text centered verticlly
% % % % % % % % % % % % % % % % % %optional - comment or delete if you are not dedicating to anyone,

%------------------------------------------%

\include{tex/acknowledgements} % % % %Required - There is no requirement to acknowledge a particular person
% % % % % % % % % % % % % % % % %but you must acknowledge someone (funding source, committee chair, spouse)?

%------------------------------------------%

% This file includes the file which creates the table of contents %
\include{tex/TOC} %This file creates the Table of Contents, List of Figures, and List of Objects (if any)
% % % % % % % %delete or comment the file you want to remove

%------------------------------------------%

%%This is an optional file. A list of abbreviations is NOT even suggested.
%%Best practice is to define the item the first time it is used in the document

%\include{tex/abbreviations}

%------------------------------------------%
% This line adds the word CHAPTER to the TOC just before the listing of the chapter and subsections begins
\addtocontents{toc}{\protect\addvspace{10pt}\noindent{CHAPTER}\protect\hfill\par}{}% This extra line adds the word CHAPTER to the table of contents %
\phantomsection
\include{tex/abstract} %The abstract is created using this file and userinfo.tex
% % % % % % % % % % %If you have a c-chair you must uncomment that line in userinfo.tex AND find the
% % % % % % % % % % %co-chair lines in ufthesis.cls and un-comment those as well

%-----------------------------------------------------------------------%

% This section encompasses the main body of the paper from all the content through to the biographical sketch

% Chapters to be included (more can be added by creating a new chapter#.tex %
% file and then implementing the \inlcude{chapter#.tex} command as seen below %
\include{tex/chapter1}

\include{tex/chapter2}
\include{tex/chapter3}

\include{tex/chapter4}
\include{tex/chapter5}
%\include{tex/chapter6}  %These chapters are not included in the template
%\include{tex/chapter7}

%-----------------------------------------------------------------------%

% Use the appropriate file depending upon the number of appendices you have

\include{tex/TwoOrMoreAppendices} %Use this file if you have two or more appendices

% \include{tex/OneSingleAppendix} %Use this file if you have one and only one appendix

%------------------------------------------%

% Make List of References (BibTeX implemented using the Natbib package)
% un-comment your preferred bibliography style and replace the
% bibliography file "sample" with the name of your .bib file
% REMEMBER!!! If you want un-numbered references comment the Natbib package with
% The numbered options in the packages.tex file and un-comment the package with the authoryear option
% See the included pdfs of the various styles to see the differences.
% The citation style differences are from the \citet{key} and \citep{key} commands
% More options are available; see the Natbib documentation for details

%\interlinepenalty=10000

\bibliography {bib/dissertation_refs}
% You can have more than one library of references - put the .bib file
% in the bib folder and call it here
%------------------------------------------%

% Bio Sketch is required and should be in third person, past tense%
\include{bio}

%------------------------------------------%

\end{document}

%% file: packages.tex
\usepackage{graphicx}
\usepackage{algorithmic}
\usepackage[ruled,linesnumbered]{algorithm2e}
\usepackage{amsmath}
\usepackage{amsthm}
\usepackage{tabularx}
\usepackage{url}
\usepackage[letterpaper,hmargin=1in,vmargin=1in]{geometry}
\usepackage{lscape}
\usepackage{hanging}
\usepackage{longtable}
\usepackage{amsfonts}
\usepackage{amssymb}
\usepackage[cmbright]{sfmath} % Comment this line to use Times New Roman Math Typeface
\usepackage{subfigure}
\usepackage{rotating}
\usepackage{calc}
\usepackage{setspace}
\usepackage{ufenumerate}
\usepackage{latexsym}
\usepackage{epsf}
\usepackage{epsfig}
\usepackage{euscript}
\usepackage[format=hang,justification=raggedright,singlelinecheck=0,labelsep=period]{caption}

\usepackage{textcomp}
\usepackage{url}            % simple URL typesetting
\usepackage{booktabs}       % professional-quality tables
\usepackage{nicefrac}
\usepackage{microtype}      % microtypography
\usepackage{changes}
\usepackage{lipsum}
\usepackage{enumitem}
\usepackage{listings}
\usepackage{epstopdf}
\usepackage{array}
\usepackage{wrapfig}
\usepackage{mathtools,cuted}
\usepackage{color}
\usepackage{colortbl}
\usepackage{cases}
\usepackage{float}
\DeclarePairedDelimiter\abs{\lvert}{\rvert}
\DeclarePairedDelimiter\norm{\lVert}{\rVert}
\usepackage{stackengine,graphicx}
\stackMath

\usepackage{amssymb}
\usepackage[normalem]{ulem}
\usepackage{bm}
\usepackage{mathtools}
\usepackage{mathrsfs}
\usepackage{verbatim}
\usepackage{comment}

\newcommand{\xbf}{\mathbf{x}}
\newcommand{\ybf}{\mathbf{y}}
\newcommand{\rbf}{r}
\newcommand{\bbf}{\mathbf{b}}
\newcommand{\nbf}{\mathbf{n}}
\newcommand{\ubf}{\mathbf{u}}
\newcommand{\Ubf}{\mathbf{U}}
\newcommand{\vbf}{\mathbf{v}}
\newcommand{\wbf}{\mathbf{w}}
\newcommand{\zbf}{\mathbf{z}}
\newcommand{\cbf}{\mathbf{c}}
\newcommand{\fbf}{\mathbf{f}}
\newcommand{\sbf}{\mathbf{s}}

\newcommand{\Fbf}{\mathbf{F}}
\newcommand{\Pbf}{\mathbf{P}}
\newcommand{\gbf}{\mathbf{g}}
\newcommand{\Abf}{\mathbf{A}}

\newcommand{\Xbf}{\mathbf{X}}

\newcommand{\Sbf}{\mathbf{S}}
\DeclareMathOperator{\prox}{\mathrm{prox}}
\newcommand{\etol}{\epsilon_{\mathrm{tol}}}

\newcommand{\R}{\mathbb{R}}
\newcommand{\NN}{\mathbb{N}}

\newcommand{\C}{\mathbb{C}}
\DeclareMathOperator*{\argmin}{arg\,min}
\newcommand{\F}{\mathcal{F}}
\newcommand{\J}{\mathcal{J}}
\newcommand{\K}{\mathcal{K}}

\newcommand{\T}{\mathcal{T}}
\newcommand{\N}{\mathcal{N}}

\newcommand{\soft}{\mathcal{S}}

\newcommand{\M}{\mathcal{M}}

\newcommand{\D}{\mathcal{D}}
\newcommand{\G}{\mathcal{G}}

\newcommand{\xt}{\mathbf{x}_{t}}

\newcommand{\xtp}{\mathbf{x}_{t+1}}
\newcommand{\xpp}{\mathbf{x}_{p+1}}
\newcommand{\ztp}{\mathbf{z}_{t+1}}
\newcommand{\utp}{\mathbf{u}_{t+1}}
\newcommand{\vtp}{\mathbf{v}_{t+1}}

\newcommand{\Ccal}{\mathcal{C}}
\newcommand{\Scal}{\mathcal{S}}
\newcommand{\Xcal}{\mathcal{X}}
\newcommand{\Ycal}{\mathcal{Y}}
\newcommand{\gi}{\gbf_i}
\newcommand{\tl}{t_{l}}
\newcommand{\tlp}{t_{l+1}}
\newcommand{\epst}{\varepsilon_{t}}
\newcommand{\epsp}{\varepsilon_{p}}
\newcommand{\epstp}{\varepsilon_{t+1}}

\newcommand{\taut}{\tau_t}

\newcommand{\gj}{\gbf_j}
\newcommand{\xhat}{\hat{\mathbf{x}}}
\usepackage{tikz}
\usetikzlibrary{shapes,arrows,chains}

\usepackage[hyperfootnotes=false]{hyperref}
\hypersetup{colorlinks=true,linkcolor=blue,anchorcolor=blue,citecolor=blue,filecolor=blue,urlcolor=blue,bookmarksnumbered=true,pdfview=FitB} %

%% file: usersetcommands.tex
\hyphenchar\font=-1

\newlength{\widthOfItem}

\renewenvironment{itemize}{%
	\begin{Itemize}
		\setlength{\itemsep}{0.5\baselineskip}
		\setlength{\labelwidth}{2em}
		\setlength{\listparindent}{.32in}%
		\setlength{\leftmargin}{.32in}
		\setlength{\rightmargin}{0in}
		\settowidth{\widthOfItem}{\labelitemi}
		\setlength{\labelsep}{\leftmargin-\widthOfItem}
		
		\singlespacing}{%
	\end{Itemize}}

\newtheorem{theorem}{Theorem}[chapter]%To link the theorem to each chapter uncomment the chapter option
\newtheorem{lemma}{Lemma}%[theorem]% To link each lemma to a theorem uncomment the theorem option

%% file: tex/dedication.tex
% Add your text for the dedication here between the center tags
\addvspace{4.25in}
\begin{center}\singlespacing
I dedicate this dissertation to my family and friends.\\
\end{center}

%% file: tex/acknowledgements.tex
% Make sure to keep the text within the brackets and the output should turn out correct
\acknowledge{Thank you to my advisor, Yunmei Chen, for all her help and motivation these past five years. Thank you to my friends and family.}

%% file: tex/TOC.tex
% This creates your table of contents, list of figures, and list of tables
% the pdfbookmark line adds the word to the bookmarks of the pdf without adding it to the TOC itself
\pdfbookmark[0]{TABLE OF CONTENTS}{tableofcontents}
\tableofcontents %
\listoftables %
\listoffigures %

% Produced list of abbreviations or symbols %
%\printindex[keylist]{KEY TO ABBREVIATIONS}{KEY TO ABBREVIATIONS}{}
%\printindex[mathlist]{KEY TO SYMBOLS}{KEY TO SYMBOLS}{%
%The list shown below gives a brief description of the major mathematical symbols defined in this work. For each
%symbol, the page number corresponds to the place where the symbol is first used.} %

%% file: tex/abstract.tex
% Write in only the text of your abstract, all the extra heading jargon is automatically taken care of
\begin{abstract}
This dissertation is devoted to provide advanced nonconvex nonsmooth variational models of (Magnetic Resonance Image) MRI reconstruction, efficient learnable image reconstruction algorithms and parameter training algorithms that improve the accuracy and robustness of the optimization-based deep learning methods for compressed sensing MRI reconstruction and synthesis.

The first part introduces a novel optimization based deep neural network whose architecture is inspired by proximal gradient descent for solving a variational model of the pMRI reconstruction problem without knowledge of coil sensitivity maps. The regularization function consists of the nonlinear combination operator and sparse feature encoder.

The second part is a substantial extension of the preliminary work in the first part by solving the calibration-free fast pMRI reconstruction problem in a discrete-time optimal control framework. The network architecture is determined by the discrete-time dynamic system, which is induced by a designated variational model. The regularization in the variation model contains two parts: one enhances the sparsity of the reconstructed image in images domain; the other is in the k-space domain to further remove high-frequency artifacts.

The third part aims at developing a generalizable Magnetic Resonance Imaging (MRI) reconstruction method in the meta-learning framework. Specifically, we developed a deep reconstruction network induced by a learnable optimization algorithm (LOA) to solve the nonconvex nonsmooth variational model of MRI image reconstruction. We partition these network parameters into two parts: a task-invariant part for the common feature encoder component of the regularization, and a task-specific part to account for the variations in the heterogeneous training and testing data.

The last part aims to synthesize target modality of MRI by using partially scanned k-space data from source modalities instead of fully scanned data that is used in the state-of-the-art multimodal synthesis. We propose to learn three modality-specific feature extraction operators, one for each of these three modalities. Then, we design regularizers of these images by combining these learned operators and a robust sparse feature selection operator. To synthesize the target modality image using source modalities, we employ another feature-fusion operator which learns the mapping from the features that are generated from source modalities to the target modality.

\end{abstract}

%% file: tex/chapter1.tex
\chapter{INTRODUCTION} \label{introduction}
MRI is one of the most prominent non-invasive and non-ionizing medical imaging technologies to generate precise in-vivo tissue images for disease diagnosis and medical analysis. A major interest for the MRI imaging community is to accelerate data acquisition speed during MRI scanning. The typical scanning time for one sequence of MR images often requires at least 30 minutes depending on the part of the body getting scanned, which is much longer than most other imaging modalities. However, infants, elderly people and patients who have serious decease and cannot control body movement may not stay tranquil during the long time scanning. A prolonged scan process will cause patients discomfort or introduce motion artifacts into MR images and degrade diagnostic accessibility. Therefore, speeding up MRI scanning is essential and significant for improving the MR image quality.
In addition, MRI provides different contrast images from the same anatomy which enrich the anatomical information for both clinical diagnostic and research studies. Different contrast MRI images have similar anatomical structure but highlight different soft tissue and they are scanned in different sequences, so the acquisition time could be prolonged in order to obtain additional MRI sequences. 

MRI data are acquired from the frequency domain (k-space) and the acquisition time is roughly proportional to the k-space phase encoding steps. The common approach to speed up MRI scans is to reduce the number of k-space measurement by skipping phase encoding lines and only acquiring partial k-space data, which reduces the number of sampled signals during data acquisition. however, it violates Nyquist criterion \cite{nyquist1928certain} and causes aliasing artifacts because of the under-sampling.

MRI reconstruction is a process to  recover clear MR images from the undersampled partial k-space data that can be applied for diagnostic and clinical applications.  Compressed sensing (CS) \cite{donoho2006compressed} MRI reconstruction and MRI parallel imaging \cite{sodickson1997simultaneous, pruessmann1999sense,larkman2007parallel} are successful methods that solve for an inverse problem and expedite MRI scans and eliminate artifacts.  

This chapter first introduces the background of fast MRI undersampling and reconstruction; then gives an overview of classical reconstruction methods and deep-learning-based methods for both CS-MRI and Parallel imaging, and lastly provides several well-known methods that proposed during past decades.

%importance of sparsity,importance of different contrast, challenges,

\section{Optimization Models for Compressive Sensing in MRI}

Consider the MR image $ \xbf \in \C^{ \sqrt{N} \times \sqrt{N}} $ is a two-dimensional (2D) array, where there are $N$ pixel numbers in the image $\xbf$. It is common and convenient to write $\xbf$ as a column vector with dimension $N$ in the numerical algorithms, and $ \xbf_i$ is the $i$-th intensity value of image $\xbf$ for $i \in \{1, \cdots, N \}$. 
The under-sampled k-space measurement are related to the image by the following formula \cite{haldar2010compressed}:
\begin{equation}\label{intro_MRI}
    \ybf = \Pbf \F  \xbf + \nbf,
\end{equation}
where $\ybf \in \C^p$ is the measurements in k-space with total of $p$ sampled data points, $\xbf \in \C^{N\times 1}$ is to be reconstructed MR image with $N$ pixels, $ \F \in \C^{N \times N}$ is the 2D discrete Fourier transform (DFT) matrix, and $ \Pbf \in \R^{p \times N}$ $(p< N)$ is the binary matrix representing the sampling trajectory in k-space. $\mathbf{n}$ represents the data acquisition noise in k-space. The goal of MRI reconstruction is to solve for $\xbf$ from equation \eqref{intro_MRI} given the partial k-space data $\ybf$. 

Consider that $\ybf$ does not contain noises so $\nbf = 0$ and we can get $ \Pbf \F \xbf = \ybf$, where the operation $\Pbf \F \in \C^{p \times N}$ with $p << N$ is under-determined system which results in infinitely many complex solutions.
The straightforward solution is the zero-filled reconstruction method formulated as $ \xbf = \F^{-1} \Pbf^{-1} \ybf$, which is essentially fill the uncollected k-space data with zeroes to obtain a pseudo full k-space data $ \Pbf^{-1} \ybf$ and then apply inverse Fourier transform $\F^{-1}$.  However, this solution results in truncation effects which in the form of \emph{Gibbs Ringing} artifacts and blurring \cite{wood1985truncation}.

For multi-coil acquisitions in parallel MRI, SENSE \cite{pruessmann1999sense} reconstruct pMRI in image domain using the pre-calculated diagonal matrix called coil sensitivity map $ \Sbf_j \in \R^{N \times N}$ of the $j$th coil, which is either given or estimated in advance. $ \Pbf \F (\Sbf_j \cdot \xbf) \in \C^p$ is the vector of undersampled Fourier coefficients and k-space data acquisition at $j$-th receiver coil is expressed as $\ybf_j = \Pbf \F (\Sbf_j \cdot \xbf) + \nbf_j$ for $j = 1,\cdots, N_c $,  where $\cdot$ denotes element-wise multiplication, $ \nbf_j$ represents the measurement noise in k-space at the $j$-th receiver coil.

Due to the ill-posedness of the inverse problem \eqref{intro_MRI}, design a proper regularization by incorporating prior information of the image to be reconstructed is very necessary. Compressive sensing (CS) theory implies that when the signal $\xbf$ is sparse or can be sparsely represented, significantly less measurement is enough to reconstruct an image with well-preserved quality. Given the raw data measurements $\ybf$, the general CS-based model for reconstructing original single-coil MR image $\xbf$  can be written as:
\begin{equation}\label{eq:single_MRI}
    \min_{\xbf} \| \Pbf \F \xbf - \ybf\|^2 +  \mu R(\xbf),
\end{equation}
and SENSE-based model for pMRI reconstruction  can be written as:
\begin{equation}\label{eq:PFS}
    \min_{\xbf} \ \sum^{N_c}_{j=1} \frac{1}{2} \| \Pbf \Fbf (\sbf_j \cdot \xbf)- \fbf_j\|^2 + \mu R(\xbf),
\end{equation}
where the regularizer $R$ provides prior information of reconstructed image and $\mu > 0$ is a weight parameter that balance data fidelity term and regularization term. 
In past decades, total variation (TV) is widely recognized as a successful candidate of regularizer, which is formulated as $ TV(\xbf) = \sum^N_{i=1} \| D_i \xbf \|$ where $ D_i \in \R^{2\times N}$ is binary with entries $-1$ and $1$ corresponding to forward finite differences to partial derivatives along the first and second coordinates. Since TV is $l_1$ norm of the gradient $\| D_i \xbf \|$, it is a common choice to be used as sparsifying transform in CS method \cite{rudin1992nonlinear}.

\section{Learnable Optimization Algorithm for MRI reconstruction}

%1. method don't need regularization, learn proximal point directly. data-determined 2. do have a math formula of regularization. ISTA ADMM but they just imitate their algorithm. Existing problem they are not interpretable. Motivation of develop our algorithm.

Successful conventional methods for single-coil MRI reconstruction including but not limit to the following methods: 1) TV based models \cite{lustig2007sparse, lysaker2003noise, wu2010augmented, zhang2010bregmanized, valkonen2014primal};
Several popular algorithms for solving TV regularized image reconstruction problems including the Alternating
Direction Method of Multipliers (ADMM) \cite{boyd2011distributed, goldstein2009split, wu2010augmented}, the Primal-Dual Hybrid Gradient (PDHG) algorithm \cite{chambolle2011first, esser2010general, zhu2008efficient} and the iterative shrinkage-thresholding algorithm (ISTA) \cite{beck2009fast}. 2) Wavelet-based models \cite{guerquin2011fast, chen2014exploiting} which utilizes sparsity assumptions in the wavelet domain. 3) Low-rank based models \cite{lyra2012improved, dong2014compressive} often used for dynamic MRI. 4) Patch or non-local based models \cite{yang2012nonlocal, trzasko2011local, eksioglu2016decoupled} consider the sparse representation by capturing geometric self-similarity between patches. 5) Dictionary learning based models \cite{ravishankar2010mr, zhan2015fast, huang2014bayesian} that take advantage of dictionary learning to further promote sparsity of the signal, etc. 6) Parallel Imaging contains two classes: k-spcae methods that use coil-by-coil auto-calibration such as GRAPPA \cite{griswold2002generalized}, SPIRiT \cite{lustig2010spirit}  and image domain methods such as SENSE \cite{pruessmann1999sense} which solves for an optimization problem that requires perfect coil sensitivity maps knowledge.

Traditional methods employ handcrafted regularization terms, and the solution algorithm follows a theoretical justification. However, these regularization terms are excessively simplified, for example, TV regularized model tend to reconstruct images that are "piecewisely constant" where the fine structure details can be smeared due to its promotion of enforcing sparse gradient.  
In addition, it is hard to tune the associated parameter  and often require hundreds even thousands of iterations to converge to capture subtle details and satisfy clinic diagnostic quality \cite{chen2021variational}.

Deep learning based model leverages large dataset and further explore the potential improvement of reconstruction performance comparing to traditional methods and has successful applications in clinic field. However, training generic deep neural networks (DNNs) may prone to over-fitting when data is scarce. Also, the deep network structure behaves like a black box without mathematical interpretation. To improve the interpretability of the relation between the topology of the deep model and reconstruction results, a new emerging class of deep learning-based methods known as \emph{learnable optimization algorithms} (LOA) have attracted much attention e.g. \cite{lundervold2019overview, liang2020deep, sandino2020compressed, mccann2017convolutional, zhou2020review, singha2021deep, chandra2021deep, ahishakiye2021survey,liu2020deep, liang2020deep}. LOA was proposed to map existing optimization algorithms to structured networks where each phase of the networks correspond to one iteration of an optimization algorithm.

For example, proximal point network \cite{bian2020deep,bian2022optimal} which does not require specified form of regularization and iterates the following two steps:
\begin{subequations}
\begin{align}
    r_k & = \xbf_{k-1} - \alpha_k \nabla f(\xbf_{k-1}),\\
    \xbf_{k} & = \prox_{\alpha_k R} (r_k) ,
\end{align}
\end{subequations}
where $\alpha >0$ is the step size, and function $f$ is the data fidelity in \eqref{intro_MRI}. The proximity operator $\prox_{\alpha_k R} $ can be parametrized as a deep learnable denoiser which could be replaced as a denoising network \cite{zhang2017learning}. Chapters \ref{pMRI} and \ref{optimalcontrol} present more detail and some variations of proximal point network.
In ADMM-Net \cite{NIPS2016_6406}, every epoch of the deep reconstruction network architecture mimics one iteration of the ADMM algorithm, where ADMM-Net replaces the discrete gradient operator $D$ with a convolution operator which is a linear combination of a set of given filters. 

CS theory for image reconstruction problems often require a sparsifying transform $\Psi \in \R^{n \times n}$ (eg: wavelet transform) in the regularization since natural image are usually not sparse,  a single $\ell_1$ is not enough. Therefore we can obtain a different reconstruction model: 
\begin{equation}
    \min_{\xbf} f(\xbf) +  \mu R( \Psi \xbf),
\end{equation}
where we write $f$ as the data fidelity term and $\Psi \xbf$ is a sparse vector. 
ISTA-NET$^+$ \cite{zhang2018ista} consider solving $ \ubf = \Psi \xbf$ and reconstruct $\xbf = \Psi^{\top} \ubf $. The transformers $\Psi$ and $\Psi^{\top}$ are replaced by multilayer CNNs $H^{(k)}$ and $\tilde{H}^{(k)}$ respectively and then iterate the following  scheme:
\begin{subequations}
\begin{align}
    r_k & = \xbf_{k-1} - \alpha_k \nabla f(\xbf_{k-1}),\\
    \xbf_{k} & = \tilde{H}^{(k)} \prox_{\alpha_k R} ( H^{(k)} r_k)  = \tilde{H}^{(k)} S_{\theta_k} ( H^{(k)} r_k),
\end{align}
\end{subequations}
where  $S_{\theta_k}$ represents the shrinkage operator since $R(\ubf) = \mu \| \ubf\|_1$ and $\theta_k = \alpha_k \mu $. The network $H^{(k)}$ and $\tilde{H}^{(k)}$ are parametrized in a symmetric structure but their parameters are learned separately. The network architecture of iterative algorithms that proposed in Chapters \ref{pMRI} and \ref{optimalcontrol} are inspired from ISTA-NET$^+$.
Primal-dual networks (PD-Nets) \cite{adler2018learned,cheng2019model,heide2014flexisp,meinhardt2017learning} are deep networks where each consecutive phase mimics the corresponding iteration of the Primal-Dual Hybrid Gradient algorithm \cite{chambolle2011first}.  PD-net \cite{cheng2019model} solves CS-based MRI reconstruction problem where the primal and dual proximal operators are replaced by multilayer learnable CNN denoisers.

The above mentioned LOAs are only specious imitations of the iterative algorithms and hence lack the backbone of the variational model and any convergence guarantee. For instance, proximal operator \cite{cheng2019model, bian2020deep, zhang2018ista}, matrix transformations \cite{yang2018admm, hammernik2018learning, zhang2018ista}, non-linear operators \cite{yang2018admm, hammernik2018learning}, and denoiser/regularizer \cite{aggarwal2018modl, schlemper2017deep} etc., by CNNs to avoid difficulty for solving non-smooth non-convex problems. These methods simply embeds networks into the unrolling scheme of optimization iterations therefore their network structures is lack of convergence analysis and interpretability, also is prone to overfitting or underfitting problems.

The networks for image reconstruction proposed in Chapters \ref{meta_learning} and \ref{JointRecSyn} are trying to addressed this issue by conducting a convergence guaranteed learnable optimization algorithm (LOA) where the network structure exactly follows the algorithm. Since the minimization problem in \eqref{eq:single_MRI} is nonsmooth and non convex, we first smooth the $\ell_{2,1}$ norm in the regularizer, then design algorithmic unrolling methods with provable convergence. In this dissertation, we consider LOA as the forward image reconstruction optimization process, and we also analyzed backward parameter optimization problem in the training process for updating the learnable parameters. Minimizing the loss function to get iterative training algorithm is a principal problem in machine learning, and bilevel optimization algorithms for parameter training are major considerations in this dissertation.

%For instance, gradient decent algorithm based CNN \cite{hammernik2018learning}, proximal gradient inspired network \cite{cheng2019model, bian2020deep, zhang2018ista}, ADMM inspired \cite{yang2018admm}, Primal dual algorithm inspired \cite{adler2018learned, cheng2019model, heide2014flexisp,meinhardt2017learning}.

\section{Bilevel Optimization Algorithm for Parameter Training}

Bilevel optimization can be formulated in the following form:
\begin{subequations}
\begin{align}
    & \min_{u} \phi(v^*(u), u) \label{introupper}\\
    \text{subject to } &  v^*(u) = \argmin_v \varphi(v,u) , \label{introlower}
\end{align}    
\end{subequations}
where \eqref{introupper} is the \emph{upper/outer-level} optimization problem and the constrained problem \eqref{introlower} is called \emph{lower/inner-level} optimization problem. The solution of lower-level is an argument of upper-level.

In conventional supervised machine learning, the training data is given as a set of pairs of inputs and outputs  $ \D = \{ (\ybf_m, \xbf^*_m )\}_{m=1}^{\M}$, and we want to train the prediction model $ \phi_{\theta}$ which is parameterized by $ \theta \in \Theta$ and maps inputs $ \ybf_m \in \Ycal $ to their corresponding outputs $\xbf^*_m \in \Xcal $. The parameters $\theta$ is learned by minimizing a predefined loss function $ L : \Xcal \times \Xcal \to\R $, which measures the discrepancy between the output $ F_{\theta}(\ybf_m)$ and the reference $\xbf^*_m$, here we denote
\begin{equation}
    \ell( \theta; \D) := \sum_{m = 1}^{\mathcal{M}}L( F_{\theta}(\ybf_m), \xbf^*_m ).
\end{equation}
The goal of the conventional supervised learning is to find optimal parameters $  \hat{\theta}$ such that 
\begin{equation}
     \hat{\theta} = \argmin_{\theta} \ell( \theta; \D).
\end{equation}

In recent machine learning researches, improving the generalization ability of the machine learning model attract extensive attentions. Some supervised learning models including Deep bilevel learning \cite{jenni2018deep} and optimization-based Meta-learning \cite{finn2017model,rajeswaran2019meta,li2017meta,nichol2018reptile,rusu2018meta} are successful strategies that dedicated to model generalization and prevent the estimated model from overfitting. 

In this thesis, we more focused on meta-learning.
In the supervised meta-learning, \emph{meta-knowledge} $ \omega$ is introduced to provide guidance of ``how to learn'' and improve the generalizability to learn new task by learning an algorithm for general tasks. Pre-specified assumptions such as learning-rate, initial parameters and choices of optimizer can vary in meta-learning \cite{huisman2021survey}.  The goal is to find the optimal meta-knowledge  $ \omega$ over a distribution of tasks $ p(\T)$, and $ \omega$ also known as \emph{cross-task} knowledge \cite{hospedales2021meta}. The general setting can be separated as meta-training, meta-validation and meta-testing,  each of these stages associates with disjoint sets of tasks. We denote $ \T_i = ( \D^{tr}_{\tau_i}, \D^{val}_{\tau_i},  \D^{test}_{\tau_i}) , \T_i  \sim p(\T)$, and  meta-training stage uses dataset $ \{ ( \D^{tr}_{\tau_i},  \D^{val}_{\tau_i}) \}^P_{i=1}$ with $P$ tasks where each task has training set and validation set. The meta-testing stage or evaluation stage uses dataset with $Q $ tasks $ \{ ( \D^{tr}_{\tau_i},  \D^{test}_{\tau_i}) \}^Q_{i=1}$, at which point the model is allowed to update parameters for individual task on training set and then the model performance can be evaluated on the testing set after task-specific parameter updating.
The meta-training process can be cast as a bilevel optimization problem:
\begin{subequations}\label{meta-bilevel}
\begin{align}
& \hat{\omega} =   \argmin_{\omega} \sum^{P}_{i=1} \ell_{\tau_i} ( \theta_i(\omega), \omega; \D^{val}_{\tau_i}) \label{outer}\\
 & s.t. \ \ \theta_i(\omega) = \argmin_{\theta} \ell_{\tau_i}(\theta, \omega; \D^{tr}_{\tau_i}), \label{inner}
\end{align}
\end{subequations}
where $\ell_{\tau_i}$ is the task-aware loss function associated with task $\tau_i$. \eqref{inner} is inner/lower/base level optimization that solves for specific parameter for each individual task from its corresponding training set with the task-specific parameter $\omega$ fixed. \eqref{outer} is outer/upper/meta level optimization that learns 
hyper-parameter $\omega$ by minimizing the validation loss.

In recent years, meta-learning methods have demonstrated promising results in various fields with different techniques \cite{hospedales2021meta}. 
Meta-learning techniques can be categorized into three groups  \cite{yao2020automated, lee2018gradient, huisman2021survey}: metric-based methods \cite{koch2015siamese, vinyals2016matching, snell2017prototypical}, model-based methods  \cite{mishra2017simple, ravi2016optimization, qiao2018few, graves2014neural}, and optimization-based methods \cite{finn2017model, rajeswaran2019meta, li2017meta}. Optimization-based methods are often cast as a bilevel optimization problem and exhibit relatively better generalizability for wider task distributions. We mainly focus on optimization-based meta-learning in this paper.

For meta-learning applications, the lower-level problem encounters new tasks and tries to learn the associated features quickly from the training observations, the outer level accumulates task-specific meta-knowledge across previous tasks and the meta-learner provides support for the inner level so that it can quickly adapt to new tasks. e.g. 
The lower-level problem is approximated by one or a few gradient descent steps in many existing optimization-based meta learning applications, such as Model-Agnostic Meta-Learning (MAML) \cite{finn2017model}, and a large number of followup works of MAML proposed to improve  generalization using similar strategy \cite{lee2018gradient, rusu2018meta, finn2018probabilistic, grant2018recasting, nichol2018first, vuorio2019multimodal, yao2019hierarchically,yin2020metalearning}.
Some variant models \cite{rusu2018meta, vuorio2019multimodal, yao2019hierarchically} focus on multiple initial conditions for fast learning which usually relies on fixed optimizers such as SGD with momentum or its variance.
Deep bilevel learning~\cite{jenni2018deep} seeks to obtain better generalization than when trained on one task and generalize well to another task. The model is used to optimize a regularized loss function to find network parameters from the training set and identify hyperparameters so that the network performs well on the validation dataset.

Alternatively,  optimizer oriented methods \cite{andrychowicz2016learning, ravi2016optimization, li2016learning, wichrowska2017learned} focus on learning inner optimizer, the meta-knowledge $\omega$ can be used to define gradient-based optimization steps  for each base learning iteration.
Hyperparameter optimization (HO) shares the same merit with meta-learning, the major difference is HO often considers a single task that split as train data  and validation data, the inner objective is the regularized empirical loss function on train data that seeks to tune the model parameters and the outer objective seeks to tune hyperparameters \cite{franceschi2018bilevel,pedregosa2016hyperparameter,franceschi2017forward,micaelli2020non}. 

\section{Outline of The Dissertation}

The following chapters are organized as follows: \\
Chapter \ref{pMRI} introduces a novel deep neural network  architecture by mapping the robust proximal gradient scheme for fast image reconstruction in parallel MRI (pMRI) with regularization function trained from data. The proposed network does not require knowledge of sensitivity maps, it learns to adaptively combine the multi-coil images from incomplete pMRI data into a single image with homogeneous contrast, which is then passed to a nonlinear encoder to efficiently extract sparse features of the image. 

Chapter \ref{optimalcontrol} aims at developing a novel calibration-free fast parallel MRI (pMRI) reconstruction method incorporate with discrete-time optimal control framework. The reconstruction model is designed to learn a regularization that combines channels and extracts features by leveraging the information sharing among channels of multi-coil images. 
The reconstruction network is cast as a structured discrete-time optimal control system, resulting in an optimal control formulation of parameter training where the parameters of the objective function play the role of control variables. We demonstrate that the Lagrangian method for solving the control problem is equivalent to back-propagation, ensuring the local convergence of the training algorithm.

Chapter \ref{meta_learning} aims at developing a generalizable MRI reconstruction model in the meta-learning framework.
The standard benchmarks in meta-learning are challenged by learning on diverse task distributions. The proposed network learns the regularization function in a variational model and reconstructs MR images with various under-sampling ratios or patterns that may or may not be seen in the training data by leveraging a  heterogeneous dataset.

Generating multi-contrasts/modal MRI of the same anatomy enriches diagnostic information but is limited in practice due to excessive data acquisition time. Chapter \ref{JointRecSyn} proposes a novel deep-learning model for joint reconstruction and synthesis of multi-modal MRI using incomplete k-space data of several source modalities as inputs. The output of our model includes reconstructed images of the source modalities and high-quality image synthesized in the target modality. 
Our proposed model is formulated as a variational problem that leverages several learnable modality-specific feature extractors and a multimodal synthesis module. We propose a learnable optimization algorithm to solve this model, which induces a multi-phase network whose parameters can be trained using multi-modal MRI data. Moreover, a bilevel-optimization framework is employed for robust parameter training. We demonstrate the effectiveness of our approach using extensive numerical experiments.

%% file: tex/chapter2.tex
\chapter{A Optimization based Deep Parallel MRI Reconstruction Network Without Coil Sensitivities}\label{pMRI}

\section{Introduction}
\label{sec:Introduction}

In this chapter, we propose a novel deep neural network architecture by mapping the robust proximal gradient scheme for fast image reconstruction in parallel MRI (pMRI) with regularization function trained from data. The proposed network learns to adaptively combine the multi-coil images from incomplete pMRI data into a single image with homogeneous contrast, which is then passed to a nonlinear encoder to efficiently extract sparse features of the image. Unlike most of existing deep image reconstruction networks, our network does not require knowledge of sensitivity maps, which can be difficult to estimate accurately, and have been a major bottleneck of image reconstruction in real-world pMRI applications. The experimental results demonstrate the promising performance of our method on a variety of pMRI imaging data sets.

Parallel magnetic resonance imaging (pMRI) is a state-of-the-art medical MR imaging technology which surround the scanned objects by multiple receiver coils and collect k-space (Fourier) data in parallel. To accelerate scan process, partial data acquisitions that increase the spacing between read-out lines in k-space are implemented in pMRI. However, reduction in k-space data sampling arising aliasing artifacts in images, which must be removed by image reconstruction process.
There are two major approaches to image reconstruction in pMRI: the first approach are k-space methods which interpolate the non-sampled k-space data using the sampled ones across multiple receiver coils \cite{doi:10.1002/jmri.23639}, such as the generalized auto-calibrating partially parallel acquisition (GRAPPA) \cite{griswold2002generalized}. The other approach is the class of image space methods which remove the aliasing artifacts in the image domain by solving a system of equations that relate the image to be reconstructed and partial k-spaced data through coil sensitivities, such as in SENSitivity Encoding (SENSE) \cite{pruessmann1999sense}.

In this paper, we propose a new deep learning based reconstruction method to address several critical issues of pMRI reconstruction in image space. Consider a pMRI system with $N_c$ receiver coils acquiring 2D MR images at resolution $m\times n$ (we treat a 2D image $\vbf \in \mathbb{C}^{m\times n}$ and its column vector form $\vbf \in \mathbb{C}^{mn}$ interchangeably hereafter). Let $\Pbf \in \mathbb{R}^{p\times mn}$ be the binary matrix representing the undersampling mask with $p$ sample locations in k-space, and $\sbf_i\in\mathbb{C}^{mn}$ the coil sensitivity and $\fbf_i \in \mathbb{C}^{p}$ the \emph{partial} k-space data at the $i$th receiver coil for $i=1,\dots,N_c$. Therefore $\fbf_i$ and the image $\vbf$ are related by $\fbf_i = \Pbf \Fbf (\sbf_i \cdot \vbf) + \mathbf{n}_i$ where $\cdot$ denotes pointwise multiplication of two matrices, and $\mathbf{n}_i$ is the unknown acquisition noise in k-space at each receiver coil. Then SENSE-based image space reconstruction methods can be generally formulated as an optimization problem:
\begin{equation}\label{eq:PFS_chp2}
    \min_{\vbf} \ \sum^{N_c}_{i=1} \frac{1}{2} \| \Pbf \Fbf (\sbf_i \cdot \vbf)- \fbf_i\|^2 + R(\vbf),
\end{equation}
where $\vbf\in \mathbb{C}^{m n}$ is the MR image to be reconstructed, $\Fbf \in \mathbb{C}^{mn\times mn}$ stands for the discrete Fourier transform, and $R(\vbf)$ is the regularization on the image $\vbf$.
 $\| \xbf \|^2 := \|\xbf \|_2^2 = \sum_{j=1}^n |x_j|^2$ for any complex vector $\xbf = (x_1,\dots,x_n)^{\top} \in \mathbb{C}^n$.
There are two critical issues in pMRI image reconstruction using \eqref{eq:PFS_chp2}: availability of accurate coil sensitivities $\{\sbf_i\}$ and proper image regularization $R$.
Most existing SENSE-based reconstruction methods assume coil sensitivity maps are given, which are however difficult to estimate accurately in real-world applications. 
On the other hand, the regularization $R$ is of paramount importance to the inverse problem \eqref{eq:PFS_chp2} to produce desired images from significantly undersampled data, but a large number of existing methods employ handcrafted regularization which are incapable to extract complex features from images effectively.

In this paper, we tackle the two aforementioned issues in an unified deep-learning framework dubbed as pMRI-Net.
Specifically, we consider the reconstruction of multi-coil images $\ubf = (\ubf_1,\dots,\ubf_{N_c}) \in \mathbb{C}^{mnN_c}$ for all receiver coils to avoid use of coil sensitivity maps (but can recover them as a byproduct), and design a deep residual network  which can jointly learn the adaptive combination of multi-coil images and an effective regularization from training data.

The contribution of this paper could be summarized as follows:
Our method is the first ``combine-then-regularize" approach for deep-learning based pMRI image reconstruction. The combination operator integrates multichannel images into single channel and this approach performs better than the linear combination the root of sum-of-squares (SOS) method \cite{pruessmann1999sense}.
This approach has three main advantages: (i) the combined image has homogeneous contrast across the FOV, which makes it suitable for feature-based image regularization and less affected by the intensity biases in coil images; (ii) the regularization operators are applied to this single body image in each iteration, and require much fewer network parameters to reduce overfitting and improve robustness; and (iii) our approach naturally avoids the use of sensitivity maps, which has been a thorny issue in image-based pMRI reconstruction.

\section{Related Work}
\label{sec:related}
Most existing deep-learning (DL) based methods rendering end-to-end neural networks mapping from the partial k-space data to the reconstructed images \cite{WANG2020136,7493320,doi:10.1002/mp.12600,Quan2018CompressedSM,8417964}. The common issue with this class of methods is that the DNNs require excessive amount of data to train, and the resulting networks perform similar to ``black-boxes'' which are difficult to interpret and modify.

In recent years, a class of DL based methods improve over the end-to-end training by selecting the scheme of an iterative optimization algorithm and prescribe a phase number $T$, map each iteration of the scheme to one phase of the network. These methods are often known as the learned optimization algorithms (LOAs),  \cite{Aggarwal_2019,cheng2019model,doi:10.1002/mrm.26977,8550778,NIPS2016_6406,zhang2018ista,8067520}.
For instance, ADMM-Net \cite{NIPS2016_6406},  ISTA-Net$^+$ \cite{zhang2018ista}, and cascade network \cite{8067520} are regular MRI reconstruction.
For pMRI: Variational network (VN)\cite{doi:10.1002/mrm.26977} introduced gradient descent method by applying given sensitivities $\{\sbf_i\}$. MoDL \cite{Aggarwal_2019} proposed a recursive network by unrolling the conjugate gradient algorithm using a weight sharing strategy.
 Blind-PMRI-Net \cite{10.1007/978-3-030-32251-9_80}  designed three network blocks to alternately update multi-channel images, sensitivity maps and the reconstructed MR image using an iterative algorithm based on half-quadratic splitting. The network in \cite{10.1007/978-3-030-32248-9_5} developed a Bayesian framework for joint MRI-PET reconstruction. VS-Net \cite{10.1007/978-3-030-32251-9_78} derived a variable splitting optimization method. However, existing methods still face the lack of accurate coil sensitivity maps and proper regularization in the pMRI problem.

Recently, a method called  DeepcomplexMRI \cite{WANG2020136} developed an end-to-end learning without explicitly using coil sensitivity maps to recover channel-wise images, and then combine to a single channel image in testing.

This paper proposes a novel deep neural network architecture which integrating the robust proximal gradient scheme for pMRI reconstruction without knowledge of coil sensitivity maps. Our network learns to adaptively combine the channel-wise image from the incomplete data to assist the reconstruction and learn a nonlinear mapping to efficiently extract sparse features of the image by using a set of training data on the pairs of under-sampled channel-wise k-space data and corresponding images. The roles of the multi-coil image combination operator and sparse feature encoder are clearly defined and jointly learned in each iteration. As a result, our network is more data efficient in training and the reconstruction results are more accurate.

\section{Proposed Method}
\label{sec:proposed}
\subsection{Joint-channel Image Reconstruction pMRI without Coil Sensitivities}
We propose an alternative pMRI reconstruction approach to \eqref{eq:PFS_chp2} by recovering images from individual receiver coils jointly.
Denote $\ubf_i$ the MR image at the $i$th receiver coil, i.e., $\ubf_i = \sbf_i \cdot \vbf$, where the sensitivity $\sbf_i$ and the full FOV image $\vbf$ are both unknown in practice.
Thus, the image $\ubf_i$ relates to the partial k-space data $\fbf_i$ by $\fbf_i = \Pbf \Fbf \ubf_i + \mathbf{n}_i$, and hence the data fidelity term is formulated as least squares $(1/2) \cdot \|\Pbf \Fbf \ubf_i - \fbf_i\|^2$.
We also need a suitable regularization $R$ on the images $\{\ubf_i\}$.
However, these images have vastly different contrasts due to the significant variations in the sensitivity maps at different receiver coils.
Therefore, it is more appropriate to apply regularization to the (unknown) image $\vbf$.

To address the issue of regularization, we propose to first learn a nonlinear operator $\J$ that combines $\{\ubf_i\}$ into the image $\vbf = \J (\ubf_1,\dots,\ubf_{N_c}) \in \mathbb{C}^{m\times n}$ with homogeneous contrast, and apply a regularization on $\vbf$ with a parametric form $\| \G (\vbf) \|_{2,1}$ by leveraging the robust sparse selection property of $\ell_{2,1}$-norm and shrinkage threshold operator. Here $\G(\vbf)$ represents a nonlinear sparse encoder trained from data to effectively extract complex features from the image $\vbf$.
Combined with the data fidelity term above, we propose the following pMRI image reconstruction model:
\begin{equation}\label{eq:m}
\ubf(\fbf; \Theta) = \argmin_{\ubf}\ \frac{1}{2} \sum^{N_c}_{i=1} \| \textbf{PF} \ubf_i - \fbf_i \|^2_2  +  \| \G \circ \J  (\ubf)\|_{2,1},
\end{equation}
where $\ubf = (\ubf_1,\dots,\ubf_{N_c})$ is the multi-channel image to be reconstructed from the pMRI data $\fbf=(\fbf_1,\dots,\fbf_{N_c})$, and $\Theta=(\G,\J)$ represents the parameters of the deep networks $\G$ and $\J$.
The key ingredients of \eqref{eq:m} are the nonlinear combination operator $\J$ and sparse feature encoder $\G$, which we describe in details in Section \ref{subsec:network}.
Given a training data set consisting of  $J$ pairs $\{ (\fbf^{[j]}, \hat{\ubf}^{[j]} ) \,|\, \ 1\le j\le J \}$, where $\fbf^{[j]}=(\fbf^{[j]}_1,\dots,\fbf^{[j]}_{N_c})$ and $\hat{\ubf}^{[j]}=(\hat{\ubf}^{[j]}_1,\dots,\hat{\ubf}^{[j]}_{N_c})$ are respectively the partial k-space data and the ground truth image reconstructed by full k-space data of the $j$th image data, our goal is to learn $\Theta$ (i.e., $\G$ and $\J$) from the following bi-level optimization problem:
\begin{equation}
    \label{eq:bilevel}
    \min_{\Theta} \frac{1}{J} \sum_{j=1}^J \ell(\ubf(\fbf^{[j]};\Theta), \hat{\ubf}^{[j]}),\ \mbox{s.t.}\ \ubf(\fbf^{[j]};\Theta)\ \mbox{solves \eqref{eq:m} with data $\fbf^{[j]}$},
\end{equation}
where $\ell(\ubf,\hat{\ubf})$ measures the discrepancy between the reconstruction $\ubf$ and the ground truth $\hat{\ubf}$.
To tackle the lower-level minimization problem in \eqref{eq:bilevel}, we construct a proximal gradient network with residual learning as an (approximate) solver of \eqref{eq:m}.
Details on the derivation of this network are provided in the next subsection.

\subsection{Proximal Gradient Network with Residual Learning}
\label{subsec:pg}
If the operators $\J$ and $\G$ were given, we can apply proximal gradient descent algorithm to approximate a (local) minimizer of \eqref{eq:m} by iterating
\begin{subequations}\label{eq:bui}
\begin{align}
\bbf_i^{(t)} & =  \ubf_i^{(t)} - \rho_t \Fbf^{\top}  \Pbf^{\top} (\Pbf \Fbf \ubf_i^{(t)} - \fbf_i), \label{eq:bi}  \\
\ubf_i^{(t+1)} &  = [\prox_{\rho_t\|\G\circ \J (\cdot) \|_{2,1}}(\bbf^{(t)})]_i, \quad 1\le i \le N_c \label{eq:ui} \end{align}
\end{subequations}
where $\bbf^{(t)} = (\bbf_1^{(t)},\dots,\bbf_{N_c}^{(t)})$, $[\xbf]_i = \xbf_i \in \mathbb{C}^{mn}$ for any vector $\xbf \in \mathbb{C}^{mn N_c}$, $\rho_t>0$ is the step size, and $\prox_{g}$ is the proximal operator of $g$ defined by
\begin{equation}
    \prox_{g}(\bbf) = \argmin_{\xbf} g(\xbf) + \frac{1}{2} \| \xbf - \bbf \|^2.
\end{equation}
The gradient update step \eqref{eq:bi} is straightforward to compute and fully utilizes the relation between the partial k-space data $\fbf_i$ and the image $\ubf_i$ to be reconstructed as derived from MRI physics.
The proximal update step \eqref{eq:ui}, however, presents several difficulties: the operators $\J$ and $\G$ are unknown and need to be learned from data, and the proximal operator $\prox_{\rho_t\|\G\circ \J (\cdot) \|_{2,1}}$ most likely will not have closed form and can be difficult to compute.
Assuming that we have both $\J$ and $\G$ parametrized by convolutional networks, we adopt a residual learning technique by leveraging the shrinkage operator (as the proximal operator of $\ell_{2,1}$-norm $\|\cdot \|_{2,1}$) and converting \eqref{eq:ui} into an explicit update formula.
To this end, we parametrize the proximal step \eqref{eq:ui} as an implicit residual update:
\begin{equation}\label{eq:ui_update}
    \ubf_i^{(t+1)} = \bbf_i^{(t)} + [\rbf(\ubf_1^{(t+1)}, \cdots, \ubf_{N_c}^{(t+1)})]_i,
\end{equation}
where $\rbf=\tilde{\J} \circ \tilde{\G} \circ \G \circ \J$ is the residual network as the composition of $\J$, $\G$, and their adjoint operators $\tilde{\J}$ and $\tilde{\G}$. These four operators are learned separately to increase the capacity of the network.
To reveal the role of nonlinear shrinkage selection in \eqref{eq:ui_update}, consider the original proximal update \eqref{eq:ui} where
\begin{equation}\label{eq:u_prox}
    \ubf^{(t+1)} = \argmin_{\ubf} \|\G \circ \J (\ubf) \|_{2,1} + \frac{1}{2 \rho_t} \| \ubf - \bbf^{(t)}\|^2.
\end{equation}
For certain convolutional networks $\J$ and $\G$ with rectified linear unit (ReLU) activation, $\| \ubf - \bbf^{(t)}\|^2$ can be approximated by $\alpha \| \G\circ \J(\ubf) - \G\circ \J(\bbf^{(t)})\|^2$ for some $\alpha > 0$ dependent on $\J$ and $\G$ \cite{zhang2018ista}.
Substituting this approximation into \eqref{eq:u_prox}, we obtain that
\begin{equation}\label{eq:GJu}
    \G \circ \J (\ubf^{(t+1)}) = \soft_{\alpha_t} ( \G \circ \J (\bbf^{(t)})),
\end{equation}
where $\alpha_t = \rho_t/\alpha$, $\soft_{\alpha_k}(\xbf) = \prox_{\alpha_k \|\cdot \|_{2,1}} (\xbf) = [\mathrm{sign}(x_i) \max(|x_i| - \alpha_k, 0)] \in \mathbb{R}^n$ for any vector $\xbf = (x_1,\dots,x_n)\in \mathbb{R}^n$ is the soft shrinkage operator.
Plugging \eqref{eq:GJu} into \eqref{eq:ui_update}, we obtain an explicit form of \eqref{eq:ui}, which we summarize together with \eqref{eq:bi} in the following scheme:
\begin{subequations}\label{eq:bu}
\begin{align}
\bbf_i^{(t)} & =  \ubf_i^{(t)} - \rho_t \Fbf^{\top}  \Pbf ^{\top} (\Pbf \Fbf \ubf_i^{(t)} - \fbf_i), \label{eq:b}  \\
\ubf_i^{(t)} & = \bbf_i^{(t)} + [\tilde{\J} \circ \tilde{\G} \circ \soft_{\alpha_t} ( \G \circ \J (\bbf^{(t)}))]_i, \quad 1\le i \le N_c
\label{eq:u}
\end{align}
\end{subequations}
Our proposed reconstruction network thus is composed of a prescribed $T$ phases, where the $t$th phase performs the update of \eqref{eq:bu}.
With a zero initial  $ \{ \ubf_i^{(0)} \} $ and partial k-space data $\{\fbf_i\}$ as input, the network performs the update \eqref{eq:bu} for $1\le t\le T$ and finally outputs $\ubf^{(T)}$.
This network serves as a solver of \eqref{eq:m} and uses $\ubf^{(T)}$ as an approximation of the true solution $\ubf(\fbf;\Theta)$. Hence, the constraint in \eqref{eq:bilevel} is replaced by this network for every input data $\fbf^{[j]}$.

\subsection{Network Architectures and Training}\label{subsec:network}
 \begin{figure}
\includegraphics[width=1\textwidth]{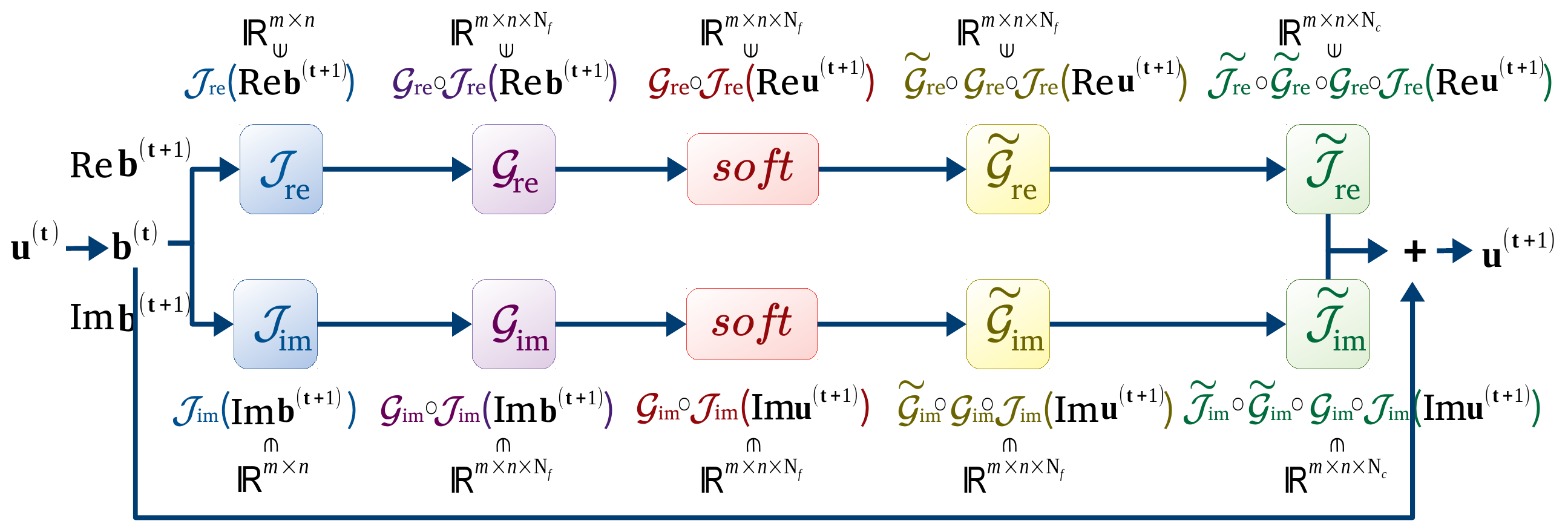}
\caption{Network structure paradigm in $t+1$th phase. $ \text{Re} \bbf^{(t+1)} , \text{Re} \ubf^{(t+1)} \in  \R^{ m \times n \times N_c}$ and $ \text{Im} \bbf^{(t+1)},  \text{Im} \ubf^{(t+1)} \in \R^{ m \times n \times N_c} $ represent for real and imaginary part of $ \bbf^{(t)} $ and $  \ubf^{(t)}$ respectively. } \label{fig1}
\end{figure}

\begin{figure}
\includegraphics[width=\textwidth]{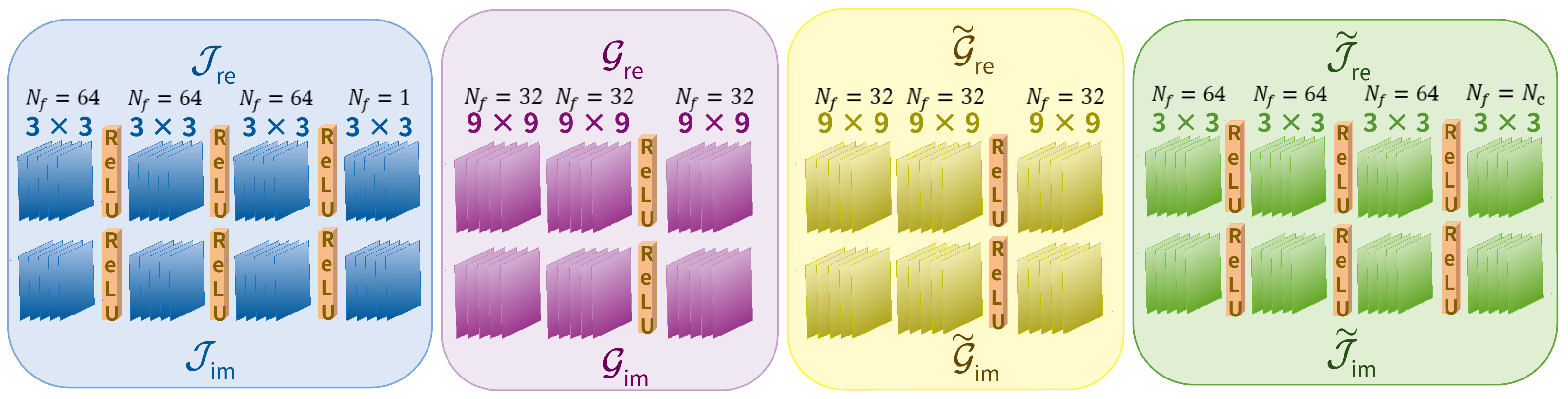}
\caption{ Structure of each convolutional operator. The weights of real and imaginary part are not shared in the network, while both real and imaginary convolutional operators have the same structure.} \label{fig2}
\end{figure}

We set $\J$ as a convolutional network with $N_l=4$ layers and each linear convolution of kernel size $3 \times 3$. The first $N_l-1$ layers have $N_f=64$ filter kernels, and $N_f=1$ in the last layer. Each layer follows an activation ReLU except for the last layer. The operator $\G$ being set as the same way except that $N_f=32$ and kernel size is $ 9\times 9$.
Operators $\tilde{ \J}$ and $\tilde{ \G }$ are designed in symmetric structures as $\J$ and $ \G$ respectively.
We treat a complex tensor as a real tensor of doubled size, and apply convolution separately.  More details on network structure are provided in Fig. \ref{fig1} and \ref{fig2}.  The output of each real and imaginary convolutional operators corresponds to different colors.

The training data $(\fbf,\hat{\ubf})$ consists of $J$ pairs $\{( \fbf_{i}^{[j]}, \hat{\ubf}_{i}  ^{[j]}) \,|\, 1 \le i \le N_c,\ 1\le j\le J \}$.   
To increase network capacity, we allow varying operators of \eqref{eq:bu} in different phases.
Hence  $\Theta = \{ \rho_t,\alpha_t,\J^{(t)},\G^{(t)},\tilde{ \G}^{(t)}, \tilde{ \J}^{(t)}\,|\, 1\le t \le T \}$ are the parameters to be trained.
Based on the analysis of loss functions \cite{7797130, hammernik2017l2, doi:10.1002/mp.13628}, the optimal parameter $ \Theta$ can be solved by minimizing the loss function:
We set the discrepancy measure $\ell$ between the reconstruction $\ubf$ and the corresponding ground truth $\hat{\ubf}$ in \eqref{eq:bilevel} as follows,
\begin{equation}
\label{eq:loss}
\begin{aligned}
\ell(\ubf,\hat{\ubf}) = \| \sbf(\ubf) - \sbf(\hat{\ubf})\|_2 + \gamma \| |\J(\ubf)| - \sbf(\hat{\ubf})\|_2
\end{aligned}
\end{equation}
where $\sbf(\ubf)= (\sum_{i=1}^{N_c} |\ubf_i|^2)^{1/2} \in\mathbb{R}^{mn}$ is the pointwise root of sum of squares across the $N_c$ channels of $\ubf$, $|\cdot|$ is the pointwise modulus, and $\gamma>0$ is a weight function. We also tried replacing the first by $(1/2)\cdot \|\ubf - \hat{\ubf}\|_2^2$, but it seems that the one given in \eqref{eq:loss} yields better results in our experiments. The second term of \eqref{eq:loss} can further improve accuracy of the magnitude of the reconstruction. The initial guess (also the input of the reconstruction network) of any given pMRI $\fbf^{[j]}$ is set to the zero-filled reconstruction $ \Fbf^{-1} \fbf^{[j]}$, and the multi-channel image $\ubf^{(T)}(\fbf^{[j]}; \Theta)$ is the output of the network \eqref{eq:bu} after $T$ phases. In addition, $  \J( {\ubf}^{(T)}(\fbf^{[j]}; \Theta))$ is the final single body image reconstructed as a by-product (complex-valued).

\section{Experimental Results}
\begin{table}[t]
\centering
\caption{Quantitative measurements for reconstruction of Coronal FSPD data. } \label{tab:fs}
\resizebox{\linewidth}{16mm}{ 
\begin{tabular}{cccc}
\toprule
Method   & PSNR                    & SSIM              & RMSE          \\\midrule
GRAPPA\cite{griswold2002generalized}   & 24.9251$\pm$0.9341    & 0.4827$\pm$0.0344  & 0.2384$\pm$0.0175 \\
SPIRiT\cite{doi:10.1002/mrm.22428}  & 28.3525$\pm$1.3314    & 0.6509$\pm$0.0300  & 0.1614$\pm$0.0203 \\
VN\cite{doi:10.1002/mrm.26977}       & 30.2588$\pm$1.1790    & 0.7141$\pm$0.0483  & 0.1358$\pm$0.0152 \\
DeepcomplexMRI\cite{WANG2020136} & 36.6268$\pm$1.9662   & 0.9094$\pm$0.0331 & 0.0653$\pm$0.0085 \\
pMRI-Net ~&~ \textbf{37.8475$\pm$1.2086} ~&~ \textbf{0.9212$\pm$0.0236} ~&~ \textbf{0.0568$\pm$0.0069}~ \\ \bottomrule
\end{tabular}}
\end{table}
\label{sec:experiment}
 \begin{table}[t]
\centering
\caption{Quantitative measurements for reconstruction of Coronal PD data. } \label{tab:pd}
\resizebox{\linewidth}{16mm}{ 
\begin{tabular}{cccc}
\toprule
Method         & PSNR                & SSIM                & RMSE           \\ \midrule
GRAPPA\cite{griswold2002generalized} & 30.4154$\pm$0.5924  & 0.7489$\pm$0.0207   & 0.0984$\pm$0.0030 \\
SPIRiT\cite{doi:10.1002/mrm.22428}   & 32.0011$\pm$0.7920  & 0.7979$\pm$0.0306   & 0.0824$\pm$0.0082 \\
VN \cite{doi:10.1002/mrm.26977}              & 37.8265$\pm$0.4000    & 0.9281$\pm$0.0114    & 0.0422$\pm$0.0036 \\
DeepcomplexMRI \cite{WANG2020136} & 41.5756$\pm$0.6271  & 0.9679$\pm$0.0031   & 0.0274$\pm$0.0018 \\
pMRI-Net   ~&~ \textbf{42.4333$\pm$0.8785} ~&~ \textbf{0.9793$\pm$0.0023} ~&~ \textbf{0.0249$\pm$0.0024}\\ \bottomrule
\end{tabular}}
\end{table}
\begin{figure*}
\centering
\includegraphics[width=0.14\linewidth, angle=180]{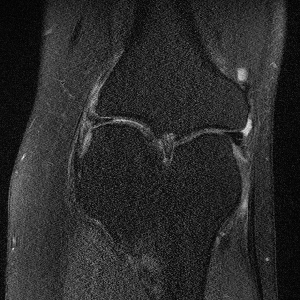}
\includegraphics[width=0.14\linewidth, angle=180]{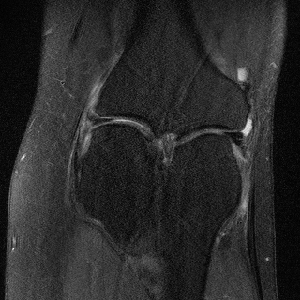}
\includegraphics[width=0.14\linewidth, angle=180]{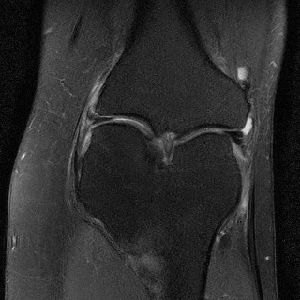}
\includegraphics[width=0.14\linewidth, angle=180]{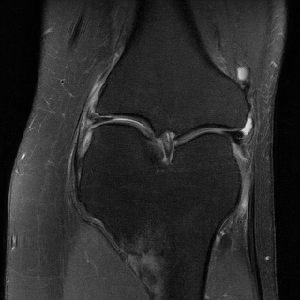}
\includegraphics[width=0.14\linewidth, angle=180]{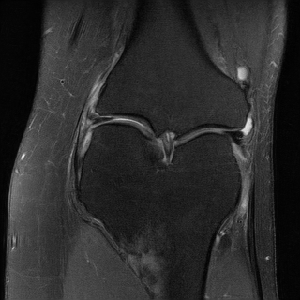}
\includegraphics[width=0.14\linewidth, angle=180]{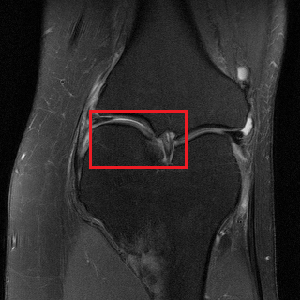}\\
\includegraphics[width=0.14\linewidth, angle=180]{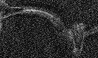}
\includegraphics[width=0.14\linewidth, angle=180]{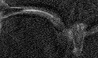}
\includegraphics[width=0.14\linewidth, angle=180]{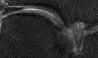}
\includegraphics[width=0.14\linewidth, angle=180]{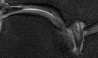}
\includegraphics[width=0.14\linewidth, angle=180]{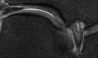}
\includegraphics[width=0.14\linewidth, angle=180]{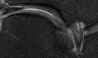}\\
\includegraphics[width=0.14\linewidth, angle=180]{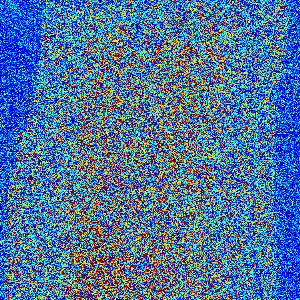}
\includegraphics[width=0.14\linewidth, angle=180]{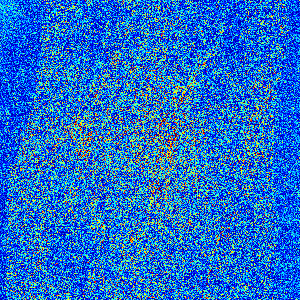}
\includegraphics[width=0.14\linewidth, angle=180]{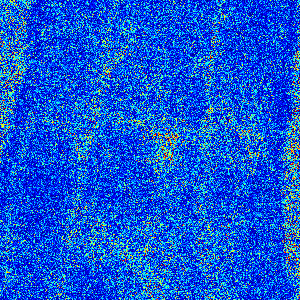}
\includegraphics[width=0.14\linewidth, angle=180]{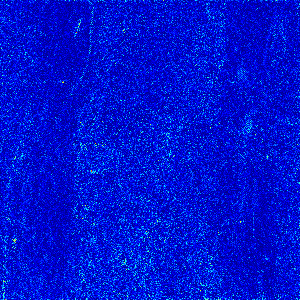}
\includegraphics[width=0.14\linewidth, angle=180]{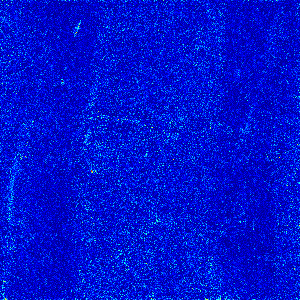}
\includegraphics[width=0.14\linewidth, angle=180]{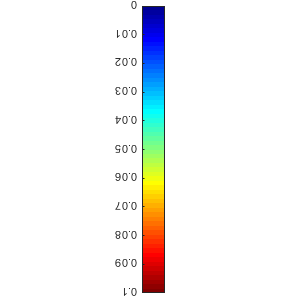}
%%detail Y position 110, X position 90, width 98, hight 58
\caption{Results on the Coronal FSPD knee image with regular Cartesian sampling (31.56\% rate).
From left to right columns: GRAPPA, SPIRiT, VN, deepcomplexMRI, pMRI-Net, and ground truth.}
\label{PDFS}
\end{figure*}

\begin{figure*}
\centering
\includegraphics[width=0.14\linewidth, angle=180]{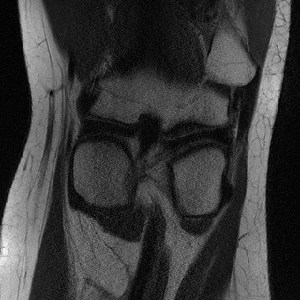}
\includegraphics[width=0.14\linewidth, angle=180]{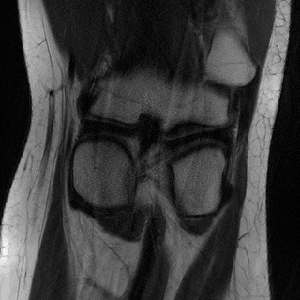}
\includegraphics[width=0.14\linewidth, angle=180]{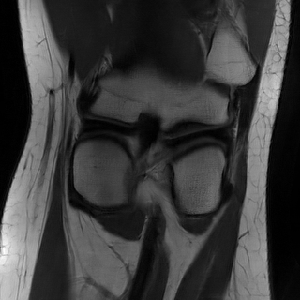}
\includegraphics[width=0.14\linewidth, angle=180]{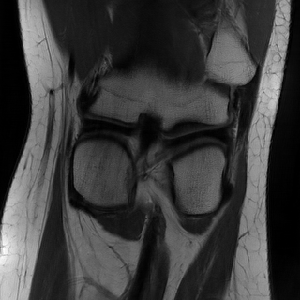}
\includegraphics[width=0.14\linewidth, angle=180]{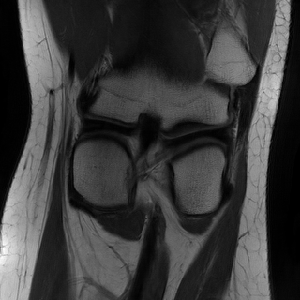}
\includegraphics[width=0.14\linewidth, angle=180]{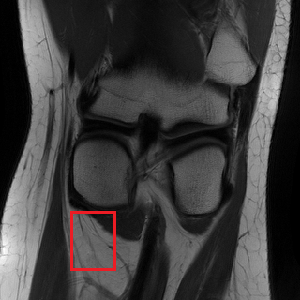}\\
\includegraphics[width=0.14\linewidth, angle=180]{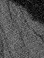}
\includegraphics[width=0.14\linewidth, angle=180]{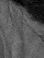}
\includegraphics[width=0.14\linewidth, angle=180]{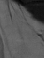}
\includegraphics[width=0.14\linewidth, angle=180]{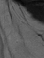}
\includegraphics[width=0.14\linewidth, angle=180]{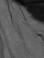}
\includegraphics[width=0.14\linewidth, angle=180]{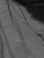}\\
\includegraphics[width=0.14\linewidth, angle=180]{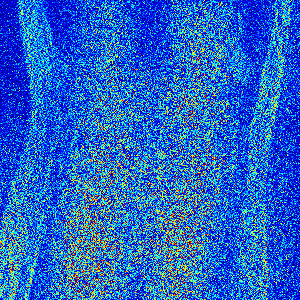}
\includegraphics[width=0.14\linewidth, angle=180]{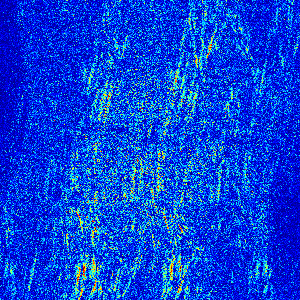}
\includegraphics[width=0.14\linewidth, angle=180]{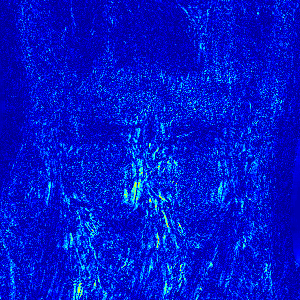}
\includegraphics[width=0.14\linewidth, angle=180]{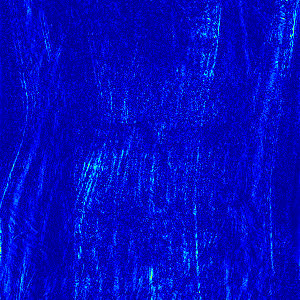}
\includegraphics[width=0.14\linewidth, angle=180]{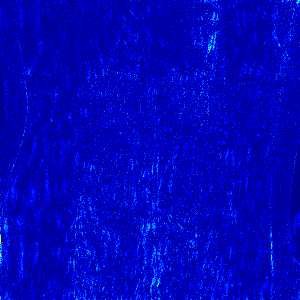}
\includegraphics[width=0.14\linewidth, angle=180]{fig_chp2/colorbar.png}
%detail Y position 71, X position 212, width 46, hight 58
\caption{Results on the Coronal PD knee image with regular Cartesian sampling (31.56\% rate). From left to right columns: GRAPPA, SPIRiT, VN, DeepcomplexMRI, pMRI-Net, and ground truth.}
\label{PD}
\end{figure*}
\begin{figure}[t]
\centering
\includegraphics[width=0.14\linewidth, angle=180]{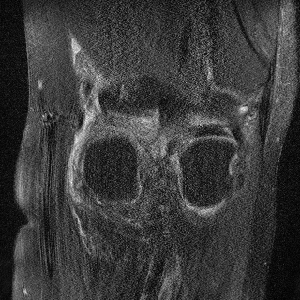}
\includegraphics[width=0.14\linewidth, angle=180]{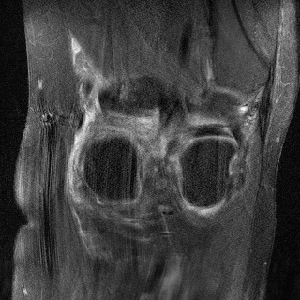}
\includegraphics[width=0.14\linewidth, angle=180]{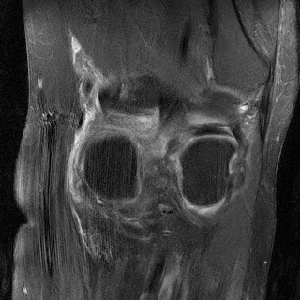}
\includegraphics[width=0.14\linewidth, angle=180]{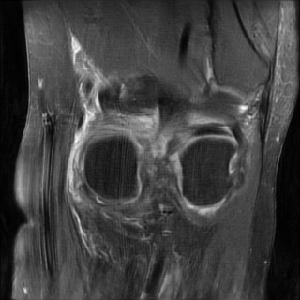}
\includegraphics[width=0.14\linewidth, angle=180]{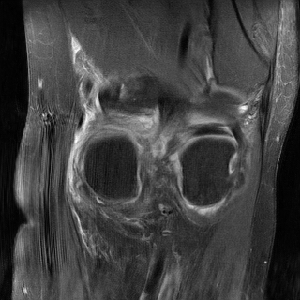}
\includegraphics[width=0.14\linewidth, angle=180]{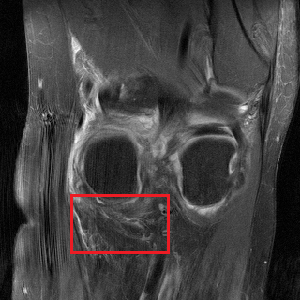}\\
\includegraphics[width=0.14\linewidth, angle=180]{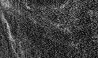}
\includegraphics[width=0.14\linewidth, angle=180]{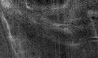}
\includegraphics[width=0.14\linewidth, angle=180]{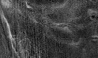}
\includegraphics[width=0.14\linewidth, angle=180]{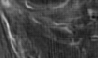}
\includegraphics[width=0.14\linewidth, angle=180]{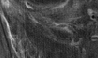}
\includegraphics[width=0.14\linewidth, angle=180]{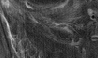}\\
\includegraphics[width=0.14\linewidth, angle=180]{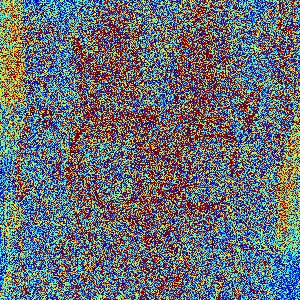}
\includegraphics[width=0.14\linewidth, angle=180]{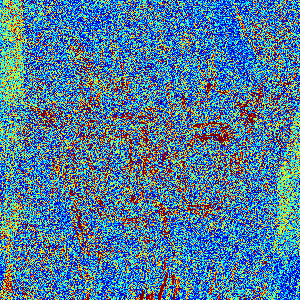}
\includegraphics[width=0.14\linewidth, angle=180]{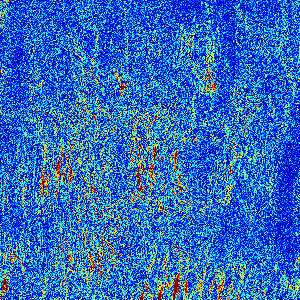}
\includegraphics[width=0.14\linewidth, angle=180]{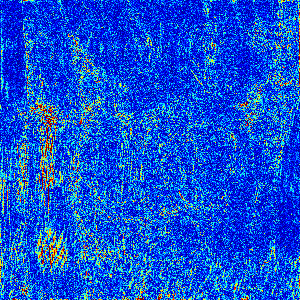}
\includegraphics[width=0.14\linewidth, angle=180]{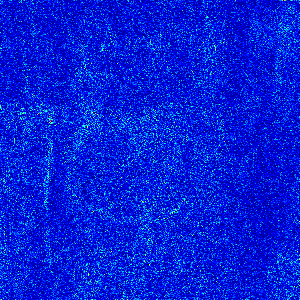}
\includegraphics[width=0.14\linewidth, angle=180]{fig_chp2/colorbar.png}
%Y 195, X 71, width 98, height 58
\caption{ Additional experiments results on the Coronal FSPD knee image with regular Cartesian sampling (31.56\% rate).
From left to right columns: GRAPPA, SPIRiT, VN, DeepcomplexMRI, proposed, and ground truth. }
\label{fig3}
\end{figure}  
\begin{figure}[t]
\centering
\includegraphics[width=0.14\linewidth, angle=180]{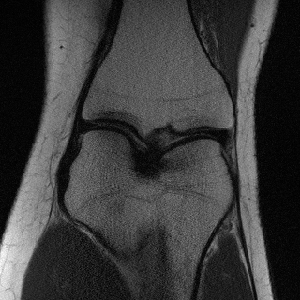}
\includegraphics[width=0.14\linewidth, angle=180]{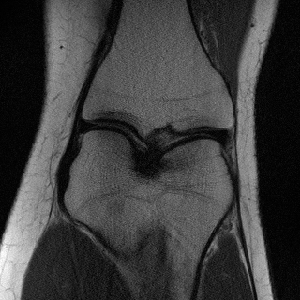}
\includegraphics[width=0.14\linewidth, angle=180]{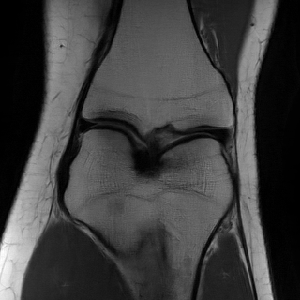}
\includegraphics[width=0.14\linewidth, angle=180]{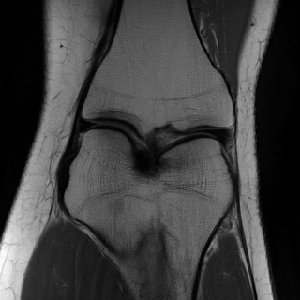}
\includegraphics[width=0.14\linewidth, angle=180]{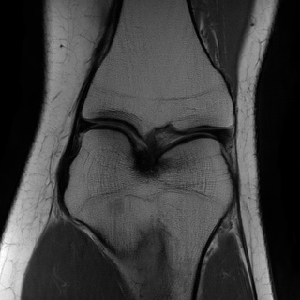}
\includegraphics[width=0.14\linewidth, angle=180]{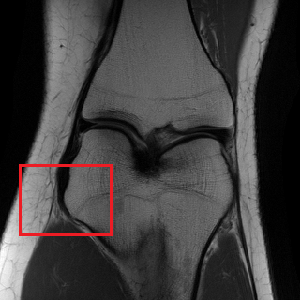}\\
\includegraphics[width=0.14\linewidth, angle=180]{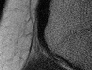}
\includegraphics[width=0.14\linewidth, angle=180]{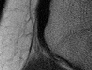}
\includegraphics[width=0.14\linewidth, angle=180]{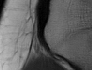}
\includegraphics[width=0.14\linewidth, angle=180]{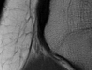}
\includegraphics[width=0.14\linewidth, angle=180]{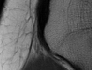}
\includegraphics[width=0.14\linewidth, angle=180]{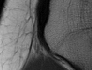}\\
\includegraphics[width=0.14\linewidth, angle=180]{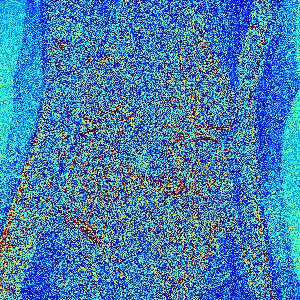}
\includegraphics[width=0.14\linewidth, angle=180]{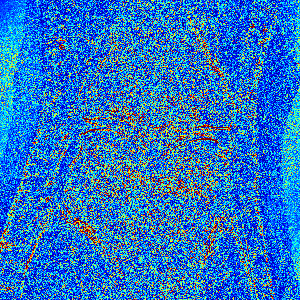}
\includegraphics[width=0.14\linewidth, angle=180]{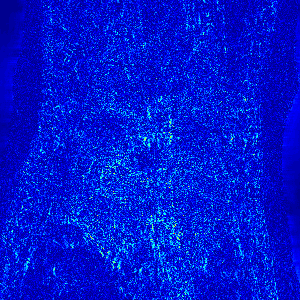}
\includegraphics[width=0.14\linewidth, angle=180]{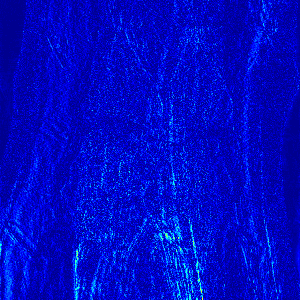}
\includegraphics[width=0.14\linewidth, angle=180]{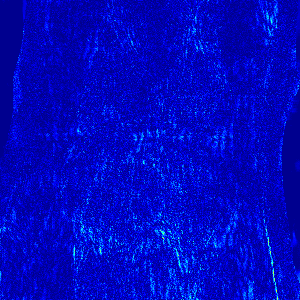}
\includegraphics[width=0.14\linewidth, angle=180]{fig_chp2/colorbar.png}
% Y 165, X 20, width 92, height 70
\caption{Additional experiments results on the Coronal PD knee image with regular Cartesian sampling (31.56\% rate). From left to right columns: GRAPPA, SPIRiT, VN, DeepcomplexMRI, Proposed and ground truth. }
\label{fig4}
\end{figure}  
\textbf{Data.} Two sequences of data named Coronal proton-density (PD) and Coronal fat-saturated proton-density (FSPD) along with the regular Cartesian sampling mask with 31.56\% sampling ratio were obtained from \url{https://github.com/VLOGroup/mri-variationalnetwork} in our experiment. Each of the two sequences data were scanned from 20 patients. The training data consists of  526 central image slices with matrix size $ 320 \times 320$ from 19 patients, and we randomly pick 15 central image slices from the one patient that not included in training data as the testing data. We normalized training data by the maximum of the absolute valued zero-filled reconstruction.
\medskip

\noindent
\textbf{Implementation.} The proposed network was implemented with $T=5$ phases. We use Xavier initialization \cite{pmlr-v9-glorot10a} to initialize network parameters and Adam optimizer for training. Experiments apply mini-batches of 2 and 3000 epochs with learning rate 0.0001 and 0.0005 for training Coronal FSPD data and PD data respectively. The initial step size $ \rho_0 =0.1$, threshold parameter $\alpha_0 = 0$ and $\gamma=10^5$ in the loss function. All the experiments were implemented in TensorFlow on a workstation with Intel Core i9-7900 CPU and Nvidia GTX-1080Ti GPU.
\medskip

\noindent
\textbf{Evaluation.}
 We evaluate traditional methods GRAPPA \cite{griswold2002generalized}, SPIRiT \cite{doi:10.1002/mrm.22428}, and deep learning methods VN \cite{doi:10.1002/mrm.26977} , DeepcomplexMRI \cite{WANG2020136} over the 15 testing  Coronal  FSPD and  PD knee images in terms of PSNR, SSIM \cite{wang2004image} and RMSE (RMSE of $ \hat{\xbf}$ to true $\xbf^*$ is defined by $\| \hat{\xbf} - \xbf^*\|/\| \xbf^* \| $).
\medskip

\noindent
\textbf{Experimental results.} The average numerical performance with standard deviations are summarized in Table  \ref{tab:fs} and \ref{tab:pd}. The comparison on reconstructed images are shown in Fig. \ref{PDFS}, \ref{fig3} and Fig. \ref{PD}, \ref{fig4} for Coronal FSPD and PD testing data respectively.  From top to bottom rows: image, zoom-in views, and pointwise absolute error to ground truth. In Fig. \ref{PDFS}, the corresponding PSNR/SSIM/RMSE for each method from left to right columns are listed below: GRAPPA: 25.6656/0.4671/0.2494, SPIRiT: 29.5550/0.6574/0.1594, VN: 31.5546/0.7387/0.1333, deepcomplexMRI: 38.6842/0.9360/0.0587,  pMRI-Net: 38.8749/0.9375/0.0574, and ground truth.  In Fig. \ref{fig3},
from left to right columns: GRAPPA: 22.2203/0.3596/0.2666, SPIRiT: 25.3434/0.5269/0.1861, VN: 29.2276/0.7591/0.1190, DeepcomplexMRI: 31.3597/0.8231/0.0931, proposed: 36.8831/0.9372/0.0493, and ground truth.
In Fig. \ref{PD}, from left to right columns: GRAPPA: 29.9155/0.7360/0.1032, SPIRiT: 33.2350/0.8461/0.0704, VN: 38.3192/0.9464/0.0393, DeepcomplexMRI: 41.2098/0.9713/0.0281, pMRI-Net: 42.9330/0.9798/0.0231, and ground truth.
In Fig. \ref{fig4}, from left to right columns: 
from left to right columns: GRAPPA: 27.3571/0.5338/0.1357, SPIRiT: 28.3076/0.5592/0.1216, VN: 38.1679/0.9259/0.0391, DeepcomplexMRI: 40.9388/0.9678/0.0284, Proposed: 41.7025/0.9697/0.0260 and ground truth.

Despite of the lack of coil sensitivities in training and testing, the proposed method still outperforms VN in reconstruction accuracy significantly while VN uses precomputed coil sensitivity maps from ESPIRiT \cite{uecker2014espirit}, which further shows that the proposed method can achieve improved accuracy without knowledge of coil sensitivities. Comparing 10 complex CNN blocks in
DeepcomplexMRI  with 5 phases in pMRI-Net, the latter requires fewer network parameters and less training time but improves reconstruction quality.

In the experiment of GRAPPA and SPIRiT, we use calibration kernel size $ 5\times 5$ with Tikhonov regularization in the calibration setted as 0.01. We implement SPIRiT with 30 iterations and set Tikhonov regularization in the reconstruction as $10^{-3}$. Default parameter settings for experiments of VN and DeepcomplexMRI were applied.  The final recovered image from VN is a full FOV single channel image, and DeepcomplexMRI produces a multi-coil image, which are combined into single channel image using adaptive multi-coil combination method \cite{Walsh2000682}. pMRI-Net reconstructs both single channel image $ \J( \ubf^{(T)}(\fbf; \Theta))$ and multi-channel image $\{\ubf_{i}^{(T)}(\fbf_{i}; \Theta)\} $.

\section{Conclusion}
\label{sec:conclusion}
We exploit a learning based multi-coil MRI reconstruction without explicit knowledge of coil sensitivity maps and the network is modeled in CS framework with proximal gradient scheme. The proposed network is designed to combine features of channel-wise images, and then extract sparse features from the coil combined image. Our experiments showed better performance of proposed ``combine-then-regularize” approach. 

%% file: tex/chapter3.tex
\chapter{An Optimal Control Framework for Joint-channel Parallel MRI  Reconstruction without Coil Sensitivities}\label{optimalcontrol}

This work aims at developing a novel calibration-free fast parallel MRI (pMRI) reconstruction method incorporate with discrete-time optimal control framework. The reconstruction model is designed to learn a regularization that combines channels and extracts features by leveraging the information sharing among channels of multi-coil images. We propose to recover both magnitude and phase information by taking advantage of structured convolutional networks in image and Fourier spaces.
We develop a novel variational model with a learnable objective function that integrates an adaptive multi-coil image combination operator and effective image regularization in the image and Fourier spaces. We cast the reconstruction network as a structured discrete-time optimal control system, resulting in an optimal control formulation of parameter training where the parameters of the objective function play the role of control variables. We demonstrate that the Lagrangian method for solving the control problem is equivalent to back-propagation, ensuring the local convergence of the training algorithm.
We conduct a large number of numerical experiments of the proposed method with comparisons to several state-of-the-art pMRI reconstruction networks on real pMRI datasets. The numerical results demonstrate the promising performance of the proposed method evidently. 
The proposed method provides a general deep network design and training framework for efficient joint-channel pMRI reconstruction.
By learning multi-coil image combination operator and performing regularizations in both image domain and k-space domain, the proposed method achieves a highly efficient image reconstruction network for pMRI. 
		
\section{Introduction}\label{sec:introduction}
	
	Magnetic resonance imaging (MRI) is one of the most prominent medical imaging technologies with extensive clinical applications. In clinical applications, an advanced medical MRI technique known as parallel MRI (pMRI) is widely used. PMRI surrounds the scanned objects with multiple receiver coils and collects k-space (Fourier) data in parallel. PMRI can reduce the data acquisition time and has become the state-of-the-art technology in modern MRI applications. To accelerate the scan process, partial data acquisitions that increase the spacing between read-out lines in k-space are implemented in pMRI. However, this results in aliasing artifacts, and a proper image reconstruction process is necessary to recover the high-quality artifact-free images from the partial data.
	
	Two major approaches are commonly addressed to image reconstruction in pMRI: the first approach is k-space method which interpolates the missing k-space data using the sampled ones across multiple receiver coils \cite{doi:10.1002/jmri.23639}, such as the generalized auto-calibrating partially parallel acquisition (GRAPPA) \cite{griswold2002generalized} and  simultaneous acquisition of spatial harmonics (SMASH) \cite{sodickson1997simultaneous}. The other approach is the class of image space method that eliminate the aliasing artifacts in the image domain by solving a system of equations that relate the image to be reconstructed and partial k-space data through coil sensitivities, such as SENSitivity Encoding (SENSE) \cite{pruessmann1999sense}.
	Coil sensitivity maps are indispensable and required to be accurately pre-estimated in traditional SENSE-based methods.
	Traditional pMRI reconstruction methods in image space follows SENSE-based framework, which is formulated as an optimization problem that minimizes a summation of a data fidelity term and a weighted regularization term. The detailed explanation about this formulation can be refer to \cite{knoll2020deep}.

	In recent years, we have witnessed fast developments of SENSE-based pMRI  reconstruction incorporate with deep-learning based methods \cite{knoll2020deep, lu2020pfista, tavaf2021grappa,hammernik2021systematic, lv2021pic}. There are two critical issues that need to be carefully addressed.
	The first issue is on the choice of regularization including the regularization weight.
	The regularization term is of paramount importance to the severely ill-posed inverse problem of pMRI reconstruction due to the significant undersampling of k-space data. 
	In the past decades, most traditional image reconstruction methods employ handcrafted regularization terms, such as the total variation (TV). %, which are oversimplified and the reconstructed image quality is often not comparable with deep-learning based regularizers.
	In recent years, a class of unrolling methods that mimic classical optimization schemes are developed, where the regularization is realized by deep networks whose parameters are learned from data. However, the learned regularization is often cast as a black-box that is difficult to interpret, and the training can be very data demanding and prone to overfitting especially when the networks are over-parameterized \cite{lecun2015deep,279181,VanishingGradient}. 
	
	The second issue is due to the unavailability of accurate coil sensitivities $\{ \sbf_i\}$ in practice. Inaccurate coil sensitivity maps lead to severe biases that degrade the quality of reconstructed $\vbf$. 
	One way to eliminate this issue is to (regularize and) reconstruct  multi-coil images, and combine channels into a full-body image in the final step by taking some hand-crafted methods such as the root of sum of squares (RSS). Different from RSS, our method proposed a learnable multi-coil combination operator to combine channel-wise multi-coil images. However, if all coils yield low SNR or data with artifacts, RSS will arise background noise level increases since it weights artifacts equally to the dark area \cite{RSS_bias, RSS_noise}.% In addition, RSS only produces the magnitude of the MR image \blue{without phase information preserved which could be useful in certain application areas such as echo planar imaging \cite{buonocore1997ghost} and phase contrast imaging \cite{pelc1991phase}.} % or susceptibility weighted imaging \cite{haacke2004susceptibility}.
	
	In this paper, we tackle the aforementioned issues in a discrete-time optimal control framework to optimize the variational pMRI reconstruction model.
	%We impose two learnable regularizers  in model \eqref{eq:PFS_chp3} to form $R$: one is the composition of a multi-coil image combination operator and a full-body image regularization using deep neural nets, and another one incorporate k-space prior information of the target coil images to increase the accuracy of the Fourier signals. 
	We highlight several main features of our framework as follows.
	
	\begin{enumerate}[leftmargin=*]
		\item Unlike most existing methods which regularize and reconstruct multi-coil images, we employ regularization in both image and Fourier spaces to improve reconstruction quality. 
		\item Our method advocates a learned adaptive combination operator that first merges multi-coil images into a full-body image with a complete field of view (FOV), followed by an effective regularization on this image. 
		This is in sharp contrast to existing methods which only combine reconstructed multi-coil images in the final step, whereas our regularizer leverages the combination operator in each iteration which improves parameter efficiency.
		\item We employ a complex-valued neural network as the coil combination operator to recover both magnitude and phase information of pMRI images when coil sensitivity is unavailable. This combination method benefits from the coil information shared among multiple channels, which is distinct from most hand-crafted coil combination methods.  
		\item We propose a novel deep reconstruction network whose structure is determined by the discrete-time optimal control system for minimizing the objective function, which yields an optimal control formulation where the parameters of the combination and regularization operators play the role of control variables of the discrete dynamic system. The optimal value of these parameters is obtained by a Lagrangian method which can be implemented using back-propagation. 
		
		%   \item We conduct a learnable initial reconstruction for the proximal gradient inspired iterative algorithm, and it provides favorable support to the reconstruction performance comparing with some other common initialization methods. 
	\end{enumerate}
	
	We consider two clinical pMRI sequences of knee images and verified the effective performance of the proposed combination operator, different initial reconstructions, complex convolutions, and domain-hybrid network in the Ablation Studies. The proposed network recovers both magnitude and phase information of pMRI images. The effect of the aforementioned techniques demonstrate evident improvement of reconstruction quality and parameter efficiency using our method. For reproducing the experiment, our code is available at \url{https://github.com/1lol/pMRI_optimal_control}.

	This remainder of the paper is organized as follows: In section \ref{sec:related_chp3}, we provide an overview of recent developments in pMRI, cross-domain reconstructions, complex-valued CNNs, and optimal control inspired deep training models that related to our work. We present our proposed problem settings and reconstruction network architecture in detail in Section \ref{sec:proposed_chp3}. Extensive numerical experiments and analyses on a variety of clinical pMRI data are presented in Section \ref{sec:experiment_chp3}. Section \ref{sec:conclusion_chp3} concludes this paper.

	\section{Related Work}
	\label{sec:related_chp3}
	
	In recent years, we have witnessed fast developments of medical imaging incorporate with deep-learning based methods \cite{huang2020medical, huang2020magnetic, huang2020mri}. Most existing deep-learning based methods rendering end-to-end neural networks mapping from the partial k-space data to the reconstructed images \cite{WANG2020136, zhou_pmri, Quan2018CompressedSM,8417964, zhu2018}. These approaches require an excessive amount of training data, and the designed networks are cast as black boxes whose underlying mechanism can be very difficult to interpret. To mitigate this issue, a number of unrolling methods were proposed to map existing optimization algorithms to structured networks where each phase of the networks correspond to one iteration of an optimization algorithm \cite{doi:10.1002/mrm.26977, adler2018learned, Aggarwal_2019,8550778, cheng2019model,NIPS2016_6406,zhang2018ista, 8067520}.
	In what follows, we focus on the recent developments in deep-learning based image reconstruction methods for pMRI.
	
	%\textbf{Deep networks for pMRI reconstruction}:
	%  by calibrating from the fully-sampled central k-space regions of each slice, this technique recover coil sensitivity as the pointwise eigenvectors to the eigenvalue one of a single reconstruction operator.
	Some of the networks for pMRI reconstruction are using ESPIRiT:
	Variational Network (VN) \cite{doi:10.1002/mrm.26977} was introduced to unroll the gradient descent algorithm as a reconstruction network which requires precalculated sensitivities $\{\sbf_i\}$ as input.
	E2E-VarNet \cite{E2E-VN} modeled a modified VN with learned sensitivity maps in cascaded refinement modules and their result shows superior performance in the list of public fastMRI leaderboard.
	MoDL \cite{Aggarwal_2019} proposed a weight sharing strategy in a recursive network to learn the regularization parameters by unrolling the conjugate gradient method.
	PI-CNN \cite{zhou_pmri} was proposed as an end-to-end cascaded network which concatenates a CNN block and a parallel imaging data consistency block in each phase.
	The techniques introduced above are using ESPIRiT \cite{uecker2014espirit} to estimate coil sensitivity maps.
	VS-Net \cite{10.1007/978-3-030-32251-9_78} applied  coil sensitivities precomputed from BART \cite{uecker2013software} that employ the center block of fully sampled k-space, and unrolls the optimization steps from the variable splitting algorithm.
	APIR-Net \cite{10.1007/978-3-030-33843-5_4} proposed an auto-calibrated k-space completion method in hierarchical way that progressively increase the size of ACS region.
	Several methods explored different strategies to avoid using pre-calculated coil sensitivity maps for pMRI reconstruction.
	Blind-PMRI-Net \cite{10.1007/978-3-030-32251-9_80} proposed pMRI model by regularizing sensitivity maps and MR image, where their network alternatively estimates coil images, sensitivities and single-body image by three subnets. %However, the individual subnets benefit each other and need to be learned cooperatively to improve the reconstruction quality, their reconstruction quality depending on estimation of the coil sensitivities.
	De-Aliasing-Net \cite{10.1007/978-3-030-32248-9_4} proposed a de-aliasing reconstruction model with that applied split Bregman iteration algorithm without explicit coil sensitivity calculation. The de-aliasing network explored cross-correlation among channels and spatial redundancy which provoked a desirable performance.
	LINDBERG \cite{wang2017learning} explored calibration-free pMRI technique which uses adaptive sparse coding to obtain joint-sparse representation precisely by equipping a joint sparsity regularization to extract desirable cross-channel relationship. This work proposed to alternatively update sparse representation, sensitivity encoded images, and K-space data. 
	SCDAE \cite{islam2021compressed} was developed to estimate coil-wise sensitivity maps via convolutions and fully connected layers. The reconstruction employed a TV-based minimization algorithm that is solved by the Bregman iteration technique. The full FOV image is obtained by RSS for multi-channel combination.
	Adaptive-CS-Net \cite{pezzotti2020adaptive} is a leading method in 2019 fastMRI challenge \cite{zbontar2018fastmri} that unrolled modified ISTA-Net$^{+}$ \cite{zhang2018ista}. 
	The proposed calibration-free pMRI method distincts from above related works in terms of the learnable multi-coil combination operator to adaptively combine channels of the updated multi-coil images through iterations.
	
	Recently, cross-domain methods exhibits its significance in medical imaging \cite{  doi:10.1002/mrm.27201, wang2020ikwi, sriram2020grappanet, 10.1007/978-3-030-59713-9_41,   pmlr-v102-souza19a,  10.1007/978-3-030-59713-9_34,  10.1007/978-3-030-59713-9_37, souza2019hybrid}.
	AUTOMAP \cite{zhu2018} was developed as fully connected layers followed by convolutional layers to learn the transform from undersampled k-space to image domain.
	GrappaNet \cite{sriram2020grappanet} takes advantage of physical information by introducing a GRAPPA operator.
	CD-SFCRF framework \cite{CD-SFCRF} proposed a stochastically fully connected graphical model to produce MRI reconstruction by taking advantage of constraints in both k-space and spatial domains. 
	KIKI-net \cite{doi:10.1002/mrm.27201} iteratively applied k-space CNN, image domain CNN and interleaved data consistency operation for single-coil image reconstruction. 
	IKWI-net \cite{wang2020ikwi} proposed CNN blocks in image domain, k-space, wavelet domain and image domain sequentially.
	CDF-Net \cite{ 10.1007/978-3-030-59713-9_41} further shows adding communication between spatial and frequency domain gives a boost in performance. Their results indicated that domain-specific network has individual strong points and disadvantages in restoring tissue-structure. Our reconstruction model is inspired of cross-domain reconstruction, the difference is that we solve for the reconstruction model with cross-domain regularization functions through a learnable optimization algorithm instead of an end-to-end network. However, these black-box sub-networks are trained in an end-to-end manner and are parameter inefficient, which degrade the training performance. In this paper we integrate the idea of cross-domain to the network inspried by proximal gradient decent, in particular, we combine the idea of cross-domain in the residual network. There is still limited investigation that being carefully manipulated with respect to the specialty of pMRI.  
	
	%\textbf{Complex deep networks for MRI reconstruction}:
	
	Our network applies complex-valued convolutions and activation functions. MRI data are complex-valued, and the phase signals also carry important pathological information such as in quantitative susceptibility mapping \cite{cole2020analysis, sandino2020compressed}. %\blue{Complex Highly-constrained back-projection (HYPR) local reconstruction (LR) \cite{wang2009ultrashort} technique and the expectation maximization algorithm based method \cite{choi2013iterative} have been proposed to solve for complex-valued image reconstruction problem. 
	Most existing deep-learning based MRI reconstruction apply standard real-valued convolutions and nonlinear activation functions to the real and imaginary parts separately \cite{10.1007/978-3-030-61598-7_2, doi:10.1002/mrm.26977, Aggarwal_2019, NIPS2016_6406}.
	Complex valued CNNs is benefit for optimization method with properly designed regularization. On the one hand, complex convolutions mitigating the effects of overfitting comparing to real-valued CNNs \cite{quan2021image} due to reduction the number of training parameters. On the other hand, the resulting  network structure is more  interpretable.
	%Several related researches are summarized as follows:
	%
	Complex convolutions and algorithms, batch-normalization, weight initialization strategies were exploited by Trabelsi $et$ $al.$ \cite{trabelsi2017deep}.
	Scardapane $et$ $al.$ \cite{8495012} extended the general idea of kernel activation functions to design nonparametric activation function for complex-valued neural networks. 
	Cole $et$ $al.$ \cite{cole2020analysis} investigated the performance of complex-valued convolution and activation functions has better reconstruction over the model with real-valued convolution in various network architectures. 
	Virtue $et$ $al.$ \cite{8297024} developed complex cardioid as activation function with complex-valued neural network to identify tissue parameters for MRI fingerprinting.
	%Daval-Frerot $et$ $al.$ \cite{complexact} designed trainable complex activation functions to improve complex neural networks in MR fingerprinting. 
	$\C$DFNet \cite{dedmari2018complex} proposed an end-to-end learning based on U-Net with complex convolution followed by complex batch normalization and  rectified linear unit (ReLU). 
	Co-VeGAN \cite{vasudeva2020covegan} developed  complex-valued generative adversarial network with complex-valued activation function to promote the reconstruction performance, which preserves phase information. 
	DeepcomplexMRI \cite{WANG2020136} was developed to recover multi-coil images by unrolling an end-to-end complex-valued network. It consists of several blocks where each block performs complex-valued convolutions on the multi-coil images for denoising and a data consistency operation that merges the original sampled k-space data to the pseudo full k-space of the denoised images.
	
	In supervised learning, deep residual neural networks can be approximated as discretizations of a classical optimal control problem of a dynamical system, where training parameters can be viewed as control variables \cite{weinan2017proposal, lidynamical,li2019deep}. Control inspired learning algorithms introduced a new family of network training models which connect with dynamical systems. Pontryagin's maximum principle (PMP) \cite{PMP} was explored as necessary optimality conditions for the optimal control \cite{li2017maximum}, \cite{pmlr}, these works devise the discrete method of successive approximations (MSA) \cite{MSA} and its variance for solving PMP. %Their methods displayed favorable convergence rate per-iteration in terms of train/test loss, train/test accuracy and error rates.
	Neural ODE \cite{chen2018neural} models the continuous dynamics of hidden states by some certain types of neural networks such as ResNet, the forward propagation is equivalent to one step of discretatized ordinary differential equations (ODE). Inspired by Neural ODE \cite{chen2018neural}, Chen $et$ $al.$ \cite{chen2020mri} modeled ODE-based deep network for MRI reconstruction. In this paper, we model the optimization trajectory as a discrete dynamic process from the view of the method of Lagrangian Multipliers.
	
	The present work is a substantial extension of the preliminary work in \cite{10.1007/978-3-030-61598-7_2} using domain-hybrid network with a trained initialization to solve for an optimal control pMRI joint-channel reconstruction problem when coil-sensitivity is unavailable. More comprehensive empirical study is conducted in this work.

	\section{Proposed Method}
	\label{sec:proposed_chp3}
	
	\subsection{Background} \label{background}
	
	PMRI as well as general MRI reconstruction can be formulated as an inverse problem. Consider a pMRI system with $c$ receiver coils acquiring 2D MR images at resolution $m\times n$ (we treat a 2D image $\vbf \in \mathbb{C}^{m\times n}$ and its column vector form $\vbf \in \mathbb{C}^{mn}$ interchangeably hereafter). Let $\Pbf \in \mathbb{R}^{p\times mn}  (p \le mn)$ be the binary matrix representing the undersampling mask with $p$ sampled locations in k-space, $\sbf_i\in\mathbb{C}^{mn}$ the coil sensitivity, and $\fbf_i \in \mathbb{C}^{p}$ the \emph{partial} k-space data at the $i$-th receiver coil for $i=1,\dots,c$. The partial data $\fbf_i$ and the image $\vbf$ are related by $\fbf_i = \Pbf \Fbf (\sbf_i \odot \vbf) + \mathbf{n}_i$, where $\sbf_i$ is the (unknown) sensitivity map at the $i$-th coil and $\odot$ denotes entrywise product of two matrices, $\Fbf \in \mathbb{C}^{mn\times mn}$ stands for the (normalized) discrete Fourier transform that maps an image to its Fourier coefficients, and $\mathbf{n}_i$ is the unknown acquisition noise in k-space at the $i$-th receiver coil. Then the variational model for image reconstruction can be cast as an optimization problem as follows:
	\begin{equation}\label{eq:PFS_chp3}
		\min_{\vbf} \ \sum\nolimits^{c}_{i=1} \frac{1}{2} \| \Pbf \Fbf (\sbf_i \odot \vbf)- \fbf_i\|^2 + R(\vbf),
	\end{equation}
	where $\vbf\in \mathbb{C}^{m n}$ is the single full-body MR image to be reconstructed, $R(\vbf)$ is a regularization on the image $\vbf$, and $\| \wbf \|^2 := \| \wbf \|_2^2 = \sum_{j=1}^n |w_j|^2$ for any complex vector $\wbf = (w_1,\dots,w_n)^{\top} \in \mathbb{C}^n$.
	Our approach is based on uniform Cartesian k-space sampling. Table \ref{tab:notations} displays the notations and their descriptions that we used in the paper.
	%\ye{Did you explain the transpose and Hermitian? Also write a sentence with reference to Table 1 (do so for every figure/table/algorithm in a paper).}
	
	\begin{table}
		\centering
		\caption{Some notations and meanings that used throughout this paper.}
		\label{tab:notations}
		%\resizebox{\linewidth}{50mm}{ 
			\begin{tabular}{ll}
				\hline
				Expression     &  Description \\
				\hline
				$c$ & total number of receiver coils\\
				$\ubf= (\ubf_1,\dots, \ubf_{c}) $ & multi-coil MRI data\\
				$\fbf= (\fbf_1,\dots,\fbf_{c}) $ & partial k-space measurement\\
				$\sbf = (\sbf_1,\dots, \sbf_{c}) $ & coil sensitivity map\\
				$\vbf $ & full fov image that need to be reconstructed \\
				$\ubf^*$ & ground truth multi-coil MRI data\\
				$\vbf^*$ & ground truth single body MRI data\\
				$\Fbf $ & discrete Fourier transform\\
				$\Fbf^{H} $ & inverse discrete Fourier transform\\
				$\Pbf $ & undersampling trajectory\\
				$\mathbf{n}$ & measurement noise\\
				$\gbf$ & algorithm unrolling network\\
				$\gbf_0$ & initial network\\
				$\J,\G,\tilde{\G}, \tilde{\J}$ & image space convolutional operators in $\gbf$ \\
				$\K$ & k-space convolutional operators in $\gbf$ \\
				$\K_0$ & k-space convolutional operator in $\gbf_0$\\
				$h$ & data fidelity term\\
				$R$ & regularization\\
				\text{RSS} & square root of the sum of squares\\
				$ t=1,\cdots,T$ & phase number \\
				$ k=1,\cdots,K$ & number of iterations for Alg \eqref{alg:1}\\
				$\Ubf=(\ubf{(0)}, \cdots, \ubf{(T)})^{\intercal}$ & collection of predicted multi-coil images\\
				& at each phase\\
				$\Theta=(\theta(0),\cdots, \theta(T))^{\intercal}$ & collection of control variable (parameters)\\
				& at each phase\\    
				$\Lambda = (\lambda(0), \cdots, \lambda(T))^{\intercal}$ & collection of Lagrangian multipliers of \eqref{eq:bilevel_chp3}\\
				\hline
		\end{tabular}
	\end{table}
	
	\subsection{Problem Settings} \label{subsec:setting}
	%The general model is to  minimize the unconstrained optimization problem \eqref{eq:PFS_chp3} by applying a dedicated learnable regularization term.
	
	%
	We propose a unified deep neural network for calibration-free pMRI reconstruction by recovering images from individual receiver coils jointly that does not require any knowledge of coil-wise sensitivity profile. 
	Denote that $\ubf_i$ is the MR image at the $i$-th receiver coil and hence is related to the full body image $\vbf$ by $\ubf_i = \sbf_i  \odot \vbf$. 
	On the other hand, the image $\ubf_i$ corresponds to the partial k-space data $\fbf_i$ by $\fbf_i = \Pbf \Fbf \ubf_i + \mathbf{n}_i$, and hence the data fidelity term is formulated as least squares $ \frac{1}{2}\sum^{c}_{i=1} \|\Pbf \Fbf \ubf_i - \fbf_i\|^2$.
	We also need a suitable regularization $R$ on the images $\{\ubf_i\}$.
	However, these images have severely geometrically inhomogeneous contrasts due to the physical variations in the sensitivities across the image domain at different receiver coils. Therefore, it is more appropriate and effective to apply regularization to the single full-body image $\vbf$ than to individual coil images in $( \ubf_1,\dots,\ubf_c )$.

	%It is a common practice to obtain such fused image $\ubf$ by the root of RSS $(\sum_{i=1}^{c} |\ubf_i|^2)^{1/2}$ in the pMRI community, however, such simple approach removes the important phase information from images $\ubf_i$ and may not yield $\ubf$ with desired uniform contrast due to the diverse presence of coil sensitivities.
	
	To address the issue of proper regularization, we propose to learn a nonlinear operator $\J$ that combines $\{\ubf_i\}$ into the image $\vbf = \J (\ubf) \in \mathbb{C}^{m\times n}$, where $\ubf = (\ubf_1,\dots, \ubf_{c}) \in \mathbb{C}^{m \times n \times c}$ represents the channel-wise multi-coil MRI data that consists of $ \ubf_i$ for $ i=1,\cdots,c$.
	Then we apply a suitable $R$ to the image $\vbf$. %and 
	%We apply $\| \N (\ubf) \|_{2,1}$ by leveraging the robust sparse selection property of $2,1$-norm. 
	% 
	We also introduce a k-space $R_f$ on $\Fbf \ubf_i$ to take advantage of Fourier information and enhance the model performance.
	
	%Combined with the data fidelity term above, we propose the pMRI reconstruction model as follows:
	
	We denote $\fbf = (\fbf_1,\dots,\fbf_{c}) \in \mathbb{C}^{p \times c}$ as the partial k-space measurements at $c$ sensor coils. Suppose that we are given $N$ data pairs $\{(\fbf^{(j)}, \ubf^{*(j)}) \}_{j=1}^N$ for training the network, where $\ubf^{*(j)}$ is the ground truth multi-coil MR data with index $j \in \{ 1, \cdots, N\}$. Let $\Theta$ represents the parameters that need to be learned from network by minimizing the loss function $\ell$. We formulate the network training as a bilevel optimization problem, where the lower level is to update $ \ubf$ with fixed trainable parameters $\Theta$ and upper level is to update $\Theta$ that learned from the training data by minimizing loss function $\ell$.
	\begin{subequations}\label{learnable_pmri}
		\begin{align}
			& \min_{\Theta} \ \ \frac{1}{N}\sum^N_{j=1}  \ell(  \ubf^{(j)}_{\Theta}) ,\ \ \ \label{learnable_pmri_upper} \\
			& \mathrm{s.t.} \ \  \ubf_{\Theta}^{(j)} = \argmin_{\ubf^{(j)}}  \phi_{\Theta}( \ubf^{(j)}), \ \  \label{learnable_pmri_lower}  
		\end{align}
	\end{subequations}
	with $\phi_{\Theta}$ defined as below:
	\begin{equation}\label{eq:m_chp3}
		\phi(\ubf):= \frac{1}{2} \sum\limits^{c}_{i=1} \| \textbf{PF} \ubf_i - \fbf_i \|^2  +  R(\J(\ubf)) + R_f(\Fbf \ubf_i).
	\end{equation}
	The objective function $\phi$ is the variational model for pMRI reconstruction, in which $\Theta$ is the set collects all the parameters that learned from the regularizers $R \circ \J$ and $R_f \circ \Fbf$, so $\phi$ is depending on $\Theta$.
	Problem \ref{learnable_pmri} is formulated in the scenario of reconstructing the multi-coil MRI data.  The final reconstruction result $\ubf_{\Theta}$ is the forward network output that dependent on network parameters $\Theta$.   %%\  \sum\nolimits^{c}_{i=1} \frac{1}{2}  \| \Pbf \Fbf \ubf_i - \fbf_i\|^2 + R(\ubf, \Theta) The primary parameter and optimal reconstruction solution is the minimizer of the network loss function: $$ \min_{ \ubf, \Theta} \ell( \ubf(\Theta)), $$ The loss $ \ell$ measures the discrepancy between reconstructed multi-coil images $\ubf(\Theta) $ and the ground truth $ \ubf^*$.
	
	The deep learning approach for pMRI reconstruction in lower level problem \eqref{learnable_pmri_lower} can be cast and formulated as a discrete-time optimal control system. 
	We use one data sample $(\fbf, \ubf^*)$ and omit the average and data indexes for notation simplicity. The architecture of deep unrolling method follows the iterations of optimization algorithms and solve for the minimizer of the following problem: %, but the proposed framework can be easily extended to a training data set consists of many instances of such data samples and implemented accordingly for parallel computation.}
	\begin{subequations}\label{eq:bilevel_chp3}
		\begin{align}
			\min_{\Theta}\ \ \  & \ell(\ubf_{\Theta}), \label{eq:bilevelupper} \\ 
			\mathrm{s.t.}\ \ \ & \ubf{(t)} =   \gbf(\ubf{(t-1)},\theta(t)),\ t= 1,\cdots, T, \label{eq:bilevelconstraint}\\
			& \ubf{(0)} =  \gbf_0(\fbf,\theta(0)), \label{eq:init}
		\end{align}
	\end{subequations}
	where $\Theta = (\theta(0),\cdots, \theta(T))^{\intercal}$ is the collection of control variables $\theta(t)$ at all time steps (phases) respectively.  In \eqref{eq:bilevelupper},  $\ubf_{\Theta} = \ubf(T)$, which is the output image from the last $T$-th phase of the entire network so that we want to find optimal $\ubf(T)$ that close to  $\argmin_{\ubf}  \phi_{\Theta}(\ubf)$. Equations \eqref{eq:bilevelconstraint}, \eqref{eq:init} are inspired by deep unrolling algorithm for solving the lower level constraint \eqref{learnable_pmri_lower}. Given $\theta(t)$, $\ubf(t)$ is the state of updated reconstruction multi-coil data from $t$-th phase for each $t=0,\cdots, T $. $\gbf$ is a multi-phase unrolling network inspired by the proximal gradient algorithm, and the output of $\gbf(\cdot) \in \C^{m \times n \times c}$  is the updated multi-coil MRI data from each phase. The network $\gbf$ is the intermediate mapping from $\ubf{(t)}$ to $\ubf{(t+1)}$ for $t=0,\cdots, T-1 $, whose structure is explained in Section \ref{MLM} for minimizing the variational model \eqref{eq:m_chp3}. The network $\gbf_0$ with initial control parameter $\theta(0)$ maps the partial k-space measurement $ \fbf$ to an initial reconstruction $ \ubf{(0)}$ as the input of this optimal control system. 
	
	To summarize in brief, we solve for the minimizer (reconstruction result) of the lower level problem \eqref{learnable_pmri_lower} in a discrete-time optimal control framework \eqref{eq:bilevel}. The dynamic system \eqref{eq:bilevelconstraint}, \eqref{eq:init} as the constrain of  \eqref{eq:bilevelupper} is modeled as the optimal control system of the variational model \eqref{eq:m_chp3}.
	
	The loss function $\ell(\ubf{(T)})$ measures the discrepancy between the final state $\ubf{(T)}$ and the reference image $\ubf^*$ obtained using full k-space data in the training data set. We set the loss function in  \eqref{learnable_pmri_upper} and \eqref{eq:bilevelupper} for the proposed method as:
	\begin{equation}\label{eq:loss_chp3}
		\begin{aligned}
			& \ell(\ubf_{\Theta}) = \ell(\ubf{(T)}) = \sum\nolimits^{c}_{i=1} \gamma \| \ubf_i{(T)} - \ubf^*_i\| +  \\ 
			& \| |\J(\Bar{\ubf}{(T)})|- \text{RSS}(\ubf^*)\| + \eta \| \text{RSS}(\Bar{\ubf}{(T)}) - \text{RSS}(\ubf^*)\|,  
		\end{aligned}
	\end{equation}
	where $\text{RSS}(\ubf^*)= (\sum_{i=1}^{c} |\ubf^*_i|^2)^{1/2} \in\mathbb{R}^{m \times n}$ is the pointwise root of sum of squares across the $c$ channels of $\ubf^*$, $|\cdot|$ is the pointwise modulus, and $\gamma, \eta > 0$ are prescribed weight parameters. %The last two terms of \eqref{eq:loss_chp3} support  $\J(\Bar{\ubf})$, $ \text{RSS}(\bar{\ubf})$ and $ \text{RSS}(\ubf^*)$ to be close.
	%We want to search for optimal $ \Theta^*$ such that $ \ubf{(T)} =  \argmin_{ \textbf{U}, \Theta}\  \ell(\ubf{(T)})$. 
	%For each $ t$, $\ubf{(t)}$ and $\theta(t) $ are constrained by \eqref{eq:bilevelconstraint} and \eqref{eq:init}.  The third term is for the purpose of narrowing the gap between $\Bar{\ubf}$ and $\ubf^*$.
	
	The motivation of applying learnable image space regularization $R \circ \J$ and k-space regularization $R_f \circ \Fbf$ in model \eqref{eq:bilevel_chp3} is explained in the following: (i) Image domain network recovers the high spatial resolution, but may not suppress some artifacts. Frequency domain network is more suitable to remove high-frequency artifacts. (ii) Image domain and k-space information are equivalent due to the global linear transformation, but adding nonlinear activations with CNNs can feasibly improve the efficacy of network learning and boost the reconstruction performance.
	We parametrize the combination operator $\J$ by CNNs since the partial k-space data were scanned by multiple coil arrays, and introduce cross-correlation among channels of coil-images which could be compatible for CNN structure.
	
	\subsection{Design the Regularization in the Variational Model}\label{subsec:g}

	Denote  function $ h(\ubf) $ as the data fidelity term, one of the famous traditional method for solving problem $ \min_{\ubf} h(\ubf) + R(\ubf) $ is the
	\emph{proximal gradient algorithm} \cite{parikh2014proximal}:
	\begin{subequations}\label{ista_alg}
		\begin{align}
			\bbf^l &= \ubf^l - \alpha^l \nabla h(\ubf^l),\\
			\ubf^{l+1} &= \prox_{\alpha^l R} (\bbf^l),
		\end{align}
	\end{subequations}
	where $\alpha^l >0$ is the step size, the proximity operator $\prox_{\alpha R}$ defined below: 
	\begin{equation}\label{prox}
		\prox_{\alpha R}(\bbf) := \argmin_{\xbf} \left\{\frac{1}{2 \alpha} \| \xbf - \bbf \|^2 + R(\xbf) \right\}, \end{equation} 
	where $[\xbf]_i = \xbf_i \in \mathbb{C}^{m \times n}$ for any $\xbf \in \mathbb{C}^{m \times n \times c}$.
	The proposed network structure is inspired by \eqref{ista_alg} for solving \eqref{learnable_pmri_lower} which can be cast as an iterative procedure with $T$ phases in the discrete dynamic system \eqref{eq:bilevelconstraint} and \eqref{eq:init}.
	We parametrize the $t$-th phase consists of three steps:
	\begin{small}
		\begin{subequations}\label{eq:newmodel} 
			\begin{align}
				\bbf_i{(t)}  = \ubf_i{(t)} - \rho_t  \Fbf^{H}  \Pbf^{\top}  (\Pbf \Fbf \ubf_i{(t)} - \fbf_i),  & \quad i = 1,\cdots, c, \label{eq:newbi}  \\
				\bar{\ubf}_i{(t)}   = [\prox_{\rho_t R(\J(\cdot)) } (\bbf{(t)})]_i, & \quad i = 1,\cdots, c, \label{eq:newubar} \\
				{\ubf}_i{(t+1)}  = [\prox_{ \rho_t R_f(\Fbf(\cdot))}  (\bar{\ubf}{(t)})]_i, & \quad i = 1,\cdots, c \label{eq:newui},
			\end{align}
		\end{subequations}
	\end{small}
	for $t=0,\cdots, T-1 $ and $\bbf{(t)} = (\bbf_1{(t)},\dots,\bbf_{c}{(t)}) \in \mathbb{C}^{m \times n \times c}$,  $\rho_t>0$ is the step size. $F$ denotes the normalized discrete Fourier transform and $F^{H}$ the complex conjugate transpose (i.e., Hermitian transpose) of $F$, here $F^{H}$ is the inverse discrete Fourier transform. 
	
	The step \eqref{eq:newbi} computes $ \bbf{(t)} $ by applying the gradient decent algorithm to minimize the data fidelity term in \eqref{eq:m_chp3} which is straightforward to compute.
	The first proximal update step \eqref{eq:newubar} equipped with the joint regularizer $ R(\J(\cdot))$ and upgrades its input $ \bbf_i{(t)} $ to a multi-coil image $\bar{\ubf}_i{(t)}$.
	Ideally, the regularization $R\circ \J$ can be parameterized as a deep neural network whose parameters can be adaptively learned from data, however, in such case $\prox_{\rho_t R(\J(\cdot)) }$ does not have closed form and can be difficult to compute.
	As an alternative, the proximity operator $\prox_{\rho_t R(\J(\cdot)) }$ can be directly parametrized as a learnable denoiser and solve \eqref{eq:newubar} in each iteration. 
	For the similar reason,  the proximity operator $\prox_{R_f(\Fbf(\cdot))}$  in \eqref{eq:newui} can also be
	parametrized as CNNs, which further improves the accuracy of the k-space measurement.
	In both \eqref{eq:newubar} and \eqref{eq:newui}, the regularizers $R \circ \J$ and $R_f \circ \Fbf $ extract complex features through neural networks and the proximity points can be learned in denoising network via ResNet \cite{7780459} structure.
	
	We frame step \eqref{eq:newubar} incorporates joint reconstruction to update coil-images via ResNet \cite{7780459}: $\Bar{\ubf}{(t)} =  \bbf{(t)} + \M( \bbf{(t)})$,  where $ \M$ represents a multi-layer CNN by executing approximation to the proximal mapping in the image space.
	We propose to first learn a nonlinear operator $\J$ that combines $\{ \bbf_i\}$ into the image $ \zbf = \J (\bbf_1,\dots,\bbf_{c}) \in \mathbb{C}^{m\times n}$ with homogeneous contrast. Then apply a nonlinear operator $\G$ on $\zbf$ with $ \G(\zbf) \in \mathbb{C}^{m\times n\times N_f}$ to extract a $N_f$-dimensional features. The nonlinear operator $\J$ contains four convolutions with an activation function in between, each convolution obtains kernel size $ 3\times3$. The nonlinear operator $\G$ consists of three convolutions with filter size $ 9\times 9$. For the sake of improving the capacity of the network, $ \tilde{\J}$ and $ \tilde{\G}$ are employed as adjoint operators of $ \J$ and $ \G$ respectively with symmetric structure and parameters are trained separately. $\G \circ \J $ was designed in the sense of playing a role as an encoder and $\tilde{\J} \circ \tilde{\G}$ as a decoder. Therefore, image domain network can be parametrized as compositions of four CNN operators: $\M = \tilde{\J} \circ \tilde{\G} \circ \G \circ \J $ with output $ \M( \bbf{(t)})\in \mathbb{C}^{m \times n\times c}$. 
	This step \eqref{eq:newubar} outputs the multi-coil data $ \Bar{\ubf}  = (\Bar{\ubf}_1, \dots, \Bar{\ubf}_{c})  \in \mathbb{C}^{m \times n \times c}$ from image domain network and we apply the combination operator $\J$ on $\Bar{\ubf}$ to obtain a full FOV MR single-channel image $ \vbf = \J(\Bar{\ubf})  \in \mathbb{C}^{m\times n}$ that we desired for reconstruction.
	
	Step \eqref{eq:newui} leverages k-space information and further suppresses the high-frequency artifacts. The output  $\Bar{\ubf}{(t)} $ from \eqref{eq:newubar} is passed to a k-space domain network by a ResNet structure
	$\ubf{(t+1)} = \ \Bar{\ubf}{(t)} +  \Fbf^{H}  \K \big( \Fbf(\Bar{\ubf}{(t)} ) \big)$, where $ \Fbf^{H}  \K  \Fbf$ is a CNN operator to refine and further improve the accuracy of the k-space data from each coil. The CNN  $\K$  consists four convolutions, the last convolution kernel numbers meets the channel number of coil images.
	
	Therefore, \eqref{eq:newmodel} is proceed in the following scheme:
	\begin{small}
		\begin{subequations}  \label{eq:scheme}
			\begin{align}
				\bbf_i{(t)} =  \ \ubf_i{(t)} - \rho_t \Fbf^{H}  \Pbf^{\top} (\Pbf \Fbf \ubf_i{(t)} - \fbf_i),& \quad i = 1,\cdots, c, \label{eq:schemeb} \\
				\bar{\ubf}_i{(t)}  =  \ \bbf_i{(t)} + \M( \bbf_i{(t)}), & \quad i = 1,\cdots, c, \label{eq:schemeu-bar}   \\
				\ubf_i{(t+1)} = \ \Bar{\ubf}_i{(t)} +   \Fbf^{H}  \K \big( \Fbf(\Bar{\ubf}_i{(t)} ) \big), & \quad i = 1,\cdots, c, \label{eq:schemeu} 
			\end{align}
		\end{subequations}
	\end{small}
	for $ t = 0,\cdots,T-1$.
	Now we can derive the function $\gbf$ described in \eqref{eq:bilevel_chp3} by combining \eqref{eq:scheme}:
	\begin{subequations}\label{eq:g}
		\begin{align}
			\ubf_i{(t+1)}&  =  ( \N - \rho_t \Fbf^{H}  \Pbf^{\top} \Pbf \Fbf -  \N( \rho_t \Fbf^{H}  \Pbf^{\top} \Pbf \Fbf) ) \ubf_i{(t)} \notag\\
			&  + \ubf_i{(t)} + \rho_t \Fbf^{H}  \Pbf^{\top} \fbf_i-\N(\rho_t \Fbf^{H}  \Pbf^{\top} \fbf_i) \label{eq:funcg}\\
			& = \gbf(\ubf_i{(t)}), \label{eq:gu}
		\end{align}
	\end{subequations}
	where $ \N =\M + \Fbf^{H} \K \Fbf + \Fbf^{H} \K \Fbf  \M$, and the function $\gbf$ maps $\ubf_i{(t)}$ to $\ubf_i{(t+1)}$ for $t=0,\cdots, T-1 $ defined in \eqref{eq:funcg}.
	
	Our proposed reconstruction network is composed of a prescribed $T$ phases, the initial input of the network is designed as $\ubf{(0)} =  \gbf_0(\fbf, \theta(0)):= \Fbf^{H}(\fbf+ \K_0(\fbf))$ where $\K_0$ is a CNN operator applied to $\fbf$ in residual learning to interpolate the missing data and generate a pseudo full k-space. %Therefore, the function $\gbf_0$ in \eqref{eq:bilevel_chp3} is defined as $ \gbf_0(\fbf):= \Fbf^{H}(\fbf+ \K(\fbf))$.
	With the chosen initial $\{ \ubf_i{(0)}\}$ and partial k-space data $\{\fbf_i\}$ as input, the network performs the update \eqref{eq:scheme} in the $t$-th phase for $t=0,\cdots, T-1$  and finally the entire network reconstructs coil-images $\ubf{(T)}$ and $ \J (\Bar{\ubf}{(T)}) $, which is the final single-body image reconstructed as a by-product (complex-valued). 
	
	Fig. \ref{fig:all-phases} displays the flowchart of our entire network procedure. The initial reconstruction suppresses artifacts caused by undersampling.  We display the flowchart of these CNN structures of each phase in Fig. \ref{fig:t+1phase}. %
	We apply complex-valued convolutions where multiplications between complex numbers are performed and use componentwise complex-valued activation function $\C$ReLU($a+ib$) = ReLU($a$) + $i$ReLU($b$) as suggested in \cite{cole2020analysis}. 
	The Landweber update step \eqref{eq:schemeb} plays a role in increasing communication between image space and k-space. The learnable step size $\rho_t$ controls the speed and stability of the convergence. In the image space denoising step \eqref{eq:schemeu-bar}, the operator $\J$ extracts feature across all the channels so that spatial resolution is improved and tissue details are recovered in the channel-combined image. The CNN $\M$ carrying channel-integration $\J$ proceed $T$ times with shared weights in every two phases therefore the spatial features across channels are learned in an efficient way. However, the oscillatory artifacts could be misinterpreted as real features, which might be sharpened.  In the k-space denoising update step \eqref{eq:schemeu}, the k-space network $\K$ is compatible with low-frequency information, so it releases the high-frequency artifacts and recovers the structure of the image \cite{doi:10.1002/mrm.27201, wang2020ikwi}. Therefore 
	iterating the networks $\M$ and $\K$ in different domains with their individual effects,  the performance compensates and the shortcomings of both networks offset each other. We evaluate the effect of hybrid domain reconstruction in ablation studies. 
	Furthermore, the prescribed denoising networks in \eqref{eq:schemeu-bar} and \eqref{eq:schemeu} refine and update coil-images in each iteration. This iterative procedure triggers the reconstruction quality of $ \J (\Bar{\ubf}{(t)}) $ get successively enhancement. 
	
	\begin{figure}[t]
		\centering
		\includegraphics[width=1\linewidth]{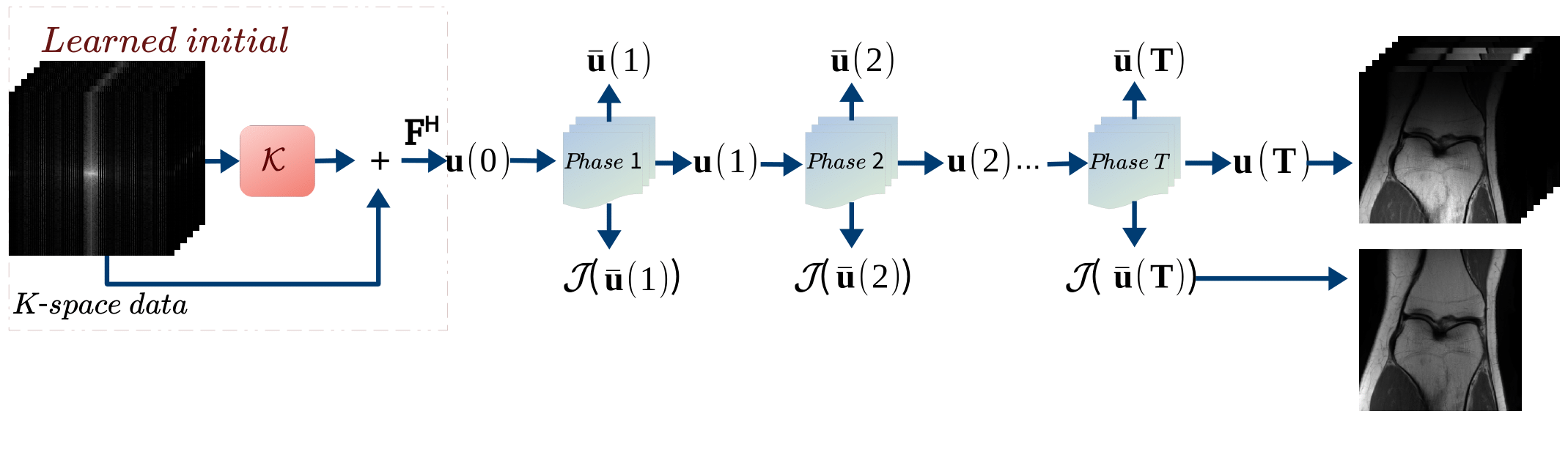}
		\caption{The proposed framework paradigm for  all  phases.}
		\label{fig:all-phases}
	\end{figure}
	\begin{figure}[t]
		\centering
		\includegraphics[width=1\linewidth]{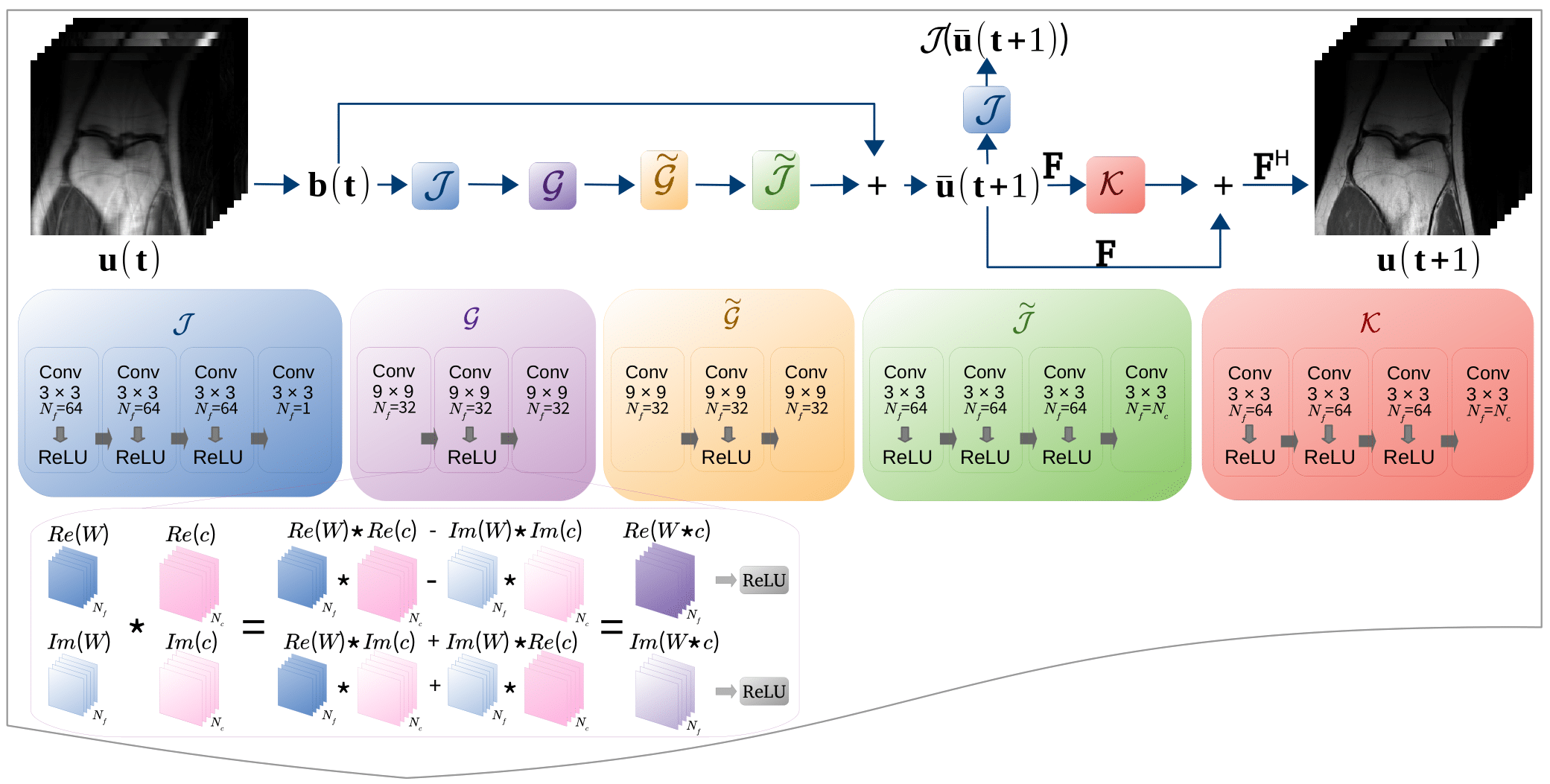}
		\caption{The proposed framework paradigm at $ t+1$-th phase.
			$W = Re(W) + i  Im(W)$ representing complex convolution filter kernels.}
		\label{fig:t+1phase}
	\end{figure}

	\subsection{Network Training from The View of The Method of Lagrangian Multipliers (MLM)}\label{MLM}
	The network parameters to be solved from \eqref{eq:bilevel_chp3} are $\Theta = \{\theta(t): t = 0,\cdots,T\}$, where $\theta(t)=\{\rho_{t},\J_{t},\G_{t},\tilde{ \G}_{t}, \tilde{ \J}_{t}, \K_{t} \}$ for $t =1, \cdots, T$ and $\theta(0)=\K_{0} $.
	%The second term of \eqref{eq:loss_chp3} is imposed if $\text{RSS}(\ubf^*)$ is expected to approximate the magnitude of the single-body image accurately.
	
	The control problem \eqref{eq:bilevel_chp3} can be solved by using MLM. The corresponding Lagrangian function is
	\begin{equation}
		\label{eq:lag}
		\begin{aligned}
			L(\textbf{U},\Theta;\Lambda) = & \ell(\ubf{(T)}) +  \langle \lambda(0) , \ubf{(0)} - \gbf_0 \big(\fbf , \theta(0) \big) \rangle \\
			&+ \sum\nolimits^{T}_{t=1} \langle \lambda(t), \ubf{(t)} - \gbf\big(\ubf{(t-1)},  \theta(t) \big) \rangle,
		\end{aligned}
	\end{equation}
	where $\textbf{U}=(\ubf{(0)}, \cdots, \ubf{(T)})^{\intercal}$ is the collection of all the states $\ubf{(t)}$,  $\Lambda = (\lambda(0), \cdots, \lambda(T))^{\intercal}$ are Lagrangian multipliers of \eqref{eq:bilevel_chp3}. 
	The algorithm proceed in the iterative scheme to update $\Theta^k$, for each training epoch  $k=0, \cdots, K$, $\Theta^{k}=(  \theta^{k}(0),\cdots, \theta^{k}(T))^{\intercal}$, $\textbf{U}^k =( \ubf^{k}{(0)}, \cdots,  \ubf^{k}{(T)})^{\intercal}$, and $\Lambda^k = (\lambda^k(0), \cdots,  \lambda^k(T))^{\intercal}$ . 
	
	If $(\textbf{U}^*, \Theta^*;\Lambda^*)$ minimizes Lagrange function \eqref{eq:lag}, by the first order optimality condition, we have
	\begin{subnumcases}{}
		\partial_{\textbf{U}}  L (\textbf{U}^*, \Theta^*; \Lambda^*) = 0 \label{dU}\\
		\partial_{\Theta} L (\textbf{U}^*, \Theta^*; \Lambda^*) = 0 \label{dTheta}\\
		\partial_{\Lambda}  L (\textbf{U}^*, \Theta^*; \Lambda^*) = 0  \label{dLambda}
	\end{subnumcases}
	\textbullet Fix $\Theta = \Theta^k$, define $(\textbf{U}^k,\Lambda^k):= \argmin_{\textbf{U},\Lambda}  L(\textbf{U},\Theta^k;\Lambda) $, then by the first order optimality condition for minimizing $L$ w.r.t $ \lambda(t)$ for $ t=0,\cdots,T$,  $ (\textbf{U}^k,\Lambda^k)$  should satisfy 
	\begin{equation}
		\begin{aligned}
			&\partial_{\lambda(0)}\langle \lambda(0), \ubf{(0)} - \gbf_0(\fbf , \theta^k(0)) \rangle \Big|_{(\ubf^k{(0)}, \lambda^k(0))} =0  \\
			&\implies \ubf^k{(0)} = \gbf_0(\fbf , \theta^k(0)) \label{eq:ut0}\\
		\end{aligned}
	\end{equation}
	\begin{equation}
		\begin{aligned}
			&\partial_{\lambda(t)}\langle \lambda(t), \ubf{(t)} - \gbf(\ubf{(t-1)},\theta^k(t) ) \rangle \Big|_{(\ubf^k{(t)}, \lambda^k(t))} =0 \\
			&\implies \ubf^k{(t)} = \gbf(\ubf^k{(t-1)}, \theta^k(t)), \ t = 1,\cdots, T. \label{eq:ut1} 
		\end{aligned}
	\end{equation}

	Then by the first order optimality condition for minimizing $L$ w.r.t  $\ubf{(t)}$, for $t=T$, we get
	\begin{equation}
		\begin{aligned}
			&\partial_{\ubf{(T)}} [\ell(\ubf{(T)}) + \langle \lambda(t), \ubf{(T)} \rangle ] \Big|_{(\ubf^k{(t)}, \lambda^k(t))} = 0\\ 
			&\implies \lambda^k(T) = - \partial_{\ubf{(T)}}  \ell(\ubf^k{(T)}), \label{eq:lambdaT} 
		\end{aligned}
	\end{equation}
	for $t = 0,\cdots,T-1$:
	\begin{small}     
		\begin{eqnarray}     
			&\partial_{\ubf{(t)}}  [\langle \lambda(t), \ubf{(t)} \rangle- \langle \lambda(t+1), \gbf(\ubf{(t)}, \theta^k{(t+1)}) \rangle] \Big|_{(\ubf^k{(t)}, \lambda^k(t))}  = 0 \notag\\ 
			&\implies \lambda^k(t) = \langle \lambda^k(t+1), \partial_{\ubf{(t)}}  \gbf(\ubf^k{(t)}, \theta^k{(t+1)}) \rangle,  \label{eq:lambdat}        
		\end{eqnarray}              
	\end{small}    
	\textbullet Fix  $(\textbf{U}^k,\Lambda^k) $ for updating $\Theta$,  we compute the gradient $\partial_{\Theta}L(\textbf{U}^k,\Theta; \Lambda^k)$:
	\begin{subequations}
		\label{eq:dtheta}
		\begin{align}
			&\partial_{\theta(t)}  L(\textbf{U}^k,\theta(t);\Lambda^k)  = \partial_{\theta(t)} [ -\langle \lambda^k(t), \gbf(\ubf^k{(t-1)}, \theta(t)) \rangle ] \notag\\
			& =  -\langle \lambda^k(t), \partial_{\theta(t)} \ \gbf(\ubf^k{(t-1)}, \theta(t)) \rangle, t = 1,\cdots, T, \label{eq:dtheta1}\\
			&\partial_{\theta(0)}  L(\textbf{U}^k,\theta(0);\Lambda^k) =  -\langle \lambda^k(0), \partial_{\theta(0)}  \gbf_0(\fbf, \theta(0)) \rangle.\label{eq:dtheta0}
		\end{align}
	\end{subequations} 
	\begin{theorem}\label{thm}
		$\partial_{\Theta}L(\textbf{U}^k,\Theta;\Lambda^k) = \partial_{\Theta} \ell(\ubf^k{(T)}(\Theta) )$.
	\end{theorem}
	\begin{proof}
		First we show the following holds for $t=0,\cdots, T$:
		\begin{equation}
			\label{eq:everyt}
			\lambda^k(t)  = - \partial_{\ubf{(t)}}\ell(\ubf^k{(T)}) 
		\end{equation}
		
		From \eqref{eq:lambdaT}, \eqref{eq:everyt} is true when $ t=T$.
		Suppose \eqref{eq:everyt} true for $ t=\tau \in \{1,\cdots ,T\} $.
		From \eqref{eq:lambdat}, we have
		\begin{subequations}
			\begin{align}
				\lambda^k(\tau-1) & = \langle \lambda^k(\tau), \partial_{\ubf{(\tau-1)}} \ \gbf(\ubf^k{(\tau-1)}, \theta^k(\tau)) \rangle  \notag\\
				& = \langle - \partial_{\ubf{(\tau)}} \ell(\ubf^k{(T)}), \partial_{\ubf{(\tau-1)}} \ubf^k{(\tau)}  \rangle  \\
				& = - \partial_{\ubf{(\tau-1)}}\ell(\ubf^k{(T)})\label{eq:lambda_T-1}
			\end{align}
		\end{subequations}
		Thus, \eqref{eq:everyt} holds for $ t= \tau-1$. By the principle of induction, \eqref{eq:everyt} is ture for all $t=0, \cdots, T$.
		
		Hence, \eqref{eq:dtheta1} reduces to
		\begin{subequations}
			\label{eq:dL}
			\begin{align}
				\partial_{\theta(t)} & L(\textbf{U}^k,\theta(t);\Lambda^k) =  -\langle \lambda^k(t), \partial_{\theta(t)} \ \gbf(\ubf^k{(t-1)}, \theta(t) ) \rangle \notag \\
				& =  \langle \partial_{\ubf{(t)} } \ell(\ubf^k{(T)})  , \partial_{\theta(t)} \ubf^k{(t)} \rangle  \\
				& = \partial_{\theta(t)} \ell ( \ubf^k{(T)} ), \ t = 1,\cdots, T. 
			\end{align}
		\end{subequations} 
		Also \eqref{eq:dtheta0} gives $\partial_{\theta(0)} L(\textbf{U}^k,\theta(0);\Lambda^k) 
		= \partial_{\theta(0)} \ell ( \ubf^k{(T)} )$.
		Therefore, we derive 
		$ \partial_{\Theta} L(\textbf{U}^k,\Theta;\Lambda^k) = \partial_{\Theta} \ell ( \ubf^k{(T)}(\Theta))$.
	\end{proof}
	This theorem further implies that:
	Since the gradients are the same,  applying SGD Algorithms or its variance such as Adam \cite{kingma2014adam} to minimize loss function $ \ell$ is equivalent to perform the same algorithm on $L$. Network training algorithm using MLM can be summarized in  Algorithm \ref{alg:1}. % i.e. This theorem shows that using  Adam  to minimize loss function has equivalent performance as using  Algorithm \ref{alg:1} to minimize $L$, which ensuring  the  local  convergence  of  the training  algorithm. 
	
	\begin{algorithm}
		\caption{Network Training by MLM}\label{alg:1}
		\textbf{Hyperparameter:} $K$ (\#Iterations)\\
		\textbf{Initialize:} Initial guess $\theta^{0}(t) \in \Theta^0, t =0,\cdots, T$\\
		\For{$k = 0$ \KwTo $K-1$}{
			Set $ \ubf^k{(0)} = \gbf_0(\fbf , \theta(0)) $\\
			\For{$t=1$ \KwTo $T$}{
				$ \ubf^k{(t)} = \gbf(\ubf^k{(t-1)}, \theta^{k}(t))$ }
			Set $ \lambda^k(T) =   - \partial_{\ubf}  \ell(\ubf^k{(T)})$ \\
			\For{$t = T-1$ \KwTo  $0$ }{$\lambda^k(t) = \langle \lambda^k{(t+1)}, \partial_{\ubf}  \gbf(\ubf^k{(t)}, \theta^{k}(t+1)) \rangle$ }
			\For{$t = 0$ \KwTo  $T$ }{
				Set $\theta^{k+1}(t) = \argmin_{\theta(t)} L (\textbf{U}^k, \theta(t);\Lambda^k) $. }
		}
		\textbf{Output:}  $ \theta^{K}(t) , t=0,\cdots,T $.
	\end{algorithm}
	
	\section{Experimental Results}
	\label{sec:experiment_chp3}
	\subsection{Data Set}
	
	The data in our experiments was acquired by a 15-channel knee coil array with two pulse sequences: a proton density weighting with (FSPD) and without (PD) fat suppression in the coronal direction from \url{https://github.com/VLOGroup/mri-variationalnetwork}. We used a regular Cartesian sampling mask with 31.56\% sampling ratio as shown in the lower-right of Fig. \ref{PD_chp3}. Each of the two sequences data includes images of 20 patients, we select 27-28 central image slices from 19 patients, which amount to 526 images each of size $ 320 \times 320$ as the training dataset, and 15 central image slices are picked from one patient that is not included in the training data set as testing dataset. 
	
	\subsection{Implementation}
	\label{Implementation}
	We evaluate classical methods GRAPPA \cite{griswold2002generalized}, SPIRiT \cite{doi:10.1002/mrm.22428}, and deep learning methods VN \cite{doi:10.1002/mrm.26977}, De-AliasingNet \cite{10.1007/978-3-030-32248-9_4}, DeepcomplexMRI \cite{WANG2020136} and Adaptive-CS-Net \cite{pezzotti2020adaptive} over the 15 testing  Coronal  FSPD and Coronal  PD knee images with regular Cartesian sampling in terms of PSNR, structural similarity (SSIM) \cite{wang2004image} and relative error RMSE.
	The following equations are computations of SSIM, PSNR and RMSE between  reconstruction $\vbf = |\J(\Bar{\ubf})| $ and ground truth single-body image $\vbf^*$:
	\begin{equation}
		SSIM = \frac{(2\mu_{\vbf} \mu_{\vbf^*} + C_1)(2\sigma_{\vbf \vbf^*} + C_2)}{(\mu_{\vbf}^2 + \mu_{\vbf^*}^2+C_1)( \sigma_{\vbf}^2 + \sigma_{\vbf^*}^2 + C_2)},
	\end{equation}
	where $\mu_{\vbf}, \mu_{\vbf^*}$ are local means of pixel intensity, $\sigma_{\vbf}, \sigma_{\vbf^*}$ denote the standard deviation and $\sigma_{\vbf \vbf^*} $ 
	is covariance between $\vbf$ and $\vbf^* $, $ C_1 = (k_1 L)^2, C_2 = (k_2 L)^2$ are two constants that avoid denominator to be zero, and $ k_1 =0.01, k_2 =0.03$. $L$ is the largest pixel value of the magnitude of coil-images.
	\begin{equation}
		PSNR = 20\log_{10} \big(  \max(\abs{\vbf^*})  \big/ \frac{1}{N}\| \vbf^* - |\J(\Bar{\ubf})|  \|^2 \big),
	\end{equation}
	where $N$ is the total number of pixels in the magnitude of ground truth.
	\begin{equation}
		RMSE = \| \vbf^* - |\J(\Bar{\ubf})|  \| / \| \vbf^* \|.
	\end{equation}
	The relative error between the multi-coil reconstruction $\ubf$ and the ground truth $\ubf^*$ is defined as
	\begin{equation}
		RMSE = \sqrt{  \sum\nolimits^{c}_{i=1} \| \ubf^*_i - \ubf_i\|^2 / \sum\nolimits^{c}_{i=1} \| \ubf^*_i  \|^2  }.
	\end{equation}
	
	All the experiments are implemented and tested in TensorFlow \cite{abadi2016tensorflow} on a Windows workstation with Intel Core i9 CPU at 3.3GHz and an Nvidia GTX-1080Ti GPU with 11GB of graphics card memory. The parameters in proposed networks are using Xavier initialization \cite{glorot2010understanding} to initialize $\theta^{0}(t) , t=0,\cdots,T $. Indeed, solving line 13 in the Algorithm \ref{alg:1} using any stochastic gradient descent algorithm such as Adam \cite{kingma2014adam} has the same performance as minimizing loss function w.r.t $\theta(t)$ using the same algorithm to update $ \theta^{k+1}(t)$, because of the equivalence of the gradients: $ \partial_{\Theta} L(\textbf{U}^k,\Theta;\Lambda^k) = \partial_{\Theta} \ell ( \ubf^k{(T)}(\Theta))$, which is proved by Theorem \ref{thm}. TensorFlow provides optimized APIs for automatic differentiation which helps to build highly performant input pipelines and keeps high GPU utilization. In order to  implement a more stable gradient calculation, 
	we replace line 13 by minimizing $\ell$ with the Adam algorithm. The network was trained with total epochs $K=700$ to update $\Theta^k$. We apply exponentially decay learning rate 0.0001, $\beta_1 = 0.9, \beta_2 = 0.999, \epsilon =10^{-8}$ and mini-batch size of 2 is used in Adam optimizer.  The initial step size is set to $\rho_0 =1$ for both real and imaginary parts and we choose  $\gamma = 10^{-3}, \eta = 10^{-4}$ in \eqref{eq:loss_chp3}. The proposed network was implemented with $T=4$ and parameters trained in the network $\M$ are shared for every two phases.

	\subsection{Comparison with Existing Methods}\label{sub:compare}
	
	The average numerical performance  with standard deviations of the proposed method and several state-of-the-art methods are summarized in Table \ref{tab:result}. The proposed method achieves the best reconstruction quality in terms of PSNR/SSIM/RMSE. Our method is parameter efficient due to the network $\M$ share parameters in every two phases so that the entire network reduces more than 1/4 of learnable parameters. We listed the trained parameter numbers and inference time in table \ref{tab:result}.  Fig. \ref{fig:phase_images} and Table \ref{tab:phase_result} indicate that reconstruction performance get improved progressively as $t$ increases.
	\begin{figure}
		\centering
		\includegraphics[width=0.24\linewidth]{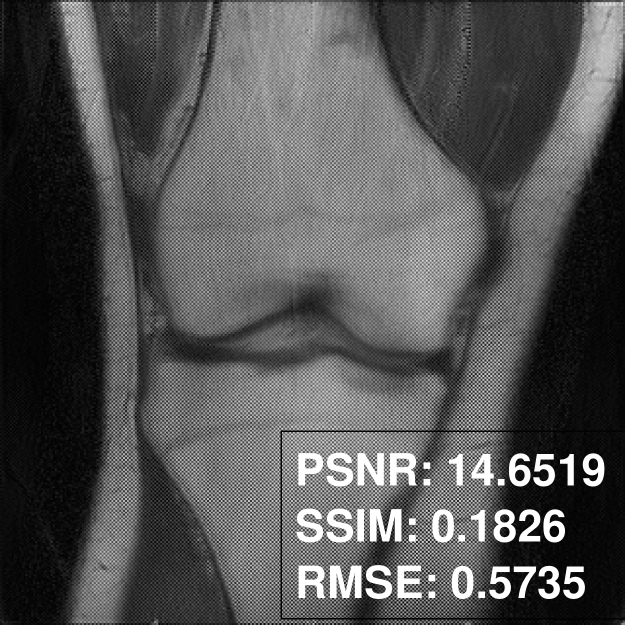}
		\includegraphics[width=0.24\linewidth]{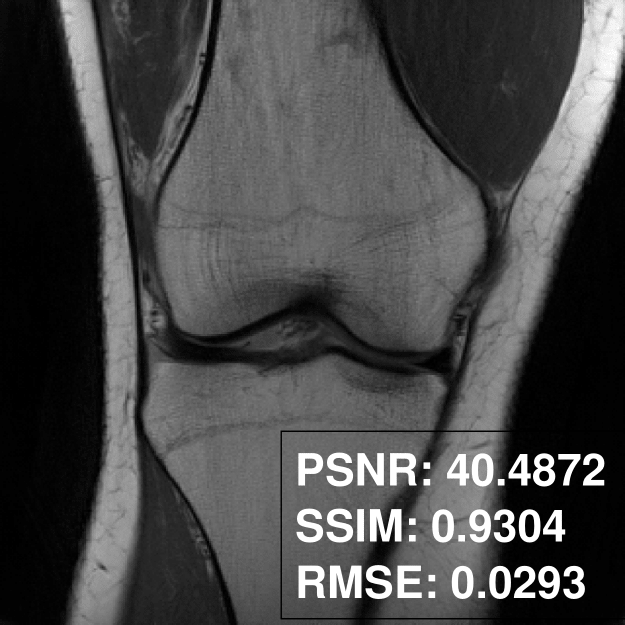}
		\includegraphics[width=0.24\linewidth]{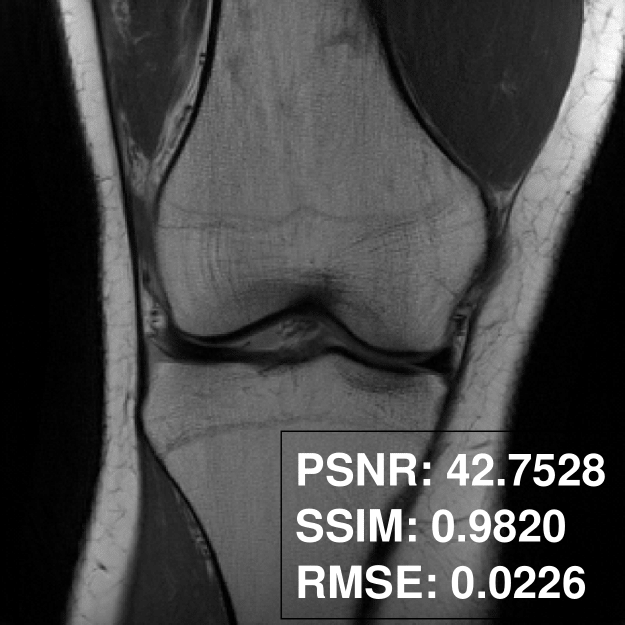}
		\includegraphics[width=0.24\linewidth]{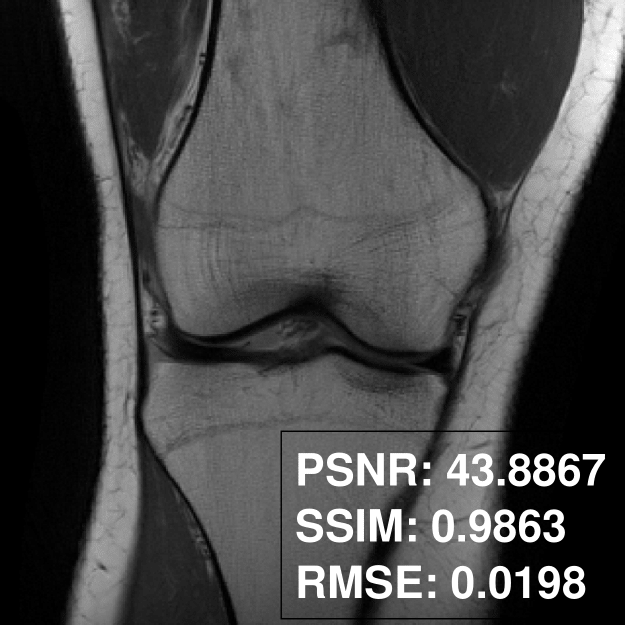}
		\caption{ Reconstructed channel-combined images $ \J( \bar{\ubf}{(t)})$ for $ t =1,2,3,4$ of proposed method.}
		\label{fig:phase_images}
	\end{figure} 
	\begin{table}
		\caption{Tested Average PSNR for FSPD data and PD data for each phase.}\label{tab:phase_result}
		\centering
	 \begin{tabular}{ccc}
				\toprule
				Phase number  &  PSNR of FSPD data  & PSNR of PD data  \\
				\midrule
				1 & 16.2007 & 15.5577\\
				2 & 33.7201 & 40.5228\\
				3 & 36.9523 & 42.6623\\
				4 & 40.7101 & 44.8120\\
				\bottomrule
		\end{tabular}
	\end{table}
	GRAPPA and SPIRiT adopted calibration kernel size $ 5\times 5$ and Tikhonov regularization in the calibration was set to be 0.01. Tikhonov regularization in the reconstruction was set as $10^{-3}$ for implementing SPIRiT, which took 30 iterations.
	In training and testing of VN and DeepcomplexMRI, the network and parameter settings are used as stated in their paper and code. We modified De-AliasingNet by increasing filter numbers to be 64 in 20 iterations to improve the performance for fair competition. In the CDF-Net, we set 128 channels as an initial layer with a depth of 4 in all the three Frequency-Informed U-Nets. We choose $ \sigma$ to be  Softplus activation function. We implemented Adaptive-CS-Net with a total of 15 reconstruction blocks including 2 of 16 filters with size $ 3\times3$, 3 of 32 filters with size $5\times5$, and 2 of 64 filters with size $5\times5$, the blocks are mixed with 2 and 3 scales, which was controlled as the largest tolerance of our GPU capacity. We input the data consistency as prior knowledge of training data to the network, neighboring slices with the center slice are used as input.  For fair competitions, CDF-Net and Adaptive-CS-Net have employed complex convolutions and activation $\C$ReLU.
	
	VN applied precalculated coil sensitivities, and output full-body image. De-AliasingNet, DeepcomplexMRI, CDF-Net, and Adaptive-CS-Net both output multi-coil images and do not require coil sensitivity maps. DeepcomplexMRI uses the adaptive coil combination method \cite{Walsh2000682}, the other networks all applied RSS.  CDF-Net and proposed network perform cross-domain reconstruction, other learning-based methods perform on image domain.
	Comparison between referenced pMRI  methods and proposed methods are shown in Fig.~\ref{PD_chp3} for PD images and FSPD images are in Fig.~\ref{FSPD}.  The top row shows reconstructed images and the referenced image, the second and third-row are corresponding zoomed-in ROIs of the red box area (draw on the rightmost ground truth image), the fourth row shows corresponding pointwise absolute error maps and the last row shows corresponding values and regular Cartesian sampling (31.56\% rate) mask.
	We observe the evidence that deep learning-based methods significantly outperform classical methods GRAPPA and SPIRiT in reconstruction accuracy. 
	In learning-based methods, the tested images from VN and De-AliasingNet are more blurry than other ones and lost sharp details in some complicated tissue, other methods display only slight differences in the detail.

	\begin{table*}
		\caption{Quantitative evaluations of the reconstructions on the Coronal FSPD \& PD data and reconstruction time for each of the referenced methods. Time (in seconds) refers to the testing time for each method.} \label{tab:result}
		\centering
		\resizebox{\linewidth}{17mm}{ 
			\begin{tabular}{lcccccccc}
				\toprule
				& & FSPD data & & & PD data & & \\
				Method   & PSNR                  & SSIM               & RMSE       & PSNR                & SSIM                & RMSE    &   Time & Parameters\\\midrule
				GRAPPA~\cite{griswold2002generalized}   & 24.9251$\pm$0.9341    & 0.4827$\pm$0.0344  & 0.2384$\pm$0.0175 & 30.4154$\pm$0.5924  & 0.7489$\pm$0.0207   & 0.0984$\pm$0.0030 &  280s  & N/A\\
				SPIRiT~\cite{doi:10.1002/mrm.22428}  & 28.3525$\pm$1.3314    & 0.6509$\pm$0.0300  & 0.1614$\pm$0.0203 & 32.0011$\pm$0.7920  & 0.7979$\pm$0.0306   & 0.0824$\pm$0.0082 &  43s & N/A\\
				VN~\cite{doi:10.1002/mrm.26977} & 30.2588$\pm$1.1790    & 0.7141$\pm$0.0483  & 0.1358$\pm$0.0152 & 37.8265$\pm$0.4000  & 0.9281$\pm$0.0114   & 0.0422$\pm$0.0036 & 0.16s & 0.13 M\\
				De-AliasingNet~\cite{10.1007/978-3-030-32248-9_4} & 36.1017$\pm$1.2981    & 0.8941$\pm$0.0269 & 0.0697$\pm$0.0119 & 41.2151$\pm$0.7872  & 0.9711$\pm$0.0033   & 0.0285$\pm$0.0015 & 0.92s & 0.34 M\\ 
				DeepcomplexMRI~\cite{WANG2020136} & 36.5706$\pm$1.0215   & 0.9008$\pm$0.0190 & 0.0654$\pm$0.0049 & 41.5756$\pm$0.6271  & 0.9679$\pm$0.0031   & 0.0274$\pm$0.0018 & 1.04s & 1.45 M\\
				CDF-Net~\cite{10.1007/978-3-030-59713-9_41} & 40.0405$\pm$1.5518 & 0.9520$\pm$0.0182 & 0.0443$\pm$0.0070 & 43.7943$\pm$1.9888 & 0.9879$\pm$0.0026 & 0.0215$\pm$0.0042 & 0.47s & 3.44 M\\
				Adaptive-CS-Net~\cite{pezzotti2020adaptive} & 40.1846$\pm$1.4780 & 0.9534$\pm$0.0175 & 0.0435$\pm$0.0065 & 44.1131$\pm$1.5596 & 0.9878$\pm$0.0026 & 0.0206$\pm$0.0027 & 1.57s & 5.59 M\\
				Proposed & \textbf{40.7101$\pm$1.5357} & \textbf{0.9619$\pm$0.0144} & \textbf{0.0408$\pm$0.0051} & \textbf{44.8120$\pm$1.3185} & \textbf{0.9886$\pm$0.0023} & \textbf{0.0189$\pm$0.0018} & 0.52s & 2.92 M\\
				\bottomrule
		\end{tabular}}
	\end{table*}
	\begin{figure*}
		\includegraphics[width=0.10\linewidth, angle=180]{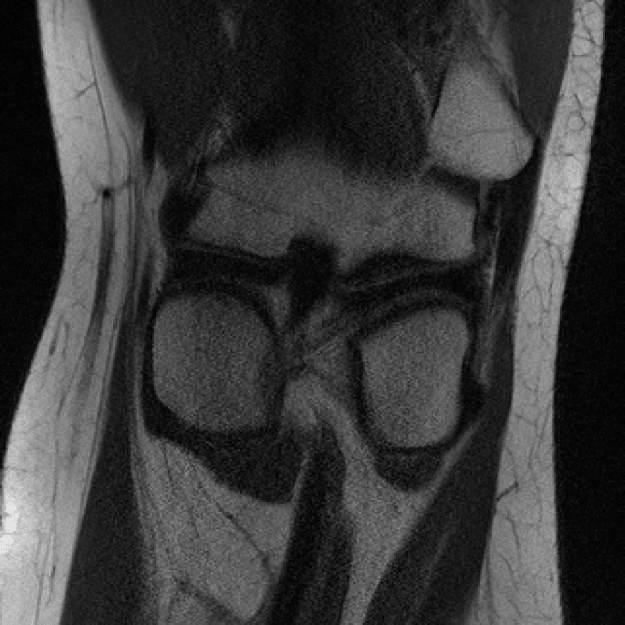}
		\includegraphics[width=0.10\linewidth, angle=180]{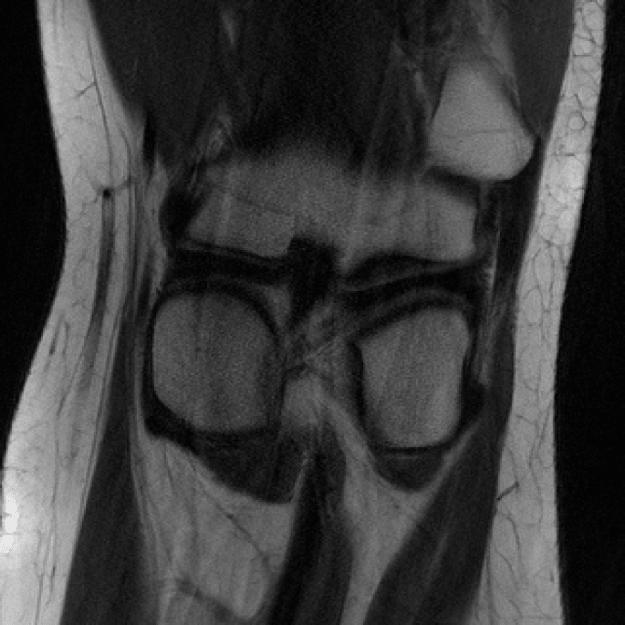}
		\includegraphics[width=0.10\linewidth, angle=180]{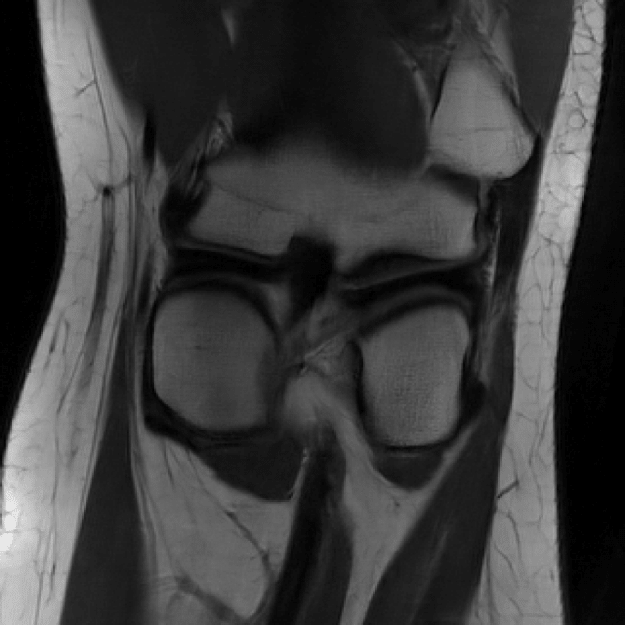}
		\includegraphics[width=0.10\linewidth, angle=180]{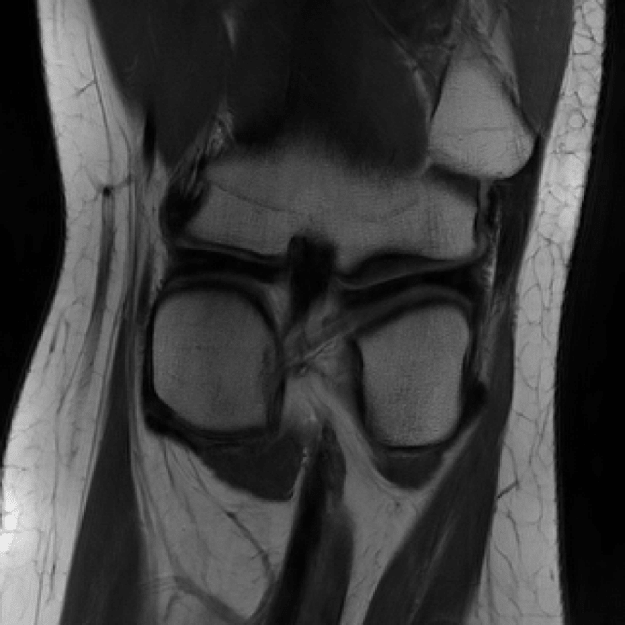}
		\includegraphics[width=0.10\linewidth, angle=180]{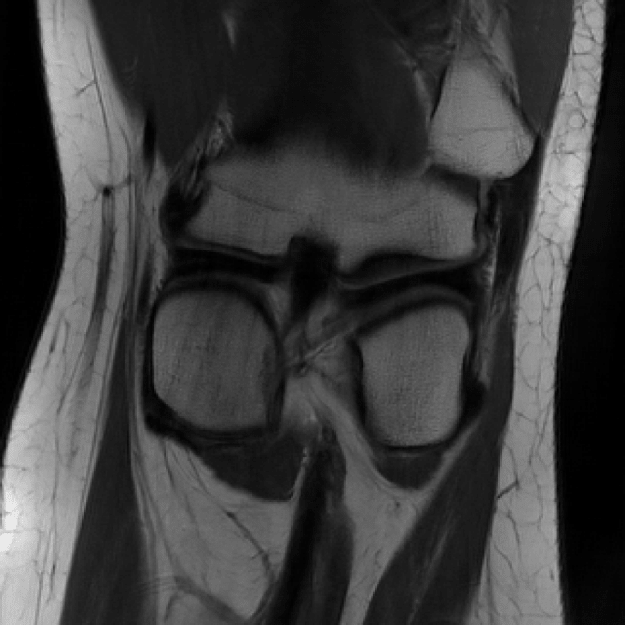}
		\includegraphics[width=0.10\linewidth, angle=180]{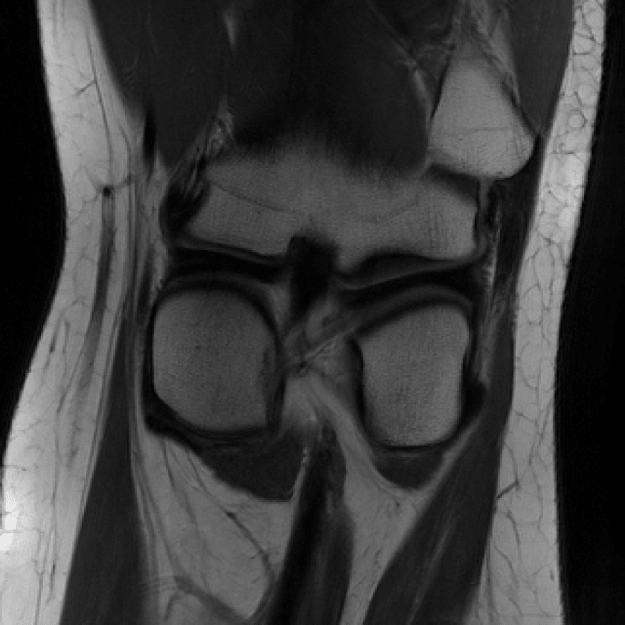}
		\includegraphics[width=0.10\linewidth, angle=180]{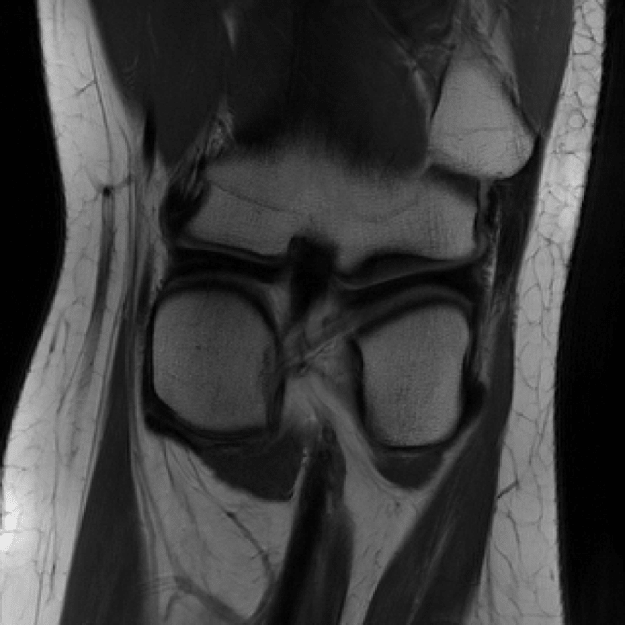}
		\includegraphics[width=0.10\linewidth, angle=180]{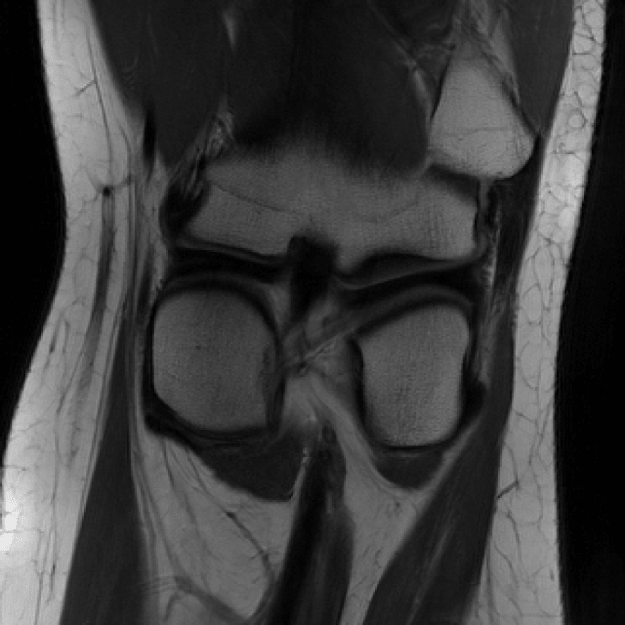}
		\includegraphics[width=0.10\linewidth, angle=180]{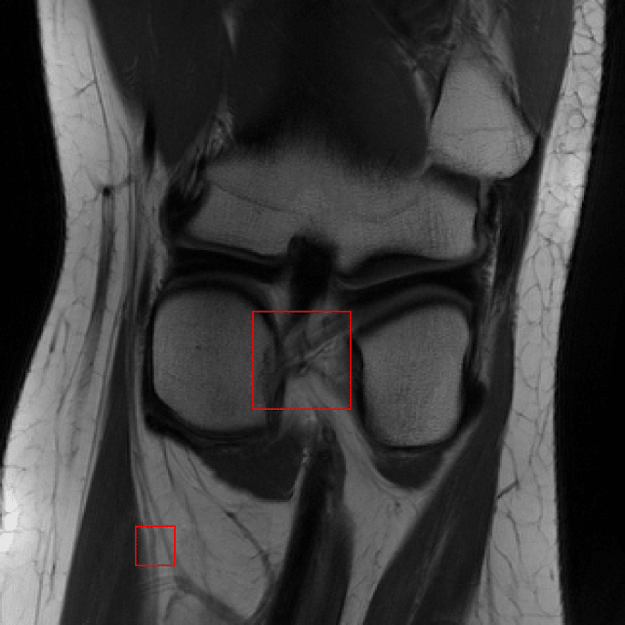}\\
		\includegraphics[width=0.10\linewidth, angle=180]{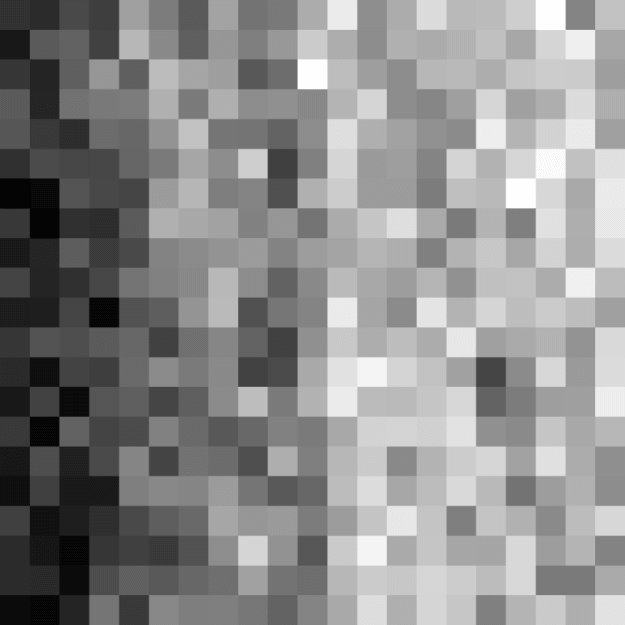}
		\includegraphics[width=0.10\linewidth, angle=180]{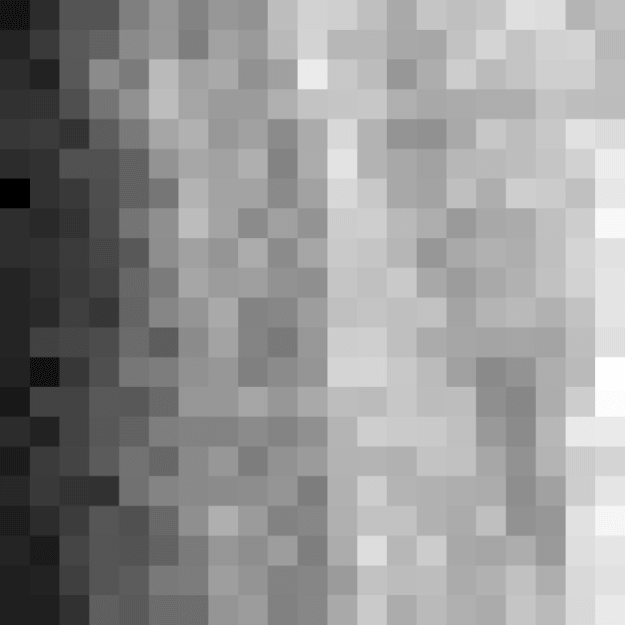}
		\includegraphics[width=0.10\linewidth, angle=180]{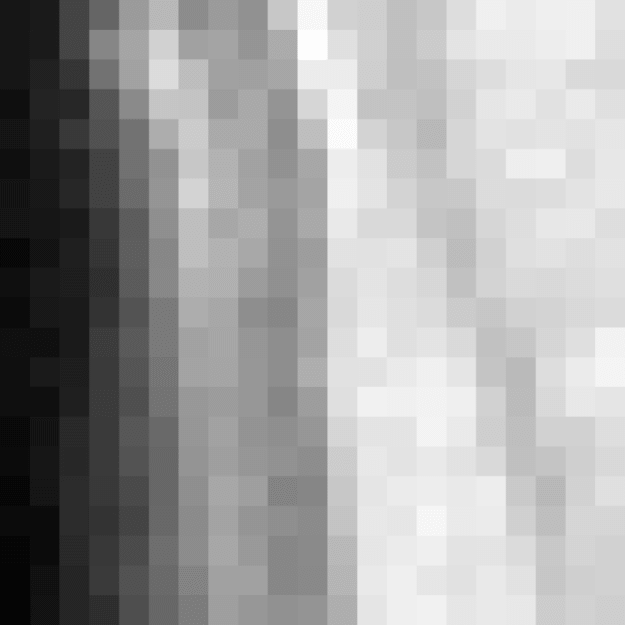}
		\includegraphics[width=0.10\linewidth, angle=180]{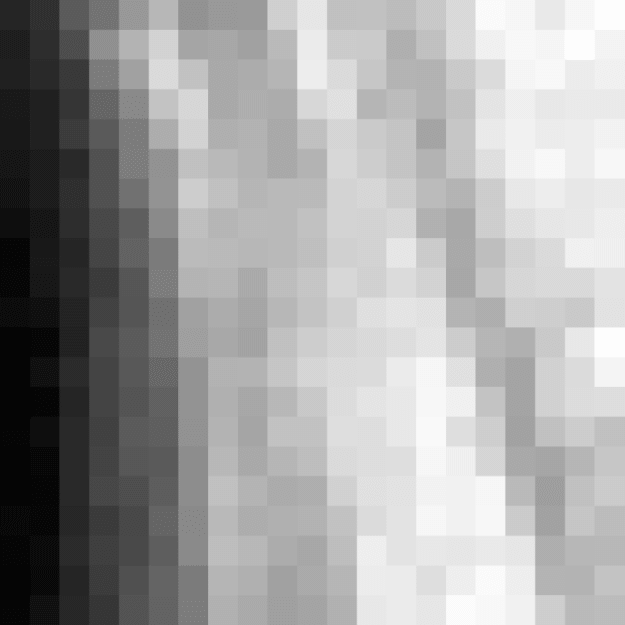}
		\includegraphics[width=0.10\linewidth, angle=180]{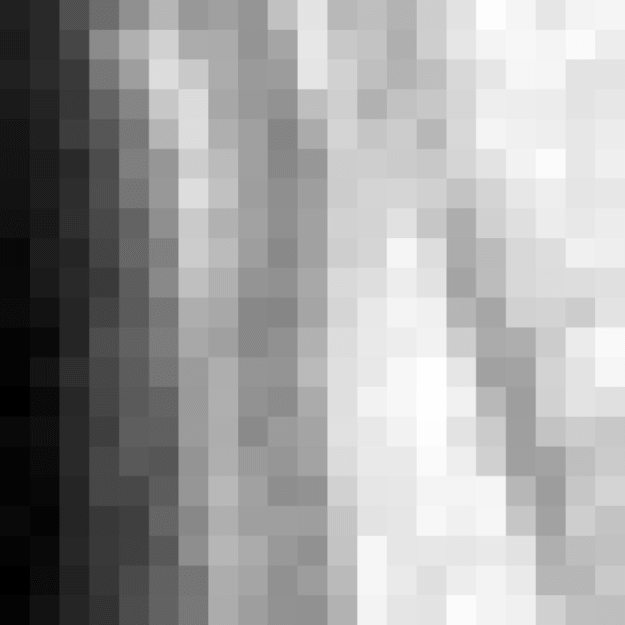}
		\includegraphics[width=0.10\linewidth, angle=180]{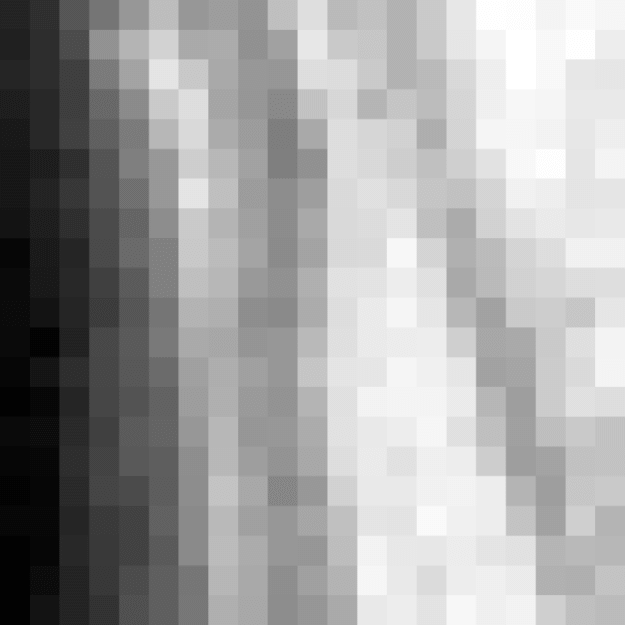}
		\includegraphics[width=0.10\linewidth, angle=180]{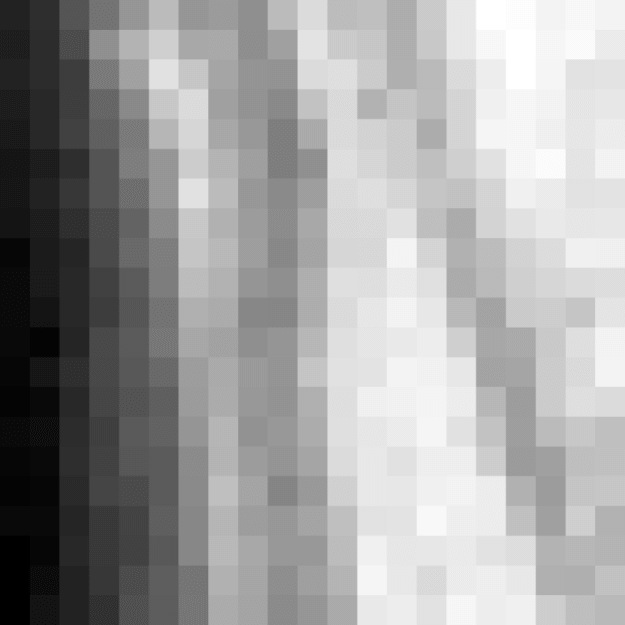}
		\includegraphics[width=0.10\linewidth, angle=180]{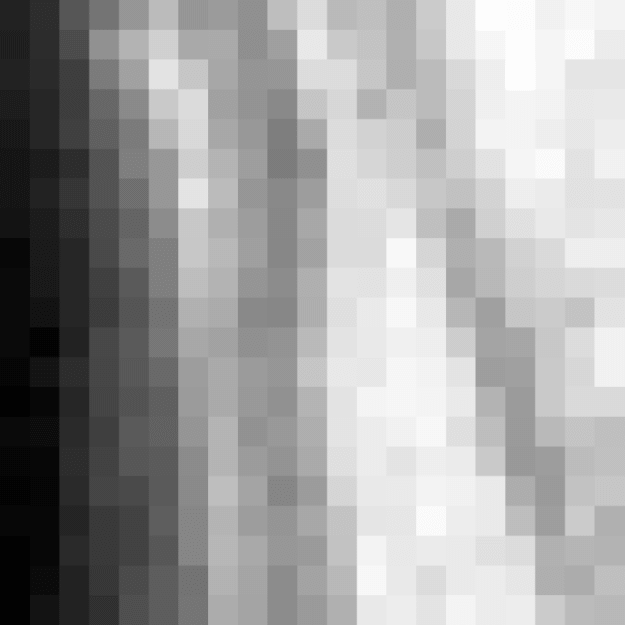}
		\includegraphics[width=0.10\linewidth, angle=180]{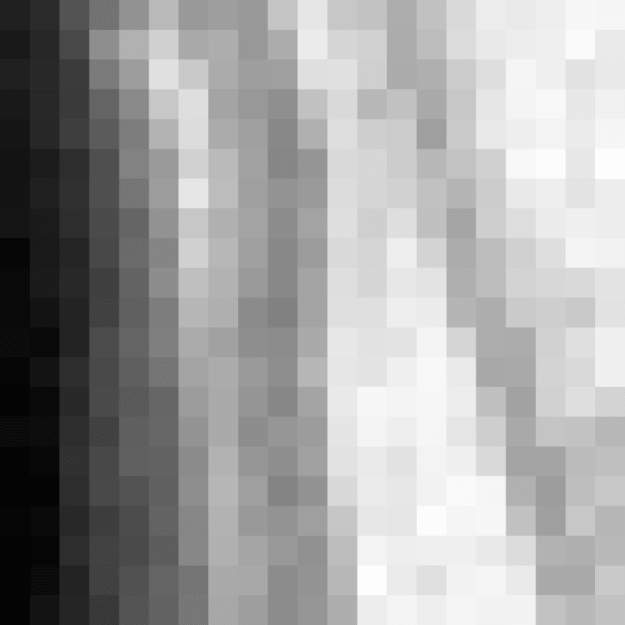}\\
		\includegraphics[width=0.10\linewidth, angle=180]{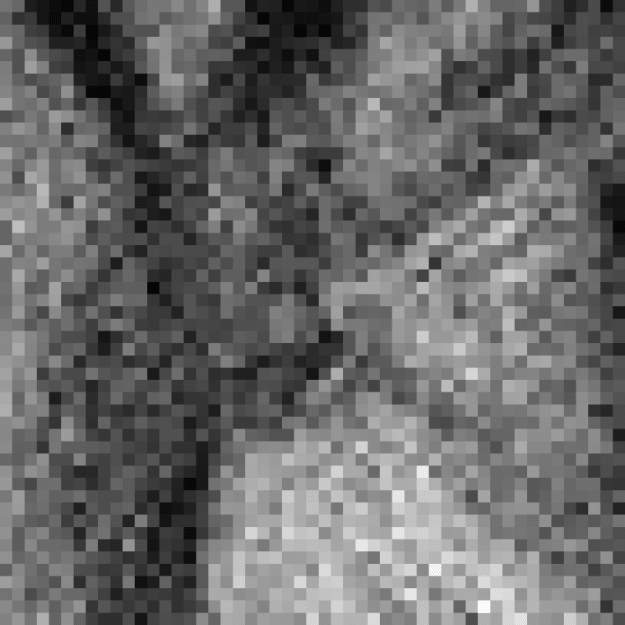}
		\includegraphics[width=0.10\linewidth, angle=180]{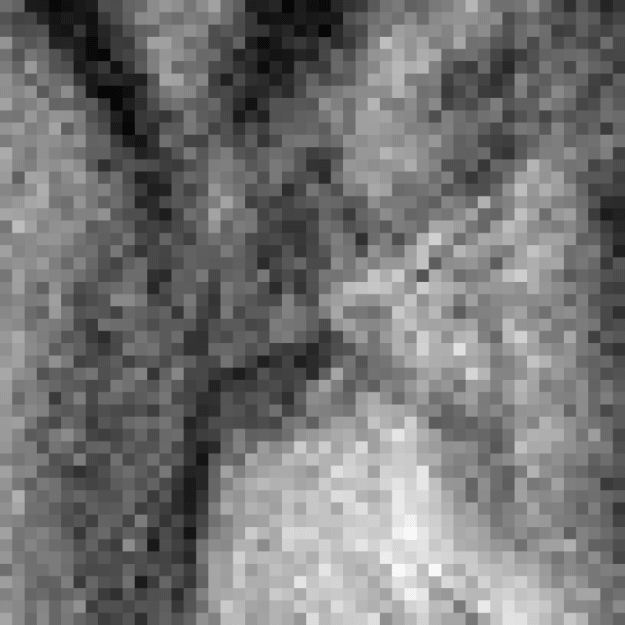}
		\includegraphics[width=0.10\linewidth, angle=180]{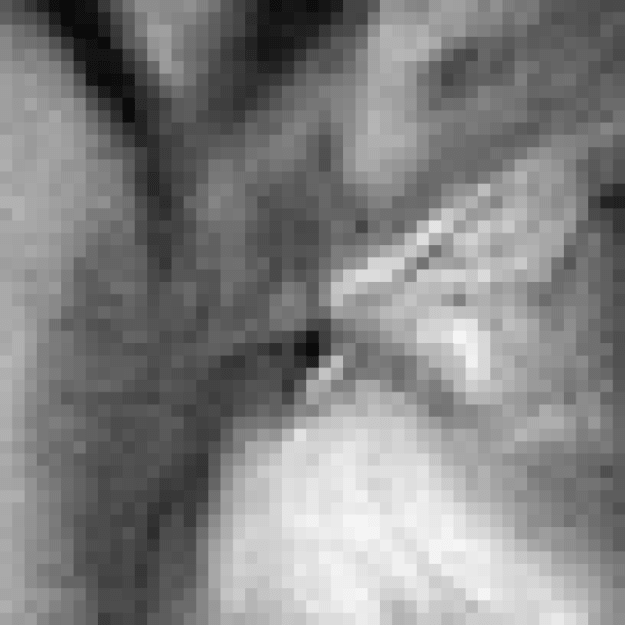}
		\includegraphics[width=0.10\linewidth, angle=180]{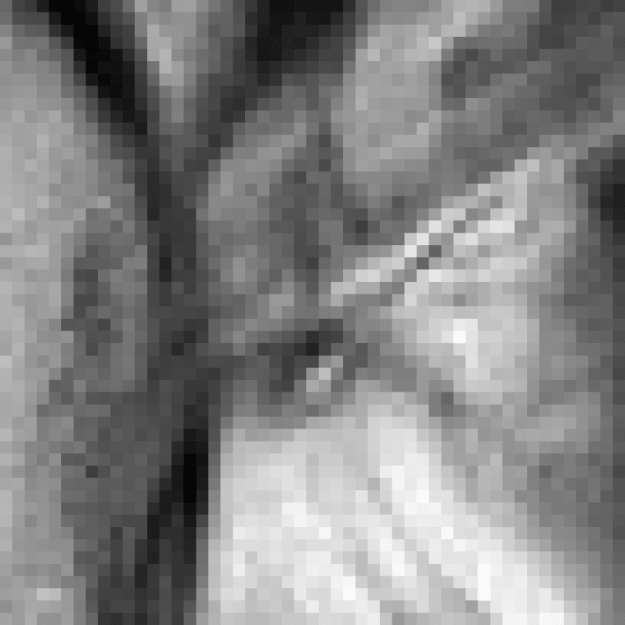}
		\includegraphics[width=0.10\linewidth, angle=180]{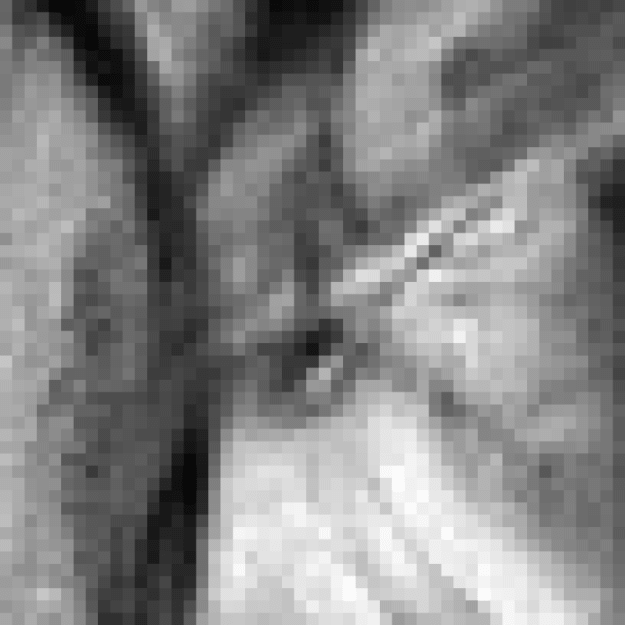}
		\includegraphics[width=0.10\linewidth, angle=180]{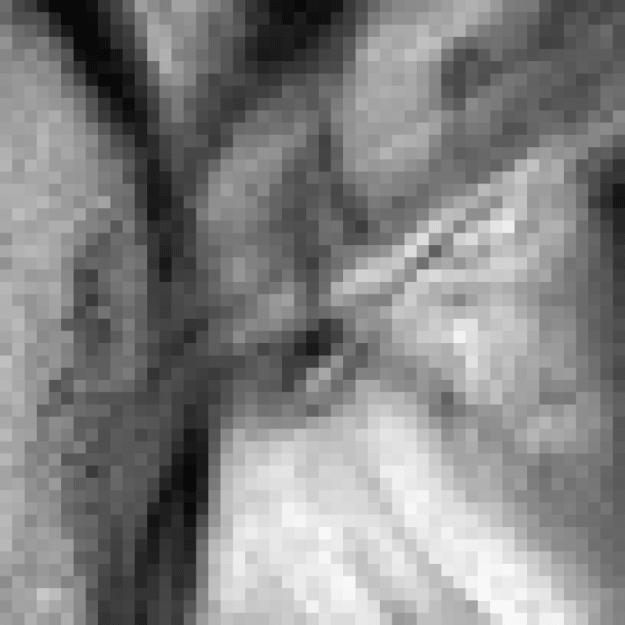}
		\includegraphics[width=0.10\linewidth, angle=180]{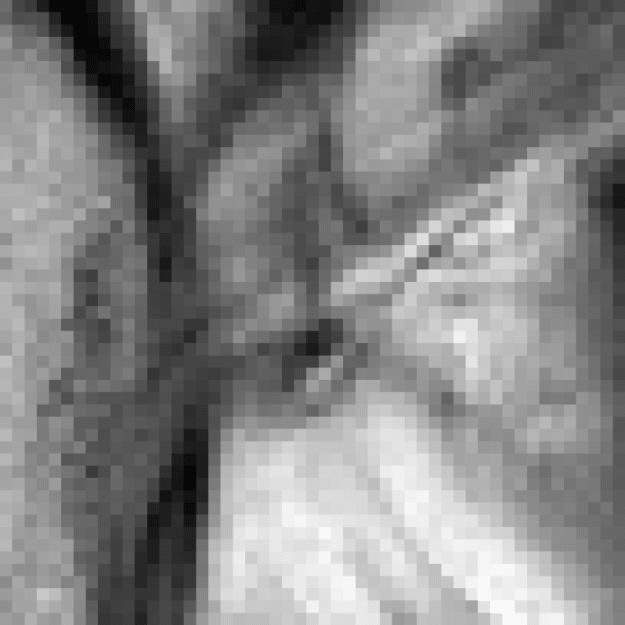}
		\includegraphics[width=0.10\linewidth, angle=180]{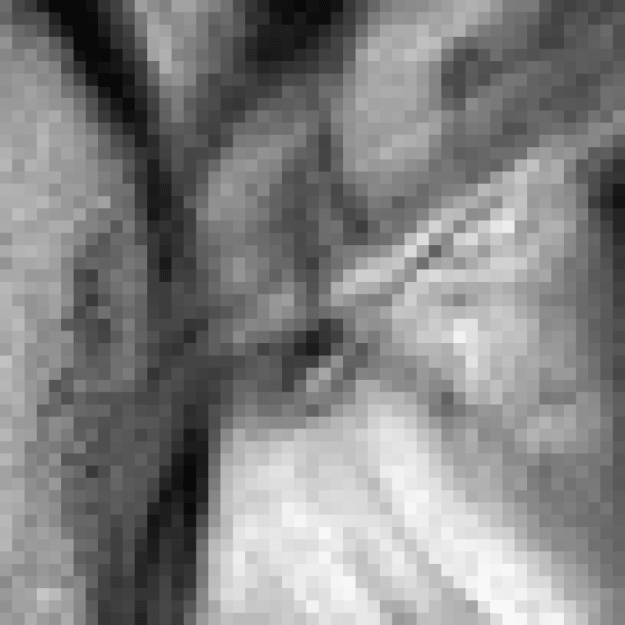}
		\includegraphics[width=0.10\linewidth, angle=180]{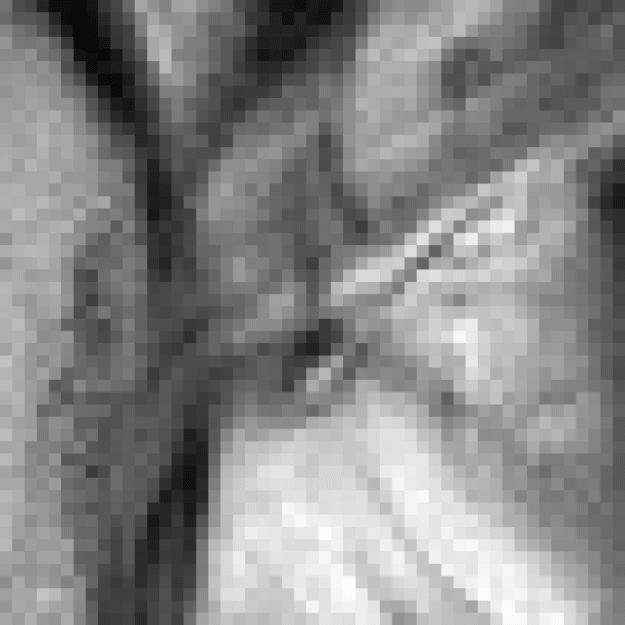}\\
		\includegraphics[width=0.10\linewidth, angle=180]{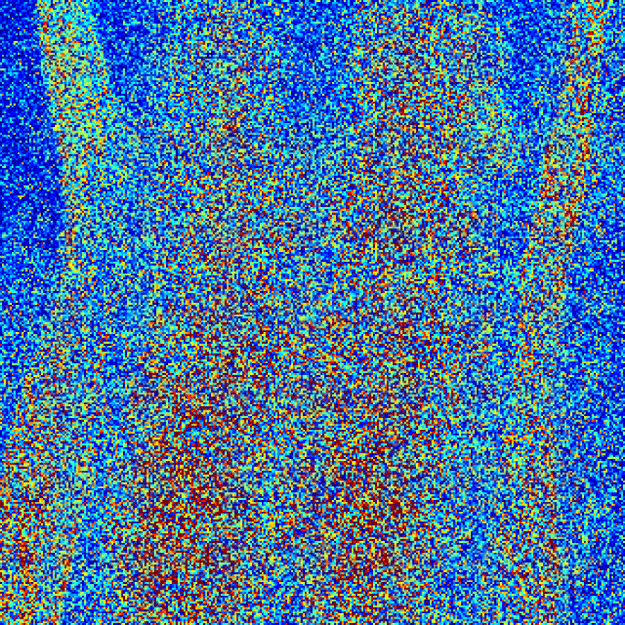}
		\includegraphics[width=0.10\linewidth, angle=180]{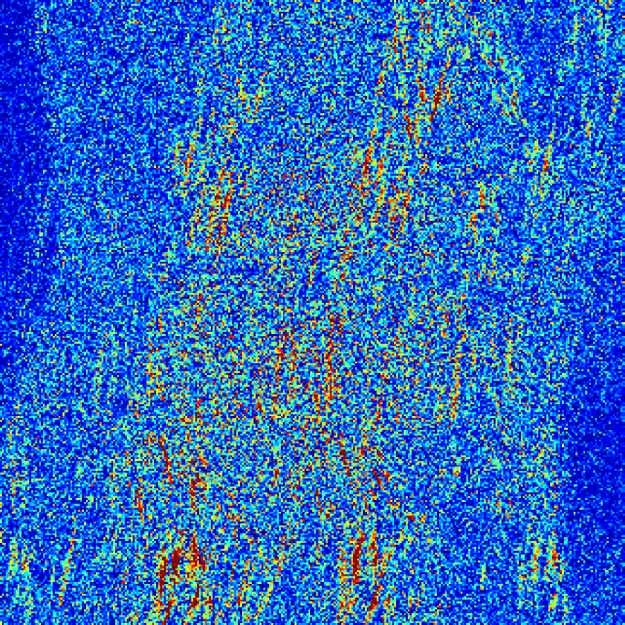}
		\includegraphics[width=0.10\linewidth, angle=180]{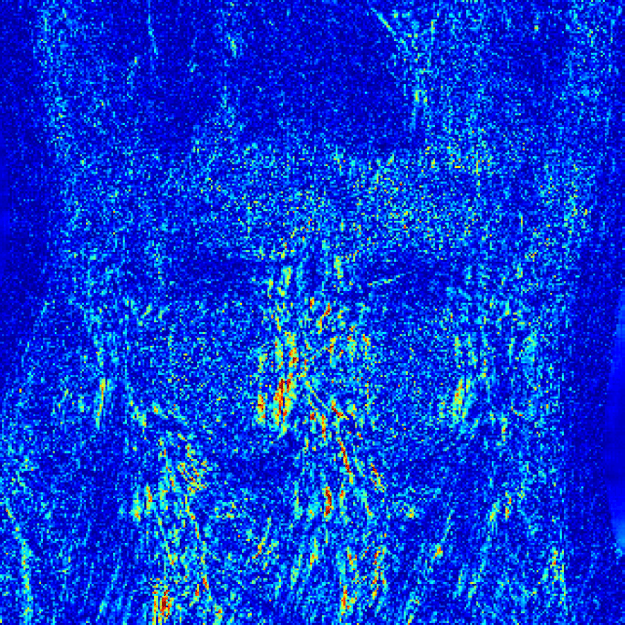}
		\includegraphics[width=0.10\linewidth, angle=180]{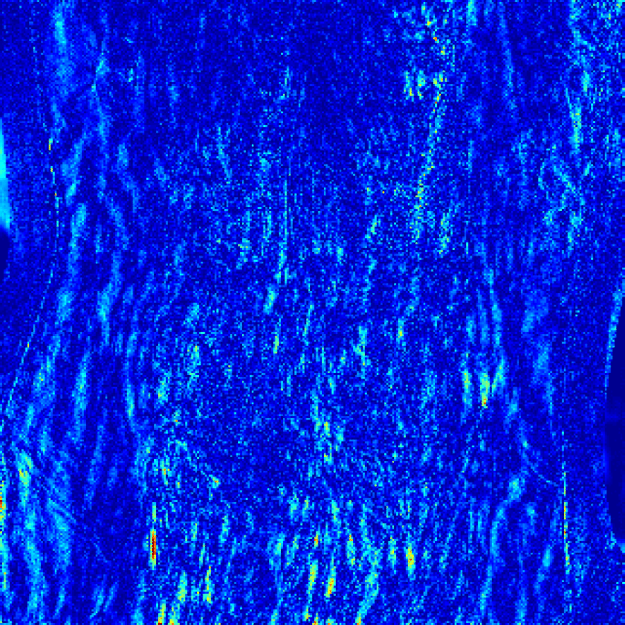}
		\includegraphics[width=0.10\linewidth, angle=180]{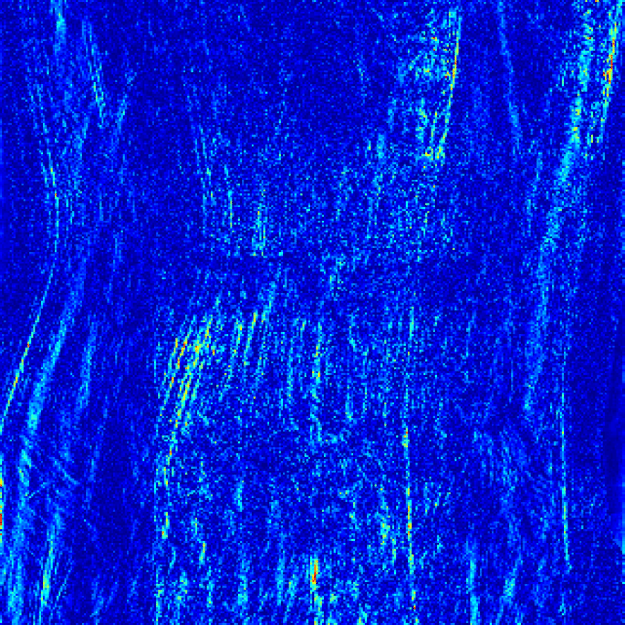}
		\includegraphics[width=0.10\linewidth, angle=180]{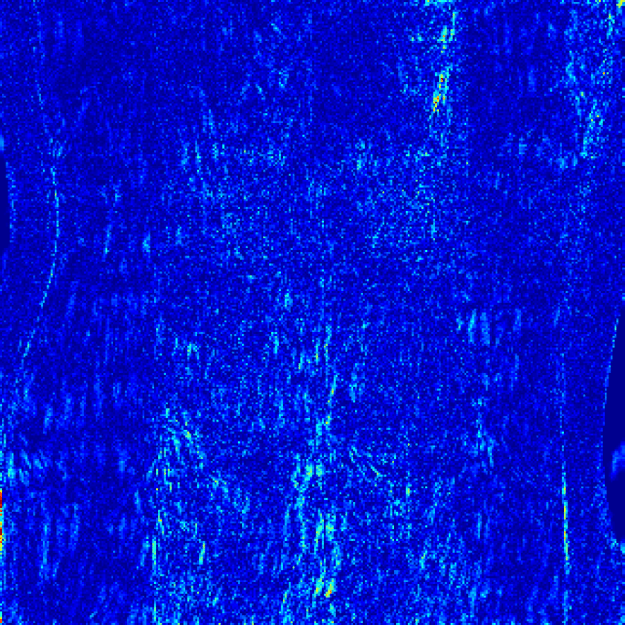}
		\includegraphics[width=0.10\linewidth, angle=180]{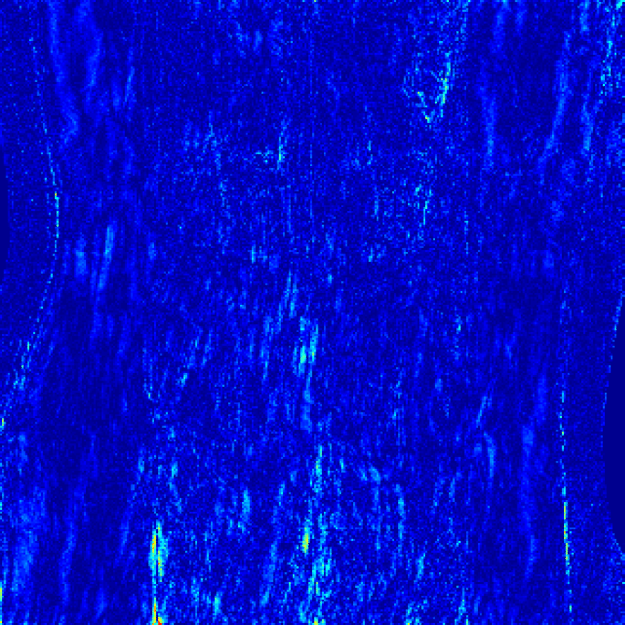}
		\includegraphics[width=0.10\linewidth, angle=180]{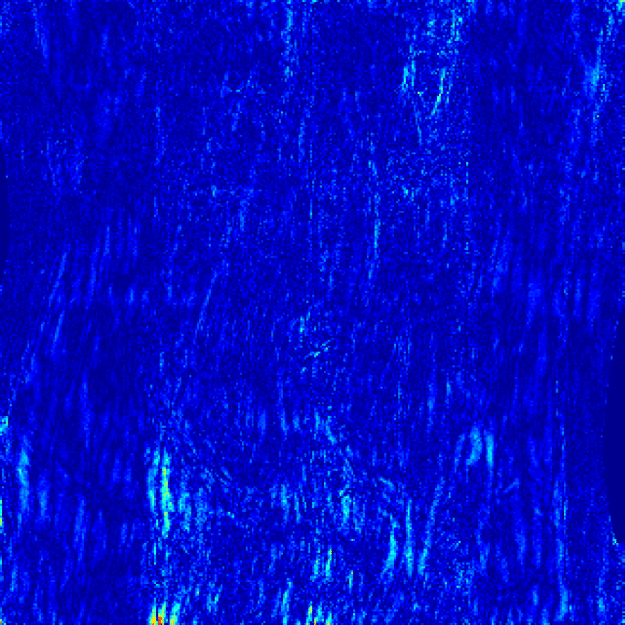}
		\includegraphics[width=0.10\linewidth, angle=180]{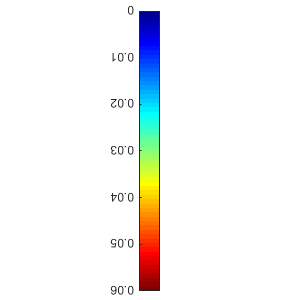}\\
		\includegraphics[width=0.10\linewidth]{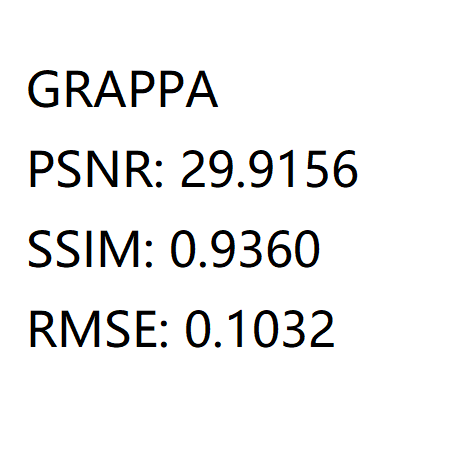}
		\includegraphics[width=0.10\linewidth]{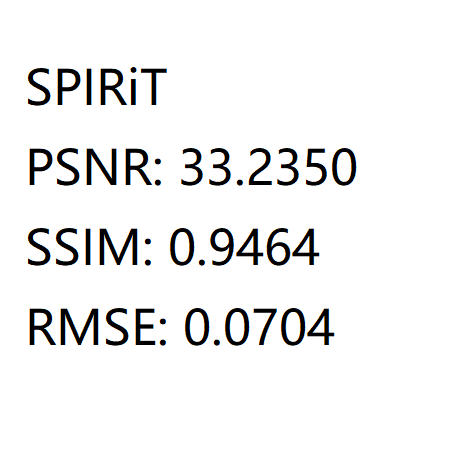}
		\includegraphics[width=0.10\linewidth]{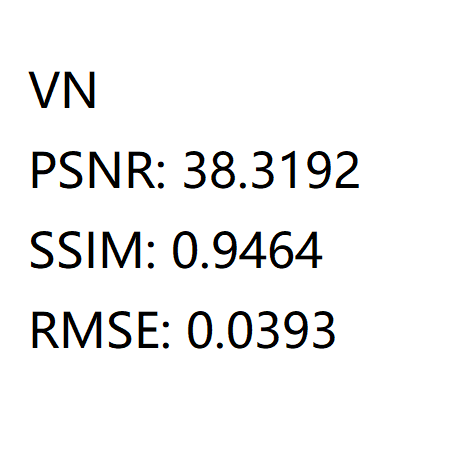}
		\includegraphics[width=0.10\linewidth]{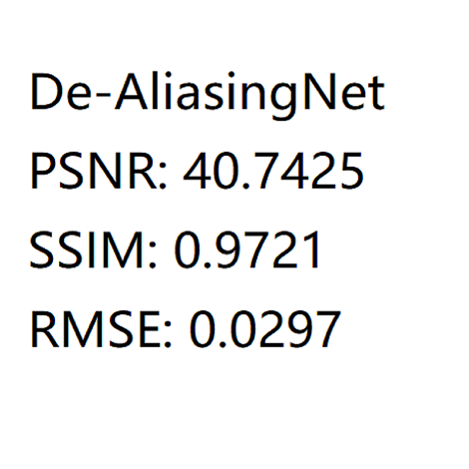}
		\includegraphics[width=0.10\linewidth]{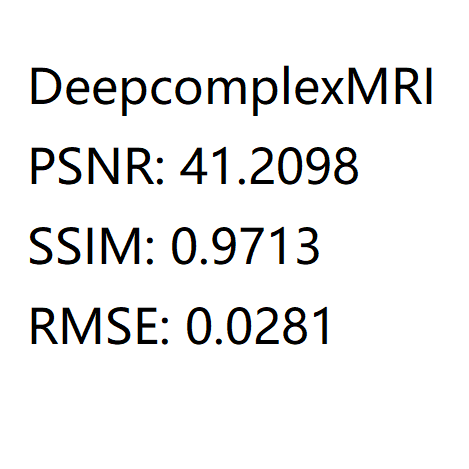}
		\includegraphics[width=0.10\linewidth]{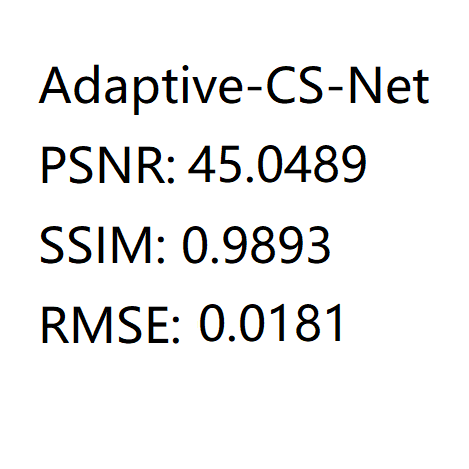}
		\includegraphics[width=0.10\linewidth]{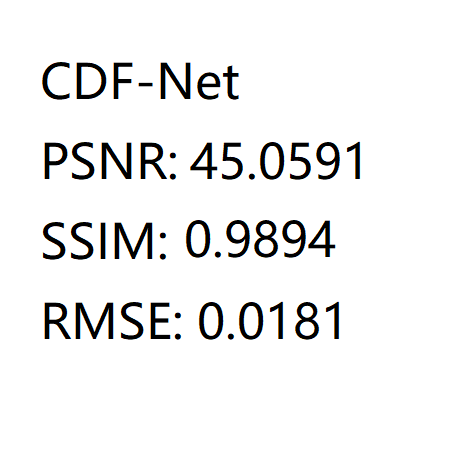}
		\includegraphics[width=0.10\linewidth]{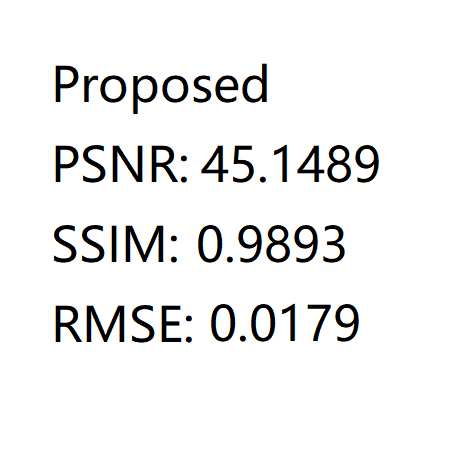}
		\includegraphics[width=0.10\linewidth]{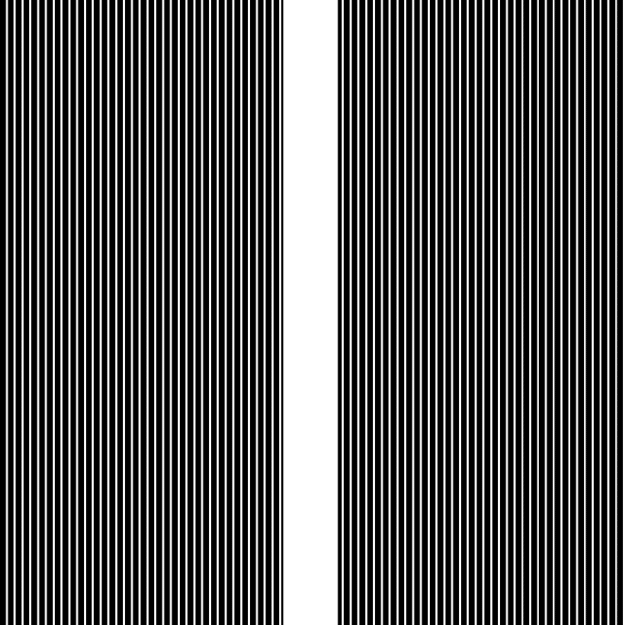}
		\caption{Qualitative comparison results of reconstruction methods on the Coronal PD knee image.} 
		\label{PD_chp3}
	\end{figure*}
	\begin{figure*}
		\centering
		\includegraphics[width=0.075\linewidth, angle=180]{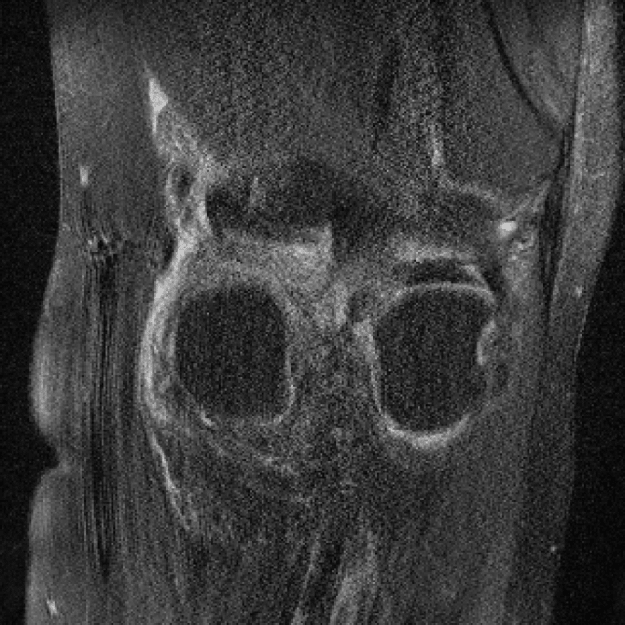}
		\includegraphics[width=0.075\linewidth, angle=180]{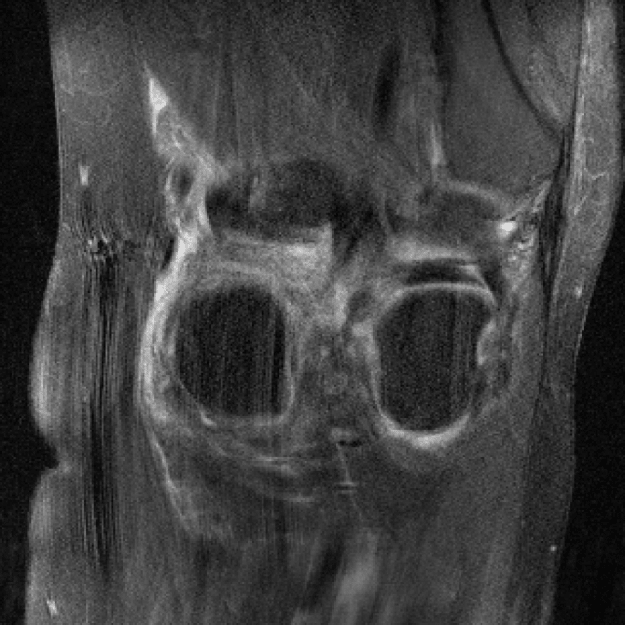}
		\includegraphics[width=0.075\linewidth, angle=180]{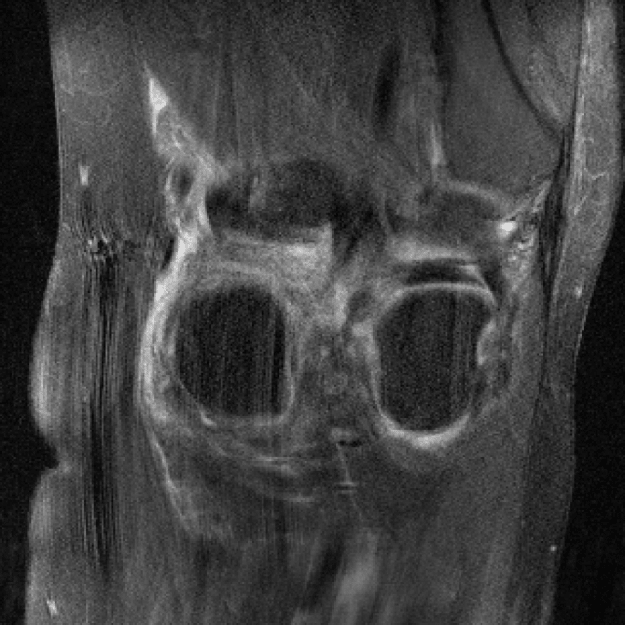}
		\includegraphics[width=0.075\linewidth, angle=180]{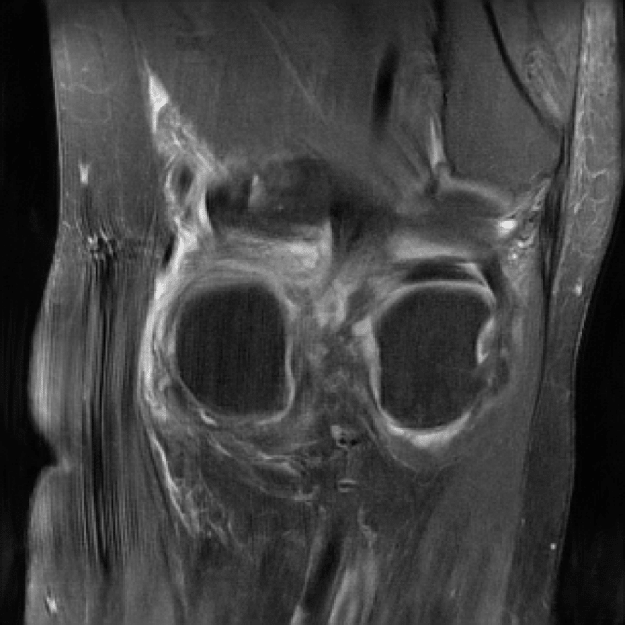}
		\includegraphics[width=0.075\linewidth, angle=180]{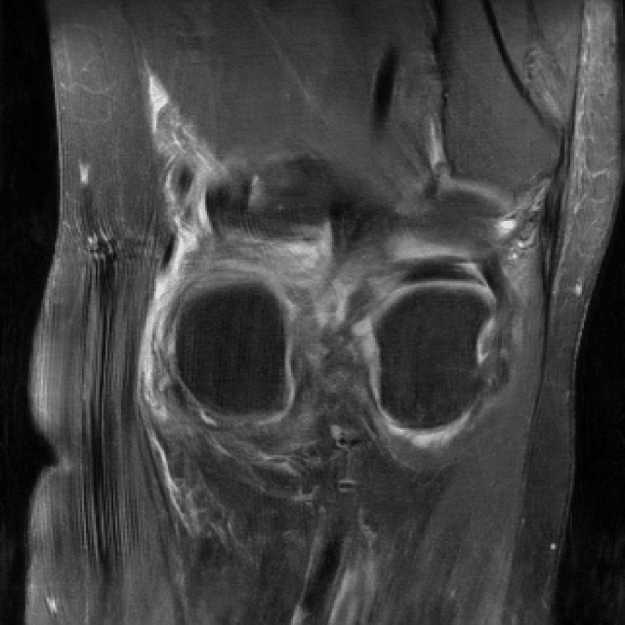}
		\includegraphics[width=0.075\linewidth, angle=180]{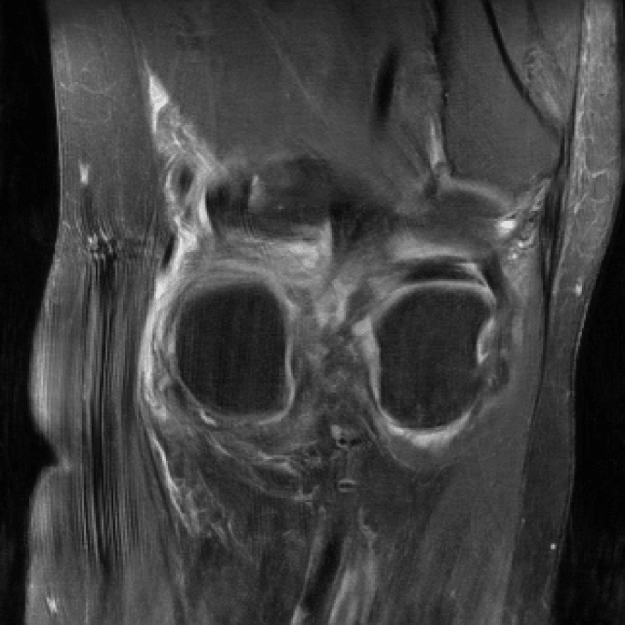}
		\includegraphics[width=0.075\linewidth, angle=180]{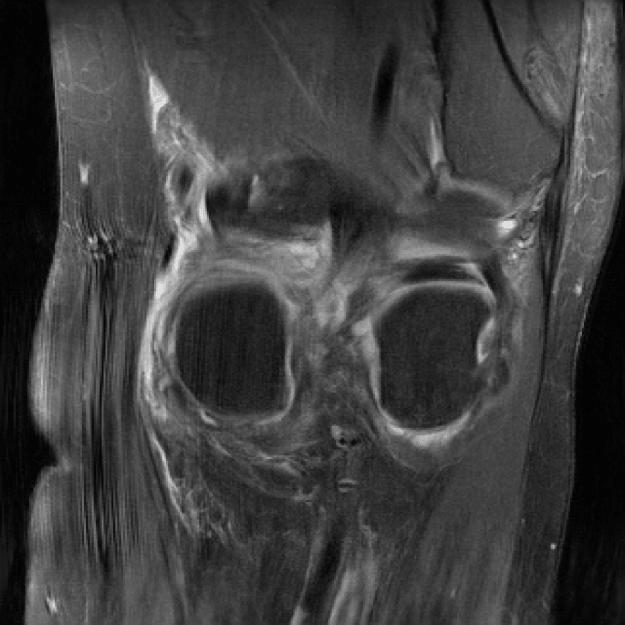}
		\includegraphics[width=0.075\linewidth, angle=180]{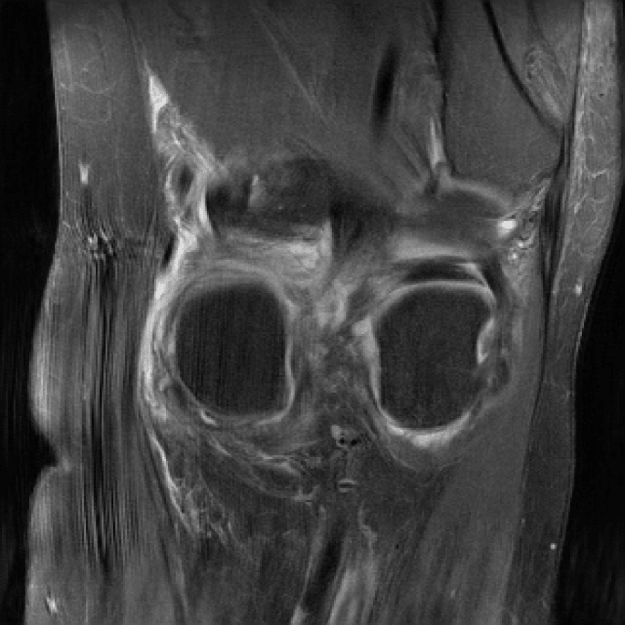}
		\includegraphics[width=0.075\linewidth, angle=180]{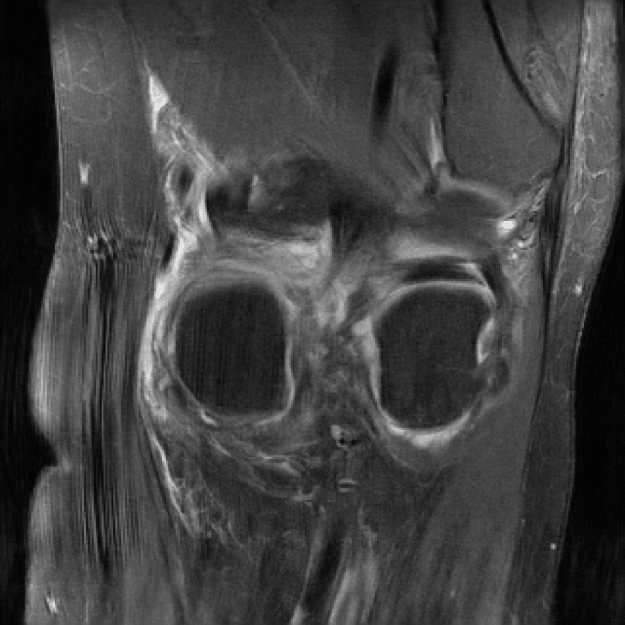}
		\includegraphics[width=0.075\linewidth, angle=180]{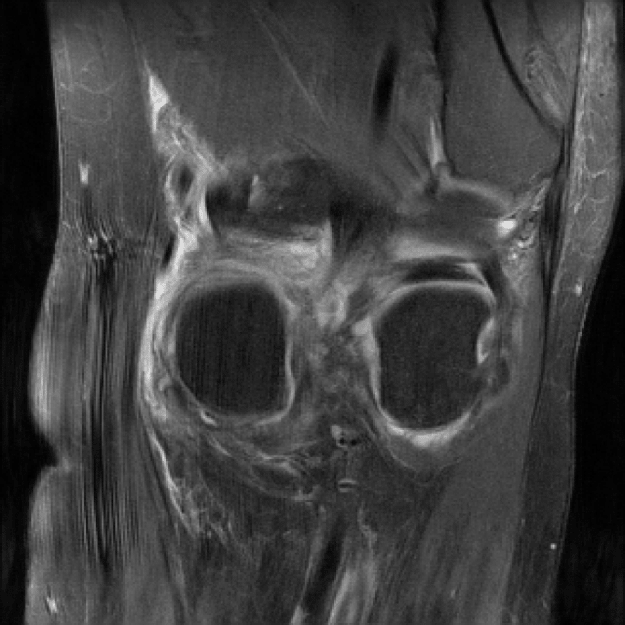}
		\includegraphics[width=0.075\linewidth, angle=180]{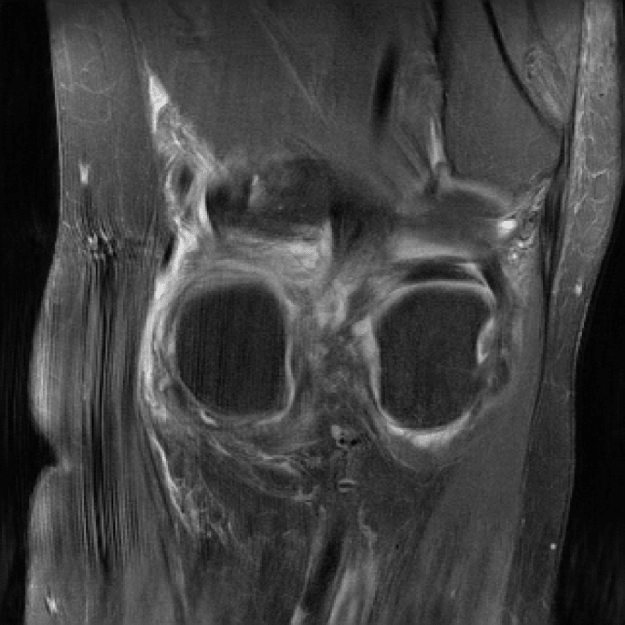}
		\includegraphics[width=0.075\linewidth, angle=180]{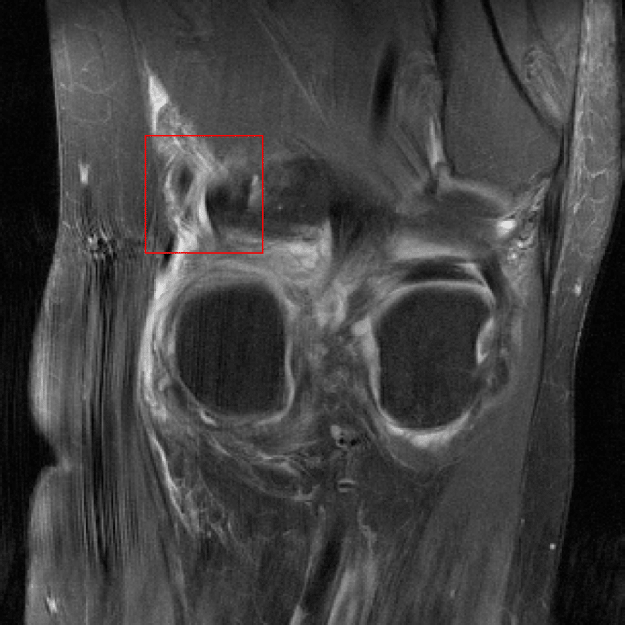}\\
		\includegraphics[width=0.075\linewidth, angle=180]{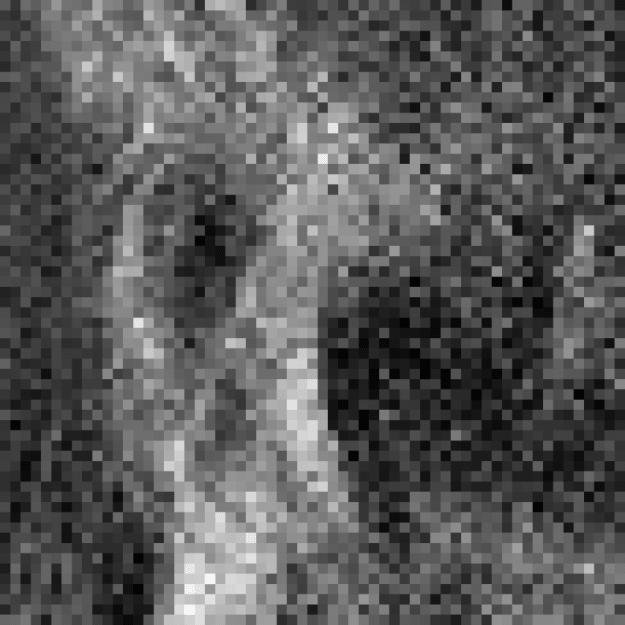}
		\includegraphics[width=0.075\linewidth, angle=180]{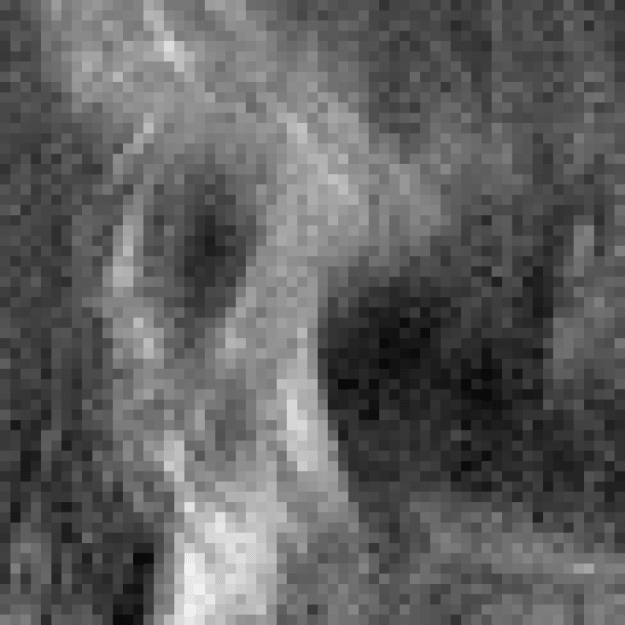}
		\includegraphics[width=0.075\linewidth, angle=180]{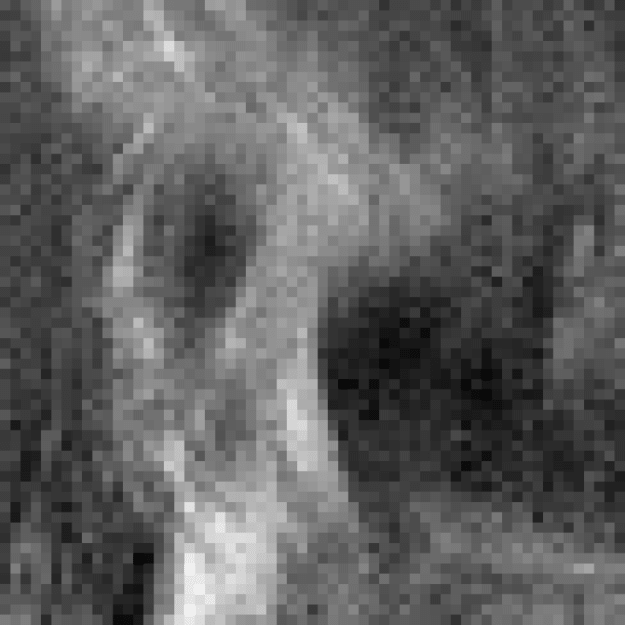}
		\includegraphics[width=0.075\linewidth, angle=180]{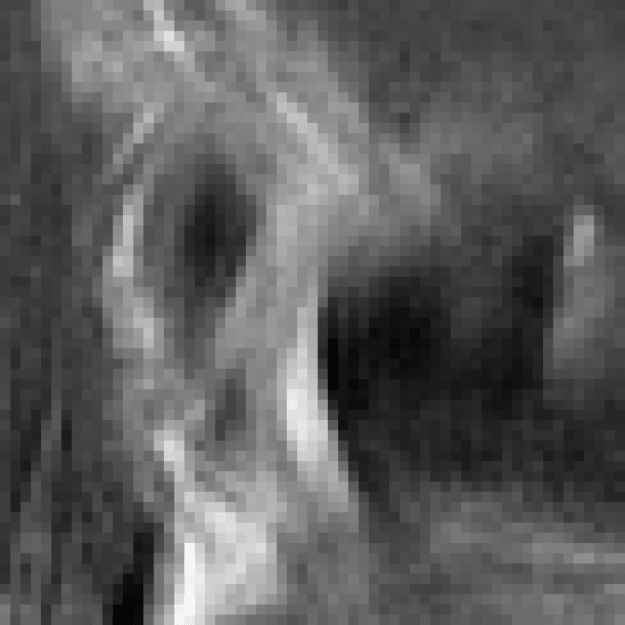}
		\includegraphics[width=0.075\linewidth, angle=180]{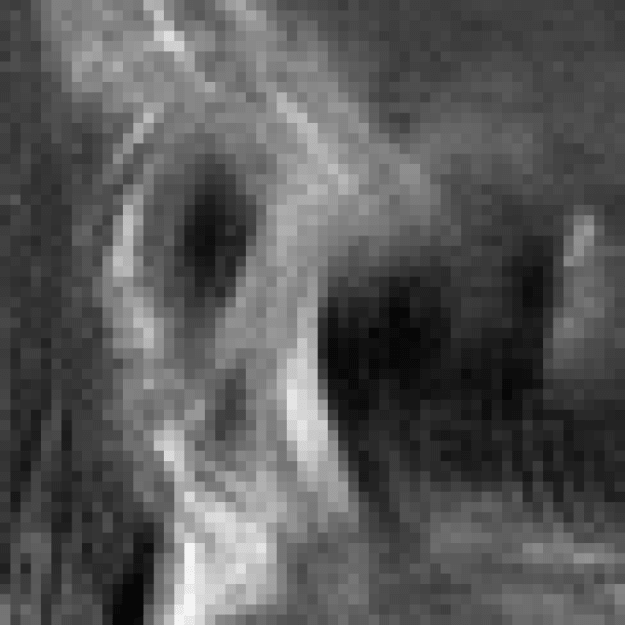}
		\includegraphics[width=0.075\linewidth, angle=180]{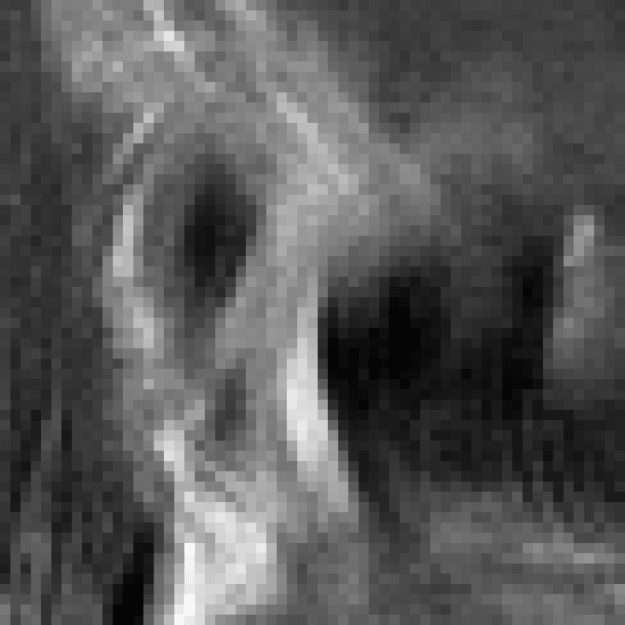}
		\includegraphics[width=0.075\linewidth, angle=180]{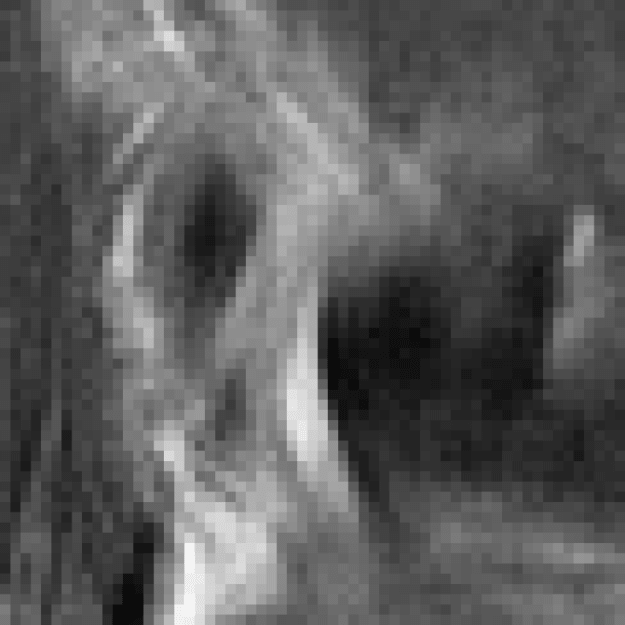}
		\includegraphics[width=0.075\linewidth, angle=180]{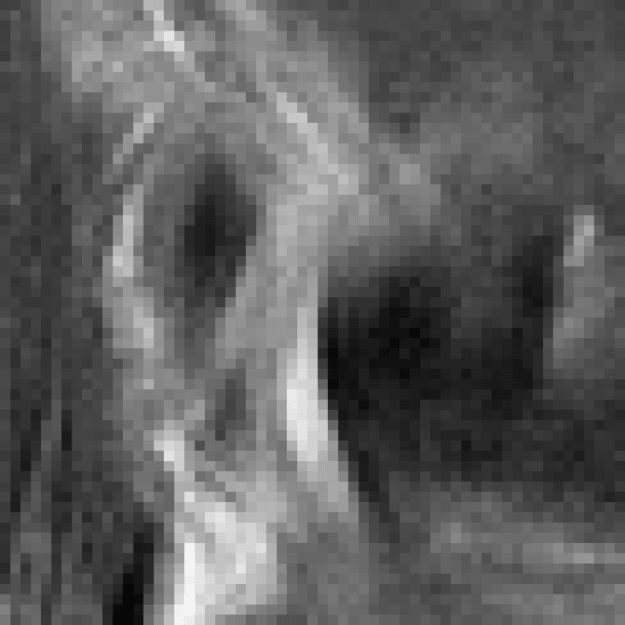}
		\includegraphics[width=0.075\linewidth, angle=180]{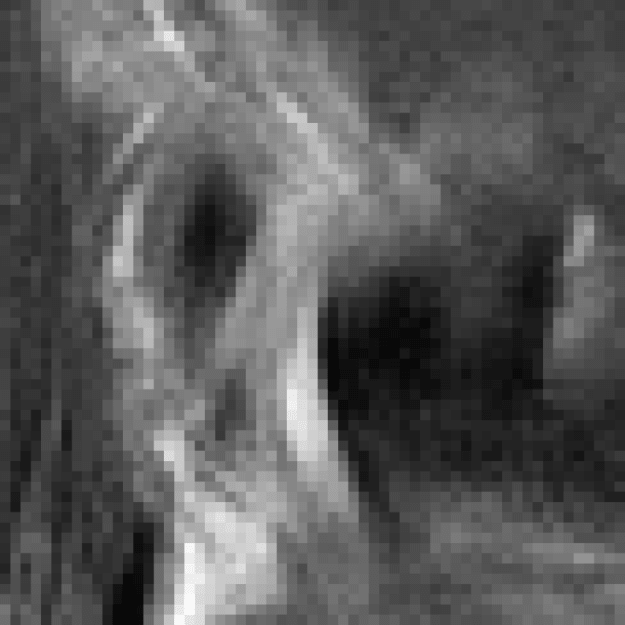}
		\includegraphics[width=0.075\linewidth, angle=180]{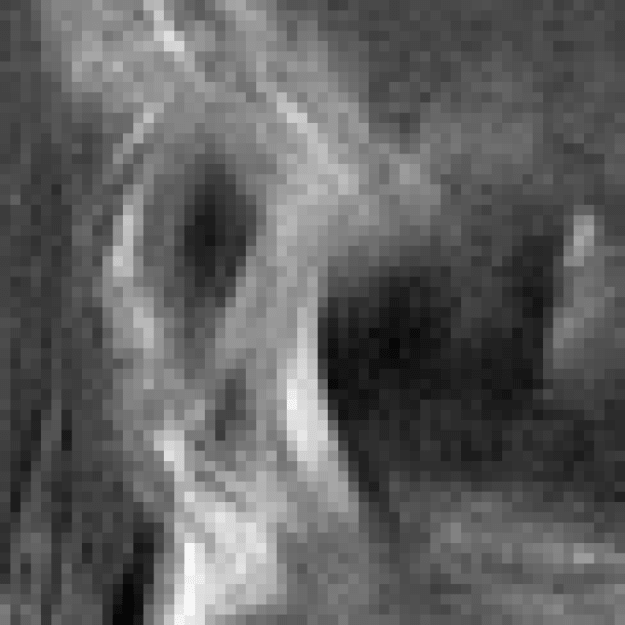}
		\includegraphics[width=0.075\linewidth, angle=180]{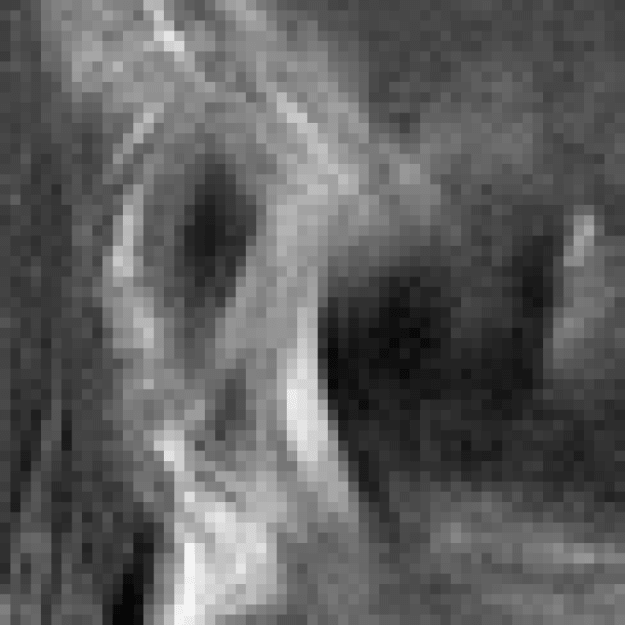}
		\includegraphics[width=0.075\linewidth, angle=180]{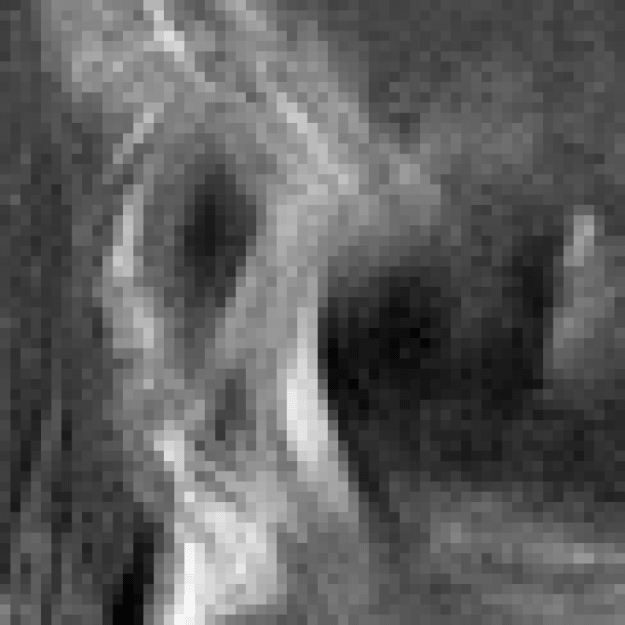}\\
		\includegraphics[width=0.075\linewidth, angle=180]{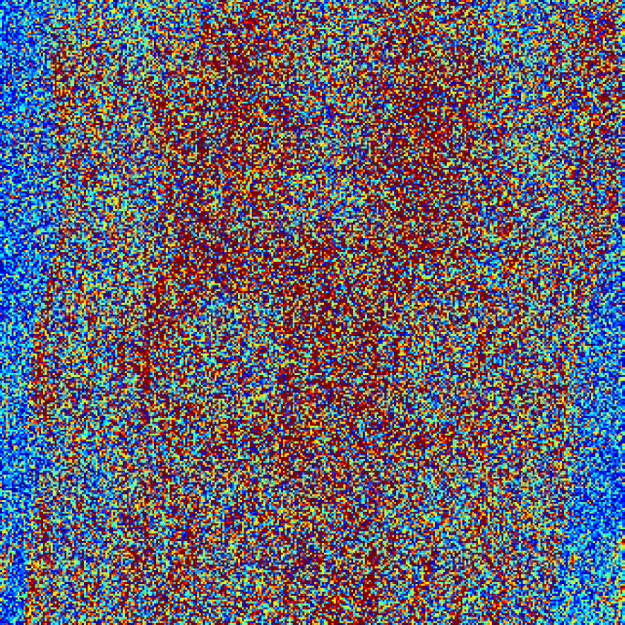}
		\includegraphics[width=0.075\linewidth, angle=180]{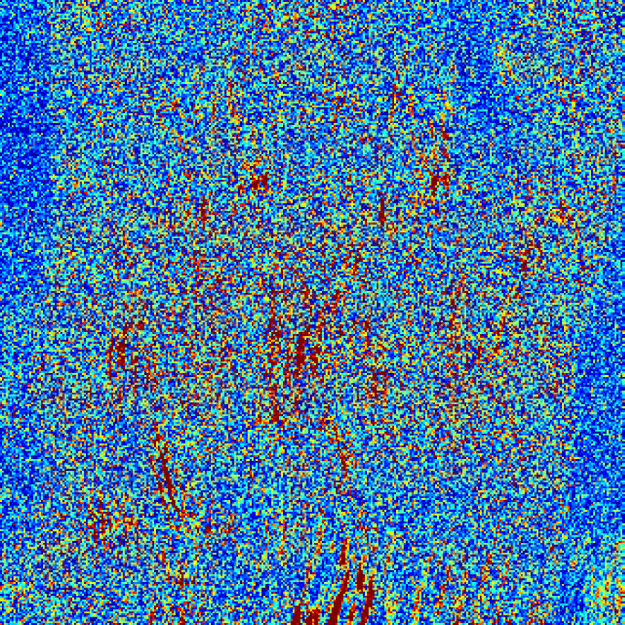}
		\includegraphics[width=0.075\linewidth, angle=180]{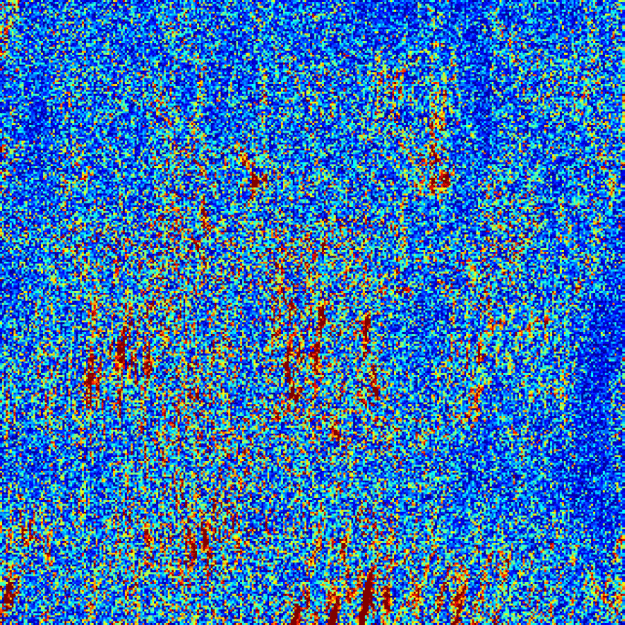}
		\includegraphics[width=0.075\linewidth, angle=180]{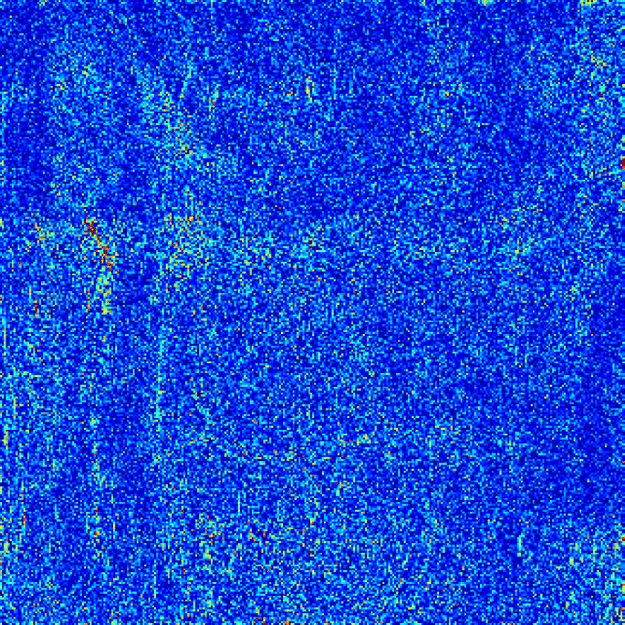}
		\includegraphics[width=0.075\linewidth, angle=180]{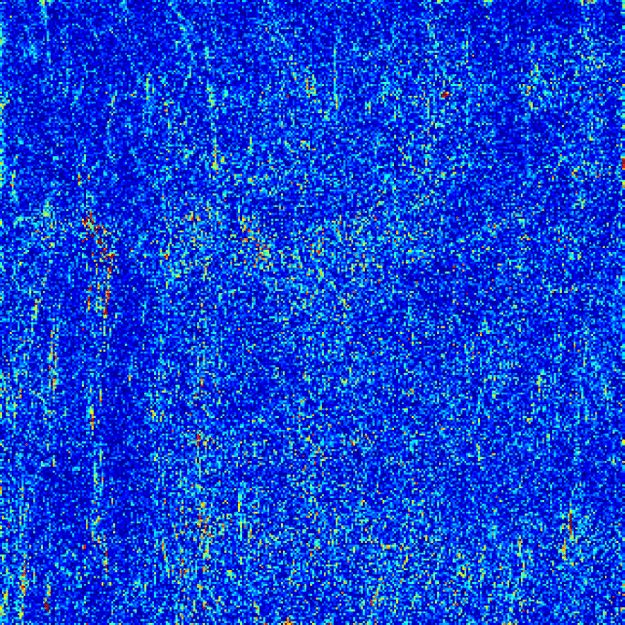}
		\includegraphics[width=0.075\linewidth, angle=180]{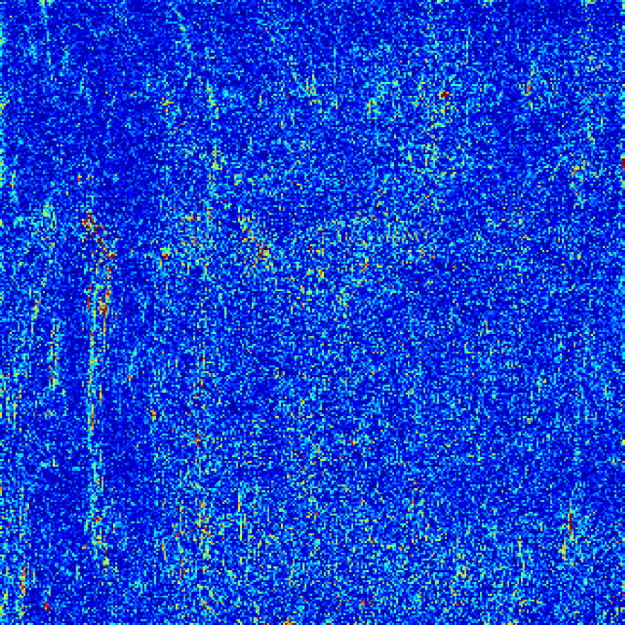}
		\includegraphics[width=0.075\linewidth, angle=180]{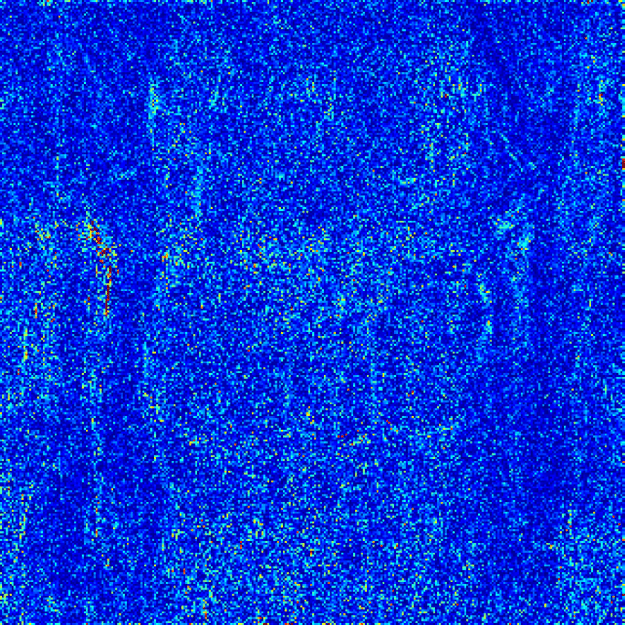}
		\includegraphics[width=0.075\linewidth, angle=180]{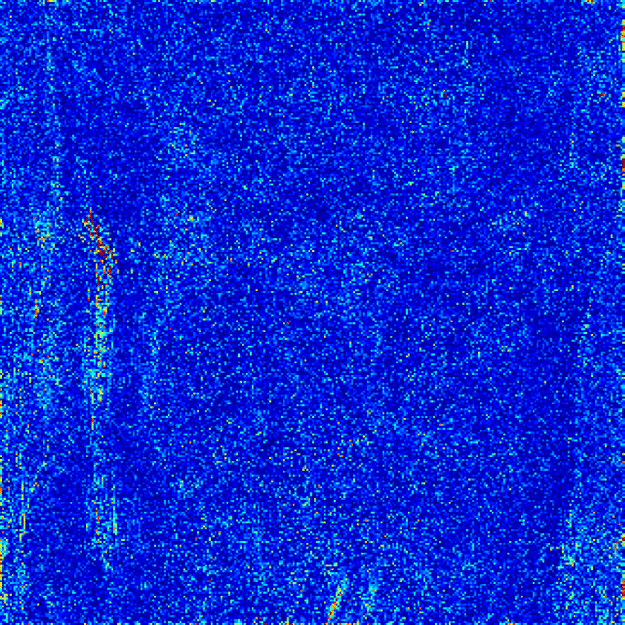}
		\includegraphics[width=0.075\linewidth, angle=180]{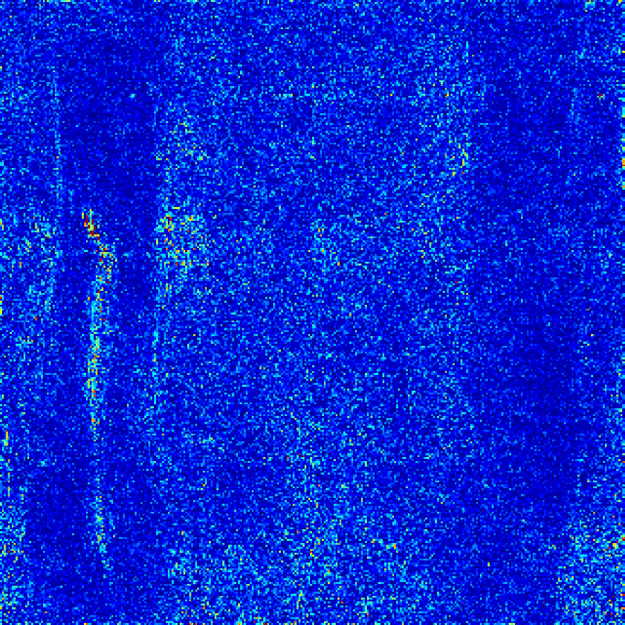}
		\includegraphics[width=0.075\linewidth, angle=180]{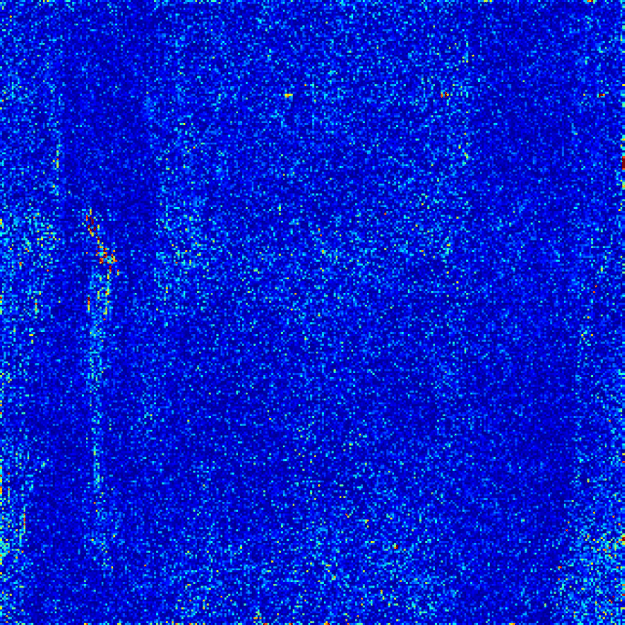}
		\includegraphics[width=0.075\linewidth, angle=180]{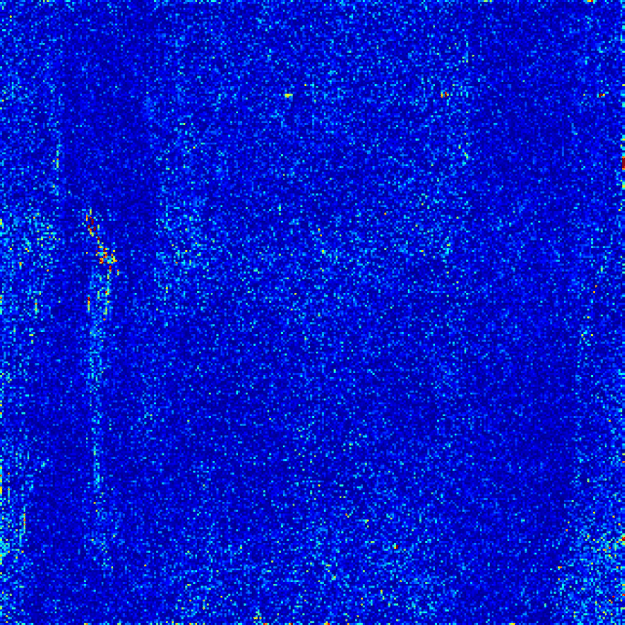}
		\includegraphics[width=0.075\linewidth, angle=180]{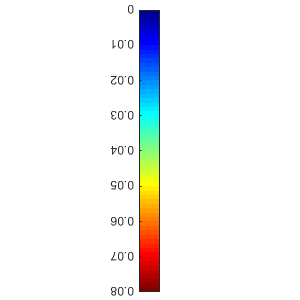}\\
		\includegraphics[width=0.075\linewidth]{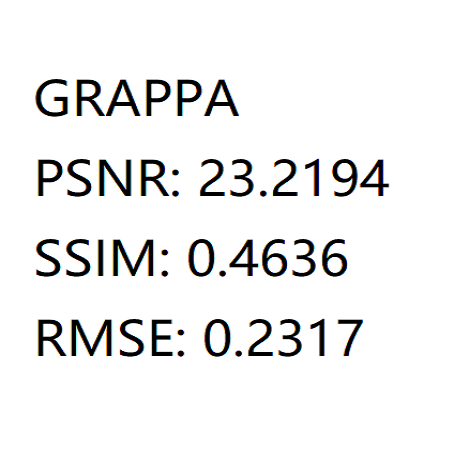}
		\includegraphics[width=0.075\linewidth]{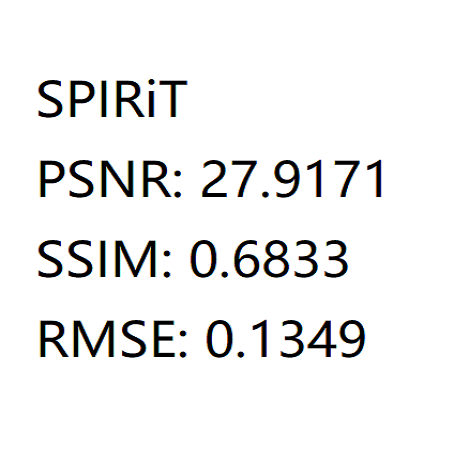}
		\includegraphics[width=0.075\linewidth]{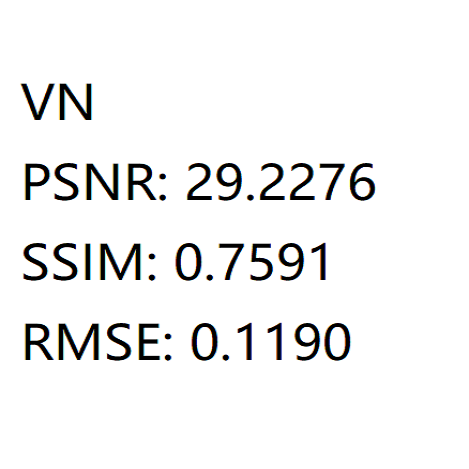}
		\includegraphics[width=0.075\linewidth]{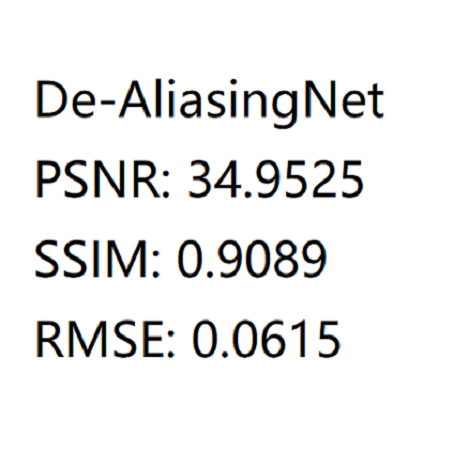}
		\includegraphics[width=0.075\linewidth]{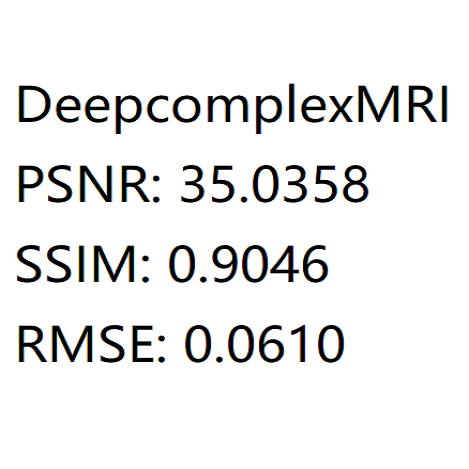}
		\includegraphics[width=0.075\linewidth]{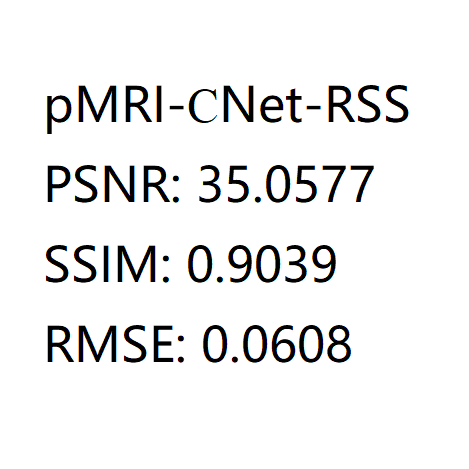}
		\includegraphics[width=0.075\linewidth]{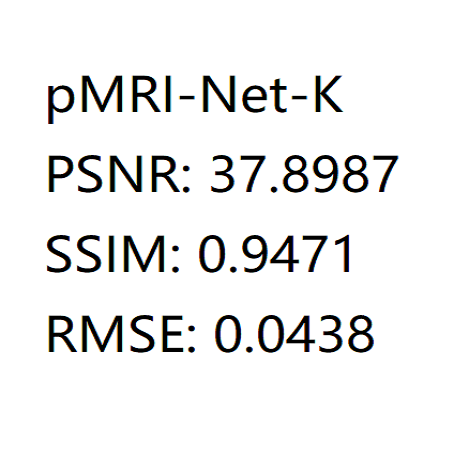}
		\includegraphics[width=0.075\linewidth]{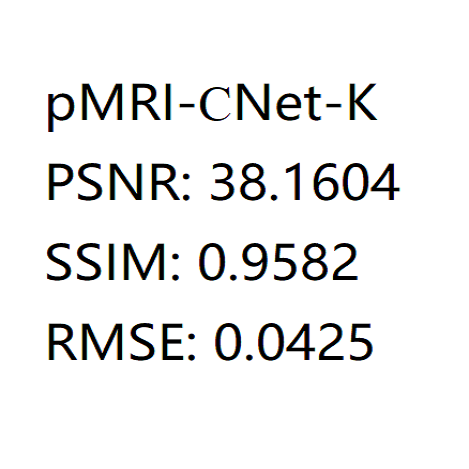}
		\includegraphics[width=0.075\linewidth]{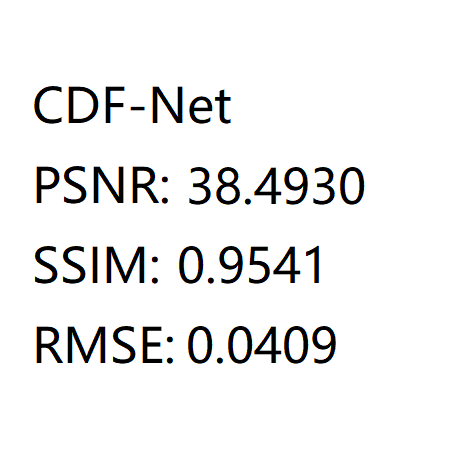}
		\includegraphics[width=0.075\linewidth]{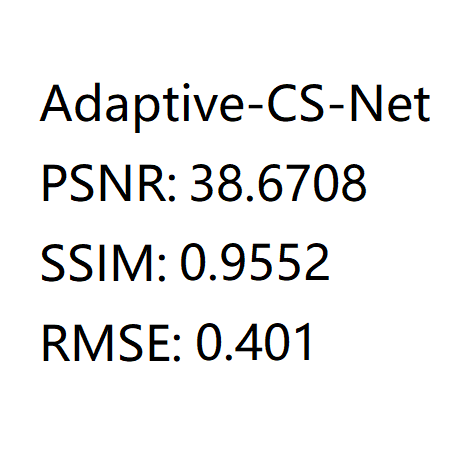}
		\includegraphics[width=0.075\linewidth]{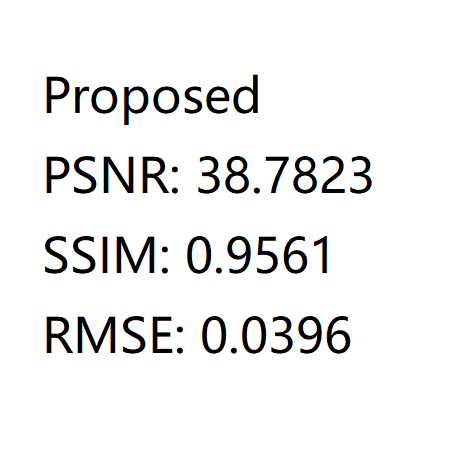}
		\includegraphics[width=0.075\linewidth]{fig_chp3/mask.png}
		\caption{Qualitative comparison results of reconstruction methods on the Coronal FSPD knee image. }
		\label{FSPD}
	\end{figure*}
	
	\begin{table}
		\caption{Mean square error between $\ubf$ and $\ubf^* $ for PD data comparing to the methods that reconstruct multi-coil images. }
		\centering
		\begin{tabular}{ccc}
			\toprule
			Method   & RMSE between $\ubf$ and $  \ubf^*$   \\
			De-AliasingNet~\cite{10.1007/978-3-030-32248-9_4} & 0.1573\\
			DeepcomplexMRI~\cite{WANG2020136} & 0.1253\\
			CDF-Net~\cite{10.1007/978-3-030-59713-9_41} & 0.1196\\
			Adaptive-CS-Net~\cite{pezzotti2020adaptive}  & 0.1096\\
			pMRI-$\C$Net-K~\eqref{eq:loss_ablation}&  0.0538\\
			Proposed~\eqref{eq:loss_chp3}   & 0.0505 \\
			\bottomrule
		\end{tabular}
		\label{tab:ui_mse}
	\end{table}
	\subsection{Ablation Studies}\label{subsec:Ablation}
	The experiments introduced in this section implemented only in the image domain, in which \eqref{eq:scheme} is replaced by:
	\begin{subequations}\label{eq:bu_ablation}
		\begin{align}
			\bbf_i{(t)} & =  \ubf_i{(t)} - \rho_t \Fbf^{H}  \Pbf^{\top} (\Pbf \Fbf \ubf_i{(t)} - \fbf_i), \label{eq:b_ablation}  \\
			\ubf_i{(t+1)} & = \bbf_i{(t)} + \tilde{\J} \circ \tilde{\G} \circ \soft_{\alpha_t} ( \G \circ \J (\bbf_i{(t)})). \label{eq:u_ablation}
		\end{align}
	\end{subequations}
	$ \soft_{\alpha_t}$ represents soft shrinkage operator, we set $ \alpha_0 = 0 $ for both real and imaginary parts. Details for this model was explained in the conference paper \cite{10.1007/978-3-030-61598-7_2}. The proximal operator is learned in residual update \eqref{eq:u_ablation} and only performs in the image domain.
	
	We conduct a series of experiments to test the effects of several important components in the network proposed. Table \ref{tab:result_our} displays the comparison of proposed ablation study. (Labels of variations are explained in Tabel \ref{tab:labels}. The experiments without being labeled loss functions are trained with \eqref{eq:loss_body_ablation}). All the experiments in Table \ref{tab:result_our} were implemented by \eqref{eq:bu_ablation} except for the last one ``proposed'' method which unrolled \eqref{eq:scheme}.
	\begin{table}
		\centering
		\caption{Labels of the variations of the proposed pMRI reconstruction network.} 
		\begin{tabular}{ll}
			\toprule
			Label     &  Meaning \\
			\midrule
			-Net & Real-valued convolution/activation \\
			-$\C$Net & Complex-valued convolution/activation \\
			-RSS & Using root of sum of squares in place of $\J$\\
			-ZF & Zero-filling as initial $\ubf{(0)}$  \\
			-SP & SPIRiT as initial $\ubf{(0)}$ \\
			-K & $\Fbf^{H}(\fbf + \K(\fbf))$ as initial $\ubf{(0)}$ \\
			\bottomrule
		\end{tabular}
		\label{tab:labels}
	\end{table}
	To specify which of the components are employed, we append the labels shown in Table \ref{tab:labels} with different variations of the ablated pMRI networks.  For instance, the network with real-valued convolution and activation function with zero-filled initialization is denoted by pMRI-Net-ZF, etc.
	\begin{table*}
		\caption{Quantitative evaluations of the reconstructions on the Coronal FSPD \& PD data using different variations of the proposed methods.}
		\begin{center}
			\resizebox{\textwidth}{18mm}{ 
				\begin{tabular}{lcccccccc}
					\toprule
					&& FSPD data & & & PD data \\
					Method   & PSNR      & SSIM      & RMSE       & PSNR     & SSIM     & RMSE     & Phases T   & Parameters \\\midrule
					pMRI-$\C$Net-RSS & 36.4887$\pm$0.9787 & 0.9002$\pm$0.0197 &  0.0661$\pm$0.0051  & 41.2897$\pm$0.8430  & 0.9281$\pm$0.0357 & 0.0285$\pm$0.0037 & 5 & 5.03 M\\
					pMRI-Net-ZF & 37.8475$\pm$1.2086 & 0.9212$\pm$0.0236 & 0.0568$\pm$0.0069 & 42.4333$\pm$0.8785 & 0.9793$\pm$0.0023 & 0.0249$\pm$0.0024 & 5 & 5.03 M \\
					pMRI-Net-SP & 38.0205$\pm$0.8125 & 0.9291$\pm$0.0183 & 0.0555$\pm$0.0057 & 42.7435$\pm$0.4856 & 0.9754$\pm$0.0047 & 0.0239$\pm$0.0019 & 5 & 5.03 M \\
					pMRI-$\C$Net-ZF &  38.1157$\pm$1.3776 &  0.9277$\pm$0.0257 & 0.0552$\pm$0.0085 & 42.7859$\pm$1.1285  & 0.9818$\pm$0.0026 & 0.0241$\pm$0.0045 & 5 & 5.03 M\\
					pMRI-$\C$Net-SP & 38.3239$\pm$1.1305 & 0.9282$\pm$0.0269  & 0.0539$\pm$0.0075 & 42.8924$\pm$0.9336 & 0.9760$\pm$0.0054 & 0.0237$\pm$0.0034 & 5 & 5.03 M\\
					pMRI-Net-K & 38.8717$\pm$1.1330  & 0.9389$\pm$0.0209 & 0.0504$\pm$0.0057 &  42.9060$\pm$0.8765 &  0.9802$\pm$0.0028  &  0.0236$\pm$0.0028 & 4 & 4.11 M\\
					pMRI-$\C$Net-K \eqref{eq:loss_ablation} & 38.9661$\pm$1.4382  & 0.9421$\pm$0.0177 & 0.0498$\pm$0.0056 &  43.2604$\pm$0.7610 & 0.9833$\pm$0.0022 & 0.0226$\pm$0.0022 & 4 & 4.11 M\\
					pMRI-$\C$Net-K & 39.3360$\pm$1.1854 & 0.9497$\pm$0.0208 &  0.0477$\pm$0.0048 &  43.5653$\pm$0.8265 & 0.9844$\pm$0.0022 & 0.0217$\pm$0.0010 & 4 & 4.11 M\\
					Proposed \eqref{eq:loss_chp3} & \textbf{40.7101$\pm$1.5357} & \textbf{0.9619$\pm$0.0144} & \textbf{0.0408$\pm$0.0051} & \textbf{44.8120$\pm$1.3185} & \textbf{0.9886$\pm$0.0023} & \textbf{0.0189$\pm$0.0018} & 4 & 2.92 M\\
					\bottomrule
			\end{tabular}}
		\end{center}
		\label{tab:result_our}
	\end{table*}
	%Proposed ($K=1$)  &  39.8226$\pm$1.3430 & 0.9488$\pm$0.0191 & 0.0454$\pm$0.0068 & 44.1633$\pm$1.0341 & 0.9877$\pm$0.0020 & 0.0203$\pm$0.0017 & 4 & 2.92 M\\
	%Proposed ($K=2$)  &  40.5482$\pm$1.2235 & 0.9613$\pm$0.0224 & 0.0435$\pm$0.0062 & 44.6822$\pm$1.2105 & 0.9882$\pm$0.0025 & 0.0193$\pm$0.0019 & 4 & 2.92 M\\
	The complete structure of the network with the three types of initialization is shown in Fig.~\ref{fig:framework}. For $ t =1, \cdots, T$, each phase follows the algorithm introduced in our previous work \cite{10.1007/978-3-030-61598-7_2}.
	\begin{figure}[t]
		\centering
		\includegraphics[width=1\linewidth]{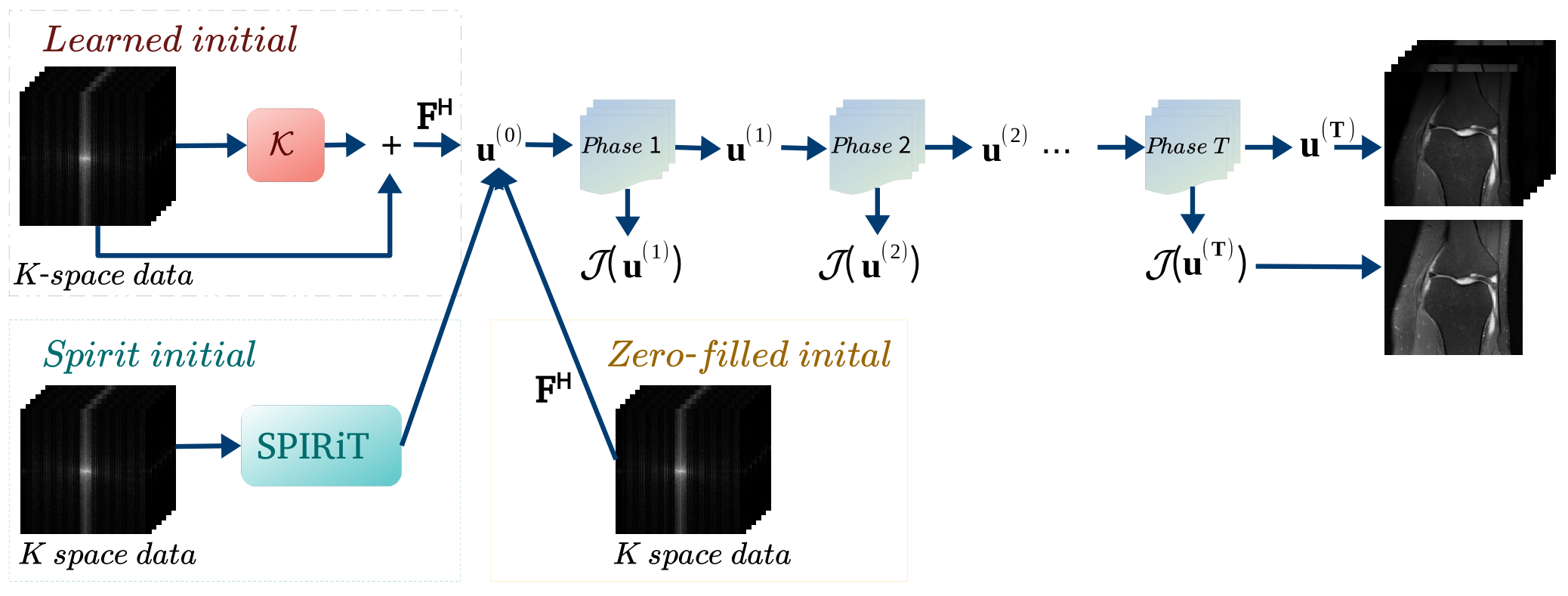}
		\caption{The framework paradigm for all phases with three different initial reconstruction methods, including zero-filled initial, SPIRiT reconstruction initial, and the learned initial.}  
		\label{fig:framework}
	\end{figure}
	
	We consider two types of training datasets in the ablation experiments and design a proper loss function for each of them in Section \ref{subsec:Ablation}.
	In the first case, the ground truth is $\ubf^*$, we set loss function:
	\begin{equation}
		\begin{aligned}
			\ell(\ubf) = & \sum\nolimits^{c}_{i=1}  \gamma \| \ubf_i - \ubf^*_i\| + \| |\J(\ubf)| - \text{RSS}(\ubf^*)\|\\
			&+ \beta  \|  \text{RSS}(\ubf_i{(0)}) - \text{RSS}(\ubf^*_i)\|, \label{eq:loss_ablation}  
		\end{aligned}
	\end{equation}
	$\ubf$ means $\ubf{(T)}$ for simplicity. If the training dataset consists ground truth single-body image $\vbf^* \in \mathbb{R}^{m \times n} = \text{RSS}(\ubf^*)$, then we test our network performance using the loss function:
	\begin{equation}
		\label{eq:loss_body_ablation}
		\ell(\ubf) = \gamma \| \text{RSS}(\ubf) - \vbf^*\| + \| |\J(\ubf)| - \vbf^*\| + \beta\| \text{RSS}( \ubf{(0)}) - \vbf^*\|.
	\end{equation}
	
	The loss functions are indicated in Tables \ref{tab:result_our} and \ref{tab:ui_mse} follow the experiments.  We set $ \gamma = 1, \beta = 10^{-3}$ in \eqref{eq:loss_ablation} and \eqref{eq:loss_body_ablation} in the implementations.
	%The proposed method uses loss function \eqref{eq:loss_chp3}, where multi-coil images $\ubf^*$ are used as ground truth. In the experiment of pMRI-$\C$Net-K, we use the loss function  \eqref{eq:loss_ablation} and \eqref{eq:loss_body_ablation}.  Both  \eqref{eq:loss_chp3} and  \eqref{eq:loss_ablation} minimize the discrepancy between the multi-coil reconstruction $\ubf$ and the ground truth $\ubf^*$. 
	
	%
	\subsubsection{Combination operator  vs. Root of Sum of Square}\label{subsec:RSS}
	In order to justify the effectiveness of the learned nonlinear combination operator $\J$, we modified pMRI-$\C$Net-ZF by substituting $ \J$ with RSS which is widely used to combine multi-coil images into a single-body image.
	The other operators $ \G,\tilde{\G}$ and $ \tilde{\J}$ remain to perform complex convolutions.
	Specifically, the output of RSS: $(\sum_{i=1}^{c} |\ubf_i{(t)}|^2)^{1/2} \in\mathbb{R}^{mn} $ is a single-channel real-valued image, which is input into both real and imaginary parts of the nonlinear operator $ \G$, so the output split into complex value. We refer  pMRI-Net-ZF/pMRI-$\C$Net-ZF as the networks with combination operator $\J$ and pMRI-$\C$Net-RSS as the network with RSS, all other settings remain the same as before.
	
	The reconstructed images and average evaluation results of these two types of networks are shown in Fig.~\ref{FSPD} and Table \ref{tab:result_our}. pMRI-Net-ZF and pMRI-$\C$Net-ZF outperform pMRI-$\C$Net-RSS with mean improvements of 1.36 dB and 1.63 dB in PSNR respectively. It indicates that the learned combination operator $\J$ gives more favorable performance compared with applying RSS.
	\subsubsection{Initialization}\label{subsec:initial}
	The choice of the input of the reconstruction network $\ubf{(0)}$, also has impacts on the final reconstruction quality. Instead of directly using the partial k-space data $\fbf$ as the input of our network, we use three different choices of input $\ubf{(0)}$: (i) Zero-filling reconstruction $\Fbf^{H} \fbf$; (ii) SPIRiT \cite{doi:10.1002/mrm.22428} reconstruction; (iii) $\Fbf^{H}(\fbf+\K(\fbf))$, one can treat $\fbf+\K(\fbf)$ as an interpolated pseudo full k-space.
	
	%In the test of pMRI-Net-SP/pMRI-$\C$Net-SP, we normalized the maximum value in the magnitude of multi-coil image intensity to be one.
	We observe that from Table \ref{tab:result_our},
	SPIRiT initial makes a slight improvement over the zero-filled initial, whereas the learned initial achieves the highest reconstruction quality compared to the other two initializations. 
	Fig.~\ref{Init} displays the three types of initials  for pMRI-Net-ZF/pMRI-$\C$Net-ZF, pMRI-Net-SP/pMRI-$\C$Net-SP, pMRI-Net-K/pMRI-$\C$Net-K, and reference image on the Coronal PD knee image. The second row shows their pointwise error maps and color bar.. We observe that both SPIRiT and the learned initial obtain higher spatial resolution over zero-filling.
	SPIRiT is a classical k-space method, this initial did a better job on reducing the aliasing artifacts and keeping edges compare to zero-filling and the learned initial, but SPIRiT introduces more noise. 
	
	Comparing to zero-filling, the learned initial preserves structure features of major tissue thanks to the k-space network $\K$. Comparing to the SPIRiT initial, the learned initial reduces resolution noises in the image space. Learning-based initial obtain a balanced performance between zero-filling and SPIRiT in the sense of avoiding the weakness of these two initials.
	
	\begin{figure}
		\centering
		\includegraphics[width=0.24\linewidth]{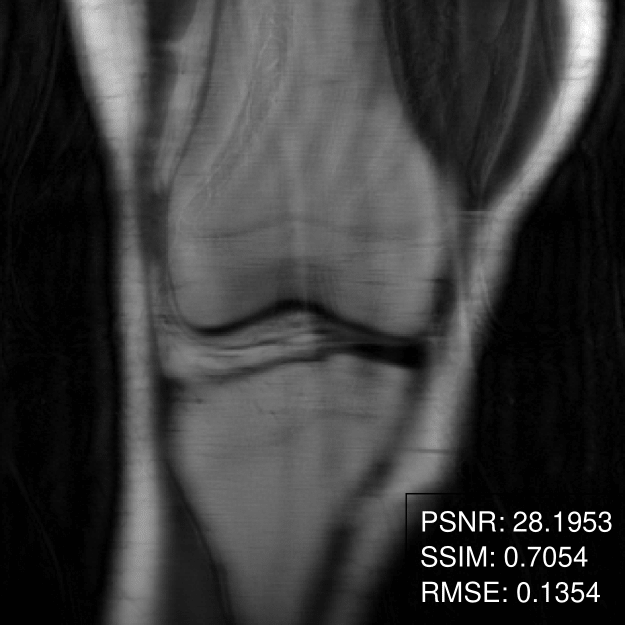}
		\includegraphics[width=0.24\linewidth]{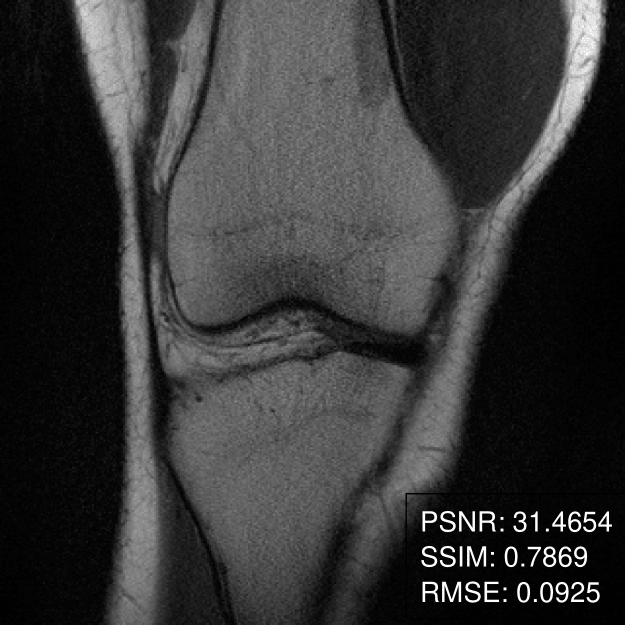}
		\includegraphics[width=0.24\linewidth]{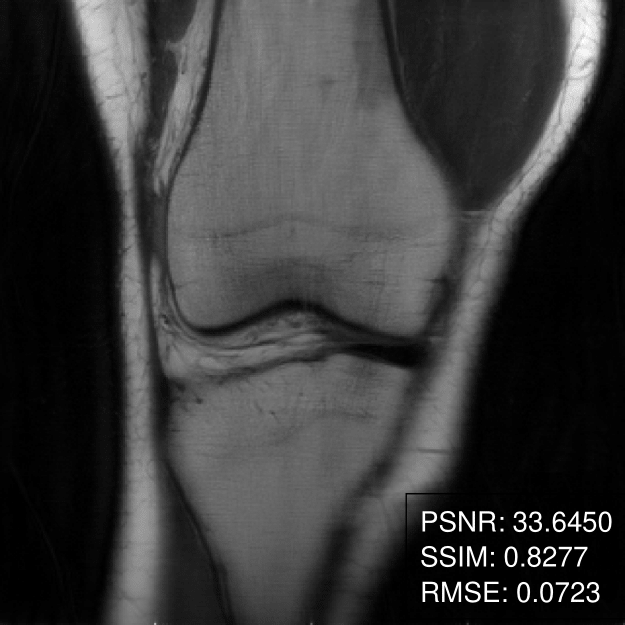}
		\includegraphics[width=0.24\linewidth]{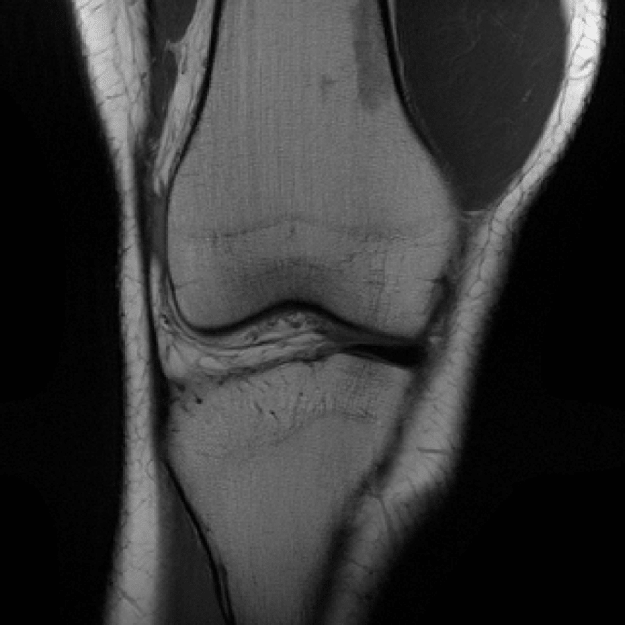}\\
		\includegraphics[width=0.24\linewidth]{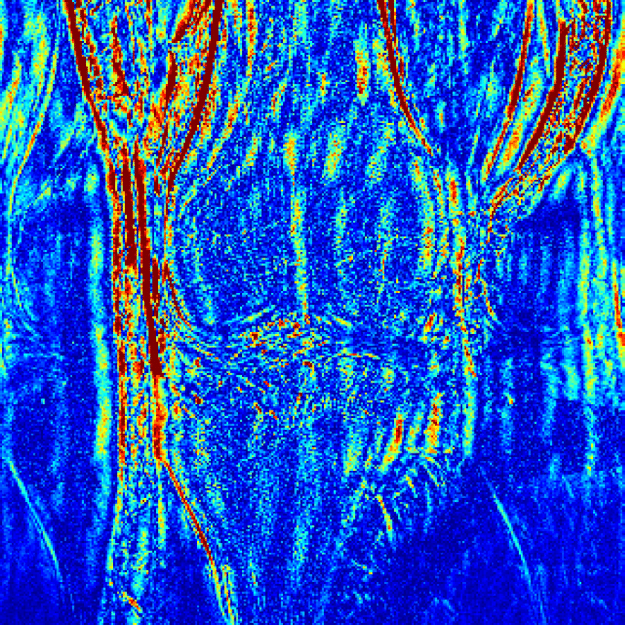}
		\includegraphics[width=0.24\linewidth]{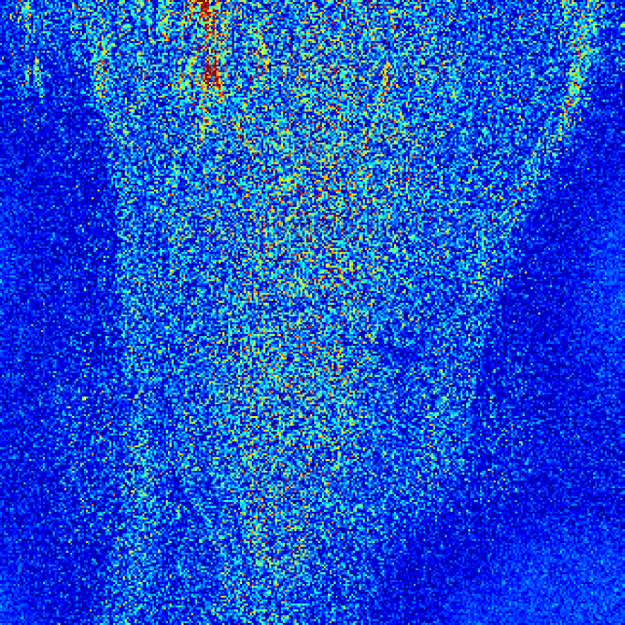}
		\includegraphics[width=0.24\linewidth]{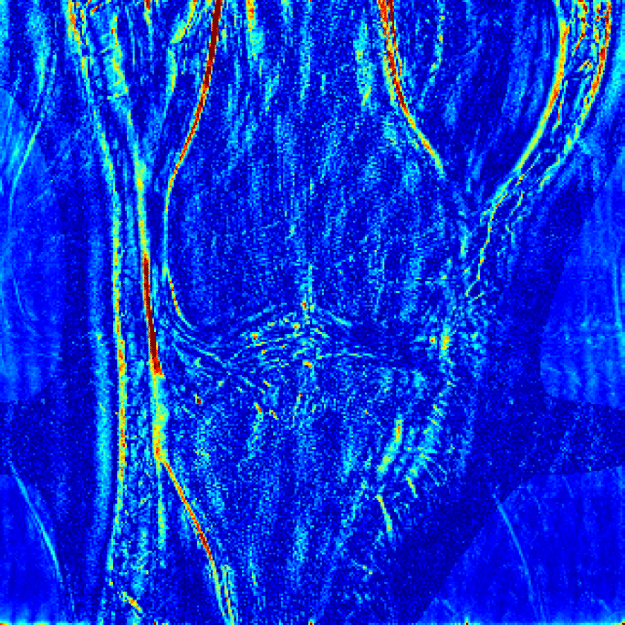}
		\includegraphics[width=0.24\linewidth]{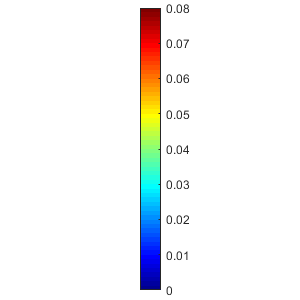}
		\caption{The first row (from left to right) shows the RSS of initial $ \Bar{\ubf}{(0)}$. }
		\label{Init}
	\end{figure}

	\subsubsection{Complex convolutions}\label{subsec:complexnetwork}
	
	%Complex-valued convolution and activation have been shown effective in recovering both magnitude and phase information in MRI reconstructions \cite{WANG2020136, cole2020analysis, vasudeva2020covegan}.
	In this experiment, we compared -Net and -$\C$Net.
	The quantitative evaluation from Table \ref{tab:result_our} indicates that the proposed complex-valued networks are outstanding in terms of PSNR/SSIM/RMSE over proposed real-valued networks. %Fig.~\ref{FSPD} shows sample reconstructions of -ZF initial and -K initial.
	Table \ref{tab:ui_mse} informs the complex convolutions extract the features in each $\ubf_i$ and obtain the lowest RMSE. The phase image for one channel of the reconstructed coil-image is displayed in Fig.~\ref{phase}. The first row shows phase information of one coil in the reconstructed coil-image for GRAPPA, SPIRiT, De-AliasingNet, DeepcomplexMRI,  CDF-Net, Adaptive-CS-Net, pMRI-$\C$Net-K with loss function \eqref{eq:loss_ablation}, the proposed method and referenced image. The second row shows the corresponding pointwise error maps and color bar, the maximum error is $30^{\circ}$.
	These results demonstrate complex-valued networks are playing important roles in updating multi-coil images and preserving phase information of each channel.
	
	\begin{figure*}
		\includegraphics[width=0.10\linewidth, angle=180]{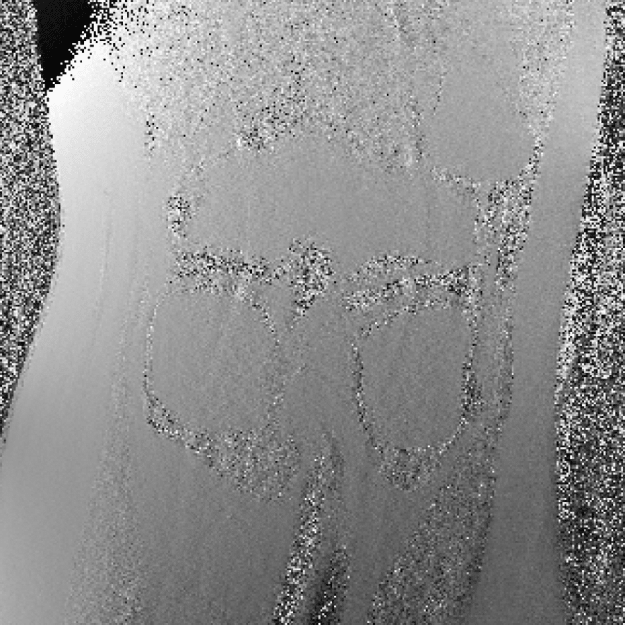}
		\includegraphics[width=0.10\linewidth, angle=180]{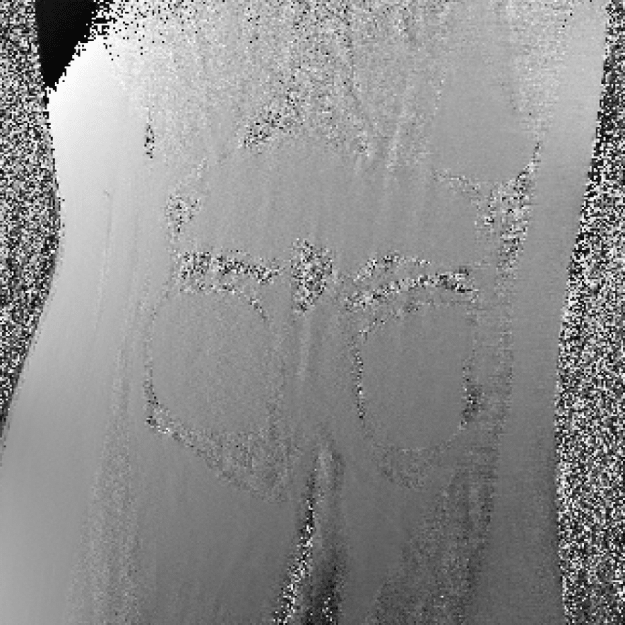}
		\includegraphics[width=0.10\linewidth, angle=180]{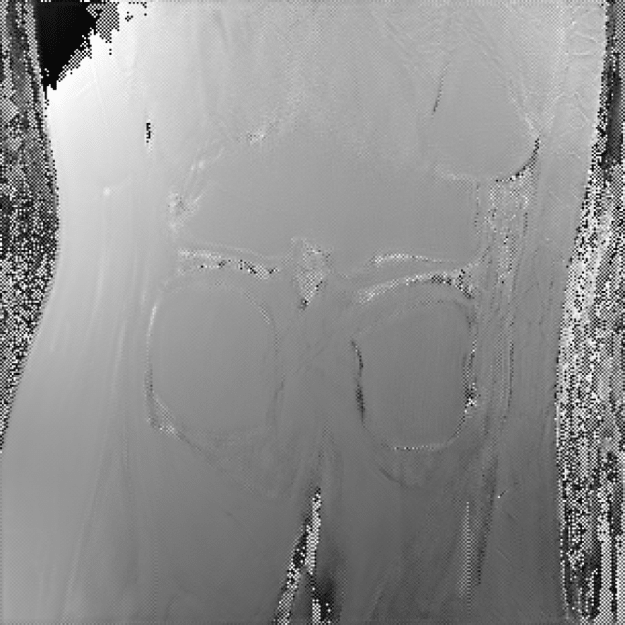}
		\includegraphics[width=0.10\linewidth, angle=180]{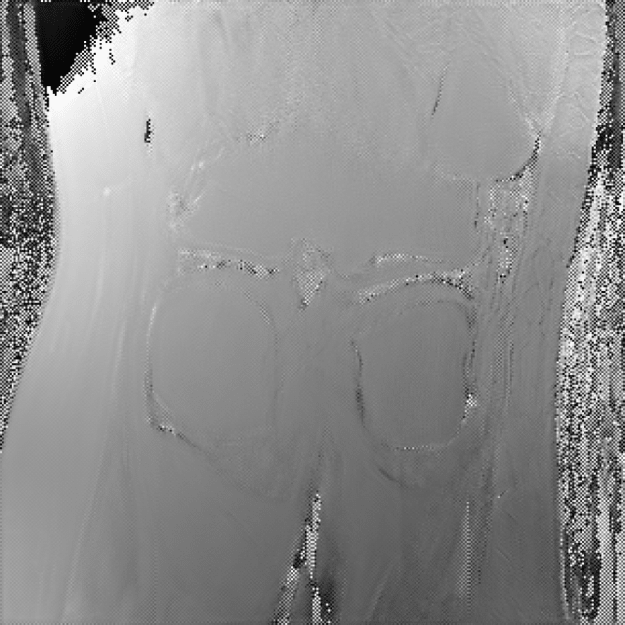}
		\includegraphics[width=0.10\linewidth, angle=180]{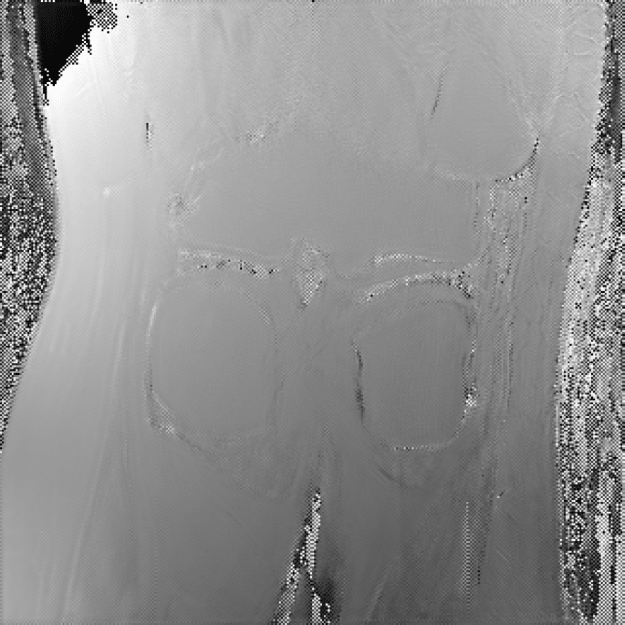}
		\includegraphics[width=0.10\linewidth, angle=180]{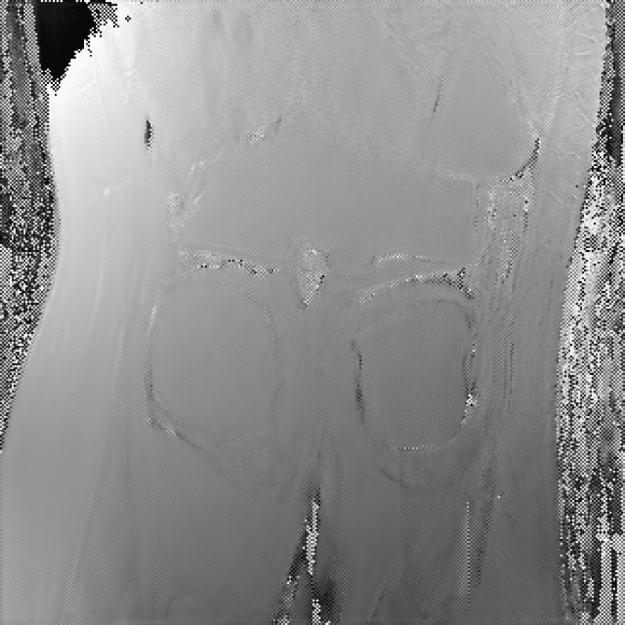}
		\includegraphics[width=0.10\linewidth, angle=180]{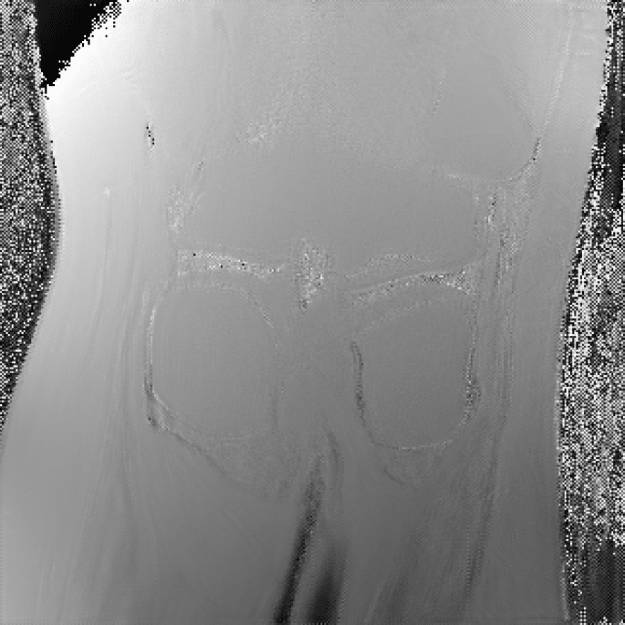}
		\includegraphics[width=0.10\linewidth, angle=180]{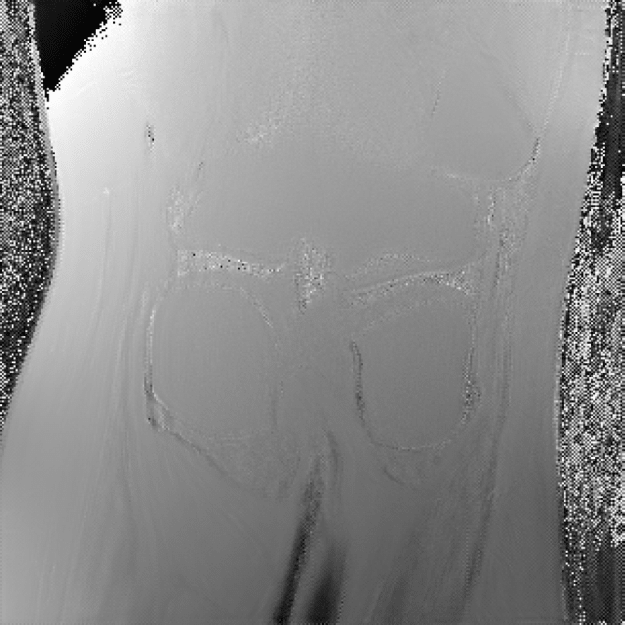}
		\includegraphics[width=0.10\linewidth, angle=180]{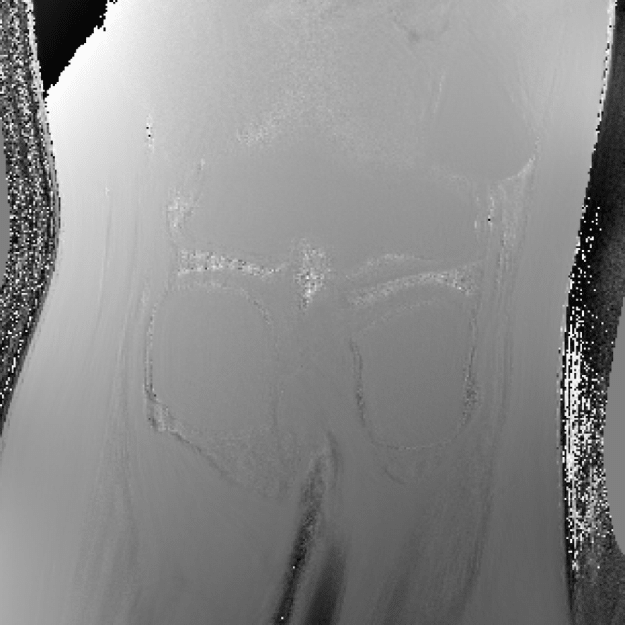}\\
		\includegraphics[width=0.10\linewidth, angle=180]{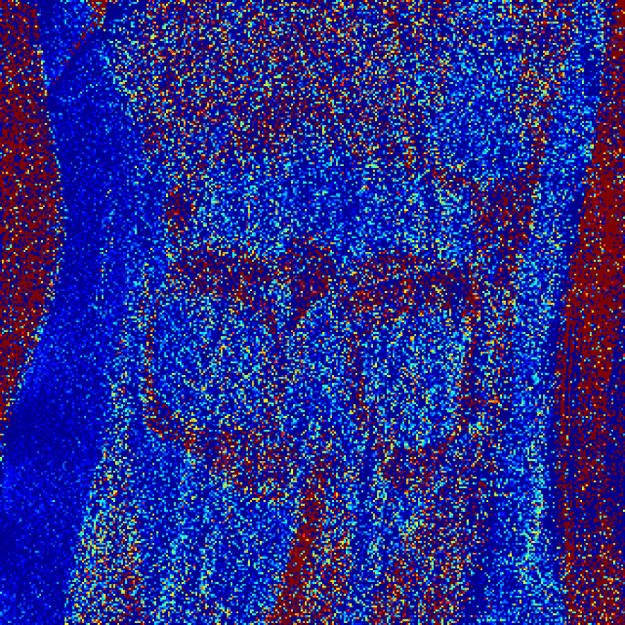}
		\includegraphics[width=0.10\linewidth, angle=180]{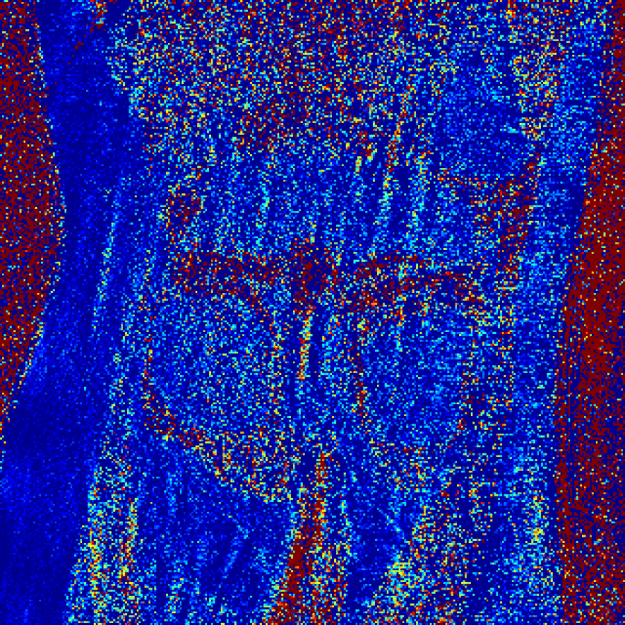}
		\includegraphics[width=0.10\linewidth, angle=180]{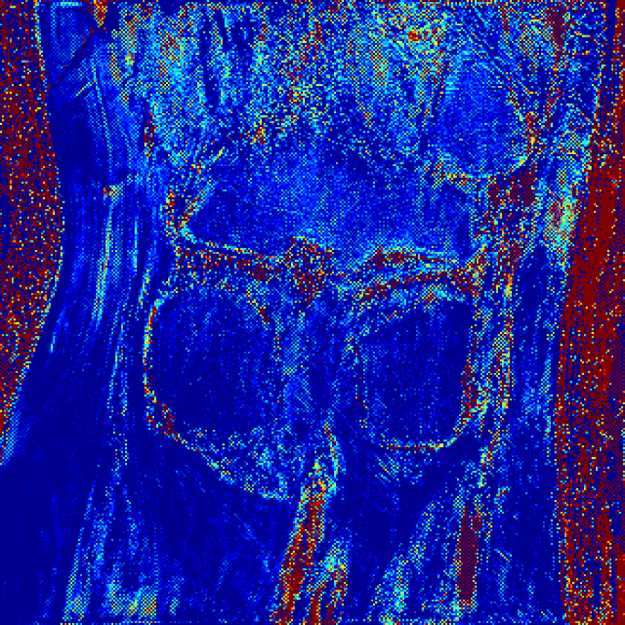}
		\includegraphics[width=0.10\linewidth, angle=180]{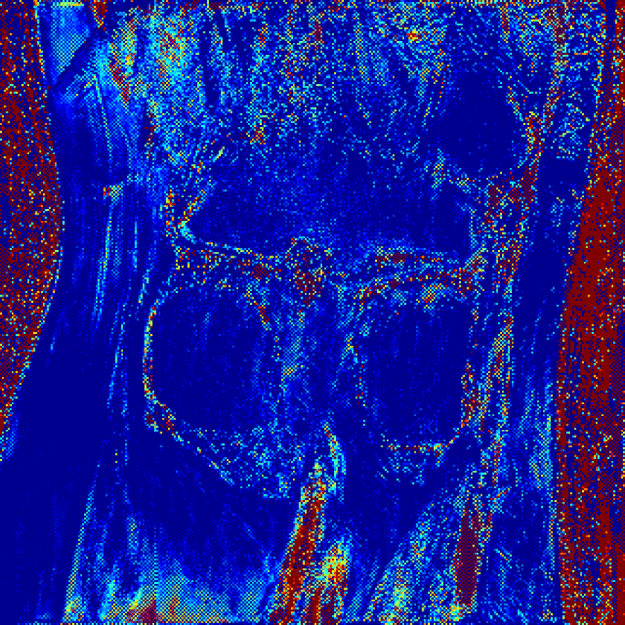}
		\includegraphics[width=0.10\linewidth, angle=180]{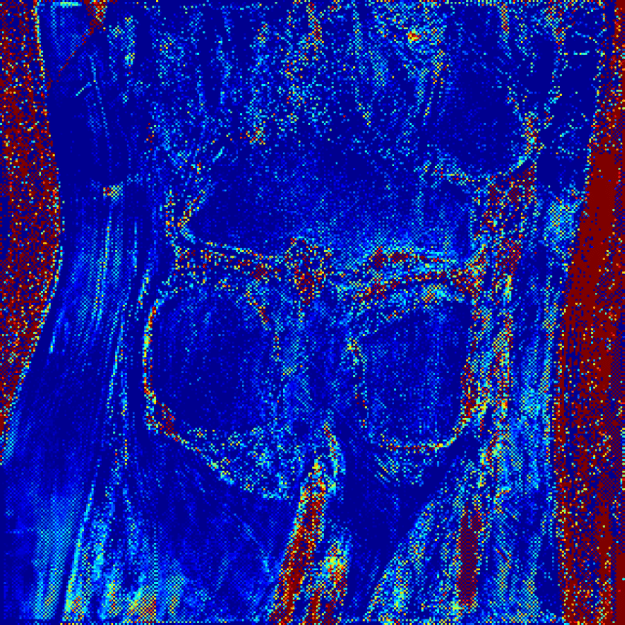}
		\includegraphics[width=0.10\linewidth, angle=180]{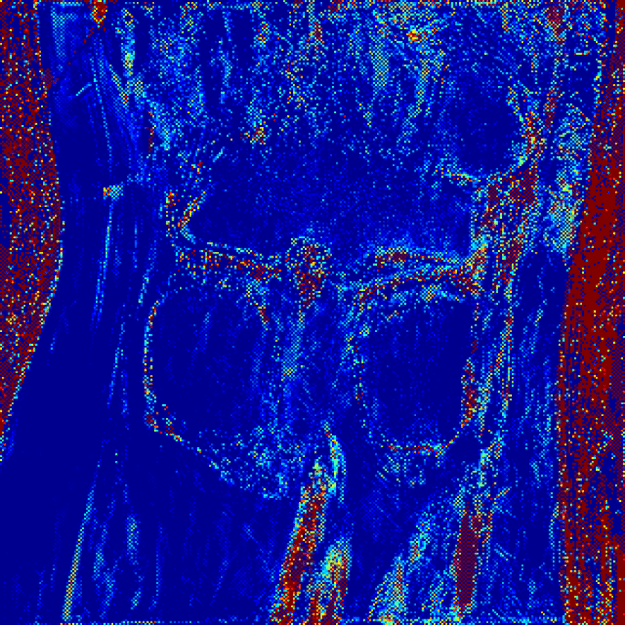}
		\includegraphics[width=0.10\linewidth, angle=180]{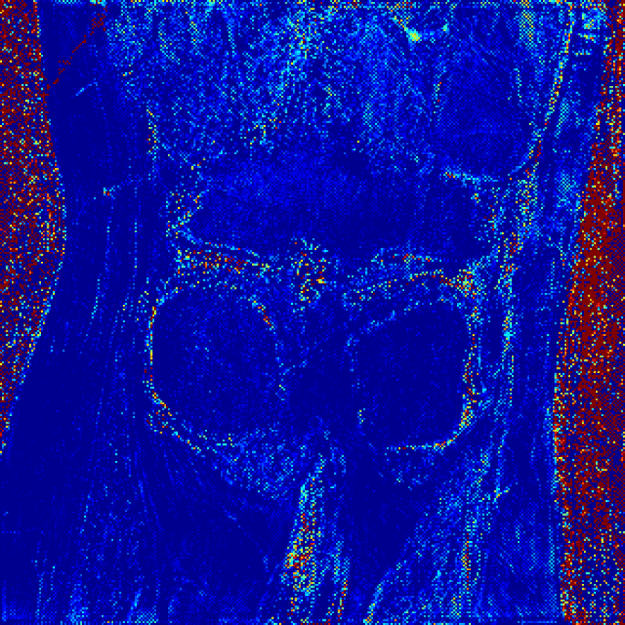}
		\includegraphics[width=0.10\linewidth, angle=180]{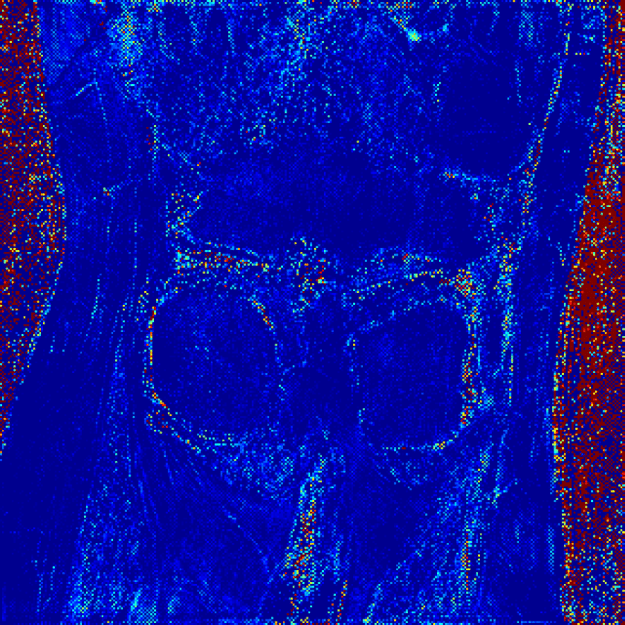}
		\includegraphics[width=0.10\linewidth, angle=180]{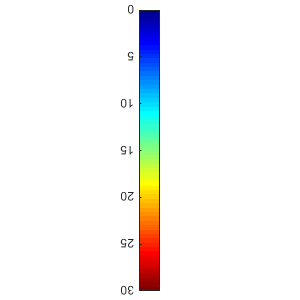}
		\caption{Phase information of one coil in the reconstructed coil-image for different methods comparing with the proposed method and referenced image. }
		\label{phase}
	\end{figure*}
	
	\subsubsection{Comparison with Proposed Image Domain Reconstruction and Domain Hybrid Reconstruction}
	\label{sec:discussion}
	The major difference between pMRI-$\C$Net-K and the proposed network is that pMRI-$\C$Net-K  only iterates \eqref{eq:u_ablation} in the image domain, while the proposed method iterates \eqref{eq:scheme} which conducts a domain hybrid reconstruction. The denoising network in \eqref{eq:u_ablation} is $ \tilde{\J} \circ \tilde{\G} \circ \soft_{\alpha_t}(  \G \circ \J)$, and in the proposed network we use $ \M + \Fbf^{H} \K \Fbf + \Fbf^{H} \K \Fbf \M$, where  $ \M = \tilde{\J} \circ \tilde{\G} \circ \G \circ \J $ is the image domain network. The soft-thresholding operator $  \soft_{\alpha_t}$ was eliminated in the proposed method since we found the results of adding the soft-thresholding does not make an obvious difference.
	
	The average reconstruction outcomes in Table \ref{tab:result_our} and Table \ref{tab:ui_mse} suggest that domain hybrid approach achieves better performance. Comparing to the pMRI-$\C$Net-K, our proposed method improved 0.598 dB in PSNR, 0.003 in SSIM, and reduced 0.014 in RMSE, which shows in Table \ref{tab:result_our} for PD dataset. 
	\section{Conclusion}
	\label{sec:conclusion_chp3}
	This paper introduces a discrete-time optimal control framework for the calibration-free pMRI reconstruction model. We apply a convolutional combination operator to combine channels of the multi-coil images and apply a parametrized regularization function to the channel-combined image to reconstruct channel-wise multi-coil images. The proposed method is inspired by the proximal gradient algorithm. The proximal point is learned by two denoising networks, which conducts in the image domain and k-space domain.
	We cast the reconstruction network as a structured discrete-time optimal control system, resulting in an optimal control formulation of parameter training, which provides an interpretable and high-performance deep architecture for pMRI reconstruction. We design network training from the view of the Method of Lagrangian Multipliers. We showed that the method of Lagrangian multipliers is equivalent to back-propagation, and we can employ SGD based algorithms to obtain a solution satisfying the necessary condition of the optimal control problem.
	The reconstruction results are of high perceived quality demonstrate the superior performance of the proposed pMRI-Net.

%% file: tex/chapter4.tex
\chapter{An Optimization-Based Meta-Learning Model for MRI Reconstruction with Diverse Dataset}\label{meta_learning}

This work aims at developing a generalizable MRI reconstruction model in the meta-learning framework.
The standard benchmarks in meta-learning are challenged by learning on diverse task distributions. The proposed network learns the regularization function in a variational model and reconstructs MR images with various under-sampling ratios or patterns that may or may not be seen in the training data by leveraging a  heterogeneous dataset.
We propose an
unrolling network induced by learnable optimization algorithms (LOA) for solving our nonconvex nonsmooth variational model for MRI reconstruction. In this model, the learnable regularization function contains a  task-invariant common feature encoder and task-specific learner represented by a shallow network. To train the network we split the training data into two parts: training and validation, and introduce a bilevel optimization algorithm. The lower-level optimization trains task-invariant parameters for the feature encoder with fixed parameters of the task-specific learner on the training dataset, and the upper-level optimizes the parameters of the task-specific learner on the validation dataset.
The PSNR increases 1.5 dB on average compared to the network trained through conventional supervised learning on the seen CS ratios. We test the result of quick adaption on the unseen tasks after meta-training, the average PSNR arises 1.22 dB compared to the conventional learning procedure that is directly trained on the unseen CS ratios in the meanwhile saving half of the training time. The average PSNR arises 1.87 dB for unseen sampling patterns comparing to conventional learning; 
We proposed a meta-learning framework consisting of the base network architecture,  design of regularization, and bi-level optimization-based training. The network inherits the convergence property of the LOA and interpretation of the variational model. The generalization ability is improved by the designated regularization and bilevel optimization-based training algorithm.   

\section{Introduction}

Deep learning methods have demonstrated promising performance in a variety of image reconstruction problems. However, deep learning models are often trained for specific tasks and require the training samples to follow the corresponding distribution. In particular, the source-domain/training samples and target-domain/testing samples need to be drawn from the same distribution {\cite{li2018learning,quinonero2009dataset,ben2007analysis,balaji2018metareg}. 
%
%However, this assumption is too strong and may not hold in practical applications. 
%
In practice, these data sets are often collected at different sources and exhibit substantial heterogeneity, and thus the samples may follow related but different distributions in real-world applications {\cite{fan2014challenges,day2017survey}}.
%
%Hence, generalizability becomes an inevitable concern of deep learning. 
%
Therefore, the robust and efficient training of deep neural networks using such data sets is theoretically important and practically relevant in the application of deep learning-based methods.
%
%(and will arise "domain shift/heterogeneity" problem, in which)(limited on future unseen data especially when the model was trained using small datasets.) (Fitting a deep learning model to a wide range of tasks with diverse data distribution is a major challenge in deep learning. This is because deep) A number of methods have been proposed to tackle this problem, including Cross-Validation \cite{browne2000cross}, Early-Stopping \cite{caruana2001overfitting, lodwich2009evaluation}, Dropout \cite{srivastava2014dropout}, Regularization \cite{xu2010robust,jaiswal2018investigation}, Ensembling \cite{perrone1992networks, zhou2002ensembling, krogh1996learning}. \ye{limitations of these methods? It would be more natural to discuss them before introducing meta-learning.}

Meta-learning provides a unique paradigm to achieve robust and efficient neural network training \cite{munkhdalai2017meta, finn2017model,li2018learning, rusu2018meta,  yao2021improving, balaji2018metareg}.
Meta-learning is known as %Is the italics necessary? If not, please make to normal font.
 \emph{learning-to-learn} and aims to quickly learn unseen tasks from the experience of learning episodes that cover the distribution of relevant tasks. In a multiple-task scenario, given a family of tasks, meta-learning has been proven to be a useful tool for extracting task-agnostic knowledge and improving the learning performance of new tasks from that family \cite{thrun1998learning, hospedales2021meta}.
%
%Most of the modern approaches of meta-learning focus on \emph{domain adaptation}, where the target domain information is available in meta-training, but this setting is still too ideal to apply in practice. Recently the study on \emph{domain generalization (DG)} has attracted much attention. DG allows the training model to learn representations from several related source domains and gain adequate generalization ability for unseen test distributions. However, DG techniques for image reconstruction are rarely explored in solving inverse problems. Leveraging large-scale heterogeneous MRI data to overcome overfitting issue caused by small-scale clinic datasets is of great interests for more precise, predictable, and powerful health care. 
%
We leverage this feature of meta-learning for network training where the MRI training data are acquired by using different under-sampling patterns (e.g., Cartesian mask,  Radial mask, Poisson mask), under-sampling ratios,  and different settings of the scanning parameters, which result in different levels of contrast (e.g., T1-weighted, T2-weighted,  proton-density (PD), and Flair). 
These data are vastly heterogeneous and can be considered as being from various tasks. 
%
%This work aims at developing a robust and generalizable MRI reconstruction method to leverage such large-scale heterogeneous data.
%to reconstruct MR images with the available small-scale dataset from specific setting of the scan.
%(only)  (The objective of this paper focuses on domain generalization for solving inverse problems and utilizes the domain invariant parameters to make predictions on the unseen domain through a meta-learning procedure.)In this paper we propose to unroll the deep network in a learned gradient descent algorithm as the forward model, and the network parameters are trained under the guidance of meta-knowledge which is learned in an optimization algorithm inspired by bilevel learning.
%(Medical image processing often suffers from data heterogeneous problems because of the inconsistency between the distribution of the training dataset and the testing dataset. Different medical image sets could be scanned from different scanners with different scanning coefficients and inharmonious protocols.  Also, acquiring labeled medical images are more expensive and laborious than natural images since medical images require intensive workforce who are expert in the radiologist or physicists with medical experiences for diagnostics. )
%
%In this work, we propose a robust and generalizable deep learning method meta-learning-based model for solving the MRI image reconstructions problem by leveraging diverse/heterogeneous dataset.
%
Thus, our goal is to develop a robust and generalizable image reconstruction method in the meta-learning framework to leverage such large-scale heterogeneous data for MRI reconstruction.

Our approach can be outlined as follows. First, we introduce a variational model rendering a nonconvex nonsmooth optimization problem for image reconstruction. In our variational model, the regularization term is parameterized as a structured deep neural network where the network parameters can be learned during the training process. We then propose a learnable optimization algorithm (LOA) with rigorous convergence guarantees to solve this optimization problem. Then, we construct a deep reconstruction network by following this LOA exactly; namely, each phase of this LOA-induced network is exactly one iteration of the LOA. This approach is inspired by \cite{chen2020learnable}, but the LOA developed in the present work is computationally more efficient than that in \cite{chen2020learnable}: the safeguard variable in this work is updated only if necessary, which can significantly reduce computational cost while retaining the convergence guarantee. 
%
%architecture is the same as algorithm, enherit %modified to improve the model performance and becomes more efficient, which slashes the computational cost. 

Second, to improve network robustness and mitigate the overfitting issue, we explicitly partition the network parameters of the regularization into two parts: a task-invariant part and a task-specific part. The former extracts common prior information of images from different tasks in the training data and learns the task-invariant component of the regularization. The latter, on the other hand, is obtained from another network which exploits proper task-specific components (also called meta-knowledge) of the regularization for different tasks. The hyperparameters of this network are also learned during training.
Furthermore, we split the available data into two sets: the training set and the validation set. Then, we introduce a bilevel optimization model for learning network parameters. Specifically, the lower-level problem (also known as inner problem) finds the task-invariant part of the regularization with the fixed task-specific part on the training dataset, whereas the upper-level (outer) problem seeks for the optimal task-specific part of the regularization parameter using the validation dataset. 
This approach greatly increases the robustness of the learned regularization, meaning that the trained LOA-induced deep reconstruction network can generalize well to unseen tasks.}
%
%The well-trained network of seen tasks can be applied to the unseen tasks with determined $\theta$. This adaptation only needs a few iterations to update $\omega_i$ with a small number of training samples of unseen tasks.
%Different from the aspiration of most optimization-based meta-learning, which searching for a universal network parameter that can generalize easily to multiple tasks, the adaptive regularization consists of task-invariant parameters $\theta$ and task-specific parameters (also called meta-knowledge) $\omega_i$.(The parameters $\theta$ are trained under the guidance of $\omega_i$ which is learned in an optimization algorithm inspired by bilevel learning in \eqref{eq:bi-level} employing training and validation data. The bilevel training enhances generalizability and reduces the overfitting of the proposed model.)

%In the proposed work, instead of directly learning a feed-forward network, we propose to learn a regularization of variational model solved by the nonsmooth nonconvex optimization algorithm in Section \ref{network}. Then we unroll it to a multi-phase shallow deep network that applies only a few kernels and convolutions of CNNs. 
%The learned regularization term can incorporate the common underlying properties of the training tasks instead of simply "fitting" the training data. 
%The common parameters in regularization are learned from a  heterogeneous dataset using method of meta-learning.
%The proposed meta-learning model with learned regularization unrolling can achieve great generalizability even on heterogeneous data. 

As demonstrated by our numerical experiments in Section \ref{experiments}, our proposed framework yields much improved image qualities using diverse data sets of various undersampling trajectories and ratios for MRI image reconstruction. 
The reason is that effective regularization can integrate common features and prior information from a variety of training samples from diverse data sets, but they need to be properly weighed against the data fidelity term obtained in specific tasks (i.e., undersampling trajectory and ratios).  %
%
% with insufficient training data through meta-training procedure so that each individual task benefit and make compensation to each other by leveraging the related cross-task information from  heterogeneous MRI datasets among tasks. 
%
%In the final analysis, the proposed LOA-based algorithm plays a key role in learning the entire regularization term in the MRI reconstruction model.
%
Our contributions can be summarized as follows:

\begin{enumerate}[leftmargin=*]
\item An LOA inspired network architecture---our network architecture exactly follows a proposed LOA with guaranteed convergence. Thus, the network is more interpretable, parameter-efficient, and stable than existing unrolling networks.

\item Adaptive design of regularization---our adaptive regularizer consists of a task-invariant part and a task-specific part, both of which can be appropriately trained from data.
%
%\deleted{Unlike the existing meta-learning methods, the proposed approach of network training can learn adaptive regularizer from diverse data sets.}
%The task-invariant portion aims at exploiting the common features and shared information across all involved tasks, whereas the task-specific parameters only learn the regularization weight to properly balance the data fidelity and learned regularization in individual tasks. 

\item Improved network robustness and generalization ability---we improve the robustness of the network parameter training process by posing it as a bilevel optimization using training data in the lower-level and validation data in the upper-level.
This approach also improves the generalization ability of the trained network so that it can be quickly adapted to image reconstruction with new unseen sampling trajectories and produces high-quality reconstructions.
\end{enumerate}

%The proposed meta-learning method improves the reconstruction performance on diversified datasets over the conventional supervised learning.\ye{I think item 1 is no longer a contribution. Just focus on 2. We can say that we improve the training efficiency and reconstruction performance of LDA from diversified data sets using a meta-learning approach.}

The remainder of the paper is organized as follows.
\textls[-15]{In Section \ref{related_work}, we discuss related work for both optimization-based meta-learning and deep unrolled networks for MRI reconstructions. We propose our meta-learning model and the neural network in Section \ref{sec:Method} and describe the implementation details in Section \ref{sec:Implementation}. Section \ref{experiments} provides the numerical results of the proposed method. Section \ref{conclusion} concludes the paper.}

%%%%%%%%%%%%%%%%%%%%%%%%%%%%%%%%%%%%%%%%%%
\section{Related Work}\label{related_work}

In recent years, meta-learning methods have demonstrated promising results in various fields with different techniques \cite{hospedales2021meta}. 
Meta-learning techniques can be categorized into three groups  \cite{yao2020automated, lee2018gradient, huisman2021survey}: metric-based methods \cite{koch2015siamese, vinyals2016matching, snell2017prototypical}, model-based methods  \cite{mishra2017simple, ravi2016optimization, qiao2018few, graves2014neural}, and optimization-based methods \cite{finn2017model, rajeswaran2019meta, li2017meta}. Optimization-based methods are often cast as a bilevel optimization problem and exhibit relatively better generalizability for wider task distributions. We mainly focus on optimization-based meta-learning in this paper.
For more comprehensive literature reviews and developments of meta-learning, we refer the readers to the recent surveys \cite{ hospedales2021meta, huisman2021survey}. 

%\subsection{ Optimization-based {meta-learning} approaches}\label{l2l}

Optimization-based meta-learning methods have been widely used in a variety of deep learning applications \cite{finn2017model,antoniou2018train,rajeswaran2019meta,li2017meta,nichol2018first,finn2019online,grant2018recasting,finn2018probabilistic,yoon2018bayesian}. The network training problems in these meta-learning methods are often cast as the bilevel optimization of a leader variable and a follower variable. The constraint of the bilevel optimization is that the follower variable is optimal for the lower-level problem for each fixed leader variable, and the ultimate goal of bilevel optimization is to find the optimal leader variable (often, the corresponding optimal follower variable as well) that minimizes the upper-level objective function under the constraint. 
%
%For meta-learning applications, the lower-level problem encounters new tasks and tries to learn the associated features quickly from the training observations, the outer level accumulates task-specific meta-knowledge across previous tasks and the meta-learner provides support for the inner level so that it can quickly adapt to new tasks. e.g. 
%
The lower-level problem is approximated by one or a few gradient descent steps in many existing optimization-based meta learning applications, such as Model-Agnostic Meta-Learning (MAML) \cite{finn2017model}, and a large number of followup works of MAML proposed to improve  generalization using similar strategy \cite{lee2018gradient, rusu2018meta, finn2018probabilistic, grant2018recasting, nichol2018first, vuorio2019multimodal, yao2019hierarchically,yin2020metalearning}.
%Some variant models \cite{rusu2018meta, vuorio2019multimodal, yao2019hierarchically} focus on multiple initial conditions for fast learning which usually relies on fixed optimizers such as SGD with momentum or its variance.
Deep bilevel learning~\cite{jenni2018deep} seeks to obtain better generalization than when trained on one task and generalize well to another task. The model is used to optimize a regularized loss function to find network parameters from the training set and identify hyperparameters so that the network performs well on the validation dataset.

% Alternatively,  optimizer oriented methods \cite{andrychowicz2016learning, ravi2016optimization, li2016learning, wichrowska2017learned} focus on learning inner optimizer, the meta-knowledge $\omega$ can be used to define gradient-based optimization steps  for each base learning iteration.
%\ye{Check Huisman's 2021 survey paper "A survey of deep meta‐learning". It summarizes many optimization-based meta-learning approaches. None of them considers to solve the bi-level optimization rigorously though, so we can also claim that we try to solve it using Algorithm 1.}
%Hyperparameter optimization (HO) shares the same merit with meta-learning, the major difference is HO often consider a single task that split as train data  and validation data, the inner objective is the regularized empirical loss function on train data that seeks to tune the model parameters and the outer objective seeks to tune hyperparameters \cite{franceschi2018bilevel,pedregosa2016hyperparameter,franceschi2017forward,micaelli2020non}. 

%\subsection{Domain generalization and latest meta-learning benchmarks}
%The feature reuse is discovered to be the key to the fast adaption ability of most existing meta-learning algorithms \cite{raghu2019rapid}. Almost No Inner Loop (ANIL) \cite{raghu2019rapid} freezes the parameters in the network body during the inner loop that performs adaptation in MAML \cite{finn2017model} yet achieves comparable performance.
When the unseen tasks lie in inconsistent domains with the meta-training tasks, as revealed in \cite{chen19closerfewshot}, the generalization behavior of the meta-learner will be compromised.
This phenomenon partially arises from the meta-overfitting on the already seen meta-training tasks, which is identified as a memorization
problem in \cite{yin2020metalearning}.
A meta-regularizer forked with information theory is proposed in \cite{yin2020metalearning} to handle the memorization
problem by regulating the information dependency during the task adaption. 

MetaReg \cite{balaji2018metareg} decouples the entire network into the feature network and task network, where the meta-regularization term is only applied to the task network. They first update the parameters of the task network with a meta-train set to obtain the domain-aligned task network and then update the parameters of the meta-regularization term on the meta-test set to learn the cross-domain generalization. In contrast to MetaReg, Feature-Critic Networks \cite{li2019feature} exploit the meta-regularization term to pursue a domain-invariant feature extraction network. The meta-regularization is designed as a feature-critic network that takes the extracted feature as an input. The parameters of the feature extraction network are updated by minimizing the new meta-regularized loss. The auxiliary parameters in the feature-critic network are learned by maximizing the performance gain over the non-meta case.  
To effectively evaluate the performance of the meta-learner, several new benchmarks \cite{Rebuffi17,48798,yu2020meta} were developed  under more realistic settings that operate well on diverse visual domains. As mentioned in \cite{48798}, the generalization to unseen tasks within multimodal or heterogeneous datasets remains a challenge to the existing meta-learning methods. 

The aforementioned methods pursue domain generalization for the classification networks that learned a regularization function to learn cross-domain generalization. Our proposed method was developed to solve the inverse problem, and we construct an adaptive regularization that not only learns the universal parameters among tasks but also the task-aware parameters. The designated adaptive regularizer assists the generalization ability of the deep model so that the well-trained model can perform well on  heterogeneous datasets of both seen and unseen tasks.

\section{Proposed Method}\label{sec:Method}
%%%%%%%%%%%%%%%%%%%%%%%%%%%%%%%%%%%%%%%%%%

{\subsection{Preliminaries}}%Variational methods for compressive MRI reconstruction
%\subsection{Background on compressed sensing MRI (CS-MRI) reconstruction}

We first provide the background of compressed sensing MRI (CS-MRI), the image reconstruction problem, and the learned optimization algorithm to solve the image reconstruction problem. CS-MRI accelerates MRI data acquisition by under-sampling the k-space (Fourier space) measurements. The under-sampled k-space measurement are related to the image by the following formula \cite{haldar2010compressed}:
\begin{equation}\label{MRI}
    \ybf = \Pbf \F \xbf + \nbf,
\end{equation}
where $\ybf \in \C^p$ represents the measurements in k-space with a total of $p$ sampled data points, $ \xbf \in \C^{N\times 1}$ is the MR image to be reconstructed  with $N$ pixels, $ \F \in \C^{N \times N}$ is the 2D discrete Fourier transform (DFT) matrix, and $ \Pbf \in \R^{p \times N}$ $(p< N)$ is the binary matrix representing the sampling trajectory in k-space. $\mathbf{n}$ is the  acquisition noise in k-space.
%For multi-coil acquisitions in MRI, $  \Ebf = \Pbf \Fbf \Sbf_j \in \CC^{p\times N} $ where $ \Sbf_j \in \R^{N \times N}$ is a diagonal matrix called coil sensitivity map of the $j$th coil, which is either given or estimated in advance. Therefore $ \Ebf \xbf \in \CC^p$ is the vector of partial Fourier coefficients and k-space data acquisition is expressed as Equation \eqref{MRI}, where $ \nbf$ represents the measurement noise.

Solving $\xbf$ from (noisy) under-sampled data $\ybf$ according to \eqref{MRI} is an ill-posed problem. 
An effective strategy to elevate the ill-posedness issue is to incorporate prior information to the reconstruction. The variational method is one of the most effective ways to achieve this.
The general framework of this method is to minimize an objective function that consists of a data fidelity term and a regularization term as follows:
 \begin{equation}\label{csmodel-1}
   \bar{\xbf} = \argmin_{\xbf} \frac{1}{2} \| \Pbf \F \xbf - \ybf \|^2 + R(\xbf),
\end{equation}
where the first term  is data fidelity, which ensures consistency between the reconstructed $ \xbf$ and the measured data $\ybf$, and the second term $ R(\xbf)$ is the regularization term, which introduces prior knowledge to the image to be reconstructed. In traditional variational methods, $R(\xbf)$ is a
hand-crafted function such as Total Variation (TV) \cite{661180}. The advances of the optimization techniques allowed more effective
algorithms to solve the variational models with theoretical justifications. However, hand-crafted regularizers may be too simple to capture subtle details and satisfy clinic diagnostic quality.
{In recent years, we have witnessed the tremendous success of deep learning in solving a variety of inverse problems, but the interpretation and generalization of these deep-learning-based methods still remain the main concerns. As an improvement over generic black-box-type deep neural networks (DNNs), several classes of learnable optimization algorithms (LOAs) inspired neural networks, known as unrolling networks, which unfold iterative algorithms to multi-phase networks and have demonstrated promising solution accuracy and efficiency empirically
\cite{lundervold2019overview, liang2020deep, sandino2020compressed, mccann2017convolutional, zhou2020review, singha2021deep, chandra2021deep, ahishakiye2021survey}. However, many of them are only specious imitations of the iterative algorithms and hence lack the backbone of the variational model and any convergence guarantee.
%(Deep learning based model leverages large dataset and further explore the potential improvement of reconstruction performance comparing to traditional methods and has successful applications in clinic field) (However, training generic deep neural networks (DNNs) may prone to over-fitting when data is scarce as we mentioned in the beginning. Also, the deep network structure behaves like a black box without mathematical interpretation.To improve the interpretability of the relation between the topology of the deep model and reconstruction results, a new emerging class of deep learning-based methods known as \emph{learnable optimization algorithms} (LOA) have attracted much attention e.g. \cite{liu2020deep, liang2020deep}. LOA was proposed to map existing optimization algorithms to structured networks where each phase of the networks correspond to one iteration of an optimization algorithm or replace some ingredients.) %For instance, gradient decent algorithm based CNN \cite{hammernik2018learning}, proximal gradient inspired network \cite{cheng2019model, bian2020deep, zhang2018ista}, ADMM inspired \cite{yang2018admm}, Primal dual algorithm inspired \cite{adler2018learned, cheng2019model, heide2014flexisp,meinhardt2017learning}.
}
%\deleted{such as proximal operator \cite{cheng2019model, bian2020deep, zhang2018ista}, matrix transformations \cite{yang2018admm, hammernik2018learning, zhang2018ista}, non-linear operators \cite{yang2018admm, hammernik2018learning}, and denoiser/regularizer \cite{aggarwal2018modl, schlemper2017deep} etc., by CNNs to avoid difficulty for solving non-smooth non-convex problems. Different from the current LOA approaches for image reconstruction that simply imitate an iterative algorithm and  without any convergence justification or just replaces some hardly solvable components with sub-networks, our proposed LOA-induced network inherits the convergence property of the proposed Algorithm \ref{alg:lda}, where we provide convergence analysis in Appendix \ref{convergence}. The proposed network retains the interpretability of the variational model, parameter efficiency and contributes to better generalization.}\chen{Combine ours to the next section, this section is related works to our work}()

{
In light of the substantial success of deep learning and the massive amount of training data now available, we can parameterize the regularization term as a deep convolutional neural network (CNN) that learns from training samples. LOA-induced reconstruction methods have been successfully applied to CS-MRI to solve inverse problems with a learnable regularizer:
\begin{equation}\label{loa_model}
\argmin_{\xbf} \frac{1}{2} \| \Pbf \F \xbf - \ybf \|^2 + R(\xbf; \Theta) .
\end{equation}
where $R(\xbf; \Theta)$ is the regularization parameterized as a deep network with parameter $\Theta$.
Depending on the specific parametric form of $R(\xbf; \Theta)$ and the optimization scheme used for unrolling, several unrolling networks have been proposed in recent years.
%
%In what follows we introduce several  LOA-type methods that optimize CS-based MRI model. 
%
\textls[-15]{For example, the variational network (VN)~\cite{doi:10.1002/mrm.26977} was introduced to unroll the gradient descent algorithm and parametrize the regularization as a combination of linear filters and nonlinear CNNs. 
MoDL~\cite{Aggarwal_2019}~proposed a weight sharing strategy in a recursive network to learn the regularization parameters by unrolling the conjugate gradient method.}
\textls[-15]{ADMM-Net \cite{sun2016deep} mimics the celebrated alternating direction method of multipliers; the regularizer is designed to be $L_1$-norm replaced by a piecewise linear function.} %is a successful MRI reconstruction method with total-variation regularization combined with variable splitting technique. Their network architecture inspired by  Alternating Direction Method of Multipliers (ADMM) algorithm. 
\textls[-15]{ISTA-Net \cite{zhang2018ista} considers the regularizer as the $L_1$-norm of a convolutional network. The network unrolls several phases iteratively, and each phase mimics one iteration of iterative shrinkage thresholding  algorithm (ISTA)~\cite{lions1979splitting,beck2009fast}. }
However, these networks only superficially mimic the corresponding optimization schemes, but they lack direct relations to the original optimization method or variational model and do not retain any convergence guarantee. In this work, we first develop a learnable optimization algorithm (LOA) for \eqref{loa_model} with comprehensive convergence analysis and obtain an LOA-induced network by following the iterative scheme of the LOA exactly.
}

\subsection{LOA-Induced Reconstruction Network}
\label{network}

In this section, we first introduce a learned optimization algorithm (LOA) to solve \eqref{loa_model} where the regularization network parameter $\Theta$ is fixed. As $\Theta$ is fixed in \eqref{loa_model}, we temporarily omit this in the derivation of the LOA below and write  $R(\xbf; \Theta)$ as $R(\xbf)$ for notation simplicity. 
%
% Then we generate a multi-phase network induced by the proposed LOA, i.e., the network architecture follows the algorithm exactly such that one phase of the network is just one iteration of the LOA.
%(Hereafter we will simply write $\phi_{\theta,\omega_i}$ as $\phi$ without ambiguity.

\textls[-15]{In this work, to incorporate sparsity along with the learned features, we parameterize the function %Please check if the dot (.) in the following equation should be chnaged to multiple sign.
 $R (\xbf, \Theta)= \kappa \cdot r(\xbf, \Theta)$, where $\kappa>0$ is a weight parameter that needs be chosen properly depending on the specific task (e.g., noise level, undersampling ratio, etc.), and $r$ is a regularizer parameterized as a composition of neural networks and can be adapted to a broad range of imaging applications. Specifically, we parameterize $r$ as the composition of the $l_{2,1}$ norm and a learnable feature extraction operator $\gbf(\xbf)$. 
That is, we set $r$ in \eqref{model} to be}

\begin{equation}\label{eq:r_chp4}
r(\xbf) := \|\gbf(\xbf)\|_{2,1} = \sum_{j = 1}^{m} \|\gbf_{j}(\xbf)\|.
\end{equation}
%
%where the feature extraction operator $\gbf$ can be easily formulated to be smooth, but the regularization term $r(\xbf) $ defined above is still nonsmooth and possibly nonconvex due to the $l_{2,1}$ norm. 
%
Here, %MDPI: Please check if all of the paragraphs after the fomula need indentation .
{``:='' stands for ``defined as''.} $\gbf(\cdot) = (\gbf_1(\cdot),\dots, \gbf_m(\cdot))$, $\gbf_j(\cdot)=\gbf_j(\cdot;\theta)$ is parametrized as a convolutional neural network (CNN) for $j = 1,\cdots, m$, and $\theta$ is the learned and fixed network parameter in $r(\cdot;\theta)$, as mentioned above. We also consider $\kappa$ to be learned and fixed as $\theta$ for now, and we discuss how to learn both of them in the next subsection.
We use a smooth activation function in $\gbf$ as formulated in \eqref{eq:sigma}, which renders $\gbf$ a smooth but nonconvex function. Due to the nonsmooth $\|\cdot\|_{2,1}$, $r$ is therefore a nonsmooth nonconvex function.
%
% We assume \eqref{model} is a nonsmooth nonconvex problem, so the regularizer $R$ need to be smoothed and we replace $R$ as $R_{\varepsilon}$ with a smoothing parameter $\varepsilon$, our smoothing method formulated in equation \eqref{eq:l21}. 

Since the minimization problem in \eqref{model} is nonconvex and nonsmooth, we need to derive an efficient LOA to solve it. 
%
% This solver will be termed as $F_{\Theta}(\ybf)$.
%
Here, we first consider smoothing the $ l_{2,1}$ norm that for any fixed $\gbf(\xbf)$ is
\begin{equation}\label{eq:l21}
r_{\varepsilon} (\xbf) = \sum\nolimits^m_{j=1}  \sqrt{\| \gbf_{j} (\xbf) \|^2 + \varepsilon^2} -\varepsilon.
\end{equation}
We denote %Please check if the dot (.) in the following equation should be chnaged to multiple sign. Also apply to other highlights.
 $R_{\varepsilon} = \kappa \cdot r_{\varepsilon}$.
%
%In \eqref{model} the network was introduced in an implicit way, in the following text we will apply some optimization algorithm to solve for \eqref{model} which will then contribute to an explicit multi-phase network that is dubbed as algorithmic unrolling network with a fixed number of phases. 
%
The LOA derived here is inspired by the proximal gradient descent algorithm and iterates the following steps to solve the smoothed problem:

\begin{subequations}\label{prox_chp4}
    \begin{align}
        \ztp & = \xt - \alpha_t \nabla f(\xt) \label{prox_u}\\
    \xtp & = \prox_{\alpha_t R_{\epst} } (\ztp ), \label{prox_sub}
    \end{align}
\end{subequations}
where $\epst$ denotes the smoothing parameter $\varepsilon$ at the specific iteration $t$, and the proximal operator is defined as $ \prox_{\alpha g}(\bbf) := \argmin_{\xbf} (1/2)\left\| \xbf-\bbf \right\|^2 + \alpha g(\xbf)$ in \eqref{prox_sub}. A quick observation from \eqref{eq:l21} is that $R_{\varepsilon} \rightarrow R$ as $\varepsilon$ diminishes, so later we intentionally push $\epst \rightarrow 0$ at Line 16 in Algorithm \ref{alg:lda}. Then, one can readily show that $ R_{\varepsilon}(x) \le R(x) \le R_{\varepsilon}(x) + \varepsilon$ for all $x$ and $\varepsilon > 0$. From Algorithm \ref{alg:lda}, line 16  automatically reduces $\varepsilon$, and the iterates will converge to the solution of the original nonsmooth nonconvex problem \eqref{model}---this is clarified precisely in the convergence analysis in Appendix \ref{convergence}. %Please check that the meaning here is correct.
%\ye{You need a bound like $R_{\epsilon}(x) \le R(x) \le R_{\epsilon}(x) + O(\epsilon)$ for all $x$ and $\epsilon > 0$. Then just say that Alg.1 will automatically reduce $\epsilon$ and the iterates will converge to the solution of the original nonsmooth nonconvex problem \eqref{model}, a rigorous sense will be made precisely in the convergence analysis in Appendix.} 

Since $R_{\epst}$ is a complex function involving a deep neural network, its proximal operator does not have a closed form and cannot be computed easily in the subproblem in \eqref{prox_sub}.
To overcome this difficulty, we consider to approximate $R_{\epst}$ by
\begin{subequations}
\begin{align}
\hat{R}_{\epst} (\xbf) & = R_{\epst}(\ztp) + \nonumber \langle  \nabla R_{\epst}(\ztp), \xbf-\ztp \rangle + \frac{1}{2\beta_t} \norm{\xbf-\ztp}^2. \label{eq:u} 
\end{align}
\end{subequations}
Then, 
we update $ \utp  = \prox_{\alpha_t \hat{R}_{\epst}  } (\ztp )$ to replace \eqref{prox_sub}; therefore, we obtain
\begin{equation}\label{ut+1}
    \utp = \ztp -  \taut \nabla R_{\epst}(\ztp), \text{ where } \taut = \frac{\alpha_t \beta_t}{\alpha_t + \beta_t}.
\end{equation}
If condition line 5 in LOA \ref{alg:lda} satisfies, we will iterate $\xtp = \utp$. In order to guarantee the convergence of the algorithm, we introduce the standard gradient descent of $\phi_{\epst} $ (where $\phi_{\epst}  := f + R_{\epst}$) at $ \xbf$:
\begin{equation}\label{grad_dst}
    \vtp = \argmin_{\xbf} \langle \nabla f(\xt), \xbf - \xt \rangle + \langle \nabla R_{\varepsilon}(\xt) , \xbf - \xt \rangle + \frac{1}{2 \alpha_t} \| \xbf - \xt \|^2,
\end{equation}
which yields 
\begin{equation}\label{vt+1}
    \vtp =\xt - \alpha_t  \nabla \phi_{\epst}(\xt) ,
\end{equation}
to serve as a safeguard for the convergence. If condition line 9 in LOA \ref{alg:lda} fails, we will reduce $\alpha_t$ by multiplying $\rho$ in line 12 finitely many times to satisfy the condition in line 9.
Specifically, we set $\xbf_{t+1} = \ubf_{t+1}$ if $\phi_{\epst}(\ubf_{t+1}) \le \phi_{\epst}(\vbf_{t+1})$; otherwise, we set $\xbf_{t+1} = \vbf_{t+1}$. Then, we repeat this process.

Our algorithm is summarized in Algorithm \ref{alg:lda}. 
%
%In this algorithm, we have two candidates: $\utp$ which arises from the linear approximation of proximal mapping, and $\vtp$ from standard gradient descent. The architecture of the unrolling sub-network for candidate $\utp$ can separate the updating schemes of the non-learnable data-fidelity term and the learnable prior term.
%
The prior term with unknown parameters has the exact residual update itself which improves the learning and training process \cite{he2016deep}. The condition checking in Line 5 is introduced to make sure that it is in the energy descending direction. Once the condition in Line 5 fails, the process moves to $\vtp$, and the line search in Line 12  guarantees that the appropriate step size can be achieved within finite steps to make the function value decrease. From Line 3 to Line 14, we consider that it solves a problem of minimizing $\phi_{\varepsilon_t}$ with $\epst$ fixed. Line 15 is used to update the value of $\epst$ depending on a reduction criterion. The detailed analysis of this mechanism and in-depth convergence justification is shown in Appendix \ref{convergence}. The corresponding unrolling network exactly follows Algorithm \ref{alg:lda} and thus shares the same convergence property. Compared to LDA \cite{chen2020learnable}, which computes both candidates $\utp$, $\vtp$ at every iteration and then chooses the one that achieves a smaller function value, we propose the criteria above in Line 5 for updating $\xtp$, 
which potentially
saves extra computational time for calculating the candidate $\vtp$ and potentially mitigates the frequent alternations between the two candidates. Besides, the smoothing method proposed in this work is more straightforward than smoothing in dual space \cite{chen2020learnable} while still keeping provable convergence, as shown in Theorem \ref{theorem a6}.

The proposed LOA-induced network is a multi-phase network whose architecture exactly follows the proposed LOA (Algoirthm \ref{alg:lda}) in the way that each phase corresponds to one iteration of the algorithm. 
Specifically, we construct a deep network, denoted by $F_{\Theta}$, that follows Algorithm \ref{alg:lda} exactly for a user-specified number of $T$ iterations. Here, $\Theta$ denotes the set of learnable parameters in $F_{\Theta}$, which includes the regularization network parameter $\theta$, weight $\kappa$, and other algorithmic parameters of Algorithm \ref{alg:lda}.
%
%denotes the set of all learnable parameters in the model. Motivated by the variational model, for an input \replaced{measurement $\ybf$ }{partial k-space data $\ybf$ in $\D_{\tau_i}$}, we desire the network output $F_{\Theta}(\ybf)$ to be an optimizer of the following {general} minimization problem \deleted{as \eqref{model}}
%
Therefore, for any input {under-sampled k-space measurement $\ybf$}, $F_{\Theta}(\ybf)$ executes the LOA (Algorithm \ref{alg:lda}) for $T$ iterations and generates an approximate solution of the minimization problem \eqref{model}:
\begin{equation}\label{model}
F_{\Theta}(\ybf) \approx \argmin_{\xbf} \big\{ \phi_{\Theta}(\xbf, \ybf) := f(\xbf, \ybf)   + R(\xbf; \Theta) \big\}.
\end{equation}
{where we use ``$\approx$'' since $F_{\Theta}$ follows only finitely many steps of the optimization algorithm to approximate the solution. 
%The function $f$ is the data fidelity term, in standard MRI setting it usually takes the same form as in \eqref{csmodel-1}. 
}%{where $f$ is the data fidelity term that usually takes the form $f(\xbf, \ybf)  = \frac{1}{2} \| \Pbf \F \xbf - \ybf \|^2$ in CS-MRI setting, where $\F$ and $\Pbf$ represent the Fourier transform and the binary under-sampling mask for k-space trajectory respectively. }
It is worth emphasizing that this approach can be readily applied to a much broader class of image reconstruction problems as long as $f$ is (possibly nonconvex and) continuously differentiable with the Lipschitz continuous gradient.
%
% These unknown parameters are trained by the method in the next section.
%
In the next subsection, we develop a meta-learning based approach for the robust training of the network parameter $\Theta$.

\begin{algorithm}[t]
\caption{Algorithmic Unrolling Method with Provable Convergence}
\label{alg:lda}
\begin{algorithmic}[1]
\STATE \textbf{Input:} Initial $\xbf_0$, $0<\rho, \gamma<1$, and $\varepsilon_0$, $a, \sigma >0$. Max total phases $T$ or tolerance \\$\etol>0$.
\FOR{$t=0,1,2,\dots,T-1$}
\STATE $\ztp =  \xt - \alpha_{t} \nabla f(\xt)$
\STATE $\utp = \ztp - \taut \nabla R_{\epst} (\ztp)$, 
\IF{ $\| \nabla \phi_{\epst} (\xt) \| \leq a \| \utp - \xt \| \  \mbox{and}   \  \phi_{\epst}(\utp) - \phi_{\epst}(\xt) \leq - \frac{1}{a}\| \utp - \xt \|^2 $} 
\STATE set $\xtp = \utp$,
\ELSE
\STATE $\vtp = \xt - \alpha_{t}  \nabla \phi_{\epst}(\xt)$, \label{marker}
\IF{ $ \phi_{\epst}(\vtp) - \phi_{\epst}(\xt) \le - \frac{1}{a} \| \vtp - \xt\|^2$ holds}
\STATE set $\xtp = \vtp$,
\ELSE
\STATE update $\alpha_{t} \leftarrow \rho \alpha_{t}$,
then \textbf{go to}~\ref{marker},
\ENDIF
\ENDIF
\STATE \textbf{if} $\|\nabla \phi_{\epst}(\xtp)\| < \sigma \gamma {\epst}$,  set $\epstp= \gamma {\epst}$;  \textbf{otherwise}, set $\epstp={\epst}$.
\STATE \textbf{if} $\sigma {\epst} < \etol$, terminate.
\ENDFOR
\STATE \textbf{Output:} $\xbf_t$.
\end{algorithmic}
\end{algorithm}

%%%%%%%%%%%%%%%%%%%%%
\subsection{Bilevel Optimization Algorithm for Network Training}
%In the previous section, we illustrate the forward network structure, which unfolds Algorithm \ref{alg:lda}. 

In this section, we consider the parameter training problem of the LOA-induced network $F_{\Theta}$. 
Specifically, we develop a bilevel optimization algorithm to train our network parameters $ \Theta$ from diverse data sets to improve network robustness and generalization ability. 
%
%The bilevel optimization problem is formulated as \eqref{eq:bi-level}, and we propose a network training algorithm in Algorithm \ref{alg:model}.

Recall that the LOA-induced network $F_{\Theta}$ exactly follows Algorithm \ref{alg:lda}, which is designed to solve the variational model \eqref{model} containing learnable regularization $R(\xbf;\Theta)$.
As shown in Section \ref{network}, we design $R (\xbf;\Theta)= \kappa \cdot r(\xbf;\Theta)$, where $r$ is learned to capture the intrinsic property of the underlying common features across all different tasks. 
To account for the large variations in the diverse training/validation data sets, we introduce a task-specific parameter $\omega_i$ to approximate the proper $\kappa$ for the $i$th task. Specifically, for the $i$th task, the weight $\kappa$ is set to $\sigma(\omega_i) \in (0, 1)$, where $\sigma(\cdot)$ is the sigmoid function. Therefore, $\kappa = \sigma(\omega_i)$ finds the proper weight of $r$ for the $i$-th task according to its specific sampling ratio or pattern.
The parameters $\omega_i$ are to be optimized in conjunction with $\Theta$ through the hyperparameter tuning process below. 
%
%The prior term $R$ consists of two parts, $R(\xbf; \Theta) = R(\xbf; \theta, \omega_i) = \sigma(\omega_i) r(\xbf;\theta)$, where $\theta$ collects the unknown parameters in the task-invariant learner $r(\xbf;\theta)$. 

Suppose that we are given $\mathcal{M}$ data pairs $\{(\ybf_m, \xbf^*_m) \}_{m = 1} ^{\mathcal{M}}$ for the use of training and validation, where $\ybf_m$ is the observation, which is the partial k-space data in our setting, and $\xbf^*_m$ is the corresponding ground truth image. The data pairs are then sampled into $\mathcal{N}$ tasks $\{ \D_{\tau_i} \}_{i = 1} ^ {\mathcal{N}}$, where each $\D_{\tau_i}$ represents the collection of data pairs in the specific task $\tau_i$.  In each task $\tau_i$, we further divide the data into the task-specific training set $\D^{tr}_{\tau_i}$ and validation set $\D^{val}_{\tau_i}$. 
The architecture of our base network
exactly follows the LOA \mbox{(Algorithm \ref{alg:lda})} developed in the previous section with learnable parameters $\theta$ and a task-specific parameter $\omega_i$ for the $i$th task.
% whose detailed structure is illustrated in Section \ref{network}. 
%\ye{shorten the sentences below (before (13)) since we have already explained what $F$ and $\omega_i$ are just now. Here just ned to say that $\Theta$ is $(\theta, \omega_i)$ for task $i$.}
More precisely, for one data sample denoted by $(\ybf^{(i)}_{j},\xbf*^{(i)}_{j})$ in task $\tau_i$ with index $j$, we propose the algorithmic unrolling network for task $\tau_i$ as
\begin{equation}\label{our_model_chp4}
F_{\theta,\omega_i}(\ybf^{(i)}_{j}) \approx \argmin_{\xbf} f(\xbf, \ybf^{(i)}_{j}) + \sigma(\omega_i) r(\xbf;\theta),    
\end{equation}
%\ye{the next sentence is not needed since it's already explained in the previous subsection. Just follow with "We define the task-specific loss..."}
where $\theta$ denotes the learnable common parameters across different tasks with task-invariant representation, whereas $\omega_i$ is a task-specific parameter for task $\tau_i$.
%The reason to use the hyper-parameter $\omega_i$ is that the training data set may contain various sampling ratios, and hence require different weights $\sigma(\omega_i)$ on each image reconstruction task $\tau_i$. 
The weight $\sigma(\omega_i)$ represents the weight of $r$ associated with the specific task $\tau_i$. {In our proposed network, $\Theta$ is the collection of $(\theta, \omega_i)$ for task $i = 1,\cdots \N$.}
We denote $\omega$ to be the set $\{\omega_i\} _ {i = 1} ^ {\mathcal{N}}$. The detailed architecture of this network is illustrated in Section \ref{network}. We define the task-specific loss
\begin{equation}\label{loss_sum}
\ell_{\tau_i}(\theta, \omega_i  ; \D_{\tau_i}) : = \sum _ {j=1}^{ |\D_{\tau_i}|} \ell \big( F_{\theta,\omega_i}(\ybf^{(i)}_{j}), \xbf*^{(i)}_{j} \big),    
\end{equation}
where $|\D_{\tau_i}|$ represents the cardinality of $\D_{\tau_i}$ and 
\begin{equation}\label{loss}
  \ell \big( F_{\theta,\omega_i}(\ybf^{(i)}_{j}), \xbf*^{(i)}_{j} \big) := \frac{1}{2} \|F_{\theta,\omega_i}(\ybf^{(i)}_{j}) - \xbf*^{(i)}_{j}\|^2.
\end{equation}

%\ye{Need some intuition/explanation to introduce the bi-level optimization below.} 
For the sake of preventing the proposed model from overfitting the training data, we introduce a novel learning framework by formulating the network training as a bilevel optimization problem to learn $\omega$ and $\theta $ in \eqref{our_model_chp4} as 
\begin{subequations}
\label{eq:bi-level}
\begin{align}
  \min_{  \omega = \{\omega_i:i\in[N]\} }  \quad & \sum^{\mathcal{N}}_{i=1} \ell _{\tau_i}( \theta (\omega) , \omega_i  ; \D^{val}_{\tau_i}) \\
 \mbox{s.t.}\quad \quad  & \theta(\omega)  = \argmin_{\theta} \sum^{\mathcal{N}}_{i=1} \ell_{\tau_i} ( \theta , \omega_i ; \D^{tr}_{\tau_i}).
\end{align}
\end{subequations}
In \eqref{eq:bi-level}, the lower-level optimization learns the task-invariant parameters $\theta$ of the feature encoder with the fixed task-specific parameter $\omega_i$ on the training dataset, and the upper-level adjusts the task-specific parameters $\{\omega_i\}$ so that the task-invariant parameters $\theta$ can perform robustly on the validation dataset as well. 
For simplicity, we omit the summation and redefine $\mathcal{L}(\theta, \omega ; \D) := \sum^{\mathcal{N}}_{i=1}\ell _{\tau_i}(\theta, \omega ; \D)$ and then briefly rewrite \eqref{eq:bi-level} as 
\begin{equation}
  \min_{  \omega} \mathcal{L}( \theta(\omega), \omega ; \D^{val}) \ \ \ \ \ \mbox{s.t.} \ \ \theta(\omega) =   \argmin_{\theta}\mathcal{L}( \theta, \omega ; \D^{tr}).
  \label{simplified bi-level}
\end{equation}
Then, we relax \eqref{simplified bi-level} into a single-level constrained optimization where the lower-level problem is replaced with its first-order necessary condition following \cite{mehra2019penalty}
\begin{equation}
  \min_{  \omega} \mathcal{L}( \theta(\omega), \omega ; \D^{val}) \ \ \ \ \ \mbox{s.t.} \ \ \nabla_{\theta} \mathcal{L}( \theta, \omega ; \D^{tr}) = 0.
  \label{simplified bi-level-1}
\end{equation}
which can be further approximated by an unconstrained problem by a penalty term as
\begin{equation}
  \min_{ \theta, \omega} \big\{ \widetilde{\mathcal{L}}( \theta, \omega ; \D^{tr}, \D^{val}) := \mathcal{L}( \theta, \omega ; \D^{val}) + \frac{\lambda}{2} \| \nabla_{\theta} \mathcal{L}( \theta, \omega ; \D^{tr}) \|^2 \big\}.
  \label{simplified bi-level-2}
\end{equation}
%With the increasing number of training pairs, an modern machine learning model has to be trained in mini-batch setting. In Section \ref{Comprehensive meta training}, we will explain the mini-batch setup in detail.

%To tackle the increasing number of training samples, a modern machine learning model need to be trained in mini-batch setting. 
%
We adopt the stochastic gradients of the loss functions on mini-batch data sets in each iteration.
In our model, we need to include the data pairs of multiple tasks in one batch; therefore, we propose the cross-task mini-batches when training. At each training iteration, we randomly sample the training data pairs $\mathcal{B}^{tr}_{\tau_i} = \{(\ybf^{(i)}_{j}, \xbf*^{(i)}_{j}) \in \D^{tr}_{\tau_i} \}_{j = 1}^{\mathcal{J}^{tr}}$ and the validation pairs $\mathcal{B}^{val}_{\tau_i} = \{(\ybf^{(i)}_{j}, \xbf*^{(i)}_{j}) \in \D^{val}_{\tau_i} \}_{j = 1}^{\mathcal{J}^{val}}$ on each task $\tau_i$. Then, the overall training and validation mini-batches $\mathcal{B}^{tr}$ and $\mathcal{B}^{val}$ used in every training iteration are composed of the sampled data pairs from the entire set of tasks; i.e., $\mathcal{B}^{tr} = \bigcup_{i = 1}^{\mathcal{N}} \{\mathcal{B}^{tr}_{\tau_i}\}$ and $\mathcal{B}^{val} = \bigcup_{i = 1}^{\mathcal{N}} \{\mathcal{B}^{val}_{\tau_i}\}$. Thus in each iteration, we have $\mathcal{N} \cdot \mathcal{J}^{tr}$ and $\mathcal{N}  \cdot \mathcal{J}^{val}$ data pairs used for training and validation, respectively. To solve the minimization problem \eqref{simplified bi-level-1}, we utilize the stochastic mini-batch alternating direction method summarized in Algorithm \ref{penelty_method}, which is modified from \cite{mehra2019penalty}.

{As analyzed in \cite{mehra2019penalty}, this penalty-type method has linear time complexity without computing the Hessian of the low level and only requires a constant space since we only need to store the intermediate $\theta, \omega$ at each training iteration, which is suitable for solving the large-scale bilevel optimization problem. Algorithm \ref{penelty_method} relaxes the bi-level optimization problem to a single-level constrained optimization problem by using the first-order necessary condition, which is not equivalent to the original problem but is much easier and efficient to solve. {In the inner-loop (Line 5--9) of Algorithm \ref{penelty_method}, we continue minimizing the converted single-level optimization function \eqref{simplified bi-level-2} with respect to $\theta$ for $K$ steps and then $\omega$ once alternatively until the condition with tolerance $\delta$ in Line 5 fails. The basic idea behind the condition in Line 5 arises from the first-order necessary condition as we would like to push the gradient of $\widetilde{\mathcal{L}}$ toward $0$. Furthermore, at Line 11 of the outer loop (Line 2--11), we decrease the tolerance $\delta$. Combining Line 5 and 11 guarantees that each time the inner loop terminates, the gradients of $\widetilde{\mathcal{L}}$ with respect to $\theta$ and $\omega$ become increasingly close to $0$. The parameter $\delta_{tol}$ is used to control the accuracy of the entire algorithm, and the outer-loop will terminate when $\delta$ is sufficiently small (i.e., $\delta \le \delta_{tol}$).}}
In addition, $\lambda$ is the weight for the second constraint term of \eqref{simplified bi-level-2}; in the beginning, we set $\lambda$ to be small to achieve a quick starting convergence, then gradually increase its value to emphasize the constraint. 
%\ye{Also include a line for the inputs, such as the data, the choice of $\delta_{tol}$ etc. In the caption explain that it is to solve (17).}
\begin{algorithm}
\caption{Stochastic mini-batch alternating direction penalty method to solve problem \eqref{simplified bi-level-1}}\label{alg:model}
\begin{algorithmic}[1]
\STATE \textbf{Input}  $\D^{tr}_{\tau_i}$, $\D^{val}_{\tau_i}$, $\delta_{tol}>0$.
\STATE \textbf{Initialize}  $ \theta$, $ {\omega}$, $\delta$, $\lambda>0$ and $\nu_\delta \in(0, 1)$, \ $\nu_\lambda > 1$.
\WHILE{$\delta > \delta_{tol}$}
\STATE Sample cross-task training batch $\mathcal{B}^{tr} = \bigcup_{i = 1}^{\mathcal{N}} \{(\ybf^{(i)}_{j}, \xbf*^{(i)}_{j}) \in \D^{tr}_{\tau_i} \}_{j = 1 : \mathcal{J}^{tr}}$
\STATE Sample cross-task validation batch $\mathcal{B}^{val} = \bigcup_{i = 1}^{\mathcal{N}} \{(\ybf^{(i)}_{j}, \xbf*^{(i)}_{j}) \in \D^{val}_{\tau_i} \}_{j = 1 : \mathcal{J}^{val}}$
\WHILE{$\|\nabla_{\theta}\widetilde{\mathcal{L}}( \theta, \omega ; \mathcal{B}^{tr}, \mathcal{B}^{val})\|^2 + \| \nabla_{\omega}\widetilde{\mathcal{L}}( \theta, \omega ; \mathcal{B}^{tr}, \mathcal{B}^{val})\| ^2 > \delta$}
\FOR{$k=1,2,\dots,K$ (inner loop)}
%\STATE  Sample batches of tasks $ \tau_i \sim p(\tau)$
\STATE $ \theta \leftarrow \theta - \rho_{\theta}^k \nabla_{\theta}\widetilde{\mathcal{L}}( \theta, \omega ; \mathcal{B}^{tr}, \mathcal{B}^{val})$
\ENDFOR
\STATE $ \omega \leftarrow \omega - \rho_{\omega} \nabla_{\omega}\widetilde{\mathcal{L}}( \theta, \omega ; \mathcal{B}^{tr}, \mathcal{B}^{val})$
\ENDWHILE
\STATE \textbf{update} $\delta \leftarrow \nu_\delta \delta$, $\ \lambda \leftarrow \nu_\lambda \lambda$
\ENDWHILE
\STATE \textbf{output:} $\theta, {\omega}$.
\end{algorithmic}
\label{penelty_method}
\end{algorithm}

\section{Implementation}
\label{sec:Implementation}
\subsection{Feature Extraction Operator}
We set the feature extraction operator $\gbf$ to {be} a vanilla $l$-layer CNN with the component-wise nonlinear activation function $\varphi$ and no bias, as follows:%Please check if * sign need to be changed to multiple sign. 
\begin{equation}\label{eq:g_chp4}
  \gbf(x) = \wbf_l * \varphi \cdots \ \varphi ( \wbf_3 * \varphi ( \wbf _2 * \varphi ( \wbf _1 * x ))),
\end{equation}
where $\{\wbf _q \}_{q = 1}^{l}$ denote the convolution weights consisting of $d$ kernels with identical spatial kernel size, and $*$ denotes the convolution operation. 
Here, $\varphi$ is constructed to be the smoothed rectified linear unit as defined below:
\begin{equation}\label{eq:sigma}
\varphi (x) = 
\begin{cases}
0, & \mbox{if} \ x \leq -\delta, \\
\frac{1}{4\delta} x^2 + \frac{1}{2} x + \frac{\delta}{4}, & \mbox{if} \ -\delta < x < \delta, \\
x, & \mbox{if} \ x \geq \delta,
\end{cases}
\end{equation}
where the prefixed parameter $\delta$ is set to be $0.001$ in our experiment.
The default configuration of the feature extraction operator is set as follows: the feature extraction operator $\gbf$ consists of $l=3$ convolution layers and all convolutions are with $4$ kernels of a spatial size of $3 \times 3$.

\subsection{Setups}
\label{Task-specific network pre-training}
As our method introduces an algorithmic unrolling network, there exists a one-to-one correspondence between the algorithm iterations and the neural network phases (or blocks). Each phase of the forward propagation can be viewed as one algorithm iteration, which motivates us to imitate the iterating of the optimization algorithm and use a stair training strategy \cite{chen2020learnable}. At the first stage, we start training the network parameters using one phase, then after the the loss converges, we add more phases (one phase each time) then continue the training process. We repeat this procedure and stop it when the loss does not decrease any further when we add more blocks. We minimize the loss for $100$ epochs/iterations each time using the SGD-based optimizer Adam \cite{kingma2014adam} with $\beta_1 = 0.9$, $\beta_2 = 0.999$, and the initial learning rate set to $10^{-3}$ as well as a mini-batch size of $8$. The Xavier Initializer \cite{Glorot10understandingthe} is used to initialize the weights of all convolutions. The initial smoothing parameter $\varepsilon_0$ is set to be $0.001$ and then learned together with other network parameters. The input $\xbf_0$ of the unrolling network is obtained by the zero-filling strategy \cite{bernstein2001effect}.  The deep unrolling network was implemented using the Tensorflow toolbox \cite{tensorflow2015-whitepaper} in the Python programming language. 

\section{Numerical Experiments}\label{experiments}
\subsection{Dataset}
To validate the performance of the proposed method, the data we used were from Multimodal Brain Tumor Segmentation Challenge 2018 \cite{menze2014multimodal}, in which the training dataset contains four modalities (T1, T1$_{c}$, T2 and FLAIR )%Please chefck if this italic is neccessary.
 scanned from 285 patients and the validation dataset contains images from 66 patients, each with a volume size  of $240 \times 240 \times 155$.  Each modality consists of two types of gliomas: 75 volumes of low-grade gliomas (LGG) and 210 volumes of high-grade gliomas (HGG). Our implementation involved HGG MRI in two modalities---T1 and T2 images---and we chose 30 patients from each modality
in the training dataset to train our network. In the validation dataset, we randomly picked 15 patients as our validation data and 6 patients in the training dataset as testing data, which were distinct from our training set and validation set.  We cropped the 2D image size to be $160 \times 180$ in the center region and picked  $10$ adjacent slices in the center of each volume, resulting in a total of $300$ images as our training data, $150$ images as our validation data, and a total of $60$ images as testing data. The amount of data mentioned here is for a single task, but since we emploedy multi-task training, the total number of images in each dataset should be multiplied by the number of tasks. For each 2D slice, we normalized the spatial intensity by dividing the maximum pixel value.

\subsection{Experiment Settings } %\deleted{and Comparison results}
\label{Experiment settings}
All the experiments were implemented on a Windows workstation with an Intel Core i9 CPU at 3.3GHz and an Nvidia GTX-1080Ti GPU with 11 GB of graphics card memory via TensorFlow \cite{abadi2016tensorflow}. The parameters in the proposed network were initialized by using Xavier initialization \cite{glorot2010understanding}.
We trained the meta-learning network with four tasks synergistically associated with four different CS ratios---10\%, 20\%, 30\%, and 40\%---and tested the well-trained model on the testing dataset with the same masks of these four ratios. We used 300 training data for each CS ratio, amounting to a total of 1200 images in the training dataset. The results for T1 and T2 MR reconstructions are shown in Tables \ref{results_same_ratio_t1} and \ref{results_same_ratio_t2}, respectively. The associated reconstructed images are displayed in Figures \ref{figure_same_ratio_t1} and \ref{figure_same_ratio_t2}.  We also tested the well-trained meta-learning model on unseen tasks with {radial} masks for {unseen} ratios of 15\%, 25\%, and 35\% and random Cartesian masks with ratios of 10\%, 20\%, 30\%, and 40\%. The task-specific parameters for the unseen tasks were retrained for different masks with different sampling ratios individually with fixed task-invariant parameters $\theta$. In this experiments, we only needed to learn $ \omega_i$ for three {unseen} CS ratios with {radial} mask and four regular CS ratios with Cartesian masks. The experimental training proceeded with fewer data and iterations, where we used 100 MR images with 50 epochs. For example, to reconstruct MR images with a CS ratio of 15\% from the {radial} mask, we fixed the parameter $\theta$ and retrained the task-specific parameter $\omega$ on 100 raw data points with 50 epochs, then tested with renewed $\omega$ on our testing data set with raw measurements sampled from the {radial} mask with a CS {radial} of 15\%. The results associated with  {radial} masks are shown in Tables \ref{results_dif_ratio_t1} and \ref{results_dif_ratio_t2}, Figures \ref{figure_dif_ratio_t1} and \ref{figure_dif_ratio_t2}  for T1 and T2 images, respectively. The results associated with Cartesian masks are listed in 
Table \ref{results_same_ratio_t2_cts} and reconstructed images are displayed in Figure  \ref{figure_same_ratio_t2_cts}. 

We compared our proposed meta-learning method with conventional supervised learning, which was trained with one task at each time and only learned the task-invariant parameter $\theta$ without the task-specific parameter $ \omega_i$. The forward network of conventional learning  unrolled Algorithm \ref{alg:lda} with 11 phases, which was the same as meta-learning. We merged the training set and validation set, resulting in $450$ images for the training of the conventional supervised learning. The training batch size was set as 25 and we applied a total of 2000 epochs, while in meta-learning, we applied 100 epochs with a batch size of 8.  The same testing set was used in both meta-learning and conventional learning to evaluate the performance of these two methods.

We made comparisons between meta-learning and the conventional network on the seven different CS ratios (10\%, 20\%, 30\%, 40\%, 15\%, 25\%, and 35\%) in terms of two types of random under-sampling patterns: {radial sampling} mask and Cartesian {sampling} mask. The parameters for both meta-learning and conventional learning networks were trained via the Adam optimizer \cite{kingma2014adam}, and they both learned the forward unrolled task-invariant parameter $\theta$. The network training of the conventional method used the same network configuration as the meta-learning network in terms of the number of convolutions, depth and size of CNN kernels, phase numbers and parameter initializer, etc.
The major difference in the training process between these two methods is that meta-learning is performed for multi-tasks by leveraging the task-specific parameter $\omega_i$ learned from Algorithm \ref{alg:model}, and the common features among tasks are learned from the feed-forward network that unrolls \mbox{Algorithm \ref{alg:lda},} while conventional learning solves the task-specific problem by simply unrolling the forward network via Algorithm \ref{alg:lda}, where both training and testing are implemented on the same task. To investigate the generalizability of meta-learning, we tested the well-trained meta-learning model on MR images in different distributions in terms of two types of sampling masks with various trajectories. The training and testing of conventional learning were applied with the same CS ratios; that is, if the conventional method was trained with a CS ratio 10\%, then it was also tested on a dataset with a CS ratio of 10\%, etc.

Because MR images are represented as complex values, we applied complex convolutions \cite{WANG2020136} for each CNN; that is, every kernel consisted of a real part and imaginary part. Three convolutions were used in $\gbf$, where each convolution contained four filters with a spatial kernel size of $3\times3$. {In Algorithm \ref{alg:lda},} a total of 11 phases can be achieved if we set the termination condition $\etol = 1\times 10^{-3}$, and the parameters of each phase are shared except for the step sizes. For the hyperparameters in Algorithm \ref{alg:lda}, we chose an initial learnable step size $\alpha_0= 0.01, \tau_0=0.01, \varepsilon_0 = 0.001$, and we set prefixed values of $ a = 10^5, \sigma = 10^3, \rho =0.9$, and $\gamma =0.9$. The principle behind the choices of those parameters is based on the convergence of the algorithm and effectiveness of the computation. The parameter $0 < \rho < 1$ is the reduction rate of the step size during the line search used to guarantee the convergence. The parameter $0 < \gamma < 1$ at step 15 is the reduction rate for  $\varepsilon$. In Algorithm 1, from step 2 to step 14, the smoothing level $\varepsilon$ is fixed. When the gradient of the smoothed function is small enough, we reduce $\varepsilon$ by a fraction factor $\gamma$ to find an approximate accumulation point of the original nonsmooth nonconvex problem. We chose a larger $a$ in order to have more iterations $k$ for which $u_{k + 1}$ satisfies the conditions in step 5, so that there would be fewer iterations requiring the computation of  $v_{k + 1}$.  Moreover, the scheme for computing  $u_{k + 1}$ is in accordance with the residual learning architecture that has been proven effective for reducing training error.

%because we want $ \alpha_t$ and $\epsilon_t$ are in the descending direction to terminate the iterations.}  
In Algorithm \ref{alg:model}, we set $\nu_\delta =0.95 $ and the parameter $\delta$  was initialized as $\delta_0 =1 \times 10^{-3}$ and stopped at value $\delta_{tol} = 4.35 \times 10 ^ {-6}$, and a total of 100 epochs were performed. To train the conventional method, we set 2000 epochs with the same number of phases, convolutions, and kernel sizes as  used to train the meta-learning approach. The initial $\lambda$ was set as $1 \times 10^{-5} $ and $ \nu_\lambda = 1.001$. %We applied total of inner loop iterations $K=5$ for line 6-8 to update $ \theta$.

We evaluated our reconstruction results on the testing data sets using three metrics: {peak signal-to-noise ratio (PSNR) \cite{hore2010image}}, structural similarity (SSIM) \cite{wang2004image}, and normalized mean squared error {(}NMSE{) \cite{NMSE}}.
The following formulations compute the PSNR, SSIM, and NMSE between the reconstructed image $ \xbf$ and ground truth $\xbf^*$. 
{PSNR can be induced by the mean square error (MSE) where \begin{equation}
    PSNR(\xbf,\xbf^*) =  20\log_{10} \big(  \frac{\max(\abs{\xbf^*}) } { \sqrt{MSE(\xbf,\xbf^*)}} \big),
\end{equation}
where $N$ is the total number of pixels of the ground truth  and MSE is defined by $MSE(\xbf,\xbf^*) = \frac{1}{N}\| \xbf^* - \xbf \|^2$.}
%{\begin{equation}
%    PSNR = 20\log_{10} \big(  \max(\abs{\xbf^*})  \big/ \frac{1}{N}\| \xbf^* - \xbf \|^2 \big),
%\end{equation}where $N$ is the total number of pixels of ground truth.}
%
\begin{equation}
    SSIM(\xbf,\xbf^*) = \frac{(2\mu_{\xbf} \mu_{\xbf^*} + C_1)(2\sigma_{\xbf \xbf^*} + C_2)}{(\mu_{\xbf}^2 + \mu_{\xbf^*}^2+C_1)( \sigma_{\xbf}^2 + \sigma_{\xbf^*}^2 + C_2)},
\end{equation}
 where $\mu_{\xbf}, \mu_{\xbf^*}$ represent local means, $\sigma_{\xbf}, \sigma_{\xbf^*}$ denote standard deviations, $\sigma_{\xbf \xbf^*} $ represents the covariance between $\xbf$ and $\xbf^* $, $ C_1 = (k_1 L)^2, C_2 = (k_2 L)^2$ are two constants which avoid the zero denominator, and $ k_1 =0.01, k_2 =0.03$. $L$ is the largest pixel value of MR image. 
\begin{equation}
   NMSE(\xbf,\xbf^*)=  \frac{\|  \xbf-\xbf^*\|_2^2}{\|  \xbf\|_2^2},
\end{equation}
{where NMSE is used to measure the mean relative error. For detailed information of these three metrics mentioned above, please refer to \cite{hore2010image, wang2004image, NMSE}.}

\subsection{ {Experimental Results with Different CS Ratios in Radial Mask}} %{Quantitative and Qualitative Comparisons at different trajectories in radial mask}

In this section, we evaluate the performance of well-trained meta-learning and conventional learning approaches. Tables \ref{results_same_ratio_t1}, \ref{results_same_ratio_t2} and \ref{results_same_ratio_t2_cts} report the quantitative results of averaged numerical performance with standard deviations and associated descaled task-specific meta-knowledge $ \sigma(\omega_i)$. From the experiments implemented with {radial} masks, we observe that the average PSNR value of meta-learning improved by 1.54 dB in the T1 brain image for all four CS ratios compared with the conventional method, and for the T2 brain image, {the average PSNR of }meta-learning improved by 1.46 dB.  Since the general setting of meta-learning aims to take advantage of the information provided from each individual task, with each task associated with an individual sampling mask that may have complemented sampled points, the performance of the reconstruction from each task benefits from other tasks. Smaller CS ratios will inhibit the reconstruction accuracy, due to the sparse undersampled trajectory in raw measurement, while meta-learning exhibits a favorable potential ability to solve this issue even in the situation of insufficient amounts of training data. 

In general supervised learning,  training data need to be in the same or a similar distribution; heterogeneous data exhibit different structural variations of features, which hinder CNNs from extracting features efficiently. In our experiments, raw measurements sampled from different ratios of compressed sensing display different levels of incompleteness;  these undersampled measurements do not fall in the same distribution but they are related. Different sampling masks are shown at the bottom of Figures \ref{figure_same_ratio_t1} and \ref{figure_dif_ratio_t1}, and these may have complemented sampled points, in the sense that some of the points which a $40\%$ sampling ratio mask did not sample were captured by other masks. In our experiment, different sampling masks provided their own information from their sampled points, meaning that four reconstruction tasks helped each other to achieve an efficient performance. Therefore, this explains why meta-learning is still superior to conventional learning when the sampling ratio is large.

Meta-learning expands a new paradigm for supervised learning---the purpose is to quickly learn multiple tasks. Meta-learning only learns task-invariant parameters once for a common feature that can be shared with four different tasks, and each $ \sigma(\omega_i)$ provides  task-specific weighting parameters according to the principle of  ``learning to learn''. In conventional learning, {the network parameter} needs to be trained four times with four different masks since the task-invariant parameter cannot be generalized to other tasks, which is time-intensive. From Tables \ref{results_same_ratio_t1} and \ref{results_same_ratio_t2}, we observe that a small CS ratio needs a higher value of $\sigma(\omega_i) $. In fact, in our model \eqref{model}, the task-specific parameters behave as weighted constraints for task-specific regularizers, and the tables indicate that lower CS ratios require larger weights to be applied for the regularization. 

A qualitative comparison between conventional and meta-learning methods is shown in Figures \ref{figure_same_ratio_t1} and \ref{figure_same_ratio_t2}, displaying the reconstructed MR images of the same slice for T1 and T2%The previous variable are in italic, please check if this should be italic too. Remember to keep the consistancy on the format through the whole text
, respectively. We label the zoomed-in details of HGG in the red boxes. We observe  evidence that conventional learning is more blurry and loses sharp edges, especially with lower CS ratios. From the point-wise error map, we find that meta-learning has the ability to reduce noises, especially in some detailed and complicated regions, compared to conventional learning.

We also tested the performance of meta-learning with two-thirds of the training and validation data for the T1-weighted image used in the previous experiment, denoted as ``meta-learning''. For conventional learning, the network was also trained  by using two-thirds of the training samples in the previous experiment. The testing dataset remained the same as before. These results are displayed in Table \ref{results_same_ratio_t1}, where we denote the reduced data experiments as ``meta-learning$^*$'' and ``conventional$^*$''. These experiments reveal that the accuracy of test data decreases when we reduce the training data size, but it is not a surprise that meta-learnining$^*$ still outperforms conventional learning$^*$, and even conventional learning.

To verify the reconstruction performance of the proposed LOA \ref{alg:lda}, we compared the proposed conventional learning with ISTA-Net$^+$ \cite{zhang2018ista}, which is a state-of-the-art deep unfolded network for MRI reconstruction. We retrained ISTA-Net$^+$ with the same training dataset and testing dataset as conventional learning  on the T1-weighted image. For a fair comparison, we  used the same number of convolution kernels, the same dimension of kernels for each convolution during training, and the same phase numbers as conventional learning. The testing numerical results are listed in Table \ref{results_same_ratio_t1} and the MRI reconstructions are displayed in Figure \ref{figure_same_ratio_t1}. In figures \ref{results_same_ratio_t1}, \ref{figure_same_ratio_t1}, the pictures (from top to bottom) display the T1 brain image reconstruction results, zoomed-in details, point-wise errors with a color bar, and associated
 \textbf{{radial}} masks for meta-learning and conventional learning
 with four different CS ratios of 10\%, 20\%, 30\%, 40\% (from left to right).  Conventional$^*$ and meta-learning$^*$ are trained with two-thirds of the dataset used in training conventional and meta-learning approaches, respectively. We can observe that the conventional learning which unrolls Algorithm \ref{alg:lda} outperforms ISTA-Net$^+$ in any of the CS ratios. From the corresponding point-wise absolute error, the conventional learning attains a much lower error and much better reconstruction quality.

\subsection{ {Experimental Results with Different Unseen CS Ratios in Different Sampling Patterns}} %{Quantitative and Qualitative Comparisons at skewed trajectories in different sampling patterns}

In this section, we test the generalizability of the proposed model for unseen tasks. We fixed the well-trained task-invariant parameter $\theta$  and only trained $\omega_i$ for sampling ratios of 15\%, 25\%, and 35\% with {radial} masks and sampling ratios of 10\%, 20\%, 30\%, and 40\% with Cartesian masks. In this experiment, we only used 100 training data points for each CS ratio and applied a total of 50 epochs. The averaged evaluation values and standard deviations are listed in Tables \ref{results_dif_ratio_t1} and \ref{results_dif_ratio_t2} for reconstructed T1 and T2 brain images, respectively,  with {radial} masks, and Table \ref{results_same_ratio_t2_cts} shows the qualitative performance for the reconstructed T2 brain image with random Cartesian sampling masks applied. In Table \ref{results_dif_ratio_t1} and \ref{results_dif_ratio_t2},
Meta-learning was trained with CS ratios of 10\%, 20\%, 30\%, and 40\% and tested with {unseen} ratios of 15\%, 25\%, and 35\%. The conventional method was subjected to regular training and testing with the same CS ratios of 15\%, 25\%, and 35\%. 
In the T1 image reconstruction results, meta-learning showed an improvement of 1.6921 dB in PSNR for the 15\% CS ratio, 1.6608 dB for the 25\% CS ratio, and 0.5764 dB for the 35\% ratio compared to the conventional method,  showing the tendency that the level of reconstruction quality for lower CS ratios improved more than higher CS ratios. A similar trend was found for T2 reconstruction results with different sampling masks. The qualitative comparisons are illustrated in Figures \ref{figure_dif_ratio_t1}, \ref{figure_dif_ratio_t2}, and \ref{figure_same_ratio_t2_cts} for T1 and T2 images tested with {unseen} CS ratios in {radial} masks and T2 images tested with Cartesian masks with regular CS ratios, respectively.  In figures \ref{figure_dif_ratio_t1}, \ref{figure_dif_ratio_t2}, Meta-learning was trained with CS ratios of 10\%, 20\%, 30\%, and 40\% and tested with three different {unseen} CS ratios of 15\%, 25\%, and 35\% (from left to right). Conventional learning was trained and tested with the same CS ratios of 15\%, 25\%, and 35\%. The top-right image is the ground truth fully-sampled image. The top-right image is the ground truth fully-sampled image. In \ref{figure_same_ratio_t2_cts}, the pictures (from top to bottom) display the T2 brain image reconstruction results, zoomed-in details, point-wise errors with a color bar, and associated
 \textbf{Cartesian} masks for meta-learning and conventional learning
 with four different CS ratios of 10\%, 20\%, 30\%, and 40\% (from left to right). The top-right image is the ground truth fully-sampled image.
In the experiments conducted with {radial} masks,
meta-learning was superior to conventional learning, especially at a CS ratio of 15\%---one can observe that the detailed regions in red boxes maintained their edges and were closer to the true image, while the conventional method reconstructions are hazier and lost details in some complicated tissues. The point-wise error map also indicates that meta-learning has the ability to suppress noises.

Training with Cartesian masks is more difficult than {radial} masks, especially for conventional learning, where the network is not very deep since the network only applies three convolutions each with four kernels. Table \ref{results_same_ratio_t2_cts} indicates that the average performance of Meta-learning improved about 1.87 dB compared to conventional methods with T2 brain images. These results further demonstrate that meta-learning has the benefit of parameter efficiency, and the performance is much better than conventional learning even if we apply a shallow network with a small amount of training data.

The numerical experimental results discussed above show that meta-learning is capable of fast adaption to new tasks and has more robust generalizability for a broad range of tasks with heterogeneous, diverse data. Meta-learning can be considered as an efficient technique for solving difficult tasks by leveraging the features extracted from easier tasks.

\begin{table}
\caption{ Quantitative evaluations of the reconstructions of T1 brain image associated with various sampling ratios of 
 \textbf{radial} masks.}   \label{results_same_ratio_t1}
%\addtolength{\tabcolsep}{-2pt}
\resizebox{\linewidth}{46mm}{ 
\begin{tabular}{cccccc}
\toprule
\textbf{CS Ratio} & \textbf{Methods} & \textbf{PSNR}	& \textbf{SSIM}	& \textbf{NMSE} & \boldmath{$ \sigma(\omega_i)$}\\
\midrule
 & {ISTA-Net$^+$ \cite{zhang2018ista}} & 21.2633 $\pm$ 1.0317 & 0.5487 $\pm$ 0.0440 & 0.1676 $\pm$ 0.0253 & \\
   & {Conventional$^*$}	& 21.6947 $\pm$ 1.0264 & 0.5689 $\pm$ 0.0404 & 0.1595 $\pm$ 0.0240  &  \\
10\%   & Conventional		& 21.7570 $\pm$ 1.0677  & 0.5650 $\pm$ 0.0412 &  0.0259 $\pm$ 0.0082  &  \\
   & {Meta-learning$^*$}	& 22.9633 $\pm$ 1.0969 & 0.5962 $\pm$ 0.0415 & 0.0194 $\pm$ 0.0065 & 0.9339 
\\
   & Meta-learning		& 23.2672 $\pm$ 1.1229  & 0.6101 $\pm$ 0.0436 & 0.0184 $\pm$ 0.0067 & 0.9218\\
   \midrule
 & {ISTA-Net$^+$ \cite{zhang2018ista}} & 26.2734 $\pm$ 1.0115 & 0.7068 $\pm$ 0.0364 & 0.0944 $\pm$ 0.0155 & \\
   & {Conventional$^*$}	& 26.4639 $\pm$ 1.0233 & 0.7107 $\pm$ 0.0357 & 0.0924 $\pm$ 0.0154 & \\
20\%   & Conventional		& 26.6202 $\pm$ 1.1662  & 0.7121 $\pm$ 0.0397 &  0.0910 $\pm$ 0.0169 & \\ 
   & {Meta-learning$^*$}	& 27.9381 $\pm$ 1.1121 & 0.7541 $\pm$ 0.0360 & 0.0063 $\pm$ 0.0023 & 0.8150\\
   & Meta-learning		& 28.2944 $\pm$ 1.2119  & 0.7640 $\pm$ 0.0377  &  0.0058 $\pm$ 0.0022 & 0.7756\\
   \midrule
 & {ISTA-Net$^+$ \cite{zhang2018ista}} & 28.8309 $\pm$ 1.3137 & 0.7492 $\pm$ 0.0407 & 0.0708 $\pm$ 0.0142 & \\
   & {Conventional$^*$}	& 29.2923 $\pm$ 1.3194 & 0.7522 $\pm$ 0.0399 & 0.0671 $\pm$ 0.0136 & \\
30\%   & Conventional		& 29.5034 $\pm$ 1.4446  & 0.7557 $\pm$ 0.0408 &  0.0657 $\pm$ 0.0143 & \\
   & {Meta-learning$^*$}	& 30.8691 $\pm$ 1.5897 & 0.8310 $\pm$ 0.0394 & 0.0033 $\pm$ 0.0015 & 0.6359\\
   & Meta-learning		& 31.1417 $\pm$ 1.5866  & 0.8363 $\pm$ 0.0385 & 0.0031 $\pm$ 0.0014 & 0.6501\\
   \midrule
 & {ISTA-Net$^+$ \cite{zhang2018ista}} & 30.7282 $\pm$ 1.5482 & 0.8008 $\pm$ 0.0428 & 0.0572 $\pm$ 0.0127 & \\
   & {Conventional$^*$}	& 31.3761 $\pm$ 1.5892 & 0.8035 $\pm$ 0.0420 & 0.0532 $\pm$ 0.0121 & \\
40\%   & Conventional		& 31.4672 $\pm$ 1.6390  & 0.8111 $\pm$ 0.0422 &  0.0029 $\pm$ 0.0014 & \\
   & {Meta-learning$^*$}  & 32.7330 $\pm$ 1.6386 & 0.8623 $\pm$ 0.0358 & 0.0022 $\pm$ 0.0010 & 0.6639\\
   & Meta-learning	&  32.8238 $\pm$ 1.7039  & 0.8659 $\pm$ 0.0370  & 0.0022 $\pm$ 0.0010 & 0.6447\\
\bottomrule
\end{tabular}}
\end{table}

\begin{table}
\caption{Quantitative evaluations of the reconstructions of T2 brain image associated with various sampling ratios of
 \textbf{radial} masks.  \label{results_same_ratio_t2}}
%\addtolength{\tabcolsep}{-2pt}
\begin{tabular}{cccccc}
\toprule
\textbf{CS Ratio} & \textbf{Methods} & \textbf{PSNR}	& \textbf{SSIM}	& \textbf{NMSE} & \boldmath{$ \sigma(\omega_i)$}\\
\midrule
10\%
 & Conventional	& 23.0706 $\pm$ 1.2469  & 0.5963 $\pm$ 0.0349 &  0.2158  $\pm$  0.0347  &\\
   & Meta-learning	& 24.0842 $\pm$ 1.3863 & 0.6187 $\pm$ 0.0380 & 0.0112 $\pm$ 0.0117  & 0.9013
	\\
   \midrule
20\% & Conventional & 27.0437  $\pm$ 1.0613  & 0.6867 $\pm$ 0.0261 &  0.1364 $\pm$ 0.0213 & \\ 
   & Meta-learning	& 28.9118 $\pm$ 1.0717 & 0.7843 $\pm$ 0.0240	& 0.0122 $\pm$ 0.0030 & 0.8742\\
   \midrule
30\% & Conventional	 & 29.5533 $\pm$ 1.0927 & 0.7565 $\pm$ 0.0265 & 0.1023 $\pm$ 0.0166 & \\
   & Meta-learning	& 31.4096 $\pm$ 0.9814 & 0.8488 $\pm$ 0.0217 & 0.0069 $\pm$ 0.0019 & 0.8029\\
   \midrule
40\% & Conventional	& 32.0153 $\pm$ 0.9402 & 0.8139 $\pm$ 0.0238 & 0.0770 $\pm$ 0.0128 & \\
   & Meta-learning	& 33.1114 $\pm$ 1.0189 & 0.8802 $\pm$ 0.0210 & 0.0047 $\pm$ 0.0015 & 0.7151\\
\bottomrule
\end{tabular}
\end{table}

\begin{table}
\caption{Quantitative evaluations of the reconstructions of T1 brain image associated with various sampling ratios of
 \textbf{radial} masks.  \label{results_dif_ratio_t1}}
%\addtolength{\tabcolsep}{-2pt}
\begin{tabular}{cccccc}
\toprule
\textbf{CS Ratio} & \textbf{Methods} & \textbf{PSNR}	& \textbf{SSIM}	& \textbf{NMSE} & \boldmath{$ \sigma(\omega_i)$} \\
\midrule
15\%
 & Conventional		& 24.6573 $\pm$ 1.0244  & 0.6339 $\pm$ 0.0382 &  0.1136 $\pm$ 0.0186  &\\
   & Meta-learning		& 26.3494 $\pm$ 1.0102  & 0.7088 $\pm$ 0.0352 & 0.0090 $\pm$ 0.0030 & 0.9429
\\
   \midrule
25\% & Conventional		& 28.4156 $\pm$ 1.2361  & 0.7533 $\pm$ 0.0368 &  0.0741 $\pm$ 0.0141 & \\ 
   & Meta-learning		& 30.0764 $\pm$ 1.4645  & 0.8135 $\pm$ 0.0380 & 0.0040 $\pm$ 0.0017 & 0.8482\\
   \midrule
35\% & Conventional		& 31.5320 $\pm$ 1.5242  & 0.7923 $\pm$ 0.0420 &  0.0521 $\pm$ 0.0119  &\\
   & Meta-learning		& 32.1084 $\pm$ 1.6481  & 0.8553 $\pm$ 0.0379 &  0.0025 $\pm$ 0.0011 & 0.6552\\
\bottomrule
\end{tabular}
\end{table}
\vspace{-6pt}

\begin{table}
\caption{Quantitative evaluations of the reconstructions of T2 brain image associated with various sampling ratios of
 \textbf{radial} masks. \label{results_dif_ratio_t2}}
%\addtolength{\tabcolsep}{-2pt}
\begin{tabular}{cccccc}
\toprule
\textbf{CS Ratio} & \textbf{Methods} & \textbf{PSNR}	& \textbf{SSIM}	& \textbf{NMSE} & \boldmath{$ \sigma(\omega_i)$} \\
\midrule
15\%
 & Conventional		& 24.8921 $\pm$ 1.2356 & 0.6259 $\pm$ 0.0285 & 0.1749 $\pm$ 0.0280  & \\
   & Meta-learning		& 26.7031 $\pm$ 1.2553 & 0.7104 $\pm$ 0.0318 & 0.0205 $\pm$ 0.0052  & 0.9532
\\
   \midrule
25\% & Conventional		& 29.0545 $\pm$ 1.1980 & 0.7945 $\pm$ 0.0292 & 0.1083 $\pm$ 0.0173  & \\ 
   & Meta-learning		& 30.0698 $\pm$ 0.9969 & 0.8164 $\pm$ 0.0235 &  0.0093 $\pm$ 0.0022 & 0.8595\\
   \midrule
35\% & Conventional		& 31.5201 $\pm$ 1.0021 & 0.7978 $\pm$ 0.0236 & 0.0815 $\pm$ 0.0129  & \\
   & Meta-learning		& 32.0683 $\pm$ 0.9204 & 0.8615 $\pm$ 0.0209 & 0.0059 $\pm$ 0.0014  & 0.7388\\
\bottomrule
\end{tabular}
\end{table}
\vspace{-6pt}

\begin{table}
\caption{Quantitative evaluations of the reconstructions of T2 brain image associated with various sampling ratios of random
 \textbf{Cartesian} masks.  \label{results_same_ratio_t2_cts}}
%\addtolength{\tabcolsep}{-2pt}
\begin{tabular}{cccccc}
\toprule
\textbf{CS Ratio} & \textbf{Methods} & \textbf{PSNR}	& \textbf{SSIM}	& \textbf{NMSE} & \boldmath{$ \sigma(\omega_i)$}\\
\midrule
10\%
 & Conventional	& 20.8867 $\pm$ 1.2999 & 0.5082 $\pm$ 0.0475 & 0.0796 $\pm$ 0.0242  &\\
   & Meta-learning	& 22.0434 $\pm$ 1.3555 & 0.6279 $\pm$ 0.0444 & 0.0611 $\pm$ 0.0188  & 0.9361
	\\
   \midrule
20\% & Conventional & 22.7954 $\pm$ 1.2819 & 0.6057 $\pm$ 0.0412 & 0.0513 $\pm$ 0.0157 & \\ 
   & Meta-learning	& 24.7162 $\pm$ 1.3919 & 0.6971 $\pm$ 0.0380 & 0.0329 $\pm$ 0.0101 & 0.8320\\
   \midrule
30\% & Conventional	 & 24.2170 $\pm$ 1.2396 & 0.6537 $\pm$ 0.0360 & 0.0371 $\pm$ 0.0117 & \\
   & Meta-learning	&  26.4537 $\pm$ 1.3471 & 0.7353 $\pm$ 0.0340 & 0.0221 $\pm$ 0.0068 & 0.6771\\
   \midrule
40\% & Conventional	& 25.3668 $\pm$ 1.3279 & 0.6991 $\pm$ 0.0288 & 0.1657 $\pm$ 0.0265 & \\
   & Meta-learning	& 27.5367 $\pm$ 1.4107 & 0.7726 $\pm$ 0.0297 & 0.0171 $\pm$ 0.0050 & 0.6498\\
\bottomrule
\end{tabular}
\end{table}

\begin{figure}[H]
\includegraphics[width=0.11\linewidth, angle=270]{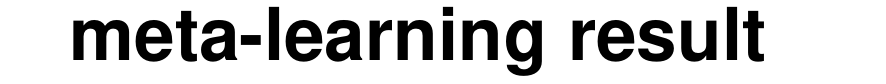}
\includegraphics[width=0.1\linewidth, angle=180]{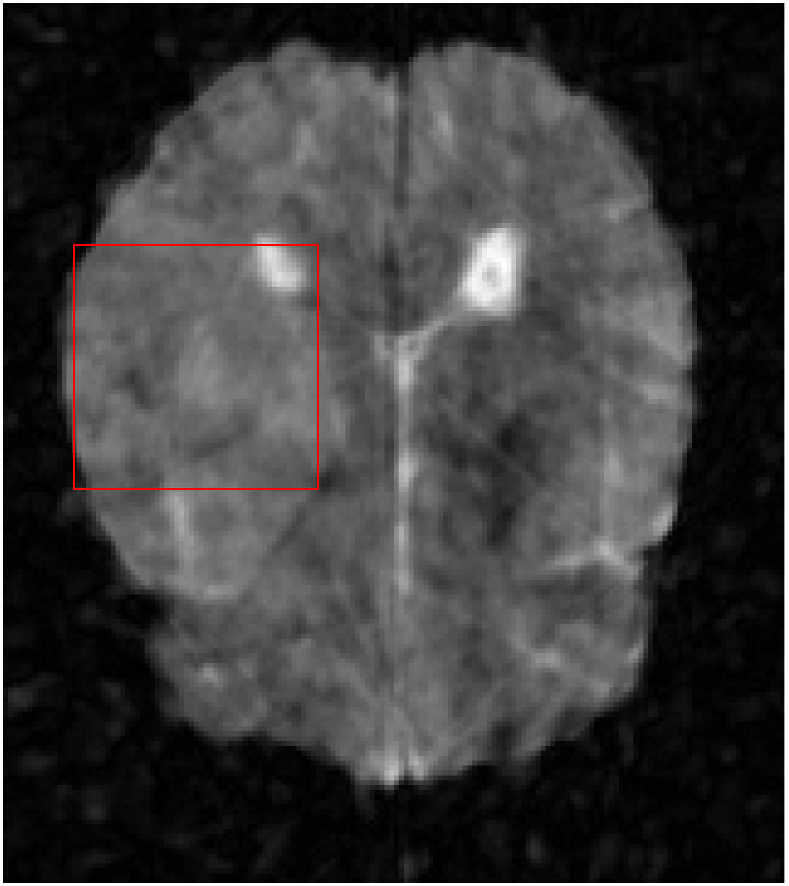}
\includegraphics[width=0.1\linewidth, angle=180]{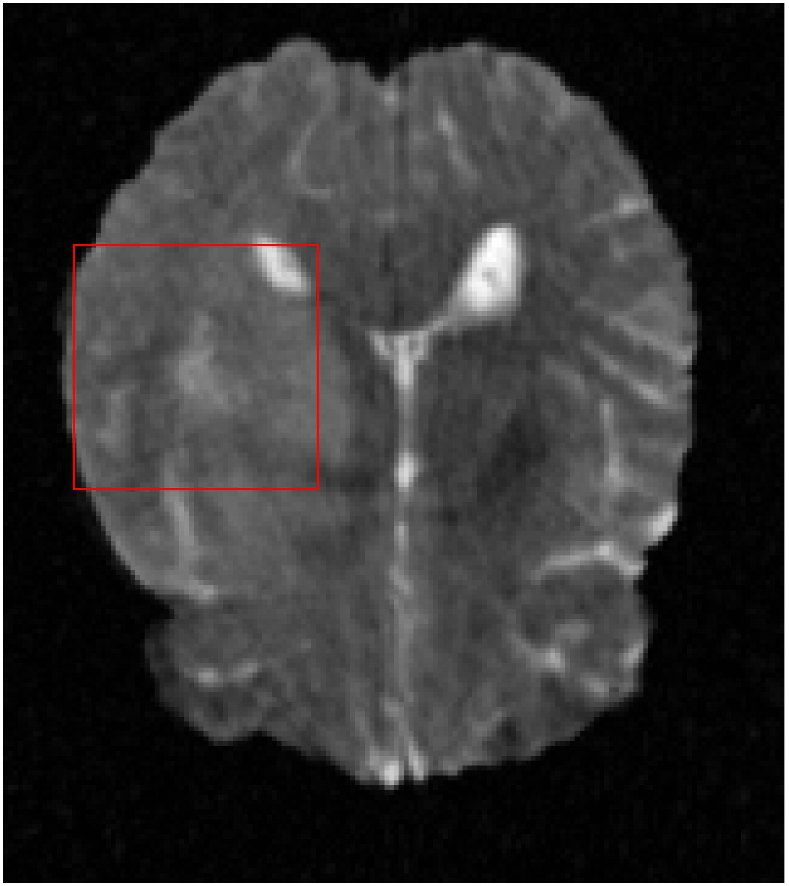}
\includegraphics[width=0.1\linewidth, angle=180]{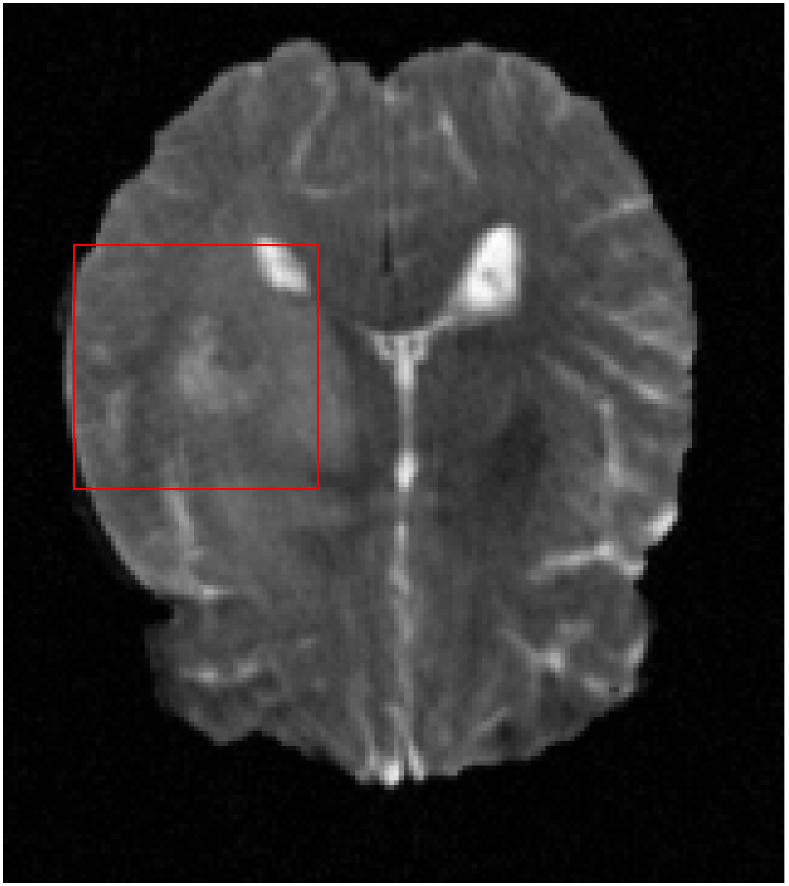}
\includegraphics[width=0.1\linewidth, angle=180]{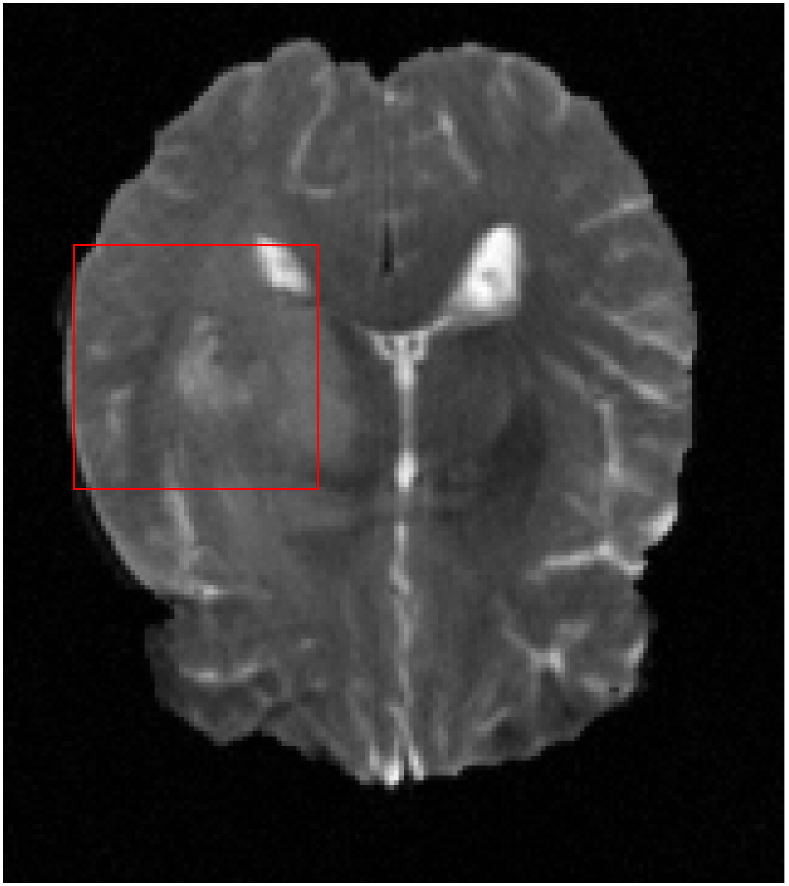}
\includegraphics[width=0.1\linewidth, angle=180]{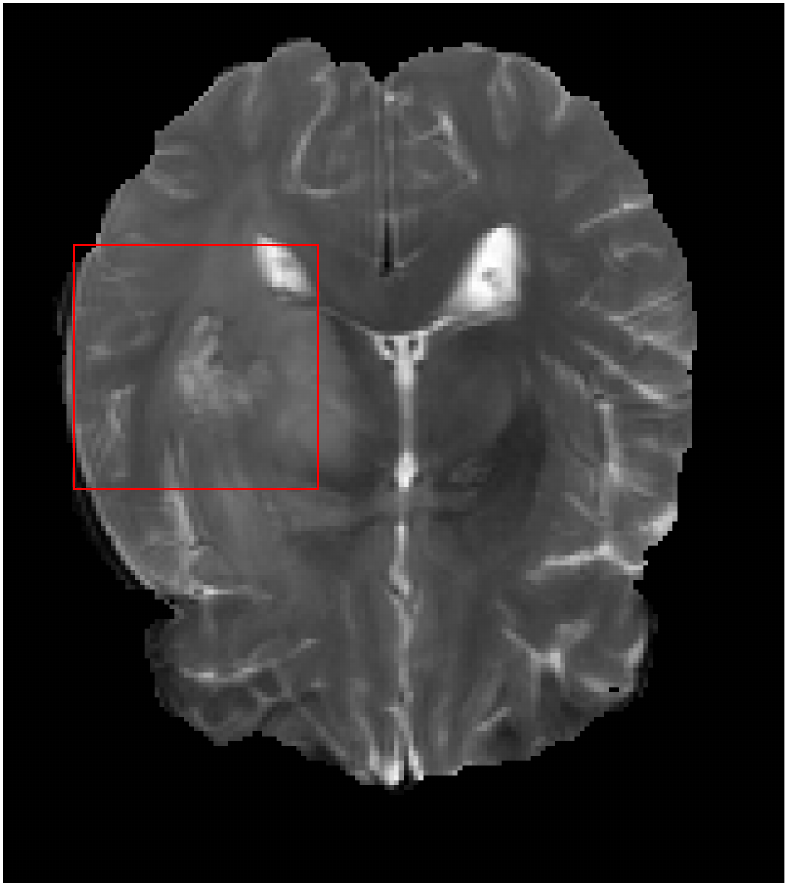}\\
\includegraphics[width=0.11\linewidth, angle=270]{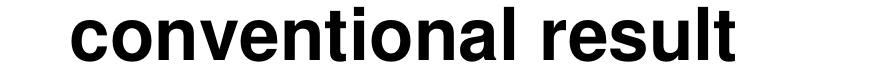}
\includegraphics[width=0.1\linewidth, angle=180]{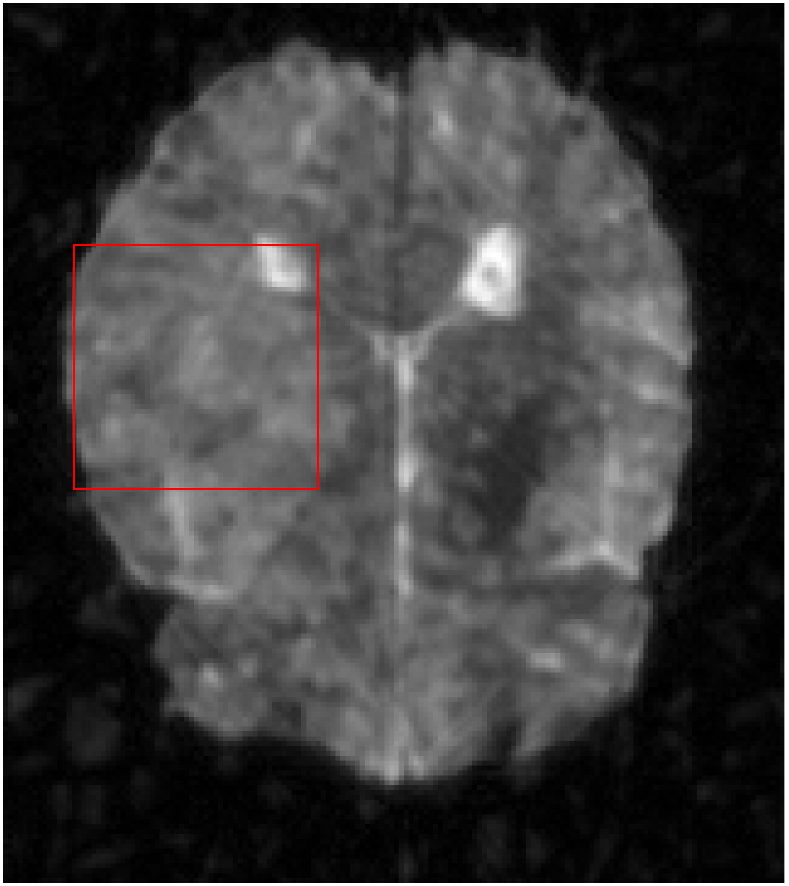}
\includegraphics[width=0.1\linewidth, angle=180]{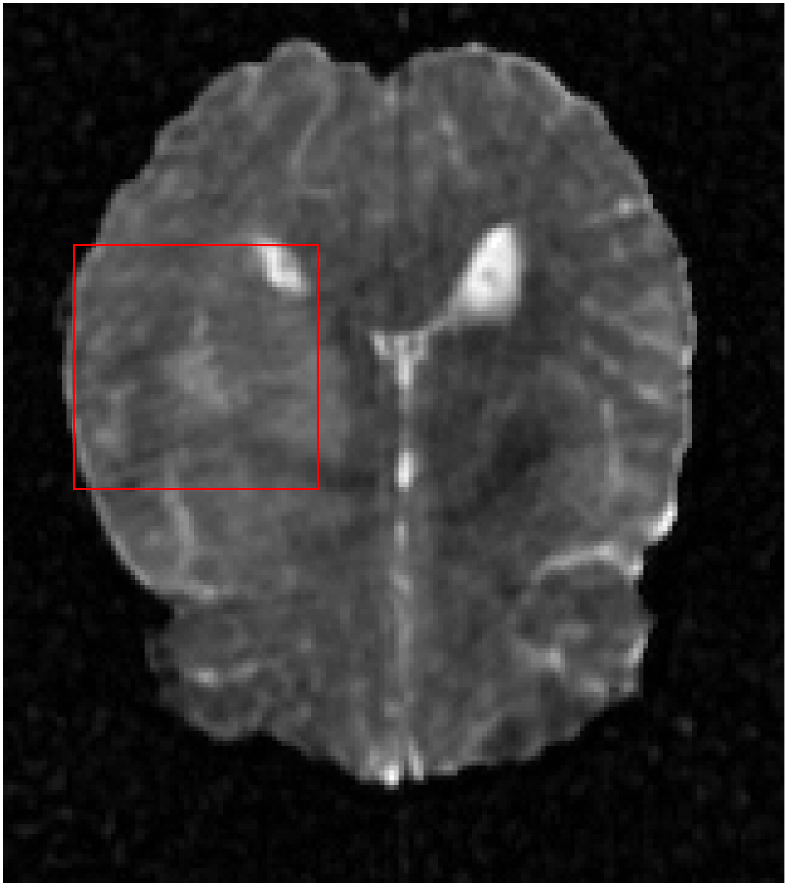}
\includegraphics[width=0.1\linewidth, angle=180]{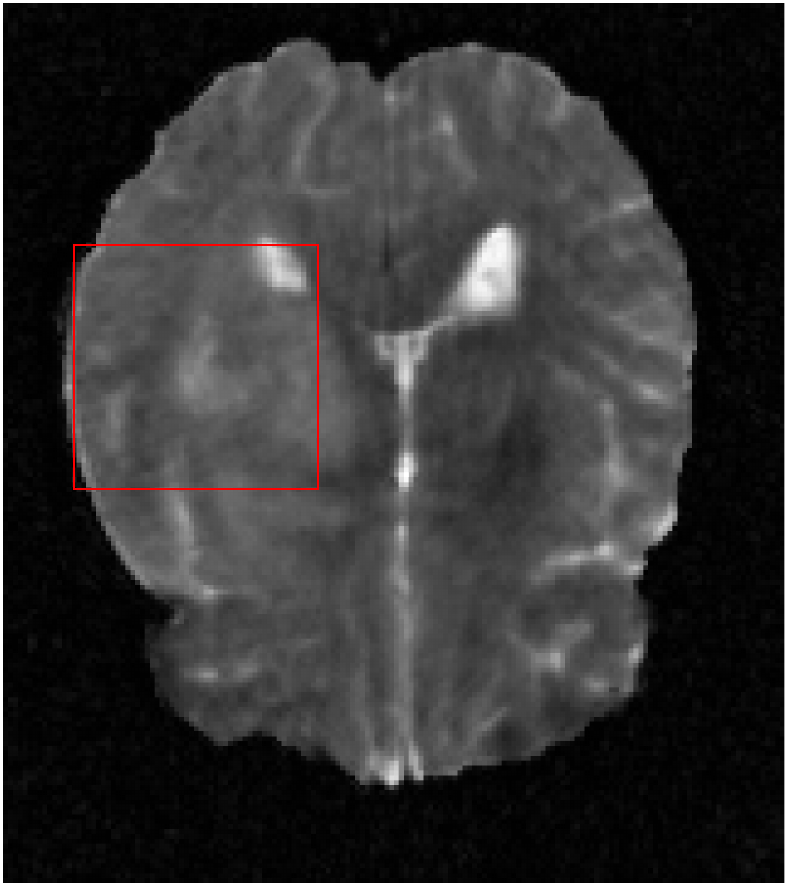}
\includegraphics[width=0.1\linewidth, angle=180]{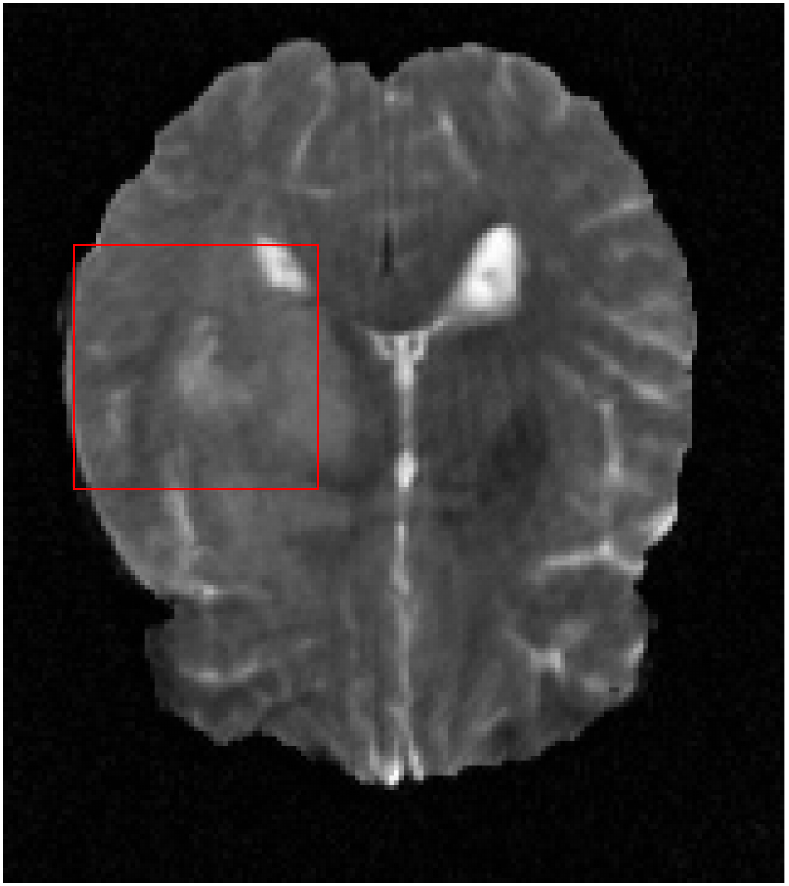}
\includegraphics[width=0.1\linewidth, angle=180]{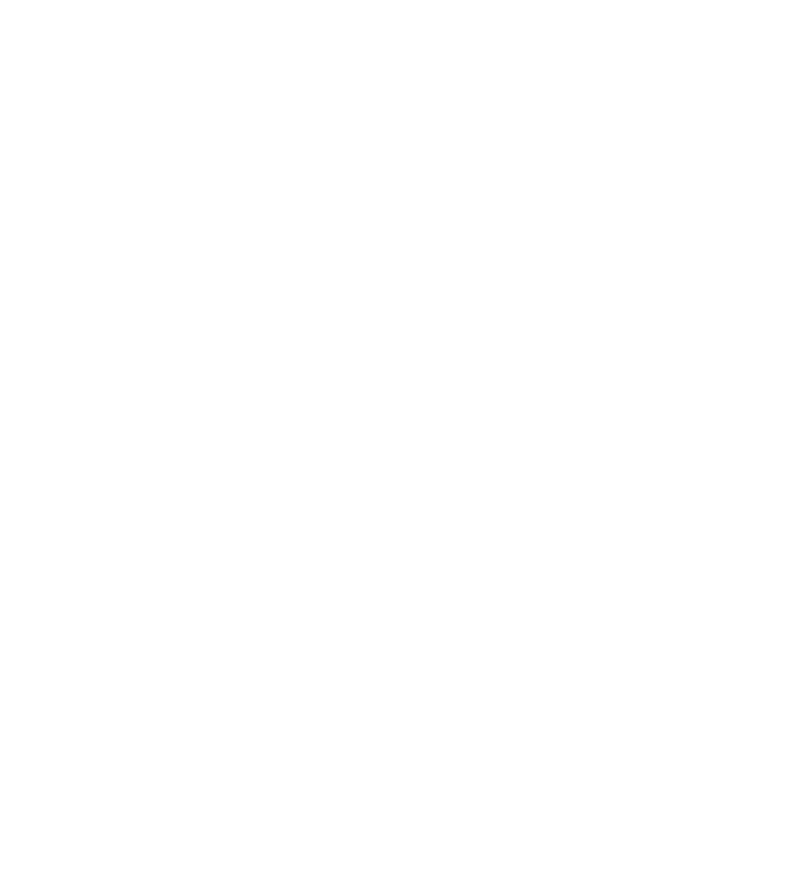}\\
\includegraphics[width=0.11\linewidth, angle=270]{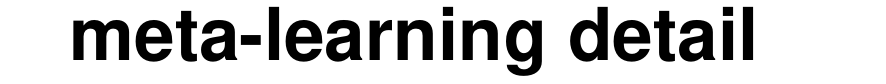}
\includegraphics[width=0.1\linewidth, angle=180]{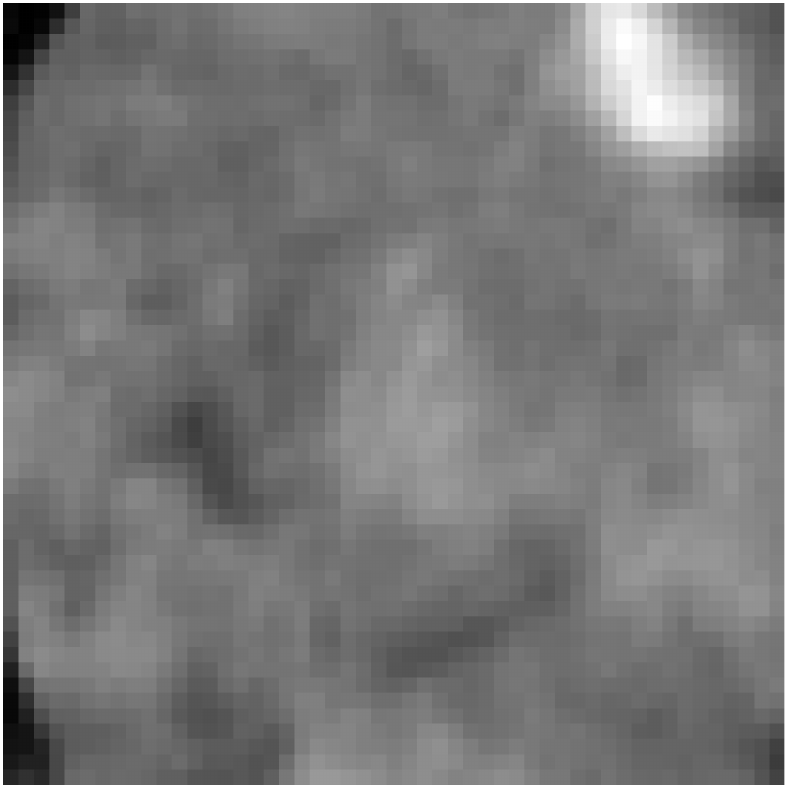}
\includegraphics[width=0.1\linewidth, angle=180]{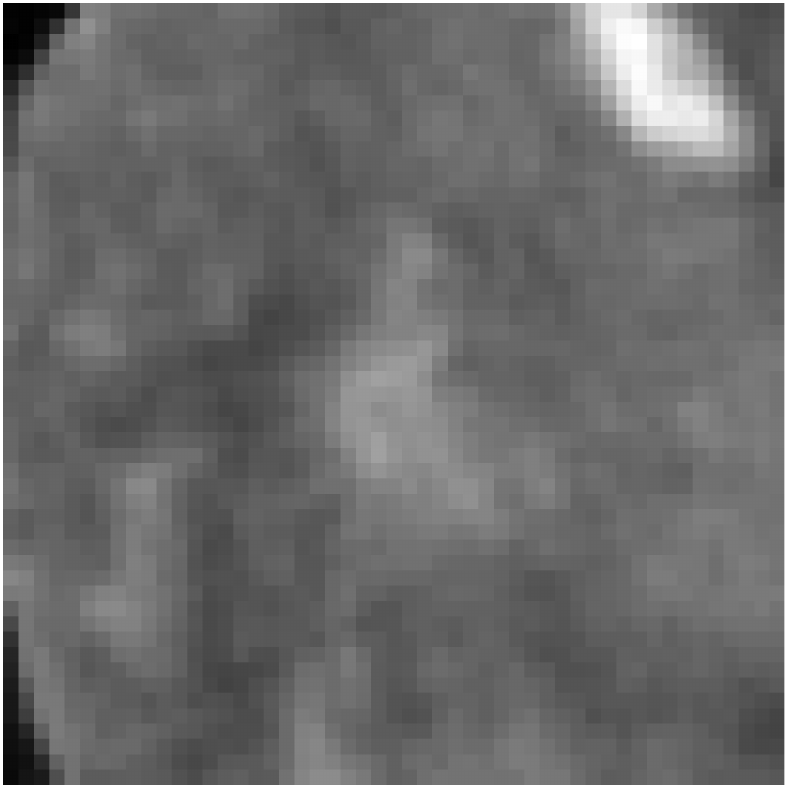}
\includegraphics[width=0.1\linewidth, angle=180]{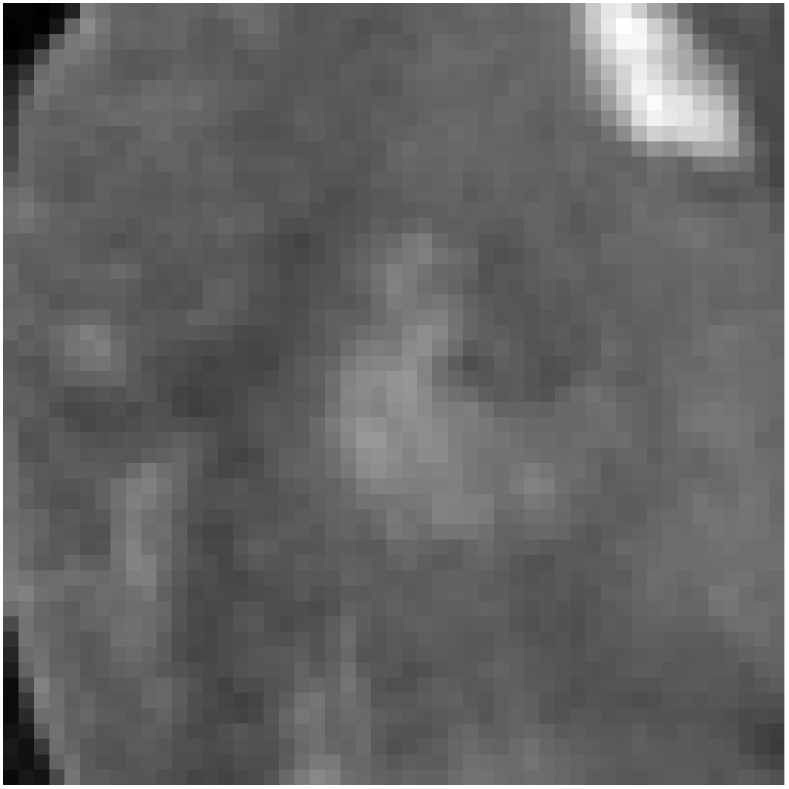}
\includegraphics[width=0.1\linewidth, angle=180]{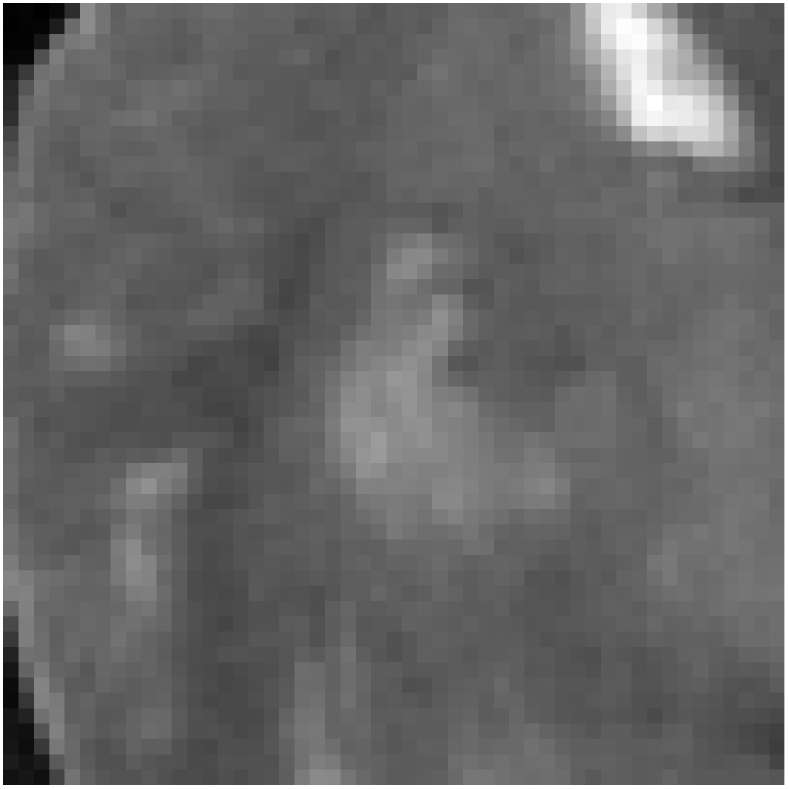}
\includegraphics[width=0.1\linewidth, angle=180]{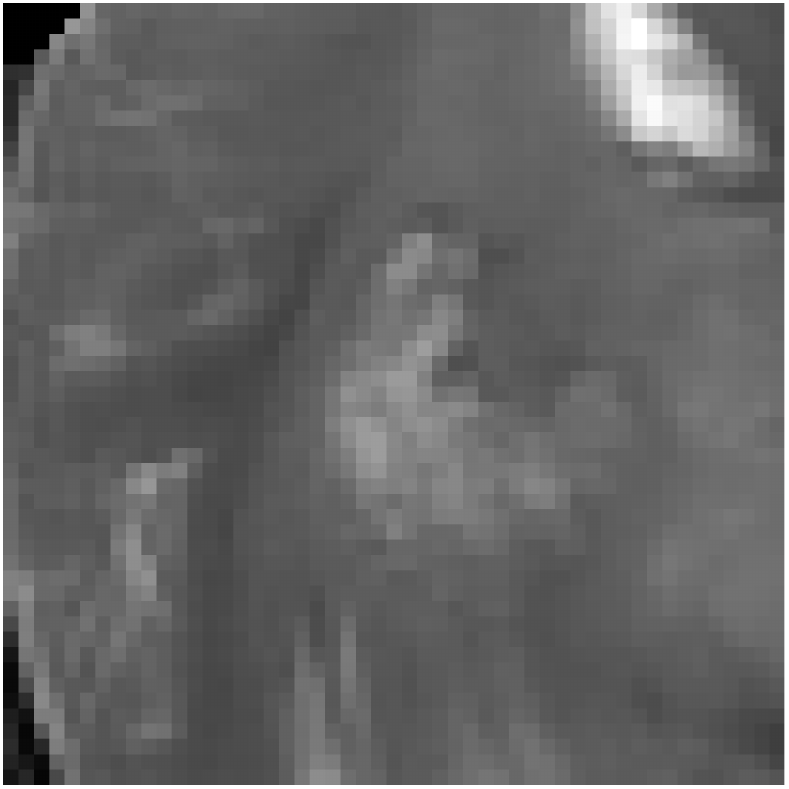}\\
\includegraphics[width=0.11\linewidth, angle=270]{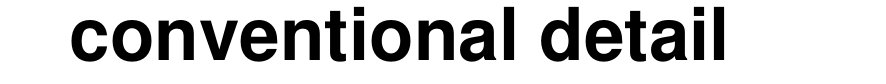}
\includegraphics[width=0.1\linewidth, angle=180]{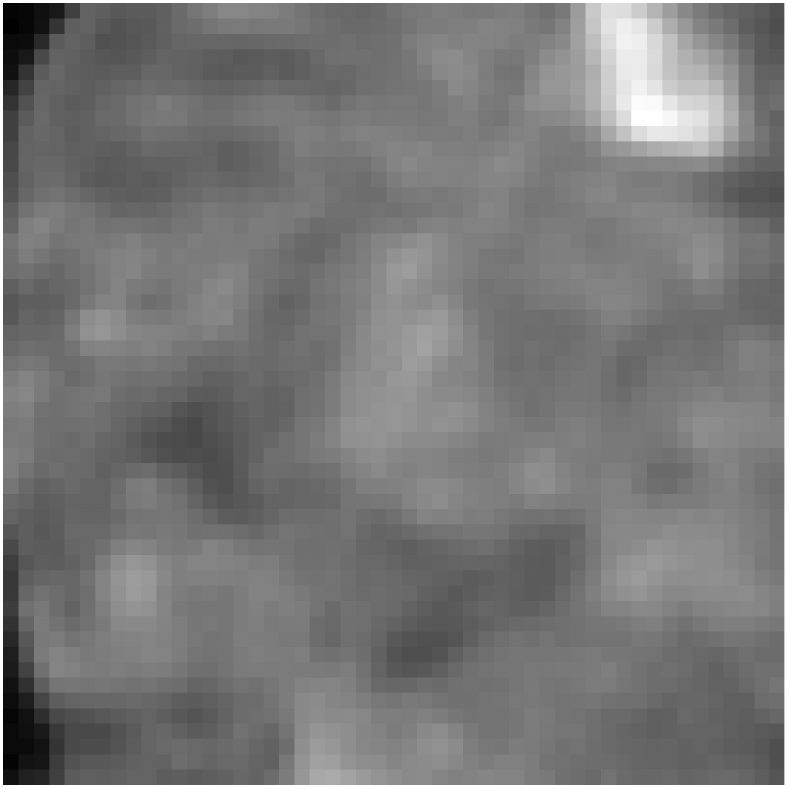}
\includegraphics[width=0.1\linewidth, angle=180]{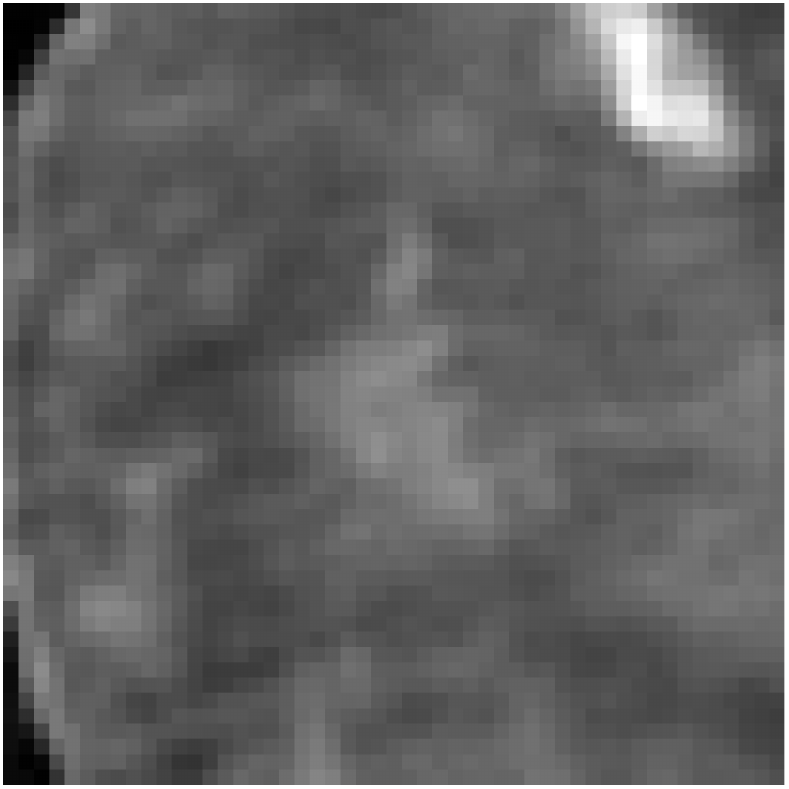}
\includegraphics[width=0.1\linewidth, angle=180]{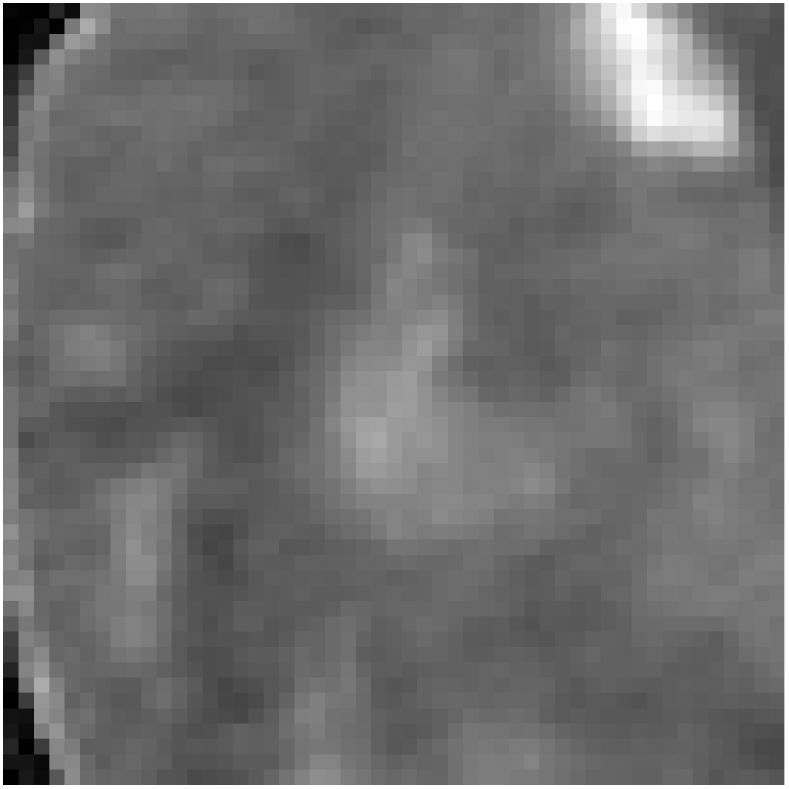}
\includegraphics[width=0.1\linewidth, angle=180]{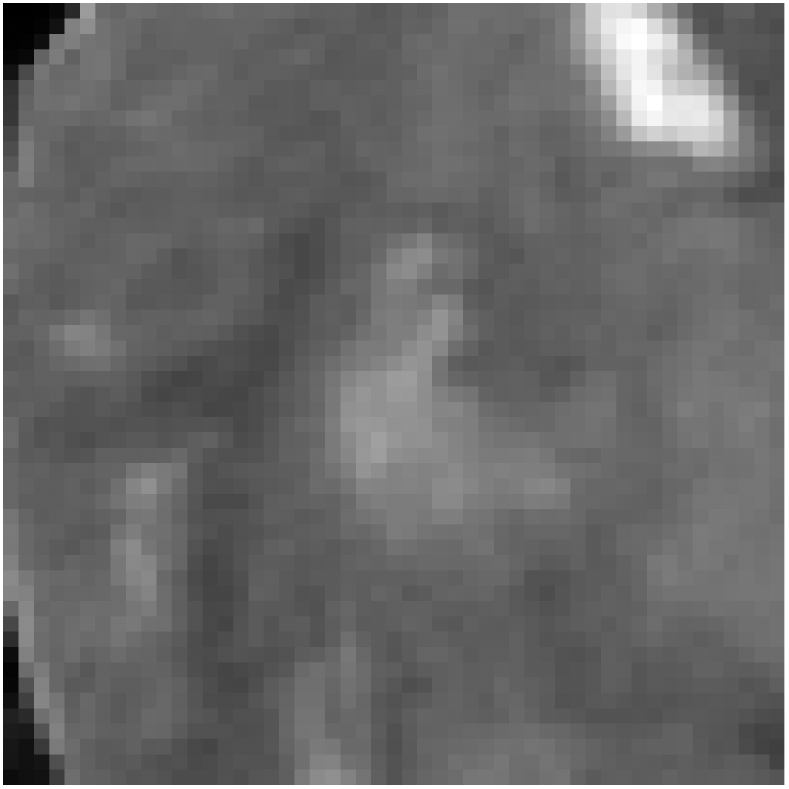}
\includegraphics[width=0.1\linewidth, angle=180]{fig_chp4/white.pdf}\\
\includegraphics[width=0.11\linewidth, angle=270]{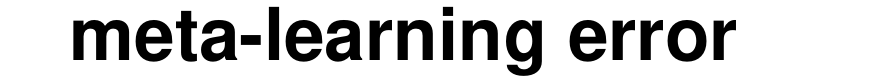}
\includegraphics[width=0.1\linewidth, angle=180]{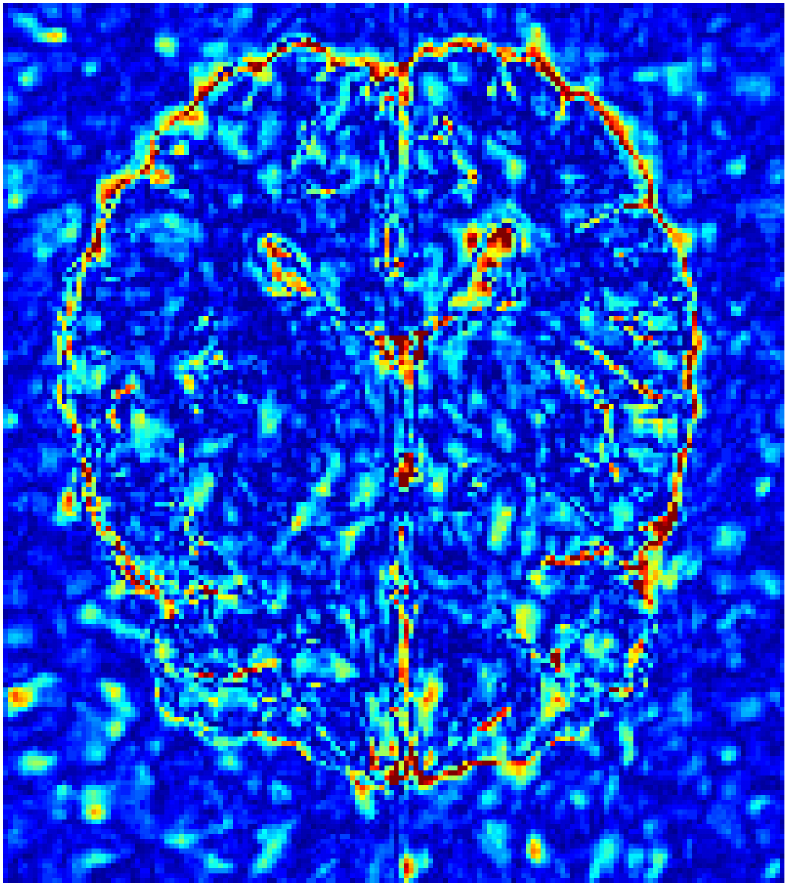}
\includegraphics[width=0.1\linewidth, angle=180]{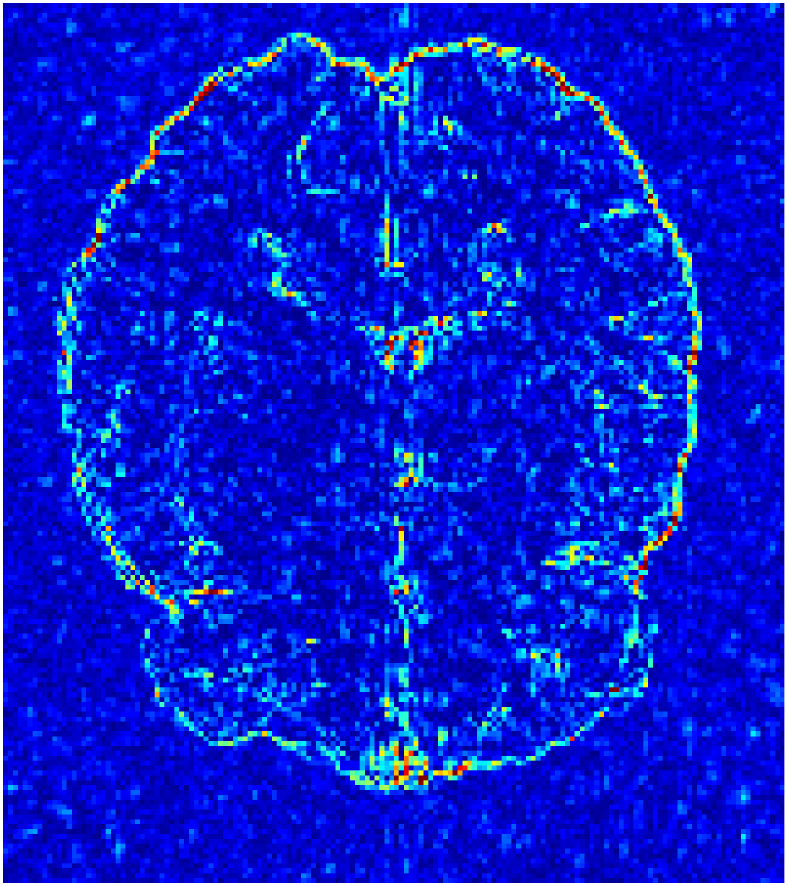}
\includegraphics[width=0.1\linewidth, angle=180]{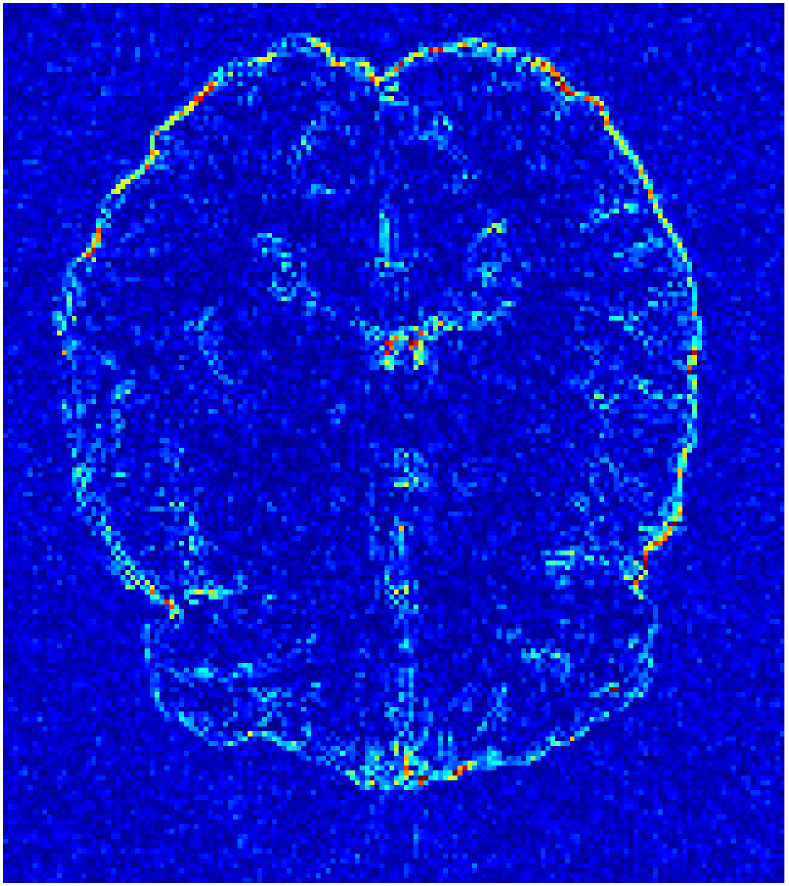}
\includegraphics[width=0.1\linewidth, angle=180]{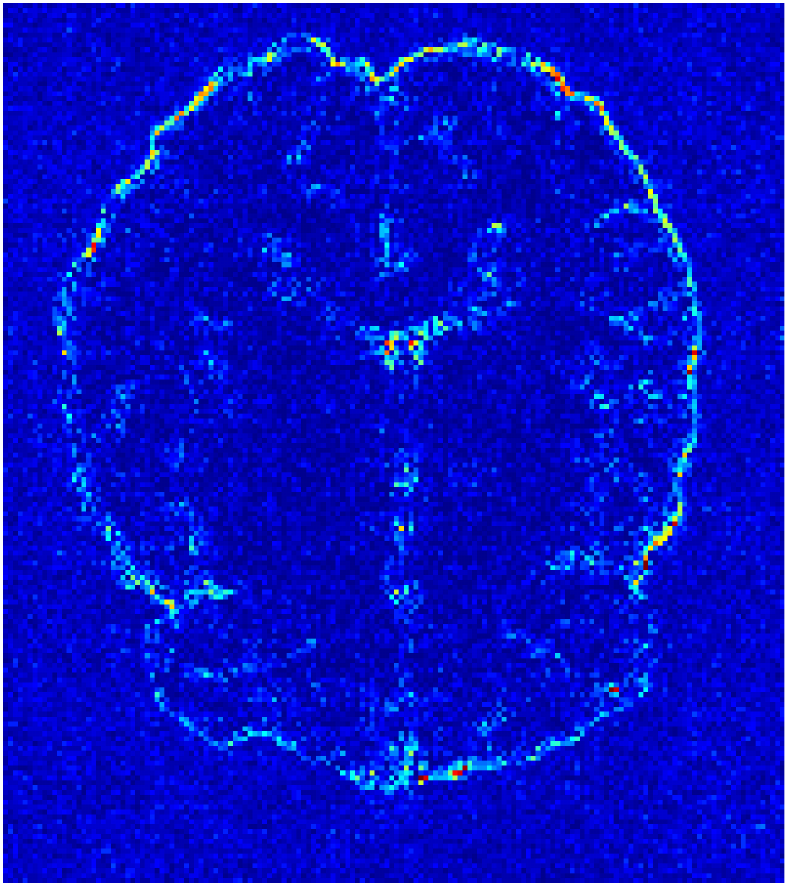}
\includegraphics[width=0.1\linewidth, angle=180]{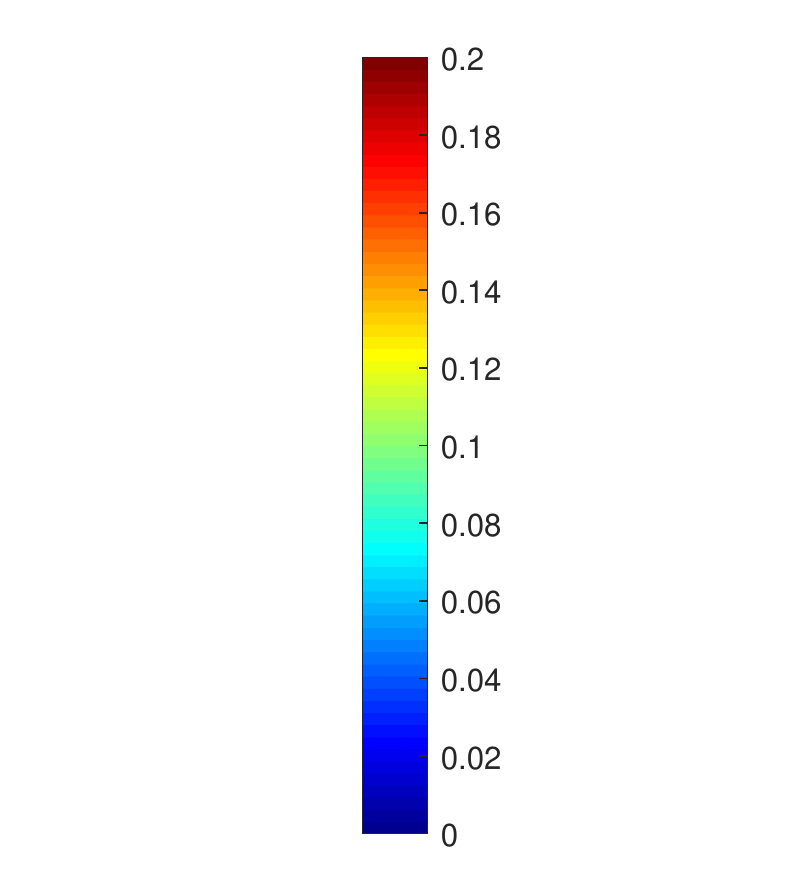}\\
\includegraphics[width=0.11\linewidth, angle=270]{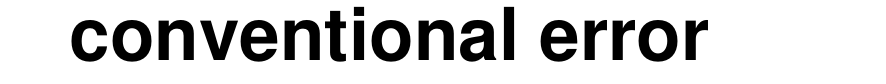}
\includegraphics[width=0.1\linewidth, angle=180]{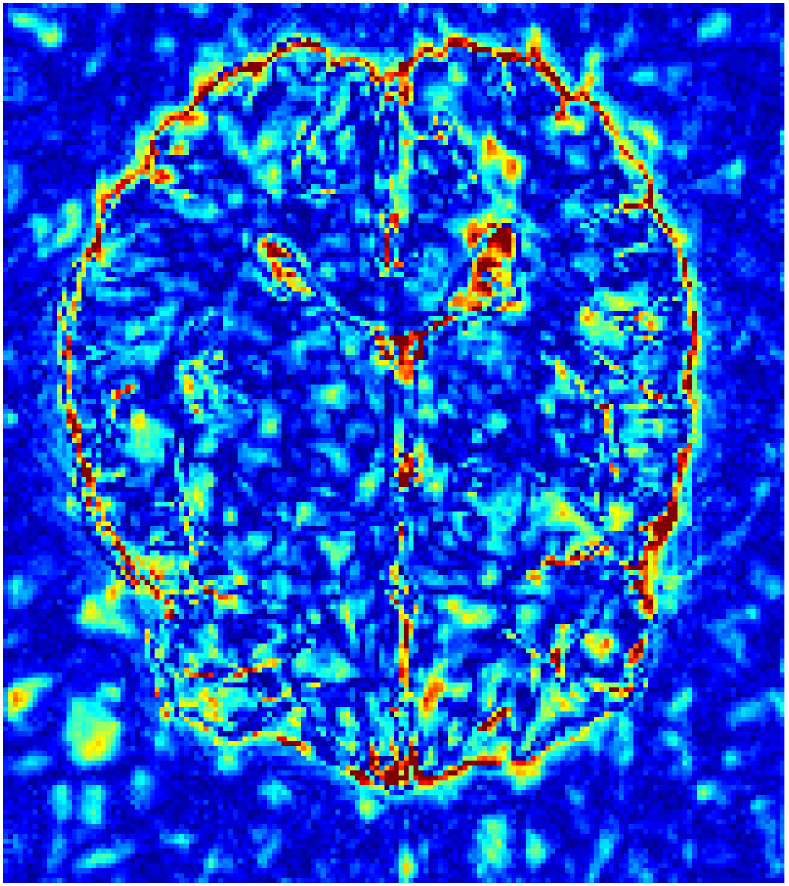}
\includegraphics[width=0.1\linewidth, angle=180]{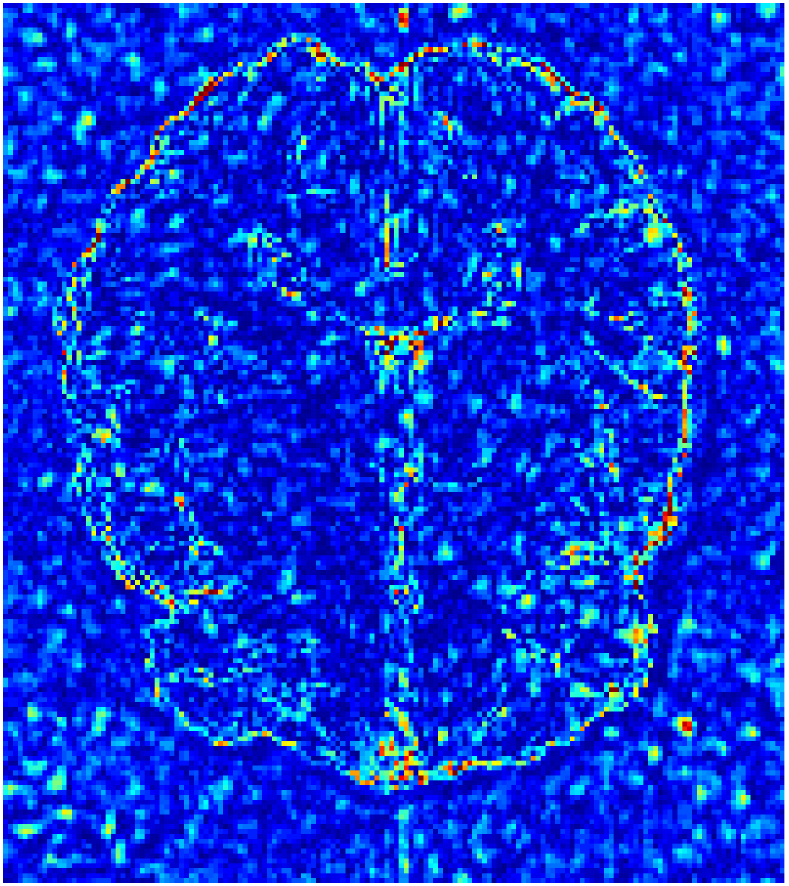}
\includegraphics[width=0.1\linewidth, angle=180]{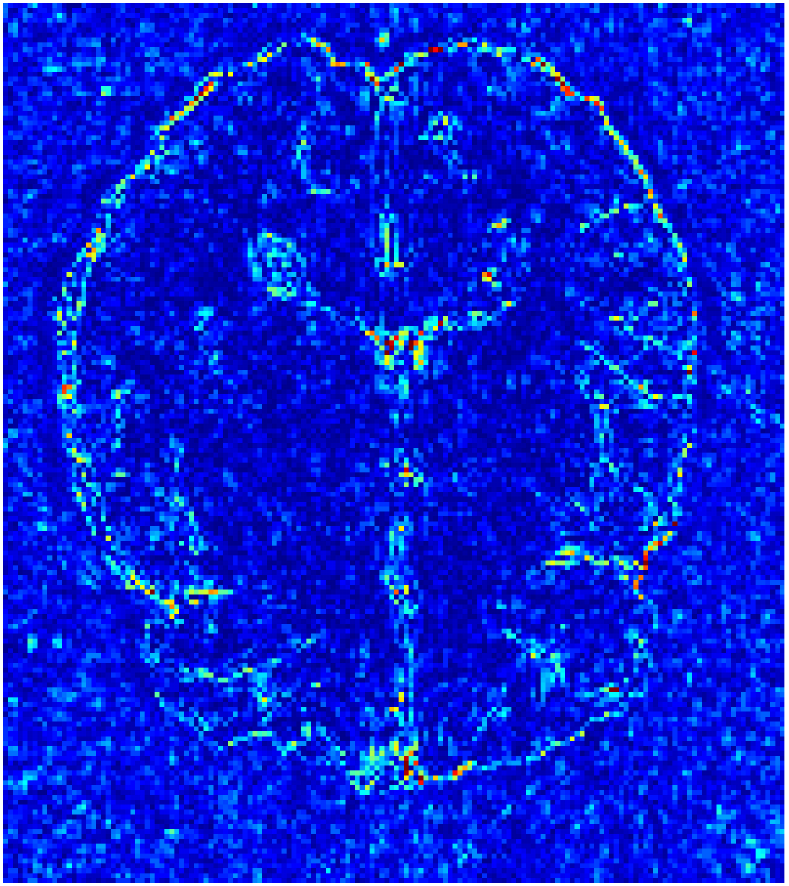}
\includegraphics[width=0.1\linewidth, angle=180]{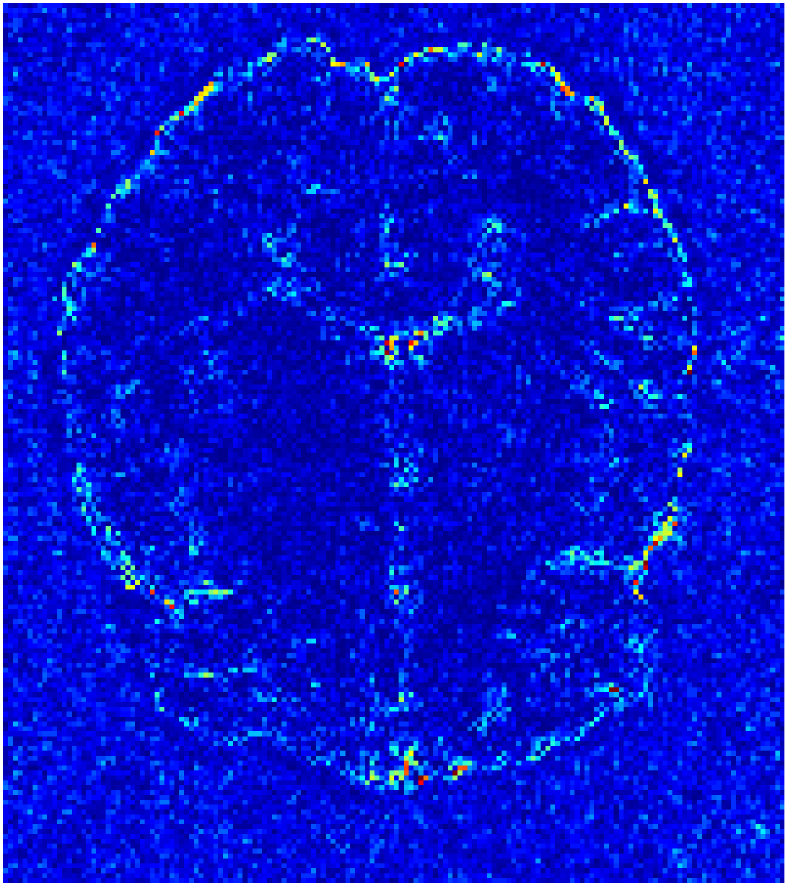}
\includegraphics[width=0.1\linewidth, angle=180]{fig_chp4/white.pdf}\\
\includegraphics[width=0.11\linewidth, angle=90]{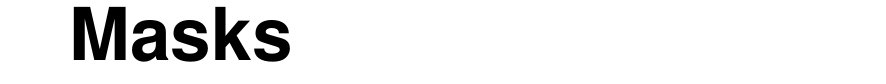}
\includegraphics[width=0.1\linewidth]{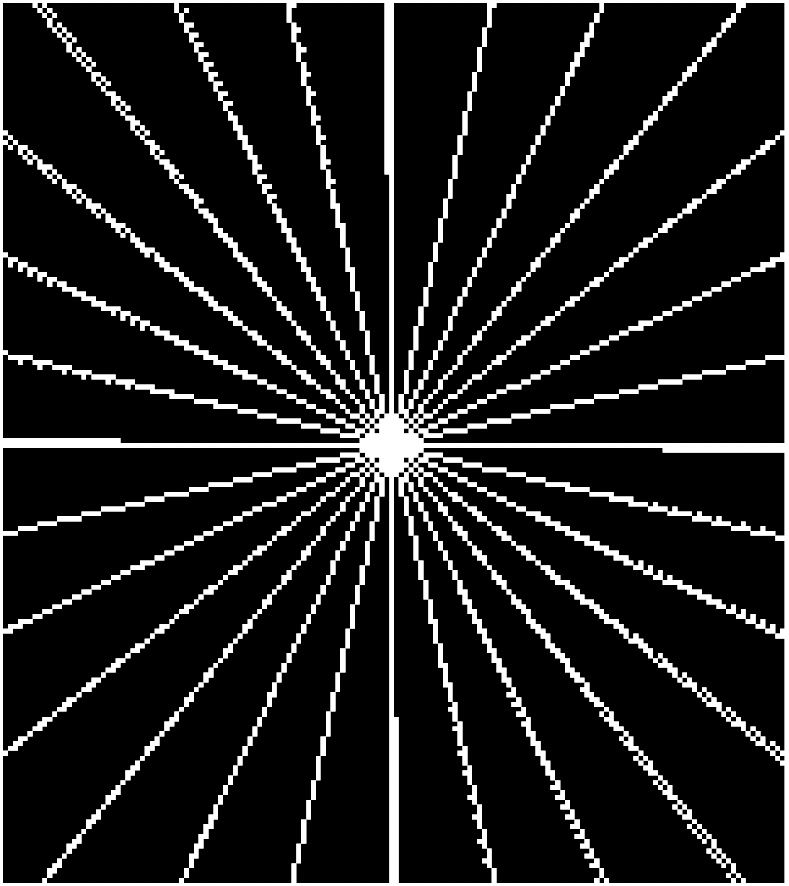}
\includegraphics[width=0.1\linewidth]{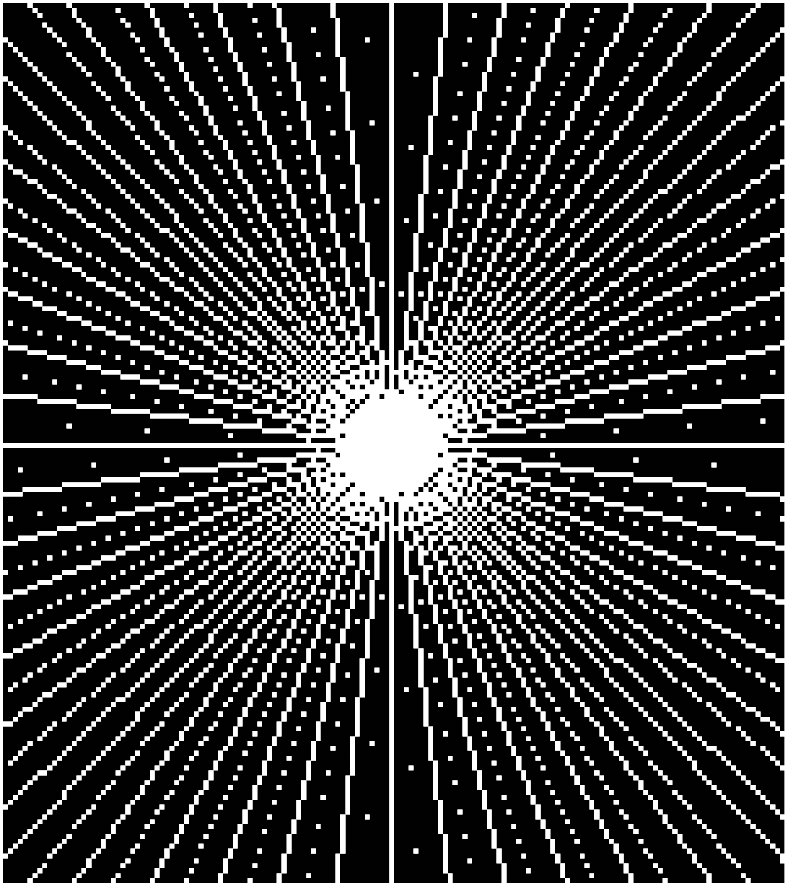}
\includegraphics[width=0.1\linewidth]{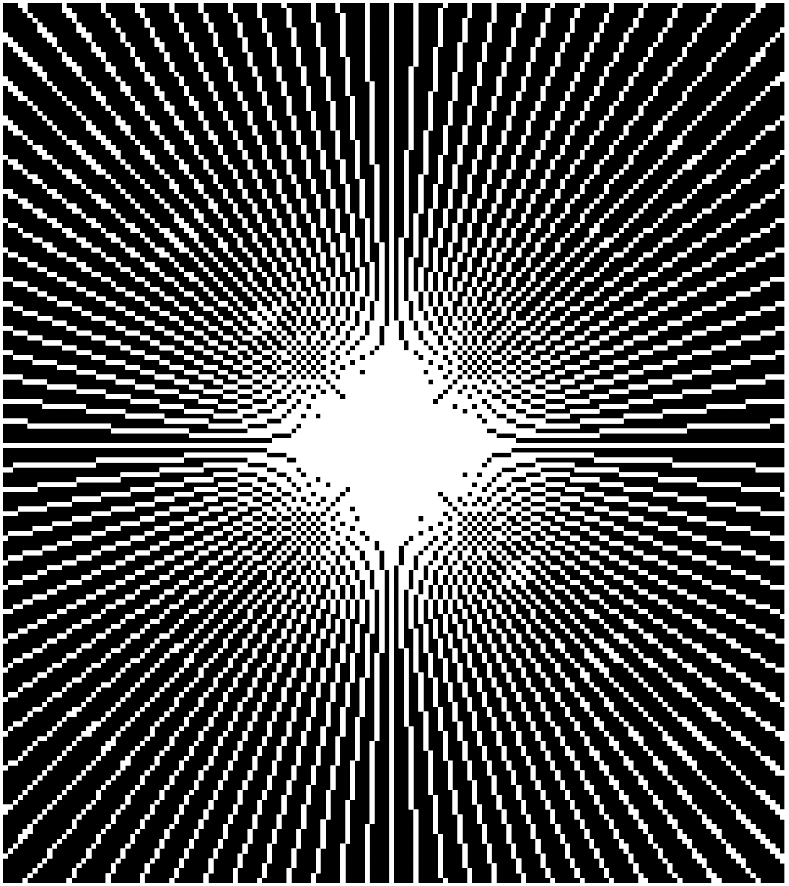}
\includegraphics[width=0.1\linewidth]{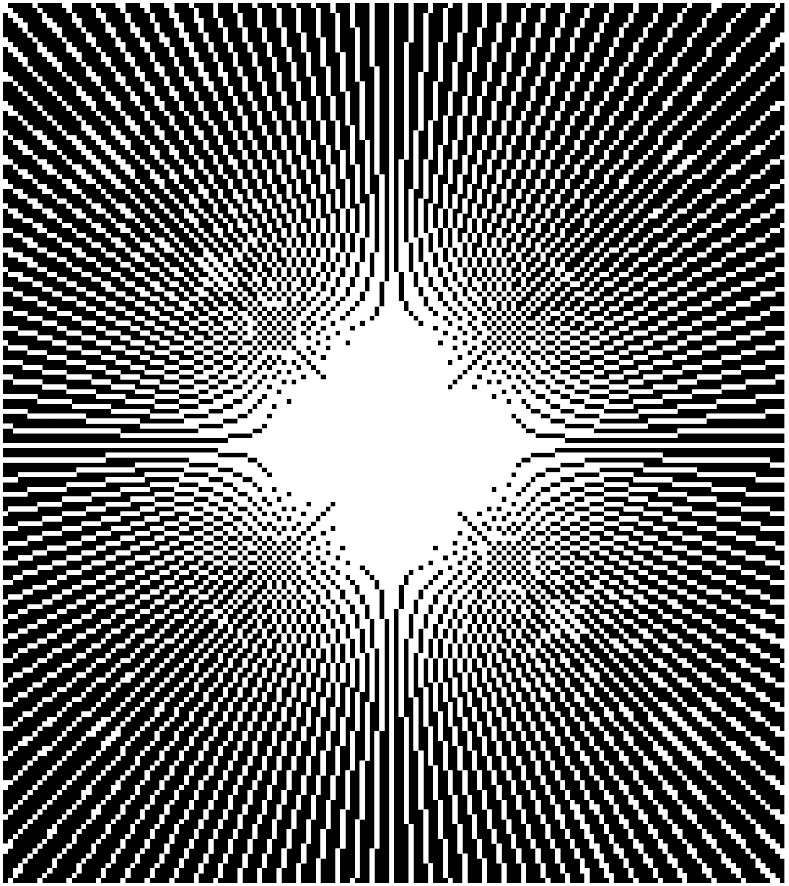}
\includegraphics[width=0.1\linewidth]{fig_chp4/white.pdf}
\caption{From top to bottom: The T2 brain image reconstruction results, zoomed-in details, point-wise errors with a color bar, and associated
 \textbf{{radial}} masks. The top-right image is the ground truth fully-sampled image. }
\label{figure_same_ratio_t2}
\end{figure}

\begin{figure}[H]
\centering
\includegraphics[width=0.1\linewidth, angle=270]{fig_chp4/meta_result.pdf}%, height=0.1\linewidth
\includegraphics[width=0.1\linewidth, angle=180]{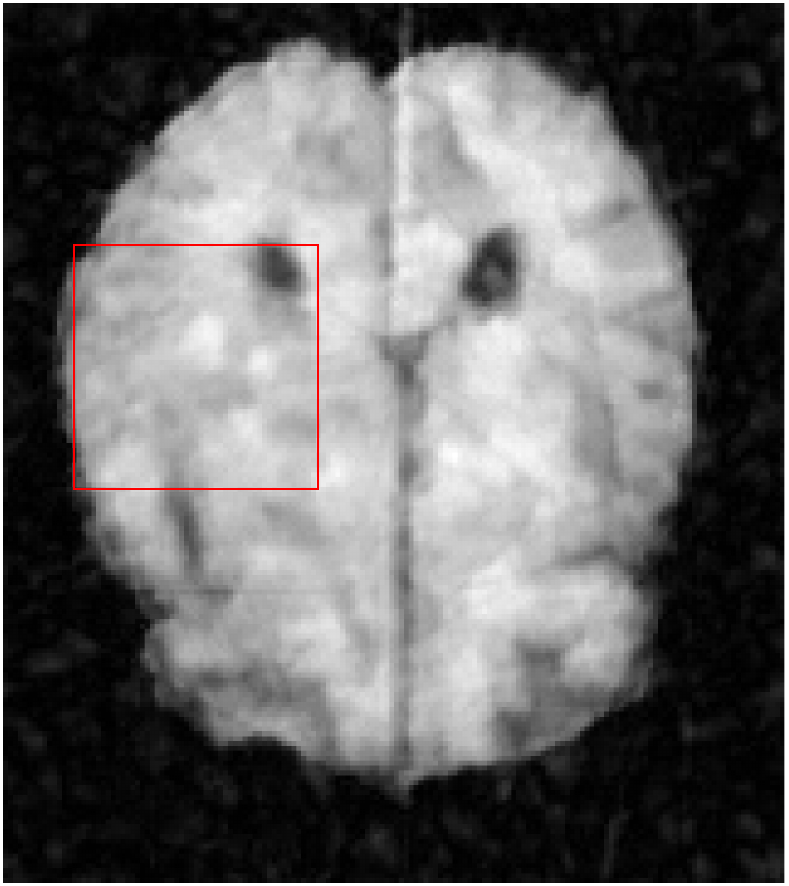}%, height=0.01\linewidth
\includegraphics[width=0.1\linewidth, angle=180]{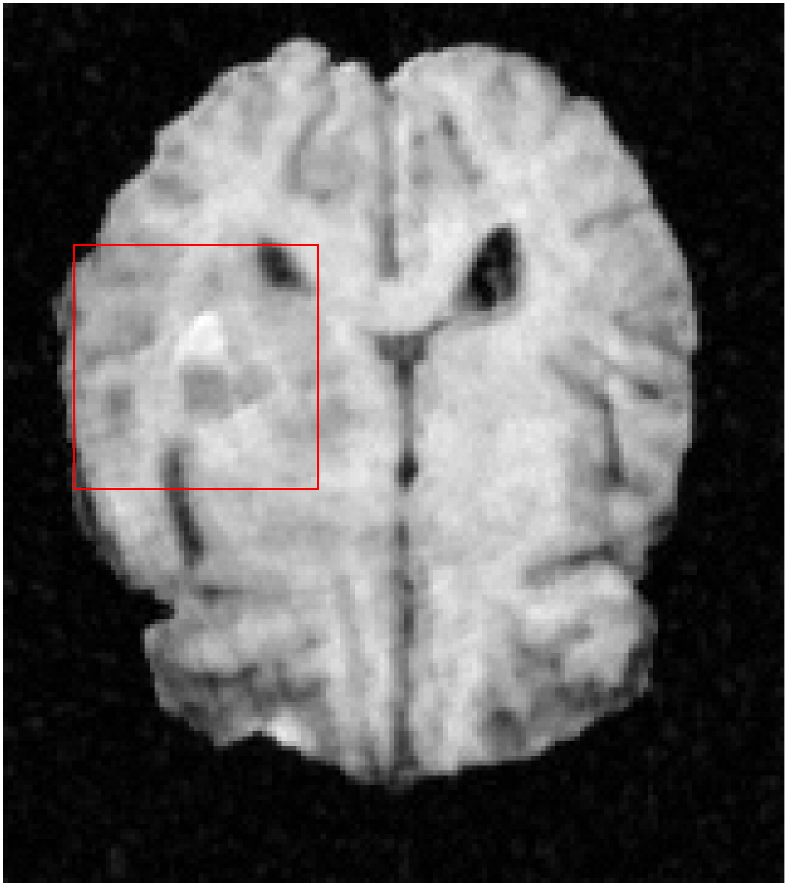}
\includegraphics[width=0.1\linewidth, angle=180]{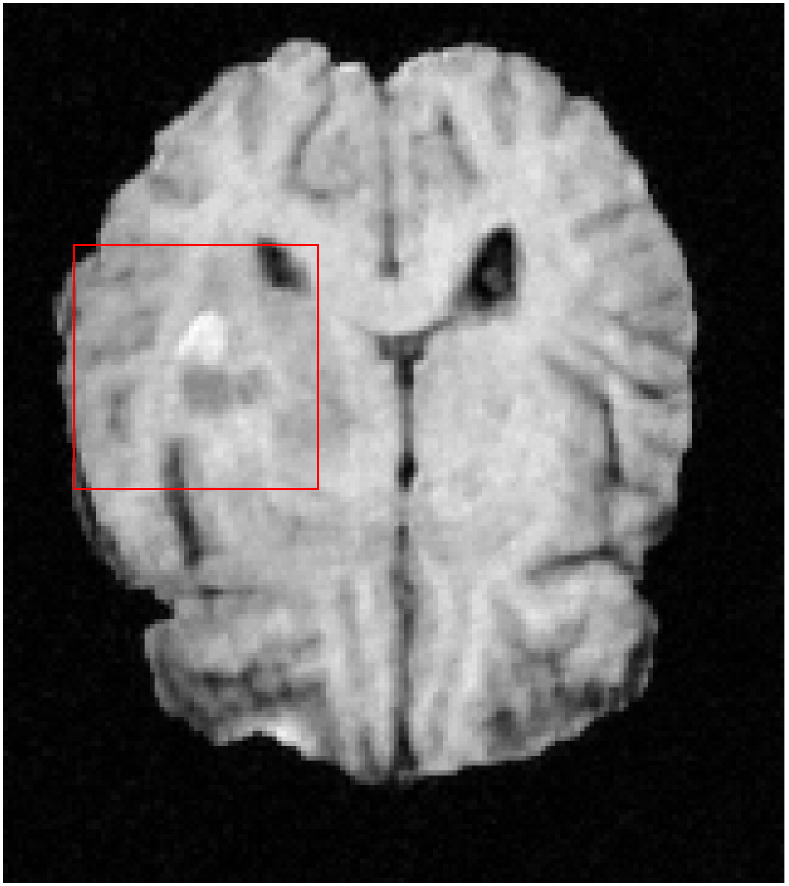}
\includegraphics[width=0.1\linewidth, angle=180]{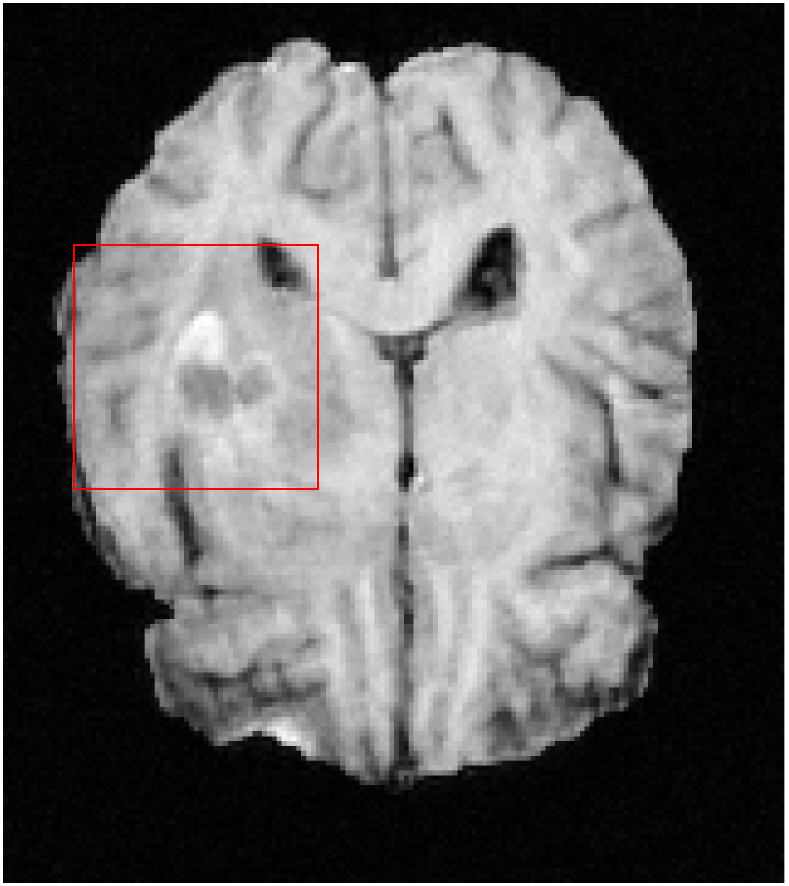}
\includegraphics[width=0.1\linewidth, angle=180]{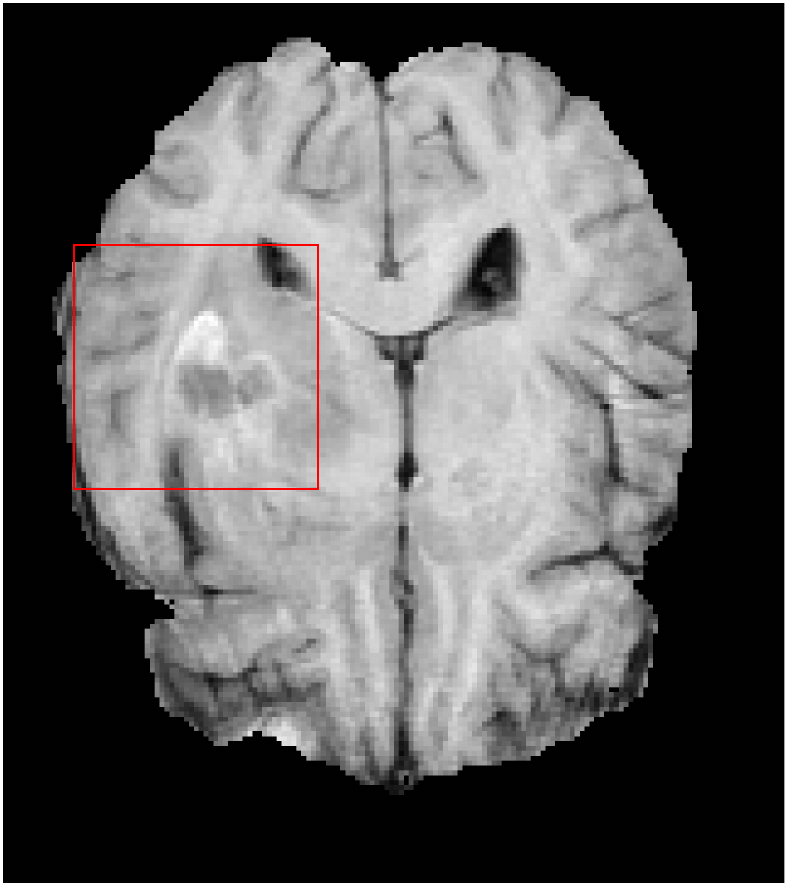}\\
\includegraphics[width=0.1\linewidth, angle=270]{fig_chp4/conventional_result.pdf}%, height=0.17\linewidth
\includegraphics[width=0.1\linewidth, angle=180]{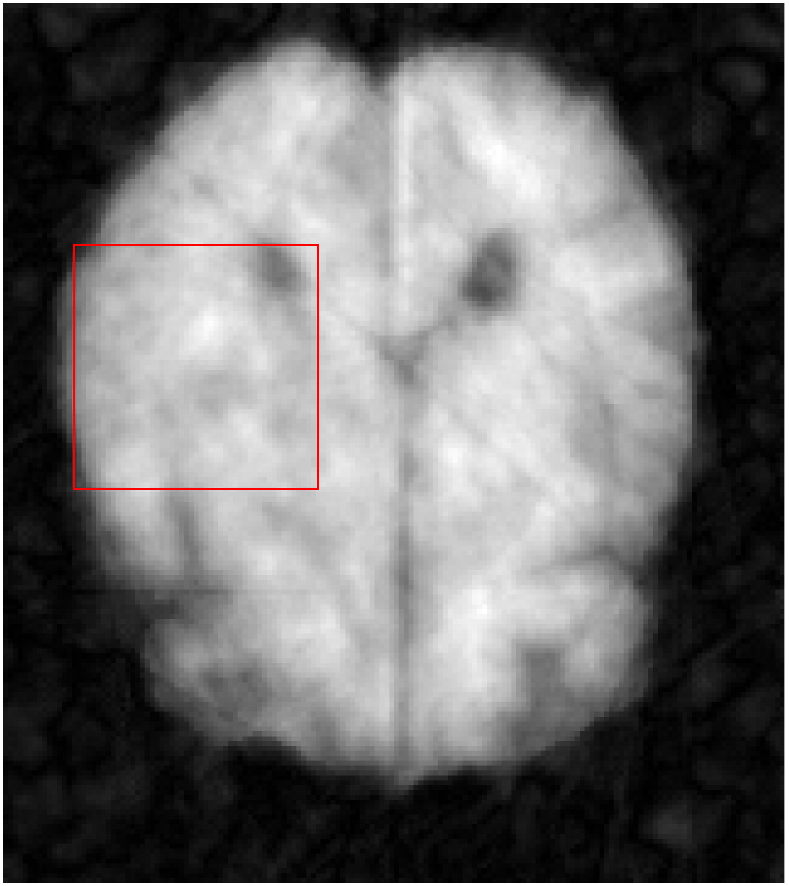}
\includegraphics[width=0.1\linewidth, angle=180]{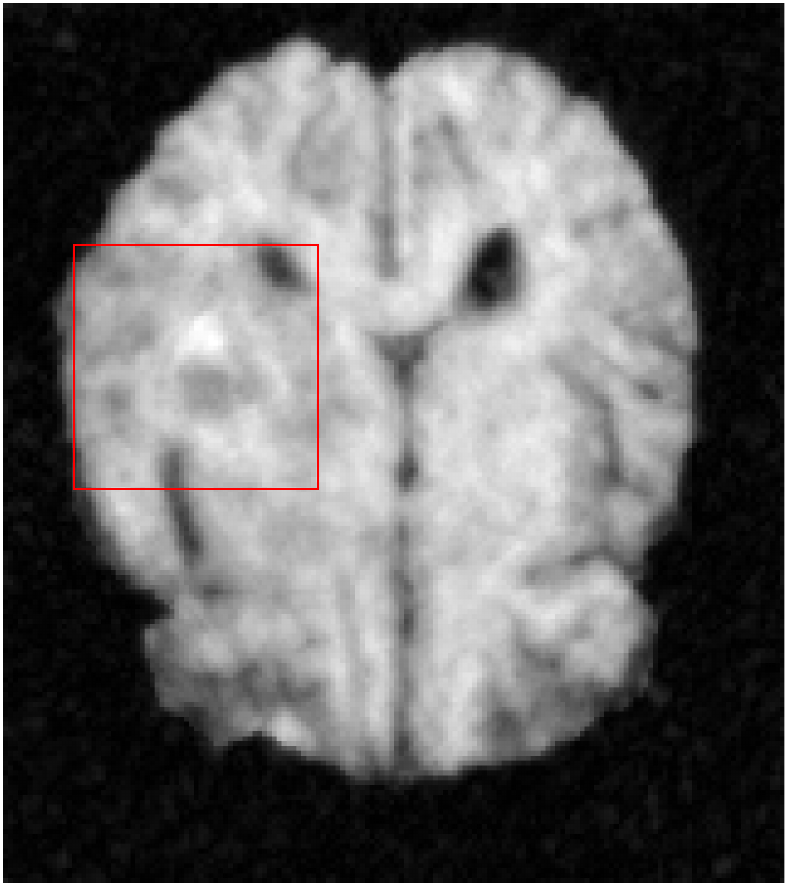}
\includegraphics[width=0.1\linewidth, angle=180]{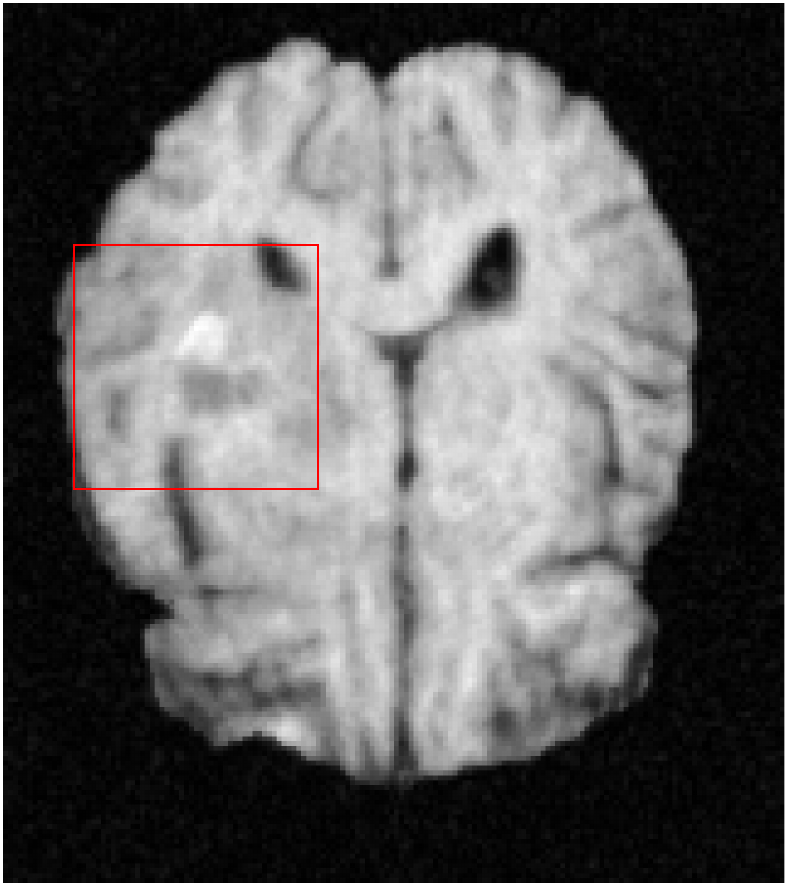}
\includegraphics[width=0.1\linewidth, angle=180]{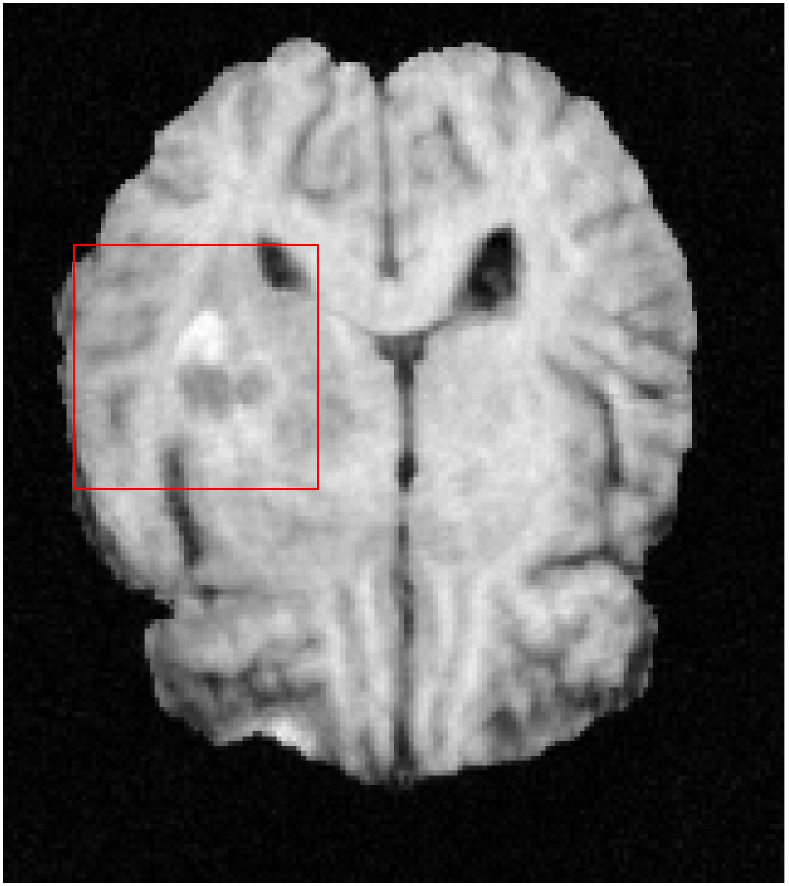}
\includegraphics[width=0.1\linewidth, angle=180]{fig_chp4/white.pdf}\\
\includegraphics[width=0.1\linewidth, angle=270]{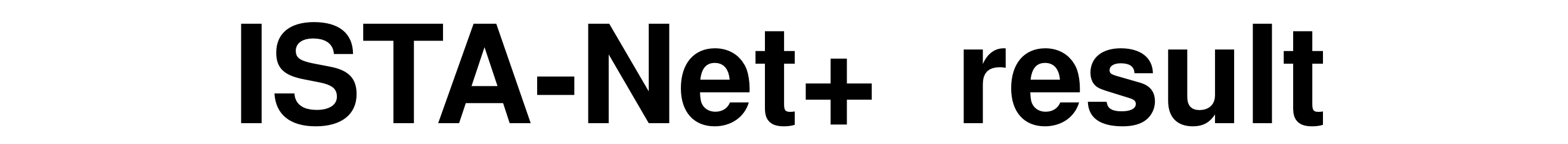}%, height=0.17\linewidth
\includegraphics[width=0.1\linewidth, angle=180]{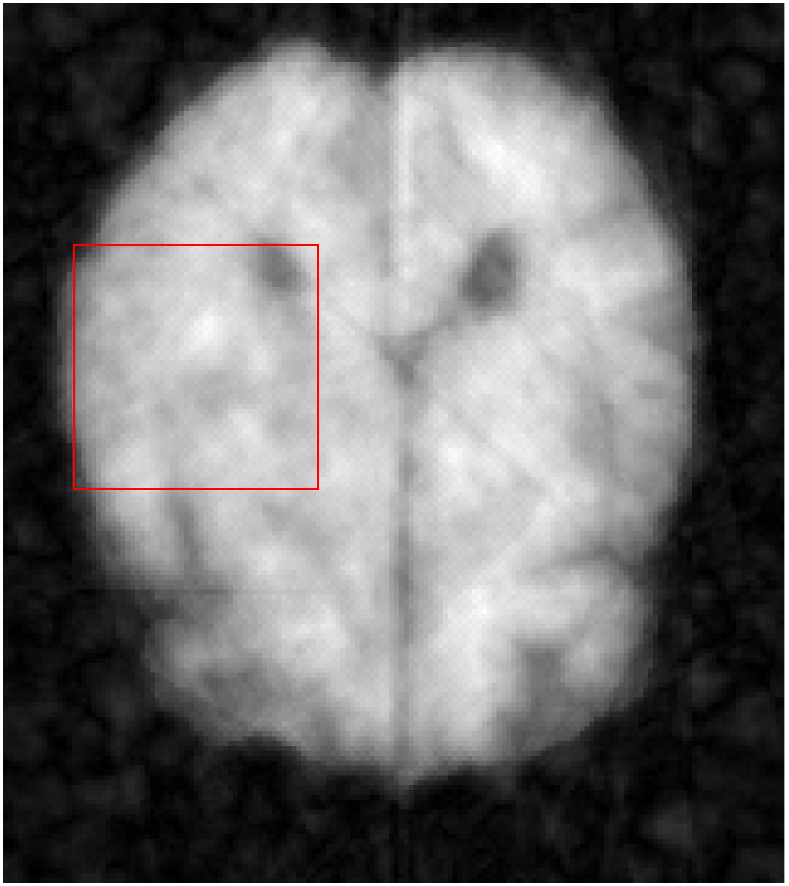}
\includegraphics[width=0.1\linewidth, angle=180]{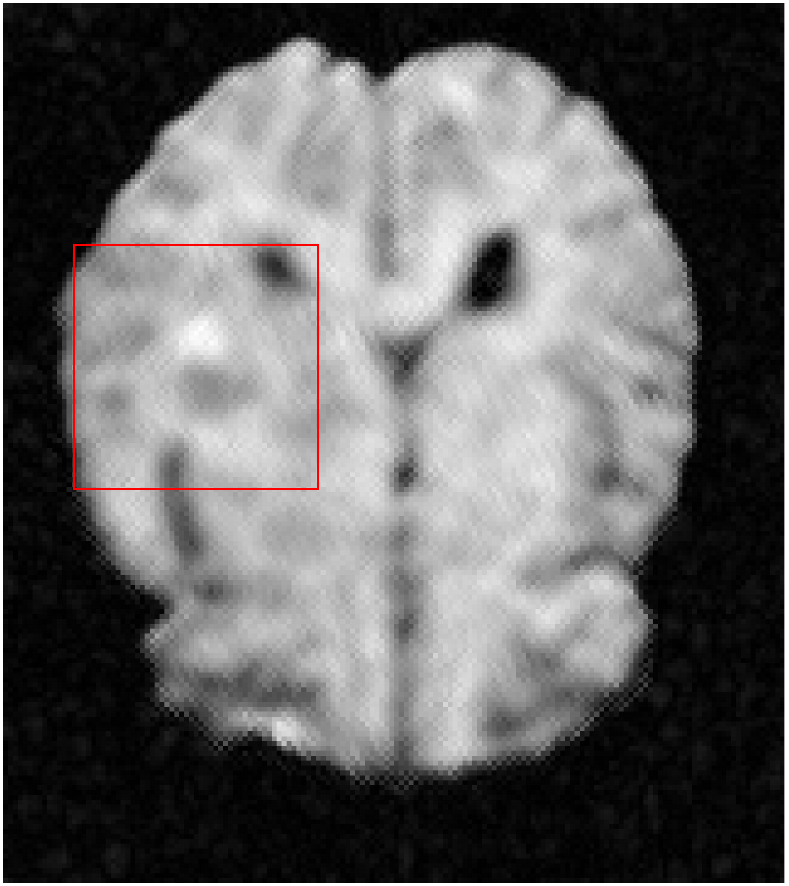}
\includegraphics[width=0.1\linewidth, angle=180]{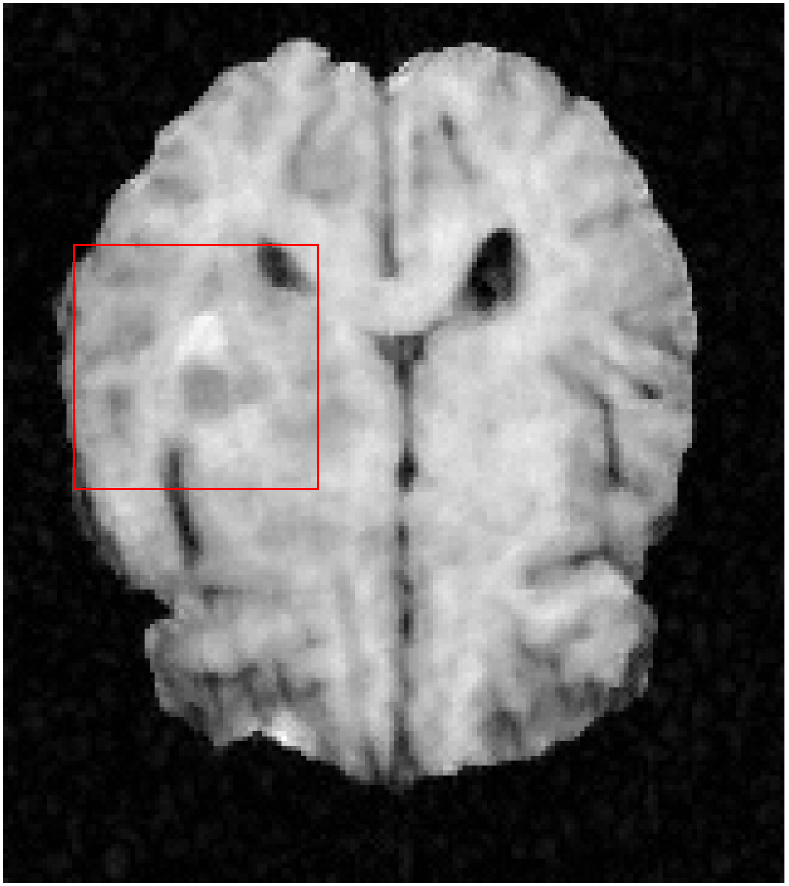}
\includegraphics[width=0.1\linewidth, angle=180]{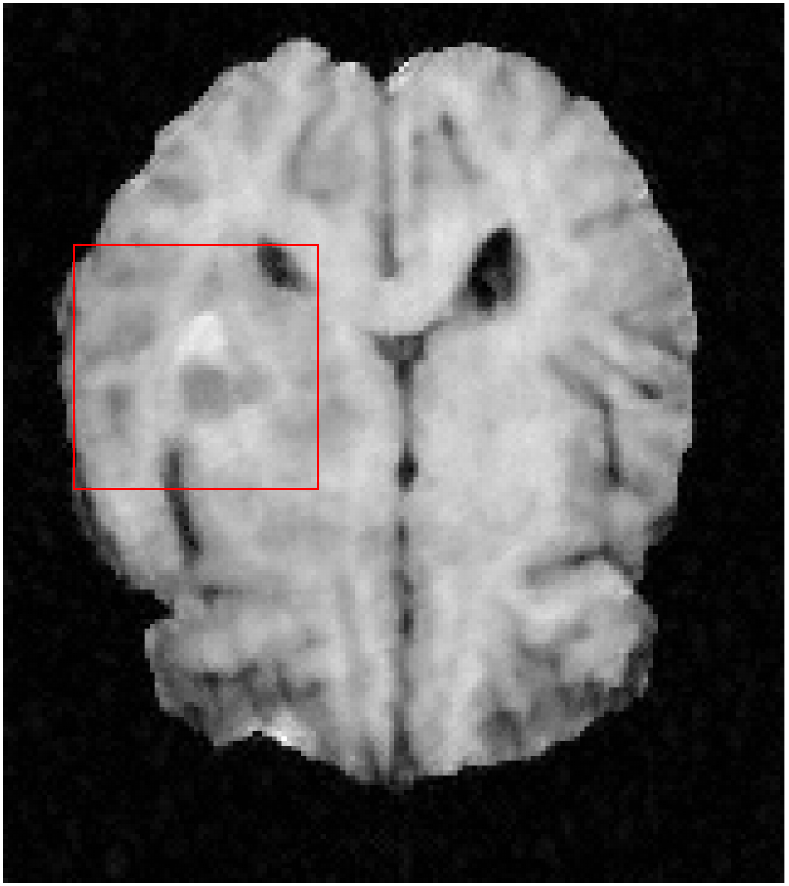}
\includegraphics[width=0.1\linewidth, angle=180]{fig_chp4/white.pdf}\\
\includegraphics[width=0.1\linewidth, angle=270]{fig_chp4/meta_detail.pdf}%, height=0.01\linewidth
\includegraphics[width=0.1\linewidth, angle=180]{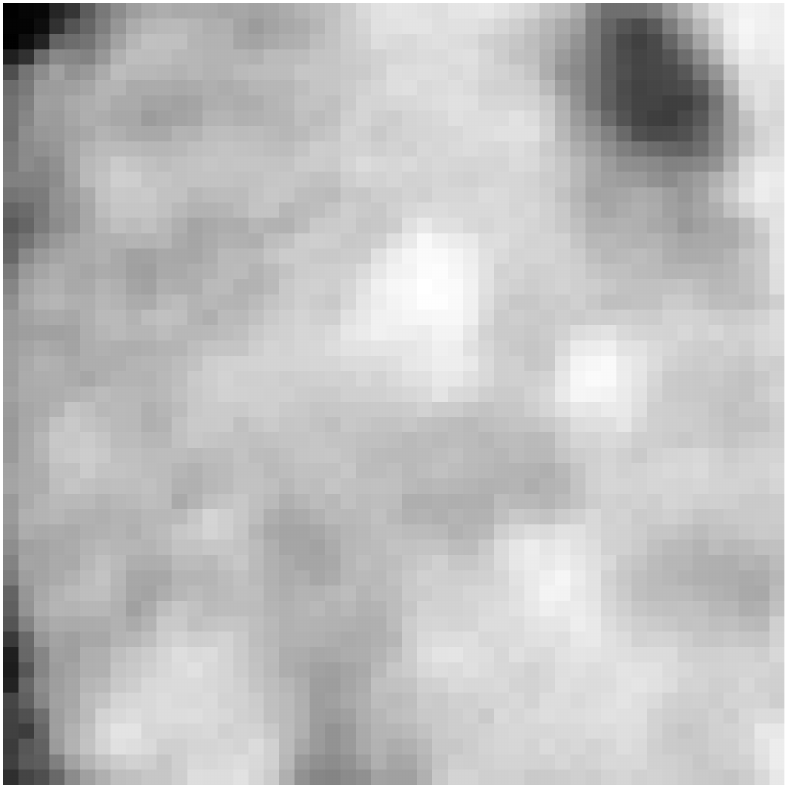}
\includegraphics[width=0.1\linewidth, angle=180]{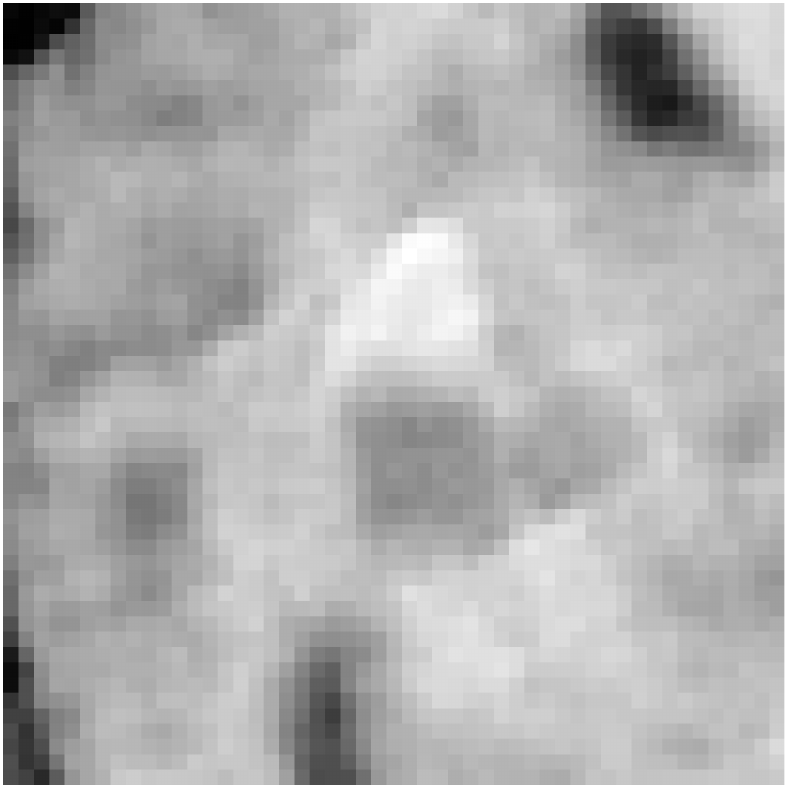}
\includegraphics[width=0.1\linewidth, angle=180]{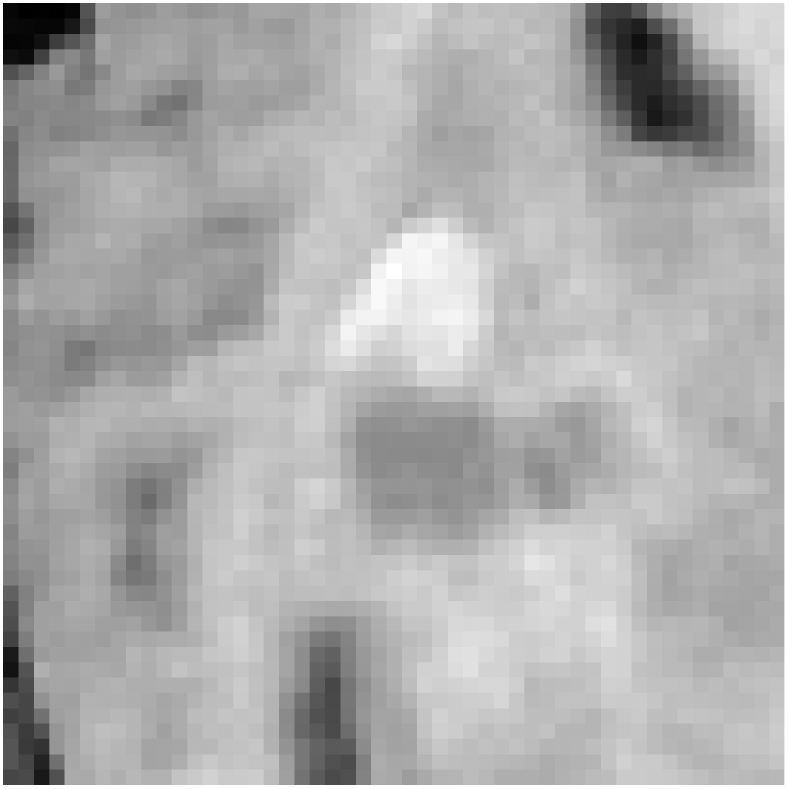}
\includegraphics[width=0.1\linewidth, angle=180]{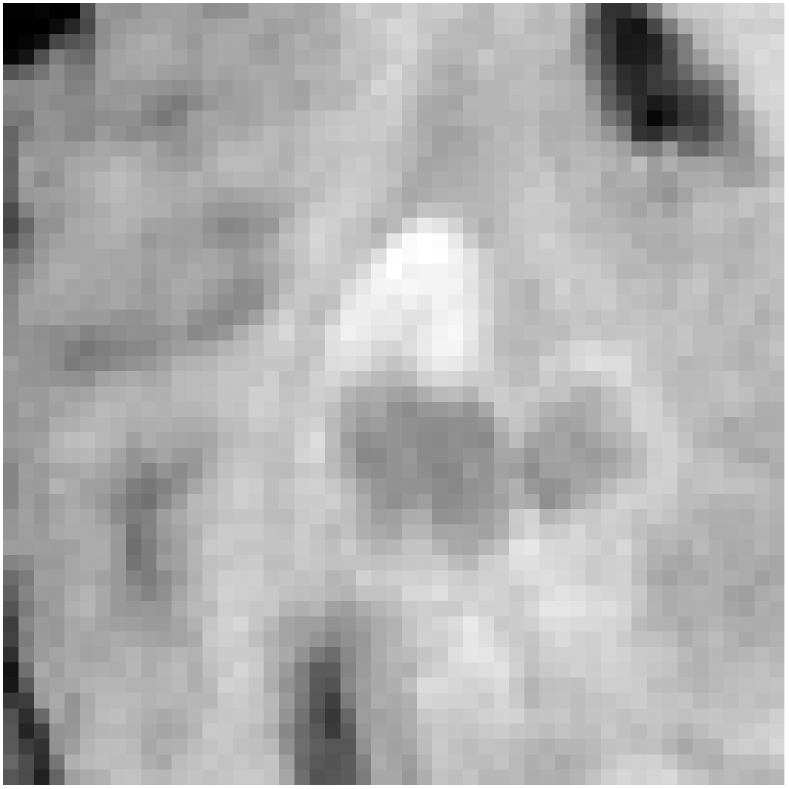}
\includegraphics[width=0.1\linewidth, angle=180]{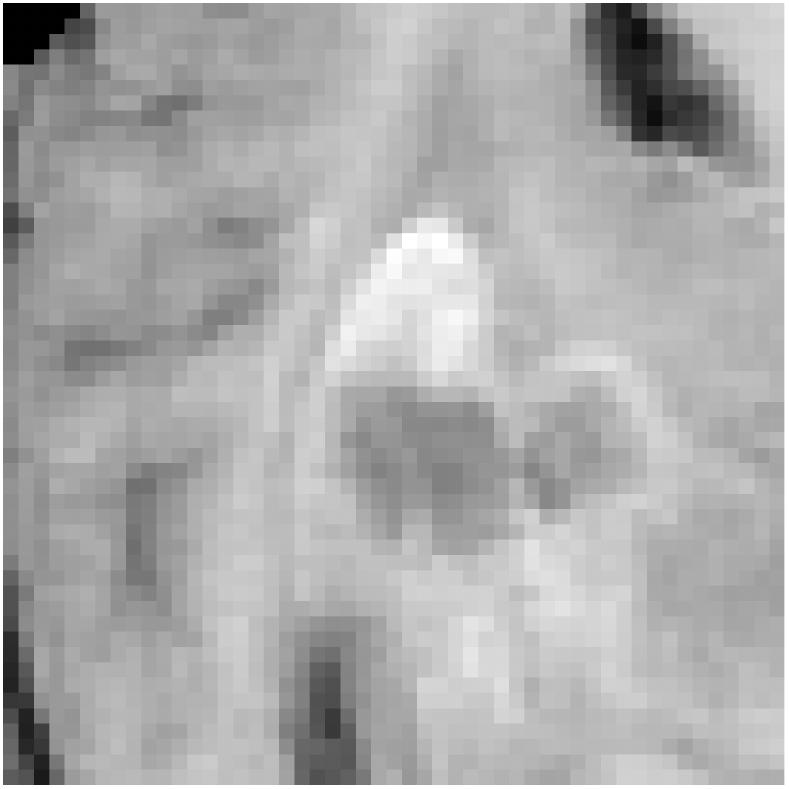}\\
\includegraphics[width=0.1\linewidth, angle=270]{fig_chp4/conventional_detail.pdf}%, height=0.17\linewidth
\includegraphics[width=0.1\linewidth, angle=180]{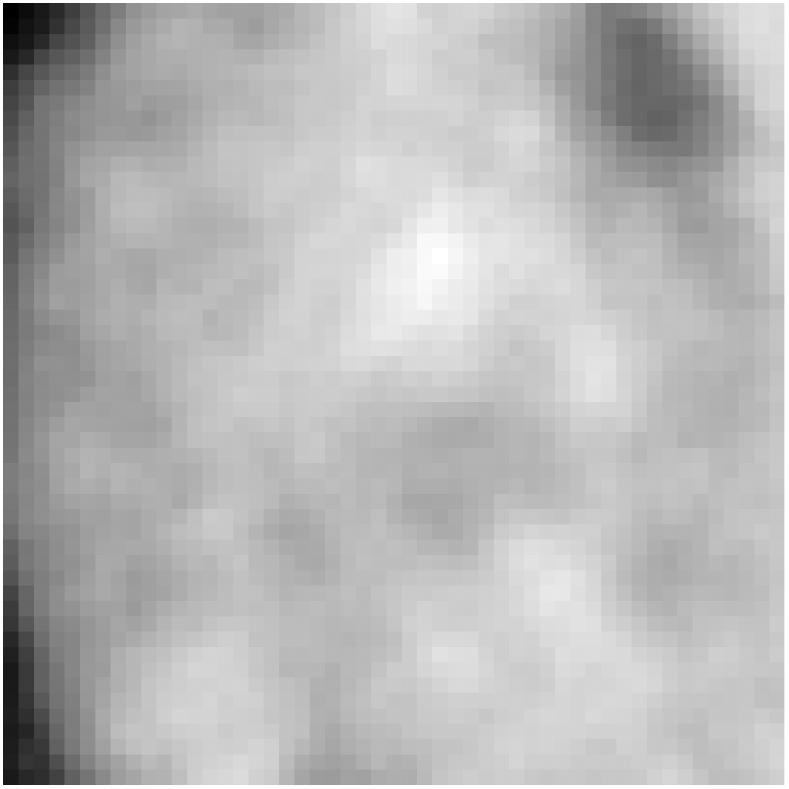}
\includegraphics[width=0.1\linewidth, angle=180]{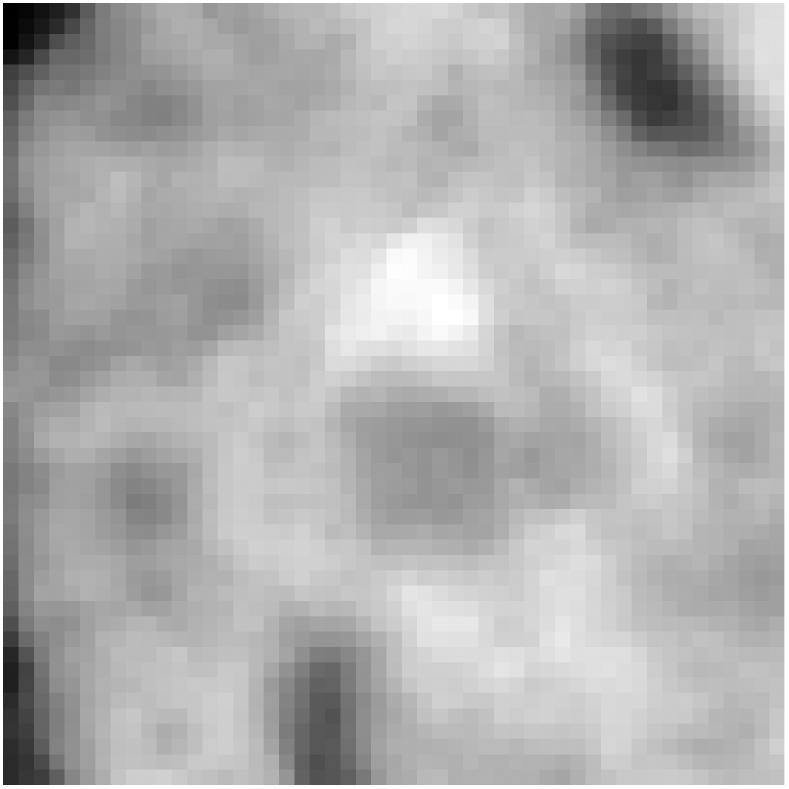}
\includegraphics[width=0.1\linewidth, angle=180]{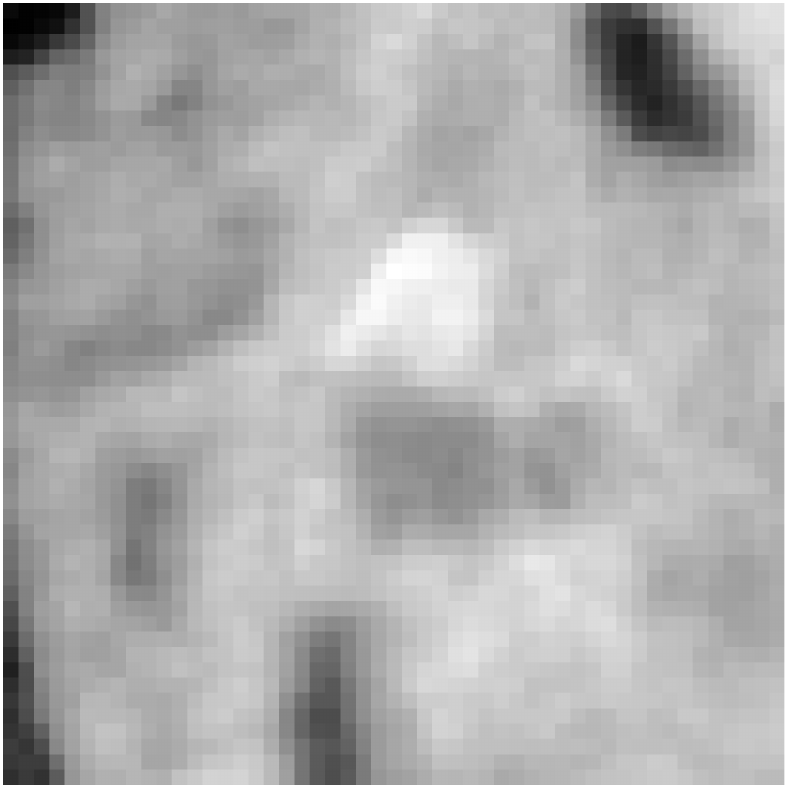}
\includegraphics[width=0.1\linewidth, angle=180]{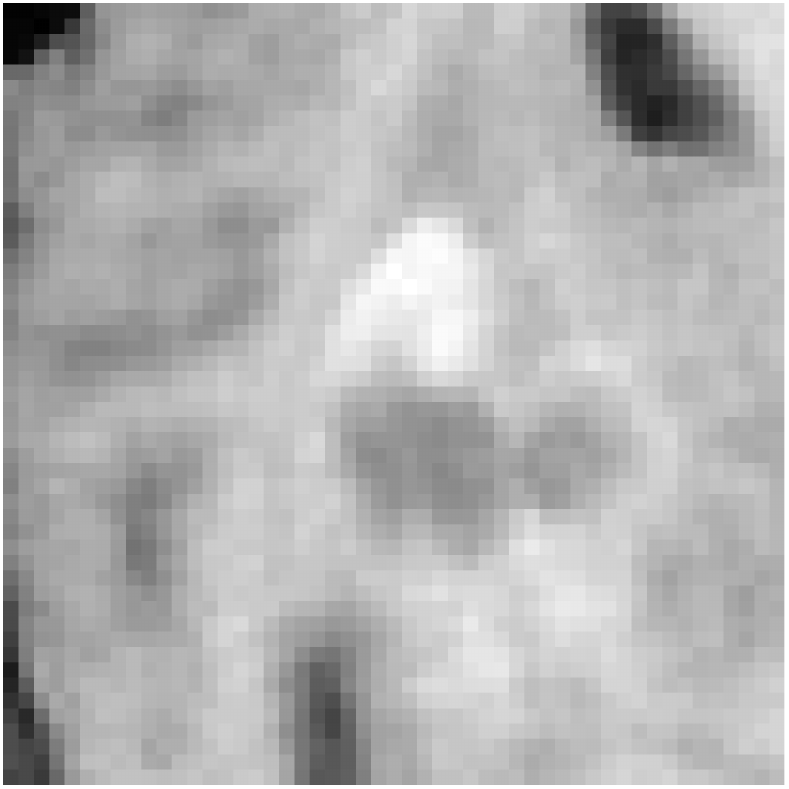}
\includegraphics[width=0.1\linewidth, angle=180]{fig_chp4/white.pdf}\\
\includegraphics[width=0.1\linewidth, angle=270]{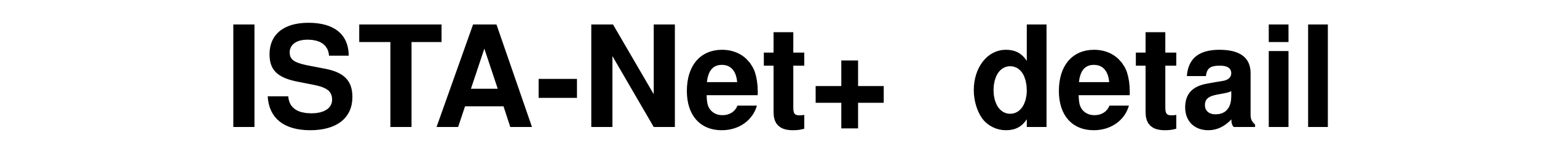}%, height=0.17\linewidth
\includegraphics[width=0.1\linewidth, angle=180]{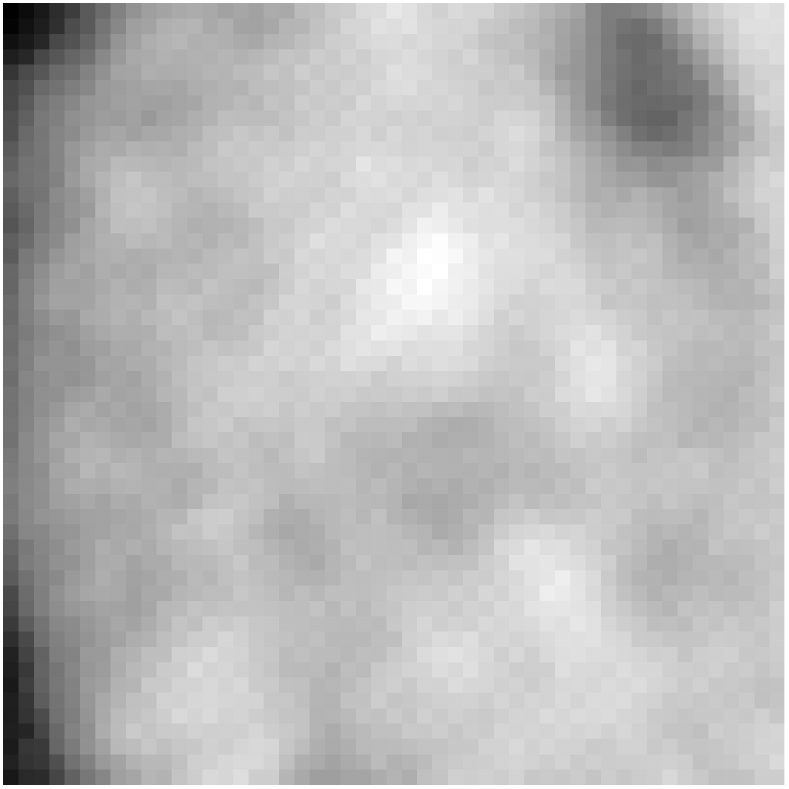}
\includegraphics[width=0.1\linewidth, angle=180]{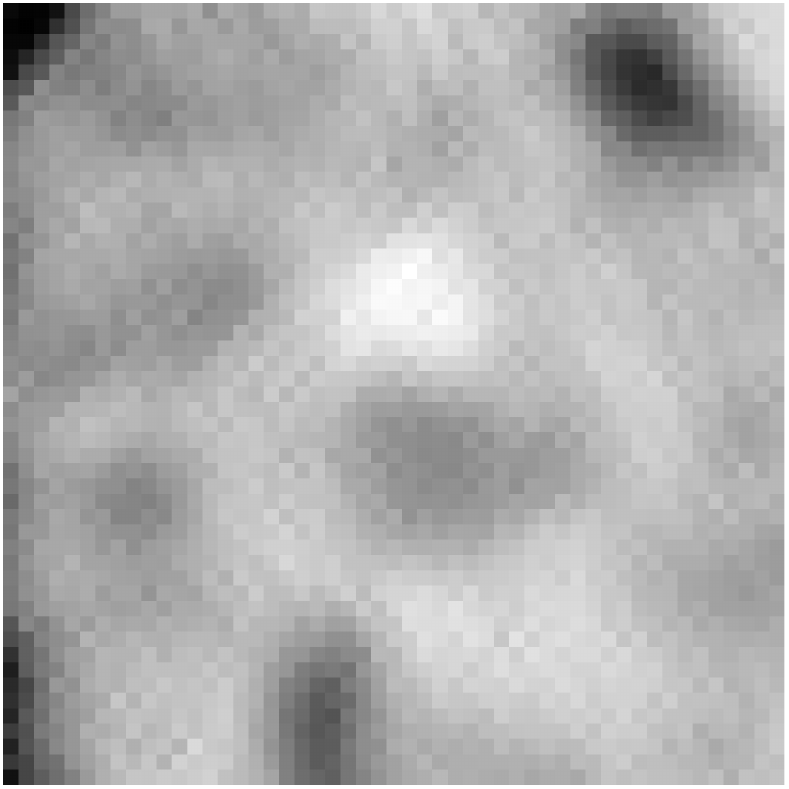}
\includegraphics[width=0.1\linewidth, angle=180]{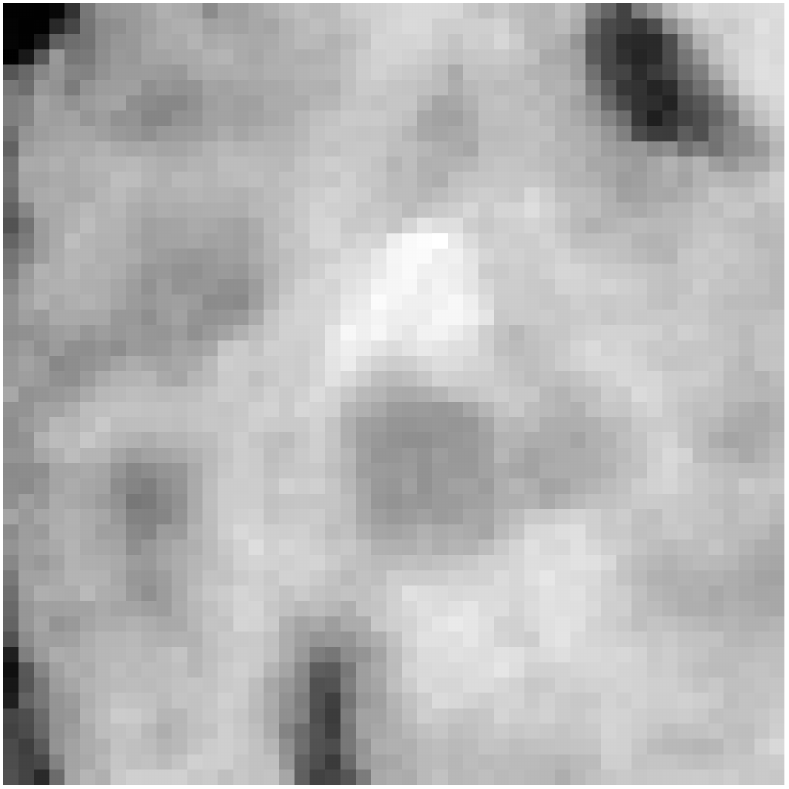}
\includegraphics[width=0.1\linewidth, angle=180]{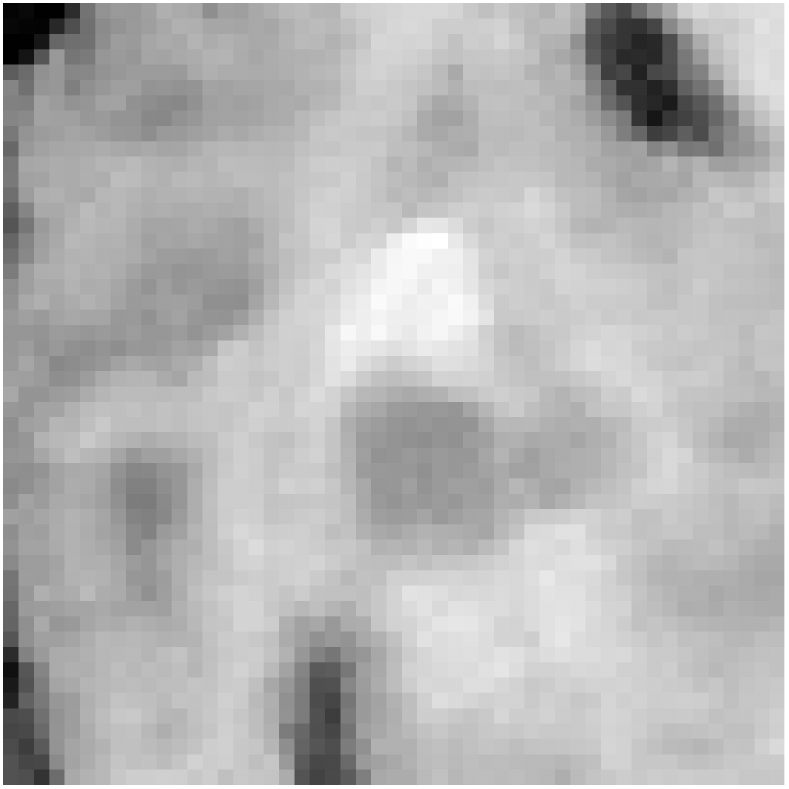}
\includegraphics[width=0.1\linewidth, angle=180]{fig_chp4/white.pdf}\\
\includegraphics[width=0.1\linewidth, angle=270]{fig_chp4/meta_error.pdf}%, height=0.16\linewidth
\includegraphics[width=0.1\linewidth, angle=180]{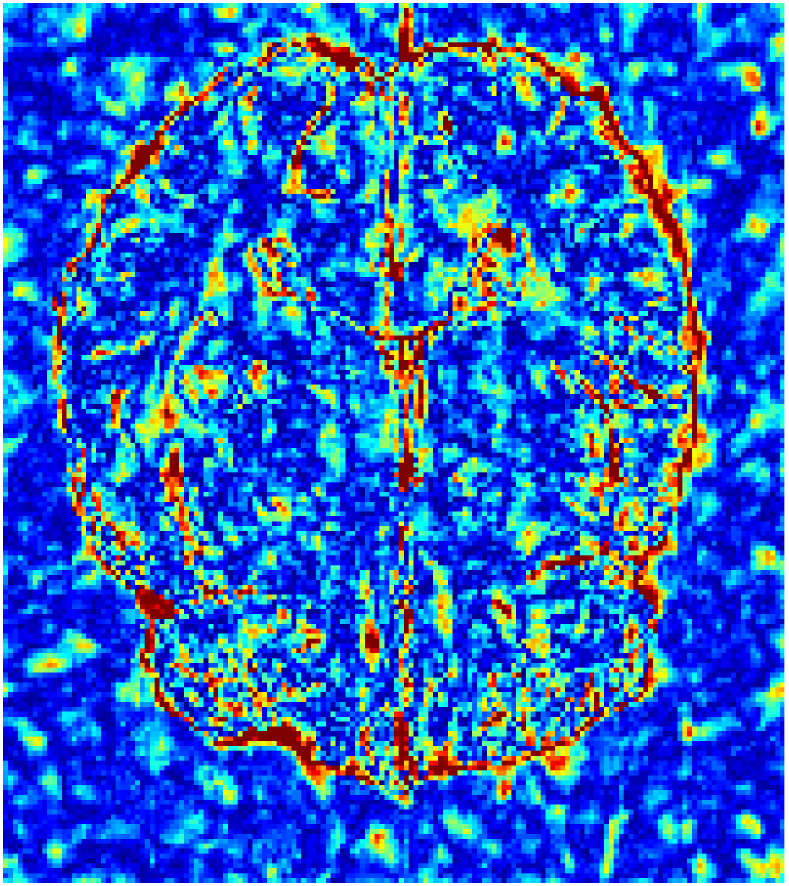}
\includegraphics[width=0.1\linewidth, angle=180]{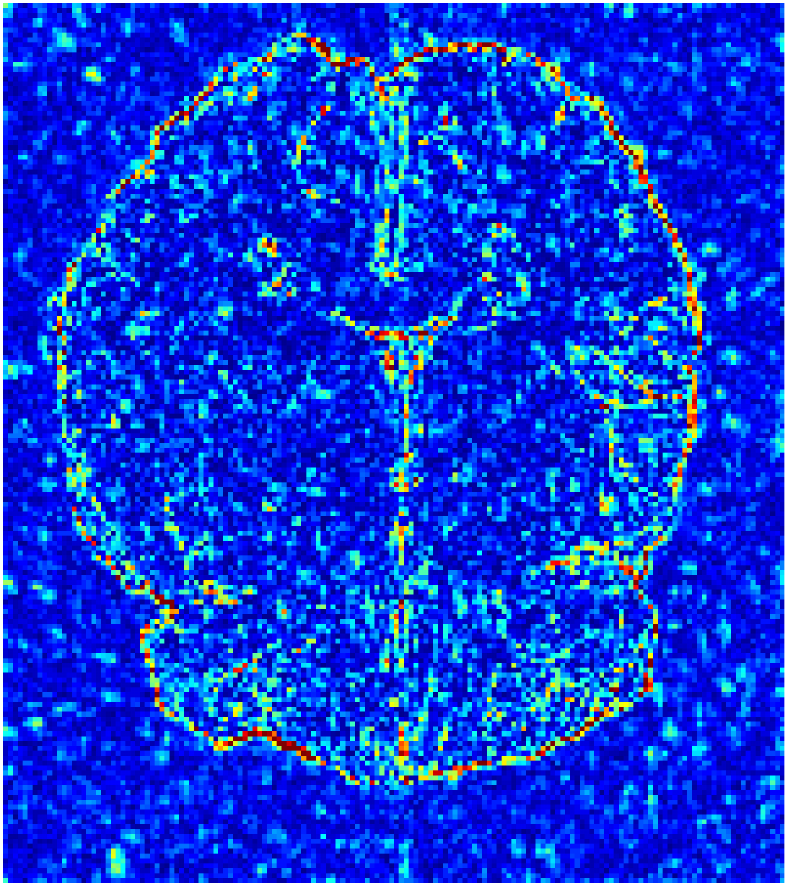}
\includegraphics[width=0.1\linewidth, angle=180]{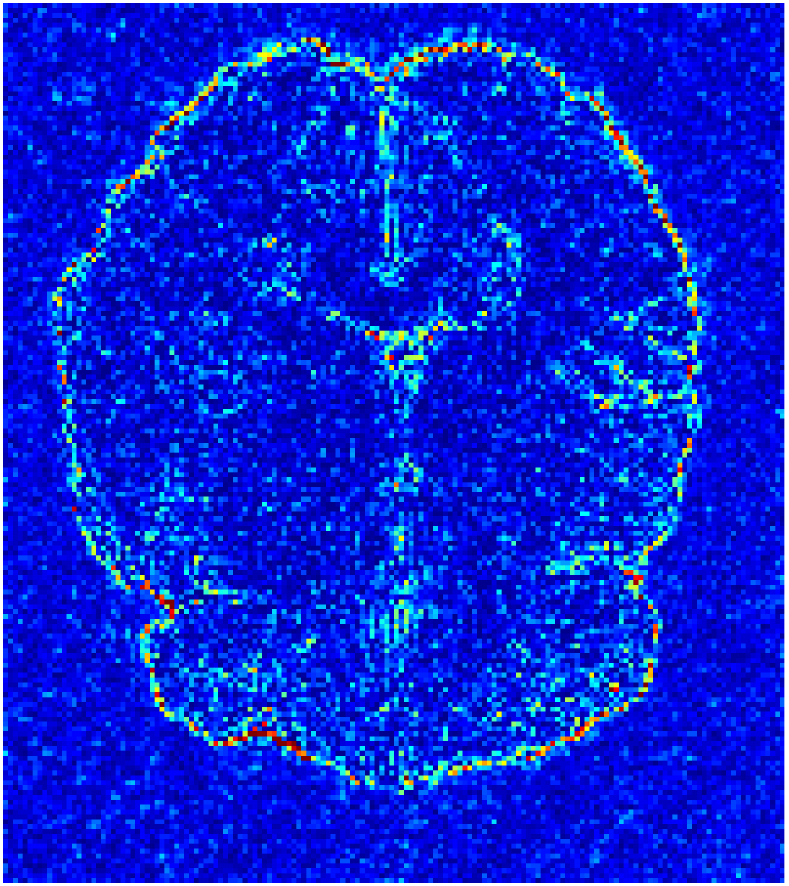}
\includegraphics[width=0.1\linewidth, angle=180]{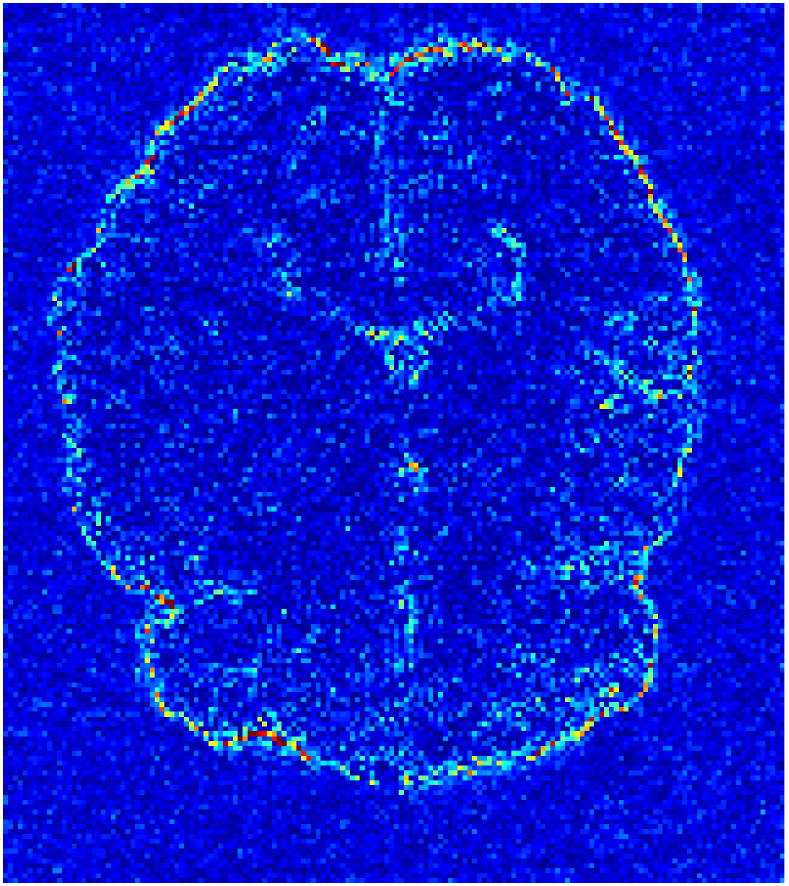}
\includegraphics[width=0.1\linewidth, angle=180]{fig_chp4/colorbar.pdf}\\
\includegraphics[width=0.1\linewidth, angle=270]{fig_chp4/conventional_error.pdf}%, height=0.17\linewidth
\includegraphics[width=0.1\linewidth, angle=180]{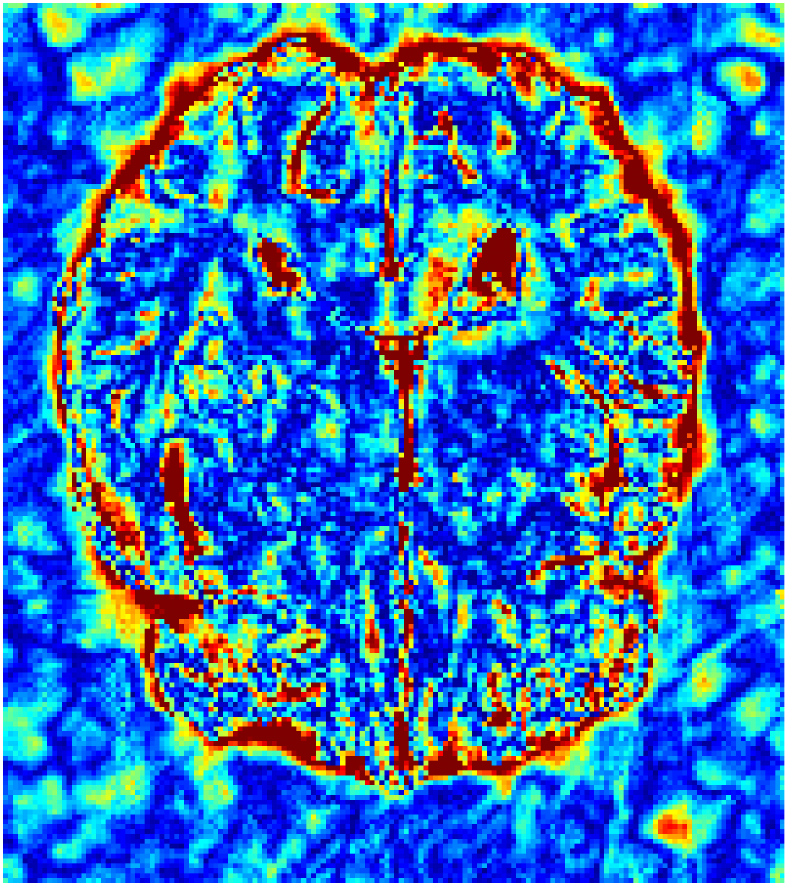}
\includegraphics[width=0.1\linewidth, angle=180]{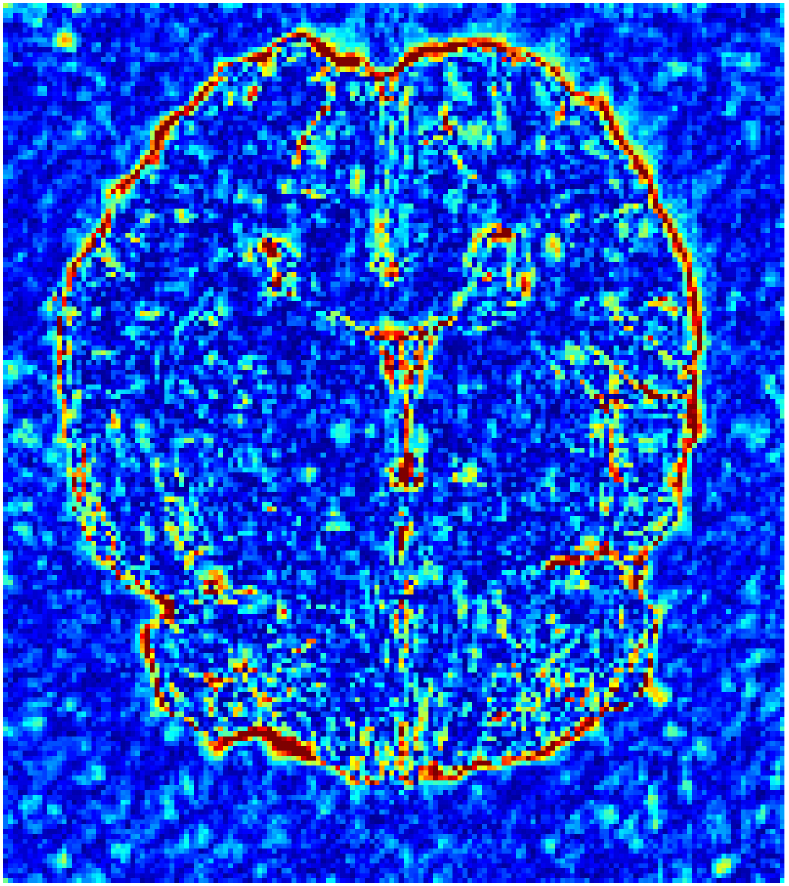}
\includegraphics[width=0.1\linewidth, angle=180]{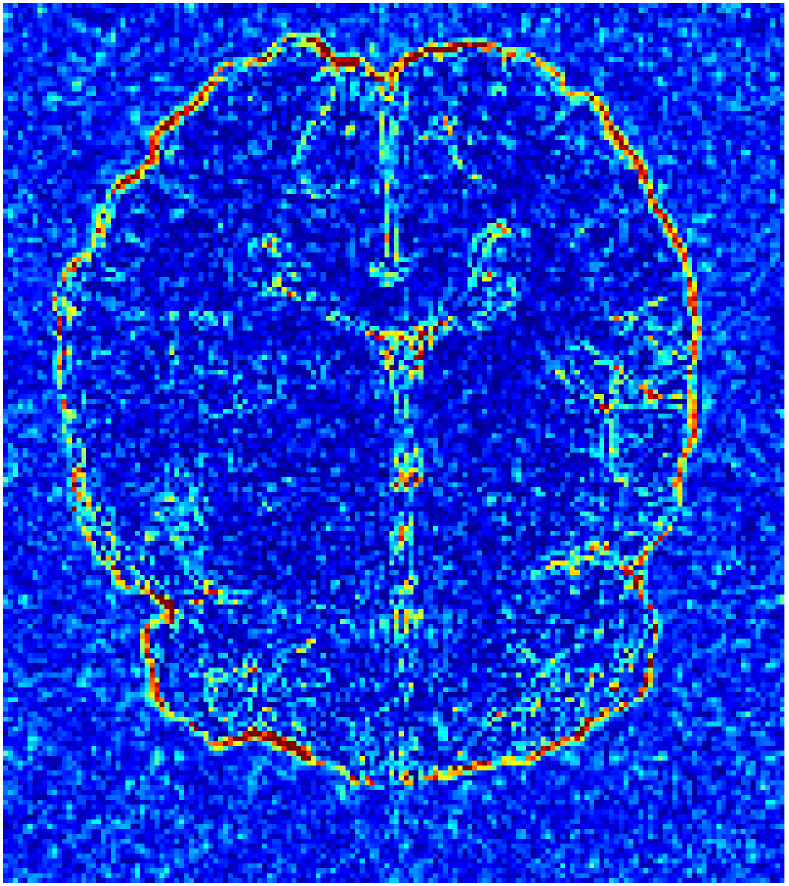}
\includegraphics[width=0.1\linewidth, angle=180]{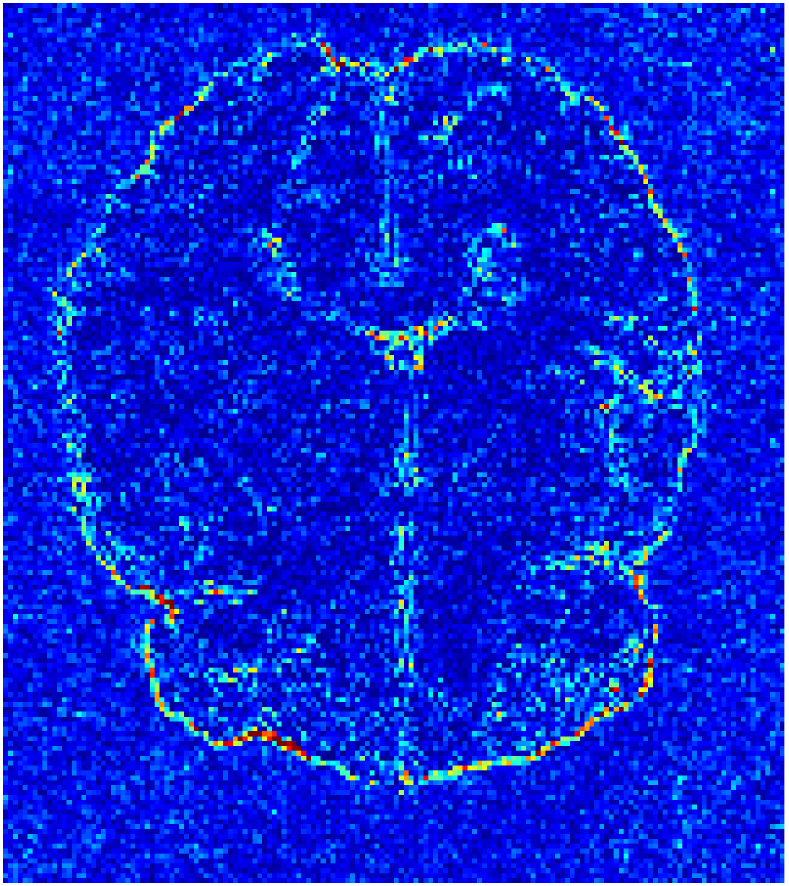}
\includegraphics[width=0.1\linewidth, angle=180]{fig_chp4/white.pdf}\\
\includegraphics[width=0.1\linewidth, angle=270]{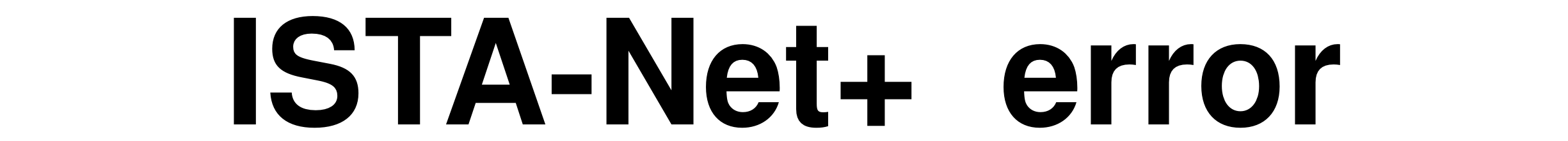}%, height=0.17\linewidth
\includegraphics[width=0.1\linewidth, angle=180]{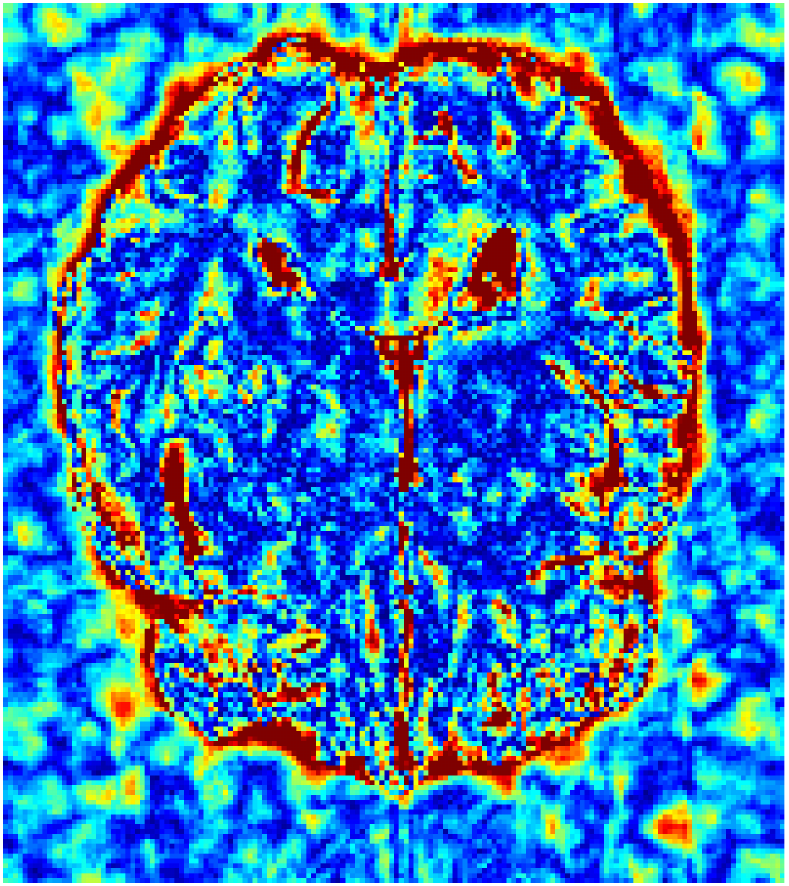}
\includegraphics[width=0.1\linewidth, angle=180]{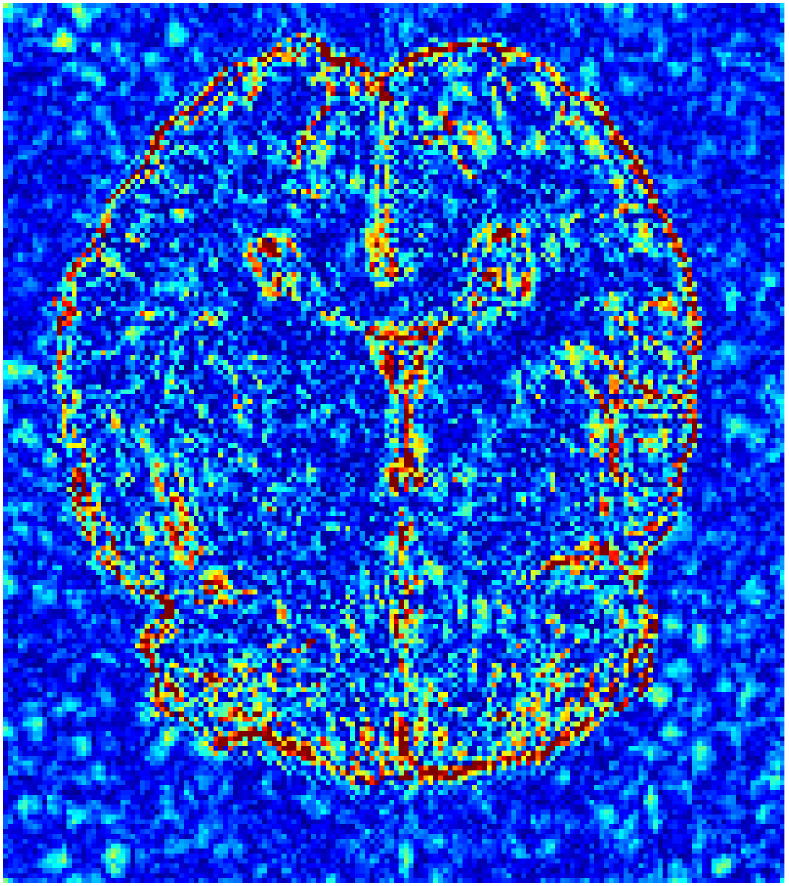}
\includegraphics[width=0.1\linewidth, angle=180]{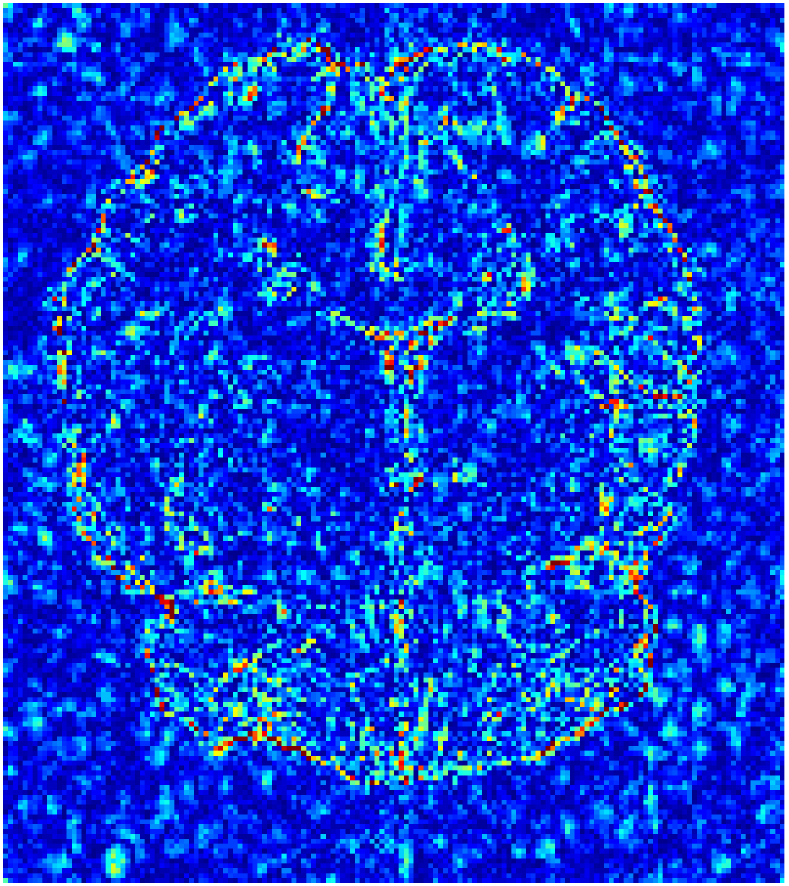}
\includegraphics[width=0.1\linewidth, angle=180]{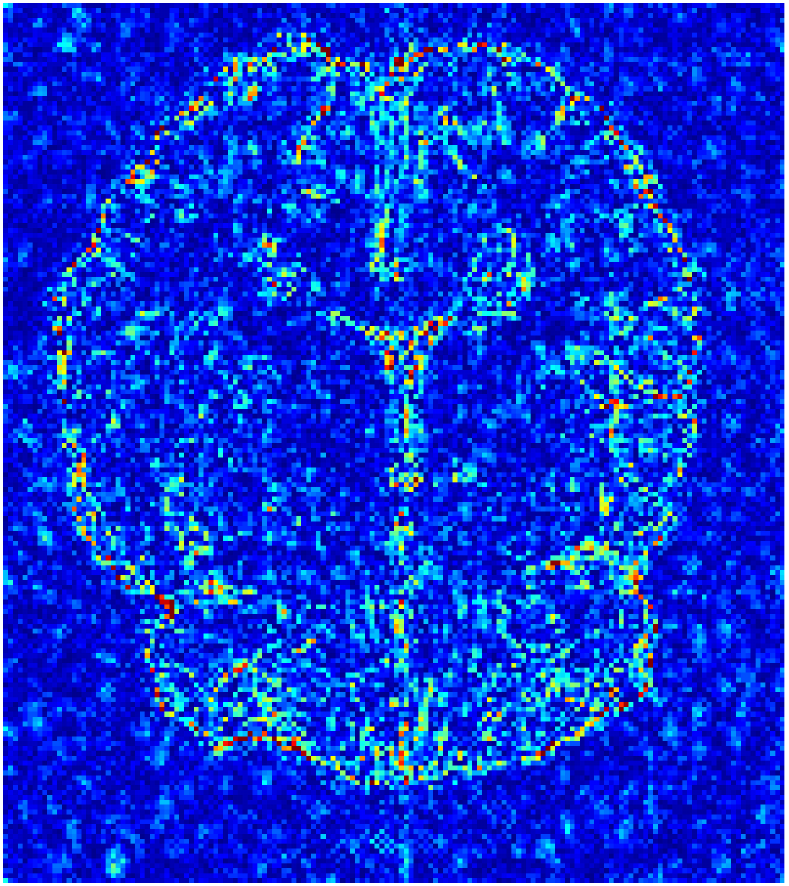}
\includegraphics[width=0.1\linewidth, angle=180]{fig_chp4/white.pdf}\\
\includegraphics[width=0.1\linewidth, angle=90]{fig_chp4/masks.pdf}
\includegraphics[width=0.1\linewidth]{fig_chp4/mask10_t1.pdf}
\includegraphics[width=0.1\linewidth]{fig_chp4/mask20_t1.pdf}
\includegraphics[width=0.1\linewidth]{fig_chp4/mask30_t1.pdf}
\includegraphics[width=0.1\linewidth]{fig_chp4/mask40_t1.pdf}
\includegraphics[width=0.1\linewidth]{fig_chp4/white.pdf}
\caption{From top to bottom: The reconstruction results, zoomed-in details, point-wise errors with a color bar, and associated \textbf{{radial}} masks for meta-learning, conventional learning, and ISTA-Net$^+$. }
\label{figure_same_ratio_t1}
\end{figure}

\begin{figure}[H]
\includegraphics[width=0.13\linewidth, angle=270]{fig_chp4/meta_result.pdf}
\includegraphics[width=0.13\linewidth, angle=180]{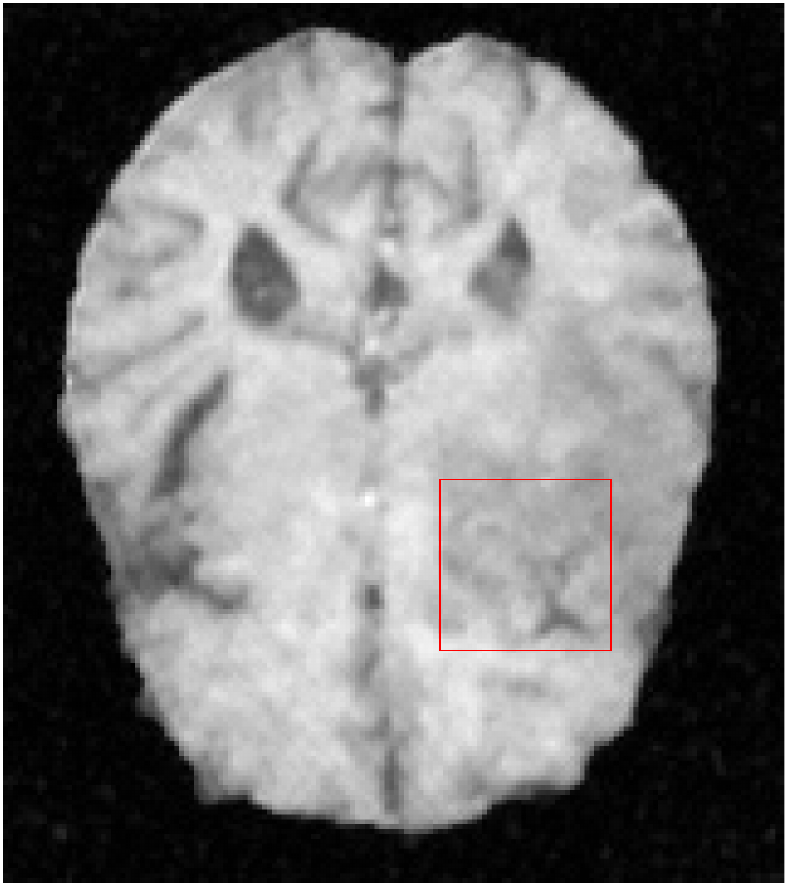}
\includegraphics[width=0.13\linewidth, angle=180]{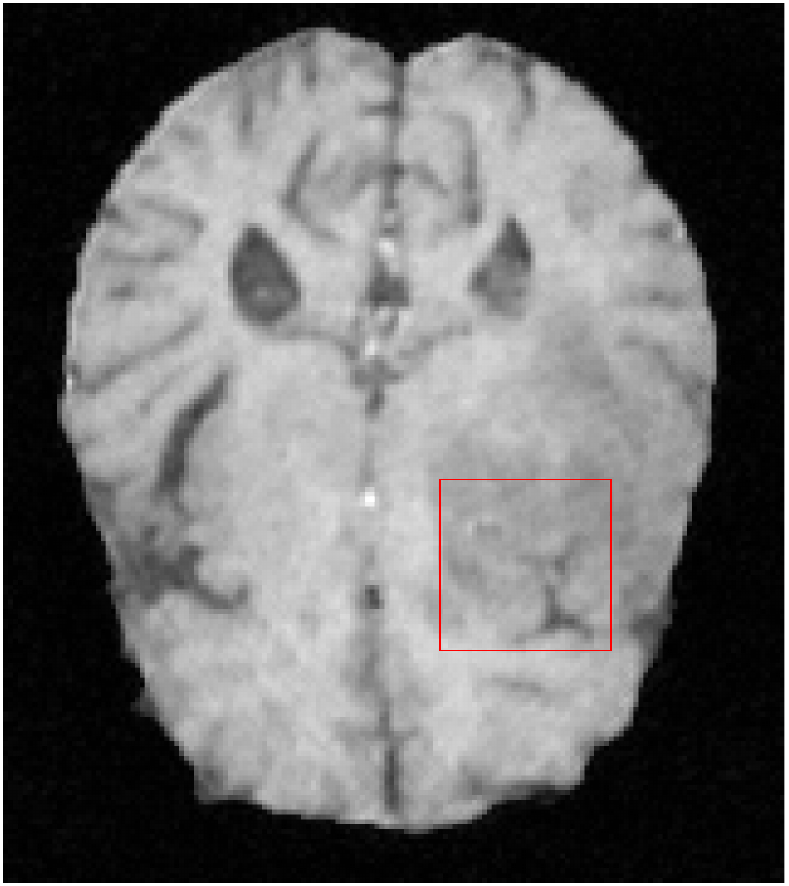}
\includegraphics[width=0.13\linewidth, angle=180]{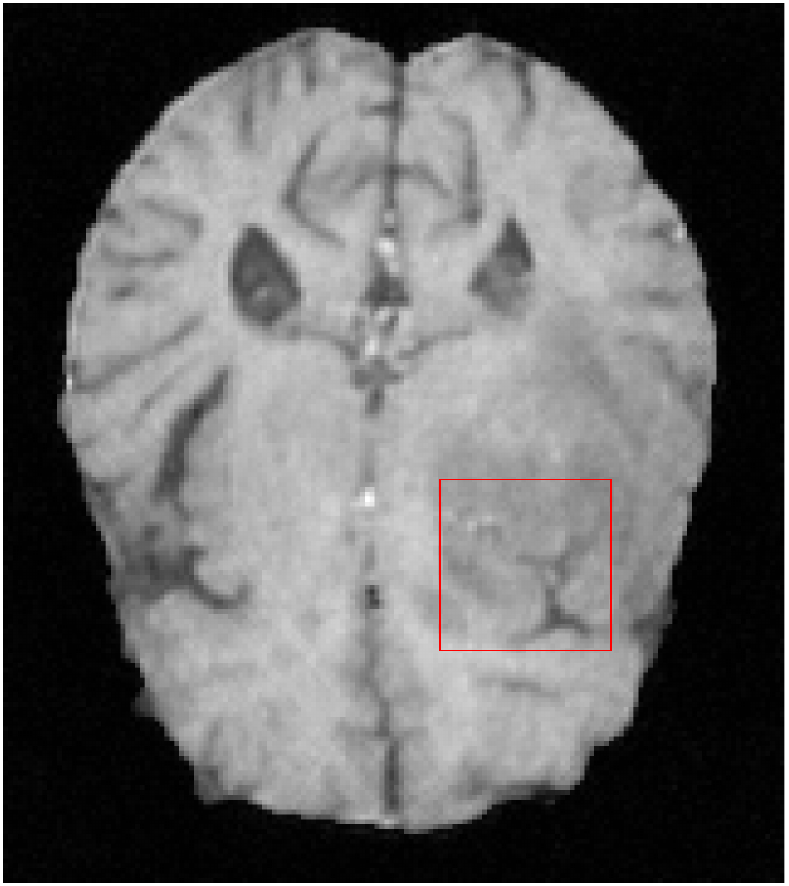}
\includegraphics[width=0.13\linewidth, angle=180]{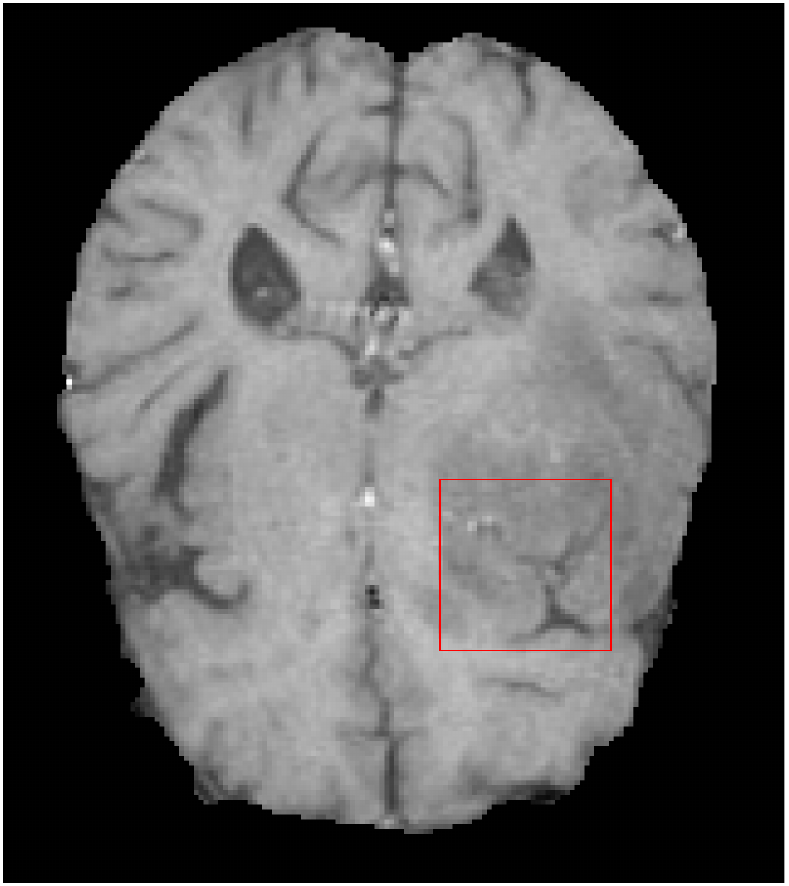}\\
\includegraphics[width=0.13\linewidth, angle=270]{fig_chp4/conventional_result.pdf}
\includegraphics[width=0.13\linewidth, angle=180]{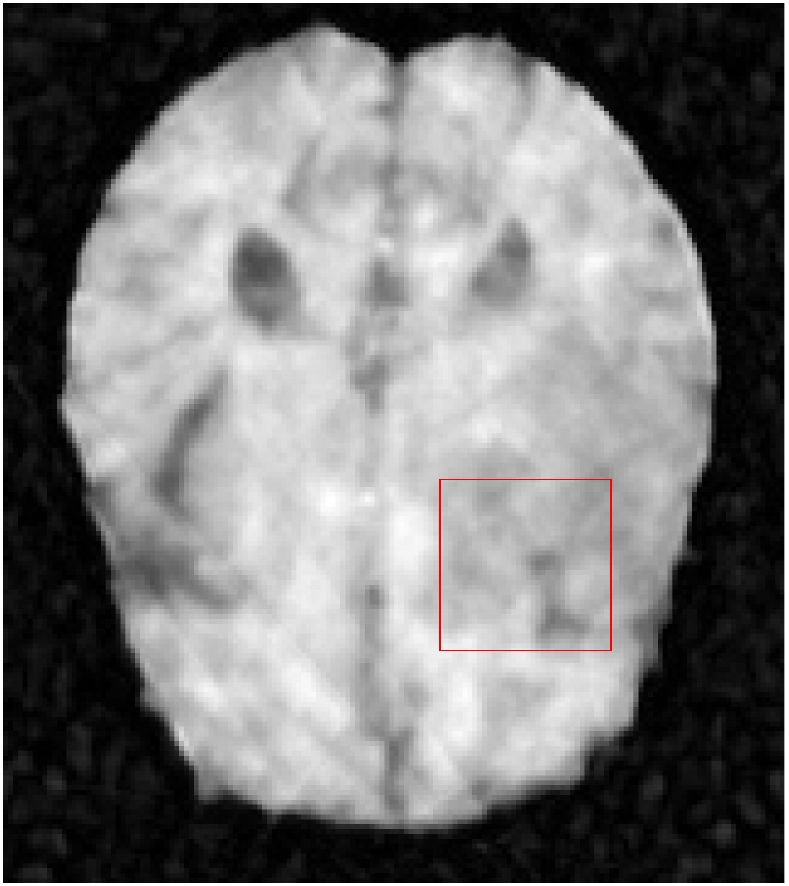}
\includegraphics[width=0.13\linewidth, angle=180]{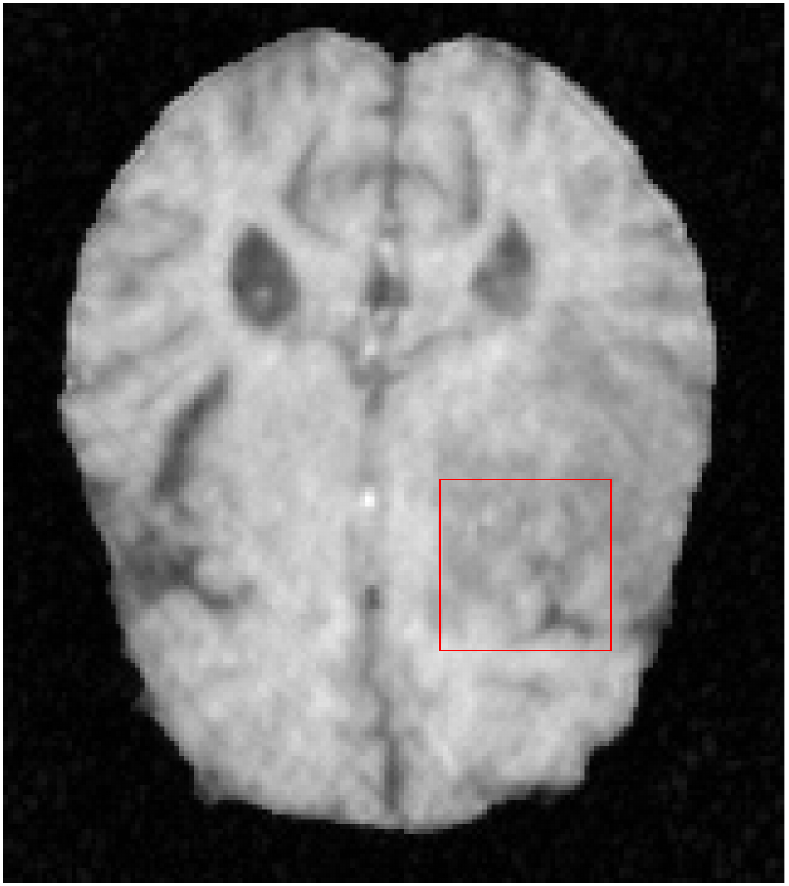}
\includegraphics[width=0.13\linewidth, angle=180]{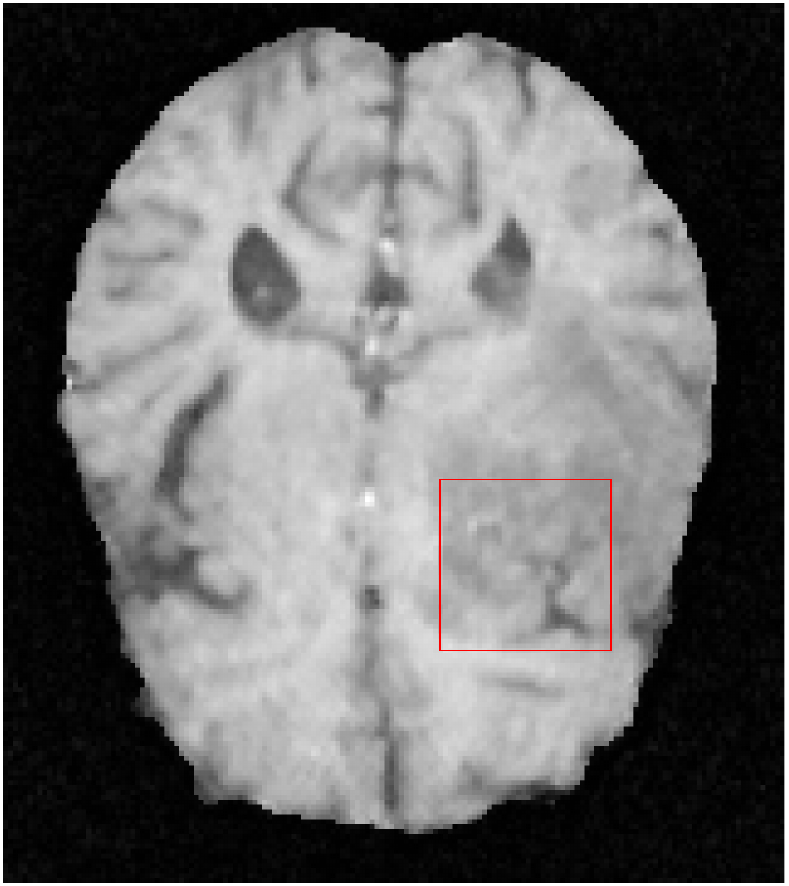}
\includegraphics[width=0.13\linewidth, angle=180]{fig_chp4/white.pdf}\\
\includegraphics[width=0.13\linewidth, angle=270]{fig_chp4/meta_detail.pdf}
\includegraphics[width=0.13\linewidth, angle=180]{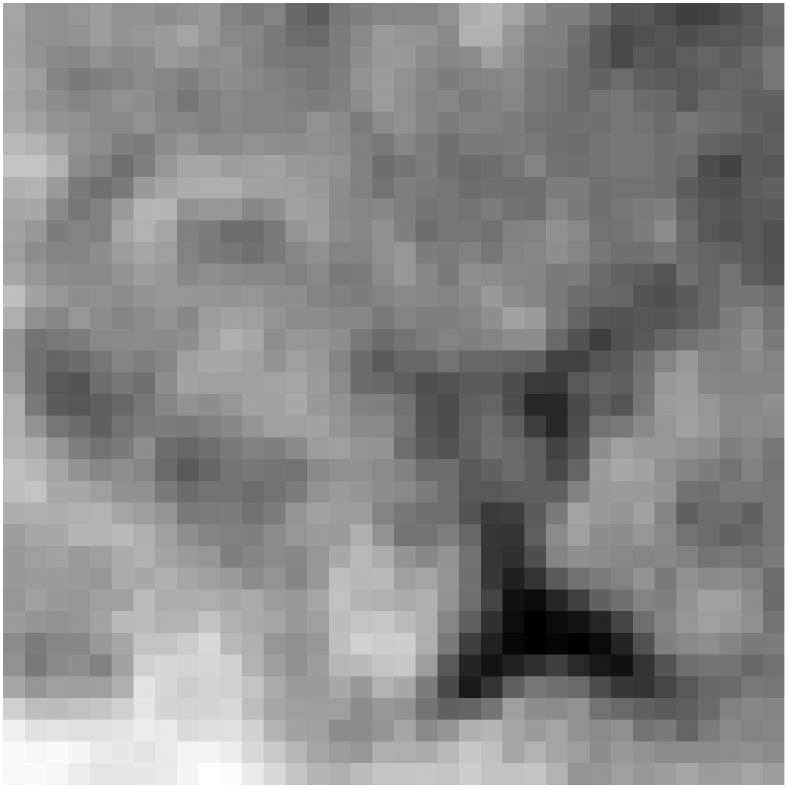}
\includegraphics[width=0.13\linewidth, angle=180]{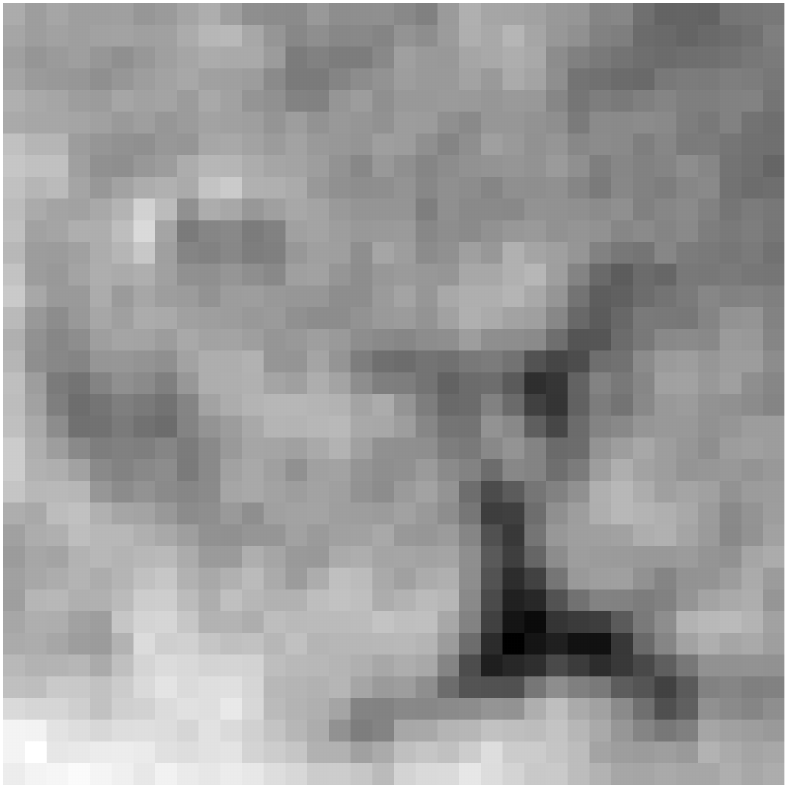}
\includegraphics[width=0.13\linewidth, angle=180]{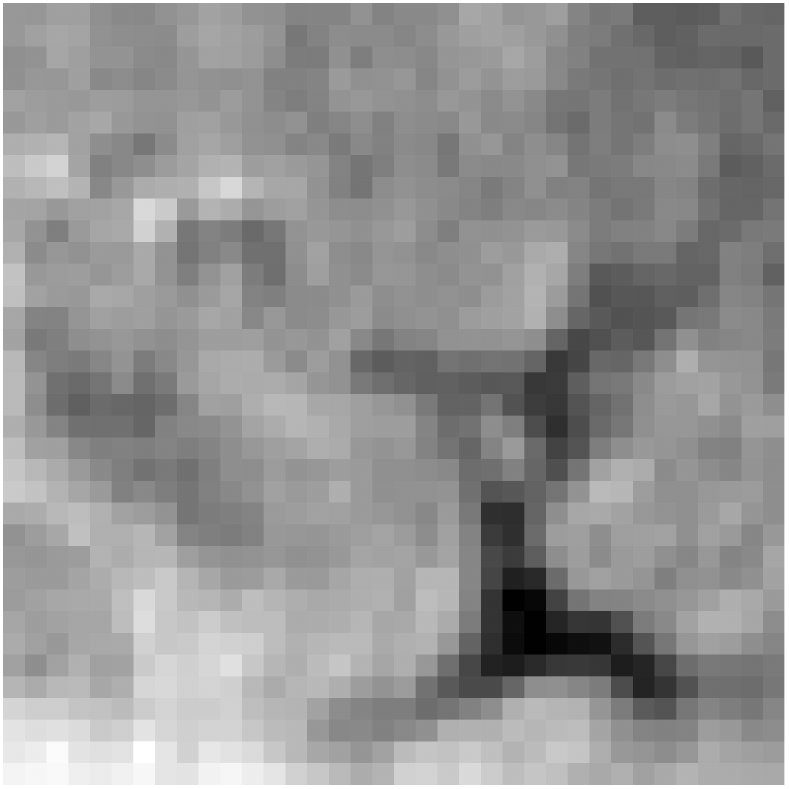}
\includegraphics[width=0.13\linewidth, angle=180]{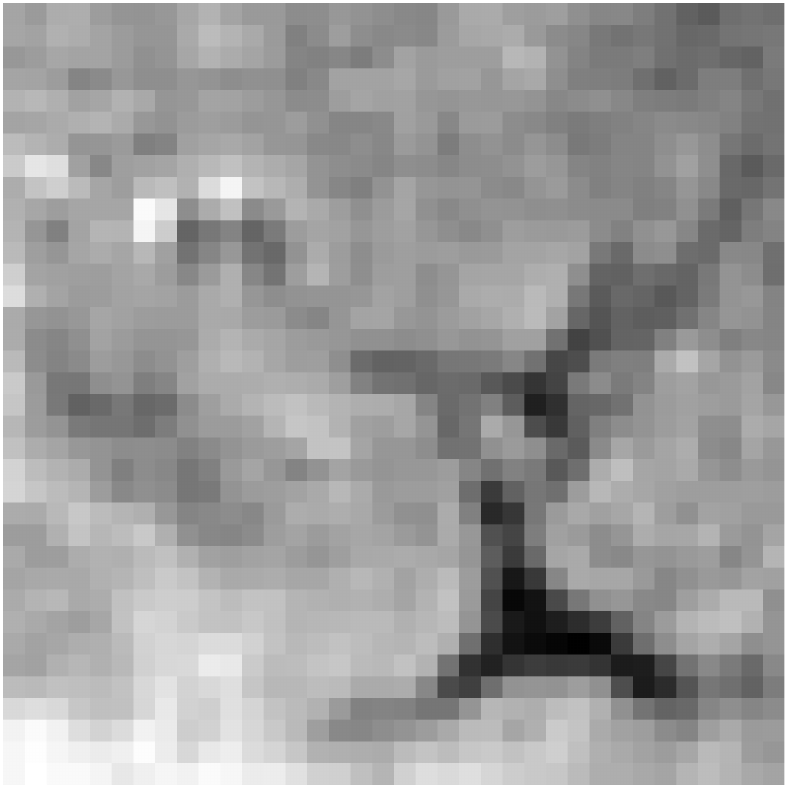}\\
\includegraphics[width=0.13\linewidth, angle=270]{fig_chp4/conventional_detail.pdf}
\includegraphics[width=0.13\linewidth, angle=180]{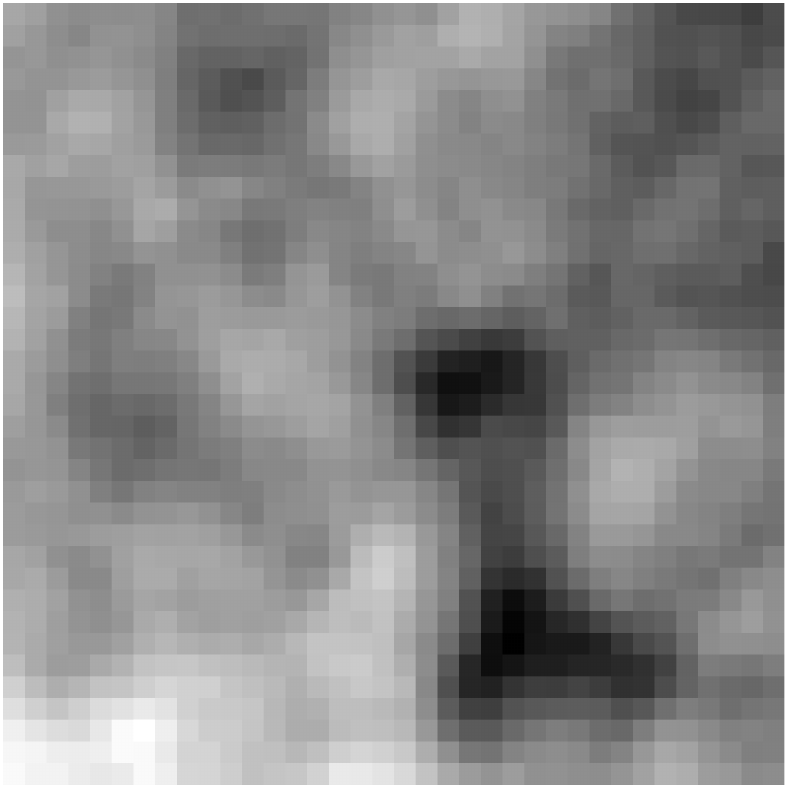}
\includegraphics[width=0.13\linewidth, angle=180]{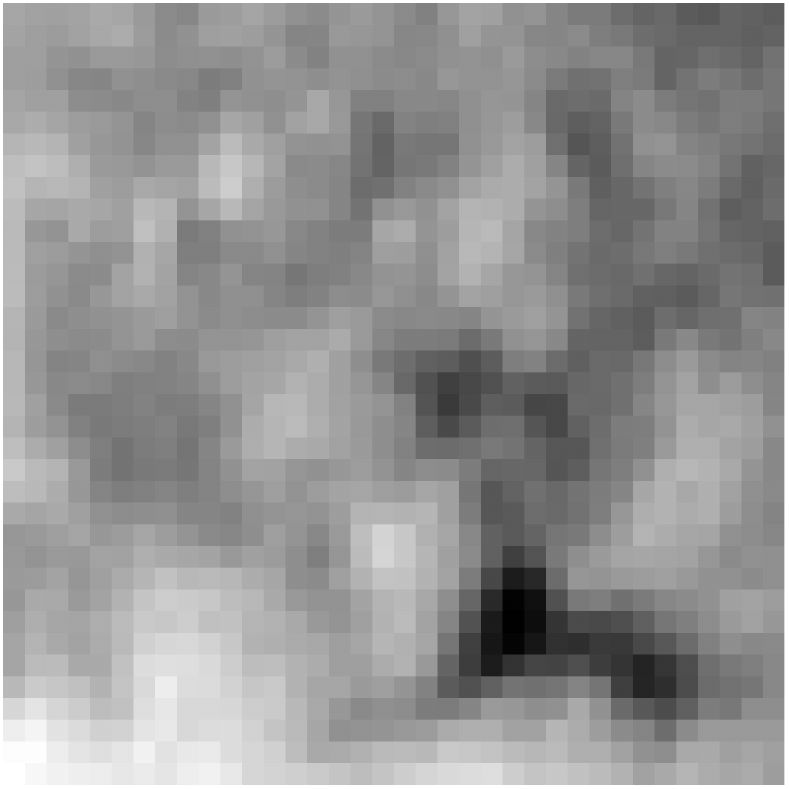}
\includegraphics[width=0.13\linewidth, angle=180]{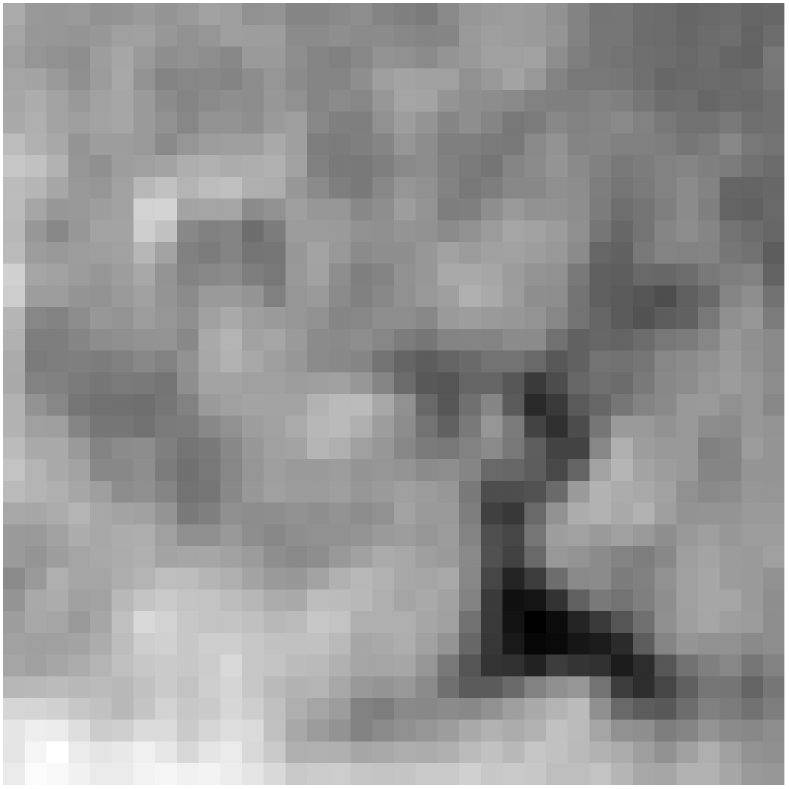}
\includegraphics[width=0.13\linewidth, angle=180]{fig_chp4/white.pdf}\\
\includegraphics[width=0.13\linewidth, angle=270]{fig_chp4/meta_error.pdf}
\includegraphics[width=0.13\linewidth, angle=180]{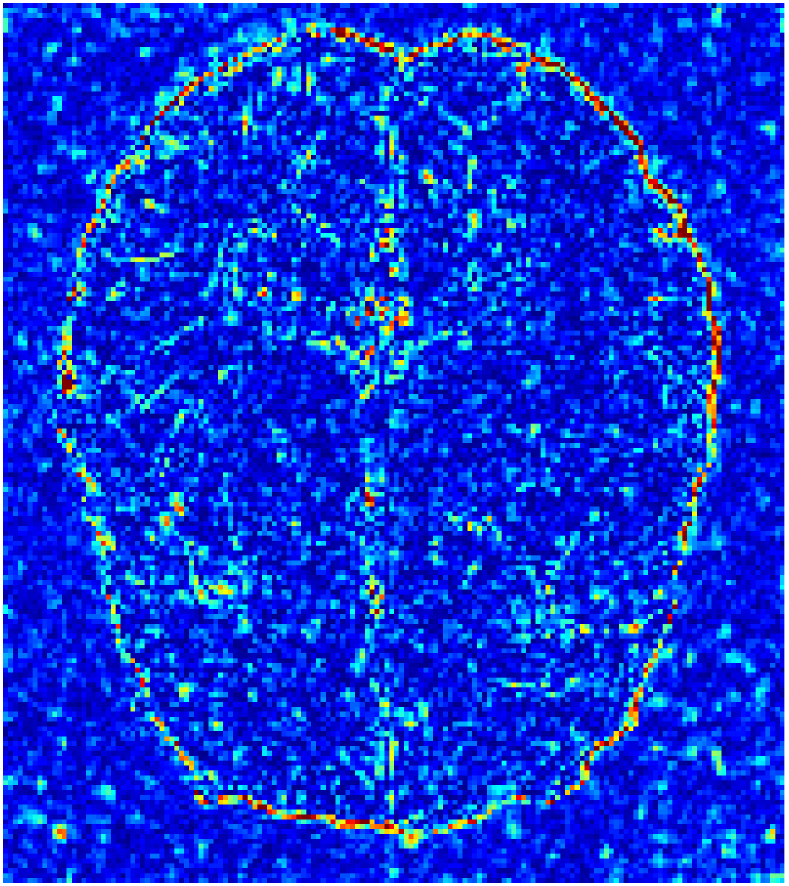}
\includegraphics[width=0.13\linewidth, angle=180]{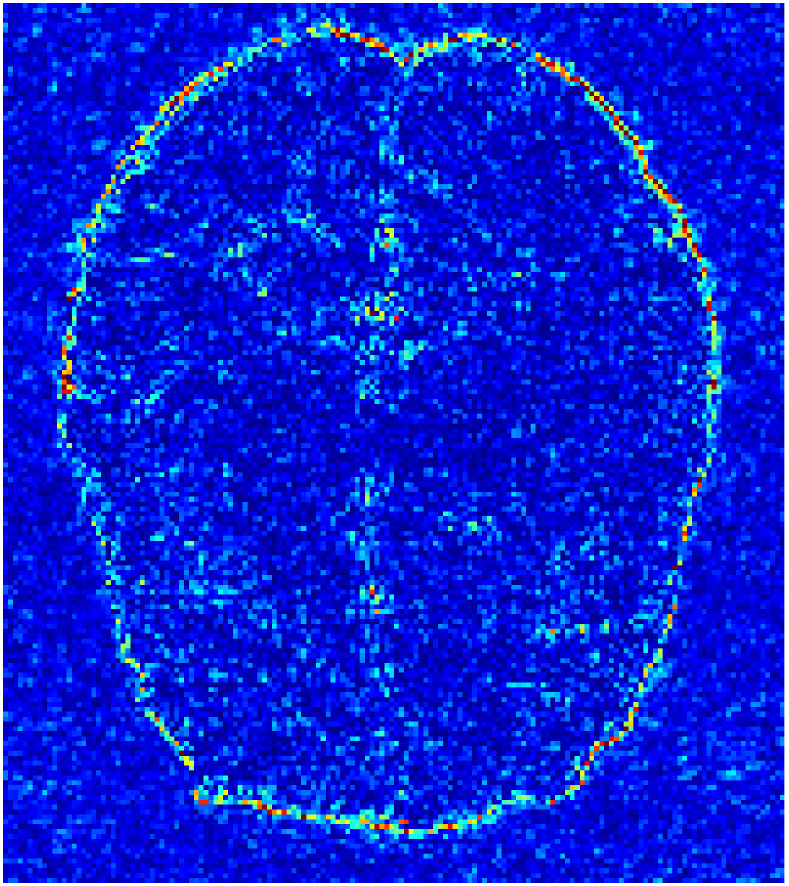}
\includegraphics[width=0.13\linewidth, angle=180]{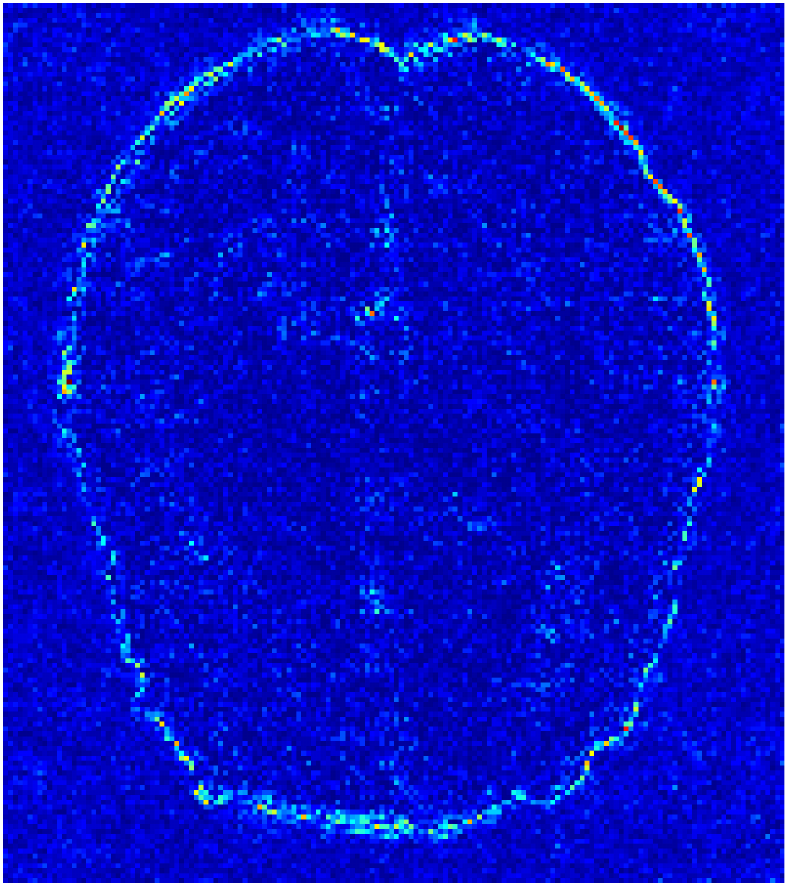}
\includegraphics[width=0.13\linewidth, angle=180]{fig_chp4/colorbar.pdf}\\
\includegraphics[width=0.13\linewidth, angle=270]{fig_chp4/conventional_error.pdf}
\includegraphics[width=0.13\linewidth, angle=180]{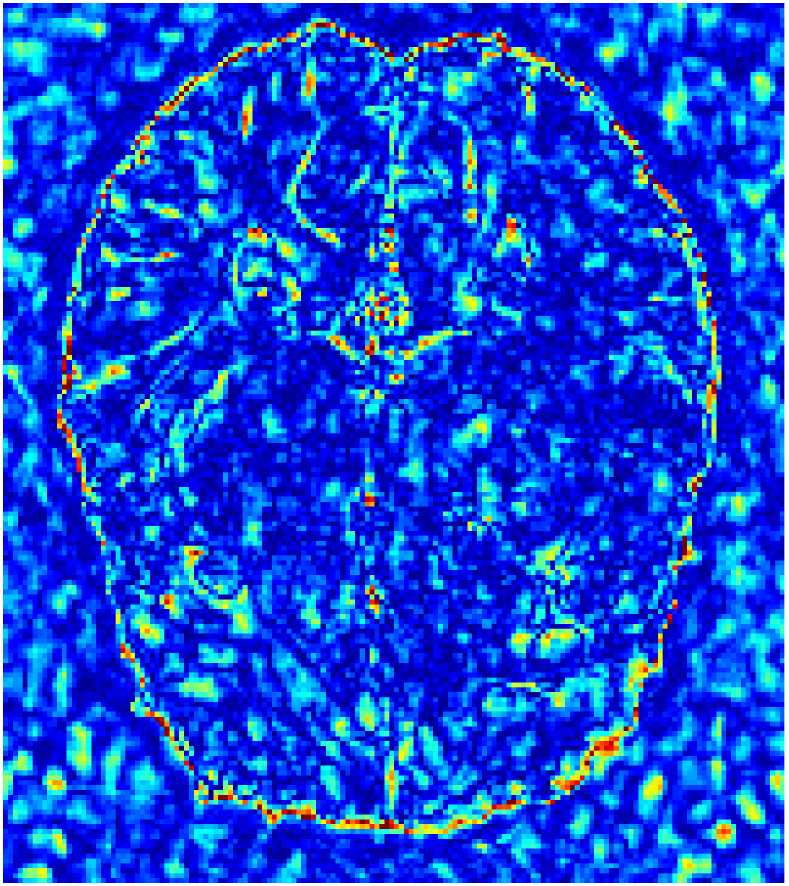}
\includegraphics[width=0.13\linewidth, angle=180]{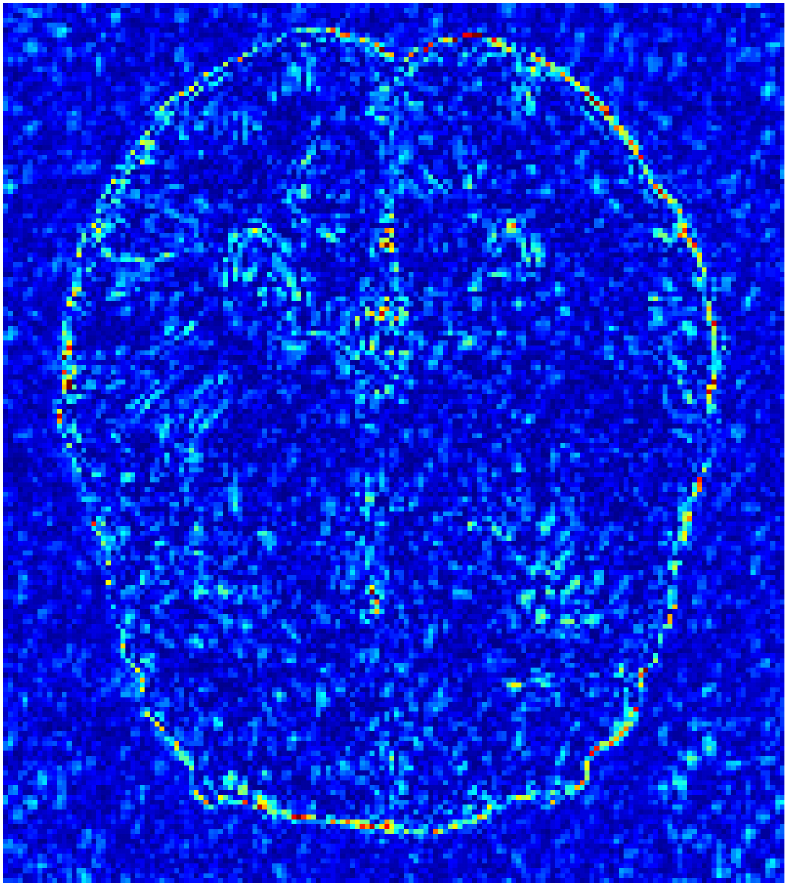}
\includegraphics[width=0.13\linewidth, angle=180]{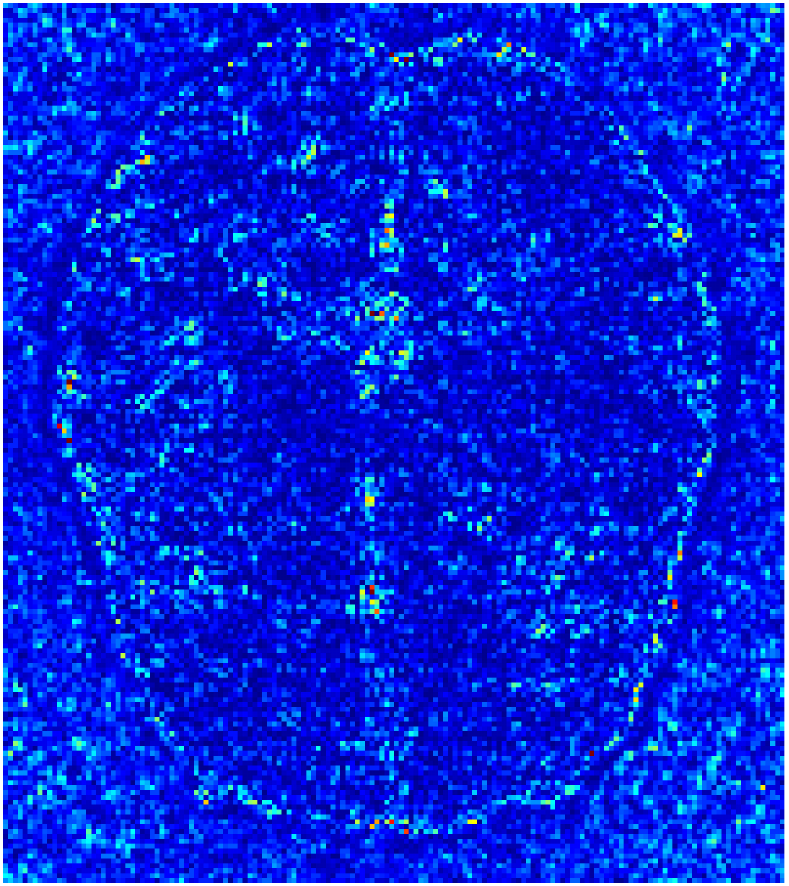}
\includegraphics[width=0.13\linewidth, angle=180]{fig_chp4/white.pdf}\\
\includegraphics[width=0.13\linewidth, angle=90]{fig_chp4/masks.pdf}
\includegraphics[width=0.13\linewidth]{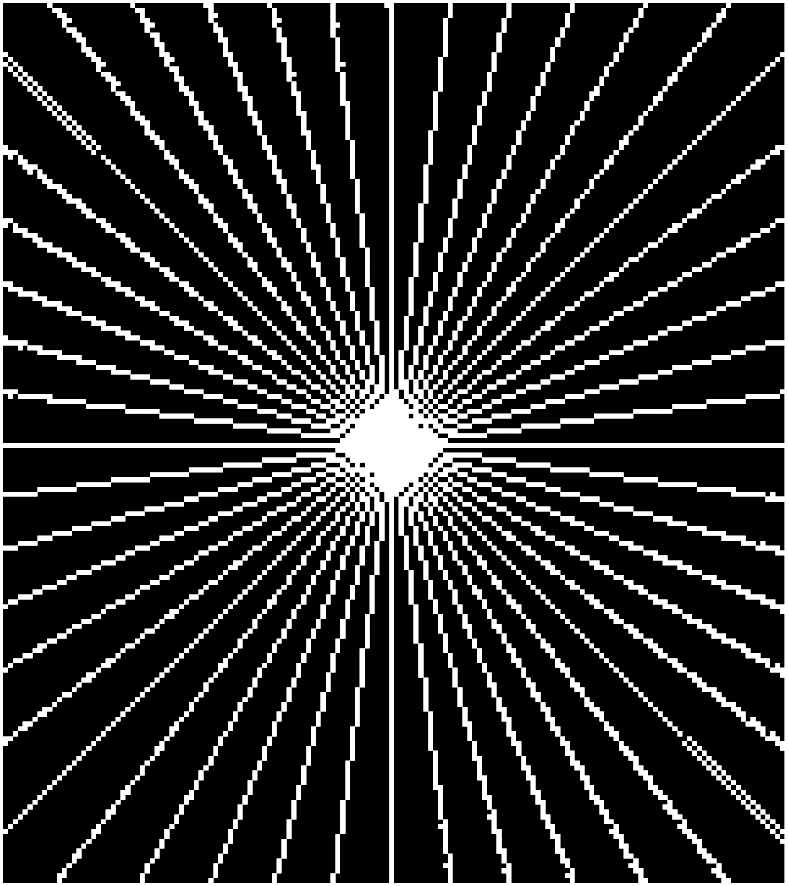}
\includegraphics[width=0.13\linewidth]{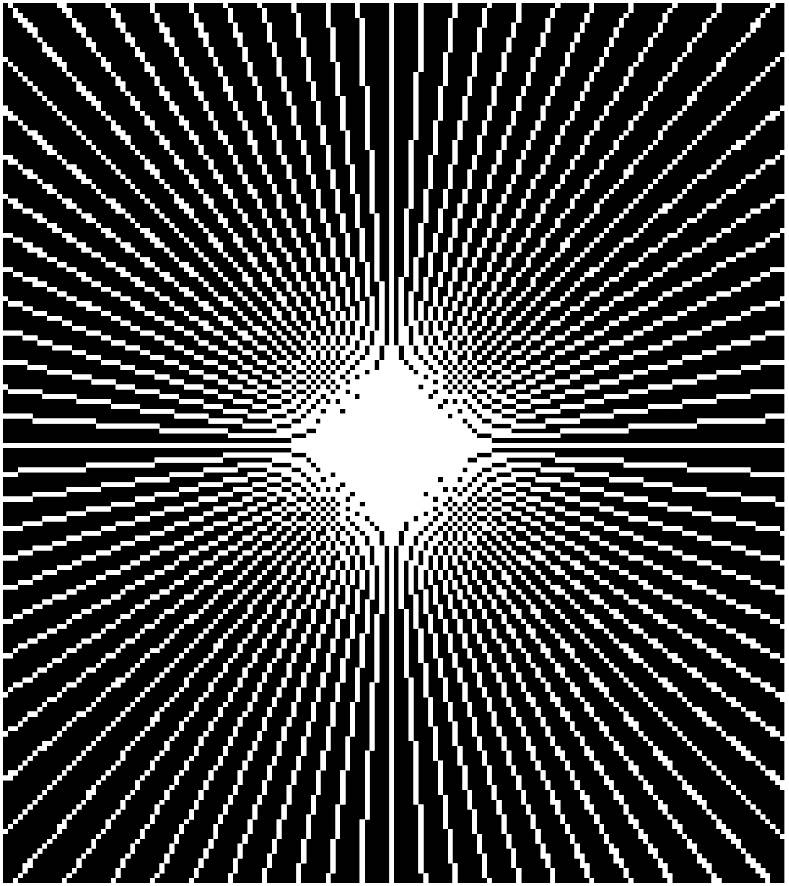}
\includegraphics[width=0.13\linewidth]{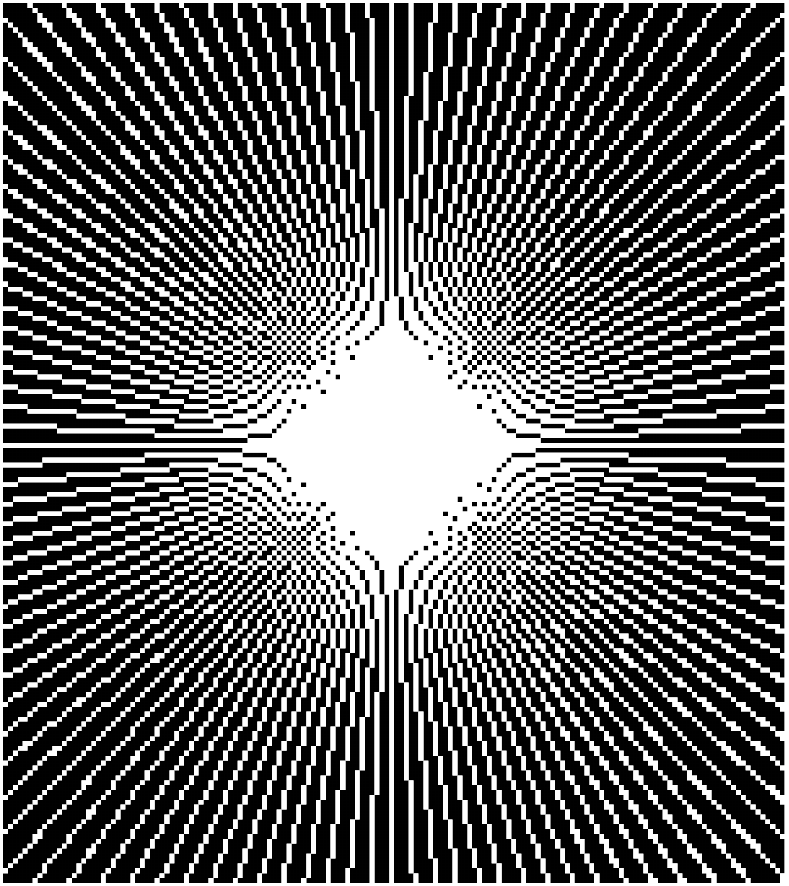}
\includegraphics[width=0.13\linewidth]{fig_chp4/white.pdf}
\caption{The pictures (from top to bottom) display the T1 brain image reconstruction results, zoomed-in details, point-wise errors with a color bar, and associated
 \textbf{{radial}} masks.}
\label{figure_dif_ratio_t1}
\end{figure}

\begin{figure}[H]
\includegraphics[width=0.13\linewidth, angle=270]{fig_chp4/meta_result.pdf}
\includegraphics[width=0.13\linewidth, angle=180]{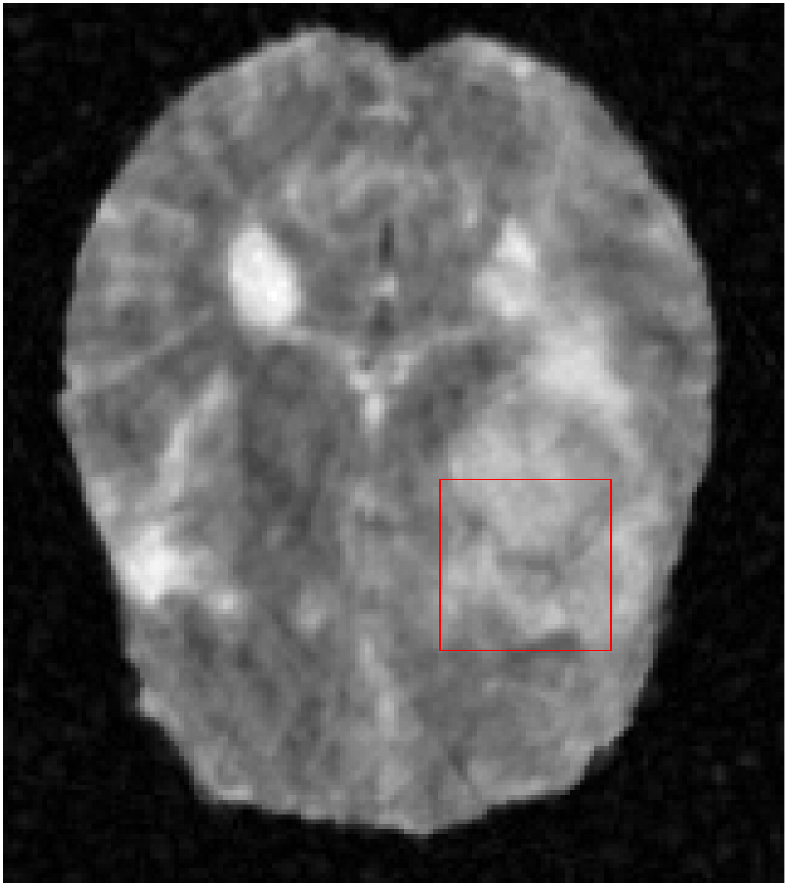}
\includegraphics[width=0.13\linewidth, angle=180]{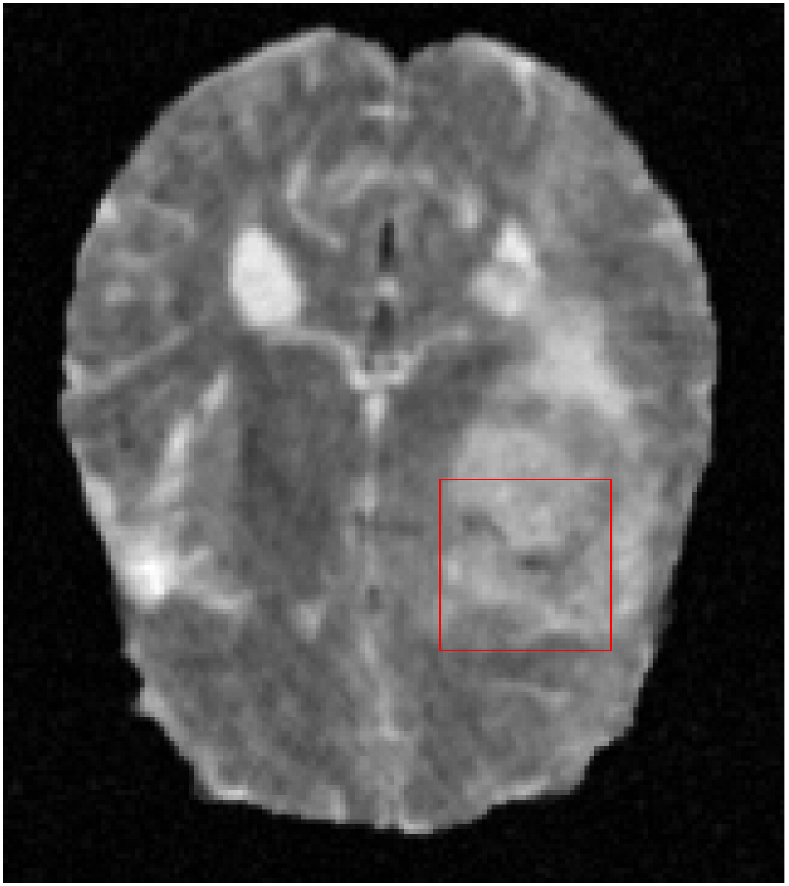}
\includegraphics[width=0.13\linewidth, angle=180]{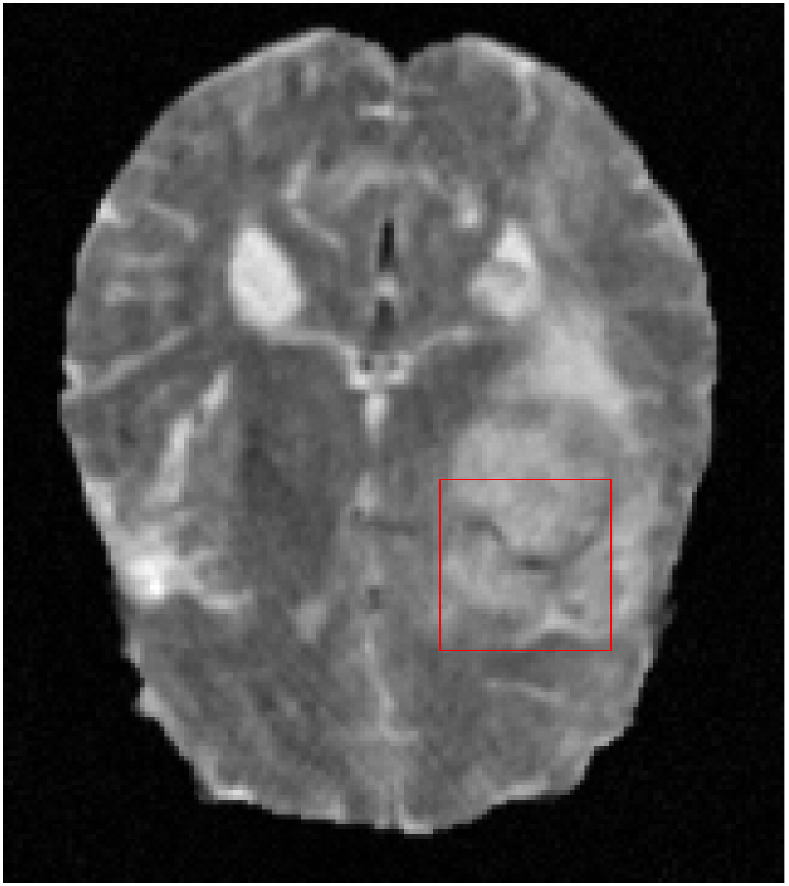}
\includegraphics[width=0.13\linewidth, angle=180]{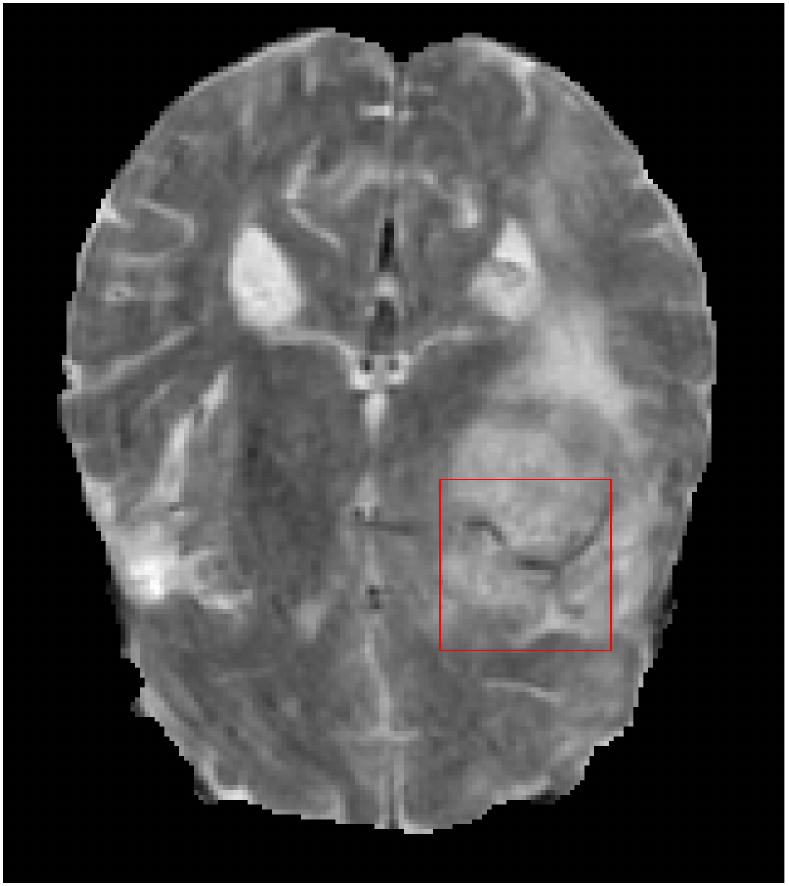}\\
\includegraphics[width=0.13\linewidth, angle=270]{fig_chp4/conventional_result.pdf}
\includegraphics[width=0.13\linewidth, angle=180]{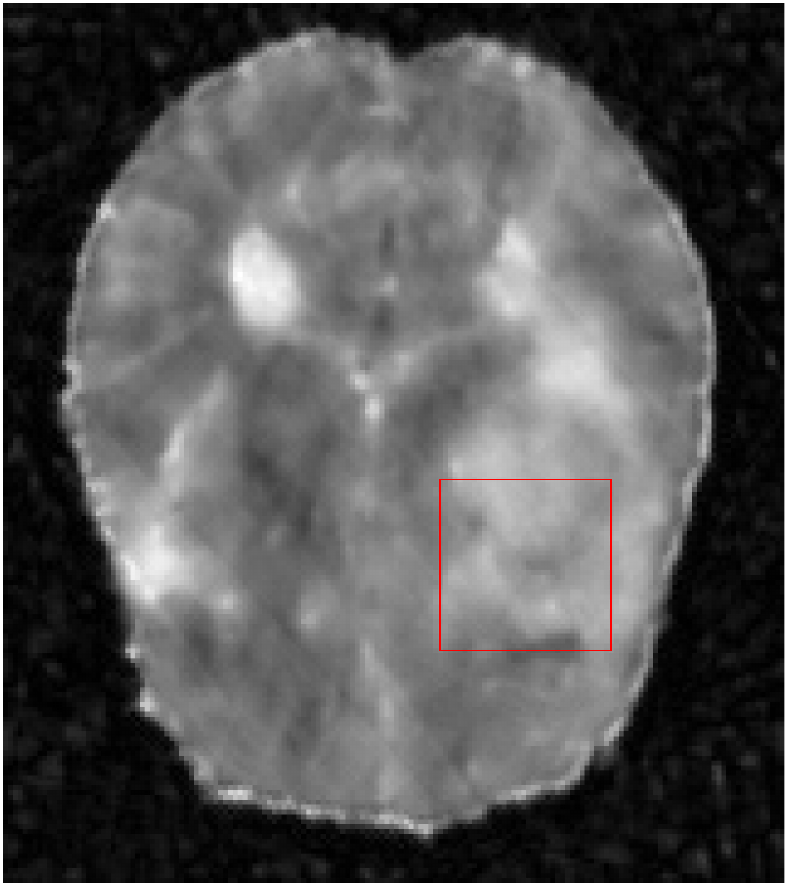}
\includegraphics[width=0.13\linewidth, angle=180]{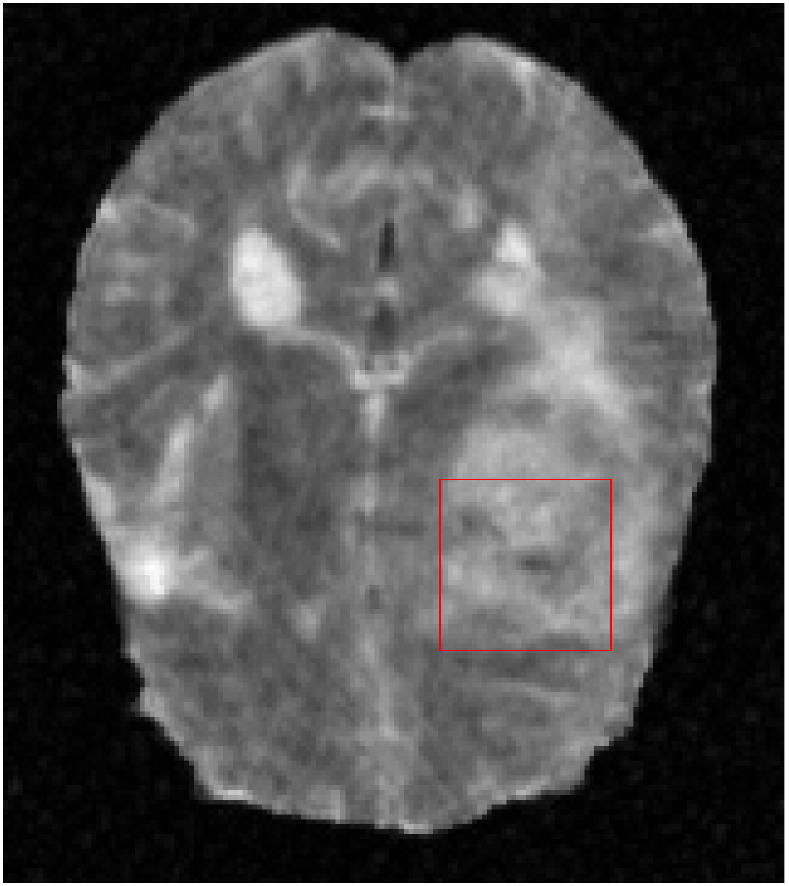}
\includegraphics[width=0.13\linewidth, angle=180]{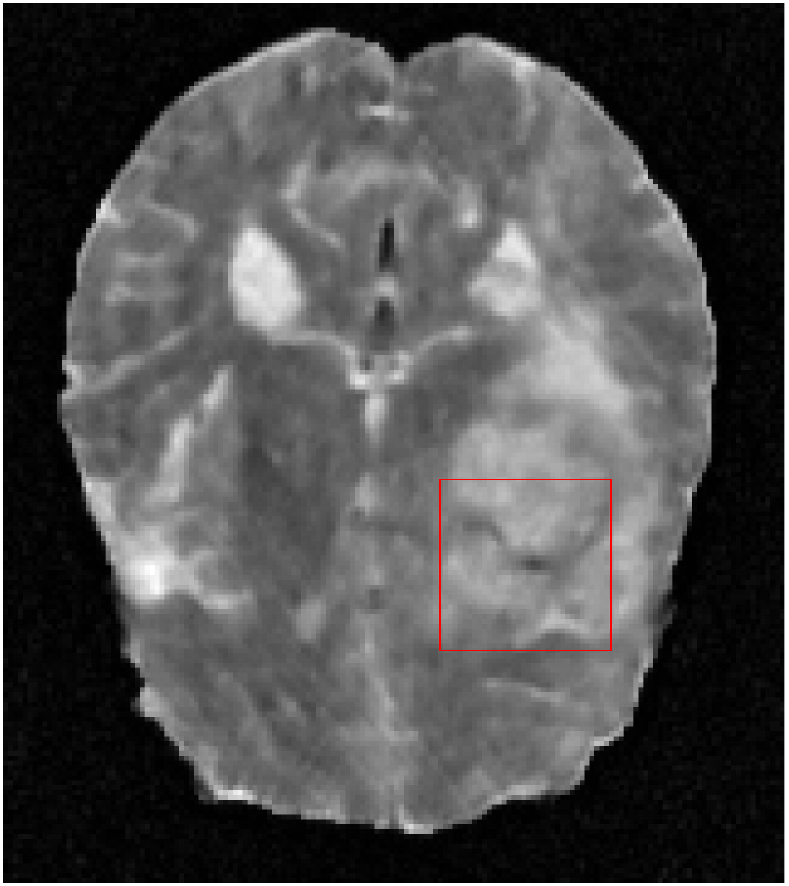}
\includegraphics[width=0.13\linewidth, angle=180]{fig_chp4/white.pdf}\\
\includegraphics[width=0.13\linewidth, angle=270]{fig_chp4/meta_detail.pdf}
\includegraphics[width=0.13\linewidth, angle=180]{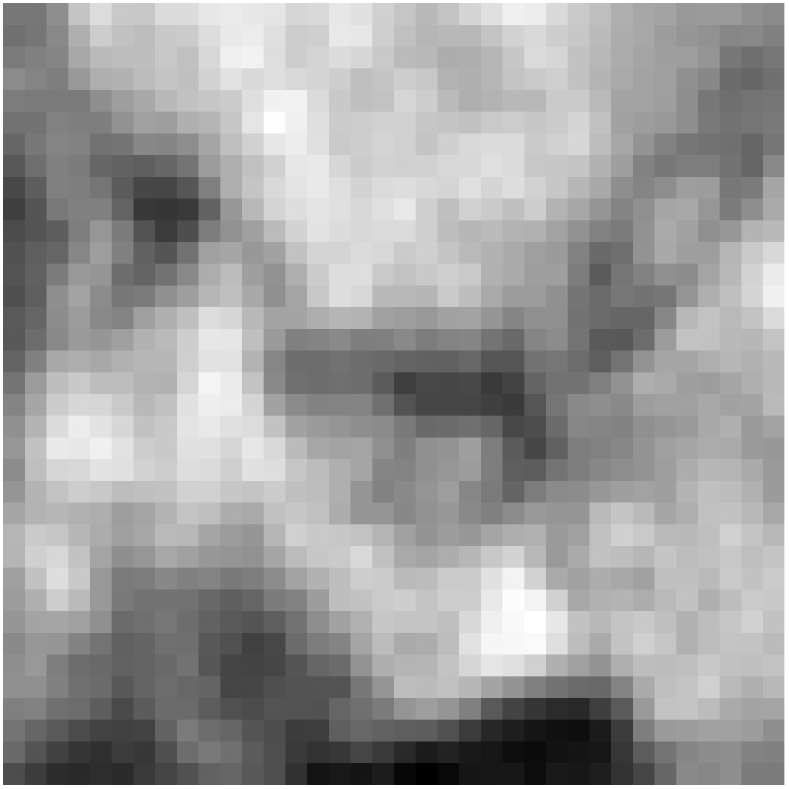}
\includegraphics[width=0.13\linewidth, angle=180]{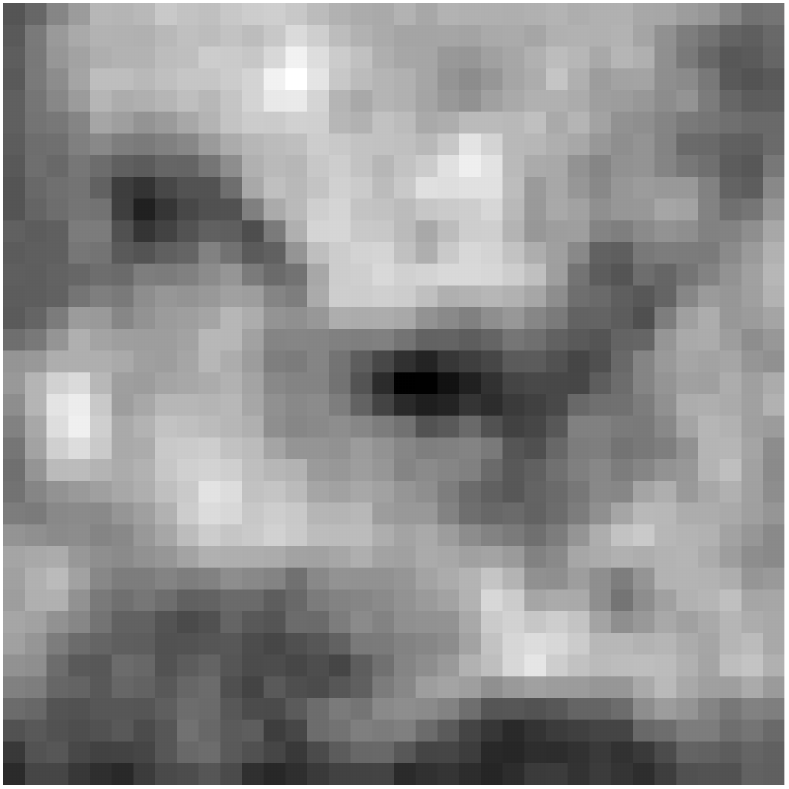}
\includegraphics[width=0.13\linewidth, angle=180]{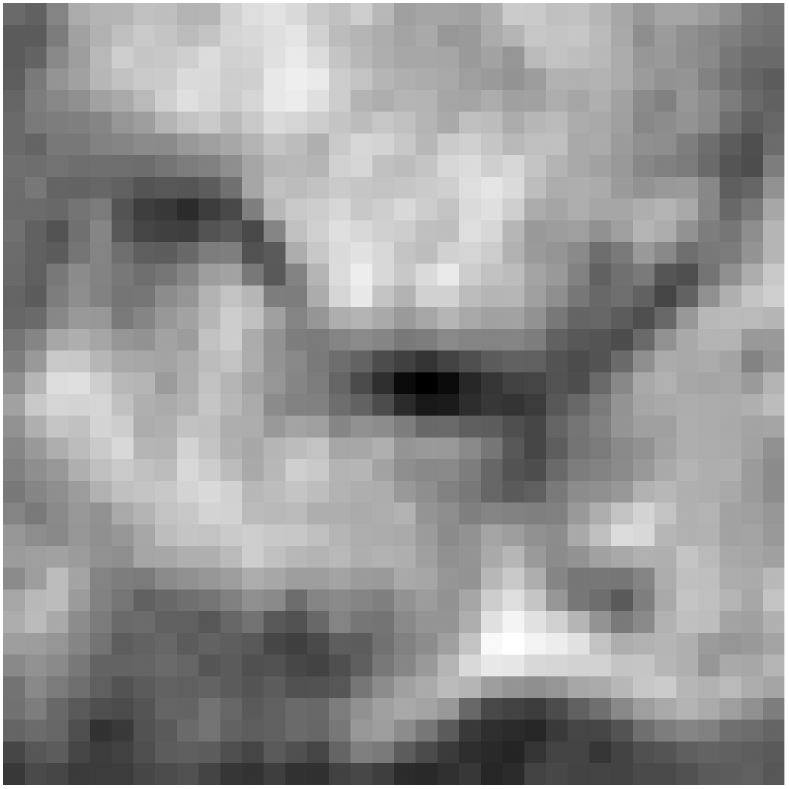}
\includegraphics[width=0.13\linewidth, angle=180]{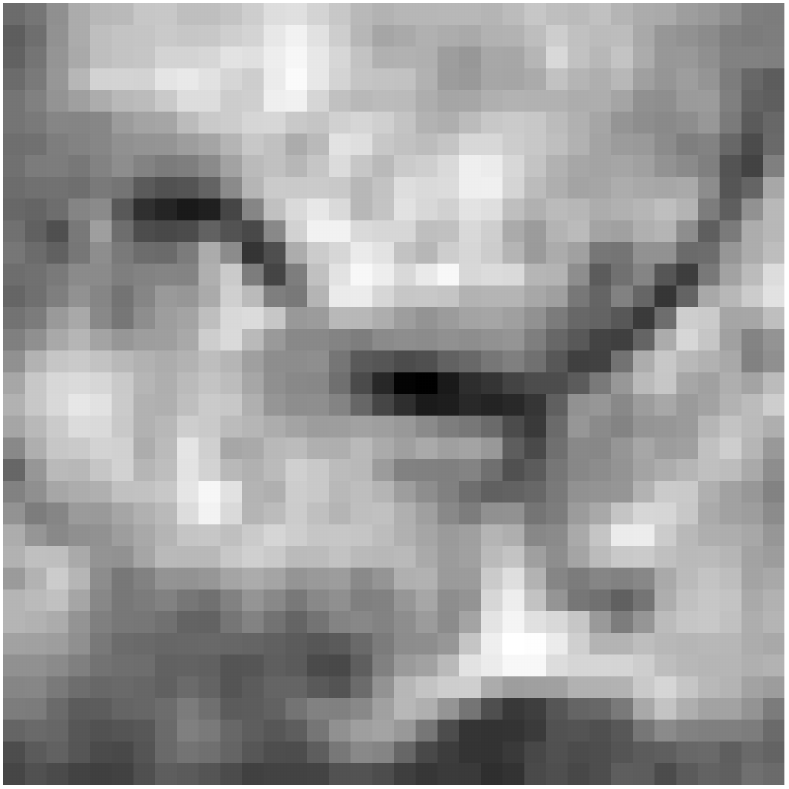}\\
\includegraphics[width=0.13\linewidth, angle=270]{fig_chp4/conventional_detail.pdf}
\includegraphics[width=0.13\linewidth, angle=180]{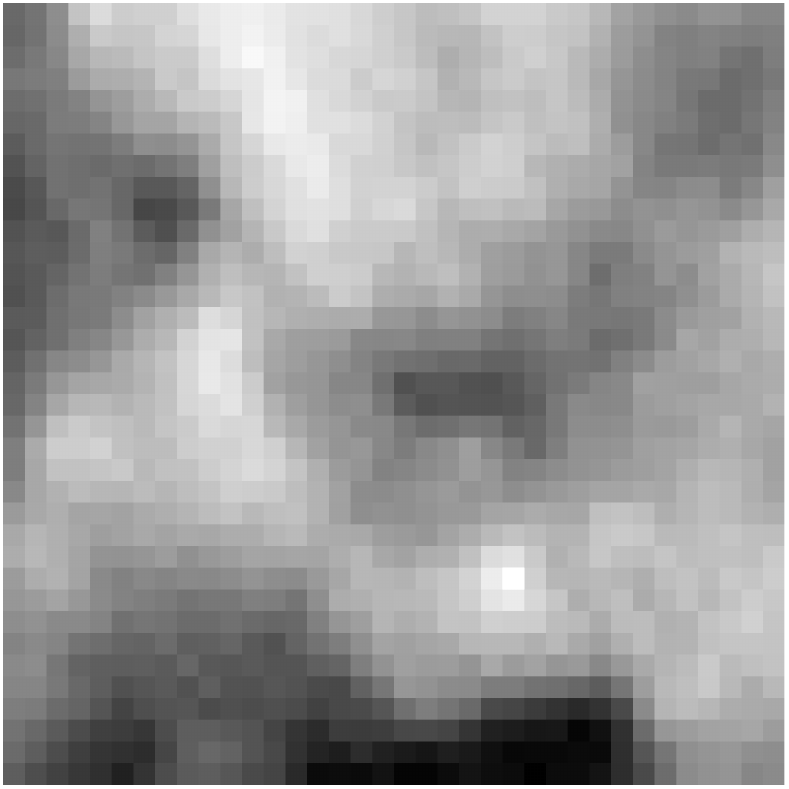}
\includegraphics[width=0.13\linewidth, angle=180]{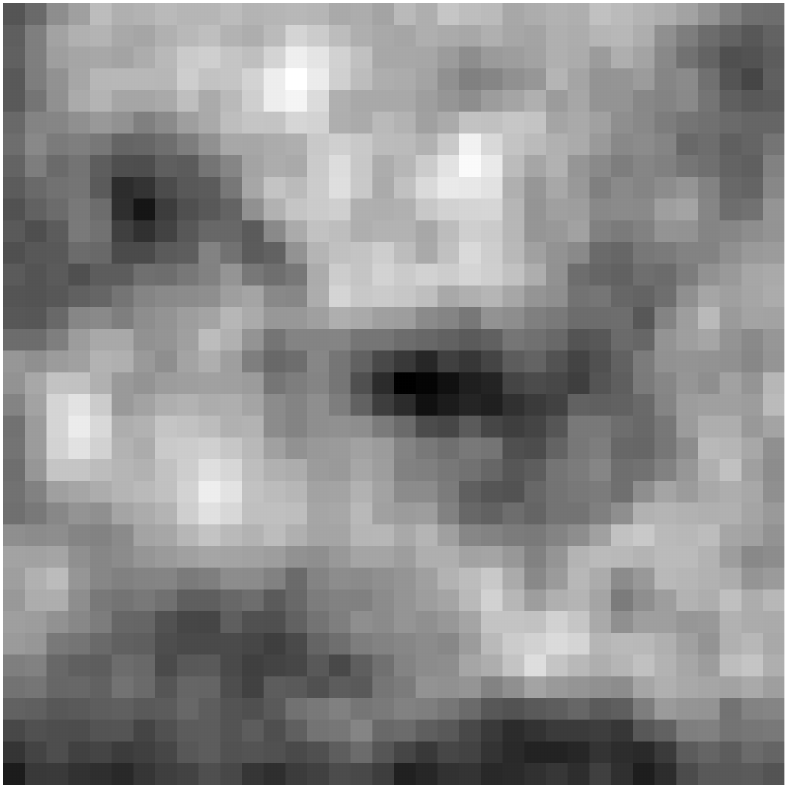}
\includegraphics[width=0.13\linewidth, angle=180]{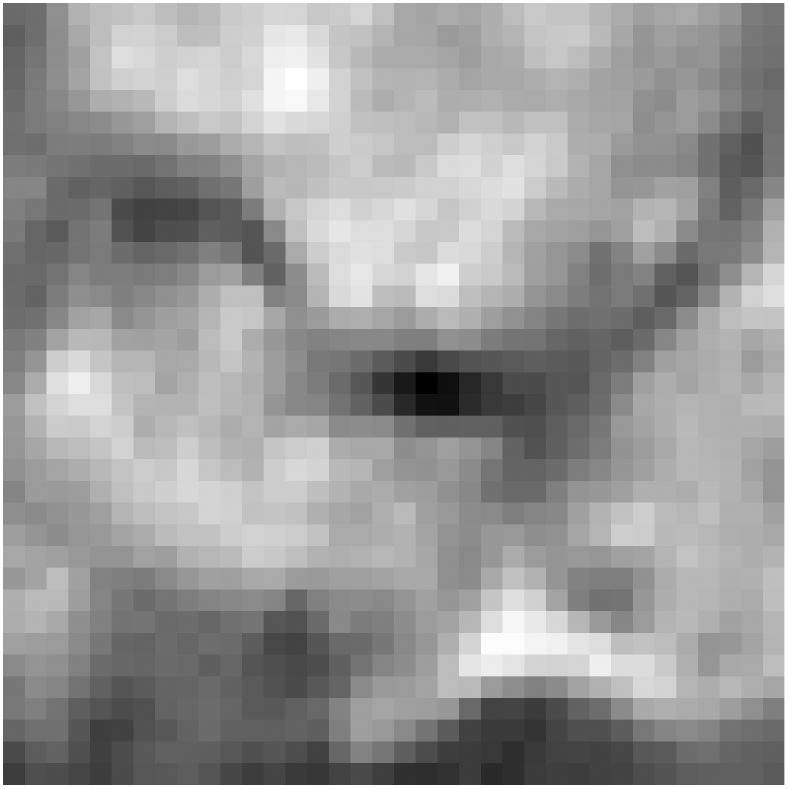}
\includegraphics[width=0.13\linewidth, angle=180]{fig_chp4/white.pdf}\\
\includegraphics[width=0.13\linewidth, angle=270]{fig_chp4/meta_error.pdf}
\includegraphics[width=0.13\linewidth, angle=180]{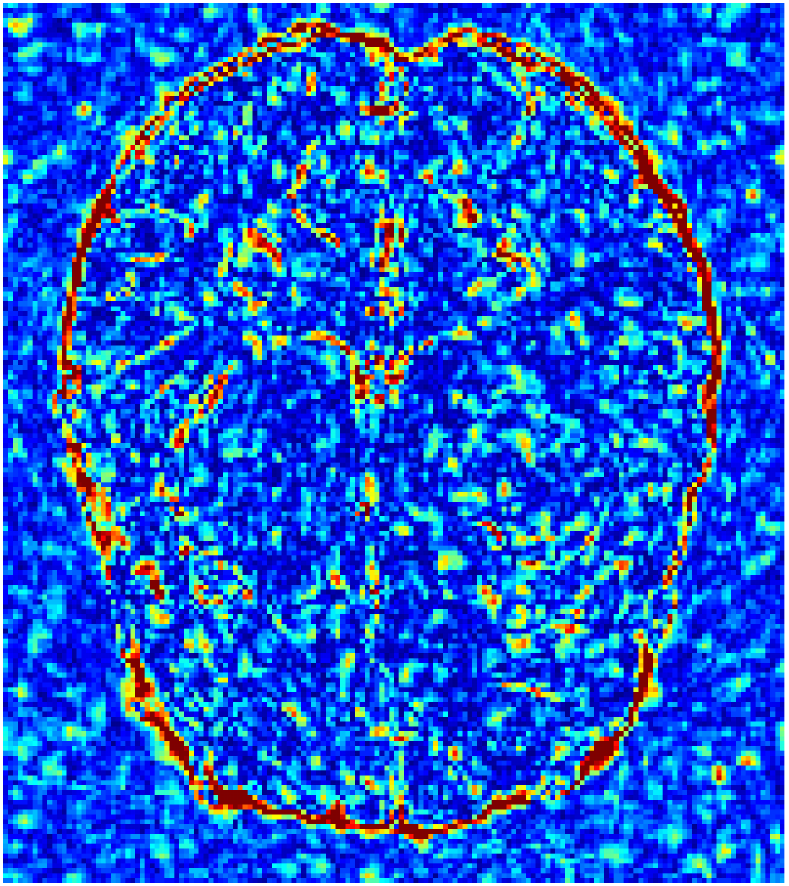}
\includegraphics[width=0.13\linewidth, angle=180]{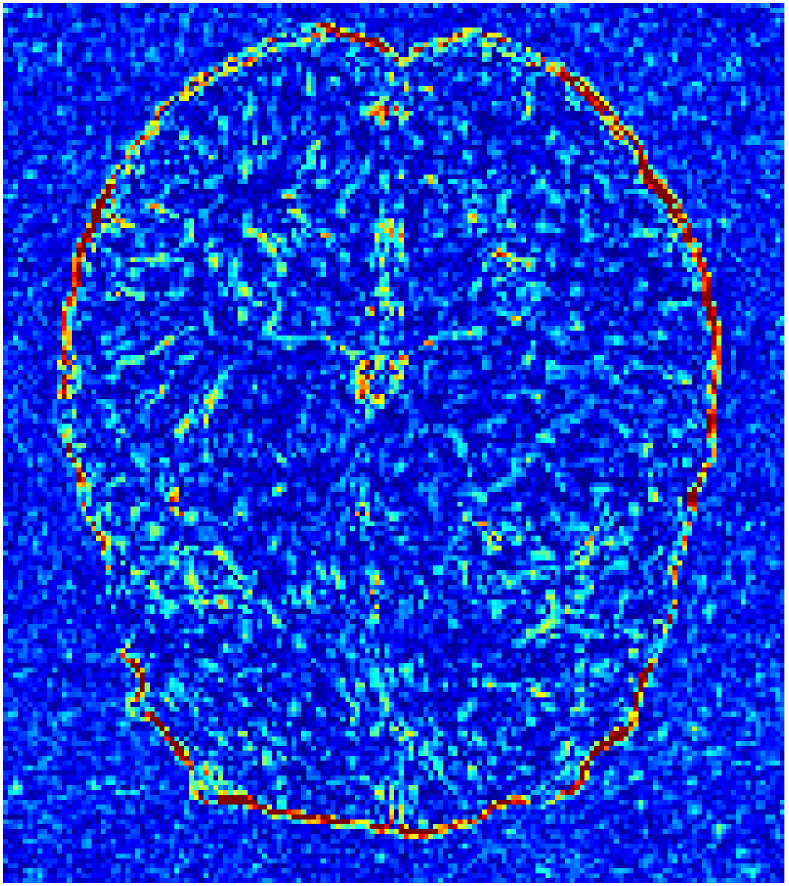}
\includegraphics[width=0.13\linewidth, angle=180]{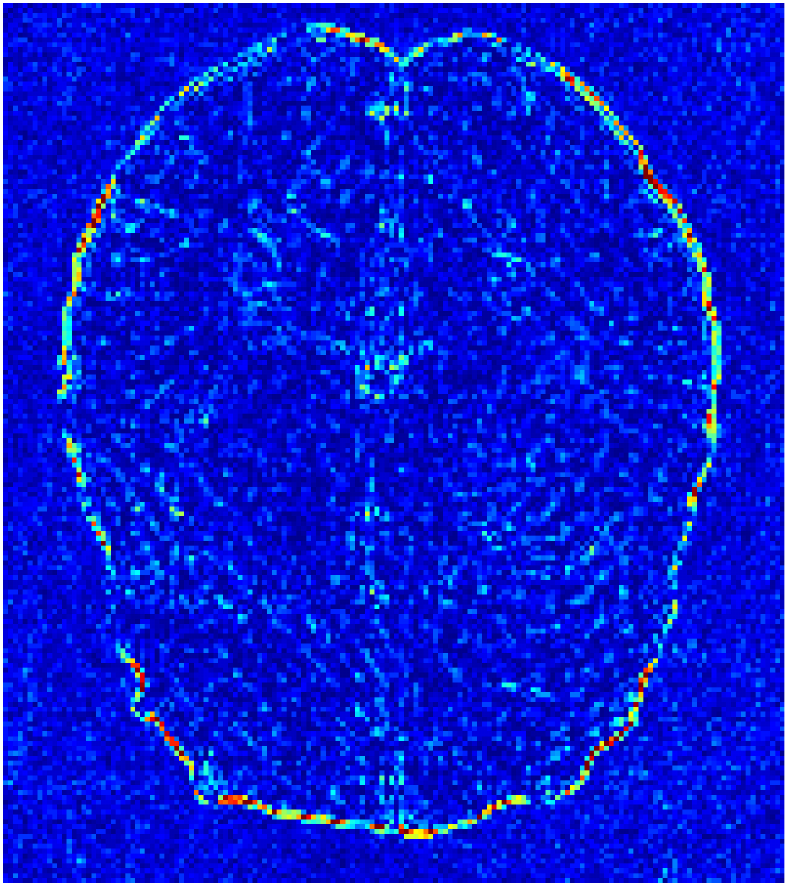}
\includegraphics[width=0.13\linewidth, angle=180]{fig_chp4/colorbar.pdf}\\
\includegraphics[width=0.13\linewidth, angle=270]{fig_chp4/conventional_error.pdf}
\includegraphics[width=0.13\linewidth, angle=180]{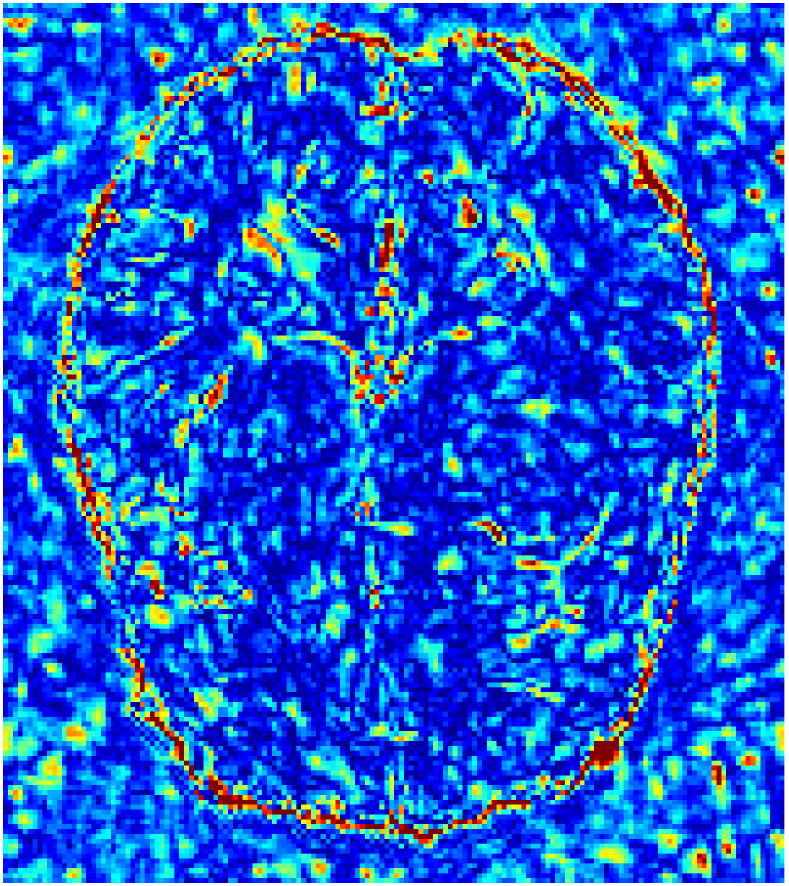}
\includegraphics[width=0.13\linewidth, angle=180]{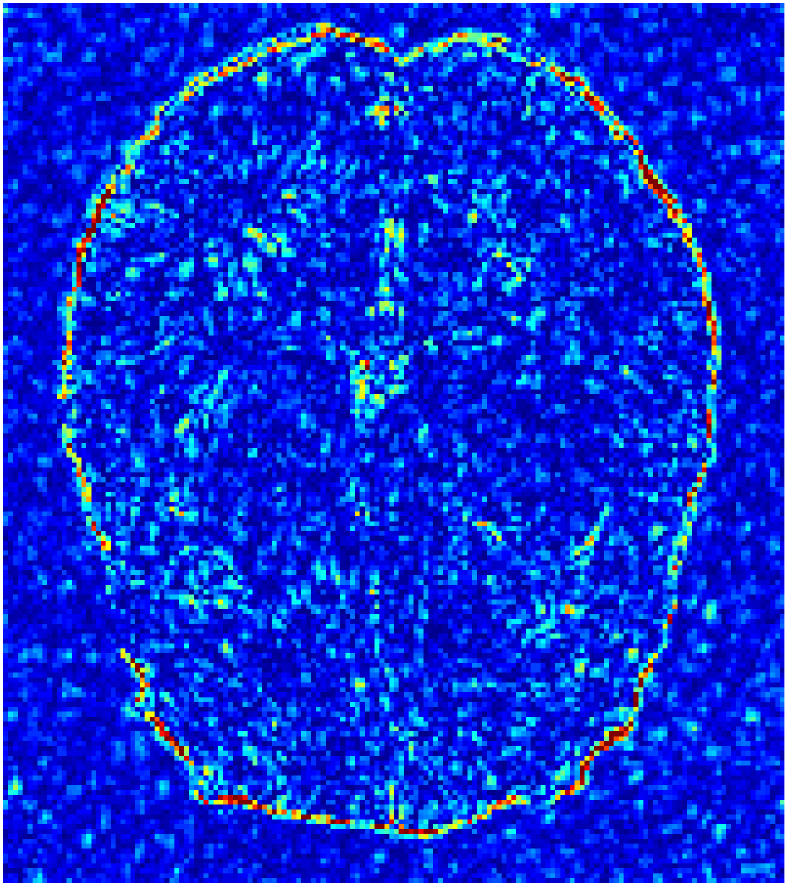}
\includegraphics[width=0.13\linewidth, angle=180]{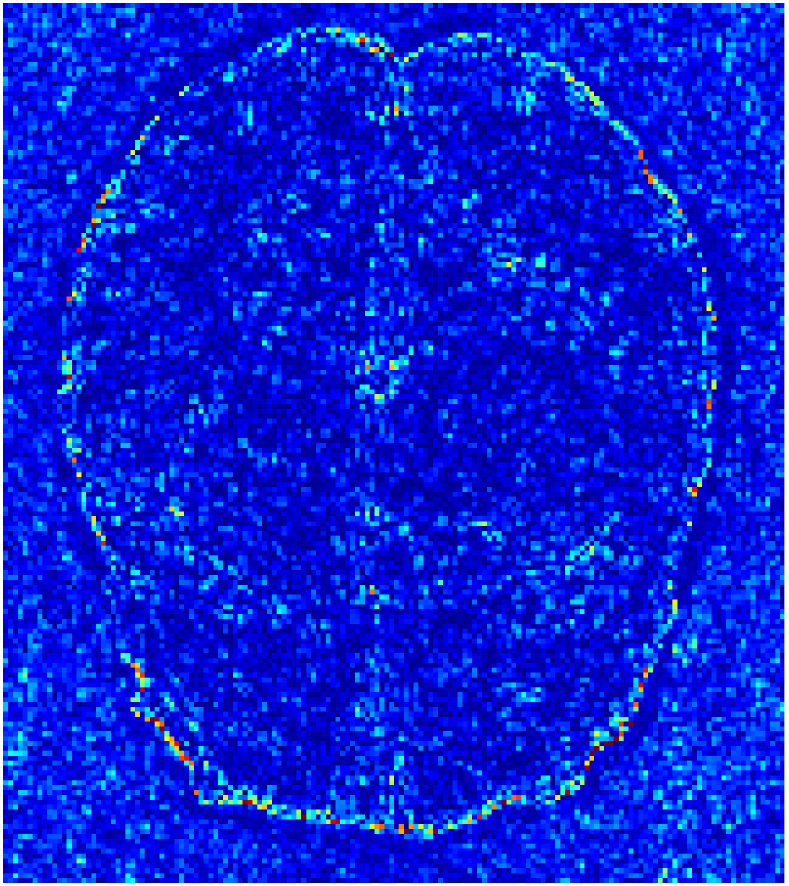}
\includegraphics[width=0.13\linewidth, angle=180]{fig_chp4/white.pdf}\\
\includegraphics[width=0.13\linewidth, angle=90]{fig_chp4/masks.pdf}
\includegraphics[width=0.13\linewidth]{fig_chp4/mask15_t1.pdf}
\includegraphics[width=0.13\linewidth]{fig_chp4/mask25_t1.pdf}
\includegraphics[width=0.13\linewidth]{fig_chp4/mask35_t1.pdf}
\includegraphics[width=0.13\linewidth]{fig_chp4/white.pdf}
\caption{The pictures (from top to bottom) display the T2 brain image reconstruction results, zoomed-in details, point-wise errors with a color bar, and associated
 \textbf{{radial}} masks.  }
\label{figure_dif_ratio_t2}
\end{figure}

\begin{figure}[H]
\includegraphics[width=0.13\linewidth, angle=270]{fig_chp4/meta_result.pdf}
\includegraphics[width=0.13\linewidth, angle=180]{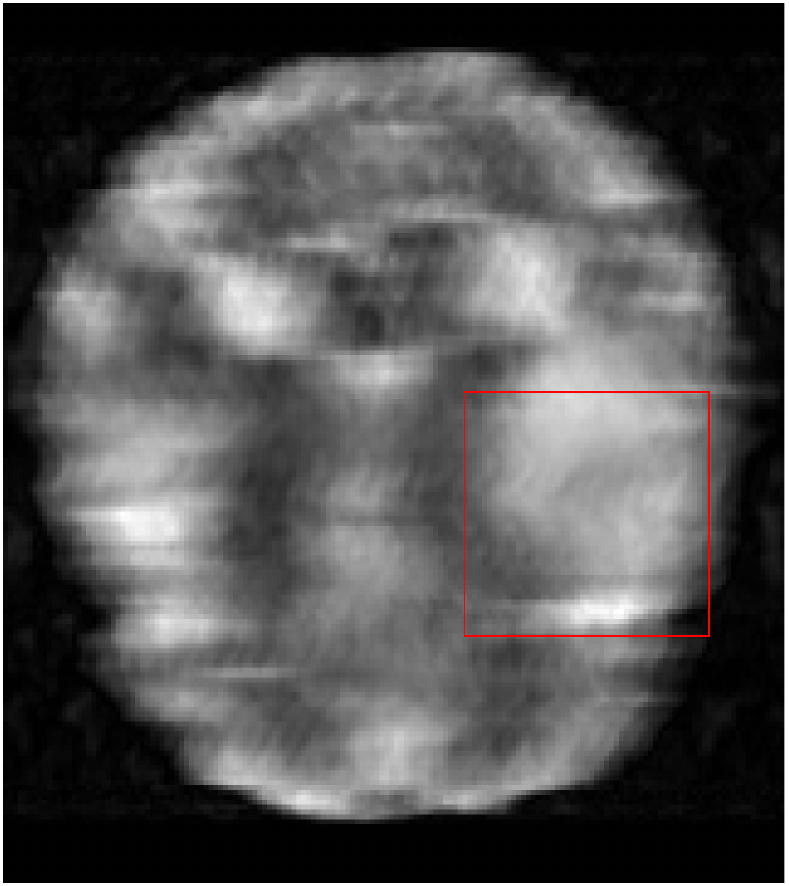}
\includegraphics[width=0.13\linewidth, angle=180]{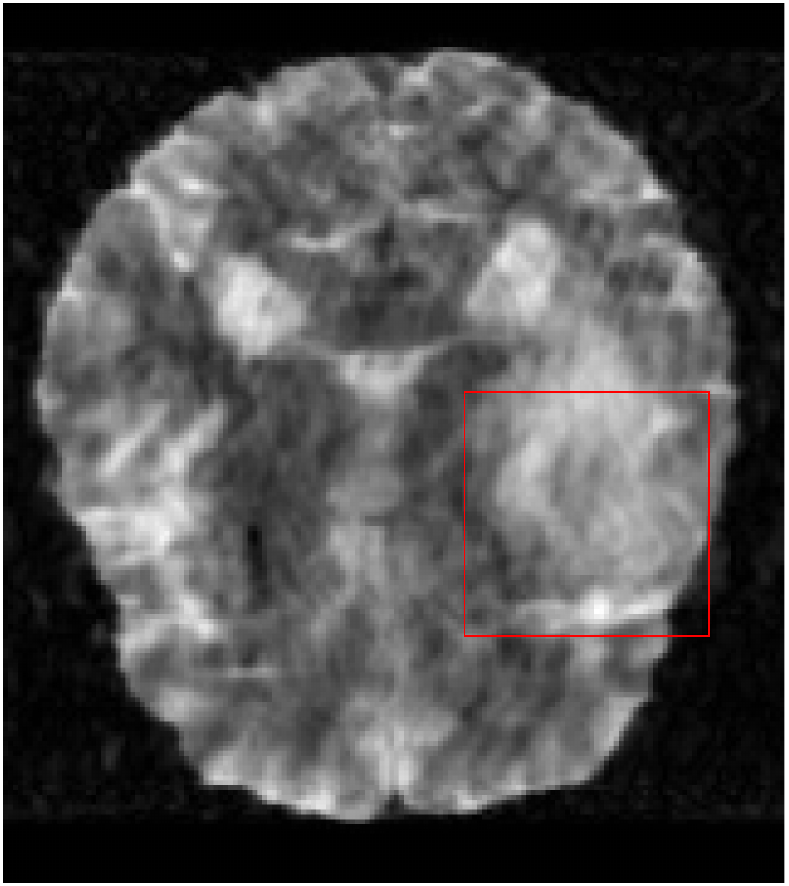}
\includegraphics[width=0.13\linewidth, angle=180]{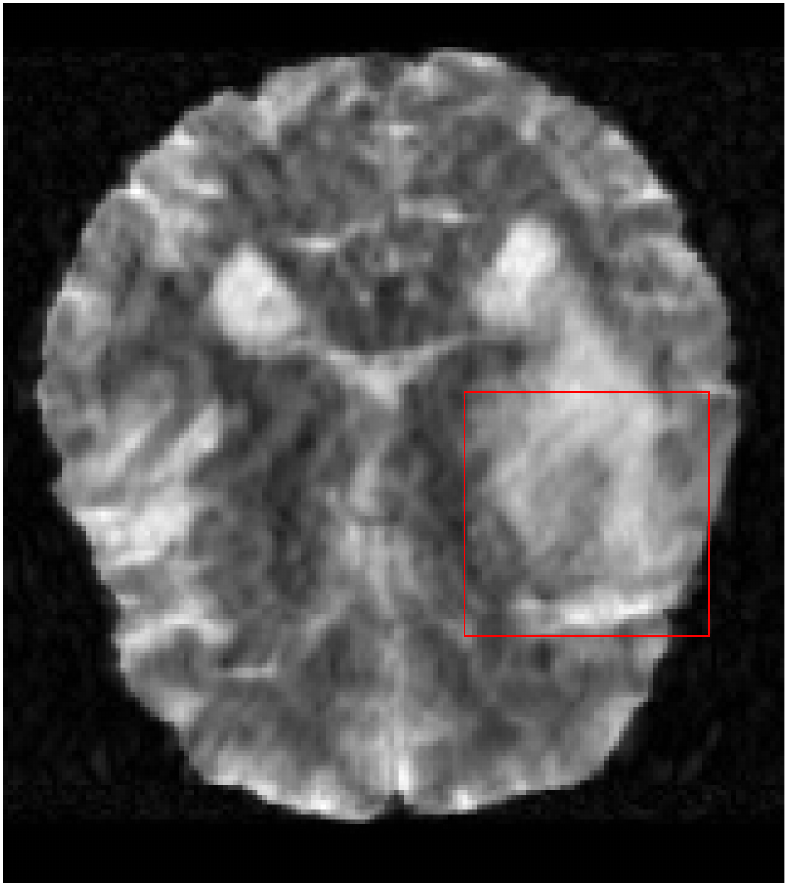}
\includegraphics[width=0.13\linewidth, angle=180]{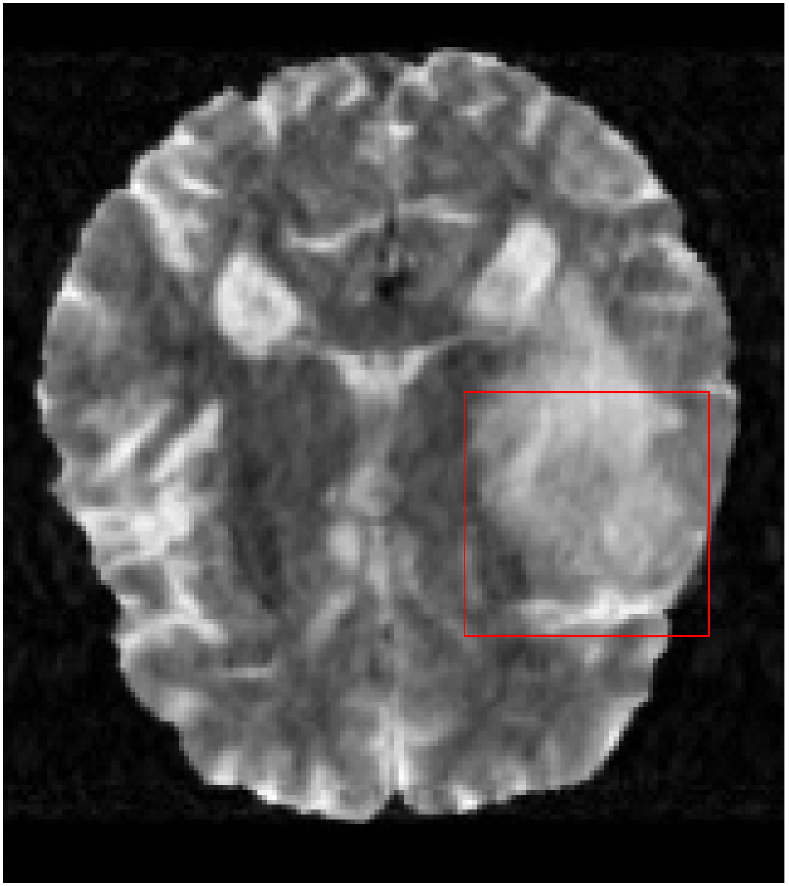}
\includegraphics[width=0.13\linewidth, angle=180]{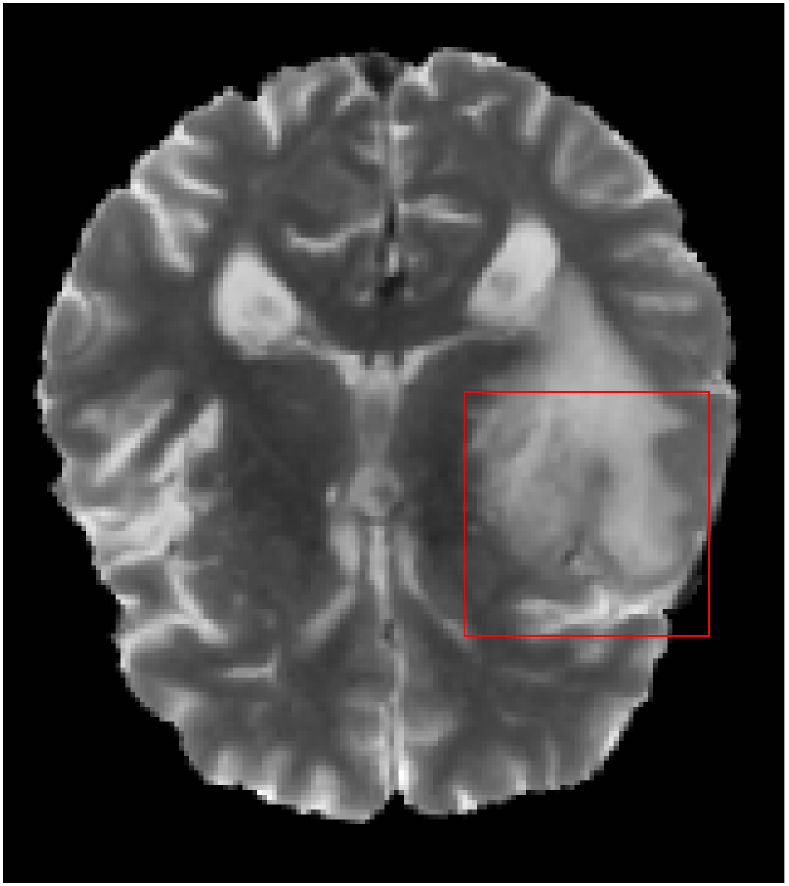}\\
\includegraphics[width=0.13\linewidth, angle=270]{fig_chp4/conventional_result.pdf}
\includegraphics[width=0.13\linewidth, angle=180]{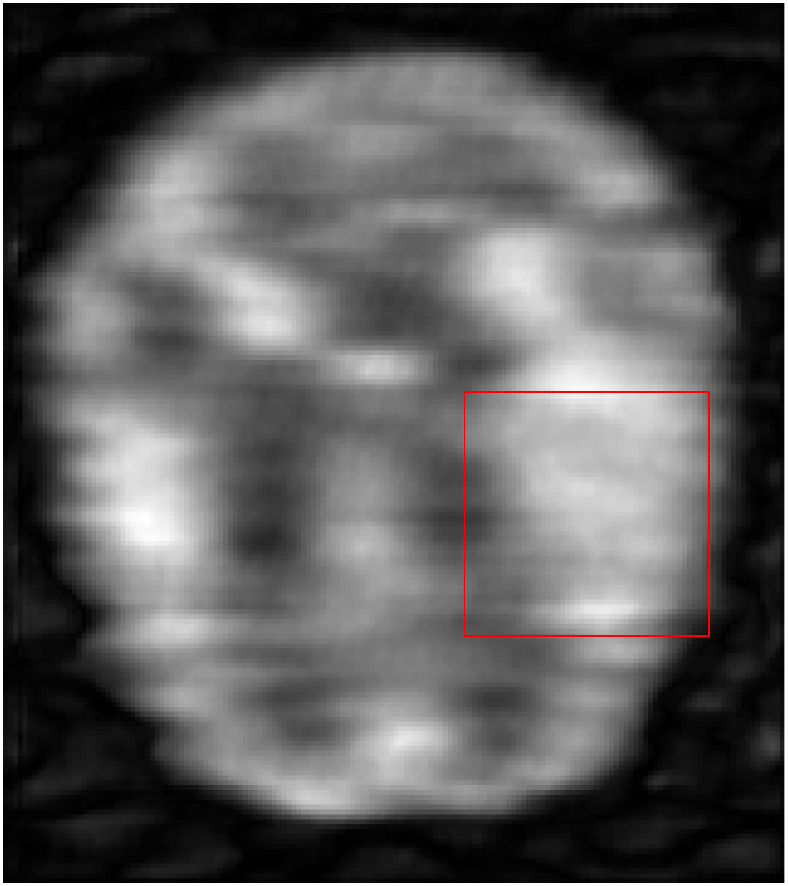}
\includegraphics[width=0.13\linewidth, angle=180]{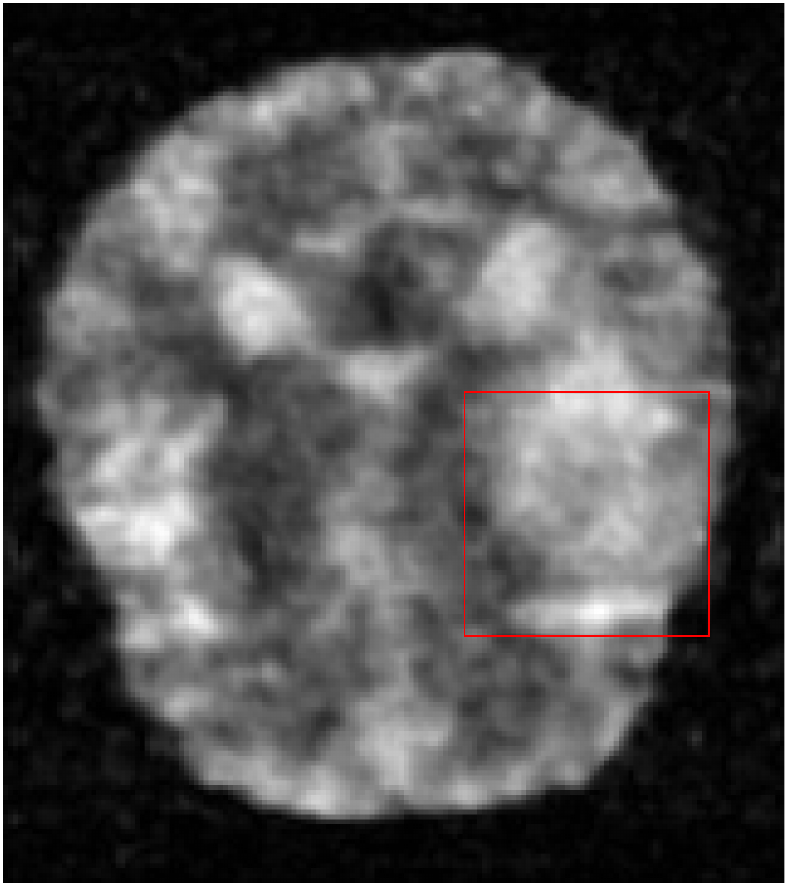}
\includegraphics[width=0.13\linewidth, angle=180]{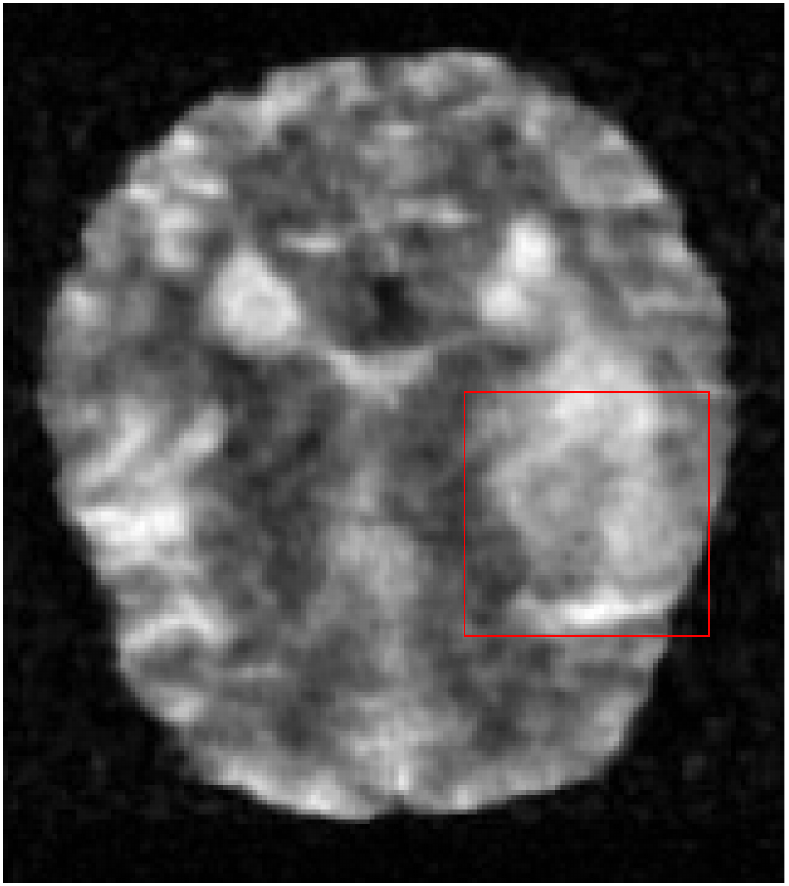}
\includegraphics[width=0.13\linewidth, angle=180]{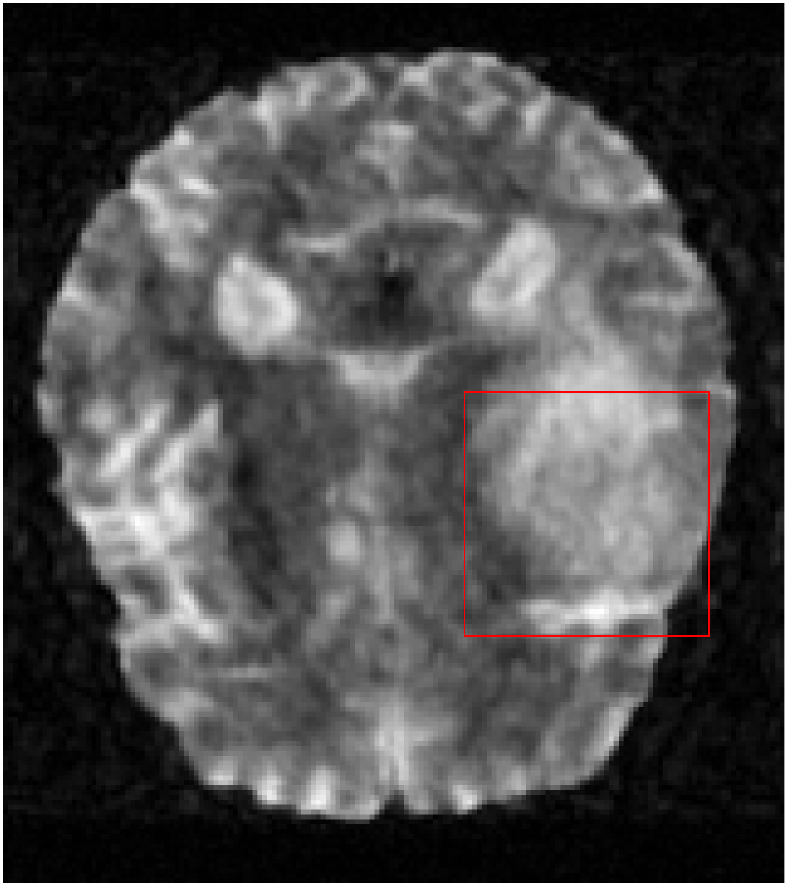}
\includegraphics[width=0.13\linewidth, angle=180]{fig_chp4/white.pdf}\\
\includegraphics[width=0.13\linewidth, angle=270]{fig_chp4/meta_detail.pdf}
\includegraphics[width=0.13\linewidth, angle=180]{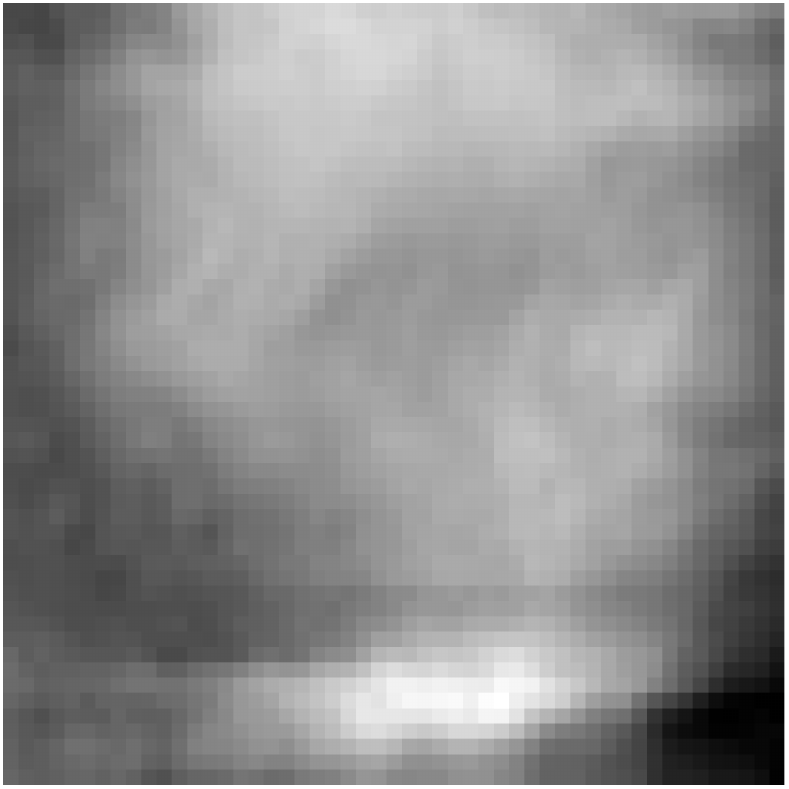}
\includegraphics[width=0.13\linewidth, angle=180]{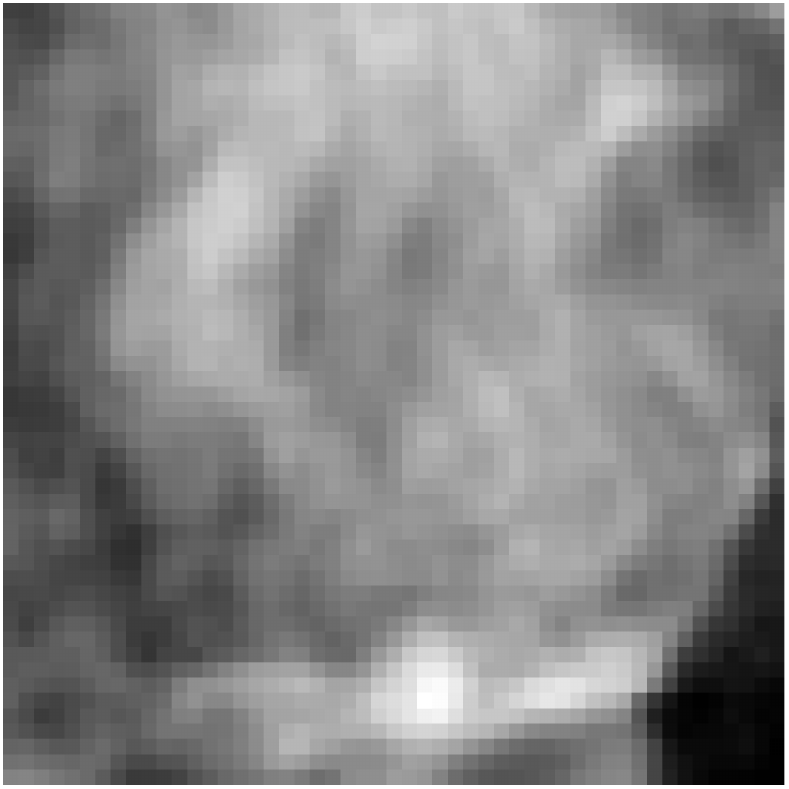}
\includegraphics[width=0.13\linewidth, angle=180]{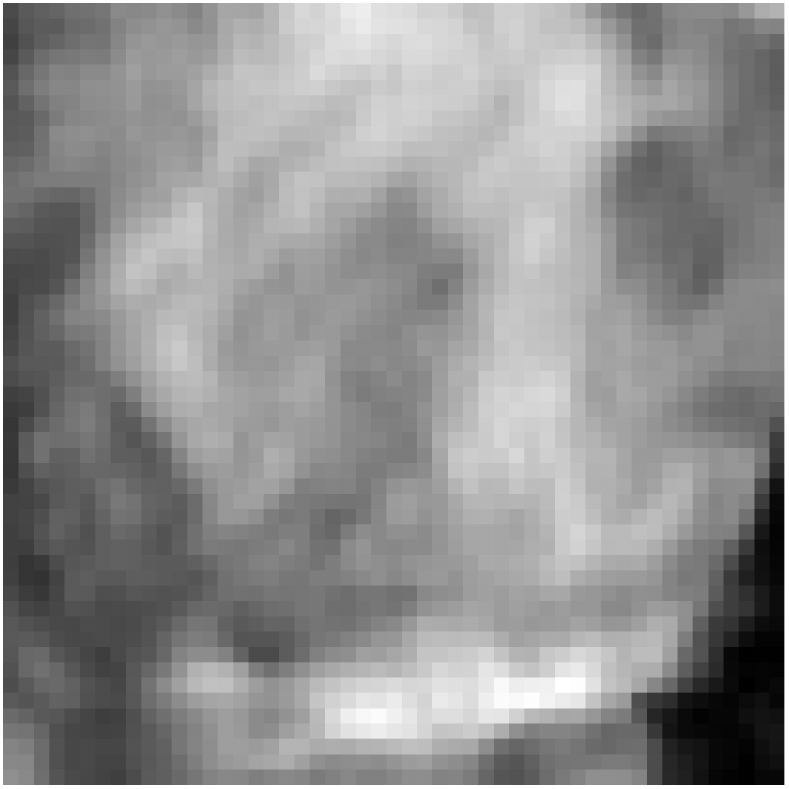}
\includegraphics[width=0.13\linewidth, angle=180]{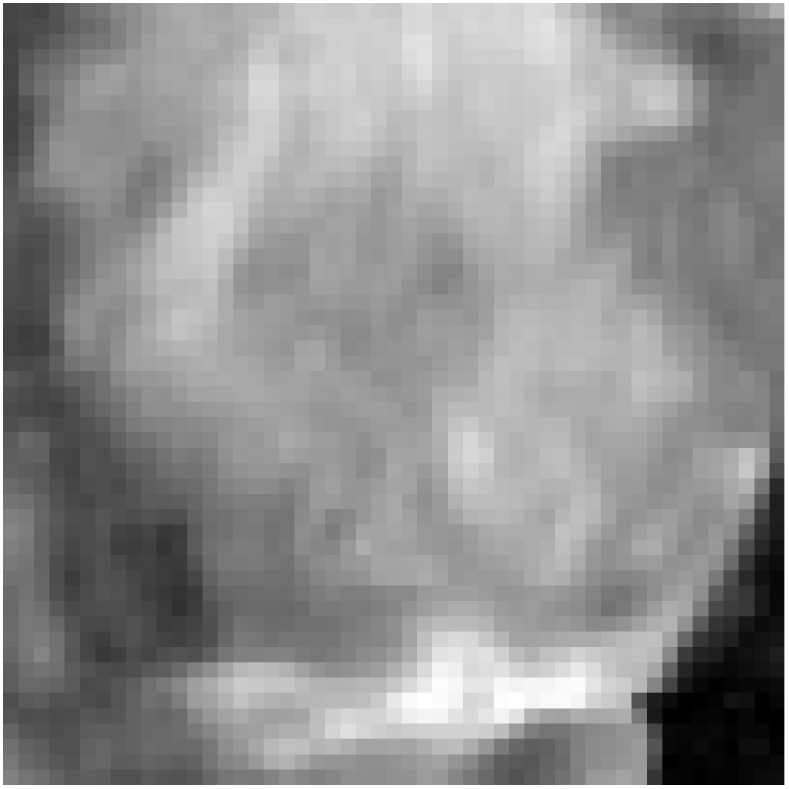}
\includegraphics[width=0.13\linewidth, angle=180]{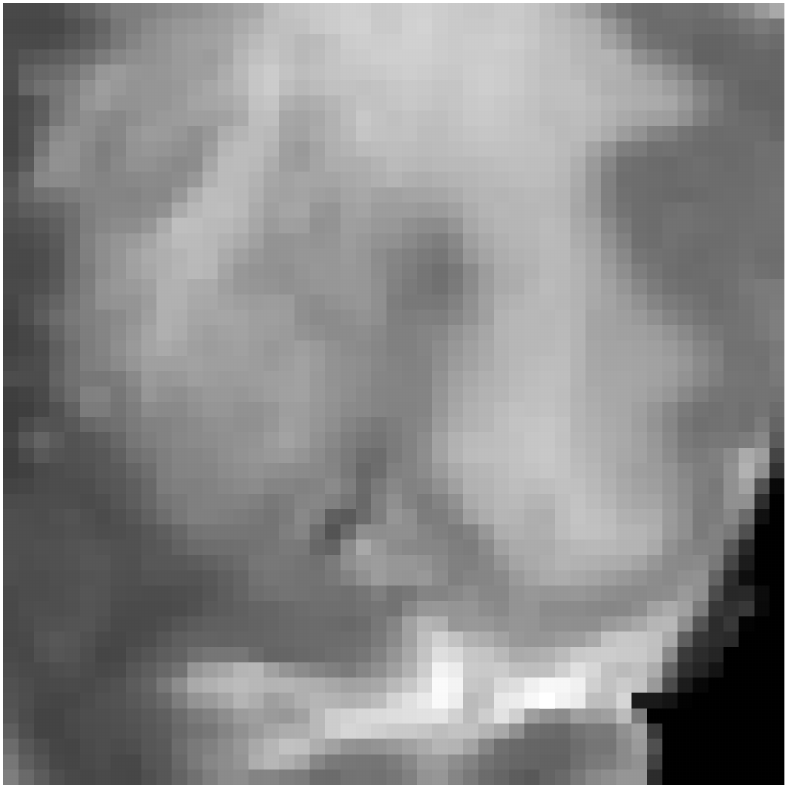}\\
\includegraphics[width=0.13\linewidth, angle=270]{fig_chp4/conventional_detail.pdf}
\includegraphics[width=0.13\linewidth, angle=180]{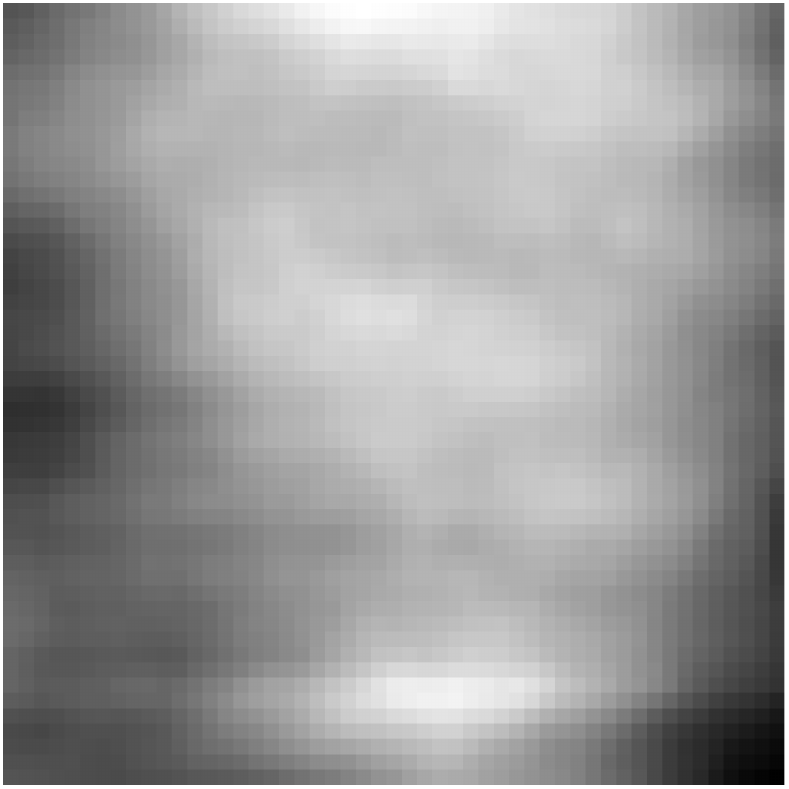}
\includegraphics[width=0.13\linewidth, angle=180]{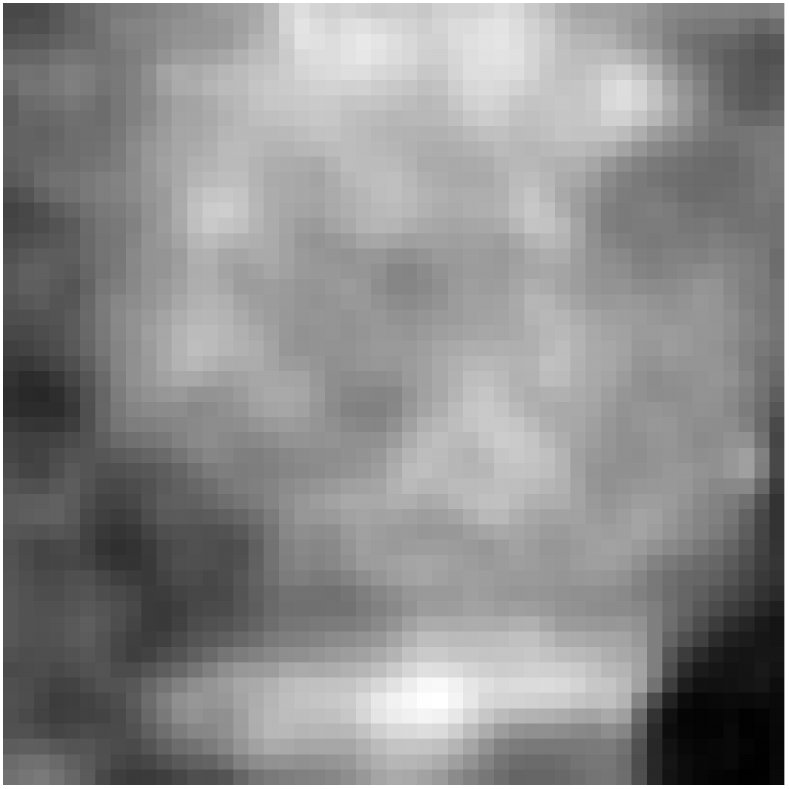}
\includegraphics[width=0.13\linewidth, angle=180]{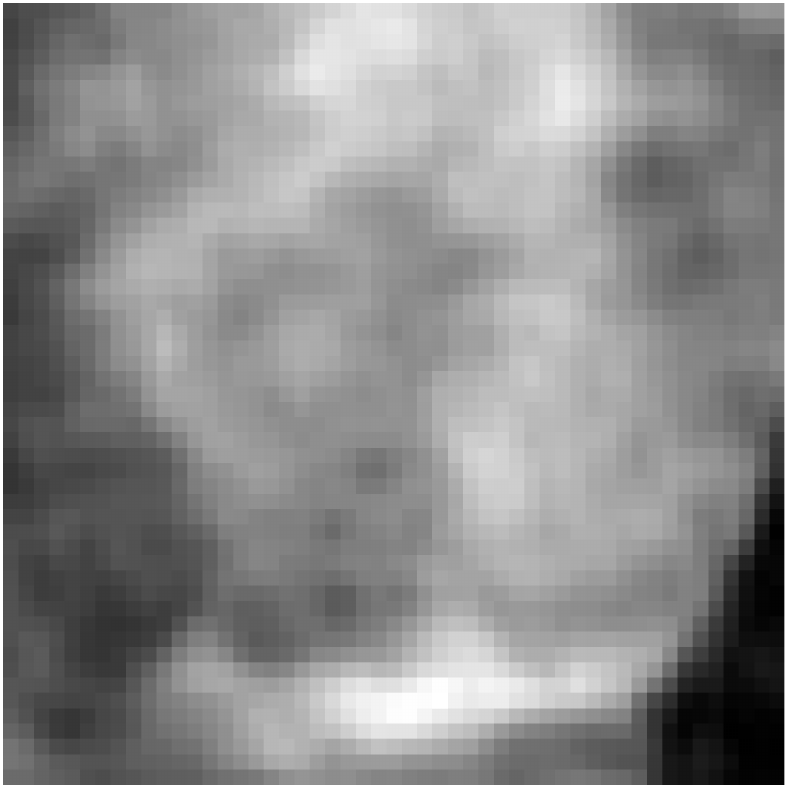}
\includegraphics[width=0.13\linewidth, angle=180]{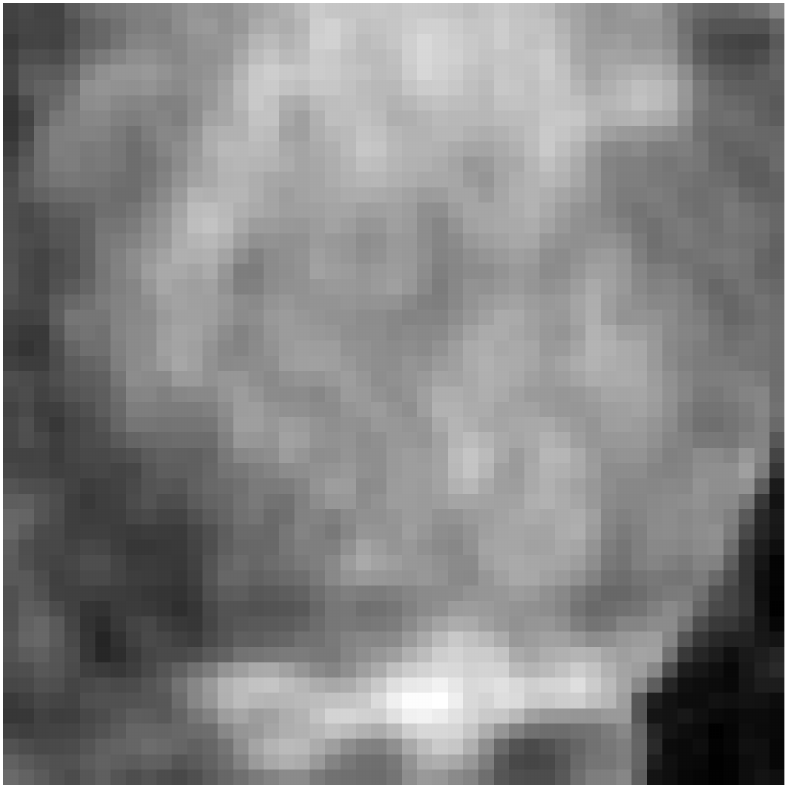}
\includegraphics[width=0.13\linewidth, angle=180]{fig_chp4/white.pdf}\\
\includegraphics[width=0.13\linewidth, angle=270]{fig_chp4/meta_error.pdf}
\includegraphics[width=0.13\linewidth, angle=180]{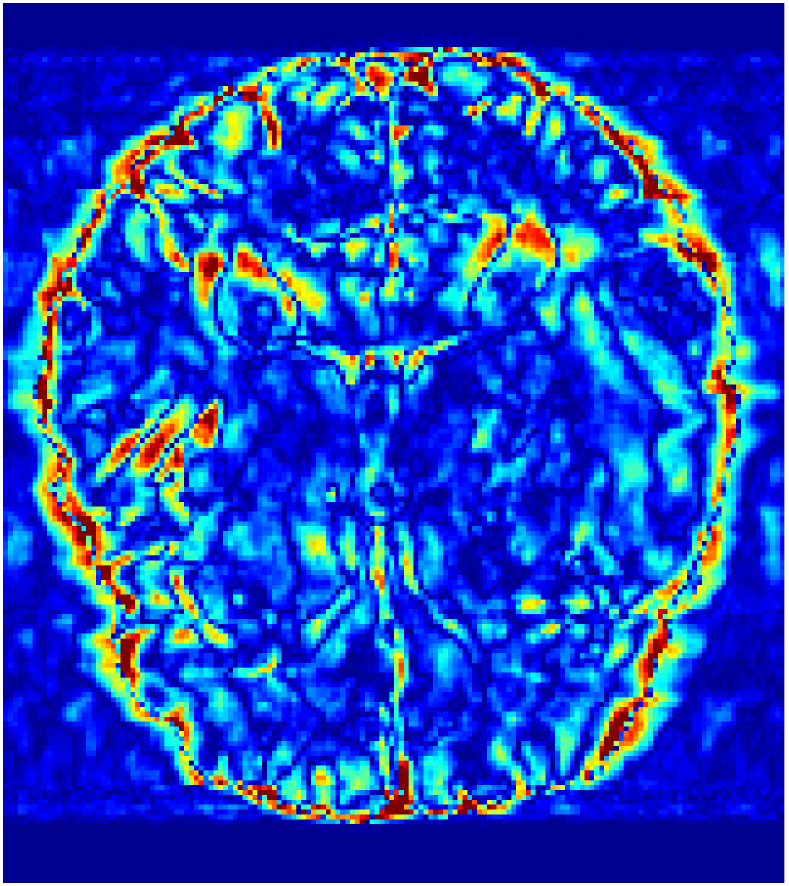}
\includegraphics[width=0.13\linewidth, angle=180]{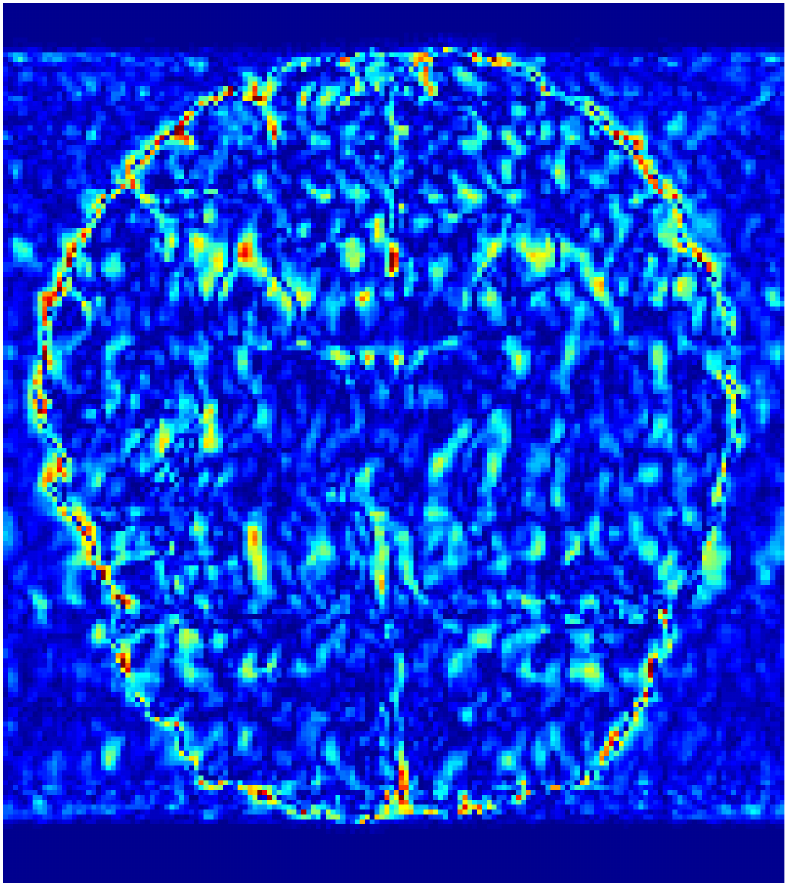}
\includegraphics[width=0.13\linewidth, angle=180]{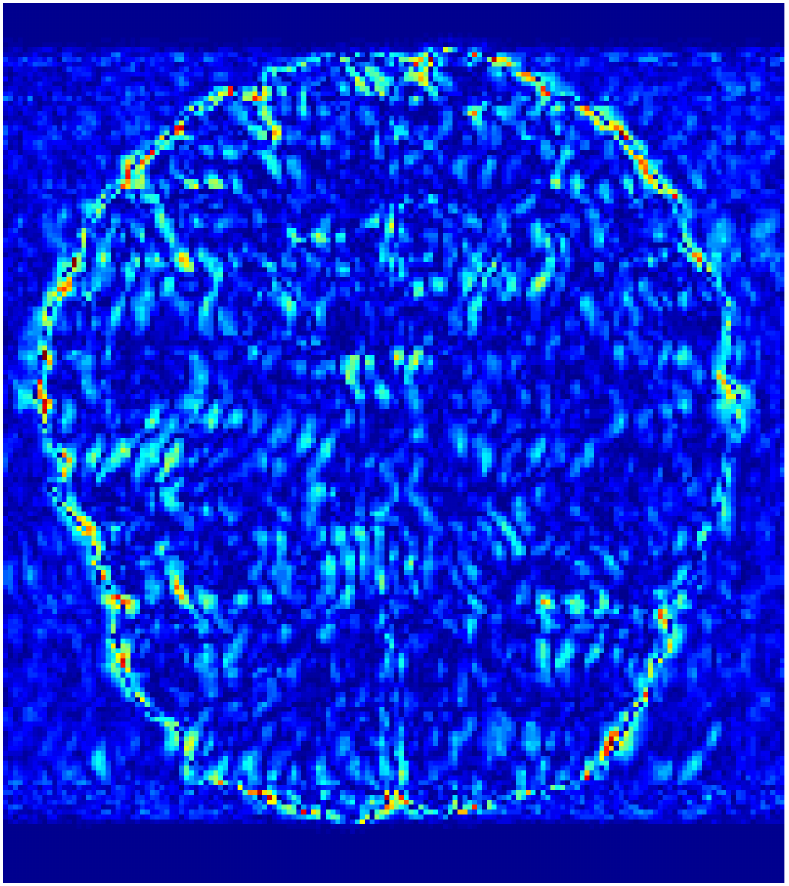}
\includegraphics[width=0.13\linewidth, angle=180]{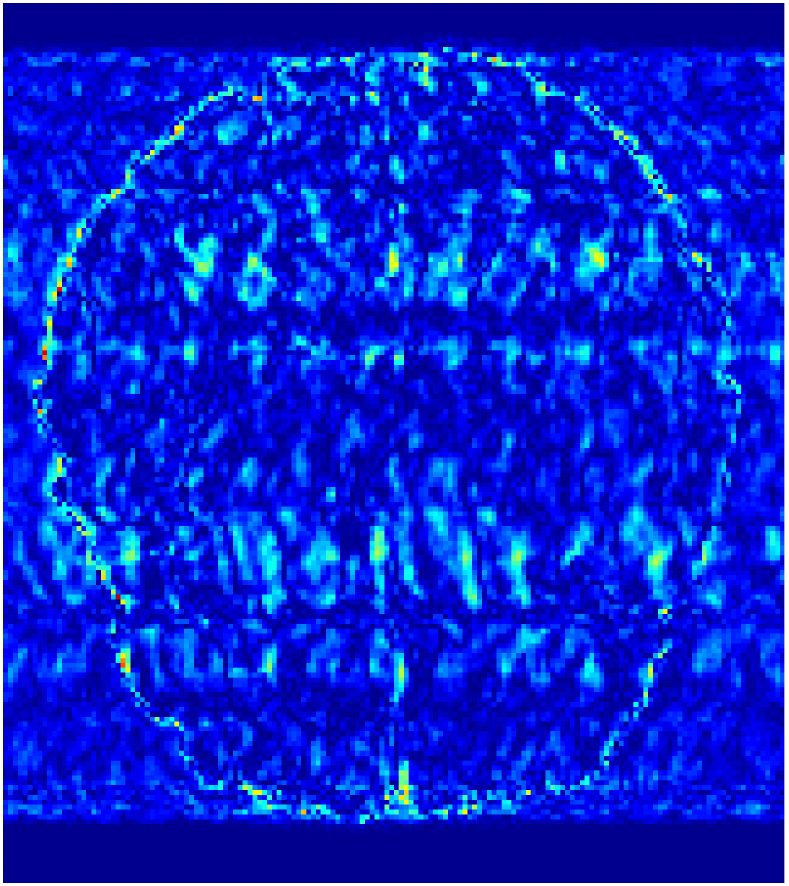}
\includegraphics[width=0.13\linewidth, angle=180]{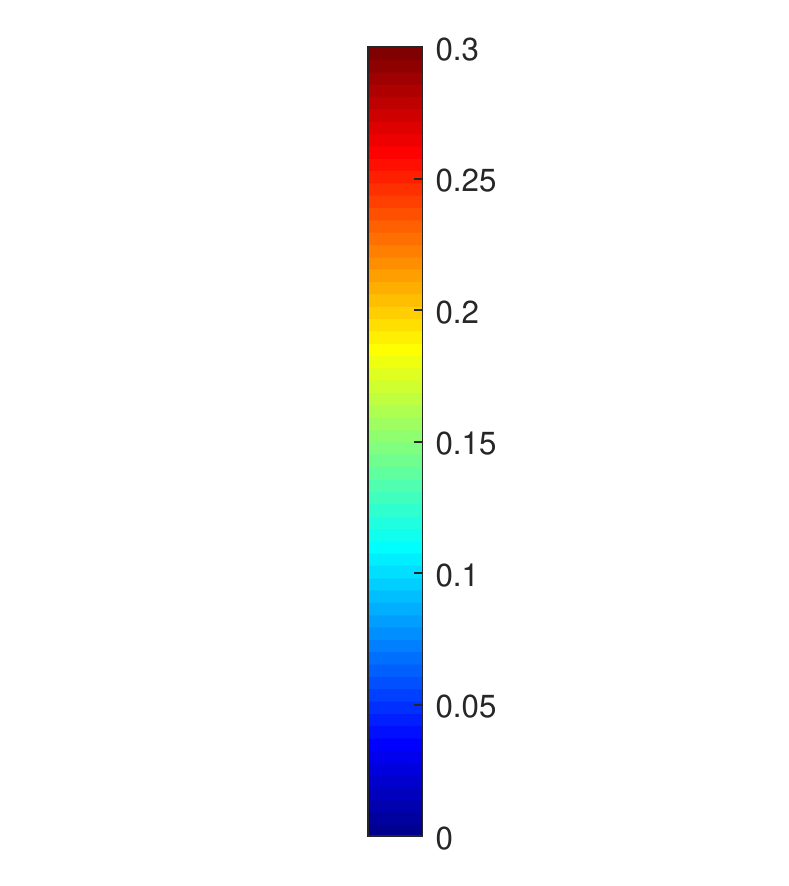}\\
\includegraphics[width=0.13\linewidth, angle=270]{fig_chp4/conventional_error.pdf}
\includegraphics[width=0.13\linewidth, angle=180]{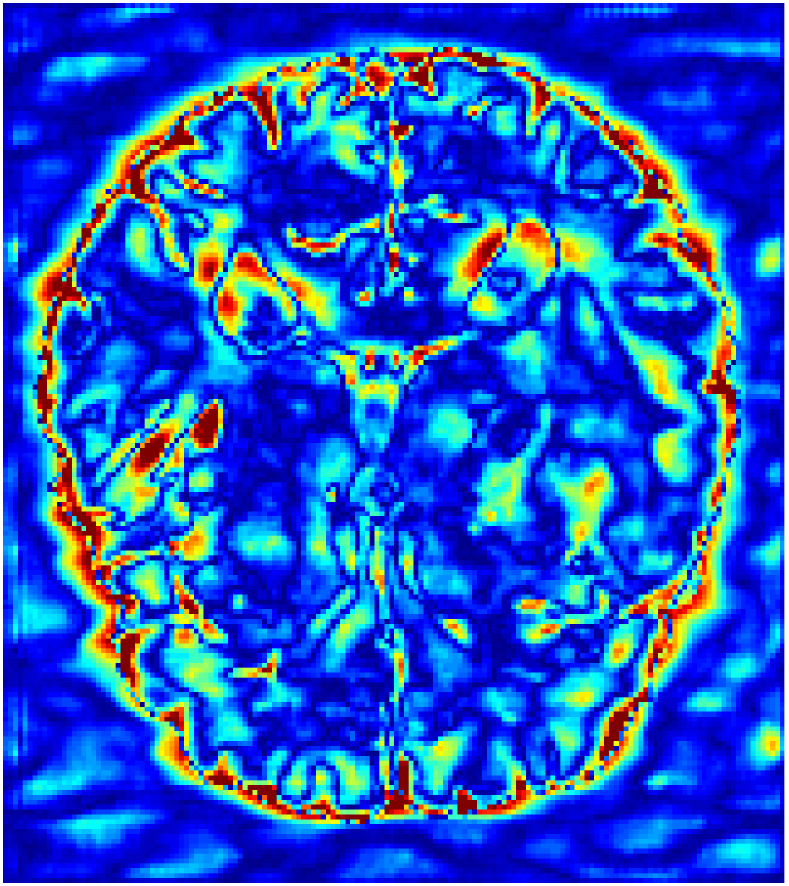}
\includegraphics[width=0.13\linewidth, angle=180]{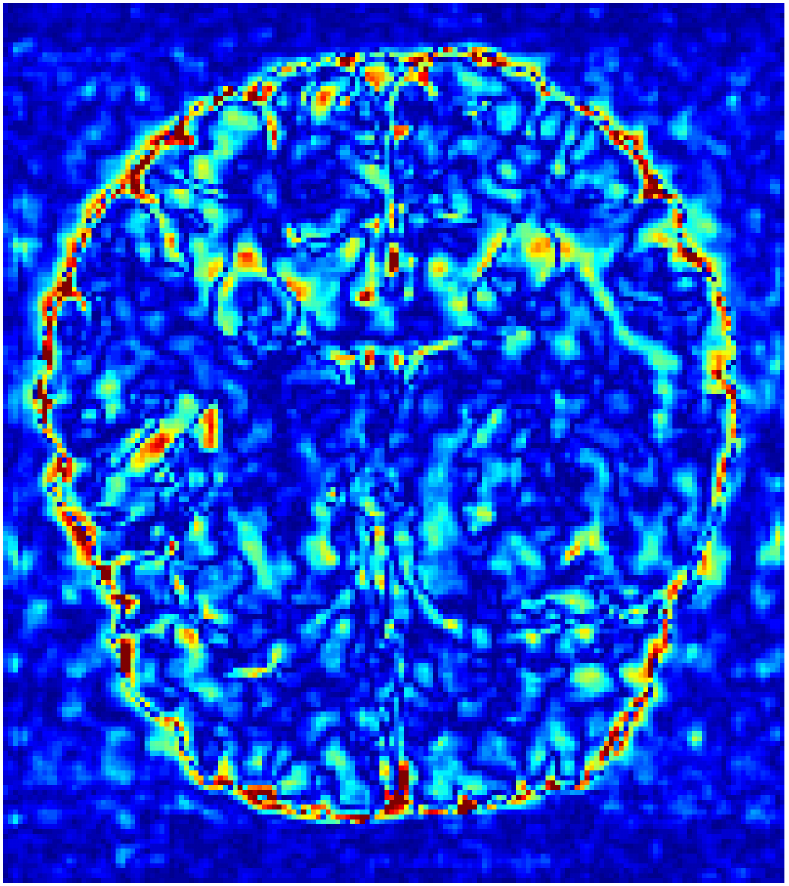}
\includegraphics[width=0.13\linewidth, angle=180]{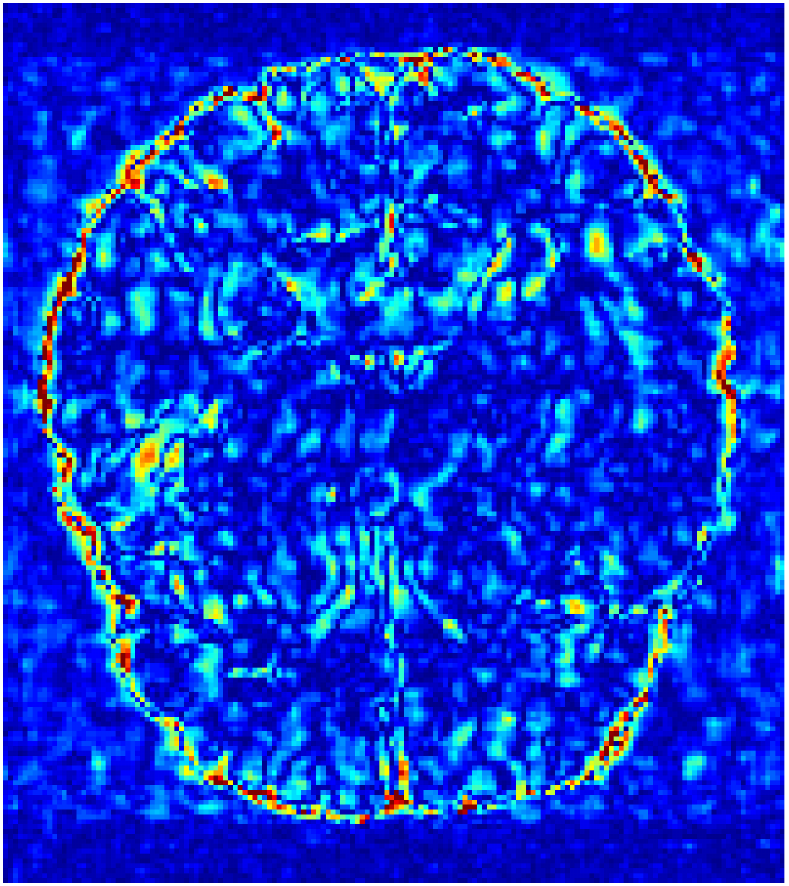}
\includegraphics[width=0.13\linewidth, angle=180]{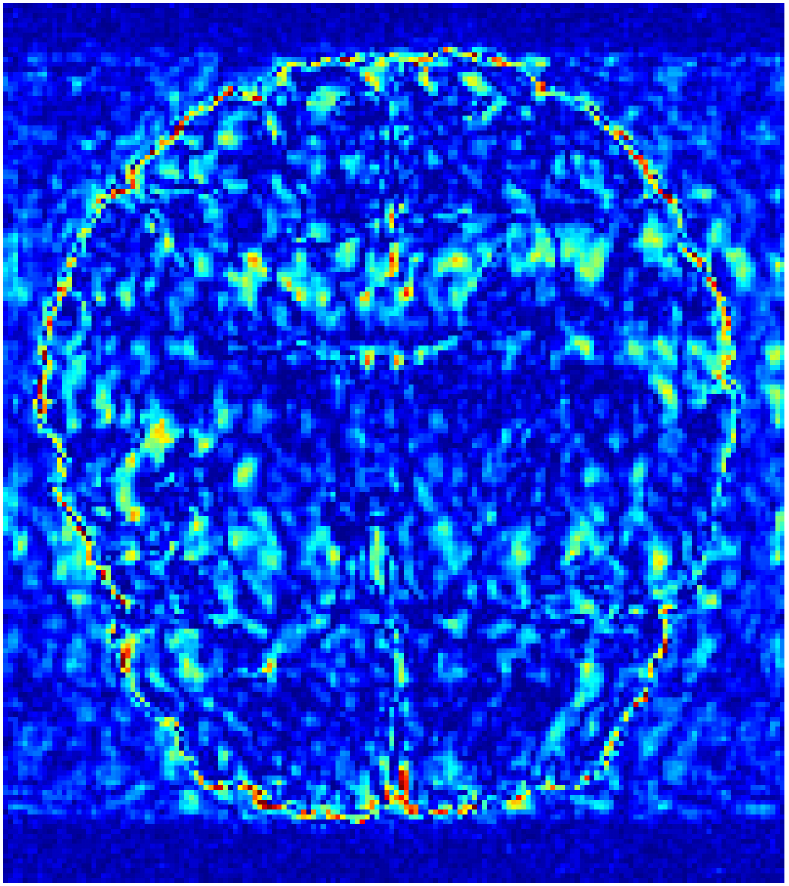}
\includegraphics[width=0.13\linewidth, angle=180]{fig_chp4/white.pdf}\\
\includegraphics[width=0.13\linewidth, angle=90]{fig_chp4/masks.pdf}
\includegraphics[width=0.13\linewidth]{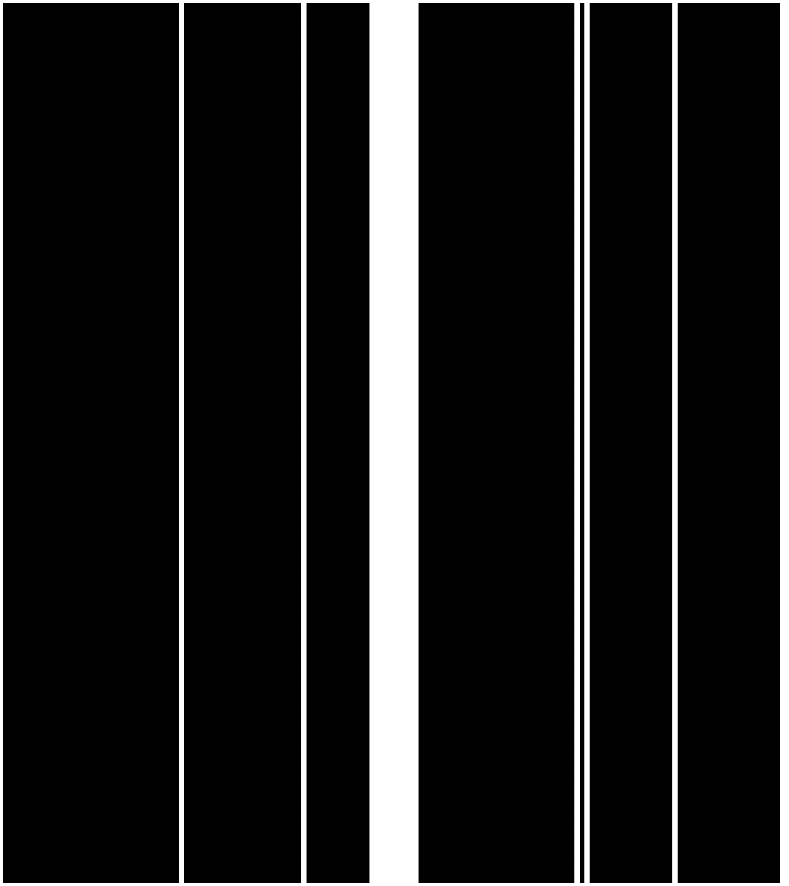}
\includegraphics[width=0.13\linewidth]{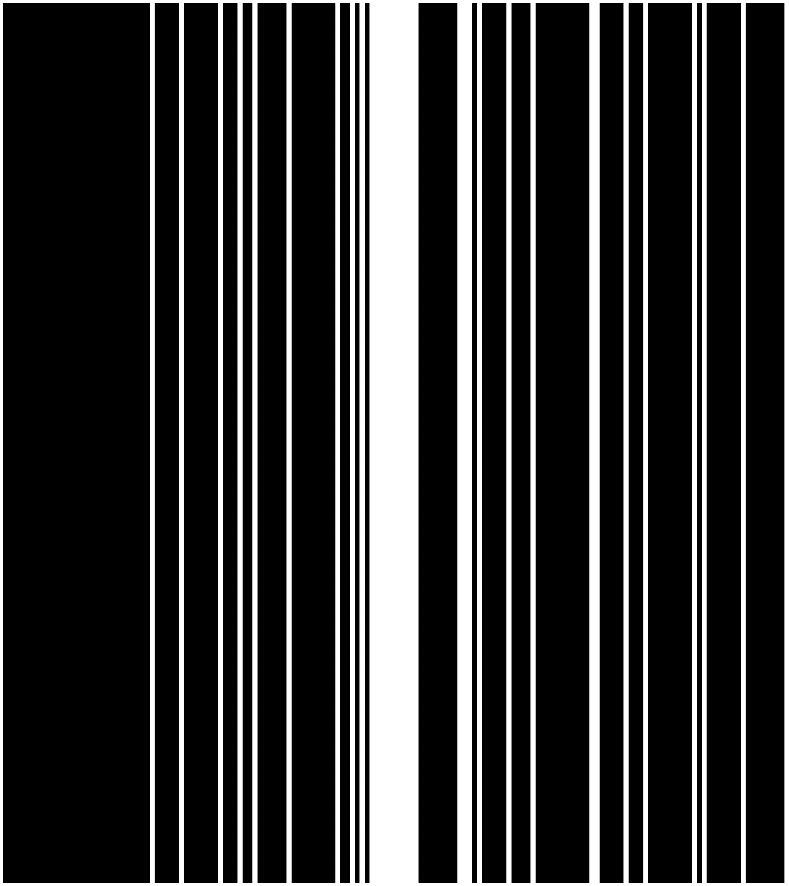}
\includegraphics[width=0.13\linewidth]{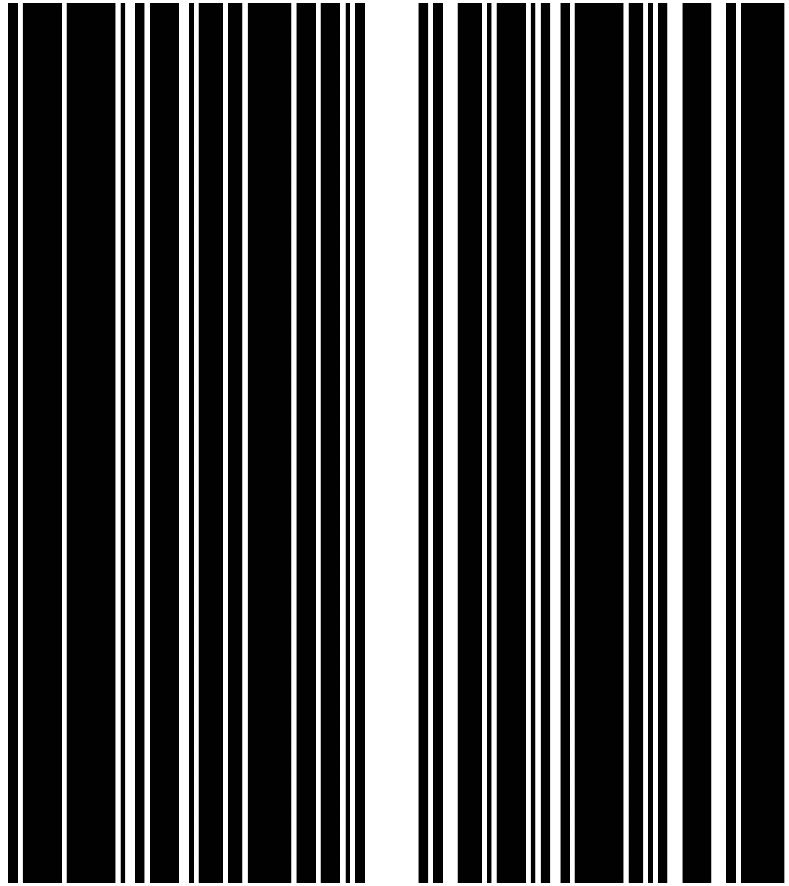}
\includegraphics[width=0.13\linewidth]{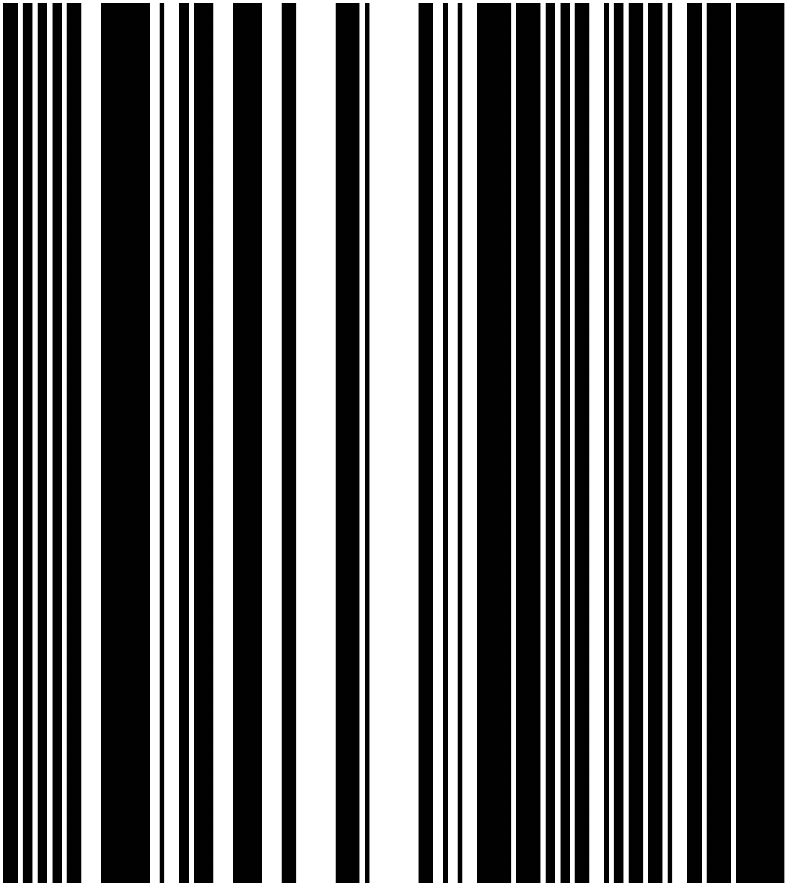}
\includegraphics[width=0.13\linewidth]{fig_chp4/white.pdf}
\caption{The pictures (from top to bottom) display the T2 brain image reconstruction results, zoomed-in details, point-wise errors with a color bar, and associated
 \textbf{Cartesian} masks. }
\label{figure_same_ratio_t2_cts}
\end{figure}

{Next, we empirically demonstrate the convergence of Algorithm \ref{alg:lda} in Figure \ref{fig:convergence}. This shows that the objective function value $\phi$ decreases and the PSNR value for testing data increases steadily as the number of phases increases, which indicates that the learned algorithm is indeed minimizing the learned function as we desired.}
\begin{figure}[H]

\includegraphics[width=0.45\linewidth]{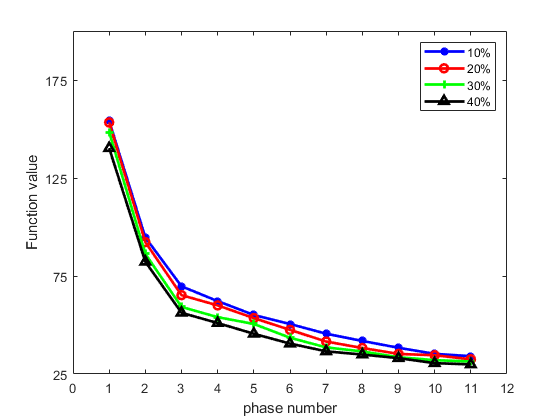}
\includegraphics[width=0.45\linewidth]{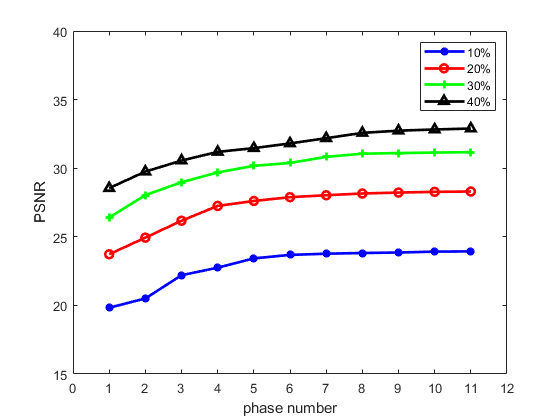}
\caption{Convergence behavior of Algorithm \ref{alg:lda} on  T1 weighted MRI image reconstruction with four different CS ratios using radial mask. \textbf{Left}: Objective function value $\phi$ versus phase number. \textbf{Right}: PSNR value versus phase number.}
\label{fig:convergence}
\end{figure}

\subsection{Future Work and Open Challenges}

Deep optimization-based meta-learning techniques have shown great generalizability, but there are several open challenges that can be discussed and can potentially be addressed in future work. 
A major issue is the memorization problem, since the base learner needs to be optimized for a large number of phases and the training algorithm contains multiple gradient steps; furthermore, the computation is very expensive in terms of time and memory costs. 
In addition to reconstructing MRI through different trajectories, another potential application for medical imaging could be multi-modality reconstruction and synthesis. Capturing images of anatomy with multi-modality acquisitions enhances the diagnostic information and could be cast as a multi-task problem that could benefit from meta-learning.

%%%%%%%%%%%%%%%%%%%%%%%%%%%%%%%%%%%%%%%%%%

\section{Conclusions}\label{conclusion}
In this paper, we put forward a novel deep model for MRI reconstructions via meta-learning. The proposed method has the ability to solve multi-tasks synergistically, and the well-trained model could generalize well to new tasks. Our baseline network is constructed by unfolding an LOA, which inherits the convergence property, improves the interpretability, and promotes the parameter efficiency of the designed network structure. The designated adaptive regularizer consists of a task-invariant learner and task-specific meta-knowledge. Network training follows a bilevel optimization algorithm that minimizes task-specific parameter $\omega$ in the upper level for the validation data and minimizes task-invariant parameters $\theta$ in the lower level for the training data with fixed $\omega$. The proposed approach is the first model designed to solve the inverse problem by applying meta-training on the adaptive regularization in the variational model. We consider the recovery of undersampled raw data across different sampling trajectories with various sampling patterns as different tasks. Extensive numerical experiments on various MRI datasets demonstrate that the proposed method generalizes well at various
sampling trajectories and is capable of fast adaption to unseen trajectories and sampling patterns. The reconstructed images achieve higher quality compared to conventional supervised learning for both seen and unseen k-space trajectory cases.

%%%%%%%%%%%%%%%%%%%%%%%%%%%%%%%%%%%%%%%%%%

\section{Convergence Analysis}
\label{convergence}
We make the following assumptions regarding $f$ and $\gbf$ throughout this work:

\begin{itemize}
\item (a1): $f$ is differentiable and (possibly) nonconvex, and $ \nabla f$ is $ L_f$-Lipschitz continuous.

\item (a2): Every component of $\gbf$ is differentiable and (possibly) nonconvex, and $\nabla \gbf$ is $ L_g$-Lipschitz continuous. 

\item (a3): $\sup_{\xbf \in \Xcal}   \| \nabla \gbf(\xbf)\| \leq M$ for some constant $ M>0$.

\item (a4): $\phi$ is coercive, and $\phi^* = \min_{\xbf \in \Xcal} \phi(\xbf) > -\infty$.
\end{itemize}

First, we state the Clark subdifferential \cite{chen2020learnable} of $r(\xbf)$ in Lemma \ref{lem:r_subdiff}, and we show that the gradient of $\rbf_\varepsilon$ is Lipschitz continuous in Lemma \ref{r_lips}.
\begin{lemma}\label{lem:r_subdiff}
Let $r(\xbf)$ be defined in \eqref{eq:r_chp4}; then, the Clarke subdifferential of $r$ at $\xbf$ is
%\begin{adjustwidth}{-4.6cm}{0cm}
\begin{equation}\label{eq:r_subdiff}
\partial \rbf(\xbf) = \{\sum_{j\in I_0}\nabla \gbf_i(\xbf)^{\top}  \wbf_j + \sum_{j \in I_1}\nabla \gj(\xbf)^{\top}\frac{\gj(\xbf)}{\|\gj(\xbf)\|} \ \bigg\vert \ \wbf_j \in \mathbb{R}^d, \ \|\Pi(\wbf_j; \Ccal(\nabla \gbf_i(\xbf)))\|\leq 1,\ \forall\, j \in I_0 \} ,  
\end{equation}
%\end{adjustwidth}
where $I_0=\{j \in [m] \ | \ \|\gj(\xbf) \|= 0 \}$, $I_1=[m] \setminus I_0$, and $\Pi(\wbf;\Ccal(\Abf))$ is the projection of $\wbf$ onto $\Ccal(\Abf)$ which stands for the column space of $\Abf$.
\end{lemma}
\begin{lemma}\label{r_lips}
The gradient of $\rbf_\varepsilon$ is Lipschitz continuous with constant $m ( L_g +\frac{2M^2}{\varepsilon})$.
\end{lemma}

\begin{proof}
From $\rbf_\varepsilon(\xbf) = \sum^m_{j=1} (\| \gbf_j(\xbf) \|^2 +\varepsilon^2)^{\frac{1}{2}} - \varepsilon $, it follows that 
\begin{equation}
    \nabla \rbf_\varepsilon (\xbf) = \sum^m_{j=1} \nabla \gbf_j(\xbf)^\top \gbf_j(\xbf) (\| \gbf_j(\xbf)\|^2 +\varepsilon^2)^{-\frac{1}{2}}.
\end{equation}
For any $\xbf_1, \xbf_2 \in \Xcal$, we first define $h(\xbf) = \gbf_j(\xbf)(\| \gbf_j(\xbf)\|^2 +\varepsilon^2)^{-\frac{1}{2}}$, so $ \norm{h(\xbf)} <1$.

\begin{subequations}\label{eq:h1-h2}
    \begin{align}
   & \left\| h(\xbf_1) - h(\xbf_2) \right\| \\
    & = \left \| \frac{\gj(\xbf_1)}{\sqrt{\| \gj(\xbf_1)\|^2 +\varepsilon^2}} - \frac{\gj(\xbf_2)}{\sqrt{\| \gj(\xbf_2)\|^2 +\varepsilon^2}} \right \|  \\  
    & = \left \| \frac{\gj(\xbf_1)}{\sqrt{\| \gj(\xbf_1)\|^2 +\varepsilon^2}} - \frac{\gj(\xbf_1)}{\sqrt{\| \gj(\xbf_2)\|^2 +\varepsilon^2}} +
    \frac{\gj(\xbf_1)}{\sqrt{\| \gj(\xbf_2)\|^2 +\varepsilon^2}}
    - \frac{\gj(\xbf_2)}{\sqrt{\| \gj(\xbf_2)\|^2 +\varepsilon^2}}  \right \| \\
    & \leq \left \| \gj(\xbf_1)  \left(\frac{\sqrt{\| \gj(\xbf_2)\|^2 +\varepsilon^2} - \sqrt{\| \gj(\xbf_1)\|^2 +\varepsilon^2} }{\sqrt{\| \gj(\xbf_1)\|^2 +\varepsilon^2} \sqrt{\| \gj(\xbf_2)\|^2 +\varepsilon^2} } \right) \right \| + \left\| \frac{ \gj(\xbf_1) - \gj(\xbf_2)  }{\sqrt{\| \gj(\xbf_2)\|^2 +\varepsilon^2}} \right\| \\
    & \leq   \left\| \frac{ \gj(\xbf_1) }{\sqrt{\| \gj(\xbf_1)\|^2 +\varepsilon^2}} \right\|  \left \| \frac{\sqrt{\| \gj(\xbf_2)\|^2 +\varepsilon^2} - \sqrt{\| \gj(\xbf_1)\|^2 +\varepsilon^2} }{ \sqrt{\| \gj(\xbf_2)\|^2 +\varepsilon^2} }  \right\| + \frac{1}{\varepsilon} \left\| \gj(\xbf_1) - \gj(\xbf_2)  \right\| \\
    & \leq  \frac{1}{\varepsilon} \left\| \sqrt{\| \gj(\xbf_2)\|^2 +\varepsilon^2} - \sqrt{\| \gj(\xbf_1)\|^2 +\varepsilon^2} \right\| + \frac{1}{\varepsilon} \left\| \gj(\xbf_1) - \gj(\xbf_2)  \right\| \label{A3e}\\
    & \leq \frac{1}{\varepsilon} \frac{\| \gj(\xbf_2)\|^2 - \| \gj(\xbf_1)\|^2}{\sqrt{\| \gj(\xbf_2)\|^2 +\varepsilon^2} + \sqrt{\| \gj(\xbf_1)\|^2 +\varepsilon^2} } +  \frac{1}{\varepsilon} \left\| \gj(\xbf_1) - \gj(\xbf_2)  \right\| \\
    & \leq \frac{1}{\varepsilon} \underbrace{\frac{\| \gj(\xbf_2)\| + \| \gj(\xbf_1)\| }{\sqrt{\| \gj(\xbf_2)\|^2 +\varepsilon^2} + \sqrt{\| \gj(\xbf_1)\|^2 +\varepsilon^2}} }_{<1} \left( \| \gj(\xbf_2)\| - \| \gj(\xbf_1)\|\right)  +  \frac{1}{\varepsilon} \left\| \gj(\xbf_1) - \gj(\xbf_2)  \right\| \\
    & \leq \frac{1}{\varepsilon}  \left\| \gj(\xbf_2) - \gj(\xbf_1)  \right\|+ \frac{1}{\varepsilon}  \left\| \gj(\xbf_1) - \gj(\xbf_2)  \right\|\\
    & = \frac{2}{\varepsilon}  \left\| \gj(\xbf_1) - \gj(\xbf_2)  \right\|.
    \end{align}
\end{subequations}
where to obtain \eqref{A3e} we used $\left\| \frac{ \gj(\xbf_1) }{\sqrt{\| \gj(\xbf_1)\|^2 +\varepsilon^2}} \right\| < 1 \text{ and }  \frac{ 1 }{\sqrt{\| \gj(\xbf_1)\|^2 +\varepsilon^2}} < \frac{1}{\varepsilon}$.

Therefore, we have
\begin{subequations}
    \begin{align}
   & \left\|\nabla \rbf_\varepsilon (\xbf_1) - \nabla  \rbf_\varepsilon (\xbf_2) \right\| \\
   &  = \sum^m_{j=1}  \left\|\nabla \gj(\xbf_1)^\top h(\xbf_1) - \nabla \gj(\xbf_2)^\top h(\xbf_2) \right\| \\
    & = \sum^m_{j=1}  \left\|\nabla \gbf_j(\xbf_1)^\top h(\xbf_1) - \nabla \gj(\xbf_2)^\top h(\xbf_1) +  \nabla \gj(\xbf_2)^\top  h(\xbf_1) - \nabla \gj(\xbf_2)^\top h(\xbf_2) \right\|  \\
    & \leq \sum^m_{j=1}  \left\| \left( \nabla \gbf_j(\xbf_1) - \nabla \gbf_j(\xbf_2) \right)^\top h(\xbf_1)  \right\| + \left\| \nabla \gbf_j(\xbf_2) \left( h(\xbf_1) - h(\xbf_2) \right) \right\| \\
& \leq \sum^m_{j=1}  \left\| \nabla \gbf_j(\xbf_1) - \nabla \gbf_j(\xbf_2) \right\|  \norm{h(\xbf_1)}  +\norm{\nabla \gbf_j(\xbf_2)} \left\| h(\xbf_1) - h(\xbf_2) \right\| \notag \\
& \leq \sum^m_{j=1} \left\| \nabla \gbf_j(\xbf_1) - \nabla \gbf_j(\xbf_2) \right\| + \norm{\nabla \gbf_j(\xbf_2)} \frac{2}{\varepsilon} \norm{\gbf_j(\xbf_1)- \gbf_j(\xbf_2)} \text{ by } \eqref{eq:h1-h2} \text{ and }\\
& \leq m (L_g \norm{\xbf_1 - \xbf_2}+ M \frac{2}{\varepsilon} \cdot M \norm{\xbf_1 - \xbf_2}),
    \end{align}
\end{subequations}
where the first term  of the last inequality is due to the $L_g$-Lipschitz continuity of $ \nabla \gj$. The second term is because of $\norm{\gbf_j(\xbf_1)- \gbf_j(\xbf_2)}  = \norm{\nabla \gbf_j(\tilde{\xbf})(\xbf_1-\xbf_2)}  $ for some $\tilde{\xbf} \in \Xcal$ due to the mean value theorem and  $ \| \nabla \gj(\tilde{\xbf}) \| \le \sup_{\xbf \in \Xcal} \| \nabla \gbf_j(\xbf) \| \leq M$. Therefore, we obtain
\begin{equation}
    \left\|\nabla \rbf_\varepsilon (\xbf_1) - \nabla  \rbf_\varepsilon (\xbf_2) \right\| \leq m(L_g + \frac{2M^2}{\varepsilon})  \norm{\xbf_1 - \xbf_2}.
\end{equation}
\end{proof}

\begin{lemma}\label{lem:inner}
Let $\varepsilon, \eta, \taut, a>0$,  $ 0<\rho<1$, and choose the initial $ \xbf_0 \in \Xcal$. Suppose the sequence $ \{ \xt \}$ is generated by  executing Lines 3--14 of Algorithm  \ref{alg:lda} with fixed $ \epst = \varepsilon$ and that $  0<\delta<\frac{L_\varepsilon}{2/a+L_\varepsilon} <1 $ exists, where $ L_{\varepsilon} = L_f  + m (L_g + \frac{2M^2}{\varepsilon})$ and $ \phi^* := \min_{\xbf \in  \Xcal} \phi(\xbf)$. Then, the following statements hold:
\begin{enumerate}
\item $ \| \nabla \phi_\varepsilon (\xt) \| \to 0$ as $t \to \infty$.

\item $ \max \{ t\in \NN \ |  \  \norm{\nabla \phi_{\varepsilon} (\xtp) } \geq \eta \} \leq \max \{ \frac{{a L^2_\varepsilon} }{ \delta^2 \eta^2}, \frac{a^3}{\eta^2}\}  \left(\phi_\varepsilon(\xbf_0) - \phi^*+\varepsilon \right) $.
\end{enumerate}
\end{lemma}
\begin{proof}

\begin{enumerate}
\item
In each iteration, we compute $\utp = \ztp - \taut\sigma(\omega_i) \nabla \rbf_{\epst} (\ztp)$. 
\begin{enumerate}[label*=\arabic*.]
\item \label{case1} In the case the condition  
\begin{equation}\label{eq:con1}
   \| \nabla \phi_{\varepsilon} (\xt) \| \leq a \| \utp - \xt \| \ \ \ \mbox{and}  \ \ \  \phi_{\varepsilon}(\utp) - \phi_{\varepsilon}(\xt) \leq - \frac{1}{a}\| \utp - \xt \|^2 
\end{equation} holds with $a>0$,
we put $ \xtp = \utp$, and we have $\phi_{\varepsilon}(\utp) \leq \phi_{\varepsilon}(\xt)$. 
\item \label{case2} Otherwise, we compute $\vtp = \xt - \alpha_{t} \nabla \phi_{\varepsilon}(\xt)$, where $ \alpha_{t}$ is found through the line search until the criteria 
\begin{equation}\label{eq:con2}
  \phi_{\varepsilon}(\vtp) - \phi_{\varepsilon}(\xt) \le - \frac{1}{a} \| \vtp - \xt\|^2  
\end{equation} holds, and then put $ \xtp = \vtp $. From Lemma \ref{r_lips}, we know that the gradient $ \nabla \rbf_{\varepsilon} (\xbf)$ is Lipschitz continuous with constant $m (L_g +\frac{2M^2}{\varepsilon})$. Furthermore, we assumed in (a1) that $ \nabla f$ is $L_f$-Lipschitz continuous. Hence, putting $ L_\varepsilon = L_f + m( L_g +\frac{2M^2}{\varepsilon})$, we find that $\nabla \phi_\varepsilon$ is $L_\varepsilon$-Lipschitz continuous, which implies
\begin{equation}\label{eq:phi}
    \phi_\varepsilon(\vtp) \leq \phi_\varepsilon(\xt) + \langle \nabla \phi_\varepsilon(\xt) , \vtp-\xt \rangle + \frac{L_\varepsilon}{2} \| \vtp-\xt \|^2.
\end{equation}
Furthermore, by the optimality condition of $$ \vtp = \argmin_{\xbf} \langle \nabla f(\xt), \xbf - \xt \rangle +  \sigma(\omega_i) \langle \nabla \rbf_{\varepsilon}(\xt) , \xbf - \xt \rangle + \frac{1}{2 \alpha_t} \| \xbf - \xt \|^2, $$ we have 
\begin{equation}\label{eq:opt}
\langle \nabla \phi_\varepsilon (\xt), \vtp - \xt \rangle + \frac{1}{2 \alpha_t} \norm{\vtp - \xt}^2 \leq 0.
\end{equation}
Combining \eqref{eq:phi} and \eqref{eq:opt} and $ \vtp = \xt -\alpha_t \nabla \phi_\varepsilon (\xt)$ in line 8 of Algorithm \ref{alg:lda} yields 
\begin{equation}
    \phi_\varepsilon(\vtp) - \phi_\varepsilon(\xt) \leq -\left(\frac{1}{2\alpha_t} - \frac{L_\varepsilon}{2} \right) \norm{\vtp - \xt}^2.  %= -\frac{\alpha_t(1 - \alpha_t L_\varepsilon)}{2} \norm{\nabla \phi_\varepsilon (\xt)}^2 \leq 0,
\end{equation}
Therefore, it is sufficient for $ \alpha_t \leq \frac{1}{ 2/a + L_\varepsilon}$ for the criteria \eqref{eq:con2}
 to be satisfied. This process only take finitely many iterations since we can find a finite $t$ such that $ \rho^t \alpha_t \leq \frac{1}{ 2/a + L_\varepsilon}$, and through the line search, we can obtain $ \phi_\varepsilon(\vtp) \leq \phi_\varepsilon (\xt)$.
\end{enumerate}
Therefore, in either case of \ref{case1} or \ref{case2} where we take $\xtp = \utp$ or $\vtp$,  we can obtain
\begin{equation}\label{eq:dec}
    \phi_\varepsilon (\xtp) \leq \phi_\varepsilon (\xt) , \text{ for all } t\geq 0.
\end{equation}
Now, from case \ref{case1}, \eqref{eq:con1} gives 
\begin{subequations}
\begin{align}
\norm{\nabla \phi_\varepsilon(\xt)}^2 \leq a^2\left\| \utp - \xt \right\|^2 \leq &  a^3 \left(\phi_\varepsilon(\xt) - \phi_\varepsilon(\utp) \right), \\
\text{therefore if $\xtp = \utp $ we get } & \norm{\nabla \phi_\varepsilon(\xt)}^2 \leq a^3 \left(\phi_\varepsilon(\xt) - \phi_\varepsilon(\xtp) \right).   \label{eq:c1}    
\end{align}
\end{subequations}
From  case \ref{case2} and $\vtp = \xt - \alpha_{t} \nabla \phi_{\varepsilon}(\xt)$, we have  
\begin{subequations}
    \begin{align}
     \phi_\varepsilon(\vtp) - \phi_\varepsilon(\xt) \leq  & - \frac{1}{a} \norm{\vtp - \xt}^2 = -\frac{1}{a} \alpha_t^2 \norm{\nabla \phi_\varepsilon(\xt)}^2\\
    \Longrightarrow 
    &  \norm{\nabla \phi_\varepsilon(\xt)}^2 \leq \frac{a}{\alpha_t^2} \big(\phi_\varepsilon(\xt) - \phi_\varepsilon(\vtp) \big),\\
     \text{ then if $\xtp =\vtp$, we have }
    &  \norm{\nabla \phi_\varepsilon(\xt)}^2 \leq \frac{a}{ \alpha_t^2} \big(\phi_\varepsilon(\xt) - \phi_\varepsilon(\xtp) \big). \label{eq:c2}
    \end{align}
\end{subequations}
Since $ \frac{\delta}{L_\varepsilon}\leq \alpha_t \leq \frac{1}{ 2/a + L_\varepsilon}$, we have
\begin{equation}\label{eq:grad_phi_bound}
     \norm{\nabla \phi_\varepsilon(\vtp)}^2  \leq \frac{{a L^2_\varepsilon} }{ \delta^2}\big(\phi_\varepsilon(\vtp) - \phi_\varepsilon(\xtp) \big).
\end{equation}

Combining \eqref{eq:c1} and \eqref{eq:grad_phi_bound} and selecting $C = \max \{ \frac{{a L^2_\varepsilon} }{ \delta^2}, a^3\} $, we obtain
\begin{equation}\label{eq:grad_phi_bound_2}
     \norm{\nabla \phi_\varepsilon(\xt)}^2  \le  C \big(\phi_\varepsilon(\xt) - \phi_\varepsilon(\xtp) \big).
\end{equation}

Summing up \eqref{eq:grad_phi_bound_2} for $t=0,\cdots, T$, we have
\begin{equation}\label{eq:sum}
    \sum^{T}_{t=0} \norm{\nabla \phi_\varepsilon(\xt)}^2  \leq C \big(\phi_\varepsilon(\xbf_0) - \phi_\varepsilon(\xbf_{T+1}) \big).
\end{equation}
Combined with the fact that $ \phi_\varepsilon(\xbf) \geq \phi(\xbf) - \varepsilon \geq \phi^* -\varepsilon$ for every $ \xbf\in \Xcal$, we have 
\begin{equation}\label{eq:sum_}
    \sum^T_{t=0} \norm{\nabla \phi_\varepsilon(\xt)}^2 \leq C  \left(\phi_\varepsilon(\xbf_0) - \phi^*+\varepsilon \right). 
\end{equation}
The right-hand side is a finite constant, and hence by letting $t\to \infty$, we know that $\norm{\nabla \phi_\varepsilon(\xt)} \to 0$, which proves the first statement.

\item Denote $ \kappa :=  \max \{ t\in \NN \ |  \  \norm{\nabla \phi_{\varepsilon} (\xtp) } \geq \eta \}$; then, we know that $ \norm{\nabla \phi_{\varepsilon} (\xtp) } \geq \eta $ for all $ t\leq \kappa-1$. Hence, we have 
\begin{equation}
    \kappa \eta^2 \leq \sum^{\kappa-1}_{t=0} \norm{ \nabla \phi_{\varepsilon} (\xtp) }^2 = \sum^{\kappa}_{t=1}  \norm{ \nabla \phi_{\varepsilon} (\xt) }^2\leq C \left(\phi_\varepsilon(\xbf_0) - \phi^*+\varepsilon \right). 
\end{equation}
which implies the second statement.
\end{enumerate}
\end{proof}

\begin{lemma}\label{lem:phi_decay}
Suppose that the sequence $\{ \xt\}$ is generated by Algorithm 1 with an initial guess $\xbf_0$. Then, for any $t\geq 0$, we have 
$ \phi_{\epstp}(\xtp) + \epstp \leq \phi_{\epst}(\xt) + \epst$.
\end{lemma}
\begin{proof}
To prove this statement, we can prove
\begin{equation}\label{eq:ineq}
\phi_{\epstp}(\xtp) + \epstp \leq  \phi_{\epst}(\xtp) + \epst \leq \phi_{\epst}(\xt) + \epst.
\end{equation}
The second inequality is immediately obtained from \eqref{eq:dec}. Now, we prove the first inequality. 

For any $\varepsilon>0$, denote
\begin{equation}\label{eq:r_}
\rbf_{\varepsilon, j} (\xbf) = \sqrt{\| \gbf_j (\xbf) \|^2_{2} + \varepsilon^2} -\varepsilon.
\end{equation}

Since $ \phi_\varepsilon(\xbf)=f(\xbf)+\sigma(\omega_i)\sum^m_{j=1} \rbf_{\varepsilon, j}(\xbf)$, it suffices to show that 
\begin{equation}
    \rbf_{\epstp, j}(\xtp) + \epstp \leq  \rbf_{\epst, j}(\xtp) + \epst 
\end{equation}
If $ \epstp =\epst$, then the two quantities above are identical and the first inequality holds. Now, suppose $ \epstp = \gamma\epst \le \epst$; then, 
\begin{equation}
    \rbf_{\epstp, j}(\xtp) + \epstp =  \sqrt{\| \gbf_j (\xbf) \|^2_{2} + \epstp^2} \leq   \sqrt{\| \gbf_j (\xbf) \|^2_{2} + \epst^2} = \rbf_{\epst, j}(\xtp) + \epst,
\end{equation}
which implies the first inequality of \eqref{eq:ineq}.
\end{proof}
\begin{theorem}
Suppose that $\{\xt \}$ is the sequence generated by Algorithm \ref{alg:lda} with any initial $\xbf_0$, $\etol=0$ and $T=\infty$. Let $ \{ \xbf_{t_l+1}\}$ be the subsequence that satisfies the reduction criterion  in step 15 of Algorithm \ref{alg:lda}, i.e., $  \norm{\nabla \phi_{\varepsilon_{t_l}} (\xbf_{t_l+1})} \leq \sigma  \varepsilon_{t_l} \gamma $ for $t=t_l$ and $ l=1,2,\cdots$. Then $ \{ \xbf_{t_l+1}\}$ has at least one accumulation point, and every accumulation point of $\{ \xbf_{t_l+1}\}$ is a clarke stationary point of $ \min_{\xbf} \phi(\xbf) := f(\xbf) +\sigma(\omega_i) r(\xbf)$.
    \label{theorem a6}
\end{theorem}
\begin{proof}
By Lemma \ref{lem:phi_decay} and $ \phi(\xbf) \leq \phi_\varepsilon(\xbf) +\varepsilon$ for all $\varepsilon>0$ and $ \xbf\in \Xcal$, we know that
\begin{equation}
    \phi(\xt) \leq \phi_{\epst}(\xt) +\epst\leq \cdots \leq \phi_{\varepsilon_0}(\xbf_0) +\varepsilon_0 <\infty.
\end{equation}
Since $\phi$ is coercive, we know that $\{ \xt\}$ is bounded, and the selected subsequence $ \{ \xbf_{t_l+1} \}$ is also bounded and has at least one accumulation point.

Note that $ \norm{\nabla \phi_{\varepsilon_{t_l}} (\xbf_{t_l+1})} \leq \sigma  \varepsilon_{t_l} \gamma = \sigma \varepsilon_{0} \gamma^{l+1} \to 0$ as $l\to \infty$.  Let $ \{ \xpp\} $ be any convergent subsequence of $\{ \xbf_{t_l+1} \}$ and denote $\epsp$ as the corresponding $\epst$ used in \mbox{Algorithm~\ref{alg:lda}} that generates $\xpp$. Then, there exists $ \xbf^* \in \Xcal$ such that $ \xpp \to  \xbf^*$ as $ \varepsilon_p \to 0,$ and $ \nabla \phi_{\epsp}(\xpp) \to 0$ as $ p\to \infty$.

\textls[-15]{Note that the Clarke subdifferential of $\phi$ at $\xbf^*$ is given by $\partial \phi(\xbf^*) =  \partial f(\xbf^*) + \sigma(\omega_i) \partial \rbf(\xbf^*) $:}

\begin{multline} \label{eq:d_phi_xhat}
\partial \phi(\xhat) = \{\nabla f(\xhat) + \sigma(\omega_i) \sum_{j \in I_0} \nabla \gj(\xhat)^{\top} \wbf_j + \sigma(\omega_i) \sum_{ j \in I_1} \nabla \gj(\xhat)^{\top} \frac{\gj(\xhat)}{\| \gj(\xhat) \|} \ \bigg\vert \\
 \ \norm{\Pi(\wbf_j; \Ccal(\nabla \gj(\xhat)))}  \le 1,\ \forall\, j\in I_0\},
\end{multline}
where $I_0 = \{j\in[m]\ \vert \ \|\gi(\xhat)\| = 0 \}$ and $I_1 = [m] \setminus I_0$.

If $j\in I_0$, we have $ \norm{\gj(\xbf)} =0 \iff \gj(\xbf) =0$: then,  
\begin{subequations}
    \begin{align}
    \partial \rbf_{\varepsilon}(\xbf) & = \sum_{j\in I_0} \nabla \gbf_j(\xbf)^{\top} \frac{\gbf_j(\xbf) }{\Big(\norm{\gbf_j(\xbf)}^2+\varepsilon^2 \Big)^{\frac{1}{2}}}  +  \sum_{j\in I_1} \nabla \gbf_j(\xbf)^{\top} \frac{\gbf_j(\xbf) }{\Big(\norm{\gbf_j(\xbf)}^2+\varepsilon^2 \Big)^{\frac{1}{2}}}  \\
    & = \mathbf{0} + \sum_{j\in I_1} \nabla \gbf_j(\xbf)^{\top} \frac{\gbf_j(\xbf) }{\Big(\norm{\gbf_j(\xbf)}^2+\varepsilon^2 \Big)^{\frac{1}{2}}} 
    \end{align}
\end{subequations}

Therefore, we obtain
\begin{equation}\label{eq:d_phi_epsj}
\nabla \phi_{\epsp}(\xpp)  =   \nabla f(\xpp) +  \sigma(\omega_i) \sum_{j\in I_1} \nabla \gbf_j(\xbf)^{\top} \frac{\gbf_j(\xbf) }{\Big(\norm{\gbf_j(\xbf)}^2+\varepsilon_p^2 \Big)^{\frac{1}{2}}} . 
\end{equation}
Comparing \eqref{eq:d_phi_xhat} and \eqref{eq:d_phi_epsj}, we can see that the first term on the right-hand side of \eqref{eq:d_phi_epsj} converges to that of \eqref{eq:d_phi_xhat}, due to the fact that $\xpp \to \xhat$ and the continuity of $\nabla f$. Together with the continuity of $\gi$  and
$\nabla \gi$, the last term of \eqref{eq:d_phi_xhat} converges to the last term of \eqref{eq:d_phi_epsj} as $\varepsilon_p \rightarrow 0$ and $\norm{\gbf_j(\xbf)} > 0$. Furthermore, apparently $\mathbf{0}$ is a special case of the second term in \eqref{eq:d_phi_epsj}.
Hence, we know that
\[
\mathrm{dist}( \nabla \phi_{\epsp}(\xpp), \partial \phi(\xhat)) \to 0,
\]
as $p \to \infty$. Since $\nabla \phi_{\epsp}(\xpp) \to 0$ and $\partial \phi(\xhat)$ is closed, we conclude that $0 \in \partial \phi(\xhat)$.
\end{proof}

%% file: tex/chapter5.tex
\chapter{An Optimization-Based Model for Joint Multimodal MRI Reconstruction and Synthesis}\label{JointRecSyn}

Generating multi-contrasts/modal MRI of the same anatomy enriches diagnostic information but is limited in practice due to excessive data acquisition time. In this paper, we propose a novel deep-learning model for joint reconstruction and synthesis of multi-modal MRI using incomplete k-space data of several source modalities as inputs. The output of our model includes reconstructed images of the source modalities and high-quality image synthesized in the target modality.
Our proposed model is formulated as a variational problem that leverages several learnable modality-specific feature extractors and a multimodal synthesis module. We propose a learnable optimization algorithm to solve this model, which induces a multi-phase network whose parameters can be trained using multi-modal MRI data. Moreover, a bilevel-optimization framework is employed for robust parameter training. We demonstrate the effectiveness of our approach using extensive numerical experiments.

\section{Introduction}
Magnetic resonance imaging (MRI) is a prominent leading medical imaging technology, which provides diverse image contrasts under the same anatomy and enriches the diagnostic information. Multimodal MR images can provide more diagnostic information for clinical application and research studies comparing to single modality \cite{HiNet,iglesias2013synthesizing,cordier2016extended,van2015does,huo2018synseg}. Multiple different contrasts images are generated by varying the acquisition parameters: repetition time (TR) and echo time (TE), they have similar anatomical structure but highlight different soft tissue. For example,
acquiring T1-weighted images uses short TR and TE times, high fat content appear bright and compartments filled with  celebrospinal fluid (CSF) appear dark and they provide more anatomical information. T1 brain images distinguish the gray and white matter tissue. T2-weighted images require longer TR and TE times. In general, T2-weighted images appear to be a reversal of T1-weighted images in contrast and they often provide more pathological information for delineation of edema. T2 images distinguish fluid from cortical tissue. Fluid Attenuated Inverseion Recovery (FLAIR) need very long TR and TE times, their contrast appear similar as T2-weighted images with supressed CSF so that the free water becomes dark and minimizes contrast between gray matter and white matter. 
 A major limitation of MRI is the relatively long data acquisition time during scanning,  which will arise patients discomfort or introduce motion artifacts and degrade diagnostic accessibility. One of the predominate method to reduce the scanning time is to reconstruct undersampled k-space acquisitions, another method is to synthesize missing contrast MR image from fully-sampled acquisitions \cite{yang2020model,dar2020prior}. In this paper, we refer different contrast of MR images as different modalities. We propose to jointly reconstruct undersampled k-space MR data from multiple available  modalities (i.e, source modalities) and synthesis the missing modality (i.e, target modality) image  from the source modalities.

Compressed sensing MRI (CS-MRI) reconstruction is a predominant approach for accelerating MR acquisitions, which solves an inverse problem formulated in a variational model. Traditional CS-MRI incorporate the hand-crafted regularization term (eg. Total Variation) to introduce prior information to the image to be reconstructed  \cite{block2007undersampled,huang2014fast,lustig2007sparse,eksioglu2016decoupled,qu2014magnetic,yang2010fast}. In recent decades, deep learning based model leverages large dataset and further explore the potential improvement of reconstruction performance comparing to traditional methods and has successful applications in clinic field. Most of the deep learning based reconstruction methods employ end-to-end deep networks such as GAN-based methods: DAGAN \cite{yang2017dagan}, RefineGAN \cite{quan2018compressed}; Cascade network methods: E2E-VN \cite{sriram2020end}, Cascade-Net \cite{schlemper2017deep}; The methods incorporate deep residual learning \cite{lee2018deep,lee2017deep,dai2019compressed}, etc. To overcome the weakness of the black-box model in end-to-end networks, several methods inspired by learnable optimization
algorithms (LOAs) developed, which possess of a more interpretable network architecture. LOA-based reconstruction methods unroll the iterative optimization algorithm into a multi-phase network where the regularization parameters and image transformations are learned effectively by optimization algorithm to improve network performance \cite{monga2021algorithm,liang2020deep,jimaging7110231,hammernik2018learning,sun2016deep,zhang2018ista,zhang2020deep,hosseini2020dense}. For example: ADMM-Net \cite{sun2016deep} is proposed by unrolling the Alternating Direction Method of Multipliers (ADMM) algorithm. ISTA-Net$^+$ \cite{zhang2018ista} embedded deep neural networks into the iterative
shrinkage-thresholding algorithm (ISTA) for solving the CS-MRI reconstruction.  PD-Net \cite{cheng2019model} unfolds primal dual hybrid gradient algorithm where the proximal operators are parametrized as two network blocks that learned from training data.

MR image synthesis has recently been gaining popularity using various deep learning frameworks, which can be roughly categorized into unimodal synthesis \cite{dar2019image,sohail2019unpaired,welander2018generative,yang2018mri,yu20183d} and multimodal synthesis \cite{chartsias2017multimodal,dar2020prior,sharma2019missing,zhou2020hi,yurt2021mustgan}. Unimodal synthesis is an one-to-one approach that aims to estimate the image of a target modality from the corresponding single source modality. A common deep learning based strategy is adversarial method such as pGAN \cite{dar2019image}, which minimizes adversarial loss function to capture reliable high-frequency texture information. For paired data translation, patch-based \cite{jog2013magnetic,torrado2016fast} and atlas-based  methods \cite{miller1993mathematical,roy2013magnetic,burgos2014attenuation} achieved great success. For instance, \cite{roy2016patch} proposed a patch matching method that first register multiple atlases with T1 and T2 to the target T1 and then combine atlas T2 patches to synthesis T2 based on atlas patchs of target T1. \cite{jog2015mr} estimate the subject pulse sequence parameters and synthesize an additional atlas image to learn the transformation of intensity between additional atlas image and target atlas image by learning random forest regression on image patches, and finally apply the regression on given subject image to obtain the target modality synthetic image.   For unpaired data translation, image-to-image translation in recent years attracts attention in multimodal medical imaging,  which including GAN-based models \cite{chen2019one,welander2018generative} where the generator is designated as deterministic mapping and flow-based methods \cite{bui2020flow} where the generator is an invertible mapping that allows the cycle consistant translation by mapping data from source domain to target domain and guarantee the data point return to the source domain.  Multimodal synthesis is many-to-one or many-to-many type of approach. Most of Multimodal synthesis uses GAN-based network. A main stream of the aforementioned synthesis methods are end-to-end networks that use encoder-decoder architectures, specifically contained in the generator network of adversarial learning. For example, MM-GAN \cite{sharma2019missing}  concatenates all the available modalities channel-wisely with a zero image as missing modality and imputes the missing input incorporating curriculum learning for GAN. Multimodal MR (MM) \cite{liu2020multimodal}, MMGradAdv \cite{chartsias2017multimodal} and Hi-Net \cite{zhou2020hi} exploit the correlations among multimodal data and using robust feature fusion method to form a unified latent representation. 

In order to synthesis target modality by using partially scanned k-space data from source modalities in stead of fully scanned data that used in the state-of-the-art multimodal synthesis.  In this paper, we propose a LOA-based jointly MRI reconstruction and synthesis deep neural network. The inputs of the network are multiple partially sampled source modalities and the outputs consist of the reconstruction target modalities and the synthesized target modality. Our contributions are summarised as follows:\\ (1) We propose a novel LOA for joint multimodal MRI reconstruction and synthesis with theoretical analysis guarantee;\\
(2) The network parameters are trained by a bilevel optimization algorithm to mitigate the overfitting problem;\\
(3) Extensive experimental results demonstrate the effectiveness and superiority of the proposed method.

\section{Proposed Method}
\subsection{Model}
In this section, we provide the details of the proposed algorithm and the corresponding network for joint MRI reconstruction and synthesis. Given the partial k-space data $\{ \fbf_1, \fbf_2 \}$ of the source modalities (eg: T1 and T2), our mission is to reconstruct the corresponding images $\{\xbf_1, \xbf_2\}$ as well as synthesizing the $\xbf_3$ (eg: FLAIR) without providing any k-space information. Our model jointly reconstructs and synthesizes the modalities by solving the following optimization problem
%%
%Our model takes the down-sampled k-space data $\{ f_i \} _{i = 1}^{N}$ as input, and reconstructs the corresponding images $\{ x_i \} _{i = 1}^{N}$ as well as synthesizing the modalities $\{ x_j \} _{j = N+1}^{M}$ without any k-space data. We perform the joint reconstruction and synthesis by minimizing the following function

\begin{subequations}
\label{our_model}
\begin{align}
\min_{\xbf_1, \xbf_2, \xbf_3}\Psi_{\Theta, \gamma}(\xbf_1, \xbf_2, \xbf_3) & :=  \sum_{i = 1}^{2} \frac{1}{2} \| P_i F \xbf_i  - \fbf_i\|_2^2 + \sum_{i = 1}^{3}  \|h_{w_i} (\xbf_i) \|_{2,1} +   \\
& \frac{\gamma}{2} \|g_{\theta} ([h_{w_1}(\xbf_1),h_{ w_2} (\xbf_2)]) - \xbf_3\|_2^2,
\end{align}
\end{subequations}
where the first term is the data fidelity of source modalities that ensures consistency between the reconstructed images $\{\xbf_1, \xbf_2\}$ and the sensed partial k-space data $\{ \fbf_1, \fbf_2 \}$. And $h_{w_i}$ represents the modality-specific feature extraction operator which maps the input $\xbf_i \in \C^n$ to a high-dimensional feature tensor $h_{w_i} (\xbf_i) \in \C^{m \times d}$, where $m$ is the spatial dimension and $d$ is the channel number of the feature map. The second term is the prior term of all modalities $\{\xbf_1, \xbf_2, \xbf_3\}$ to enhance sparsity under learned transforms which defined as
\begin{equation}\label{eq:r_chp5}
\|h_{w_i}(\xbf_i)\|_{2,1} = \sum_{j = 1}^{m} \|h_{w_i, j}(\xbf_i)\|,
\end{equation}
where $h_{w_i}(*) = (h_{w_i,1}(*),\dots, h_{w_i,m}(*))$, and each $h_{w_i,j}(*)$ is parametrized as a convolutional neural network (CNN) for $j = 1,\cdots, m$, and $w_i$ is the associated learnable network parameter. Here each $h_{w_i, j}\in \mathbb{R}^d$ can be viewed as a feature vector at spatial position $j$.

In the third term of \eqref{our_model}, $[\cdot,\cdot]$ represents the concatenation of the arguments and $g_{\theta} : \C^{m \times 2d} \rightarrow \C^{n}$ is the multi-modal synthesis module which takes the concatenated features of $ \xbf_1 $ and $\xbf_2$ as input and synthesizes the missing modality $\xbf_3$. To achieve this, we add the third term which tries to minimize the discrepancy between the synthesized result $g_{\theta} ([h_{w_1}(\xbf_1),h_{ w_2} (\xbf_2)])$ and the target modality $\xbf_3$. 

Moreover, $\Theta$ in \eqref{our_model} collects the parameters in the convolution layers of the function $\Psi$, i.e. $\Theta = \{w_1, w_2, w_3, \theta\} $. The coefficient $\lambda$ is parameterized to be a hyper-parameter which is tuned by minimizing the reconstruction loss on validation set. The weight $\gamma$ is a hyper-parameter plays a critical role in balancing the reconstruction part (first two terms in \eqref{our_model}) and the image synthesis part (last term in \eqref{our_model}) of the model \eqref{our_model}, and hence has significant impact to the final image reconstruction and synthesis quality. To address this important issue, we propose to use a bi-level hyper-parameter tuning framework to learn $\gamma$ by minimizing the reconstruction loss on both validation and training data sets. Details of this hyper-parameter training will be provided in Section \ref{section: bilevel}.
In the following subsections, we will present the detailed structure of the mentioned modules.

\subsubsection{Modality-specific feature extraction operator}

We design the modality-specific feature extraction operator $h_{w_i}$ to be a vanilla $l$-layer CNN with nonlinear activation function $\sigma$ but no bias, as follows: 
\begin{equation}\label{eq:h}
  h_{w_i}(\xbf) = \cbf_{w_i, l} * \sigma \cdots \ \sigma ( \cbf_{w_i, 3} * \sigma ( \cbf _{w_i, 2} * \sigma ( \cbf _{w_i, 1} * \xbf ))),
\end{equation}
where $\{\cbf _{w_i, q} \}_{q = 1}^{l}$ denote the convolution weights consisting of $d$ kernels with identical spatial kernel size ($\kappa \times \kappa$), and $*$ denotes the convolution operation. 
Here, we use the smoothed rectified linear unit \cite{chen2021learnable} as activation function $\sigma$.
Besides the smooth $\sigma$, each convolution operation of $h_{w_i}$ in \eqref{eq:h} is linear operator, which enable $h_{w_i}$ to be differentiable, and $\nabla h_{w_i}$ can be easily obtained by Chain Rule where each $\cbf_{w_i , q}^{\top}$ can be implemented as transposed convolutional operation \cite{dumoulin2016guide}.

\subsubsection{Multi-modal synthesis module}
The multi-modal synthesis module $g_{\theta} : \C^{m \times 2d} \rightarrow \C^{n}$ is the multi-modal synthesis module which takes the concatenated features $[h_{w_1}(\xbf_1),h_{ w_2} (\xbf_2)]$ of $\{\xbf_1, \xbf_2\}$ as input and synthesizes the missing modality $\xbf_3$. Similar to $h_{w_i}$, $g_{\theta}$ is also parameterized as a stack of convolution operators separated by $\sigma$:
\begin{equation}\label{eq:g_chp5}
  g_{\theta}(\xbf) = \cbf_{\theta, l} * \sigma \cdots \ \sigma ( \cbf_{\theta, 3} * \sigma ( \cbf _{\theta, 2} * \sigma ( \cbf _{\theta, 1} * \xbf ))).
\end{equation}
\begin{figure}[htb]
\centering 
\includegraphics[width=1.0\textwidth]{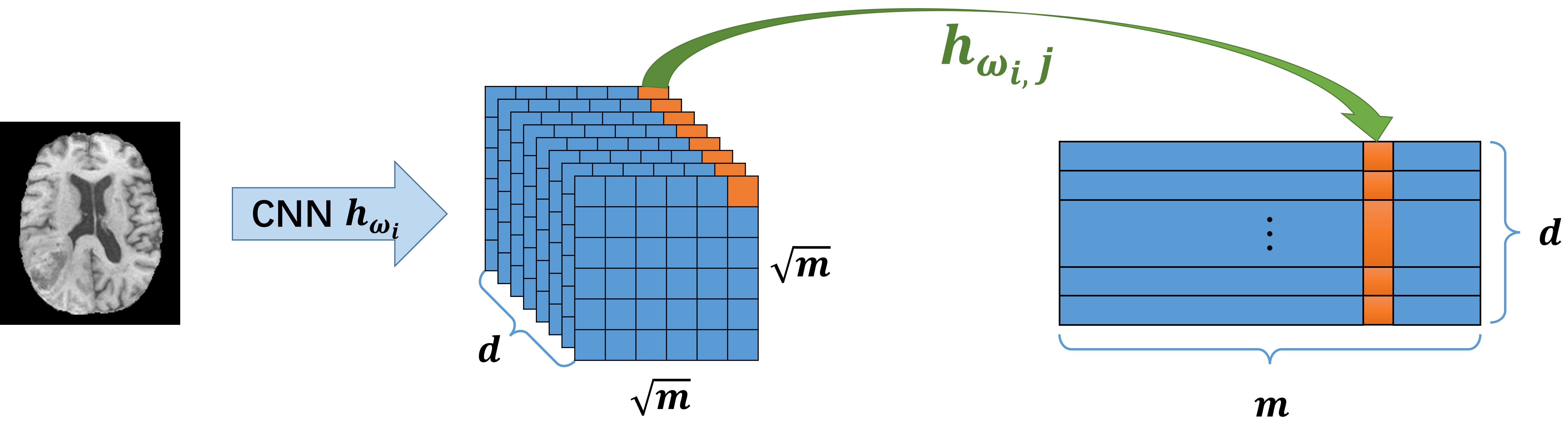}
\caption{The feature tensor $h_{w_i}$ and the feature vector $h_{w_i, j}$ at spatial position $j$.}
\vspace{10 pt}
\includegraphics[width=1.0\textwidth]{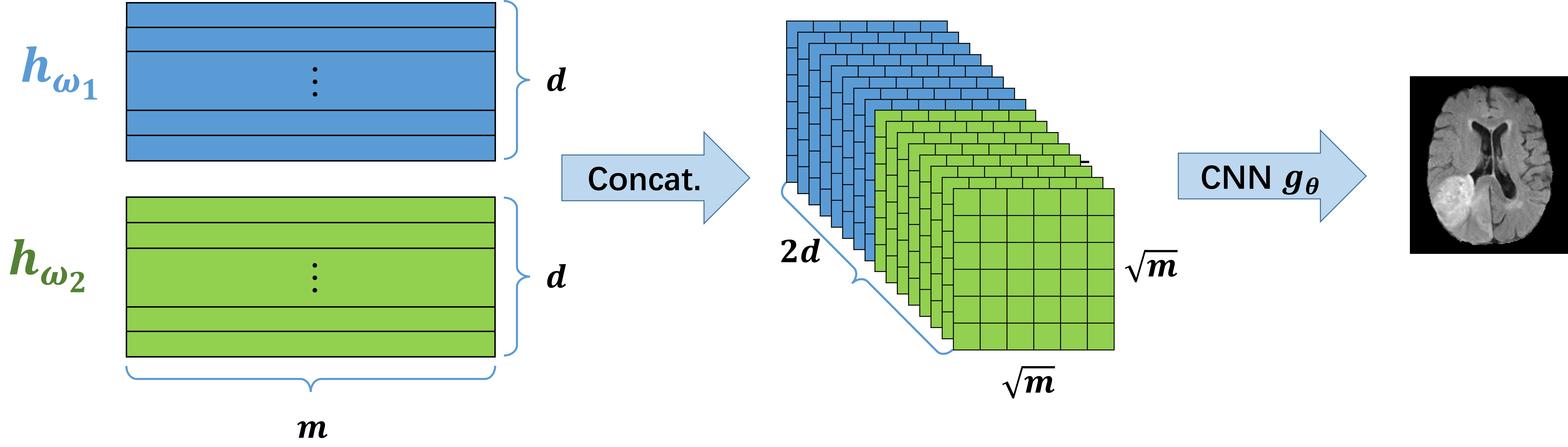}
\caption{The synthesis mapping $g_{\theta}$ maps the concatenated features $[h_{w_1}, h_{w_2}]$ to the image $\xbf_3$.}
\label{fig:operators}
\end{figure}

Figure \ref{fig:operators} indicates the illustration of the operator $h_{w_i}$ and $g_{\theta}$.

\subsection{Efficient Learnable Optimization Algorithm}\label{loa induced net}

In this section, we present a novel and efficient learnable optimization algorithm (LOA) for solving the nonconvex nonsmooth minimization problem \eqref{our_model}. Then we design a DNN whose architecture exactly follows this algorithm, and the parameters of the DNN can be learned from data. In this way, the DNN inherits all the convergence properties of the LOA.

\begin{figure}[htb]
\centering 
\includegraphics[width=1.0\textwidth]{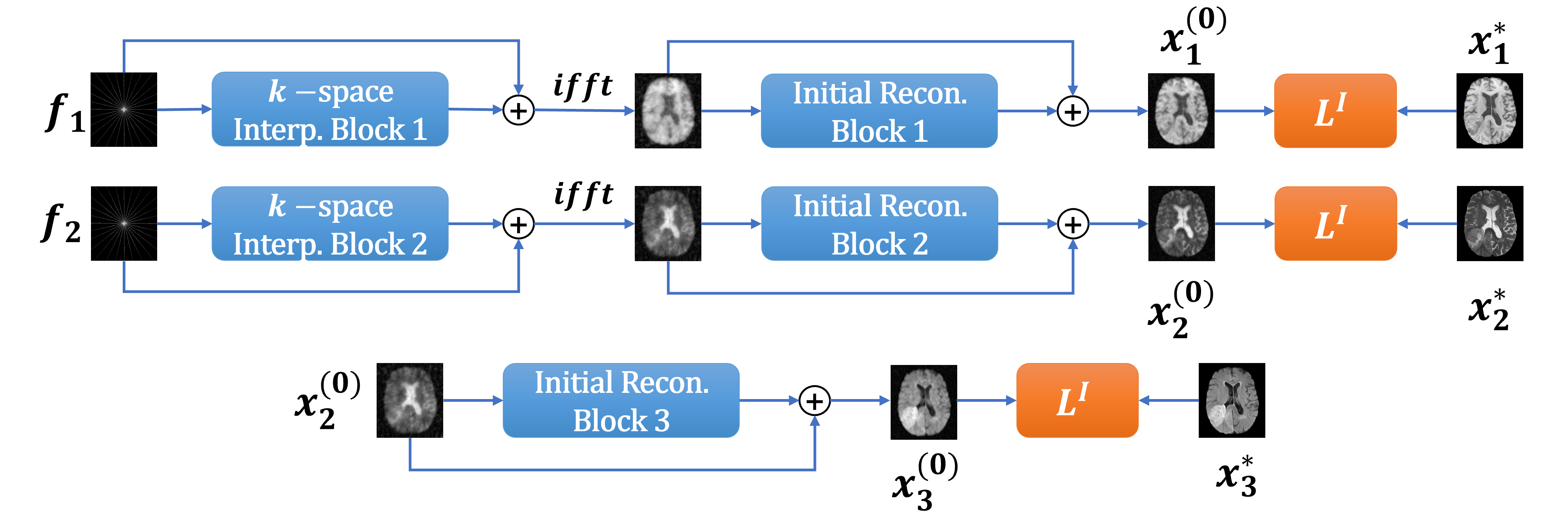}
\includegraphics[width=1.0\textwidth]{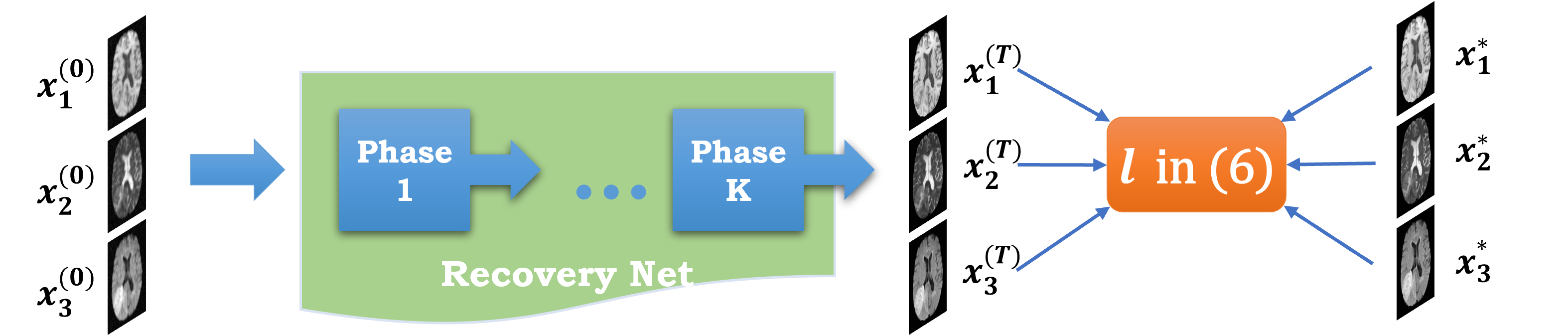}
\caption{Architecture of the proposed network. We denote $L^I$ as identity loss function that defined in equation \eqref{eq:loss_identity}. The loss function (denote as $\ell$) in the bottom orange box is defined in \eqref{loss_chp5}. }
\label{fig:flowchart}
\end{figure}

As the $l_{2, 1}$-norm in \eqref{eq:r_chp5} is nonsmooth, in this paper we consider a smoothed version of $ l_{2,1}$-norm, dubbed $ l_{\varepsilon_{2,1}}$-norm, which is defined as
$
\|h_{w_i}(\xbf_i)\|_{\varepsilon_{2,1}} = \sum\nolimits^m_{j=1}  \sqrt{\| h_{w_i, j}(\xbf_i) \|^2 + \varepsilon^2} -\varepsilon.
$
%Here we take the smoothing parameter $\varepsilon$ as $10 ^ {-6}$.
%
With this modification, the new function $\Psi_{\varepsilon}$ is achieved to be smooth, which provides us an alternative way to use gradient descent algorithm to solve the minimization problem \eqref{our_model}. As $\varepsilon$ tends to 0, the surrogates approach the original nonsmooth regularizers in \eqref{our_model}.
A key feature of our proposed algorithm is to incorporate a smoothing level reduction mechanism that automatically decreases $\varepsilon$ during the iterations, as we will explain below. 

Nevertheless, for every fixed $\epsilon>0$, we obtain a smooth surrogate function $\Psi_{\Theta,\gamma}^{\varepsilon}$, for which we can apply a gradient descent step to update our approximation to the solution of \eqref{our_model}. To see this, let $\Xbf = \{\xbf_1, \xbf_2, \xbf_3 \}$ for notation simplicity, then we solve the $\min_{\Xbf}\Psi_{\Theta, \gamma}(\Xbf)$ with initial $\Xbf^{(0)}$ by Algorithm \ref{alg:lda_chp5}. The initial $\Xbf^{(0)}$ is obtained from a pre-trained initial network, which is illustrated in detail in Section \ref{Initialization}.  At Line 3-8 of Algorithm \ref{alg:lda_chp5}, we compute a gradient descent update with step size obtained by line search while the smooth parameter $\epst > 0$ is fixed. In Line 9, the algorithm updates the smoothing parameter $\epst$ based on a reduction criterion. The reduction of $\epst$ ensures that the specified subsequence (the subsequence who met the $\epst$ reduction criterion) must have an accumulation point that is a Clarke stationary point \cite{chen2021learnable} of the optimization problem \eqref{our_model}, as given in Theorem \ref{theorem}. %The proof of this theorem is very similar to those given in \cite{chen2021learnable,jimaging7110231}, and hence is omitted due to space limitation.

\begin{theorem}
Suppose that $\{\xt \}$ is the sequence generated by Algorithm \ref{alg:lda_chp5} with any initial $\Xbf^{(0)}$, $\etol=0$ and $T=\infty$. Let $ \{ \Xbf^{(t_l+1)}\}$ be the subsequence that satisfies the reduction criterion  in step 9 of Algorithm \ref{alg:lda_chp5}. Then $ \{ \Xbf^{(t_l+1)}\}$ has at least one accumulation point, and every accumulation point of $\{ \Xbf^{(t_l+1)}\}$ is a clarke stationary point of $ \min_{\Xbf} \Psi(\Xbf)$.
    \label{theorem}
\end{theorem}

We assume that $\Psi_{\Theta, \gamma}$ is coercive and $\Psi_{\Theta, \gamma}^* = \min_{\Xbf} \Psi_{\Theta, \gamma}(\Xbf) > -\infty$. For any set $\Scal \subset \mathbb{R}^n$, we denote $\mathrm{dist}(\ybf, \Scal) := \inf\{ \|\ybf - \xbf\| \ \vert \ \xbf \in \Scal \}$.
\begin{lemma}\label{r_lips_chp5}
The gradient of $\Psi_{\Theta, \gamma}^{\varepsilon}(\Xbf)$ is Lipschitz continuous.
\end{lemma}
\begin{proof}
The $\Psi_{\Theta, \gamma}^{\varepsilon}(\Xbf)$ is the smoothing surrogate of the $\Psi_{\Theta, \gamma}(\Xbf)$ in (1) with $\|h_{w_i}(\xbf_i)\|_{2,1}$ replaced by $\|h_{w_i}(\xbf_i)\|_{\varepsilon_{2,1}}$ in the second sum. As both $h_{w_i}$ and $g_{\theta}$ are parameterized as vanilla CNNs with the smoothed activation, it is easy to verify the first and last term of $\Psi_{\Theta, \gamma}^{\varepsilon}(\Xbf)$ are Lipschitz continuous. The second sum is Lipschitz continuous proved by the Lemma A2 in \cite{jimaging7110231}.
\end{proof}
%\begin{lemma}\label{lem:bound}
%$\|h_{w_i}(\xbf_i)\|_{\varepsilon_{2,1}}  \leq\|h_{w_i}(\xbf_i)\|_{2,1} 
%\le  \|h_{w_i}(\xbf_i)\|_{\varepsilon_{2,1}}  + m\varepsilon.
%$
%\end{lemma}
%\begin{proof} $\|h_{w_i}(\xbf_i)\|_{\varepsilon_{2,1}} = \sum\nolimits^m_{j=1}  \big(\sqrt{\| h_{w_i, j}(\xbf_i) \|^2 + \varepsilon^2}   - \varepsilon \big) = \sum\nolimits^m_{j=1}  \frac{\|h_{w_i, j}(\xbf_i) \|^2}{\sqrt{\| h_{w_i, j}(\xbf_i) \|^2 + \varepsilon^2}  + \varepsilon }
%
%\le \sum\nolimits^m_{j=1}  \sqrt{\| h_{w_i, j}(\xbf_i) \|^2} = \|h_{w_i}(\xbf_i)\|_{2,1}
%\le \sum\nolimits^m_{j=1}  \sqrt{\| h_{w_i, j}(\xbf_i) \|^2 + \varepsilon^2} = \|h_{w_i}(\xbf_i)\|_{\varepsilon_{2,1}}  + m\varepsilon.
%$
%
%\end{proof}
\begin{lemma}\label{lem:inner_chp5}
Suppose the sequence $ \{ \xt \}$ is generated by  executing Lines 3 of Algorithm 1 with fixed $ \epst = \varepsilon$ then
\begin{enumerate}
\item $ \| \nabla \Psi_{\Theta, \gamma}^{\varepsilon}(\xt) \| \to 0$ as $t \to \infty$.
\item Finitely many iterations to meet the condition in Step 4 for reducing $\varepsilon$.
\end{enumerate}
\end{lemma}
\begin{lemma}\label{lem:phi_decay_chp5}
Suppose the sequence $\{ \xt\}$ is generated by Algorithm 1 with initial $\Xbf^{(0)}$, then we have 
$ \Psi_{\Theta, \gamma}^{\epstp}(\xtp) + m \epstp \leq \Psi_{\Theta, \gamma}^{\epst}(\xt) + m \epst$ for any $t\geq 0$.
\end{lemma}
\begin{proof}
First, we can easily verify that $\|h_{w_i}(\xbf_i)\|_{\varepsilon_{2,1}}  \leq \|h_{w_i}(\xbf_i)\|_{2,1} 
\le  \|h_{w_i}(\xbf_i)\|_{\varepsilon_{2,1}}  + m\varepsilon.
$  The proof of Lemma \ref{lem:inner_chp5} and \ref{lem:phi_decay_chp5} can be found in \cite{chen2021learnable,jimaging7110231}. %Lemma \ref{lem:bound} can be verified by simple calculation.
\end{proof}
\subsection{Convergence Analysis}
The proof of \textbf{Theorem 1} is outlined below.
%\begin{theorem}
%Suppose that $\{\xt \}$ is the sequence generated by Algorithm \ref{alg:lda_chp5} with any initial $\Xbf^{(0)}$, $\etol=0$ and $T=\infty$. Let $ \{ \Xbf^{(t_l+1)}\}$ be the subsequence that satisfies the reduction criterion  in step \ref{reduction_cre} of Algorithm \ref{alg:lda_chp5}. Then $ \{ \Xbf^{(t_l+1)}\}$ has at least one accumulation point, and every accumulation point of $\{ \Xbf^{(t_l+1)}\}$ is a Clarke stationary point of $ \min_{\Xbf} \Psi_{\Theta, \gamma}(\Xbf)$.
%    \label{theorem} of Lemma \ref{lem:bound}
%\end{theorem}
\begin{proof}
From the inequality $\|h_{w_i}(\xbf_i)\|_{2,1} 
\le  \|h_{w_i}(\xbf_i)\|_{\varepsilon_{2,1}}  + m\varepsilon
$, we have $ \Psi_{\Theta, \gamma}(\Xbf) \leq \Psi_{\Theta, \gamma}^{\varepsilon}(\Xbf) + m \varepsilon$ for any $\varepsilon>0$. Together with Lemma \ref{lem:phi_decay_chp5}, we get
$
    \Psi_{\Theta, \gamma}(\xt) \leq \Psi_{\Theta, \gamma}^{\epst}(\xt) +m\epst\leq \cdots \leq \Psi_{\Theta, \gamma}^{\varepsilon_0}(\Xbf^{(0)}) +m\varepsilon_0 <\infty.
$
As we assumed that $\Psi_{\Theta, \gamma}$ is coercive, we can get that $\{ \xt\}$ is bounded. Accordingly the selected subsequence $ \{ \Xbf^{(t_l+1)} \}$ is also bounded and has at least one accumulation point.
Lemma \ref{lem:inner_chp5} indicates finite many iterations from $t_{l - 1} + 1$ to $t_l + 1$. As $ \Xbf^{(t_l+1)}$ satisfies the reduction criterion  in step 4 of Algorithm \ref{alg:lda_chp5}, we have $ \norm{\nabla \Psi_{\Theta, \gamma}^{\varepsilon_{t_l}} (\Xbf^{(t_l+1)})} \leq \sigma  \varepsilon_{t_l} \eta= \sigma \varepsilon_{0} \eta^{l+1} \to 0$ as $l\to \infty$.  %Let $ \{ \xpp\} $ be any convergent subsequence of $\{ \xbf_{t_l+1} \}$ and denote $\epsp$ as the corresponding $\epst$ used in \mbox{Algorithm~\ref{alg:lda_chp5}} that generates $\xpp$.
Then, there exists at least one convergent subsequence of $\Xbf^{(t_l+1)}$, dubbed $\{\Xbf^{(k + 1)}\}$, that satisfies $ \Xbf^{(k + 1)} \to \hat{\Xbf}$ as $ \varepsilon_k \to 0,$ and $ \nabla \Psi_{\Theta, \gamma}^{\varepsilon_{k}}(\Xbf^{(k + 1)}) \to 0$ as $ k\to \infty$, where $\varepsilon_{k}$ is the corresponding $\varepsilon_{t_l}$ associated with $\Xbf^{(k + 1)}$. 

As mentioned, we denote $\Xbf = \{\xbf_1, \xbf_2, \xbf_3 \}$ for notation simplicity. 
The Clark subdifferential of each $\xbf_i$ is almost identical to \cite{jimaging7110231} but only added a smooth term in (1), so the analysis for each $\xbf_i$ will be the same as \cite{jimaging7110231}.
It has been proved in \cite{jimaging7110231} that $
\mathrm{dist}( \nabla \Psi_{\Theta, \gamma}^{\varepsilon_{k}}(\xbf^{(k + 1)}_i), \partial^C \Psi_{\Theta, \gamma}(\hat{\xbf}_i)) \to 0,
$ as $k \to \infty$, where $\partial^C$ denotes the Clark subdifferential. As this holds for each $\xbf_i$, then we get $
\mathrm{dist}( \nabla \Psi_{\Theta, \gamma}^{\varepsilon_{k}}(\Xbf^{(k + 1)}), \partial^C \Psi_{\Theta, \gamma}(\hat{\Xbf})) \to 0,
$ as $k \to \infty$. Since we already proved $ \nabla \Psi_{\Theta, \gamma}^{\varepsilon_{k}}(\Xbf^{(k + 1)}) \to 0$ and $\partial^C \Psi_{\Theta, \gamma}(\hat{\Xbf})$ is closed, we conclude that $0 \in \partial^C \Psi_{\Theta, \gamma}(\hat{\Xbf})$.
\end{proof}

\begin{algorithm}[htb]
\caption{Algorithmic Unrolling Method with Provable Convergence}
\label{alg:lda_chp5}
\begin{algorithmic}[1]
\STATE \textbf{Input:} Initial $\Xbf^0$, $0<\rho, \eta<1$, and $\varepsilon_0$, $a, \sigma >0$. Max total phases $T$ or tolerance \\$\etol>0$.
\FOR{$t=0,1,2,\dots,T-1$}
\STATE $\vtp = \xt - \alpha_{t}  \nabla \Psi_{\epst}(\xt)$, \label{marker_chp5}
\IF{ $ \Psi_{\epst}(\vtp) - \Psi_{\epst}(\xt) \le - \frac{1}{a} \| \vtp - \xt\|^2$ holds}
\STATE set $\xtp = \vtp$,
\ELSE
\STATE update $\alpha_{t} \leftarrow \rho \alpha_{t}$,
then \textbf{go to}~\ref{marker_chp5},
\ENDIF
\STATE \textbf{if} $\|\nabla \Psi_{\epst}(\xtp)\| < \sigma \eta {\epst}$,  set $\epstp= \eta {\epst}$;  \textbf{otherwise}, set $\epstp={\epst}$.
\STATE \textbf{if} $\sigma {\epst} < \etol$, terminate.
\ENDFOR
\STATE \textbf{Output:} $\Xbf^{(T)}$.
\end{algorithmic}
\end{algorithm}

The induce gradient descent network exactly follows Algorithm \ref{alg:lda_chp5}, where each phase of the network corresponds to an iteration of the algorithm. We take $\Xbf^{(T)}$ as the output of the multi-phase network, i.e. $\Xbf^{(T)} = \phi_{\hat{\Theta}}(\Xbf^{(0)}, \fbf_1, \fbf_2) $, where $\phi$ represent the algorithm-induced network and its parameters are collected by $\hat{\Theta} = \{\Theta, \lambda, \alpha_0, \alpha_1, ... \alpha_{T-1} \}$.

\subsection{Bilevel Optimization Algorithm for Network Training}
\label{section: bilevel}
%Given the initialization $\xbf_0$ represents the initialization of all modalities and the the partial k-space data $\{ \fbf_1, \fbf_2 \}$, we 
Suppose that we randomly sample $\mathcal{M}_{tr}$ data pairs $\{\D^{tr}_{i} \}_{i = 1} ^{\mathcal{M}_{tr}}$ for training and $\mathcal{M}_{val}$ data pairs $\{\D^{val}_{i} \}_{i = 1} ^{\mathcal{M}_{val}}$ for validation, where each $\D^{tr}_{i}(\mbox{or} \ \D^{val}_{i}) $ is composed of $ \{\fbf_1^{(i)}, \fbf_2^{(i)}, \Xbf^{*(i)} \}$ and $\fbf_1^{(i)}, \fbf_2^{(i)}$ denote the given partial k-space data for source modalities, and $\Xbf^{*(i)} = \{\xbf_1^{*(i)}, \xbf_2^{*(i)}, \xbf_3^{*(i)} \}$ denotes the corresponding reference images for $\xbf_1^{(i)}, \xbf_2^{(i)}, \xbf_3^{(i)}$.

For the sake of selecting a optimal coefficient for the constraint term of $\xbf_3$ in \eqref{our_model}, we introduce a novel learning framework by formulating the network training as a bilevel optimization problem to learn $\Theta $ and $\lambda$ in \eqref{our_model} as 
\begin{subequations}
\label{eq:bi-level_chp5}
\begin{align}
  \min_{ \lambda}  \quad & \sum^{\mathcal{M}_{val}}_{i=1} \ell( \Theta (\lambda) , \lambda ; \D^{val}_{i}) \\
 \mbox{s.t.}\quad \quad  & \Theta(\lambda)  = \argmin_{\Theta} \sum^{\mathcal{M}_{tr}}_{i=1} \ell ( \Theta , \lambda; \D^{tr}_{i}),
\end{align}
\end{subequations}

where
\begin{subequations}\label{loss_chp5}
\begin{align}
  \ell ( \Theta , \lambda; \D_{i}) & := \sum_{j = 1}^{3}\frac{1}{2} \|\xbf_j^T({\Theta , \lambda}; \D_{i})  - \xbf_j^{*(i)} \|^2_2 + (1 - SSIM(\xbf_j^T({\Theta , \lambda}; \D_{i}), \xbf_j^{*(i)} ) ) \\
  & + \frac{\mu}{2} \|g_{\theta} ([h_{w_1}(\xbf_1^{*(i)}),h_{ w_2} (\xbf_2^{*(i)})]) - \xbf_3^{*(i)}\|_2^2,
\end{align}
\end{subequations}
and the $\xbf_j^{(\hat{T})}(\cdot)$ denotes the output of the $\hat{T}$-phase network for the $j$th modality.
The first term of \eqref{loss_chp5} is to minimize the difference between the output of the network and the ground truth. The second term is the structural similarity index \cite{wang2004image}. The third term is to enforce the $g_{\theta}$ to synthesize $\xbf_3$.
In \eqref{eq:bi-level_chp5}, the lower-level optimization learns the network parameters $\Theta$ of the convolution layers with the fixed coefficient $\lambda$ on the training dataset, and the upper-level adjusts the coefficient $\lambda$ on the validation set so that the coefficient $\lambda$ can generalize to the validation set as well. 
Before we introduce the bi-level training algorithm, we simplify and redefine $\mathcal{L}(\Theta , \lambda ; \D) := \sum^{\mathcal{M}}_{i=1}\ell (\Theta , \lambda  ; \D_i)$ and then reformat \eqref{eq:bi-level_chp5} to be
\begin{equation}
  \min_{ \lambda} \mathcal{L}( \Theta(\lambda), \lambda ; \D^{val}) \ \ \ \ \ \mbox{s.t.} \ \ \Theta(\lambda) =   \argmin_{\Theta}\mathcal{L}( \Theta, \lambda ; \D^{tr}).
  \label{simplified bi-level_chp5}
\end{equation}

As the above bi-level optimization problem \eqref{simplified bi-level_chp5} is hard to solve, here we first consider to relax it into a single-level constrained optimization problem following \cite{mehra2019penalty}. More specifically, we use the  first-order necessary condition of the lower-level part as a constraint term, as shown below
\begin{equation}
   \min_{\lambda} \mathcal{L}( \Theta(\lambda), \lambda ; \D^{val})  \ \ \ \ \ \mbox{s.t.} \ \ \nabla_{\Theta} \mathcal{L}( \Theta, \lambda ; \D^{tr}) = 0.
  \label{simplified bi-level-1_chp5}
\end{equation}

The constraint optimization problem above can be further relaxed by replacing the constraint by a penalty term, then we obtain the following compositional optimization problem
\begin{equation}
  \min_{ \Theta, \lambda} \big\{ \widetilde{\mathcal{L}}( \Theta, \lambda ; \D^{tr}, \D^{val}) := \mathcal{L}( \Theta, \lambda ; \D^{val}) + \frac{\lambda}{2} \| \nabla_{\Theta} \mathcal{L}( \Theta, \lambda; \D^{tr}) \|^2 \big\}.
  \label{simplified bi-level-2_chp5}
\end{equation}

Due to the large volume of the training and test datasets, it is not possible to solve \eqref{simplified bi-level-2_chp5} in full-batch. Here we train the parameters using the mini-batch stochastic alternating direction penalty method summarized in Algorithm \ref{alg:model_chp5}.

\begin{algorithm}
\caption{Mini-batch alternating direction penalty algorithm}\label{alg:model_chp5}
\begin{algorithmic}[1]
\STATE \textbf{Input}  $\D^{tr}$, $\D^{val}$, $\delta_{tol}>0$.
\STATE \textbf{Initialize}  $ \Theta$, $ {\lambda}$, $\delta$, $\gamma>0$ and $\nu_\delta \in(0, 1)$, \ $\nu_\gamma > 1$.
\WHILE{$\delta > \delta_{tol}$}
\STATE Sample training and validation batch $\mathcal{B}^{tr} \subset \D^{tr}, \mathcal{B}^{val} \subset \D^{val}$.
\WHILE{$\|\nabla_{\Theta}\widetilde{\mathcal{L}}( \Theta, \lambda ; \mathcal{B}^{tr}, \mathcal{B}^{val})\|^2 + \| \nabla_{\lambda}\widetilde{\mathcal{L}}( \Theta, \lambda ; \mathcal{B}^{tr}, \mathcal{B}^{val})\| ^2 > \delta$}
\FOR{$k=1,2,\dots,K$ (inner loop)}
%\STATE  Sample batches of tasks $ \tau_i \sim p(\tau)$
\STATE $ \Theta \leftarrow \Theta - \rho_{\Theta}^k \nabla_{\Theta}\widetilde{\mathcal{L}}( \Theta, \lambda ; \mathcal{B}^{tr}, \mathcal{B}^{val})$
\ENDFOR
\STATE $ \lambda \leftarrow \lambda - \rho_{\lambda} \nabla_{\lambda}\widetilde{\mathcal{L}}( \Theta, \lambda ; \mathcal{B}^{tr}, \mathcal{B}^{val})$
\ENDWHILE
\STATE \textbf{update} $\delta \leftarrow \nu_\delta \delta$, $\ \gamma \leftarrow \nu_\gamma \gamma$
\ENDWHILE
\STATE \textbf{output:} $\Theta, \lambda$
\end{algorithmic}
\label{penelty_method_chp5}
\end{algorithm}

\section{Experiments}
\subsection{Initialization}
\label{Initialization}
The initial reconstruction $\{\xbf_1^{(0)}, \xbf_2^{(0)}, \xbf_3^{(0)} \}$ is obtained through the initialization network shown in Fig. \ref{fig:flowchart} which consists of three sub-nets for $\xbf_1^{(0)}, \xbf_2^{(0)}, \xbf_3^{(0)}$, and then fed into the multi-phase LOA-induced recovery net illustrated in Section \ref{loa induced net}. Each block of the initialization nets is designed in residual structure.
\subsubsection{Initialization-Net}
The structures of the initialization modules for $\xbf_1^{(0)}$ and $\xbf_2^{(0)}$ are identical, which consist of a k-space interpolation block followed by a inverse Fourier transform operator then a residual block in image domain. The k-space interpolation block is to interpolate the missing components of the k-space data, which is composed of 4 convolution layers separated by RELU. Then we pass the interpolated k-space data into the inverse Fourier transform operator to obtain a initial image. Next we obtain the refined initial image after a 4-convolution residual block as the input for the multi-phase recovery net.

As the k-space data for $\xbf_3^{(0)}$ is missing, here we take the initial image $\xbf_1^{(0)}$ or $\xbf_2^{(0)}$  as input to initialize the $\xbf_3^{(0)}$. We obtain the initial $\xbf_3^{(0)}$ image through a residual initialization net with 4 convolution layers and ReLU in between.

\subsubsection{Initialization-Nets Training}
To train the parameters in the initialization networks above, we minimize the difference between the outputs of the initialization networks and the ground truth images with the loss $L^I$ defined below

\begin{equation}\label{eq:loss_identity}
  L^I(\xbf_j^{(0)}, \xbf_j^{*}) = \|\xbf_j^{(0)} - \xbf_j^{*} \|_1, \ \ \ j = 1, 2, 3.
\end{equation}

We first train the initialization nets for $\xbf_1^{(0)}$ and $\xbf_2^{(0)}$ first. Then we take the output of  initialization net for avalible modalities $\xbf_1^{(0)}$ or $\xbf_2^{(0)}$ as the input to train the initialization net for $\xbf_3^{(0)}$. After we finish training all three initialization nets, we will keep them fixed when training the multi-phase recovery net.

\subsection{Experiment Setup}
The dataset we used is from BRATS 2018 \cite{menze2014multimodal} challenge, which scanned from 285 patients with high and low grade glioma cases in four modalities: T1, T2, Flair and T1ce. Each patient data was scanned with volume size $ 240 \times 240 \times 155$. Our experiments were implemented on high-grade gliomas (HGG) dataset. We randomly picked center 10 slices from 6 patients as testing dataset with cropped dimention size $ 160 \times 180$, which results in total of 60 testing images, and rest of HGG dataset are split into training data (1020 images) and validation data (1020 images). We compared with three state-of-the-art multimodal MR synthesis methods: Multimodal MR (MM) \cite{chartsias2017multimodal}, MM-GAN \cite{sharma2019missing},  MMGradAdv \cite{liu2020multimodal} and Hi-Net \cite{zhou2020hi}. We have implemented their methods with the hyper-parameters that indicated from their papers and trained their methods using our own training dataset and tested on our test dataset. The input images of MM, MMGradAdv and Hi-Net are ground truth images of source modalities, while our inputs are the partial k-sapce data that scanned with an radio mask with sampling ratio $40 \%$. We normalized the intensity values of all the dataset into $[0,1]$. 

\subsection{Hyper-parameter Selection}
All our experiments are implemented on Nvidia GTX-1080Ti GPU on Windows workstation and the trainable parameters are initialized with Xavier initialization \cite{glorot2010understanding} and trained with ADAM optimizer \cite{kingma2014adam} with initial learning rate $0.001$. We put $\lambda = 1$ in \eqref{our_model} and $\mu = 0.1$ in loss $\ell$. In our experiment, we use all complex convolution operators \cite{WANG2020136} where we set 4 convolutions with kernel size $3 \times 3 \times 64$ for $h_{w_i}$ and 6 convolutions with kernel size $3 \times 3 \times 128$ for $g_{\theta}$. For Algorithm \ref{alg:lda_chp5}, considering both algorithm convergence and the computational efficiency, we take the parameters as follows after trials: $\alpha_0= 0.01, \eta_0=0.01, \varepsilon_0 = 0.001, a = 10^5, \sigma = 10^3, \rho =0.9$, and $\gamma =0.9$. We also set the termination tolerance $\etol = 1\times 10^{-3}$. With this tolerance and the termination condition defined in Line 10, the algorithm will stop after 11 phases. We will provide the publicly shared code depending on acceptance.

 For Algorithm \ref{alg:lda_chp5}, considering both algorithm convergence and the computational efficiency, we take the parameters as follows after trials: $\alpha_0= 0.01, \eta_0=0.01, \varepsilon_0 = 0.001, a = 10^5, \sigma = 10^3, \rho =0.9$. We also set the termination tolerance $\etol = 1\times 10^{-3}$, together with the termination condition defined in Line 5, the algorithm will stop after 11 phases. 

Similarly, in Algorithm \ref{alg:model_chp5}, we select the parameters as follows: $\nu_\delta =0.95 $, $\delta =1 \times 10^{-3}$, $\lambda =10^{-4}$, $\nu_\lambda = 1.001$ and $ \rho_\gamma = 0.9$.  We decide the batch size to be $2$ considering the GPU memory and data size. We set $\delta_{tol} = 4.35 \times 10 ^ {-6}$ which makes the algorithm stop at around 1000 epochs. 

\subsection{Performance Evaluation}
We compared proposed method against multimodal MRI synthesis methods: MM \cite{chartsias2017multimodal}, MM-GAN \cite{sharma2019missing}, MMGradAdv \cite{liu2020multimodal} and Hi-Net \cite{zhou2020hi}, and take four synthesis directions including: T1 $+$ T2 $\to$ FLAIR,  T1 $+$ FLAIR  $\to$ T2,  T2 $+$ FLAIR  $\to$ T1 and T1 $+$ T2 $\to$ T1CE.  Table \ref{quant_results} reports quantitative results of the proposed method and the state-of-the-art methods. Inputs of the proposed model are partial k-space data of source modalities. $\fbf$ represents partial k-space data of corresponding modality. It is clear that our proposed method network outperforms MM and GAN-based methods (MM-GAN, MMGradAdv, Hi-Net) in terms of all the performance metrics. The average PSNR of our method improves about 0.67 dB comparing to the baseline Hi-Net, SSIM improves about 0.01,  and NMSE reduces about 0.01.   We input partial k-space data of source modalities to simultaneously reconstruct source modality images and synthesis target modality image, our model first learn a good initial input of the Algorithm \ref{alg:lda_chp5} and then get updated reconstruction images and synthesis images by iterating the Algorithm, which has theoretical convergence guarantee and thus achieves better performance. 

We also conduct the pure-synthesis experiment (T1 $+$ T2 $\to$ FLAIR) by inputting fully-scanned source data and the model \eqref{our_model} minimizes w.r.t. $\xbf_3$ only and excludes the data-fidelity terms. This result is in Table \ref{quant_results} where the PSNR value is about 1.16 dB higher than baseline method Hi-Net. 
 
Part of the reconstruction results are listed in Table \ref{t1t2}, which shows the comparison between purely reconstruct source modalities (only minimize $\xbf_1$ and $\xbf_2$ in model \eqref{our_model}) and jointly reconstruct source modalities $\xbf_1$, $\xbf_2$ and synthesis target modality $\xbf_3$. Table \ref{t1t2} shows the joint reconstruction and synthesis improves PSNR by 0.46 dB comparing to purely reconstructing T1 and T2 without synthesizing FLAIR.
%\textcolor{blue}{The reason is due to the third term in \eqref{our_model} has controlled by a weighted parameter $\gamma$, which is trained by bilevel optimization Algorithm \ref{alg:model_chp5} to promote the performance of all the three modalities.
%Also $g_\theta$ and $h_{w_3}$ generate more features and useful information of $\xbf_3$ which benefits the reconstruction of $\xbf_1$ and $\xbf_2$. This demonstrates the superiority of the proposed method in both MRI reconstruction and synthesis and these two tasks are mutually beneficial.}
%\textcolor{red}{Synthesis using common features and learn the mapping use data. x1 and x2 the joint feature maps can be refined by the information from x3 if g theta is good. }
We think this is because that the synthesis operator $g_{\theta}$ also leverages data $\xbf_3$ to assist shaping the feature maps of $\xbf_1$ and $\xbf_2$, which improves the reconstruction quality of the latter images.

Our results indicated that the performance of simply do reconstructions on source modalities without synthesis is worse than joint-reconstruction of source modalities and synthesis target modality using our model \eqref{our_model}. This indicates the effectiveness of the proposed method in both MRI reconstruction and synthesis.

Figure \ref{fig:synthesis} displays the synthetic MRI results on different source and target modality images. The proposed synthetic images preserves more details and sharper edges of the tissue boundary (indicated by enlarged red windows and green arrows) comparing to other referenced methods.

\begin{table}[ht]
\centering
\caption{Qualitative comparison of synthesis results between the state-of-the-art Multimodal synthesis methods and proposed method.  }\label{quant_results}
\resizebox{\linewidth}{48mm}{ 
\begin{tabular}{lllll}
\hline
& Methods &  PSNR & SSIM & NMSE\\
\hline
& MM \cite{chartsias2017multimodal}  & 22.8905 $\pm$ 1.4794 &  0.6671 $\pm$ 0.0586 &  0.0693 $\pm$ 0.0494 \\
& MM-GAN \cite{sharma2019missing}  & 23.3469 $\pm$ 1.0276 &  0.7084 $\pm$ 0.0370 &  0.0620 $\pm$ 0.0426 \\
T1 $+$ T2 $\to$ FLAIR & MMGradAdv \cite{liu2020multimodal}  & 24.0275 $\pm$ 1.3959 & 0.7586 $\pm$ 0.0326 & 0.0583 $\pm$ 0.0380 \\
& Hi-Net \cite{zhou2020hi} &   25.0299 $\pm$ 1.3845 & 0.8499 $\pm$ 0.0300 & 0.0254 $\pm$ 0.0097\\
& Proposed &  26.1851 $\pm$ 1.3357 & 0.8677 $\pm$ 0.0307 & 0.0205 $\pm$ 0.0087\\
\hline
$\fbf_{T1} + \fbf_{T2} \to$ FLAIR & Proposed & 25.7355 $\pm$ 1.2475 & 0.8597 $\pm$ 0.0315 & 0.0215 $\pm$ 0.0085\\
\hline
& MM \cite{chartsias2017multimodal}  & 23.8916 $\pm$ 1.6094 & 0.6895  $\pm$ 0.0511 & 0.0494 $\pm$ 0.0185 \\
& MM-GAN \cite{sharma2019missing}  & 24.1474 $\pm$ 0.8991 & 0.7217 $\pm$ 0.0432 & 0.0431 $\pm$ 0.0114 \\
 T1 $+$ FLAIR  $\to$ T2 & MMGradAdv \cite{liu2020multimodal} & 25.0563 $\pm$ 1.4895 &  0.7597 $\pm$ 0.0486 & 0.0406 $\pm$  0.0165\\ 
 & Hi-Net \cite{zhou2020hi} & 25.9490 $\pm$ 1.4958 & 0.8552 $\pm$ 0.0410 & 0.0229  $\pm$  0.0070\\
 \hline
$\fbf_{T1} + \fbf_{FLAIR} \to$ T2 & Proposed & 26.6585 $\pm$ 1.6178 & 0.8610 $\pm$ 0.0438 & 0.0207 $\pm$ 0.0072\\
\hline
& MM \cite{chartsias2017multimodal} &  23.5347 $\pm$ 2.1764 & 0.7825 $\pm$ 0.0470 &  0.0301 $\pm$ 0.0149 \\
& MM-GAN \cite{sharma2019missing}  & 23.6321 $\pm$ 2.3085  & 0.7908 $\pm$ 0.0421 & 0.0293  $\pm$ 0.0119\\
T2 $+$ FLAIR  $\to$ T1 & MMGradAdv \cite{liu2020multimodal} & 24.7342 $\pm$ 2.2255 & 0.8065 $\pm$ 0.0423 &  0.0252 $\pm$ 0.0118 \\ 
& Hi-Net \cite{zhou2020hi} &  25.6405 $\pm$ 1.5887 & 0.8729 $\pm$ 0.0349 &  0.0130 $\pm$ 0.0097 \\
\hline
$\fbf_{T2} + \fbf_{FLAIR} \to$ T1 & Proposed &  26.3063 $\pm$ 1.7955 & 0.9085 $\pm$ 0.0311 & 0.0112 $\pm$ 0.0113 \\
\hline
& MM \cite{chartsias2017multimodal} &  23.3678 $\pm$ 1.5612 & 0.7272 $\pm$ 0.0574 &  0.0312 $\pm$  0.0138 \\
& MM-GAN \cite{sharma2019missing}  & 23.6795 $\pm$ 0.9729 & 0.7577 $\pm$ 0.0637 & 0.0302 $\pm$ 0.0133\\
T1 $+$ T2 $\to$ T1CE & MMGradAdv \cite{liu2020multimodal} &  24.2311 $\pm$ 1.8961 & 0.7887 $\pm$ 0.0519  & 0.0273 $\pm$ 0.0136   \\ 
& Hi-Net \cite{zhou2020hi} &  25.2149 $\pm$ 1.2048 & 0.8650 $\pm$ 0.0328  &  0.0180 $\pm$ 0.0134   \\
\hline
$\fbf_{T1} + \fbf_{T2} \to$ T1CE & Proposed &  25.9103 $\pm$ 1.2077 & 0.8726 $\pm$ 0.0340  &  0.0167 $\pm$ 0.0133   \\
\hline
\end{tabular}}
\end{table}

\begin{table}\centering
\caption{Quantitative comparison of reconstruction results between the proposed model with joint-reconstruction and synthesis and reconstruction of T1 and T2 only. }\label{t1t2}
\begin{tabular}{|l|l|l|l|l|}
\hline
modality & FLAIR involved? & PSNR & SSIM & NMSE\\
\hline
T1 & No &  37.0040 $\pm$ 0.7411 & 0.9605 $\pm$ 0.0047 & 0.0008 $\pm$ 0.0002 \\
 & Yes & 37.4967 $\pm$ 0.8361 &  0.9628 $\pm$ 0.0074 & 0.0007 $\pm$ 0.0002 \\
T2 & No &  37.2401 $\pm$ 1.2252 & 0.9678 $\pm$ 0.0028 & 0.0027 $\pm$ 0.0010 \\
 & Yes & 37.6675 $\pm$ 1.3427 & 0.9663 $\pm$ 0.0043 & 0.0024 $\pm$ 0.0009 \\ 
\hline
\end{tabular}
\end{table}

\begin{figure}[ht]
\centering
\includegraphics[width=0.03\linewidth]{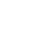}
\includegraphics[width=0.15\linewidth]{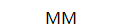}
\includegraphics[width=0.15\linewidth]{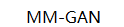}
\includegraphics[width=0.15\linewidth]{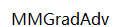}
\includegraphics[width=0.15\linewidth]{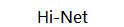}
\includegraphics[width=0.15\linewidth]{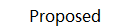}
\includegraphics[width=0.15\linewidth]{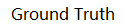}\\
\includegraphics[width=0.03\linewidth]{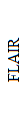}
\includegraphics[width=0.15\linewidth]{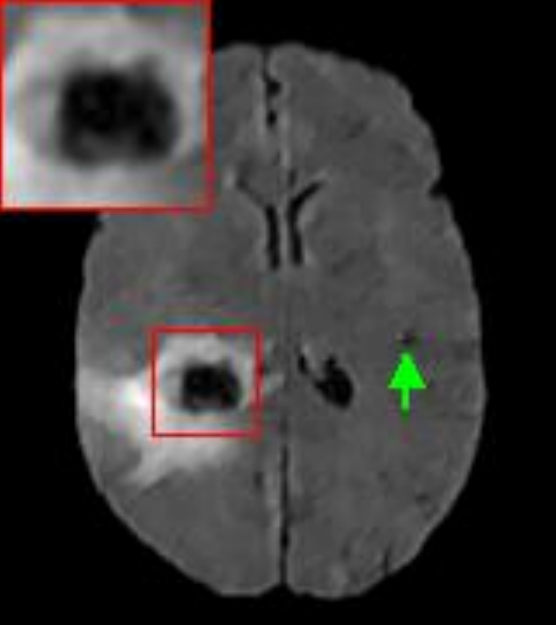}
\includegraphics[width=0.15\linewidth]{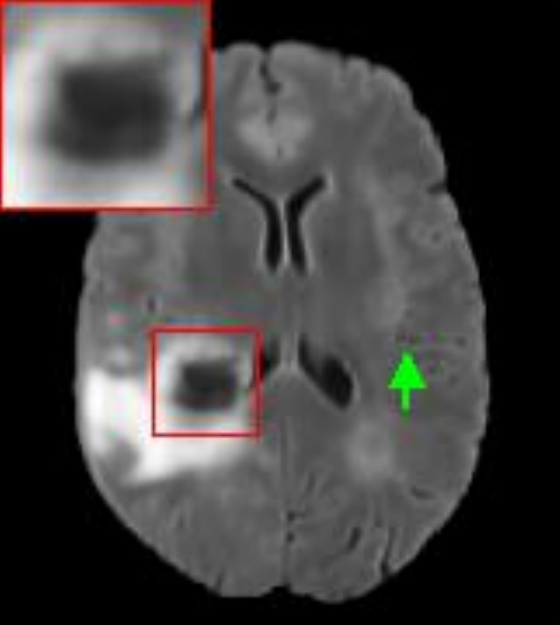}
\includegraphics[width=0.15\linewidth]{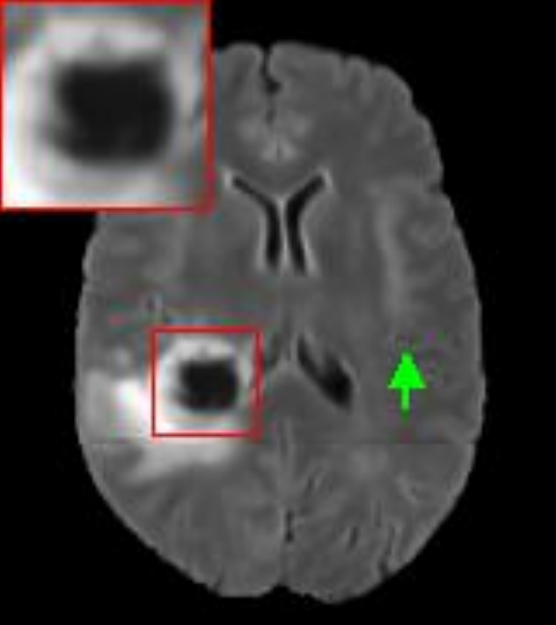}
\includegraphics[width=0.15\linewidth]{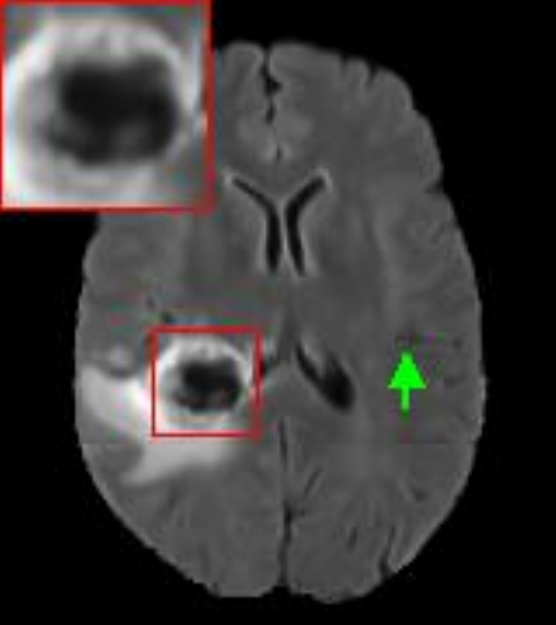}
\includegraphics[width=0.15\linewidth]{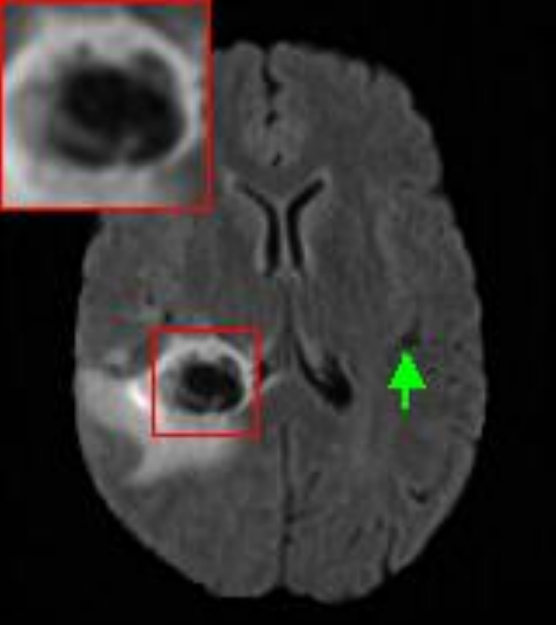}
\includegraphics[width=0.15\linewidth]{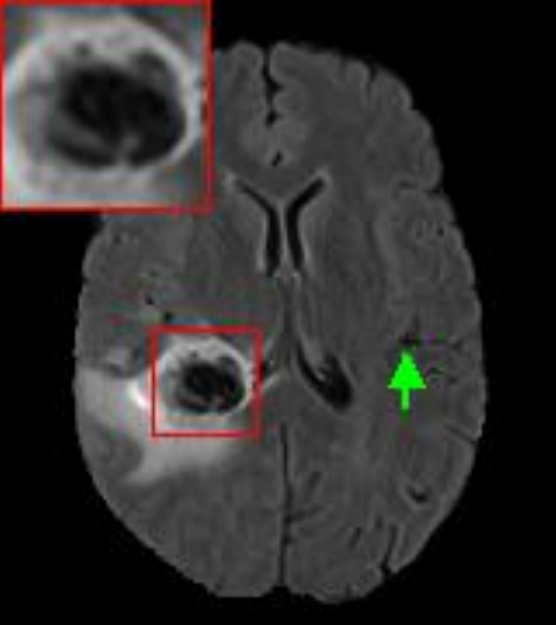}\\
\includegraphics[width=0.03\linewidth]{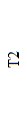}
\includegraphics[width=0.15\linewidth]{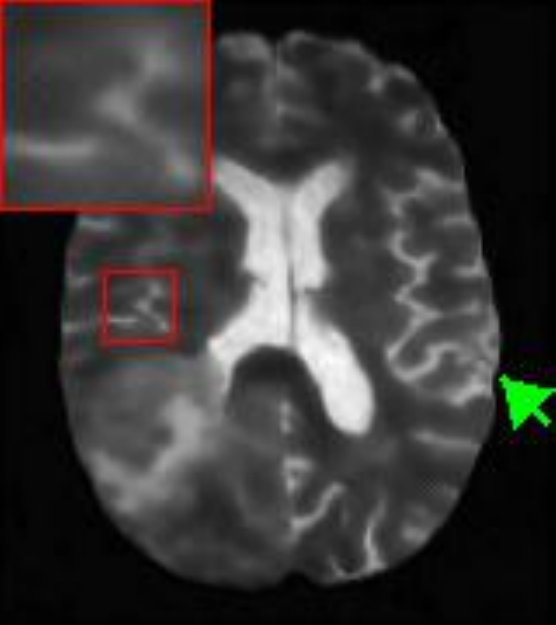}
\includegraphics[width=0.15\linewidth]{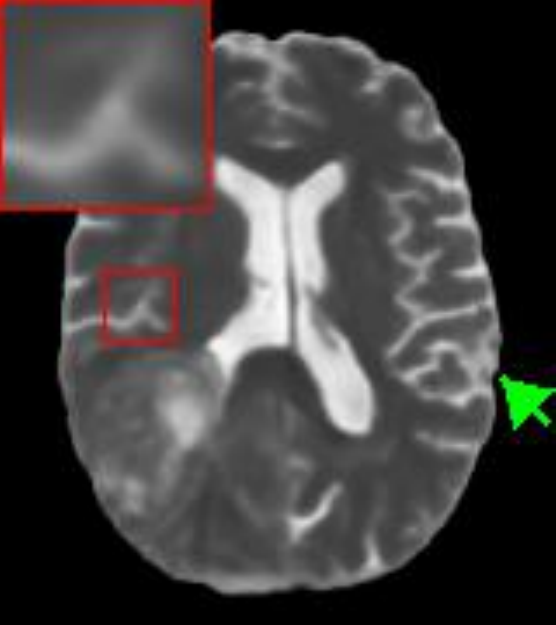}
\includegraphics[width=0.15\linewidth]{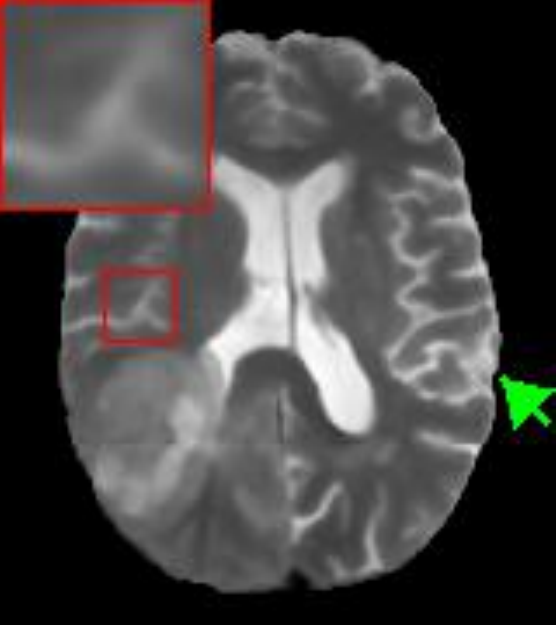}
\includegraphics[width=0.15\linewidth]{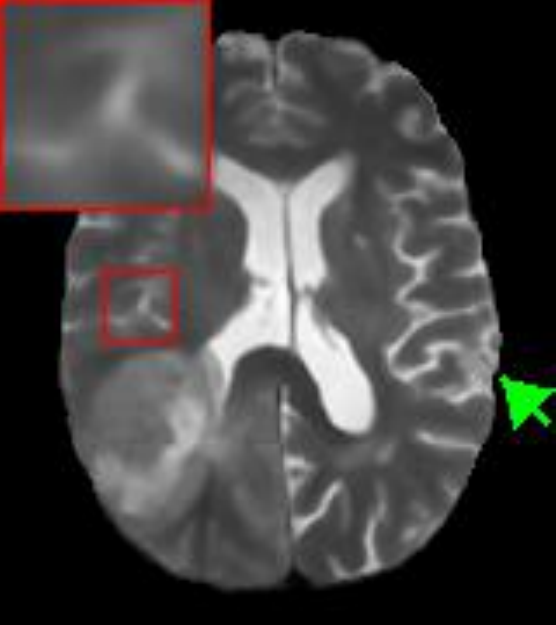}
\includegraphics[width=0.15\linewidth]{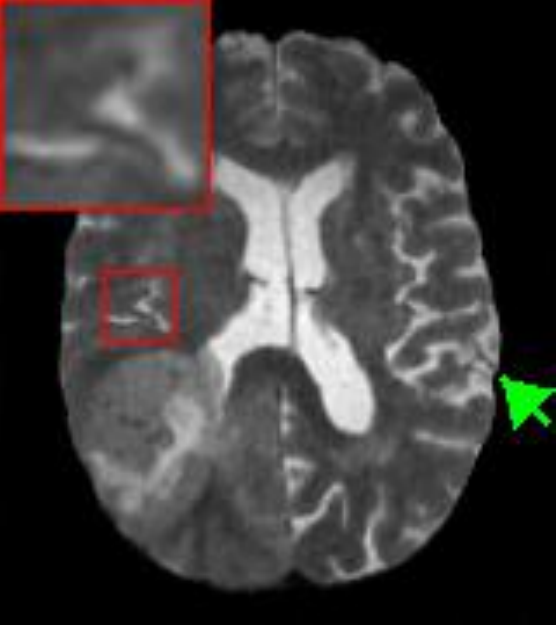}
\includegraphics[width=0.15\linewidth]{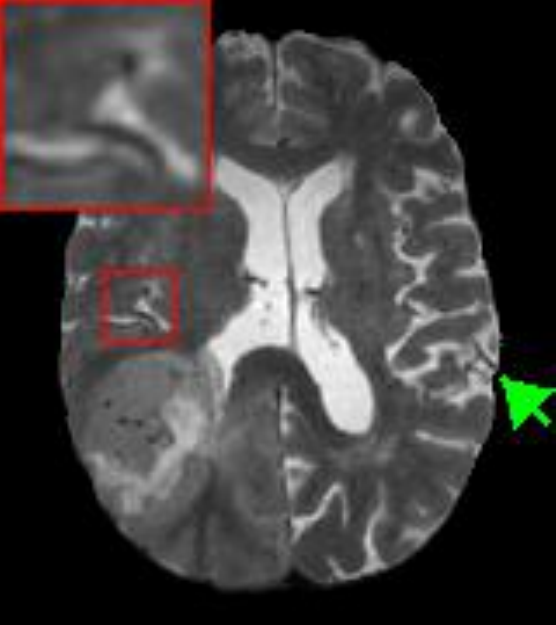}\\
\includegraphics[width=0.03\linewidth]{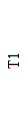}
\includegraphics[width=0.15\linewidth]{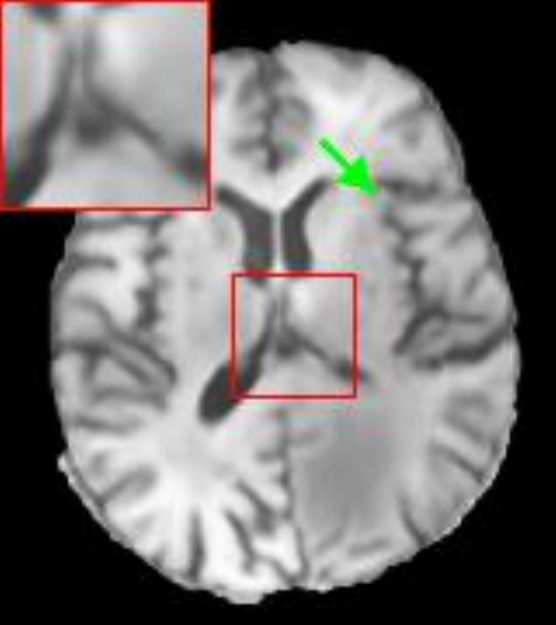}
\includegraphics[width=0.15\linewidth]{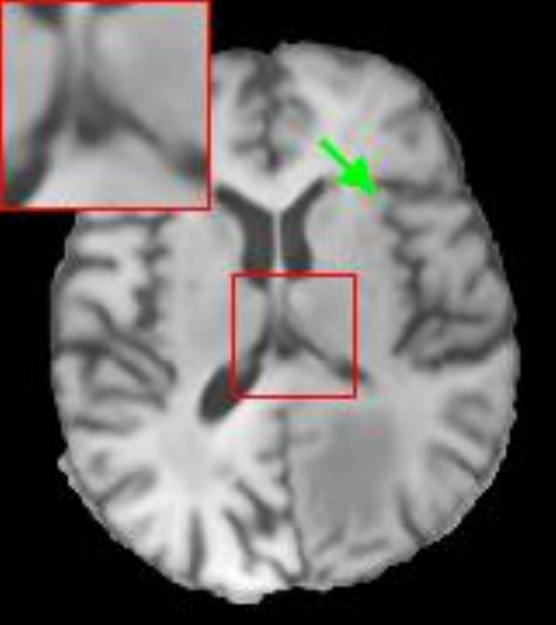}
\includegraphics[width=0.15\linewidth]{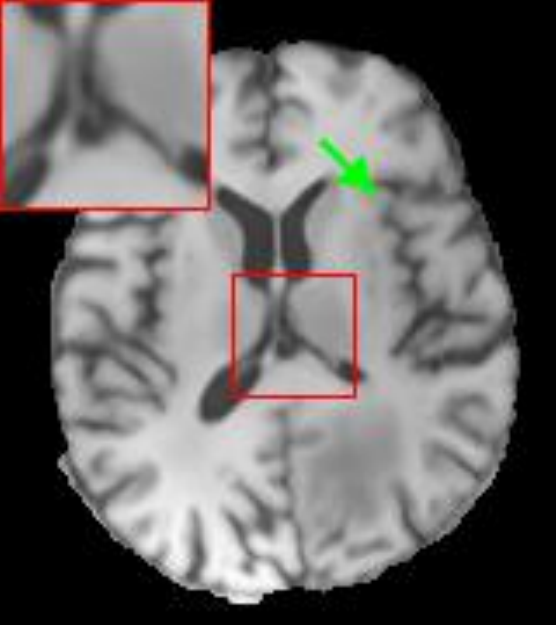}
\includegraphics[width=0.15\linewidth]{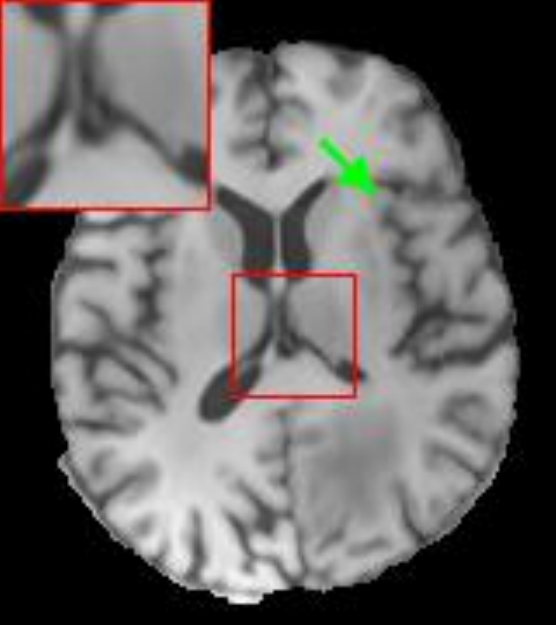}
\includegraphics[width=0.15\linewidth]{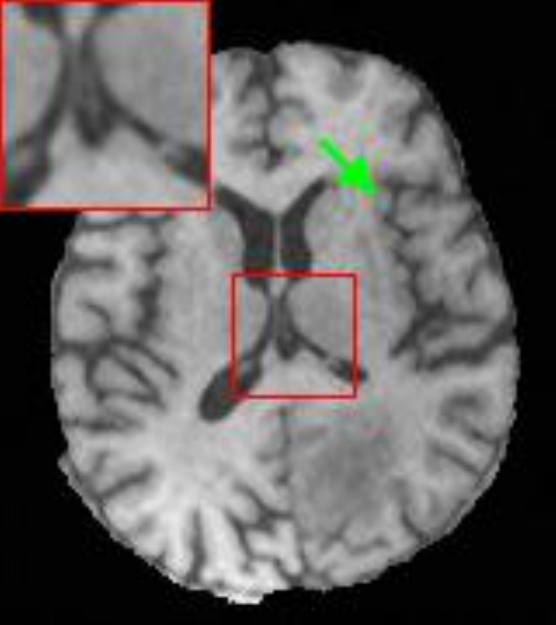}
\includegraphics[width=0.15\linewidth]{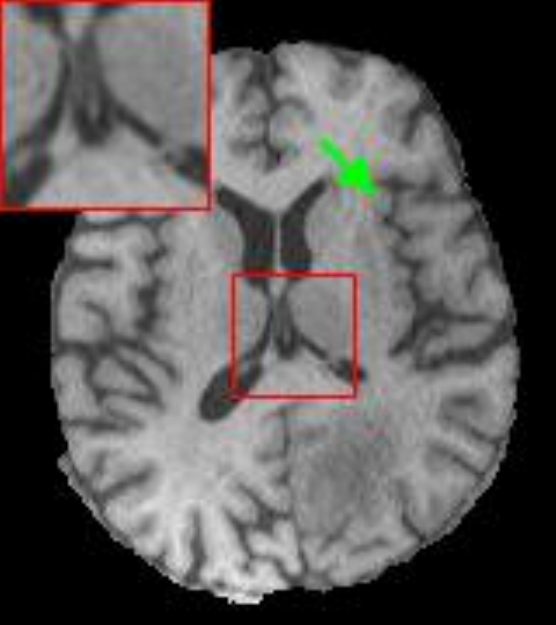}\\
\includegraphics[width=0.03\linewidth]{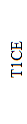}
\includegraphics[width=0.15\linewidth]{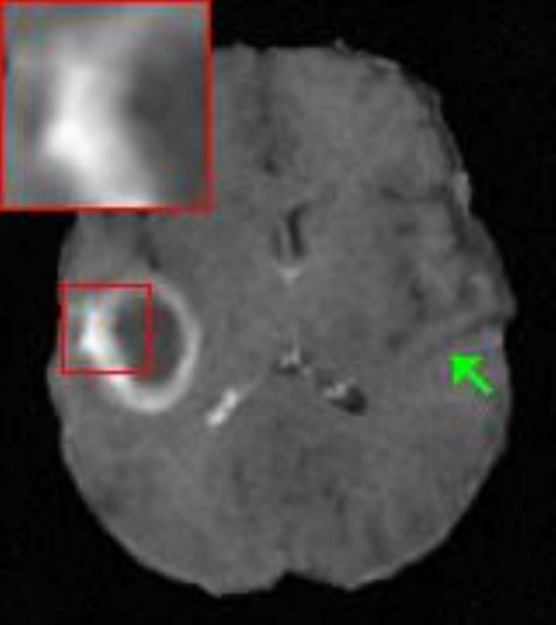}
\includegraphics[width=0.15\linewidth]{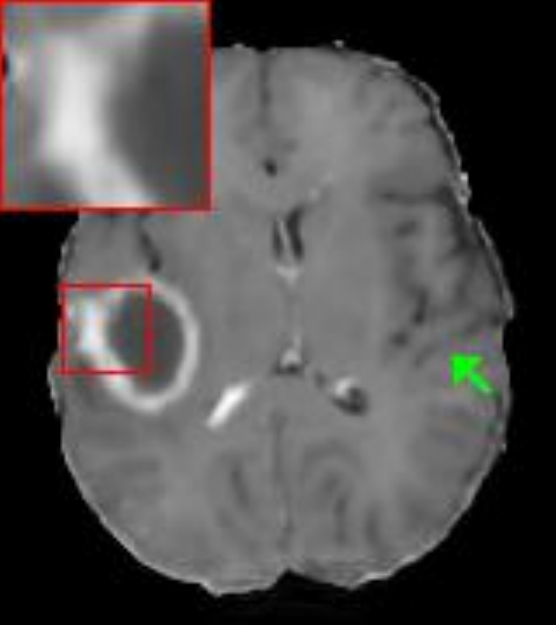}
\includegraphics[width=0.15\linewidth]{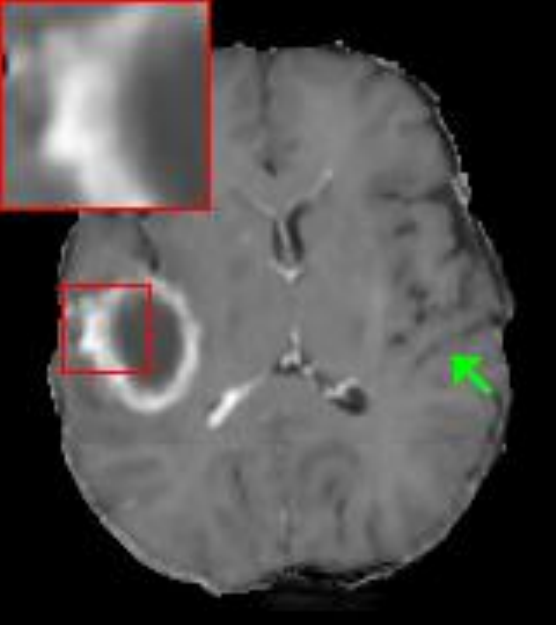}
\includegraphics[width=0.15\linewidth]{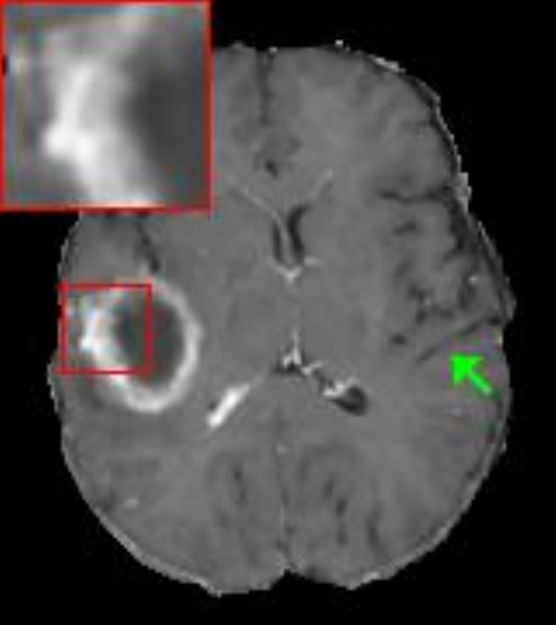}
\includegraphics[width=0.15\linewidth]{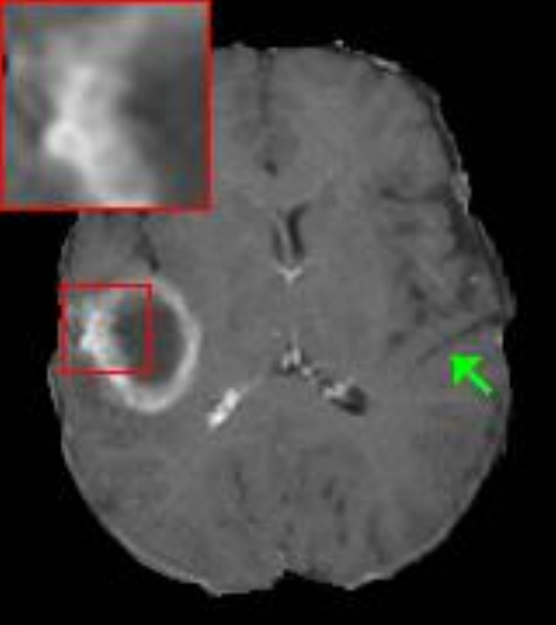}
\includegraphics[width=0.15\linewidth]{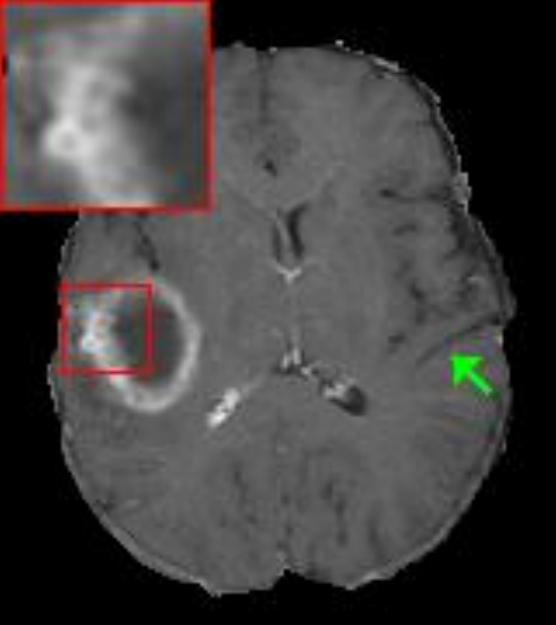}\\
\caption{Qualitative comparison between the state-of-the-art multimodal synthesis methods and proposed method. From first row to last row: T1 $+$ T2 $\to$ FLAIR, T1 $+$ FLAIR $\to$ T2, T2 $+$ FLAIR $\to$ T1 and T1 $+$ T2 $\to$ T1CE.  }\label{fig:synthesis}
\end{figure} 
\begin{figure}[ht]
\flushleft
\includegraphics[width=0.03\linewidth]{fig_chp5/smallwhite.png}
\includegraphics[width=0.16\linewidth]{fig_chp5/MM.png}
\includegraphics[width=0.16\linewidth]{fig_chp5/MM-GAN.png}
\includegraphics[width=0.16\linewidth]{fig_chp5/MMGradAdv.png}
\includegraphics[width=0.16\linewidth]{fig_chp5/Hi-Net.png}
\includegraphics[width=0.16\linewidth]{fig_chp5/Proposed.png}\\
\includegraphics[width=0.03\linewidth]{fig_chp5/FLAIR.png}
\includegraphics[width=0.16\linewidth]{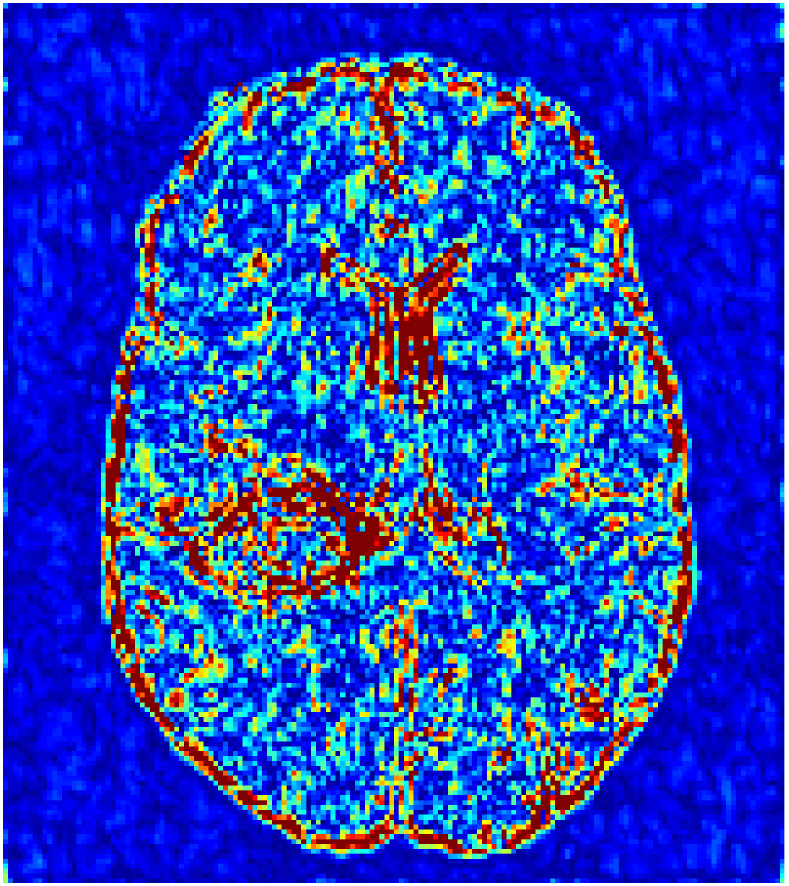}
\includegraphics[width=0.16\linewidth]{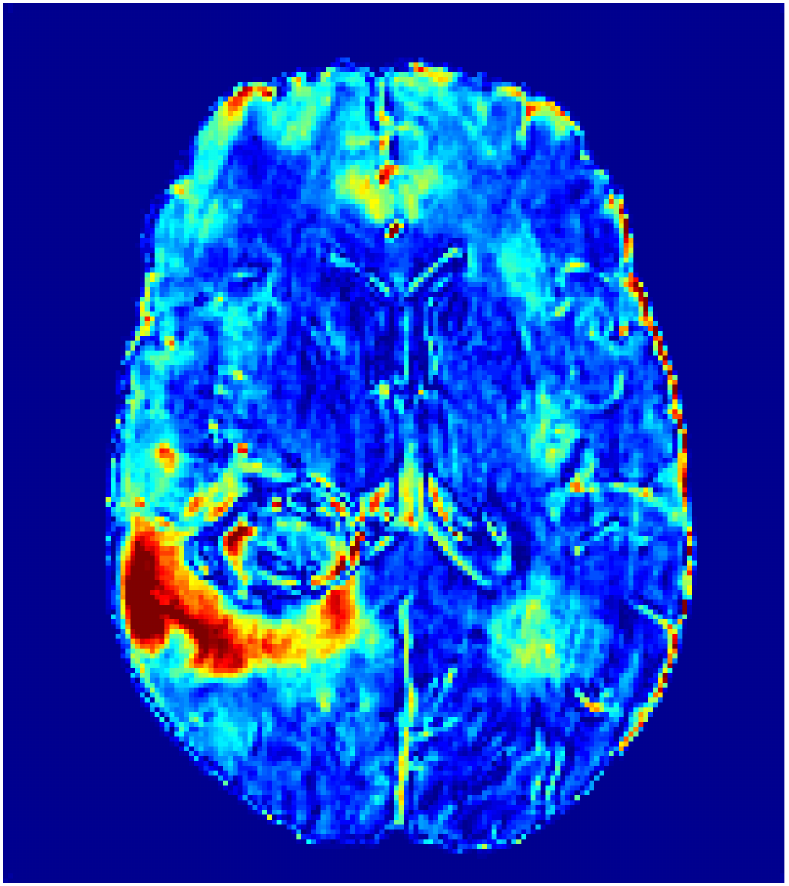}
\includegraphics[width=0.16\linewidth]{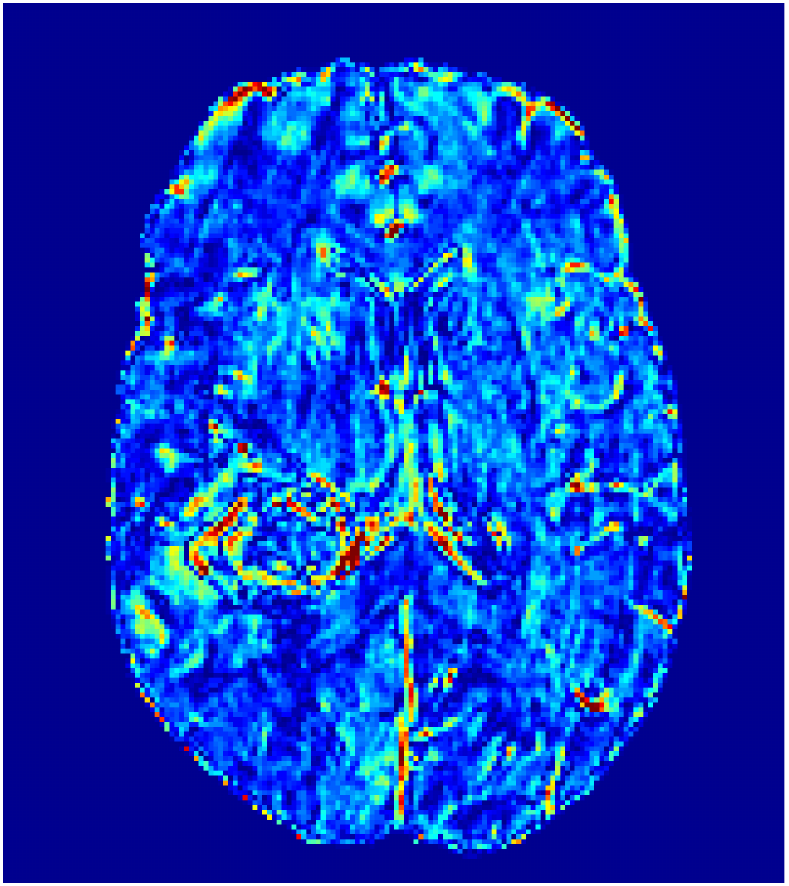}
\includegraphics[width=0.16\linewidth]{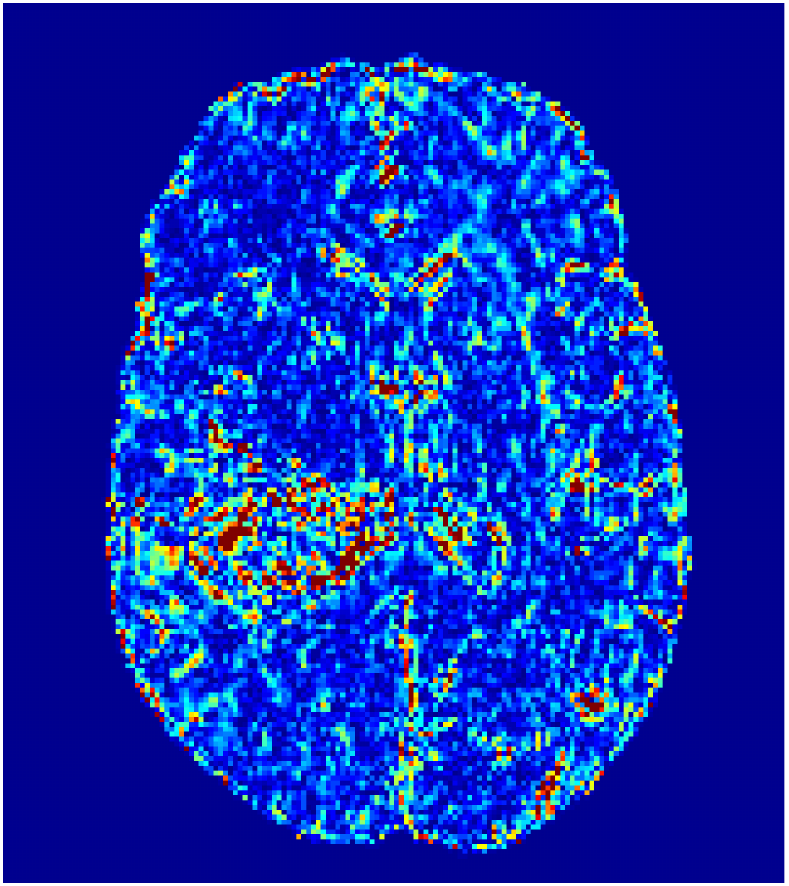}
\includegraphics[width=0.16\linewidth]{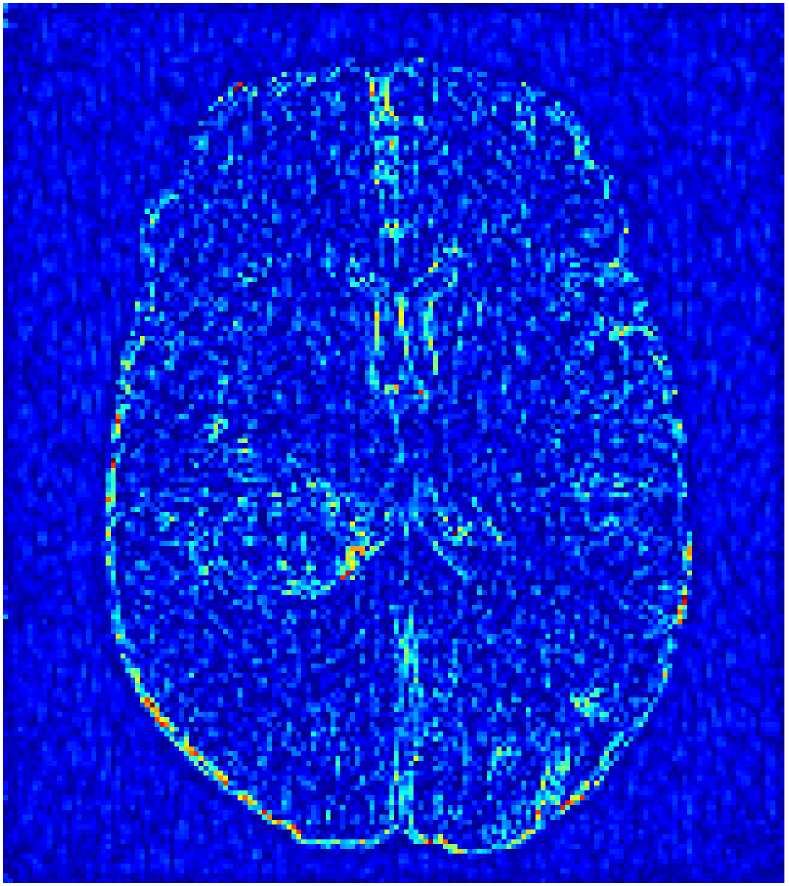}
\includegraphics[width=0.065\linewidth]{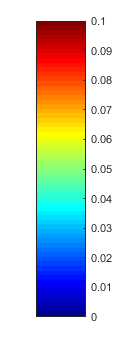}\\
\includegraphics[width=0.03\linewidth]{fig_chp5/T2.png}
\includegraphics[width=0.16\linewidth]{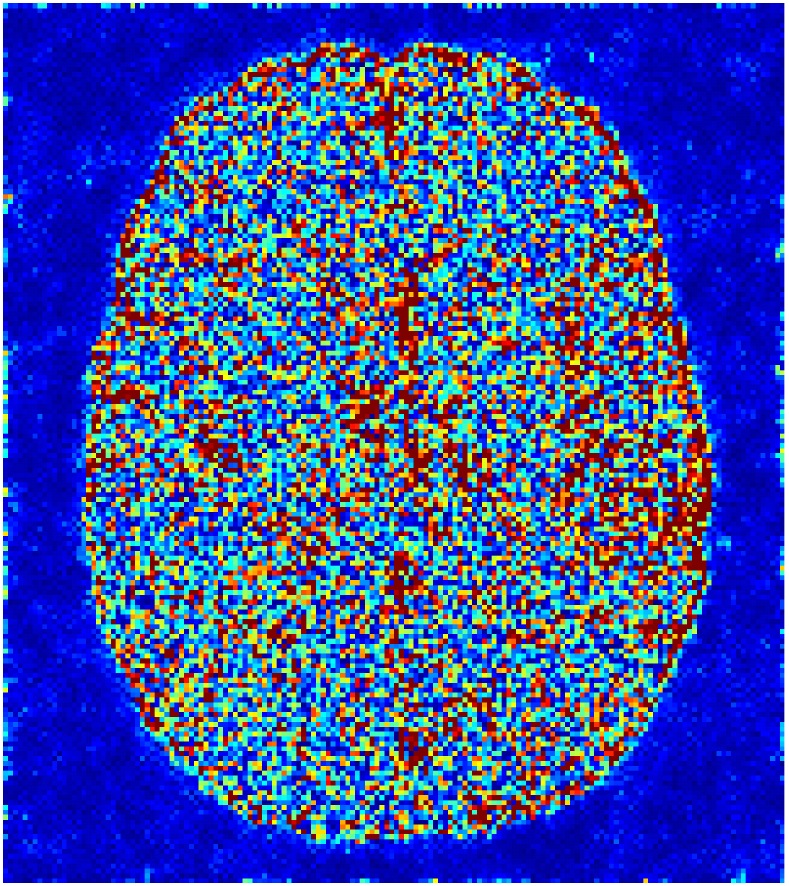}
\includegraphics[width=0.16\linewidth]{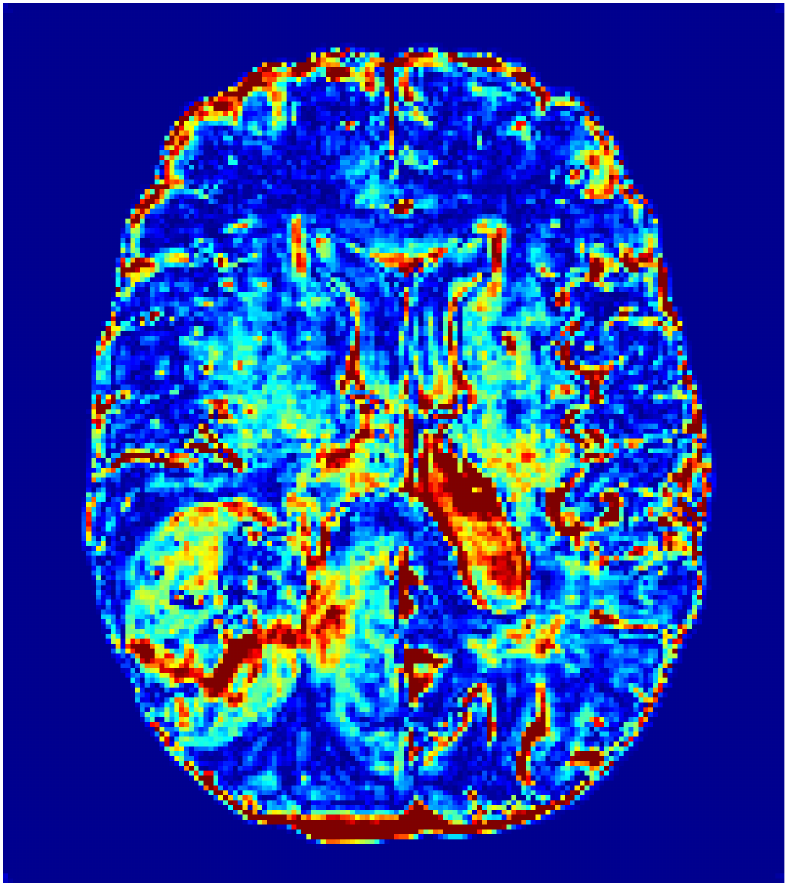}
\includegraphics[width=0.16\linewidth]{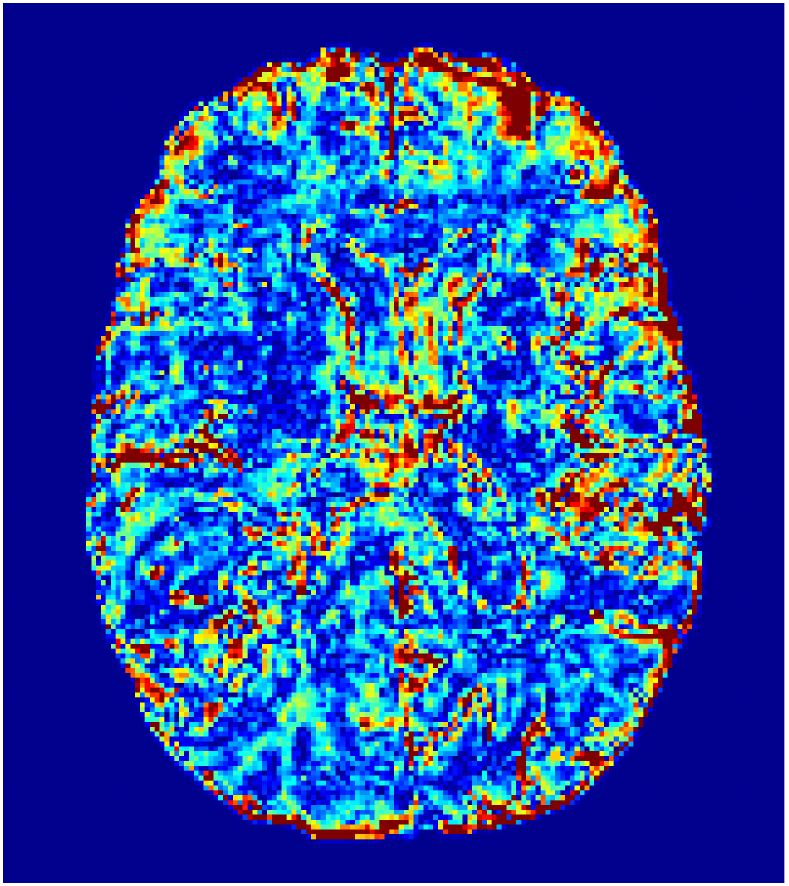}
\includegraphics[width=0.16\linewidth]{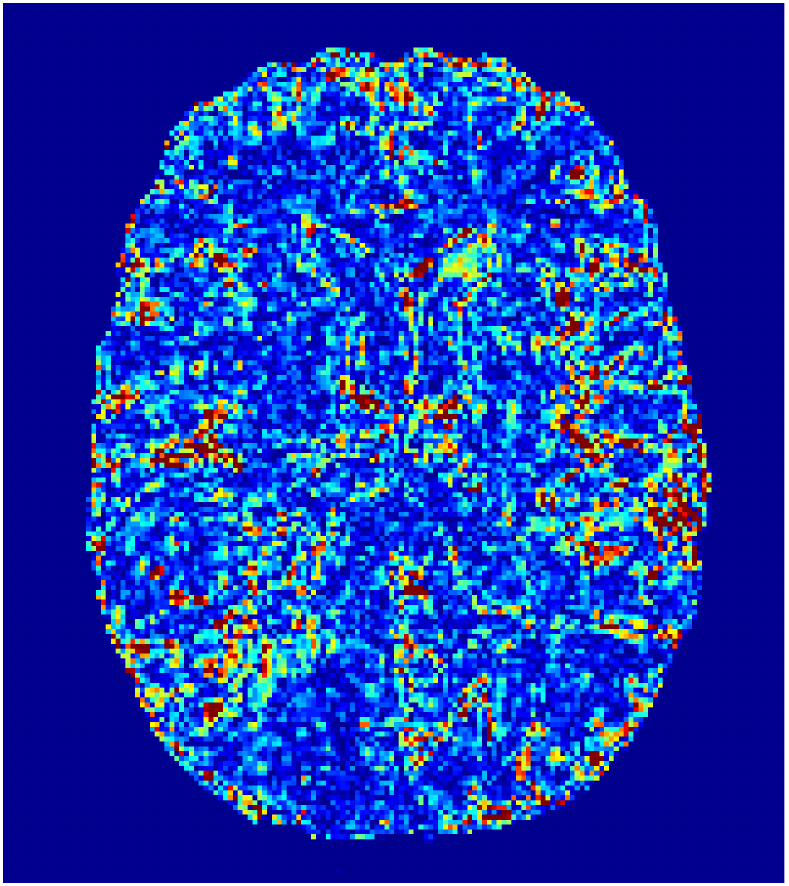}
\includegraphics[width=0.16\linewidth]{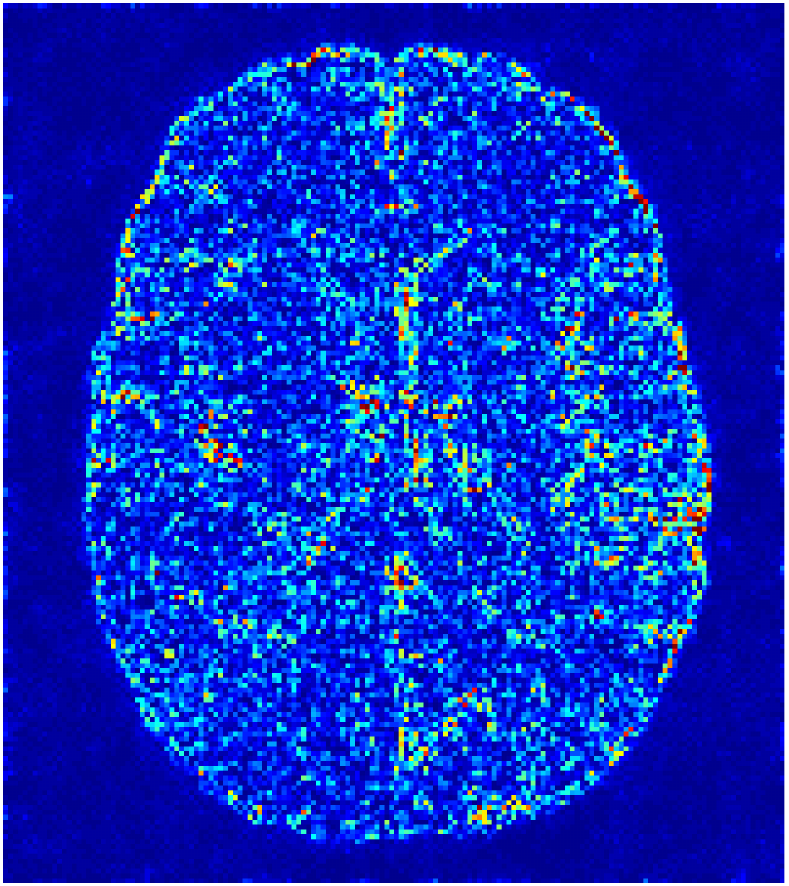}\\
\includegraphics[width=0.03\linewidth]{fig_chp5/T1.png}
\includegraphics[width=0.16\linewidth]{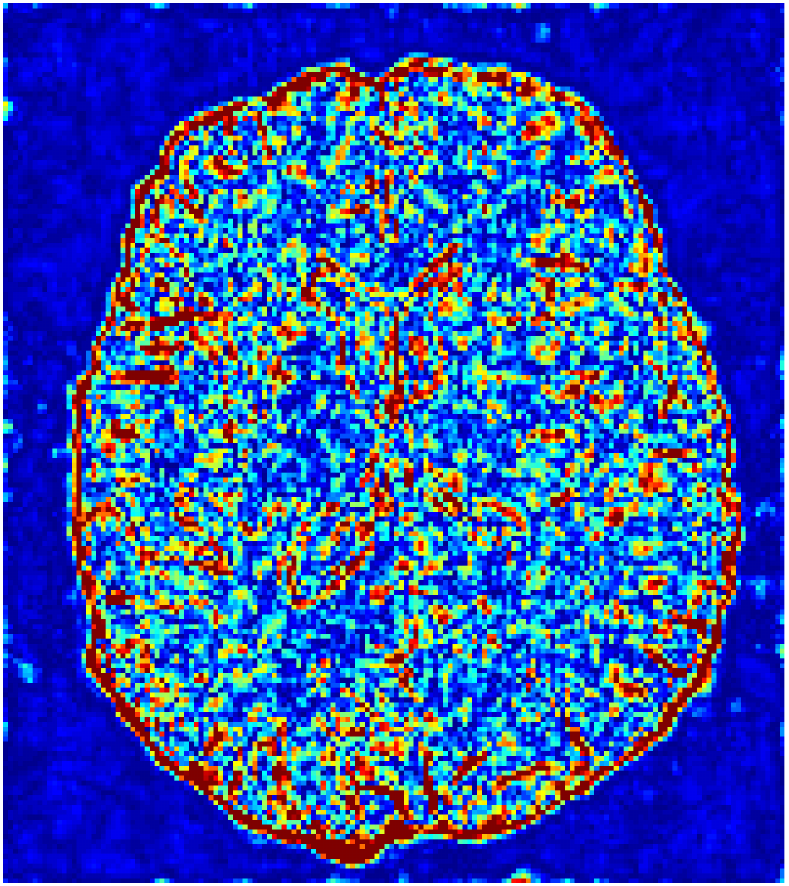}
\includegraphics[width=0.16\linewidth]{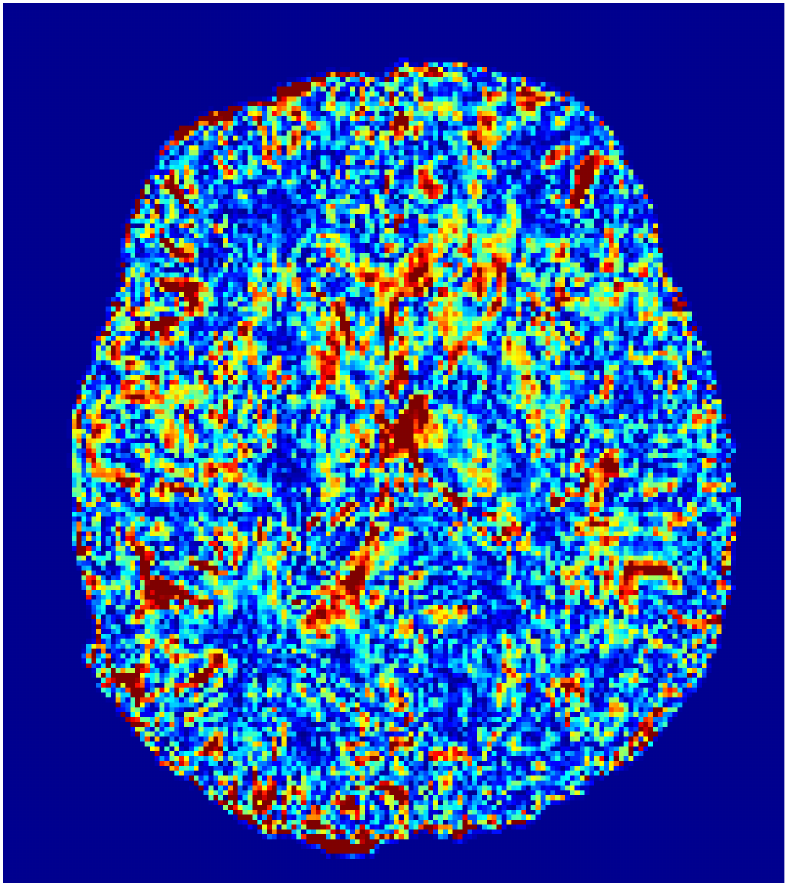}
\includegraphics[width=0.16\linewidth]{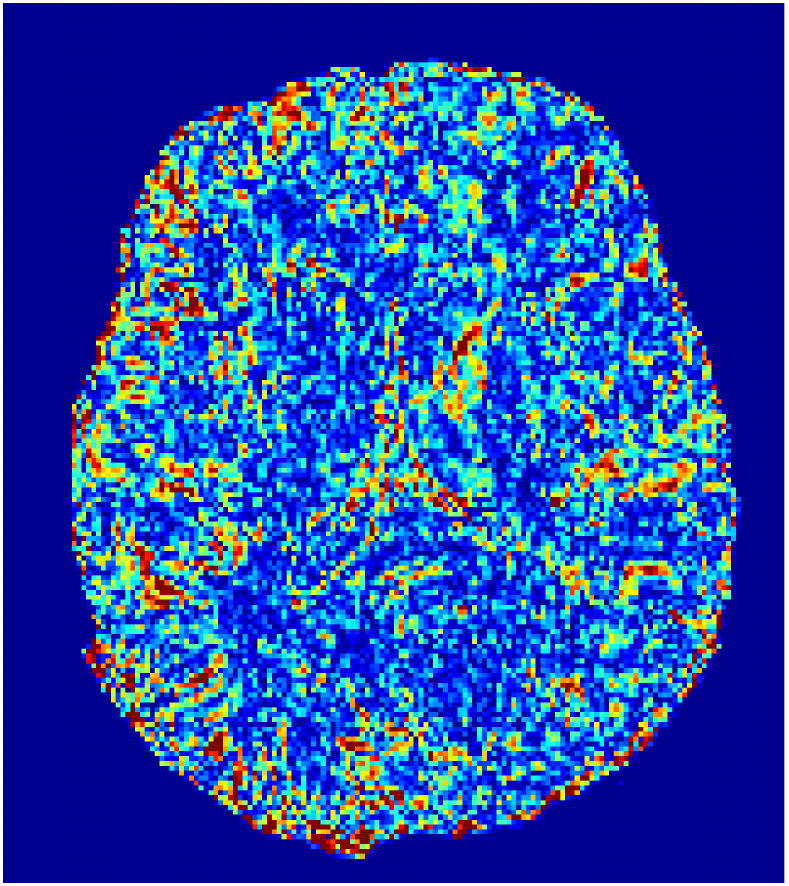}
\includegraphics[width=0.16\linewidth]{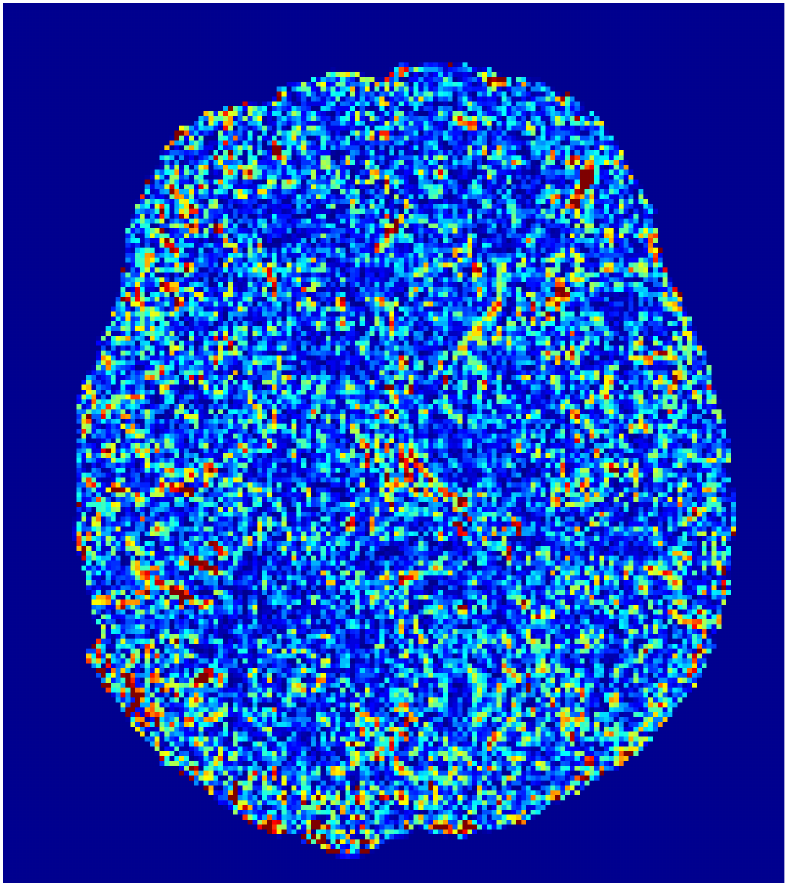}
\includegraphics[width=0.16\linewidth]{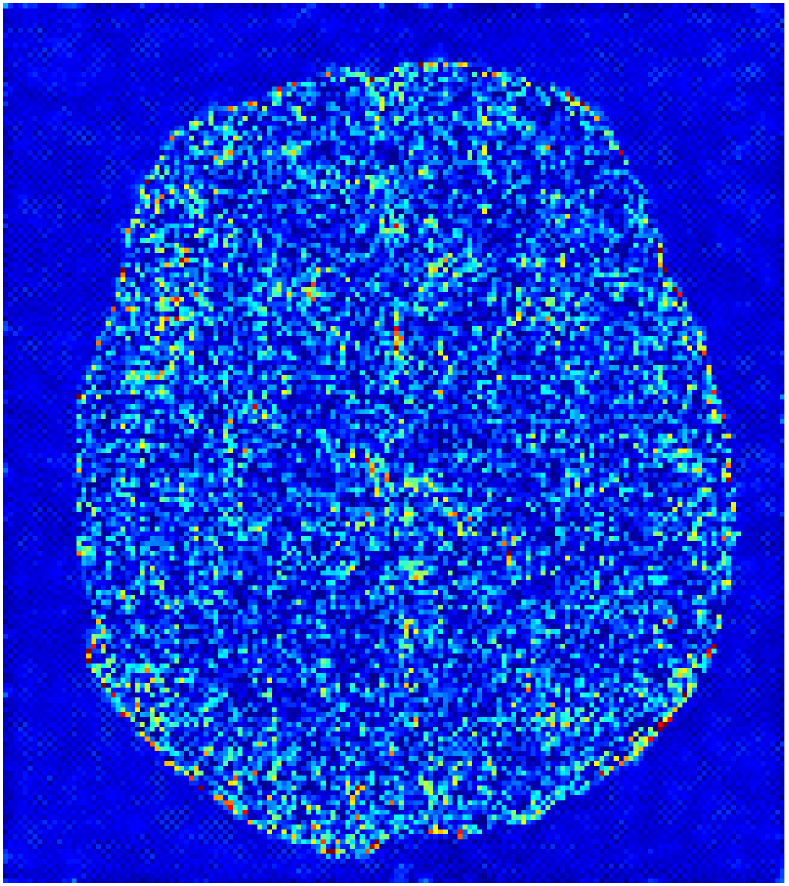}\\
\includegraphics[width=0.03\linewidth]{fig_chp5/T1CE.png}
\includegraphics[width=0.16\linewidth]{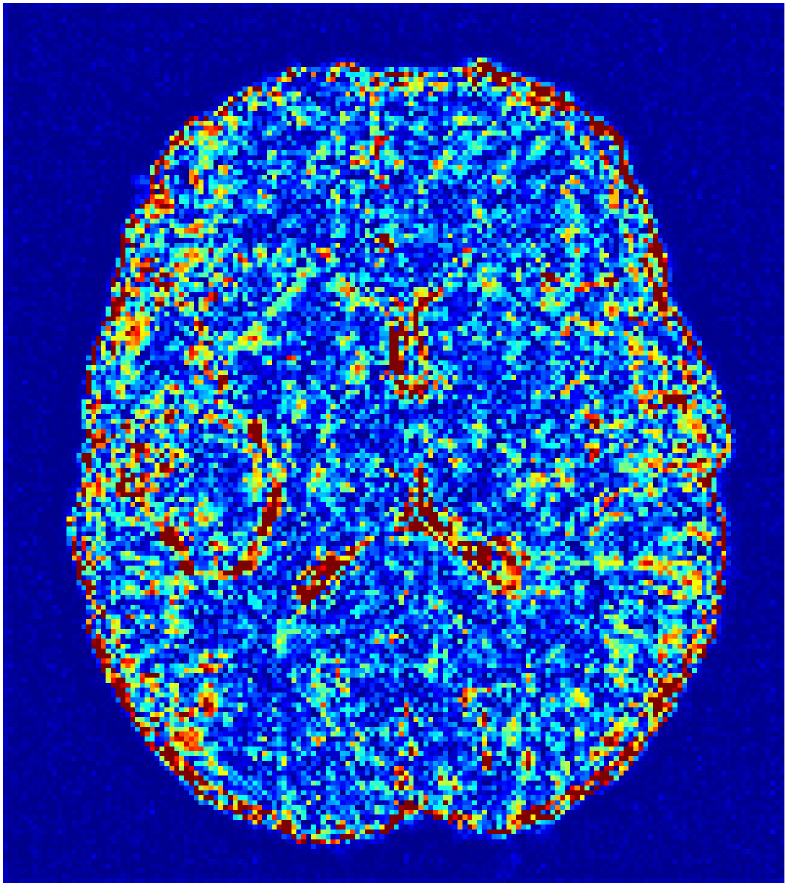}
\includegraphics[width=0.16\linewidth]{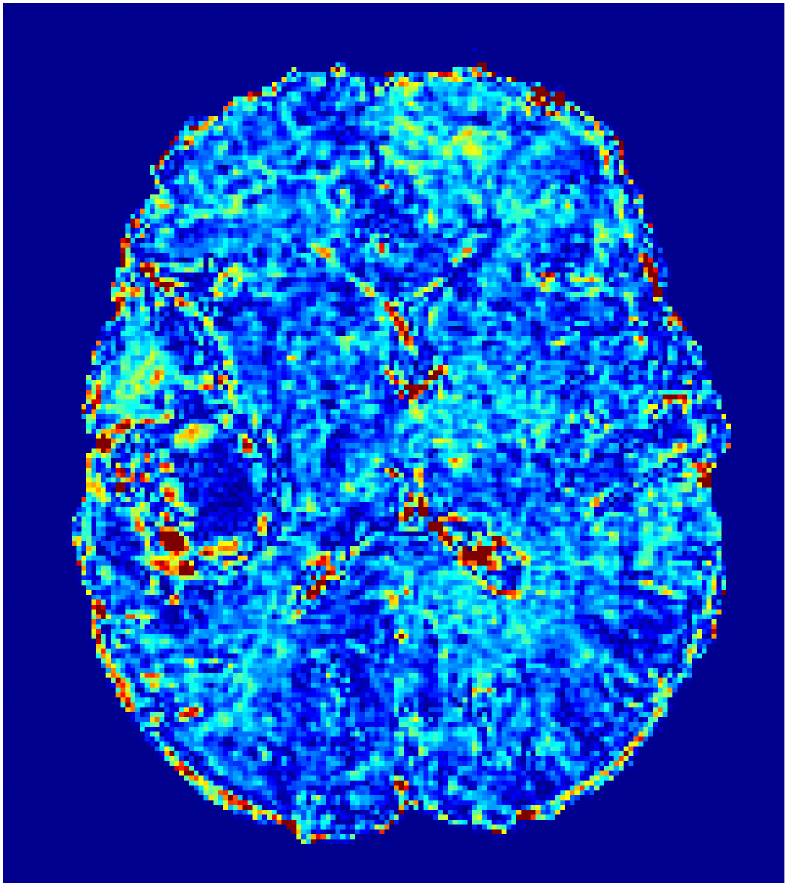}
\includegraphics[width=0.16\linewidth]{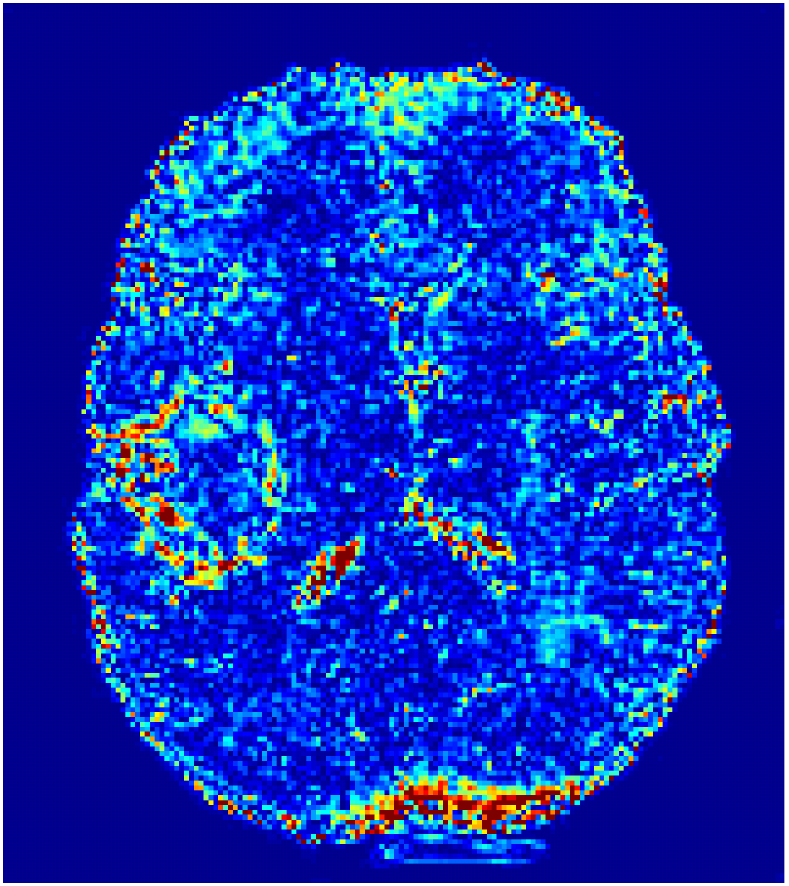}
\includegraphics[width=0.16\linewidth]{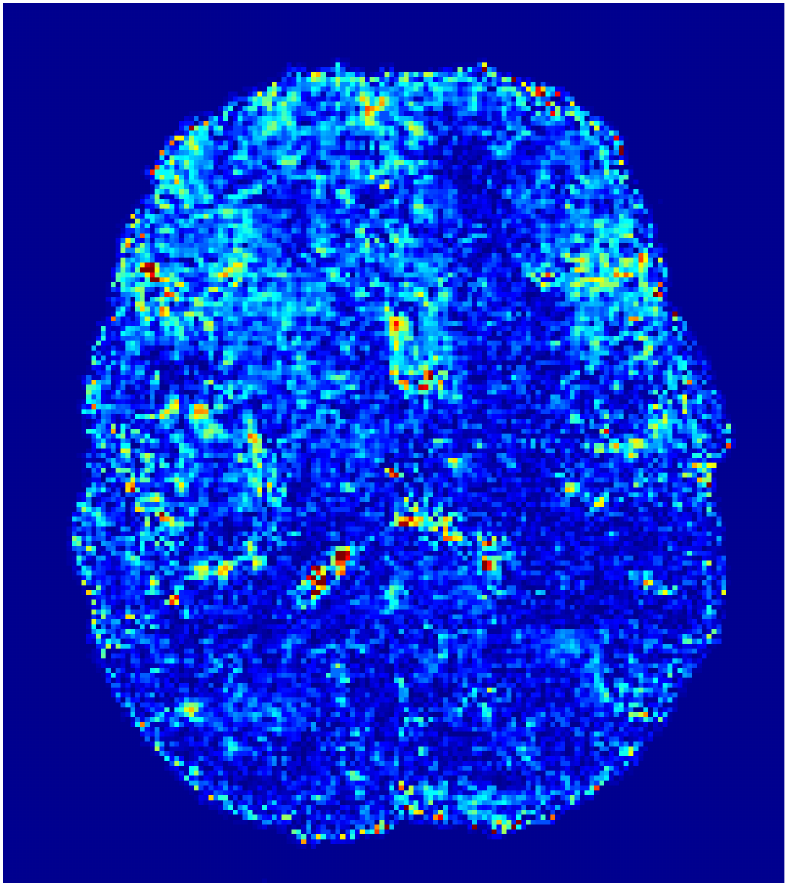}
\includegraphics[width=0.16\linewidth]{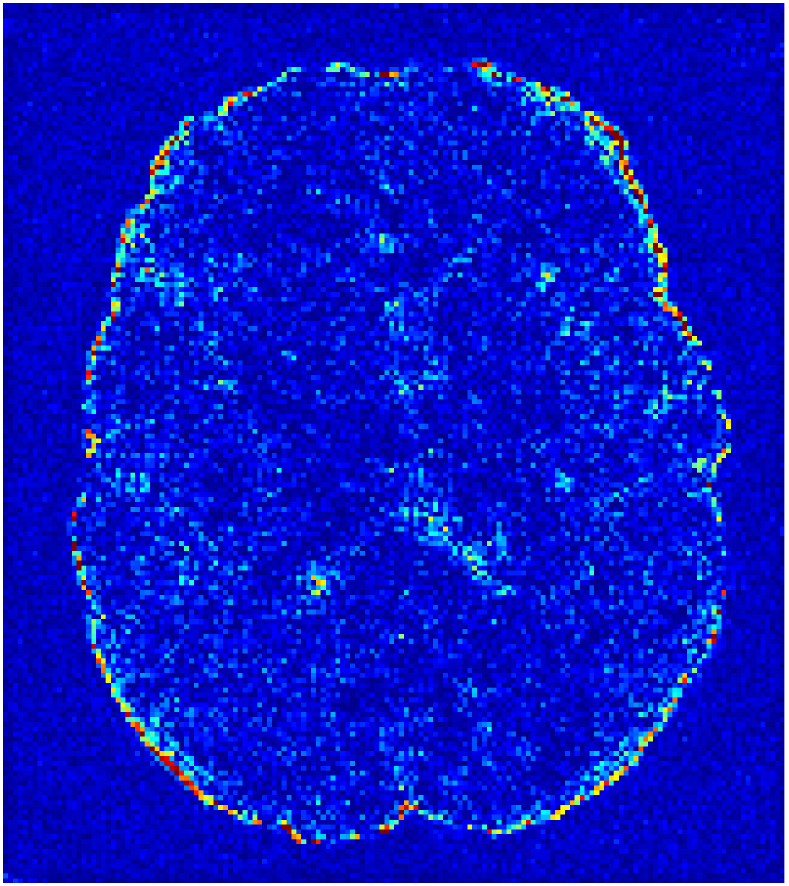}\\
\caption{Pointwise error maps between synthetic image and the its corresponding ground truth. From first row to last row: T1 $+$ T2 $\to$ FLAIR, T1 $+$ FLAIR $\to$ T2, T2 $+$ FLAIR $\to$ T1 and T1 $+$ T2 $\to$ T1CE.  }\label{fig:synthesis_error}
\end{figure}

\section{Conclusion}
We proposed a novel joint multimodal MRI reconstruction and synthesis model. The proposed model simultaneously reconstructs the source modality images from the partially scanned k-space MR data and synthesizes the target modality image without any k-space information by iterating a learnable optimization algorithm with convergence guaranteed.
Moreover, we incorporate a bilevel-optimization training algorithm with the use of both training and validation sets to further improve the performance. Extensive experiments on brain data with several different modalities validate the magnificent performance of the proposed model.

%% file: tex/TwoOrMoreAppendices.tex
% The Editorial Office Requirements for the Table of Contents cause a significant problem
%in Latex if there is only one Appendix. The Appendix is no longer labeled "A" in the TOC
%but has the word "APPENDIX" placed in front of the title of the Appendix. This can be done
%without issue IF nothing needs to be numbered by LaTeX in the Appendix. Unfortunately, most of the time
%something needs to be numbered in that single Appendix. For this reason we have included the IFTHENELSE switch
%found in this document and at the beginning of AppendixA. We assume that if you have any appendices, that you have more than one. However, you DO only have one appendix DO NOT USE THIS FILE!!!!!!!!!!!!!!!!!!!!!!!
%
% OneSingleAppendix.tex has all the settings needed to adjust for a single appendix
% you will have a major problem with your TOC if you use this file with a single appendix!!!!!

%
% Comment (or delete) all of the \input{AppendixB} commands except those you are using.
%Then open the AppendixA.tex file and continue there.

%you can add/substract individual appendices through by using the /include{appendix'X'}
% and creating/deleting the appropriate files
\appendix %
%\clearpage%

\addtocontents{toc}{\protect\addvspace{10pt}\protect\noindent \protect APPENDIX\par}

%% file: bio.tex
% Just type your bio in between the brackets
\biography{%
Wanyu Bian graduated from the University of Missouri-Columbia with a degree in mathematics
and statistics in 2017. She received her Ph.D. from the Department of Mathematics at the University of Florida in 2022. }

%% file: main.bbl
\begin{thebibliography}{100}
\providecommand{\url}[1]{\texttt{#1}}
\providecommand{\urlprefix}{URL }
\providecommand{\doi}[1]{https://doi.org/#1}

\bibitem{tensorflow2015-whitepaper}
Abadi, M., Agarwal, A., Barham, P., Brevdo, E., Chen, Z., Citro, C., Corrado,
  G.S., Davis, A., Dean, J., Devin, M., Ghemawat, S., Goodfellow, I., Harp, A.,
  Irving, G., Isard, M., Jia, Y., Jozefowicz, R., Kaiser, L., Kudlur, M.,
  Levenberg, J., Man\'{e}, D., Monga, R., Moore, S., Murray, D., Olah, C.,
  Schuster, M., Shlens, J., Steiner, B., Sutskever, I., Talwar, K., Tucker, P.,
  Vanhoucke, V., Vasudevan, V., Vi\'{e}gas, F., Vinyals, O., Warden, P.,
  Wattenberg, M., Wicke, M., Yu, Y., Zheng, X.: {TensorFlow}: Large-scale
  machine learning on heterogeneous systems (2015),
  \url{https://www.tensorflow.org/}, software available from tensorflow.org

\bibitem{abadi2016tensorflow}
Abadi, M., et~al.: Tensorflow: A system for large-scale machine learning. In:
  12th $\{$USENIX$\}$ symposium on operating systems design and implementation
  ($\{$OSDI$\}$ 16). pp. 265--283 (2016)

\bibitem{adler2018learned}
Adler, J., {\"O}ktem, O.: Learned primal-dual reconstruction. IEEE transactions
  on medical imaging  \textbf{37}(6),  1322--1332 (2018)

\bibitem{aggarwal2018modl}
Aggarwal, H.K., Mani, M.P., Jacob, M.: Modl: Model-based deep learning
  architecture for inverse problems. IEEE transactions on medical imaging
  \textbf{38}(2),  394--405 (2018)

\bibitem{Aggarwal_2019}
Aggarwal, H.K., Mani, M.P., Jacob, M.: Modl: Model-based deep learning
  architecture for inverse problems. IEEE Transactions on Medical Imaging
  \textbf{38}(2),  394–405 (Feb 2019)

\bibitem{ahishakiye2021survey}
Ahishakiye, E., Van~Gijzen, M.B., Tumwiine, J., Wario, R., Obungoloch, J.: A
  survey on deep learning in medical image reconstruction. Intelligent Medicine
   (2021)

\bibitem{andrychowicz2016learning}
Andrychowicz, M., Denil, M., Gomez, S., Hoffman, M.W., Pfau, D., Schaul, T.,
  Shillingford, B., De~Freitas, N.: Learning to learn by gradient descent by
  gradient descent. In: Advances in neural information processing systems. pp.
  3981--3989 (2016)

\bibitem{antoniou2018train}
Antoniou, A., Edwards, H., Storkey, A.: How to train your maml. arXiv preprint
  arXiv:1810.09502  (2018)

\bibitem{balaji2018metareg}
Balaji, Y., Sankaranarayanan, S., Chellappa, R.: Metareg: Towards domain
  generalization using meta-regularization. Advances in Neural Information
  Processing Systems  \textbf{31},  998--1008 (2018)

\bibitem{beck2009fast}
Beck, A., Teboulle, M.: A fast iterative shrinkage-thresholding algorithm for
  linear inverse problems. SIAM journal on imaging sciences  \textbf{2}(1),
  183--202 (2009)

\bibitem{ben2007analysis}
Ben-David, S., Blitzer, J., Crammer, K., Pereira, F., et~al.: Analysis of
  representations for domain adaptation. Advances in neural information
  processing systems  \textbf{19}, ~137 (2007)

\bibitem{279181}
{Bengio}, Y., {Simard}, P., {Frasconi}, P.: Learning long-term dependencies
  with gradient descent is difficult. IEEE Transactions on Neural Networks
  \textbf{5}(2),  157--166 (1994). \doi{10.1109/72.279181}

\bibitem{bernstein2001effect}
Bernstein, M.A., Fain, S.B., Riederer, S.J.: Effect of windowing and
  zero-filled reconstruction of {MRI} data on spatial resolution and
  acquisition strategy. Journal of Magnetic Resonance Imaging: An Official
  Journal of the International Society for Magnetic Resonance in Medicine
  \textbf{14}(3),  270--280 (2001)

\bibitem{bian2020deep}
Bian, W., Chen, Y., Ye, X.: Deep parallel mri reconstruction network without
  coil sensitivities. In: International Workshop on Machine Learning for
  Medical Image Reconstruction. pp. 17--26. Springer (2020)

\bibitem{10.1007/978-3-030-61598-7_2}
Bian, W., Chen, Y., Ye, X.: Deep parallel mri reconstruction network without
  coil sensitivities. In: Deeba, F., Johnson, P., W{\"u}rfl, T., Ye, J.C.
  (eds.) Machine Learning for Medical Image Reconstruction. pp. 17--26.
  Springer International Publishing, Cham (2020)

\bibitem{bian2022optimal}
Bian, W., Chen, Y., Ye, X.: An optimal control framework for joint-channel
  parallel mri reconstruction without coil sensitivities. Magnetic Resonance
  Imaging  \textbf{89},  1--11 (2022)

\bibitem{jimaging7110231}
Bian, W., Chen, Y., Ye, X., Zhang, Q.: An optimization-based meta-learning
  model for mri reconstruction with diverse dataset. Journal of Imaging
  \textbf{7}(11) (2021)

\bibitem{block2007undersampled}
Block, K.T., Uecker, M., Frahm, J.: Undersampled radial mri with multiple
  coils. iterative image reconstruction using a total variation constraint.
  Magnetic Resonance in Medicine: An Official Journal of the International
  Society for Magnetic Resonance in Medicine  \textbf{57}(6),  1086--1098
  (2007)

\bibitem{661180}
Blomgren, P., Chan, T.: Color tv: total variation methods for restoration of
  vector-valued images. IEEE Transactions on Image Processing  \textbf{7}(3),
  304--309 (1998). \doi{10.1109/83.661180}

\bibitem{boyd2011distributed}
Boyd, S., Parikh, N., Chu, E.: Distributed optimization and statistical
  learning via the alternating direction method of multipliers. Now Publishers
  Inc (2011)

\bibitem{bui2020flow}
Bui, T.D., Nguyen, M., Le, N., Luu, K.: Flow-based deformation guidance for
  unpaired multi-contrast mri image-to-image translation. In: International
  Conference on Medical Image Computing and Computer-Assisted Intervention. pp.
  728--737. Springer (2020)

\bibitem{burgos2014attenuation}
Burgos, N., Cardoso, M.J., Thielemans, K., Modat, M., Pedemonte, S., Dickson,
  J., Barnes, A., Ahmed, R., Mahoney, C.J., Schott, J.M., et~al.: Attenuation
  correction synthesis for hybrid pet-mr scanners: application to brain
  studies. IEEE transactions on medical imaging  \textbf{33}(12),  2332--2341
  (2014)

\bibitem{chambolle2011first}
Chambolle, A., Pock, T.: A first-order primal-dual algorithm for convex
  problems with applications to imaging. Journal of mathematical imaging and
  vision  \textbf{40}(1),  120--145 (2011)

\bibitem{chandra2021deep}
Chandra, S.S., Bran~Lorenzana, M., Liu, X., Liu, S., Bollmann, S., Crozier, S.:
  Deep learning in magnetic resonance image reconstruction. Journal of Medical
  Imaging and Radiation Oncology  (2021)

\bibitem{chartsias2017multimodal}
Chartsias, A., Joyce, T., Giuffrida, M.V., Tsaftaris, S.A.: Multimodal mr
  synthesis via modality-invariant latent representation. IEEE transactions on
  medical imaging  \textbf{37}(3),  803--814 (2017)

\bibitem{chen2014exploiting}
Chen, C., Huang, J.: Exploiting the wavelet structure in compressed sensing
  mri. Magnetic resonance imaging  \textbf{32}(10),  1377--1389 (2014)

\bibitem{chen2020mri}
Chen, E.Z., Chen, T., Sun, S.: Mri image reconstruction via learning
  optimization using neural odes. In: International Conference on Medical Image
  Computing and Computer-Assisted Intervention. pp. 83--93. Springer (2020)

\bibitem{chen2018neural}
Chen, R.T., Rubanova, Y., Bettencourt, J., Duvenaud, D.: Neural ordinary
  differential equations. arXiv preprint arXiv:1806.07366  (2018)

\bibitem{chen19closerfewshot}
Chen, W.Y., Liu, Y.C., Kira, Z., Wang, Y.C., Huang, J.B.: A closer look at
  few-shot classification. In: International Conference on Learning
  Representations (2019)

\bibitem{chen2019one}
Chen, X., Lian, C., Wang, L., Deng, H., Fung, S.H., Nie, D., Thung, K.H., Yap,
  P.T., Gateno, J., Xia, J.J., et~al.: One-shot generative adversarial learning
  for mri segmentation of craniomaxillofacial bony structures. IEEE
  transactions on medical imaging  \textbf{39}(3),  787--796 (2019)

\bibitem{10.1007/978-3-030-32248-9_4}
Chen, Y., et~al.: Model-based convolutional de-aliasing network learning for
  parallel mr imaging. In: Medical Image Computing and Computer Assisted
  Intervention -- MICCAI 2019. pp. 30--38. Springer International Publishing,
  Cham (2019)

\bibitem{chen2020learnable}
Chen, Y., Liu, H., Ye, X., Zhang, Q.: Learnable descent algorithm for nonsmooth
  nonconvex image reconstruction. SIAM Journal on Imaging Sciences
  \textbf{14}(4),  1532--1564 (2021). \doi{10.1137/20M1353368}

\bibitem{chen2021learnable}
Chen, Y., Liu, H., Ye, X., Zhang, Q.: Learnable descent algorithm for nonsmooth
  nonconvex image reconstruction. SIAM Journal on Imaging Sciences
  \textbf{14}(4),  1532--1564 (2021)

\bibitem{chen2021variational}
Chen, Y., Ye, X., Zhang, Q.: Variational model-based deep neural networks for
  image reconstruction. Handbook of Mathematical Models and Algorithms in
  Computer Vision and Imaging: Mathematical Imaging and Vision pp. 1--29 (2021)

\bibitem{cheng2019model}
Cheng, J., et~al.: Model learning: Primal dual networks for fast mr imaging.
  In: International Conference on Medical Image Computing and Computer-Assisted
  Intervention. pp. 21--29. Springer (2019)

\bibitem{MSA}
Chernousko, F.L., Lyubushin, A.A.: Method of successive approximations for
  solution of optimal control problems. Optimal Control Applications and
  Methods  \textbf{3}(2),  101--114 (1982)

\bibitem{cole2020analysis}
Cole, E.K., et~al.: Analysis of deep complex-valued convolutional neural
  networks for mri reconstruction. arXiv:2004.01738  (2020),
  \url{https://arxiv.org/abs/2004.01738}

\bibitem{RSS_noise}
Constantinides, C.D., Atalar, E., McVeigh, E.R.: Signal-to-noise measurements
  in magnitude images from nmr phased arrays. Magnetic Resonance in Medicine
  \textbf{38}(5),  852--857 (1997). \doi{10.1002/mrm.1910380524}

\bibitem{cordier2016extended}
Cordier, N., Delingette, H., L{\^e}, M., Ayache, N.: Extended modality
  propagation: image synthesis of pathological cases. IEEE transactions on
  medical imaging  \textbf{35}(12),  2598--2608 (2016)

\bibitem{dai2019compressed}
Dai, Y., Zhuang, P.: Compressed sensing mri via a multi-scale dilated residual
  convolution network. Magnetic resonance imaging  \textbf{63},  93--104 (2019)

\bibitem{dar2019image}
Dar, S.U., Yurt, M., Karacan, L., Erdem, A., Erdem, E., {\c{C}}ukur, T.: Image
  synthesis in multi-contrast mri with conditional generative adversarial
  networks. IEEE transactions on medical imaging  \textbf{38}(10),  2375--2388
  (2019)

\bibitem{dar2020prior}
Dar, S.U., Yurt, M., Shahdloo, M., Ild{\i}z, M.E., T{\i}naz, B., {\c{C}}ukur,
  T.: Prior-guided image reconstruction for accelerated multi-contrast mri via
  generative adversarial networks. IEEE Journal of Selected Topics in Signal
  Processing  \textbf{14}(6),  1072--1087 (2020)

\bibitem{day2017survey}
Day, O., Khoshgoftaar, T.M.: A survey on heterogeneous transfer learning.
  Journal of Big Data  \textbf{4}(1),  1--42 (2017)

\bibitem{dedmari2018complex}
Dedmari, M.A., et~al.: Complex fully convolutional neural networks for mr image
  reconstruction. In: International Workshop on Machine Learning for Medical
  Image Reconstruction. pp. 30--38. Springer (2018)

\bibitem{doi:10.1002/jmri.23639}
Deshmane, A., et~al.: Parallel mr imaging. Journal of Magnetic Resonance
  Imaging  \textbf{36}(1),  55--72 (2012)

\bibitem{dong2014compressive}
Dong, W., Shi, G., Li, X., Ma, Y., Huang, F.: Compressive sensing via nonlocal
  low-rank regularization. IEEE transactions on image processing
  \textbf{23}(8),  3618--3632 (2014)

\bibitem{donoho2006compressed}
Donoho, D.L.: Compressed sensing. IEEE Transactions on information theory
  \textbf{52}(4),  1289--1306 (2006)

\bibitem{10.1007/978-3-030-32251-9_78}
Duan, J., et~al.: Vs-net: Variable splitting network for accelerated parallel
  mri reconstruction. In: Medical Image Computing and Computer Assisted
  Intervention -- MICCAI 2019. pp. 713--722. Springer International Publishing,
  Cham (2019)

\bibitem{dumoulin2016guide}
Dumoulin, V., Visin, F.: A guide to convolution arithmetic for deep learning.
  arXiv preprint arXiv:1603.07285  (2016)

\bibitem{eksioglu2016decoupled}
Eksioglu, E.M.: Decoupled algorithm for mri reconstruction using nonlocal block
  matching model: Bm3d-mri. Journal of Mathematical Imaging and Vision
  \textbf{56}(3),  430--440 (2016)

\bibitem{doi:10.1002/mrm.27201}
Eo, T., et~al.: Kiki-net: cross-domain convolutional neural networks for
  reconstructing undersampled magnetic resonance images. Magnetic Resonance in
  Medicine  \textbf{80}(5),  2188--2201 (2018). \doi{10.1002/mrm.27201}

\bibitem{esser2010general}
Esser, E., Zhang, X., Chan, T.F.: A general framework for a class of first
  order primal-dual algorithms for convex optimization in imaging science. SIAM
  Journal on Imaging Sciences  \textbf{3}(4),  1015--1046 (2010)

\bibitem{fan2014challenges}
Fan, J., Han, F., Liu, H.: Challenges of big data analysis. National science
  review  \textbf{1}(2),  293--314 (2014)

\bibitem{finn2017model}
Finn, C., Abbeel, P., Levine, S.: Model-agnostic meta-learning for fast
  adaptation of deep networks. In: International Conference on Machine
  Learning. pp. 1126--1135. PMLR (2017)

\bibitem{finn2019online}
Finn, C., Rajeswaran, A., Kakade, S., Levine, S.: Online meta-learning. In:
  International Conference on Machine Learning. pp. 1920--1930. PMLR (2019)

\bibitem{finn2018probabilistic}
Finn, C., Xu, K., Levine, S.: Probabilistic model-agnostic meta-learning. arXiv
  preprint arXiv:1806.02817  (2018)

\bibitem{franceschi2017forward}
Franceschi, L., Donini, M., Frasconi, P., Pontil, M.: Forward and reverse
  gradient-based hyperparameter optimization. In: International Conference on
  Machine Learning. pp. 1165--1173. PMLR (2017)

\bibitem{franceschi2018bilevel}
Franceschi, L., Frasconi, P., Salzo, S., Grazzi, R., Pontil, M.: Bilevel
  programming for hyperparameter optimization and meta-learning. In:
  International Conference on Machine Learning. pp. 1568--1577. PMLR (2018)

\bibitem{pmlr-v9-glorot10a}
Glorot, X., Bengio, Y.: Understanding the difficulty of training deep
  feedforward neural networks. In: Teh, Y.W., Titterington, M. (eds.)
  Proceedings of the Thirteenth International Conference on Artificial
  Intelligence and Statistics. Proceedings of Machine Learning Research,
  vol.~9, pp. 249--256. PMLR, Chia Laguna Resort, Sardinia, Italy (13--15 May
  2010)

\bibitem{glorot2010understanding}
Glorot, X., Bengio, Y.: Understanding the difficulty of training deep
  feedforward neural networks. In: Proceedings of the thirteenth international
  conference on artificial intelligence and statistics. pp. 249--256 (2010)

\bibitem{Glorot10understandingthe}
Glorot, X., Bengio, Y.: Understanding the difficulty of training deep
  feedforward neural networks. In: In Proceedings of the International
  Conference on Artificial Intelligence and Statistics. Society for Artificial
  Intelligence and Statistics (2010)

\bibitem{goldstein2009split}
Goldstein, T., Osher, S.: The split bregman method for l1-regularized problems.
  SIAM journal on imaging sciences  \textbf{2}(2),  323--343 (2009)

\bibitem{grant2018recasting}
Grant, E., Finn, C., Levine, S., Darrell, T., Griffiths, T.: Recasting
  gradient-based meta-learning as hierarchical bayes. arXiv preprint
  arXiv:1801.08930  (2018)

\bibitem{graves2014neural}
Graves, A., Wayne, G., Danihelka, I.: Neural turing machines. arXiv preprint
  arXiv:1410.5401  (2014)

\bibitem{griswold2002generalized}
Griswold, M.A., et~al.: Generalized autocalibrating partially parallel
  acquisitions (grappa). Magnetic Resonance in Medicine: An Official Journal of
  the International Society for Magnetic Resonance in Medicine  \textbf{47}(6),
   1202--1210 (2002)

\bibitem{guerquin2011fast}
Guerquin-Kern, M., Haberlin, M., Pruessmann, K.P., Unser, M.: A fast
  wavelet-based reconstruction method for magnetic resonance imaging. IEEE
  transactions on medical imaging  \textbf{30}(9),  1649--1660 (2011)

\bibitem{haldar2010compressed}
Haldar, J.P., Hernando, D., Liang, Z.P.: Compressed-sensing mri with random
  encoding. IEEE transactions on Medical Imaging  \textbf{30}(4),  893--903
  (2010)

\bibitem{PMP}
Halkin, H.: A maximum principle of the pontryagin type for systems described by
  nonlinear difference equations. SIAM Journal on Control  \textbf{4}(1),
  90--111 (1966). \doi{10.1137/0304009}

\bibitem{hammernik2018learning}
Hammernik, K., Klatzer, T., Kobler, E., Recht, M.P., Sodickson, D.K., Pock, T.,
  Knoll, F.: Learning a variational network for reconstruction of accelerated
  mri data. Magnetic resonance in medicine  \textbf{79}(6),  3055--3071 (2018)

\bibitem{hammernik2017l2}
Hammernik, K., Knoll, F., Sodickson, D.K., Pock, T.: L2 or not l2: impact of
  loss function design for deep learning mri reconstruction. In: ISMRM 25th
  Annual Meeting. p.~0687 (2017)

\bibitem{hammernik2021systematic}
Hammernik, K., Schlemper, J., Qin, C., Duan, J., Summers, R.M., Rueckert, D.:
  Systematic evaluation of iterative deep neural networks for fast parallel mri
  reconstruction with sensitivity-weighted coil combination. Magnetic Resonance
  in Medicine  (2021)

\bibitem{doi:10.1002/mrm.26977}
Hammernik, K., et~al.: Learning a variational network for reconstruction of
  accelerated mri data. Magnetic Resonance in Medicine  \textbf{79}(6),
  3055--3071 (2018)

\bibitem{7780459}
{He}, K., {Zhang}, X., {Ren}, S., {Sun}, J.: Deep residual learning for image
  recognition. In: 2016 IEEE Conference on Computer Vision and Pattern
  Recognition (CVPR). pp. 770--778 (2016). \doi{10.1109/CVPR.2016.90}

\bibitem{he2016deep}
He, K., Zhang, X., Ren, S., Sun, J.: Deep residual learning for image
  recognition. In: Proceedings of the IEEE conference on computer vision and
  pattern recognition. pp. 770--778 (2016)

\bibitem{heide2014flexisp}
Heide, F., Steinberger, M., Tsai, Y.T., Rouf, M., Paj{a}k, D., Reddy, D.,
  Gallo, O., Liu, J., Heidrich, W., Egiazarian, K., et~al.: Flexisp: A flexible
  camera image processing framework. ACM Transactions on Graphics (ToG)
  \textbf{33}(6),  1--13 (2014)

\bibitem{VanishingGradient}
Hochreiter, S.: The vanishing gradient problem during learning recurrent neural
  nets and problem solutions. International Journal of Uncertainty, Fuzziness
  and Knowledge-Based Systems  \textbf{06}(02),  107--116 (1998).
  \doi{10.1142/S0218488598000094}

\bibitem{hore2010image}
Hore, A., Ziou, D.: Image quality metrics: Psnr vs. ssim. In: 2010 20th
  international conference on pattern recognition. pp. 2366--2369. IEEE (2010)

\bibitem{hospedales2021meta}
Hospedales, T.M., Antoniou, A., Micaelli, P., Storkey, A.J.: Meta-learning in
  neural networks: A survey. IEEE Transactions on Pattern Analysis and Machine
  Intelligence  (2021)

\bibitem{hosseini2020dense}
Hosseini, S.A.H., Yaman, B., Moeller, S., Hong, M., Ak{\c{c}}akaya, M.: Dense
  recurrent neural networks for accelerated mri: History-cognizant unrolling of
  optimization algorithms. IEEE Journal of Selected Topics in Signal Processing
   \textbf{14}(6),  1280--1291 (2020)

\bibitem{huang2020medical}
Huang, F.: Medical imaging using neural networks (Dec~24 2020), uS Patent App.
  16/904,981

\bibitem{huang2020mri}
Huang, F., Chen, M.: Magnetic resonance imaging method and device (Dec~1 2020),
  uS Patent 10,852,376

\bibitem{huang2020magnetic}
Huang, F., Han, D., Mei, L.: Magnetic resonance imaging with deep neutral
  networks (Jul~9 2020), uS Patent App. 16/735,874

\bibitem{huang2014fast}
Huang, J., Chen, C., Axel, L.: Fast multi-contrast mri reconstruction. Magnetic
  resonance imaging  \textbf{32}(10),  1344--1352 (2014)

\bibitem{huang2014bayesian}
Huang, Y., Paisley, J., Lin, Q., Ding, X., Fu, X., Zhang, X.P.: Bayesian
  nonparametric dictionary learning for compressed sensing mri. IEEE
  Transactions on Image Processing  \textbf{23}(12),  5007--5019 (2014)

\bibitem{huisman2021survey}
Huisman, M., van Rijn, J.N., Plaat, A.: A survey of deep meta-learning.
  Artificial Intelligence Review pp. 1--59 (2021)

\bibitem{huo2018synseg}
Huo, Y., Xu, Z., Moon, H., Bao, S., Assad, A., Moyo, T.K., Savona, M.R.,
  Abramson, R.G., Landman, B.A.: Synseg-net: Synthetic segmentation without
  target modality ground truth. IEEE transactions on medical imaging
  \textbf{38}(4),  1016--1025 (2018)

\bibitem{iglesias2013synthesizing}
Iglesias, J.E., Konukoglu, E., Zikic, D., Glocker, B., Van~Leemput, K., Fischl,
  B.: Is synthesizing mri contrast useful for inter-modality analysis? In:
  International Conference on Medical Image Computing and Computer-Assisted
  Intervention. pp. 631--638. Springer (2013)

\bibitem{islam2021compressed}
Islam, S.R., Maity, S.P., Ray, A.K.: Compressed sensing regularized
  calibrationless parallel magnetic resonance imaging via deep learning.
  Biomedical Signal Processing and Control  \textbf{66},  102399 (2021)

\bibitem{jenni2018deep}
Jenni, S., Favaro, P.: Deep bilevel learning. In: Proceedings of the European
  conference on computer vision (ECCV). pp. 618--633 (2018)

\bibitem{10.1007/978-3-030-59713-9_34}
Jiang, J., et~al.: Unified cross-modality feature disentangler for unsupervised
  multi-domain mri abdomen organs segmentation. In: Medical Image Computing and
  Computer Assisted Intervention -- MICCAI 2020. pp. 347--358. Springer
  International Publishing, Cham (2020)

\bibitem{jog2015mr}
Jog, A., Carass, A., Roy, S., Pham, D.L., Prince, J.L.: Mr image synthesis by
  contrast learning on neighborhood ensembles. Medical image analysis
  \textbf{24}(1),  63--76 (2015)

\bibitem{jog2013magnetic}
Jog, A., Roy, S., Carass, A., Prince, J.L.: Magnetic resonance image synthesis
  through patch regression. In: 2013 IEEE 10th International Symposium on
  Biomedical Imaging. pp. 350--353. IEEE (2013)

\bibitem{kingma2014adam}
Kingma, D.P., Ba, J.: Adam: A method for stochastic optimization. arXiv
  preprint arXiv:1412.6980  (2014)

\bibitem{knoll2020deep}
Knoll, F., Hammernik, K., Zhang, C., Moeller, S., Pock, T., Sodickson, D.K.,
  Akcakaya, M.: Deep-learning methods for parallel magnetic resonance imaging
  reconstruction: A survey of the current approaches, trends, and issues. IEEE
  signal processing magazine  \textbf{37}(1),  128--140 (2020)

\bibitem{koch2015siamese}
Koch, G., Zemel, R., Salakhutdinov, R., et~al.: Siamese neural networks for
  one-shot image recognition. In: ICML deep learning workshop. vol.~2. Lille
  (2015)

\bibitem{doi:10.1002/mp.12600}
Kwon, K., Kim, D., Park, H.: A parallel mr imaging method using multilayer
  perceptron. Medical Physics  \textbf{44}(12),  6209--6224 (2017)

\bibitem{larkman2007parallel}
Larkman, D.J., Nunes, R.G.: Parallel magnetic resonance imaging. Physics in
  Medicine \& Biology  \textbf{52}(7), ~R15 (2007)

\bibitem{lecun2015deep}
LeCun, Y., Bengio, Y., Hinton, G.: Deep learning. nature  \textbf{521}(7553),
  436--444 (2015)

\bibitem{lee2018deep}
Lee, D., Yoo, J., Tak, S., Ye, J.C.: Deep residual learning for accelerated mri
  using magnitude and phase networks. IEEE Transactions on Biomedical
  Engineering  \textbf{65}(9),  1985--1995 (2018)

\bibitem{lee2017deep}
Lee, D., Yoo, J., Ye, J.C.: Deep residual learning for compressed sensing mri.
  In: 2017 IEEE 14th International Symposium on Biomedical Imaging (ISBI 2017).
  pp. 15--18. IEEE (2017)

\bibitem{lee2018gradient}
Lee, Y., Choi, S.: Gradient-based meta-learning with learned layerwise metric
  and subspace. In: International Conference on Machine Learning. pp.
  2927--2936. PMLR (2018)

\bibitem{li2018learning}
Li, D., Yang, Y., Song, Y.Z., Hospedales, T.M.: Learning to generalize:
  Meta-learning for domain generalization. In: Thirty-Second AAAI Conference on
  Artificial Intelligence (2018)

\bibitem{CD-SFCRF}
Li, E., Khalvati, F., Shafiee, M.J., Haider, M.A., Wong, A.: Sparse
  reconstruction of compressive sensing mri using cross-domain stochastically
  fully connected conditional random fields. BMC Medical Imaging
  \textbf{16}(1),  1--12 (2016)

\bibitem{li2016learning}
Li, K., Malik, J.: Learning to optimize. arXiv preprint arXiv:1606.01885
  (2016)

\bibitem{lidynamical}
Li, Q.: Dynamical systems and machine learning. Summer School, Peking
  University  (2020)

\bibitem{li2017maximum}
Li, Q., Chen, L., Tai, C., et~al.: Maximum principle based algorithms for deep
  learning. arXiv preprint arXiv:1710.09513  (2017)

\bibitem{pmlr}
Li, Q., Hao, S.: An optimal control approach to deep learning and applications
  to discrete-weight neural networks. In: Dy, J., Krause, A. (eds.) Proceedings
  of the 35th International Conference on Machine Learning. Proceedings of
  Machine Learning Research, vol.~80, pp. 2985--2994. PMLR (10--15 Jul 2018)

\bibitem{li2019deep}
Li, Q., Lin, T., Shen, Z.: Deep learning via dynamical systems: An
  approximation perspective. arXiv preprint arXiv:1912.10382  (2019)

\bibitem{li2019feature}
Li, Y., Yang, Y., Zhou, W., Hospedales, T.: Feature-critic networks for
  heterogeneous domain generalization. In: International Conference on Machine
  Learning. pp. 3915--3924. PMLR (2019)

\bibitem{li2017meta}
Li, Z., Zhou, F., Chen, F., Li, H.: Meta-sgd: Learning to learn quickly for
  few-shot learning. arXiv preprint arXiv:1707.09835  (2017)

\bibitem{liang2020deep}
Liang, D., Cheng, J., Ke, Z., Ying, L.: Deep magnetic resonance image
  reconstruction: Inverse problems meet neural networks. IEEE Signal Processing
  Magazine  \textbf{37}(1),  141--151 (2020)

\bibitem{lions1979splitting}
Lions, P.L., Mercier, B.: Splitting algorithms for the sum of two nonlinear
  operators. SIAM Journal on Numerical Analysis  \textbf{16}(6),  964--979
  (1979)

\bibitem{liu2020deep}
Liu, R., Zhang, Y., Cheng, S., Luo, Z., Fan, X.: A deep framework assembling
  principled modules for cs-mri: Unrolling perspective, convergence behaviors,
  and practical modeling. IEEE Transactions on Medical Imaging
  \textbf{39}(12),  4150--4163 (2020)

\bibitem{liu2020multimodal}
Liu, X., Yu, A., Wei, X., Pan, Z., Tang, J.: Multimodal mr image synthesis
  using gradient prior and adversarial learning. IEEE Journal of Selected
  Topics in Signal Processing  \textbf{14}(6),  1176--1188 (2020)

\bibitem{lu2020pfista}
Lu, T., Zhang, X., Huang, Y., Guo, D., Huang, F., Xu, Q., Hu, Y., Ou-Yang, L.,
  Lin, J., Yan, Z., et~al.: pfista-sense-resnet for parallel mri
  reconstruction. Journal of Magnetic Resonance  \textbf{318},  106790 (2020)

\bibitem{lundervold2019overview}
Lundervold, A.S., Lundervold, A.: An overview of deep learning in medical
  imaging focusing on mri. Zeitschrift f{\"u}r Medizinische Physik
  \textbf{29}(2),  102--127 (2019)

\bibitem{lustig2007sparse}
Lustig, M., Donoho, D., Pauly, J.M.: Sparse mri: The application of compressed
  sensing for rapid mr imaging. Magnetic Resonance in Medicine: An Official
  Journal of the International Society for Magnetic Resonance in Medicine
  \textbf{58}(6),  1182--1195 (2007)

\bibitem{lustig2010spirit}
Lustig, M., Pauly, J.M.: Spirit: iterative self-consistent parallel imaging
  reconstruction from arbitrary k-space. Magnetic resonance in medicine
  \textbf{64}(2),  457--471 (2010)

\bibitem{doi:10.1002/mrm.22428}
Lustig, M., Pauly, J.M.: Spirit: Iterative self-consistent parallel imaging
  reconstruction from arbitrary k-space. Magnetic Resonance in Medicine
  \textbf{64}(2),  457--471 (2010). \doi{10.1002/mrm.22428}

\bibitem{lv2021pic}
Lv, J., Wang, C., Yang, G.: Pic-gan: A parallel imaging coupled generative
  adversarial network for accelerated multi-channel mri reconstruction.
  Diagnostics  \textbf{11}(1), ~61 (2021)

\bibitem{lyra2012improved}
Lyra-Leite, D.M., da~Costa, J.P.C.L., de~Carvalho, J.L.A.: Improved mri
  reconstruction and denoising using svd-based low-rank approximation. In: 2012
  Workshop on Engineering Applications. pp.~1--6. IEEE (2012)

\bibitem{lysaker2003noise}
Lysaker, M., Lundervold, A., Tai, X.C.: Noise removal using fourth-order
  partial differential equation with applications to medical magnetic resonance
  images in space and time. IEEE Transactions on image processing
  \textbf{12}(12),  1579--1590 (2003)

\bibitem{8417964}
{Mardani}, M., et~al.: Deep generative adversarial neural networks for
  compressive sensing mri. IEEE Transactions on Medical Imaging
  \textbf{38}(1),  167--179 (Jan 2019)

\bibitem{mccann2017convolutional}
McCann, M.T., Jin, K.H., Unser, M.: Convolutional neural networks for inverse
  problems in imaging: A review. IEEE Signal Processing Magazine
  \textbf{34}(6),  85--95 (2017)

\bibitem{mehra2019penalty}
Mehra, A., Hamm, J.: Penalty method for inversion-free deep bilevel
  optimization. arXiv preprint arXiv:1911.03432  (2019)

\bibitem{meinhardt2017learning}
Meinhardt, T., et~al.: Learning proximal operators: Using denoising networks
  for regularizing inverse imaging problems. In: Proceedings of the IEEE
  International Conference on Computer Vision. pp. 1781--1790 (2017)

\bibitem{10.1007/978-3-030-32251-9_80}
Meng, N., et~al.: A prior learning network for joint image and sensitivity
  estimation in parallel mr imaging. In: Medical Image Computing and Computer
  Assisted Intervention -- MICCAI 2019. pp. 732--740. Springer International
  Publishing, Cham (2019)

\bibitem{menze2014multimodal}
Menze, B.H., Jakab, A., Bauer, S., Kalpathy-Cramer, J., Farahani, K., Kirby,
  J., Burren, Y., Porz, N., Slotboom, J., Wiest, R., et~al.: The multimodal
  brain tumor image segmentation benchmark (brats). IEEE transactions on
  medical imaging  \textbf{34}(10),  1993--2024 (2014)

\bibitem{micaelli2020non}
Micaelli, P., Storkey, A.: Non-greedy gradient-based hyperparameter
  optimization over long horizons. arXiv preprint arXiv:2007.07869  (2020)

\bibitem{miller1993mathematical}
Miller, M.I., Christensen, G.E., Amit, Y., Grenander, U.: Mathematical textbook
  of deformable neuroanatomies. Proceedings of the National Academy of Sciences
   \textbf{90}(24),  11944--11948 (1993)

\bibitem{mishra2017simple}
Mishra, N., Rohaninejad, M., Chen, X., Abbeel, P.: A simple neural attentive
  meta-learner. arXiv preprint arXiv:1707.03141  (2017)

\bibitem{monga2021algorithm}
Monga, V., Li, Y., Eldar, Y.C.: Algorithm unrolling: Interpretable, efficient
  deep learning for signal and image processing. IEEE Signal Processing
  Magazine  \textbf{38}(2),  18--44 (2021)

\bibitem{munkhdalai2017meta}
Munkhdalai, T., Yu, H.: Meta networks. In: International Conference on Machine
  Learning. pp. 2554--2563. PMLR (2017)

\bibitem{nichol2018first}
Nichol, A., Achiam, J., Schulman, J.: On first-order meta-learning algorithms.
  arXiv preprint arXiv:1803.02999  (2018)

\bibitem{nichol2018reptile}
Nichol, A., Schulman, J.: Reptile: a scalable metalearning algorithm. arXiv
  preprint arXiv:1803.02999  \textbf{2}(3), ~4 (2018)

\bibitem{10.1007/978-3-030-59713-9_41}
Nitski, O., et~al.: Cdf-net: Cross-domain fusion network for accelerated mri
  reconstruction. In: Medical Image Computing and Computer Assisted
  Intervention -- MICCAI 2020. pp. 421--430. Springer International Publishing,
  Cham (2020)

\bibitem{nyquist1928certain}
Nyquist, H.: Certain topics in telegraph transmission theory. Transactions of
  the American Institute of Electrical Engineers  \textbf{47}(2),  617--644
  (1928)

\bibitem{parikh2014proximal}
Parikh, N., Boyd, S.: Proximal algorithms. Foundations and Trends in
  optimization  \textbf{1}(3),  127--239 (2014)

\bibitem{pedregosa2016hyperparameter}
Pedregosa, F.: Hyperparameter optimization with approximate gradient. In:
  International conference on machine learning. pp. 737--746. PMLR (2016)

\bibitem{pezzotti2020adaptive}
Pezzotti, N., et~al.: An adaptive intelligence algorithm for undersampled knee
  mri reconstruction: Application to the 2019 fastmri challenge.
  arXiv:2004.07339  (2020)

\bibitem{NMSE}
Poli, A., Cirillo, M.: On the use of the normalized mean square error in
  evaluating dispersion model performance. Atmospheric Environment. Part A.
  General Topics  \textbf{27},  2427--2434 (10 1993).
  \doi{10.1016/0960-1686(93)90410-Z}

\bibitem{pruessmann1999sense}
Pruessmann, K.P., et~al.: Sense: sensitivity encoding for fast mri. Magnetic
  Resonance in Medicine: An Official Journal of the International Society for
  Magnetic Resonance in Medicine  \textbf{42}(5),  952--962 (1999)

\bibitem{qiao2018few}
Qiao, S., Liu, C., Shen, W., Yuille, A.L.: Few-shot image recognition by
  predicting parameters from activations. In: Proceedings of the IEEE
  Conference on Computer Vision and Pattern Recognition. pp. 7229--7238 (2018)

\bibitem{qu2014magnetic}
Qu, X., Hou, Y., Lam, F., Guo, D., Zhong, J., Chen, Z.: Magnetic resonance
  image reconstruction from undersampled measurements using a patch-based
  nonlocal operator. Medical image analysis  \textbf{18}(6),  843--856 (2014)

\bibitem{Quan2018CompressedSM}
Quan, T.M., Nguyen-Duc, T., Jeong, W.K.: Compressed sensing mri reconstruction
  using a generative adversarial network with a cyclic loss. IEEE Transactions
  on Medical Imaging  \textbf{37},  1488--1497 (2018)

\bibitem{quan2018compressed}
Quan, T.M., Nguyen-Duc, T., Jeong, W.K.: Compressed sensing mri reconstruction
  using a generative adversarial network with a cyclic loss. IEEE transactions
  on medical imaging  \textbf{37}(6),  1488--1497 (2018)

\bibitem{quan2021image}
Quan, Y., Chen, Y., Shao, Y., Teng, H., Xu, Y., Ji, H.: Image denoising using
  complex-valued deep cnn. Pattern Recognition  \textbf{111},  107639 (2021)

\bibitem{quinonero2009dataset}
Qui{\~n}onero-Candela, J., Sugiyama, M., Lawrence, N.D., Schwaighofer, A.:
  Dataset shift in machine learning. Mit Press (2009)

\bibitem{rajeswaran2019meta}
Rajeswaran, A., Finn, C., Kakade, S.M., Levine, S.: Meta-learning with implicit
  gradients. Advances in neural information processing systems  \textbf{32}
  (2019)

\bibitem{ravi2016optimization}
Ravi, S., Larochelle, H.: Optimization as a model for few-shot learning. ICLR
  (2016)

\bibitem{ravishankar2010mr}
Ravishankar, S., Bresler, Y.: Mr image reconstruction from highly undersampled
  k-space data by dictionary learning. IEEE transactions on medical imaging
  \textbf{30}(5),  1028--1041 (2010)

\bibitem{Rebuffi17}
Rebuffi, S.A., Bilen, H., Vedaldi, A.: Learning multiple visual domains with
  residual adapters. In: Advances in Neural Information Processing Systems
  (2017)

\bibitem{roy2013magnetic}
Roy, S., Carass, A., Prince, J.L.: Magnetic resonance image example-based
  contrast synthesis. IEEE transactions on medical imaging  \textbf{32}(12),
  2348--2363 (2013)

\bibitem{roy2016patch}
Roy, S., Chou, Y.Y., Jog, A., Butman, J.A., Pham, D.L.: Patch based synthesis
  of whole head mr images: Application to epi distortion correction. In:
  International Workshop on Simulation and Synthesis in Medical Imaging. pp.
  146--156. Springer (2016)

\bibitem{rudin1992nonlinear}
Rudin, L.I., Osher, S., Fatemi, E.: Nonlinear total variation based noise
  removal algorithms. Physica D: nonlinear phenomena  \textbf{60}(1-4),
  259--268 (1992)

\bibitem{rusu2018meta}
Rusu, A.A., Rao, D., Sygnowski, J., Vinyals, O., Pascanu, R., Osindero, S.,
  Hadsell, R.: Meta-learning with latent embedding optimization. In:
  International Conference on Learning Representations (2018)

\bibitem{sandino2020compressed}
Sandino, C.M., et~al.: Compressed sensing: From research to clinical practice
  with deep neural networks: Shortening scan times for magnetic resonance
  imaging. IEEE Signal Processing Magazine  \textbf{37}(1),  117--127 (2020)

\bibitem{8495012}
{Scardapane}, S., et~al.: Complex-valued neural networks with nonparametric
  activation functions. IEEE Transactions on Emerging Topics in Computational
  Intelligence  \textbf{4}(2),  140--150 (2020)

\bibitem{8067520}
{Schlemper}, J., et~al.: A deep cascade of convolutional neural networks for
  dynamic mr image reconstruction. IEEE Transactions on Medical Imaging
  \textbf{37}(2),  491--503 (2018)

\bibitem{schlemper2017deep}
Schlemper, J., Caballero, J., Hajnal, J.V., Price, A.N., Rueckert, D.: A deep
  cascade of convolutional neural networks for dynamic mr image reconstruction.
  IEEE transactions on Medical Imaging  \textbf{37}(2),  491--503 (2017)

\bibitem{sharma2019missing}
Sharma, A., Hamarneh, G.: Missing mri pulse sequence synthesis using
  multi-modal generative adversarial network. IEEE transactions on medical
  imaging  \textbf{39}(4),  1170--1183 (2019)

\bibitem{singha2021deep}
Singha, A., Thakur, R.S., Patel, T.: Deep learning applications in medical
  image analysis. Biomedical Data Mining for Information Retrieval:
  Methodologies, Techniques and Applications pp. 293--350 (2021)

\bibitem{snell2017prototypical}
Snell, J., Swersky, K., Zemel, R.S.: Prototypical networks for few-shot
  learning. arXiv preprint arXiv:1703.05175  (2017)

\bibitem{sodickson1997simultaneous}
Sodickson, D.K., Manning, W.J.: Simultaneous acquisition of spatial harmonics
  (smash): fast imaging with radiofrequency coil arrays. Magnetic resonance in
  medicine  \textbf{38}(4),  591--603 (1997)

\bibitem{sohail2019unpaired}
Sohail, M., Riaz, M.N., Wu, J., Long, C., Li, S.: Unpaired multi-contrast mr
  image synthesis using generative adversarial networks. In: International
  Workshop on Simulation and Synthesis in Medical Imaging. pp. 22--31. Springer
  (2019)

\bibitem{souza2019hybrid}
Souza, R., Frayne, R.: A hybrid frequency-domain/image-domain deep network for
  magnetic resonance image reconstruction. In: 2019 32nd SIBGRAPI Conference on
  Graphics, Patterns and Images (SIBGRAPI). pp. 257--264. IEEE (2019)

\bibitem{pmlr-v102-souza19a}
Souza, R., Lebel, R.M., Frayne, R.: A hybrid, dual domain, cascade of
  convolutional neural networks for magnetic resonance image reconstruction.
  In: International Conference on Medical Imaging with Deep Learning. pp.
  437--446. PMLR (2019)

\bibitem{sriram2020end}
Sriram, A., Zbontar, J., Murrell, T., Defazio, A., Zitnick, C.L., Yakubova, N.,
  Knoll, F., Johnson, P.: End-to-end variational networks for accelerated mri
  reconstruction. In: International Conference on Medical Image Computing and
  Computer-Assisted Intervention. pp. 64--73. Springer (2020)

\bibitem{E2E-VN}
Sriram, A., et~al.: End-to-end variational networks for accelerated mri
  reconstruction. In: Medical Image Computing and Computer Assisted
  Intervention -- MICCAI 2020. pp. 64--73. Springer International Publishing,
  Cham (2020)

\bibitem{sriram2020grappanet}
Sriram, A., et~al.: Grappanet: Combining parallel imaging with deep learning
  for multi-coil mri reconstruction. In: Proceedings of the IEEE/CVF Conference
  on Computer Vision and Pattern Recognition. pp. 14315--14322 (2020)

\bibitem{10.1007/978-3-030-32248-9_5}
Sudarshan, V.P., et~al.: Joint reconstruction of pet + parallel-mri in a
  bayesian coupled-dictionary mrf framework. In: Medical Image Computing and
  Computer Assisted Intervention -- MICCAI 2019. pp. 39--47. Springer
  International Publishing, Cham (2019)

\bibitem{sun2016deep}
Sun, J., Li, H., Xu, Z., et~al.: Deep admm-net for compressive sensing mri.
  Advances in neural information processing systems  \textbf{29} (2016)

\bibitem{tavaf2021grappa}
Tavaf, N., Torfi, A., Ugurbil, K., Van~de Moortele, P.F.: Grappa-gans for
  parallel mri reconstruction. arXiv preprint arXiv:2101.03135  (2021)

\bibitem{thrun1998learning}
Thrun, S., Pratt, L.: Learning to learn: Introduction and overview. In:
  Learning to learn, pp. 3--17. Springer (1998)

\bibitem{torrado2016fast}
Torrado-Carvajal, A., Herraiz, J.L., Alcain, E., Montemayor, A.S.,
  Garcia-Canamaque, L., Hernandez-Tamames, J.A., Rozenholc, Y., Malpica, N.:
  Fast patch-based pseudo-ct synthesis from t1-weighted mr images for pet/mr
  attenuation correction in brain studies. Journal of Nuclear Medicine
  \textbf{57}(1),  136--143 (2016)

\bibitem{trabelsi2017deep}
Trabelsi, C., et~al.: Deep complex networks. arXiv:1705.09792  (2017),
  \url{https://arxiv.org/abs/1705.09792}

\bibitem{48798}
Triantafillou, E., Zhu, T., Dumoulin, V., Lamblin, P., Evci, U., Xu, K.,
  Goroshin, R., Gelada, C., Swersky, K.J., Manzagol, P.A., Larochelle, H.:
  Meta-dataset: A dataset of datasets for learning to learn from few examples.
  In: International Conference on Learning Representations (2020)

\bibitem{trzasko2011local}
Trzasko, J., Manduca, A., Borisch, E.: Local versus global low-rank promotion
  in dynamic mri series reconstruction. In: Proc. Int. Symp. Magn. Reson. Med.
  vol.~19, p.~4371 (2011)

\bibitem{uecker2013software}
Uecker, M., et~al.: Software toolbox and programming library for compressed
  sensing and parallel imaging. In: ISMRM Workshop on Data Sampling and Image
  Reconstruction. p.~41. Citeseer (2013)

\bibitem{uecker2014espirit}
Uecker, M., et~al.: Espirit—an eigenvalue approach to autocalibrating
  parallel mri: where sense meets grappa. Magnetic resonance in medicine
  \textbf{71}(3),  990--1001 (2014)

\bibitem{valkonen2014primal}
Valkonen, T.: A primal--dual hybrid gradient method for nonlinear operators
  with applications to mri. Inverse Problems  \textbf{30}(5),  055012 (2014)

\bibitem{van2015does}
Van~Tulder, G., de~Bruijne, M.: Why does synthesized data improve
  multi-sequence classification? In: International Conference on Medical Image
  Computing and Computer-Assisted Intervention. pp. 531--538. Springer (2015)

\bibitem{vasudeva2020covegan}
Vasudeva, B., Deora, P., Bhattacharya, S., Pradhan, P.M.: Co-vegan:
  Complex-valued generative adversarial network for compressive sensing mr
  image reconstruction (2020)

\bibitem{vinyals2016matching}
Vinyals, O., Blundell, C., Lillicrap, T., Wierstra, D., et~al.: Matching
  networks for one shot learning. Advances in neural information processing
  systems  \textbf{29},  3630--3638 (2016)

\bibitem{8297024}
{Virtue}, P., {Yu}, S.X., {Lustig}, M.: Better than real: Complex-valued neural
  nets for mri fingerprinting. In: 2017 IEEE International Conference on Image
  Processing (ICIP). pp. 3953--3957 (2017)

\bibitem{vuorio2019multimodal}
Vuorio, R., Sun, S.H., Hu, H., Lim, J.J.: Multimodal model-agnostic
  meta-learning via task-aware modulation. arXiv preprint arXiv:1910.13616
  (2019)

\bibitem{RSS_bias}
Walsh, D.O., Gmitro, A.F., Marcellin, M.W.: Adaptive reconstruction of phased
  array mr imagery. Magnetic Resonance in Medicine  \textbf{43}(5),  682--690
  (2000). \doi{10.1002/(SICI)1522-2594(200005)43:5<682::AID-MRM10>3.0.CO;2-G}

\bibitem{Walsh2000682}
Walsh, D., Gmitro, A., Marcellin, M.: Adaptive reconstruction of phased array
  mr imagery. Magnetic Resonance in Medicine  \textbf{43}(5),  682--690 (2000)

\bibitem{7493320}
{Wang}, S., {Su}, Z., {Ying}, L., {Peng}, X., {Zhu}, S., {Liang}, F., {Feng},
  D., {Liang}, D.: Accelerating magnetic resonance imaging via deep learning.
  In: 2016 IEEE 13th International Symposium on Biomedical Imaging (ISBI). pp.
  514--517 (April 2016)

\bibitem{wang2017learning}
Wang, S., Tan, S., Gao, Y., Liu, Q., Ying, L., Xiao, T., Liu, Y., Liu, X.,
  Zheng, H., Liang, D.: Learning joint-sparse codes for calibration-free
  parallel mr imaging. IEEE transactions on medical imaging  \textbf{37}(1),
  251--261 (2017)

\bibitem{WANG2020136}
Wang, S., et~al.: Deepcomplexmri: Exploiting deep residual network for fast
  parallel mr imaging with complex convolution. Magnetic Resonance Imaging
  \textbf{68},  136 -- 147 (2020)

\bibitem{wang2020ikwi}
Wang, Z., et~al.: Ikwi-net: A cross-domain convolutional neural network for
  undersampled magnetic resonance image reconstruction. Magnetic Resonance
  Imaging  \textbf{73},  1--10 (2020)

\bibitem{wang2004image}
Wang, Z., et~al.: Image quality assessment: from error visibility to structural
  similarity. IEEE transactions on image processing  \textbf{13}(4),  600--612
  (2004)

\bibitem{weinan2017proposal}
Weinan, E.: A proposal on machine learning via dynamical systems.
  Communications in Mathematics and Statistics  \textbf{5}(1),  1--11 (2017)

\bibitem{welander2018generative}
Welander, P., Karlsson, S., Eklund, A.: Generative adversarial networks for
  image-to-image translation on multi-contrast mr images-a comparison of
  cyclegan and unit. arXiv preprint arXiv:1806.07777  (2018)

\bibitem{wichrowska2017learned}
Wichrowska, O., Maheswaranathan, N., Hoffman, M.W., Colmenarejo, S.G., Denil,
  M., Freitas, N., Sohl-Dickstein, J.: Learned optimizers that scale and
  generalize. In: International Conference on Machine Learning. pp. 3751--3760.
  PMLR (2017)

\bibitem{wood1985truncation}
Wood, M.L., Mark~Henkelman, R.: Truncation artifacts in magnetic resonance
  imaging. Magnetic resonance in medicine  \textbf{2}(6),  517--526 (1985)

\bibitem{wu2010augmented}
Wu, C., Tai, X.C.: Augmented lagrangian method, dual methods, and split bregman
  iteration for rof, vectorial tv, and high order models. SIAM Journal on
  Imaging Sciences  \textbf{3}(3),  300--339 (2010)

\bibitem{yang2017dagan}
Yang, G., Yu, S., Dong, H., Slabaugh, G., Dragotti, P.L., Ye, X., Liu, F.,
  Arridge, S., Keegan, J., Guo, Y., et~al.: Dagan: Deep de-aliasing generative
  adversarial networks for fast compressed sensing mri reconstruction. IEEE
  transactions on medical imaging  \textbf{37}(6),  1310--1321 (2017)

\bibitem{yang2010fast}
Yang, J., Zhang, Y., Yin, W.: A fast alternating direction method for tvl1-l2
  signal reconstruction from partial fourier data. IEEE Journal of Selected
  Topics in Signal Processing  \textbf{4}(2),  288--297 (2010)

\bibitem{yang2018mri}
Yang, Q., Li, N., Zhao, Z., Fan, X., Chang, E.C., Xu, Y., et~al.: Mri
  image-to-image translation for cross-modality image registration and
  segmentation. arXiv preprint arXiv:1801.06940  (2018)

\bibitem{8550778}
{Yang}, Y., et~al.: Admm-csnet: A deep learning approach for image compressive
  sensing. IEEE Transactions on Pattern Analysis and Machine Intelligence
  \textbf{42}(3),  521--538 (March 2020)

\bibitem{yang2018admm}
Yang, Y., Sun, J., Li, H., Xu, Z.: Admm-csnet: A deep learning approach for
  image compressive sensing. IEEE transactions on pattern analysis and machine
  intelligence  \textbf{42}(3),  521--538 (2018)

\bibitem{yang2020model}
Yang, Y., Wang, N., Yang, H., Sun, J., Xu, Z.: Model-driven deep attention
  network for ultra-fast compressive sensing mri guided by cross-contrast mr
  image. In: International Conference on Medical Image Computing and
  Computer-Assisted Intervention. pp. 188--198. Springer (2020)

\bibitem{NIPS2016_6406}
Yang, Y., et~al.: Deep admm-net for compressive sensing mri. In: Lee, D.D.,
  Sugiyama, M., Luxburg, U.V., Guyon, I., Garnett, R. (eds.) Advances in Neural
  Information Processing Systems 29, pp. 10--18. Curran Associates, Inc. (2016)

\bibitem{yang2012nonlocal}
Yang, Z., Jacob, M.: Nonlocal regularization of inverse problems: a unified
  variational framework. IEEE Transactions on Image Processing  \textbf{22}(8),
   3192--3203 (2012)

\bibitem{yao2021improving}
Yao, H., Huang, L.K., Zhang, L., Wei, Y., Tian, L., Zou, J., Huang, J., et~al.:
  Improving generalization in meta-learning via task augmentation. In:
  International Conference on Machine Learning. pp. 11887--11897. PMLR (2021)

\bibitem{yao2019hierarchically}
Yao, H., Wei, Y., Huang, J., Li, Z.: Hierarchically structured meta-learning.
  In: International Conference on Machine Learning. pp. 7045--7054. PMLR (2019)

\bibitem{yao2020automated}
Yao, H., Wu, X., Tao, Z., Li, Y., Ding, B., Li, R., Li, Z.: Automated
  relational meta-learning. arXiv preprint arXiv:2001.00745  (2020)

\bibitem{yin2020metalearning}
Yin, M., Tucker, G., Zhou, M., Levine, S., Finn, C.: Meta-learning without
  memorization. In: International Conference on Learning Representations (2020)

\bibitem{yoon2018bayesian}
Yoon, J., Kim, T., Dia, O., Kim, S., Bengio, Y., Ahn, S.: Bayesian
  model-agnostic meta-learning. In: Proceedings of the 32nd International
  Conference on Neural Information Processing Systems. pp. 7343--7353 (2018)

\bibitem{yu20183d}
Yu, B., Zhou, L., Wang, L., Fripp, J., Bourgeat, P.: 3d cgan based
  cross-modality mr image synthesis for brain tumor segmentation. In: 2018 IEEE
  15th International Symposium on Biomedical Imaging (ISBI 2018). pp. 626--630.
  IEEE (2018)

\bibitem{yu2020meta}
Yu, T., Quillen, D., He, Z., Julian, R., Hausman, K., Finn, C., Levine, S.:
  Meta-world: A benchmark and evaluation for multi-task and meta reinforcement
  learning. In: Conference on Robot Learning. pp. 1094--1100. PMLR (2020)

\bibitem{yurt2021mustgan}
Yurt, M., Dar, S.U., Erdem, A., Erdem, E., Oguz, K.K., {\c{C}}ukur, T.:
  Mustgan: Multi-stream generative adversarial networks for mr image synthesis.
  Medical Image Analysis  \textbf{70},  101944 (2021)

\bibitem{zbontar2018fastmri}
Zbontar, J., et~al.: {fastMRI}: An open dataset and benchmarks for accelerated
  {MRI}. arXiv:1811.08839  (2018), \url{https://arxiv.org/abs/1811.08839}

\bibitem{zhan2015fast}
Zhan, Z., Cai, J.F., Guo, D., Liu, Y., Chen, Z., Qu, X.: Fast multiclass
  dictionaries learning with geometrical directions in mri reconstruction. IEEE
  Transactions on biomedical engineering  \textbf{63}(9),  1850--1861 (2015)

\bibitem{10.1007/978-3-030-33843-5_4}
Zhang, C., et~al.: Apir-net: Autocalibrated parallel imaging reconstruction
  using a neural network. In: Machine Learning for Medical Image
  Reconstruction. pp. 36--46. Springer International Publishing, Cham (2019)

\bibitem{zhang2018ista}
Zhang, J., Ghanem, B.: Ista-net: Interpretable optimization-inspired deep
  network for image compressive sensing. In: Proceedings of the IEEE Conference
  on Computer Vision and Pattern Recognition. pp. 1828--1837 (2018)

\bibitem{zhang2017learning}
Zhang, K., Zuo, W., Gu, S., Zhang, L.: Learning deep cnn denoiser prior for
  image restoration. In: Proceedings of the IEEE conference on computer vision
  and pattern recognition. pp. 3929--3938 (2017)

\bibitem{zhang2020deep}
Zhang, X., Lian, Q., Yang, Y., Su, Y.: A deep unrolling network inspired by
  total variation for compressed sensing mri. Digital Signal Processing
  \textbf{107},  102856 (2020)

\bibitem{zhang2010bregmanized}
Zhang, X., Burger, M., Bresson, X., Osher, S.: Bregmanized nonlocal
  regularization for deconvolution and sparse reconstruction. SIAM Journal on
  Imaging Sciences  \textbf{3}(3),  253--276 (2010)

\bibitem{7797130}
{Zhao}, H., et~al.: Loss functions for image restoration with neural networks.
  IEEE Transactions on Computational Imaging  \textbf{3}(1),  47--57 (2017)

\bibitem{zhou2020review}
Zhou, S.K., Greenspan, H., Davatzikos, C., Duncan, J.S., van Ginneken, B.,
  Madabhushi, A., Prince, J.L., Rueckert, D., Summers, R.M.: A review of deep
  learning in medical imaging: Image traits, technology trends, case studies
  with progress highlights, and future promises. Unknown Journal  (2020)

\bibitem{HiNet}
Zhou, T., Fu, H., Chen, G., Shen, J., Shao, L.: Hi-net: Hybrid-fusion network
  for multi-modal mr image synthesis. IEEE Transactions on Medical Imaging
  \textbf{39}(9),  2772--2781 (2020)

\bibitem{zhou2020hi}
Zhou, T., Fu, H., Chen, G., Shen, J., Shao, L.: Hi-net: hybrid-fusion network
  for multi-modal mr image synthesis. IEEE transactions on medical imaging
  \textbf{39}(9),  2772--2781 (2020)

\bibitem{doi:10.1002/mp.13628}
Zhou, Z., Han, F., Ghodrati, V., Gao, Y., Yin, W., Yang, Y., Hu, P.: Parallel
  imaging and convolutional neural network combined fast mr image
  reconstruction: Applications in low-latency accelerated real-time imaging.
  Medical Physics  \textbf{46}(8),  3399--3413 (2019). \doi{10.1002/mp.13628},
  \url{https://aapm.onlinelibrary.wiley.com/doi/abs/10.1002/mp.13628}

\bibitem{zhou_pmri}
Zhou, Z., et~al.: Parallel imaging and convolutional neural network combined
  fast mr image reconstruction: Applications in low-latency accelerated
  real-time imaging. Medical Physics  \textbf{46}(8),  3399--3413 (2019)

\bibitem{zhu2018}
Zhu, B., et~al.: Image reconstruction by domain-transform manifold learning.
  Nature  \textbf{555},  487--492 (2018). \doi{10.1038/nature25988}

\bibitem{zhu2008efficient}
Zhu, M., Chan, T.: An efficient primal-dual hybrid gradient algorithm for total
  variation image restoration. UCLA CAM Report  \textbf{34},  8--34 (2008)

\bibitem{10.1007/978-3-030-59713-9_37}
Zhu, Y., et~al.: Cross-domain medical image translation by shared latent
  gaussian mixture model. In: Medical Image Computing and Computer Assisted
  Intervention -- MICCAI 2020. pp. 379--389. Springer International Publishing,
  Cham (2020)

\end{thebibliography}
